\documentclass{article}

\usepackage[top=2.6cm, bottom=3.3cm, left=3.3cm, right=3.3cm]{geometry}
\usepackage{setspace}
\usepackage{latexsym}
\usepackage{amssymb}
\usepackage{amsmath}
\usepackage{amsfonts}
\usepackage{amsthm}
\usepackage{graphicx}
\usepackage[all,cmtip]{xy}
\usepackage{mathrsfs}

\newtheorem{theorem}{Theorem}[section]
\newtheorem{lemma}[theorem]{Lemma}
\newtheorem{proposition}[theorem]{Proposition}

\theoremstyle{definition}
\newtheorem{definition}{Definition}[section]
\theoremstyle{definition}
\newtheorem{remark}{Remark}[section]
\theoremstyle{remark}
\numberwithin{equation}{section}

\def\stackrel#1#2{\mathrel{\mathop{#2}\limits^{#1}}}

\newcommand{\circlenum}[1]{\raisebox{.5pt}{\textcircled{\raisebox{-.9pt} {#1}}}}
\newcommand{\mathfigns}[2]{\vcenter{\hbox{\includegraphics[height=#1 cm]{#2}}}}
\newcommand{\mathfig}[2]{\quad\vcenter{\hbox{\includegraphics[height=#1 cm]{#2}}} \quad}
\newcommand{\mathfignum}[3]{\quad\stackrel{\circlenum{#3}}{\vcenter{\hbox{\includegraphics[height=#1 cm]{#2}}}} \quad}
\newcommand{\mathfigbox}[3]{\quad\stackrel{\framebox{#3}}{\vcenter{\hbox{\includegraphics[height=#1 cm]{#2}}}} \quad}
\newcommand{\mathfignumns}[3]{\stackrel{\circlenum{#3}}{\vcenter{\hbox{\includegraphics[height=#1 cm]{#2}}}} }

\begin{document}
\unitlength = 1mm
\begin{onehalfspace}

\title{Elliptic Associators and the LMO Functor}
\author{Ronen Katz}
\maketitle
\begin{doublespace}
\begin{abstract}
The elliptic associator of Enriquez can be used to define an invariant of tangles embedded in the thickened torus, which extends the Kontsevich integral. This construction by Humbert uses the formulation of categories with elliptic structures.

In this work we show that an extension of the LMO functor also leads to an elliptic structure on the category of Jacobi diagrams which is used by the Kontsevich integral, and find the relation between the two structures. We use this relation to give an alternative proof for the properties of the elliptic associator of Enriquez. Those results can lead the way to finding associators for higher genra.
\end{abstract}
\end{doublespace}

\newpage
\tableofcontents

\newpage\section{Introduction}

The Kontsevich integral was originally defined as an invariant of knots which returns values in a space of Jacobi diagrams (\cite{Kontsevich}). Its importance comes mainly from the fact that it is universal with respect to Vassiliev invariants (see also \cite{BarNatan}). A new, combinatorial, formulation of the Kontsevich invariant for framed knots and links was given in \cite{Le}. The idea of this definition is to break the link to elementary tangles and asign values to each of them separately.

The paper \cite{Le} already contains the extension of the Konstevich integral to tangles, which is defined more explicitly in \cite{Le2} and \cite{BarNatan4}. In this context it is convenient to consider all tangles as morphisms in the category $q\tilde{T}$ of non-associative framed tangles. This category has the structure of a ribbon category - a concept which encapsulates the above elementary tangles and the relations between them. The Kontsevich integral becomes a functor to a category of Jacobi diagrams $\mathcal{A}^\partial$, which also gets a structure of a ribbon category.

In \cite{Humbert}, the Kontsevich integral is further extended to an invariant of tangles embedded in a  thickened torus $\mathbb{T}\times I$. The category of those tangles is denoted $q\tilde{T}_1$, and is an extension of the ribbon category $q\tilde{T}$. The new ingredients in this category are two elementary tangles which go around the generators of $\pi_1(\mathbb{T})$. Those tangles are called beaks. The beaks, together with the relations between them and the other elementary tangles, is encapsulated in the concept of an elliptic structure. 

In order to define the extension of the Kontsevich integral to $q\tilde{T}_1$, one needs to find a category of Jacobi diagrams extending $\mathcal{A}^\partial$ such that this extension has an elliptic structure. This category is defined in \cite{Humbert} and denoted $\mathcal{A}_1$ - the category of elliptic Jacobi diagrams. To give this extension an elliptic structure, \cite{Humbert} uses the concept of an elliptic associator. An elliptic associator is a pair of elements in the exponent of $\hat{f}(A,B)$ - the completed free Lie algebra generated by $A$ and $B$, satisfying several identities. This associator, when mapped into $\mathcal{A}_1$, determines the value of the Kontsevich invariant on the beaks. A specific elliptic associator $e(\phi)$ is introduced in \cite{Calaque} and \cite{Enriquez}.

In a different direction, the Kontsevich integral was also used to define the LMO invariant, which is an invariant of closed $3$-manifolds which returns values in some spaces of Jacobi diagrams (\cite{Le3}). It was extended to a TQFT first in \cite{Murakami} and, several years later, in \cite{ChepteaLe}. A functorial variant of this construction was given in \cite{Cheptea}, called the LMO functor.

In this work we extend the LMO functor to an invariant of $3$-cobordisms with embedded tangles. Restricted to tangles embedded in the thickened torus, this invariant can be compared to the invariant of \cite{Humbert}. It turns out that those invariants \textbf{are not} equal. However, their values on the beaks are equivalent modulo a certain relation called the homotopy relation. We use this equivalence to give an alternative, more intuitive, proof that $e(\phi)$ is indeed an elliptic associator.

This work is divided into the following sections:

In section \ref{section_tangles} we review the definitions of ribbon categories and elliptic structures. As we explained, those concepts encapsulate the strucure of the categories of tangles $q\tilde{T}$ and $q\tilde{T}_1$, and give us a tool to define invariants of tangles.

In section \ref{section_jacobi} we review several variants of categories of Jacobi diagrams which will be used later, and explain the relations between them.

In section \ref{section_LMO} we define our extension of the LMO functor to the category of embedded tangles in $3$-cobordisms. For that purpose we explain how to represent a tangle embedded in a $3$-cobordism using a representing tangle in $D^2\times I$. We also give an explicit description of the beaks using those representing tangles, and verify their properties.

In section \ref{section_t} we introduce the Lie algebras $\textbf{t}_{1,n}$, and explain how their universal enveloping algebras are mapped into the spaces $\mathcal{A}_1^{<p}(\uparrow^n)$ of Jacobi diagrams. For $n=2,3$ we prove that a certain restriction of this map is actually an injection into a certain quotient of $\mathcal{A}_1^{<p}(\uparrow^n)$.

In section \ref{section_elliptic} we introduce the concept of elliptic associators and the elliptic associator $e(\phi)$ of \cite{Enriquez}. We prove our main theorem, which determines the relation between $e(\phi)$ and the extended LMO functor. We conclude by using this theorem to give a new proof that $e(\phi)$ is indeed an elliptic associator.

Our results can lead the way to extending the concept of elliptic associators to higher genra, and to finding specific such associators using the extended LMO functor.

\newpage\section{Categories of Tangles}
\label{section_tangles}
In this section we define the concepts of ribbon categories and elliptic structures, and introduce the categories of tangles which are the universal examples for those concepts.

\subsection{Ribbon Categories}
In this subsection we recall the definition of a ribbon category, and give the main example - the category of tangles in $D^2\times I$. The definitions are all taken from \cite{Humbert}.

Let $(\mathcal{C},\otimes,\textbf{1},a)$ be a non-associative monoidal category. Assume for simplicity that $\textbf{1}\otimes U=U\otimes \textbf{1}=U$ for any object $U\in \mathcal{C}$. For shortness we will denote $U\otimes V$ by $UV$. $a$ is a family of natural isomorphisms $a_{X,Y,Z}:(XY)Z\rightarrow X(YZ)$ for any objects $X,Y,Z\in\mathcal{C}$, satisfying the pentagon relation for any objects $X,Y,Z,W\in\mathcal{C}$:
$$a_{W,X,YZ}a_{WX,Y,Z}=(id_W\otimes a_{X,Y,Z})a_{W,XY,Z}(a_{W,X,Y}\otimes id_Z)$$

\textbf{A duality} on $\mathcal{C}$ is a rule that associates to each object $V$ an object $V^*$ and $2$ morphisms $b_V:\textbf{1}\rightarrow V\otimes V^*$ and $d_V:V^*\otimes V\rightarrow\textbf{1}$, satisfying:
$$(id_V\otimes d_V)a_{V,V^*,V}(b_V\otimes id_V)=id_V$$
$$(d_V\otimes id_{V^*})a^{-1}_{V^*,V,V^*}(id_{V^*}\otimes b_V)=id_{V^*}$$

\textbf{A braiding} on $\mathcal{C}$ is a family of natural isomorphisms $c_{U,V}:UV\rightarrow VU$ for any $2$ objects $U,V\in\mathcal{C}$, satisfying for any objects $U,V,W\in\mathcal{C}$:
$$c_{UV,W}=a_{W,U,V}(c_{U,W}\otimes id_V)a^{-1}_{U,W,V}(id_U\otimes c_{V,W})a_{U,V,W}$$
$$c_{U,VW}=a^{-1}_{V,W,U}(id_V\otimes c_{U,W})a_{V,U,W}(c_{U,V}\otimes id_W)a^{-1}_{U,V,W}$$
By convension we denote $c^{-1}_{U,V}:=(c_{V,U})^{-1}$.

\textbf{A twist} on a monoidal category with braiding is a family of natural isomorphisms $\theta_V:V\rightarrow V$ for any object $V\in\mathcal{C}$, satisfying:
$$\theta_{UV}=c_{V,U}c_{U,V}(\theta_U\otimes\theta_V)$$

\begin{definition}
\label{def_tangle}
\textbf{A ribbon category} is a monoidal category $(\mathcal{C},\otimes,\textbf{1},a)$ with duality, braiding and twist as above, satisfying, for any object $V\in\mathcal{C}$:
$$(\theta_V\otimes id_{V^*})b_V=(id_V\otimes\theta_{V^*})b_V$$
\end{definition}

The main example for a ribbon category is the category of framed oriented tangles, which we will now describe.

Let $D^2\subset \mathbb{R}^2$ be the unit disk. Denote by $b_n$ a sequence of some fixed $n$ points in the interior of $D^2$ (say, the $n$ points uniformly distributed along the segment $(-1,1)$ of the $x$ axis).
Let $m,n\ge 0$ be integers, and $\omega_s,\omega_t$ non-associative words in the symbols $\{+,-\}$ of lengths $m,n$, respectively. A \textbf{tangle} in $D^2\times I$ of type $(\omega_s,\omega_t)$ is an oriented $1$-manifold $\gamma$ embedded in $D^2\times I$, such that its only boundary points are $\gamma\cap (D^2\times \{0\})=b_m\times\{0\}$ and $\gamma\cap (D^2\times \{1\})=b_n\times\{1\}$, and such that the orientations of the tangle around the points of $b_m\times\{0\}$ and $b_n\times\{1\}$ correspond to the symbols of $\omega_s$ and $\omega_t$ (a ``$+$" symbol corresponds to a strand going ``up", i.e. in the positive direction of $I$, and a ``$-$" symbol corresponds to a strand going ``down"). The embedding of $\gamma$ should be piece-wise smooth and transverse to the horizontal surfaces $D^2\times \{t\}$ at all but a finite number of points. We also require the tangle to be vertical (i.e. of the form $\{z\}\times I$) near the boundary points.

A \textbf{framing} on a tangle $\gamma$ is the homotopy class relative to the boundary of a non-zero normal vector field on the smooth points of $\gamma$, such that the limit of this vector field at the non-smooth points is the same from both sides. The vectors based on the boundary points should all be parametrized as $(0,-1,0)$. In figures we will use the convention of the blackboard framing.

\begin{definition}
The category $q\tilde{T}$ is the category whose objects are non-associative words in $\{+,-\}$, and whose sets of morphisms $q\tilde{T}(\omega_s,\omega_t)$ are the sets of ambient isotopy classes of oriented framed tangles of type $(\omega_s,\omega_t)$.
\end{definition}

The concept of a ribbon category is designed to encapsulate the structure of $q\tilde{T}$. More specifically, we have the following proposition, which is easy to verify:

\begin{proposition}
$q\tilde{T}$ is a ribbon category, where we define the dual of $\omega=(\omega_1,...,\omega_n)$ to be $\omega^*=(-\omega_n,...,-\omega_1)$, and:

$$ a_{\omega_1,\omega_2,\omega_3}=
\begin{array}{ccc}
\omega_1 & (\omega_2 & \omega_3)\\
\includegraphics[height=20pt]{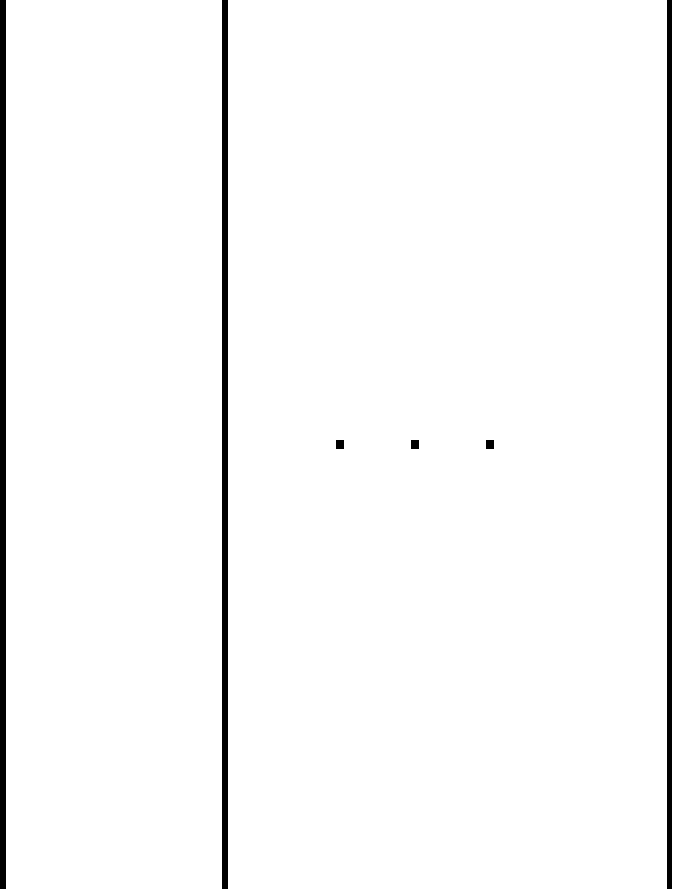} & \includegraphics[height=20pt]{section1_a.pdf} & \includegraphics[height=20pt]{section1_a.pdf}\\
(\omega_1 & \omega_2) & \omega_3\\
\end{array}
$$

$$ b_\omega=
\begin{array}{c}
\omega\qquad\qquad\quad \omega^*\\
\includegraphics[height=40pt]{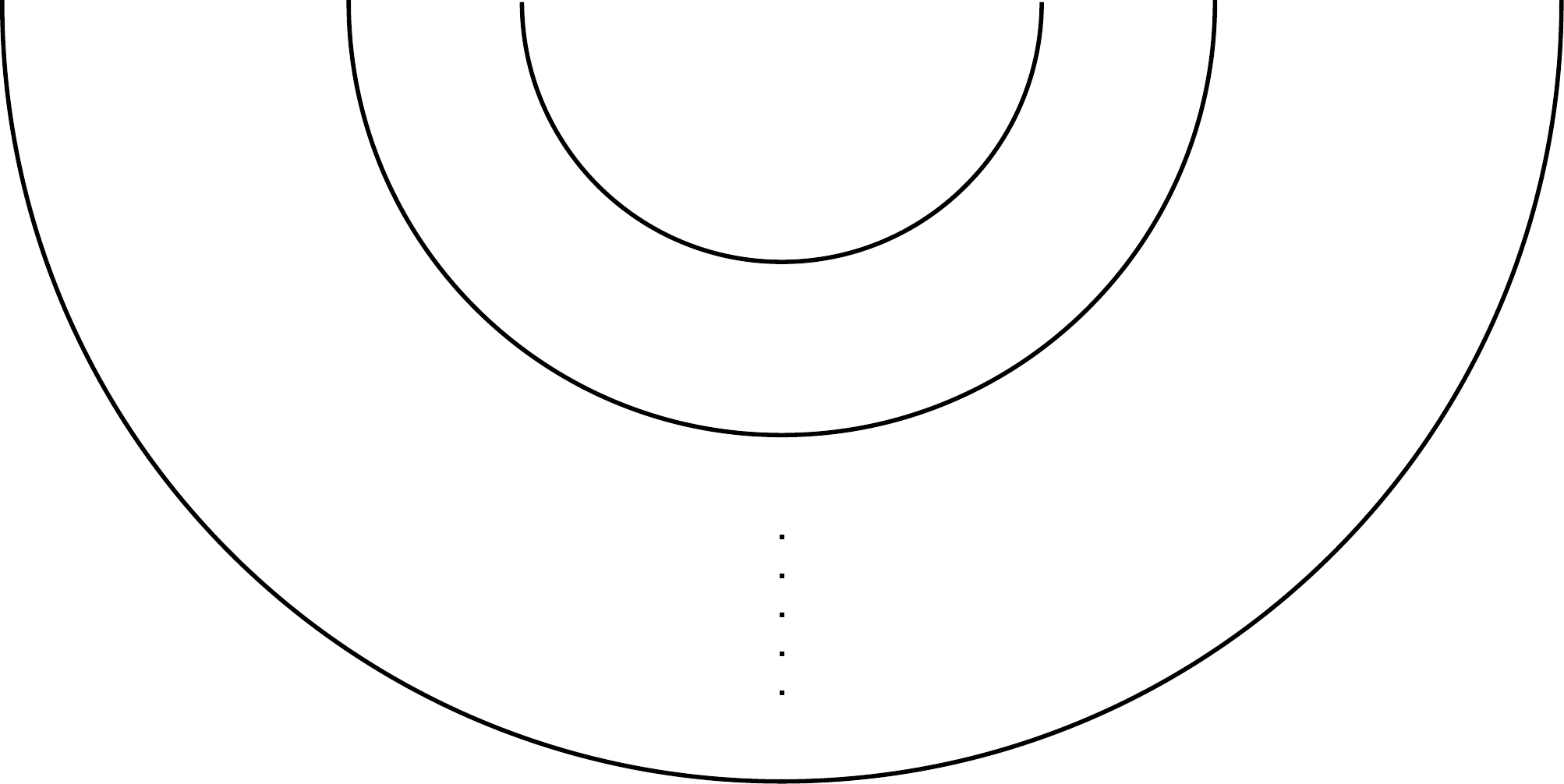}\\
\end{array}
\quad\quad\quad\quad
d_\omega=
\begin{array}{c}
\includegraphics[height=40pt]{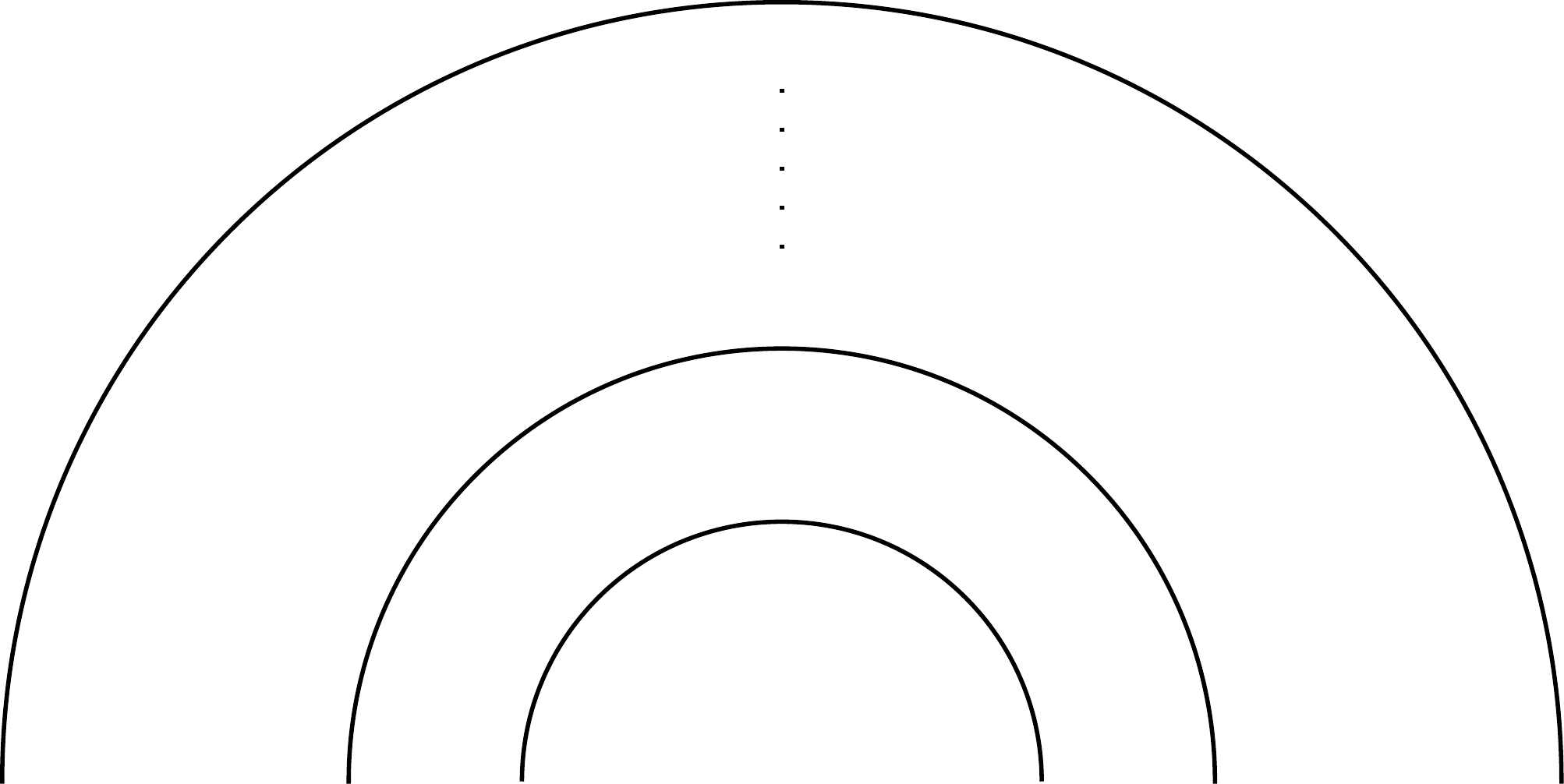}\\
\omega\qquad\qquad\quad \omega^*\\
\end{array}
$$

$$ c_{\omega_1,\omega_2}=
\begin{array}{c}
\omega_2 \qquad\quad \omega_1\\
\includegraphics[height=50pt]{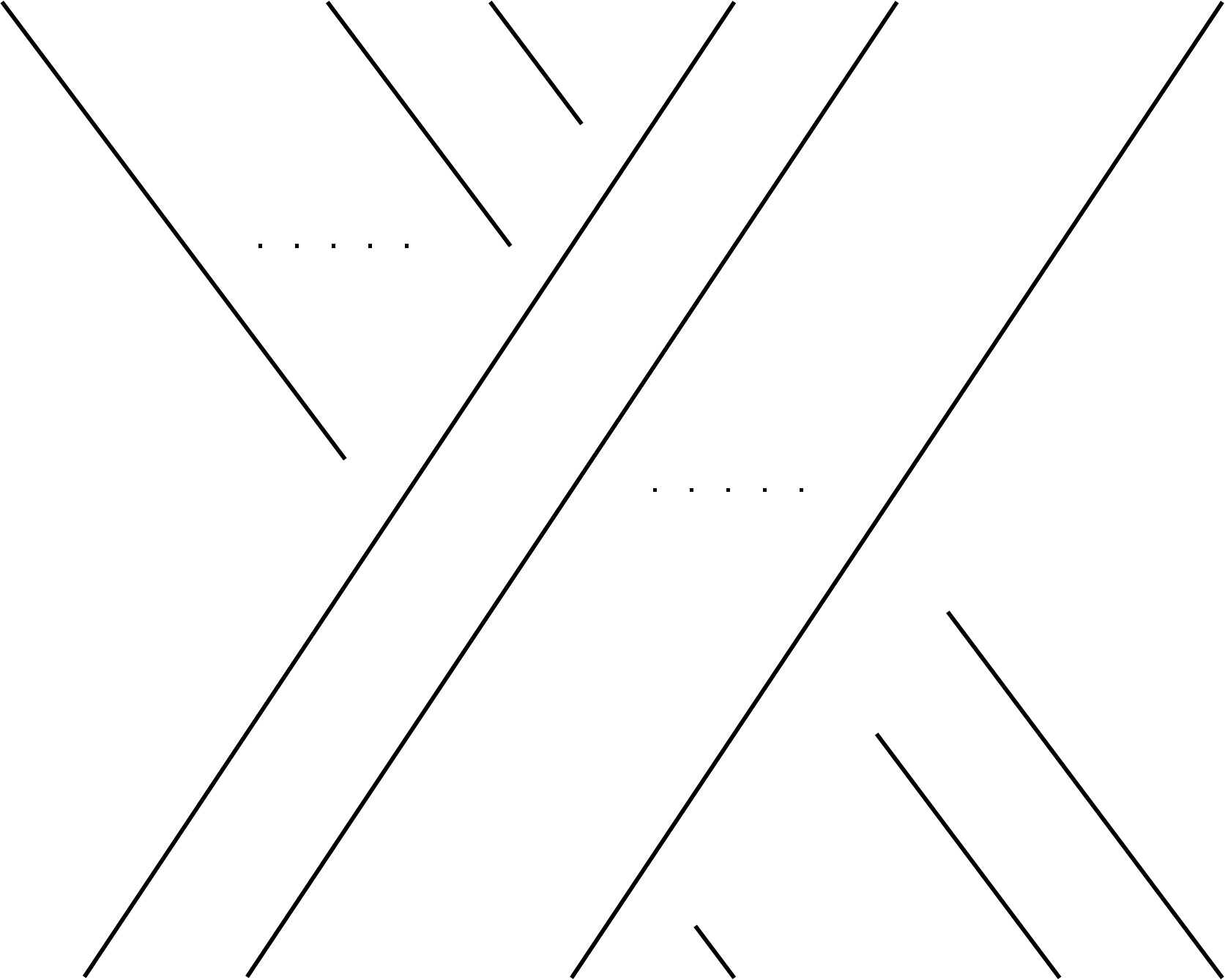}\\
\omega_1\qquad\quad\omega_2\\
\end{array}
\qquad\qquad
\theta_\omega=
\begin{array}{c}
\qquad\quad \omega\\
\includegraphics[height=70pt]{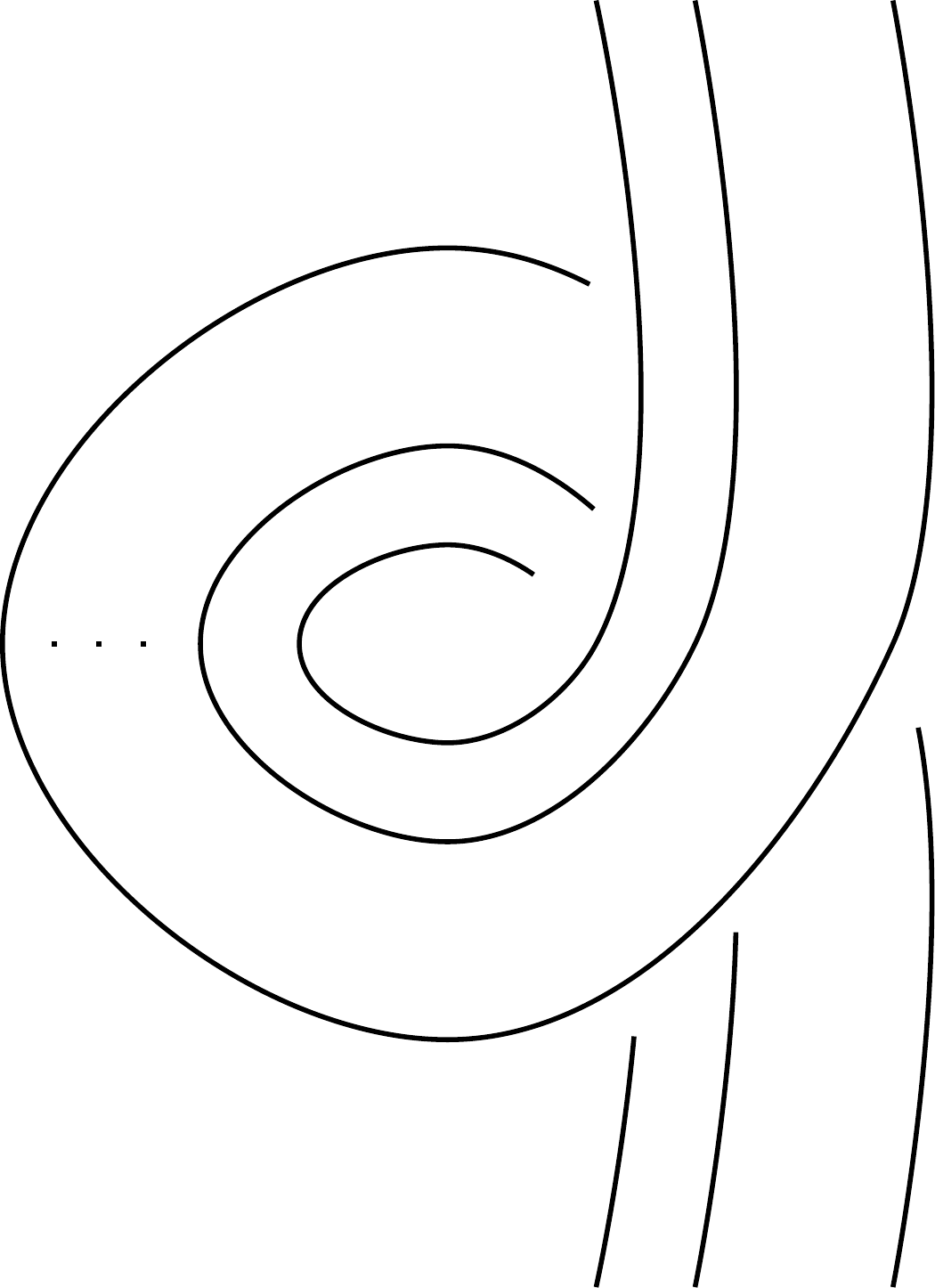}\\
\qquad\quad\omega\\
\end{array}
$$

\end{proposition}

In fact, the category $q\tilde{T}$ is the universal ribbon category, in the sense that there is a unique functor from it to any other ribbon category, preserving most of its properties. A precise formulation and proof of this theorem can be found in \cite{Shum}.

\subsection{Elliptic Structure}
We will now define the concept of elliptic structure. The definitions in this section are also taken from \cite{Humbert}.

Let $\mathcal{C}$ be a ribbon category, $\mathcal{C}_1$ any category, and $\{\cdot\}:\mathcal{C}\rightarrow\mathcal{C}_1$ a functor.
\begin{definition}
\label{def_elliptic}
An \textbf{elliptic structure} relative to $(\mathcal{C}\rightarrow\mathcal{C}_1)$ is a pair $(X,Y)$ of natural automorphisms of the functor $\{\cdot\otimes\cdot\}:\mathcal{C}\times\mathcal{C}\rightarrow\mathcal{C}_1$ (i.e. the composition of the tensor product of $\mathcal{C}$ with the given functor $\{\cdot\}$), satisfying the following identities for any objects $U,V,W\in\mathcal{C}$ (where for $Z=X$ or $Z=Y$ we denote $Z'_{U,V,W}:=\{a^{-1}_{U,V,W}\}Z_{U,VW}\{a_{U,V,W}\}$):
\begin{equation}
\label{elliptic1}
X_{UV,W}=X'_{U,V,W}\{c_{V,U}\otimes id_W\}X'_{V,U,W}\{c_{U,V}\otimes id_W\}
\end{equation}
\begin{equation}
\label{elliptic2}
Y_{UV,W}=Y'_{U,V,W}\{c^{-1}_{V,U}\otimes id_W\}Y'_{V,U,W}\{c^{-1}_{U,V}\otimes id_W\}
\end{equation}
\begin{equation}
\label{elliptic3}
Y_{U,V}X_{U,V}Y_{U,V}^{-1}X_{U,V}^{-1}=\{c_{V,U}c_{U,V}\}
\end{equation}
\begin{equation}
\label{elliptic4}
Y'_{U,V,W}\{c_{U,V}\otimes id_W\}X'_{V,U,W}\{c_{U,V}\otimes id_W\}=\{c_{V,U}\otimes id_W\}X'_{V,U,W}\{c^{-1}_{U,V}\otimes id_W\}Y'_{U,V,W}
\end{equation}
\end{definition}

We will now describe the main example for a category with an elliptic structure, which is the category of framed oriented tangles in the thickened torus.

Let $\mathbb{T}:=S^1\times S^1$ be the torus. We fix an embedding $D^2\subset\mathbb{T}$. This embedding also gives us an embedding of all the sets of points $b_n$ into $\mathbb{T}$. The torus $\mathbb{T}$ (minus an open neighborhood of a point at infinity) is depicted in figure \ref{torus}. This figure also shows the embedded disk $D^2$, and $2$ generators $x$ and $y$ of $\pi_1(\mathbb{T})$.

\begin{figure}[ht]

\centering
\includegraphics[width=12cm]{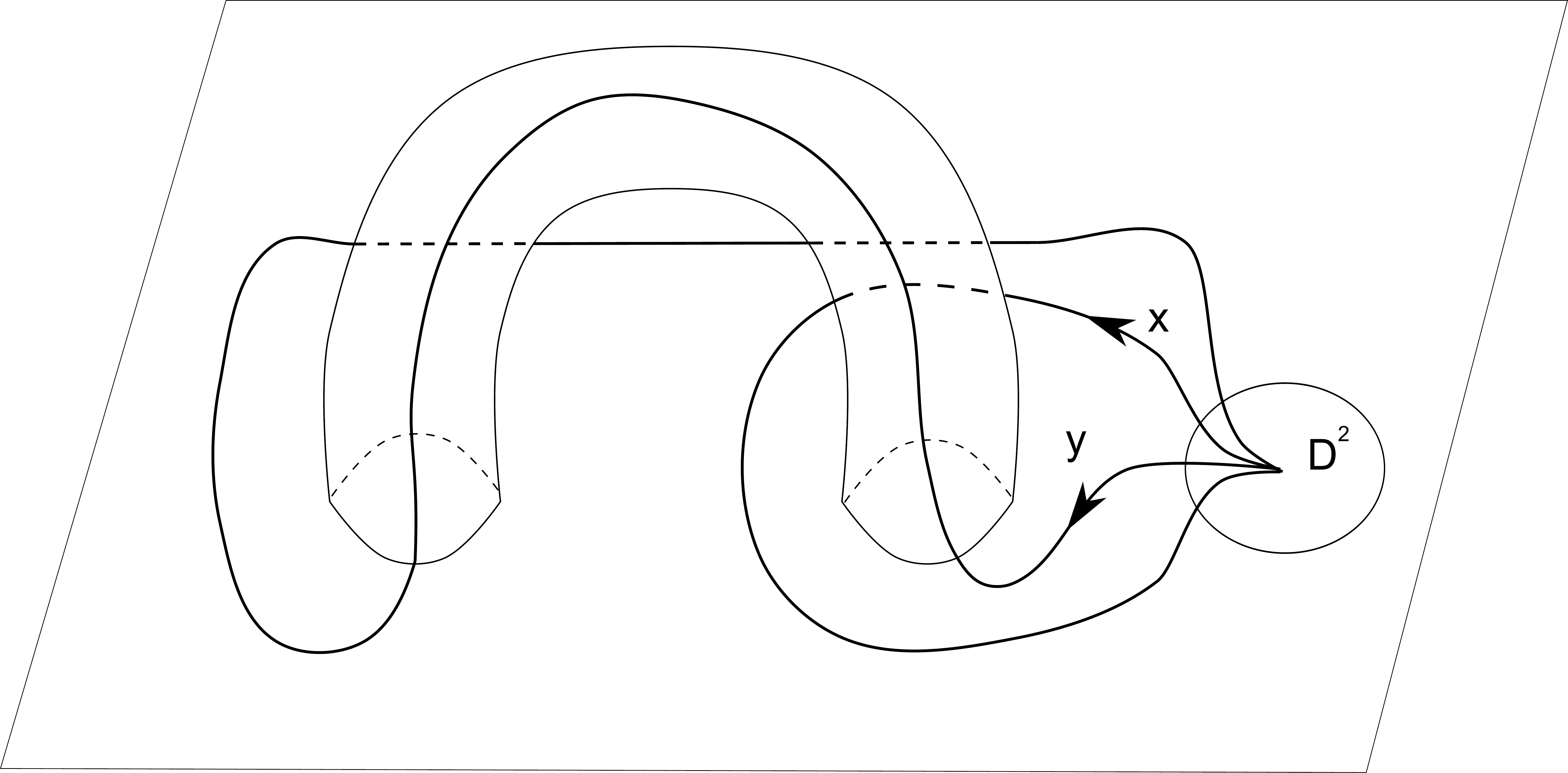}

\caption{The torus $\mathbb{T}$ with generators of $\pi_1$}
\label{torus}
\end{figure}

\label{def_category_tangles}

\begin{definition}
The category $q\tilde{T}_1$ of framed tangles in the thickened torus is defined as follows: The objects of $q\tilde{T}_1$ are non-associative words in $\{+,-\}$. For two such words $\omega_s,\omega_t$ the morphisms set $q\tilde{T}_1(\omega_s,\omega_t)$ is the set of ambient isotopy classes of oriented framed tangles in $\mathbb{T}\times I$ of type $(\omega_s,\omega_t)$. As in the definition of $q\tilde{T}$, the tangles are piece-wise smooth, transverse to the planes $\mathbb{T}\times\{t\}$ at all but a finite number of points, and vertical near the boundary points.
\end{definition}

There is an obvious functor $q\tilde{T}\rightarrow q\tilde{T}_1$, induced by the embedding $D^2\subset \mathbb{T}$. We want to describe an elliptic structure relative to this functor. In this context, the natural automorphisms $X_{\omega_1,\omega_2}$ and $Y_{\omega_1,\omega_2}$ act by composition with some invertible tangles in $q\tilde{T}_1((\omega_1)(\omega_2),(\omega_1)(\omega_2))$. We will now describe those tangles.

The tangent space at any point $p=(u,t)\in\mathbb{T}\times I$ can be decomposed as $T_{\mathbb{T}_u}\times\mathbb{R}$. If the point $u$ is in $D^2\subset\mathbb{T}$, a tangent vector at $p$ can be parametrized by $(x,y,t)$.

Let $\gamma:[0,1]\rightarrow\mathbb{T}$ be a smooth simple closed path, with a nowhere vanishing derivative. Assume that $\gamma(0)=\gamma(1)$ is the first (left) point of $b_2\subset \mathbb{T}$, and that $\gamma'(0)=c_1(-1,0)$ and $\gamma'(1)=c_2(1,0)$ for some constants $c_1,c_2$ (we will assume that the loops representing $x$ and $y$ from figure \ref{torus} have this property). We define the tangle $\tilde{\gamma}$ of type $(++,++)$ as follows:

The right strand is a constant strand at the second (right) point of $b_2$, with a framing parametrized constantly by $(0,-1,0)$.

The left strand is the tangle defined by $(\gamma(t),t)\in\mathbb{T}\times I$, and framed by the unique framing which has the following two properties: (A) it is parametrized by $(0,1,0)$ at $(\gamma(0),0)$ (thus it is actually not in $q\tilde{T}_1(++,++)$ as defined above), and (B) its parametrization has $t=0$ at all points (this parametrization is unique due to the nowhere vanishing derivative assumption).

By following closely the loops of figure \ref{torus} it can be seen that at $\tilde{\gamma}(1)$ this framing is parametrized by $(0,-1,0)$, for both $\gamma=x$ and $\gamma=y$.

Let $pt$ (=positive twist) and $nt$ (=negative twist) be the following tangles of type $(++,++)$: The right strand is constant, as in the definition of $\tilde{\gamma}$. The left strand is also a constant strand, but its framing makes a half twist from $(0,-1,0)$ to $(0,1,0)$. In $pt$ this twist is in the positive direction, and in $nt$ the twist is in the negative direction.

For any $2$ words $\omega_1,\omega_2$ and any tangle $u\in q\tilde{T}_1(++,++)$, define the cabling $\Delta^{++}_{\omega_1,\omega_2}(u)\in q\tilde{T}_1((\omega_1)(\omega_2),(\omega_1)(\omega_2))$ to be the tangle obtained from $u$ by duplicating the left strand $|\omega_1|$ times along its framing and giving the strands orientations according to the symbols of $\omega_1$, and similarly for the right strand with $\omega_2$.

\begin{proposition}
\label{elliptic_proposition}
There is an elliptic structure relative to $(q\tilde{T}\rightarrow q\tilde{T}_1)$ with:
$$X_{\omega_1,\omega_2}=\Delta^{++}_{\omega_1,\omega_2}(\tilde{x}\cdot pt)$$
$$Y_{\omega_1,\omega_2}=\Delta^{++}_{\omega_1,\omega_2}(\tilde{y}\cdot nt)$$
where $x,y$ are the representatives of the elements of $\pi_1(\mathbb{T})$ depicted in figure \ref{torus}.
\end{proposition}

The compositions $\tilde{x}\cdot pt$ and $\tilde{y}\cdot nt$ which appear in this proposition are defined by simply putting the first tangle on top of the second. Note that the framings of $\tilde{x}$ and $\tilde{y}$ at the bottom correspond to the framings of $pt$ and $nt$ at the top, and after the compositions we get elements in $q\tilde{T}_1(++,++)$.

This proposition is proved in \cite{Humbert} as a special case of the general concept of genus $g$ structures. Furthermore, it is proved there that $(q\tilde{T}\rightarrow q\tilde{T}_1)$ is universal, in the sense that there is a unique pair of functors from it to any other pair with an elliptic structure, preserving most of its properties.

It is clear that the main ingredients of this elliptic structure are the tangles $\tilde{x}\cdot pt$ and $\tilde{y}\cdot nt$. In \cite{Humbert}, these are the tangles which are represented by ``beaks" in beak diagrams. In section \ref{section_LMO} we will give a different description of those tangles, and use this description to give another, pictorial, proof  of proposition \ref{elliptic_proposition}.

\newpage
\section{Categories of Jacobi Diagrams}
\label{section_jacobi}

In this section we review the definition of several categories of Jacobi diagrams, and the relations between them. Most of those categories are obvious extensions of spaces which appear, in one way or another, in \cite{Cheptea}, \cite{Habiro12} and \cite{Humbert}. In the following sections we will see how those categories are used to define invariants of tangles.

\subsection{Category of Patterns}
\begin{definition}
A \textbf{pattern} $P$ is a compact (but not necassarily closed) oriented $1$-manifold whose boundary $\partial P$ is divided into $2$ ordered sets $\partial^s P$ and $\partial^t P$. To each pattern $P$ we can associate $2$ words $\omega_s(P)$ and $\omega_t(P)$ in the symbols $\{+,-\}$. The words $\omega_s(P)$ and $\omega_t(P)$ encode the orientations of $P$ around the ordered sets $\partial^s P$ and $\partial^t P$: A $+$ ($-$) in the $i$-th place of $\omega_s(P)$ means that the orientation of $P$ near the $i$-th point of $\partial^s P$ is going away from (towards) this point. A $+$ ($-$) in the $i$-th place of $\omega_t(P)$ means that the orientation of $P$ near the $i$-th point of $\partial^t P$ is going towards (away from) this point.

The category of patterns $\mathbb{P}$ is the category whose objects are finite words in the symbols $\{+,-\}$, and for any $2$ such words $\omega_s,\omega_t$ the morphisms set $\mathbb{P}(\omega_s,\omega_t)$ is the set of patterns $P$ such that $\omega_s(P)=\omega_s$ and $\omega_t(P)=\omega_t$. The composition $P_2\cdot P_1$ of $2$ patterns is defined by attaching $\partial^t P_1$ to $\partial^s P_2$. Graphically, we usually draw $\partial^s P$ at the bottom and $\partial^t P$ at the top. The composition is then obtained by putting $P_2$ above $P_1$.
\end{definition}

\paragraph*{Note:} There are obvious functors $q\tilde{T}\rightarrow \mathbb{P}$ and $q\tilde{T}_1\rightarrow\mathbb{P}$ which map an object $\omega$ to itself (forgetting the non-associative structure), and a tangle $u$ to its skeleton. The split $\partial P=\partial^s P\cup \partial^t P$ is determined by whether the boundary point is in $D^2\times\{0\}$ or in $D^2\times\{1\}$.

\subsection{Jacobi Diagrams and the Category $\mathcal{A}$}
Let $P$ be a pattern, and $S$ a set. A \textbf{Jacobi diagram} $D$ over $P$ and $S$ is a uni-trivalent graph whose univalent vertices are either connected to a point in $P$ or labeled by an element of $S$. The trivalent vertices are cyclically oriented. When we draw a Jacobi diagram in a $2$-dimensional plane, we assume the orientations of the trivalent vertices are always counter-clockwise. An example for a Jacobi diagram is given in figure \ref{Jacobi_example}. The pattern $P$ is given by the solid lines, and the uni-trivalent diagram is given by the dashed lines. 

\begin{figure}[ht]

\centering
\includegraphics[width=5cm]{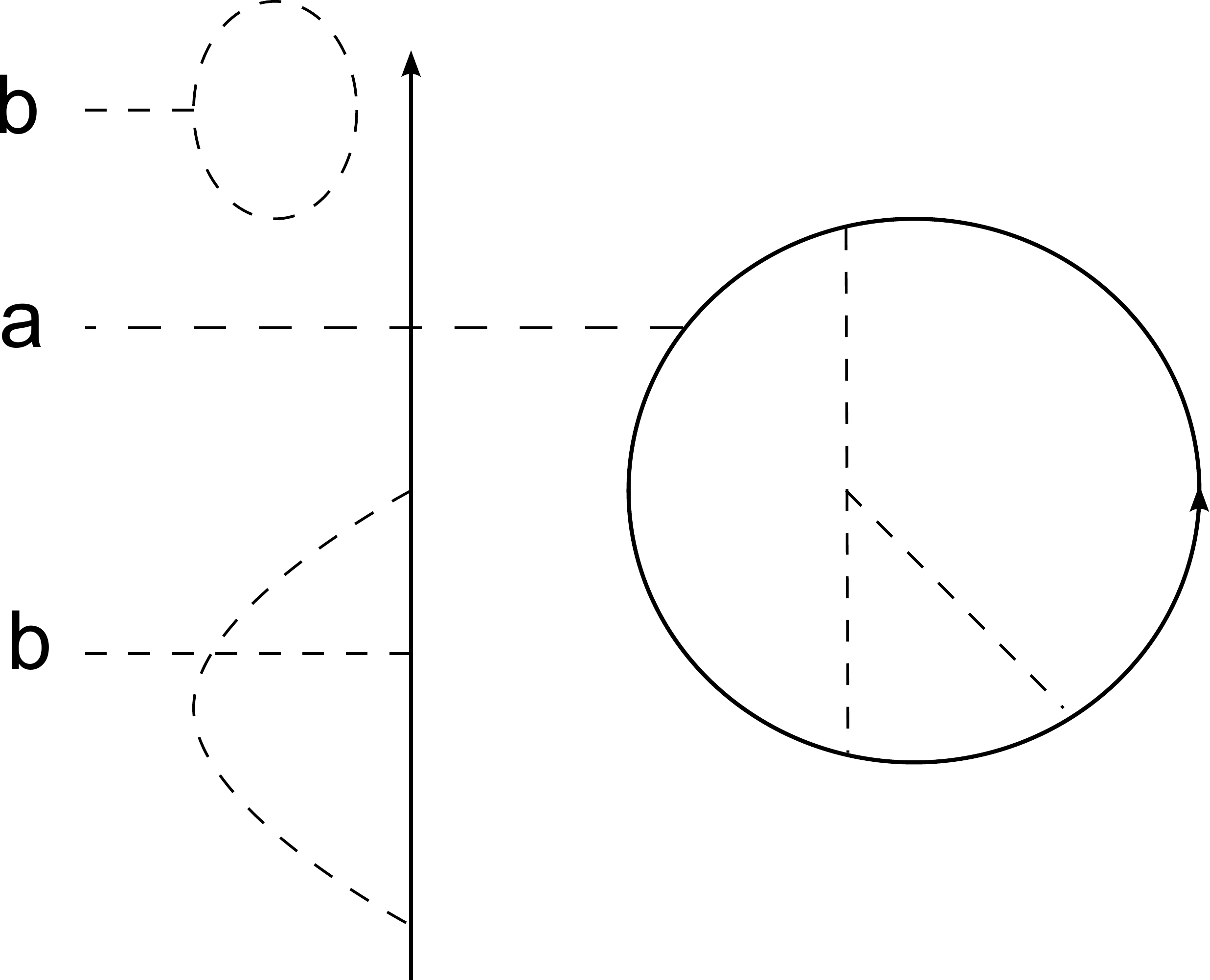}

\caption{An example for a Jacobi diagram}
\label{Jacobi_example}
\end{figure}

The \textbf{degree} of a Jacobi diagram $D$ is defined to be the number of trivalent vertices + the number of univalent vertices on $P$.

The set of all Jacobi diagrams over $P$ and $S$ is denoted $\mathbb{D}(P,S)$. 
\paragraph*{Notation:} Let $A$ be a set with a degree map $A\rightarrow \mathbb{N}$, and $\mathbb{F}$ a field ($\mathbb{F}$ will usually be a field of characteristic $0$, such as $\mathbb{R}$ or $\mathbb{C}$). Then $S_\mathbb{F}(A)$ denotes the degree completion of the vector space spanned by $A$ over the field $\mathbb{F}$.

Let $\mathcal{A}(P,S)$ be the quotient of $S_\mathbb{F}(\mathbb{D}(P,S))$ by the relations $STU$, $IHX$ and $AS$, which are defined by:
$$STU:\mathfig{2}{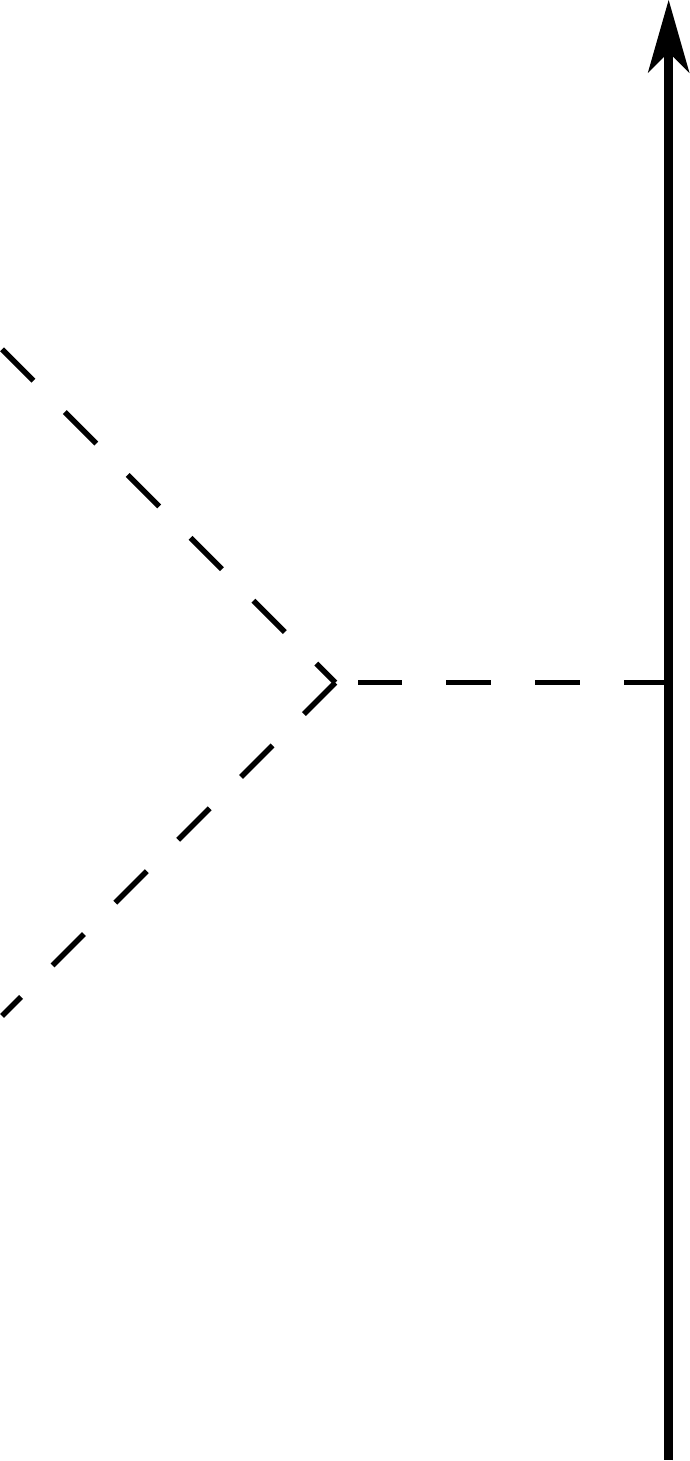} = \mathfig{2}{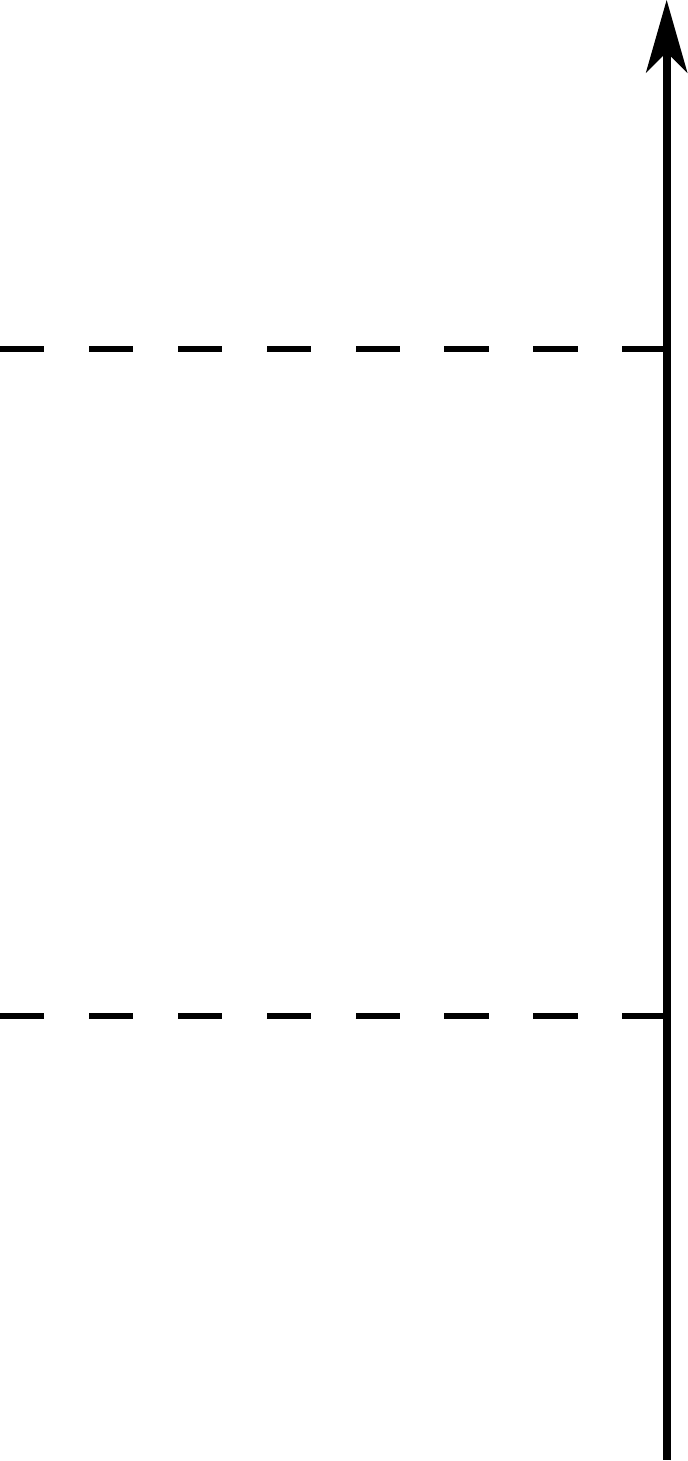} - \mathfig{2}{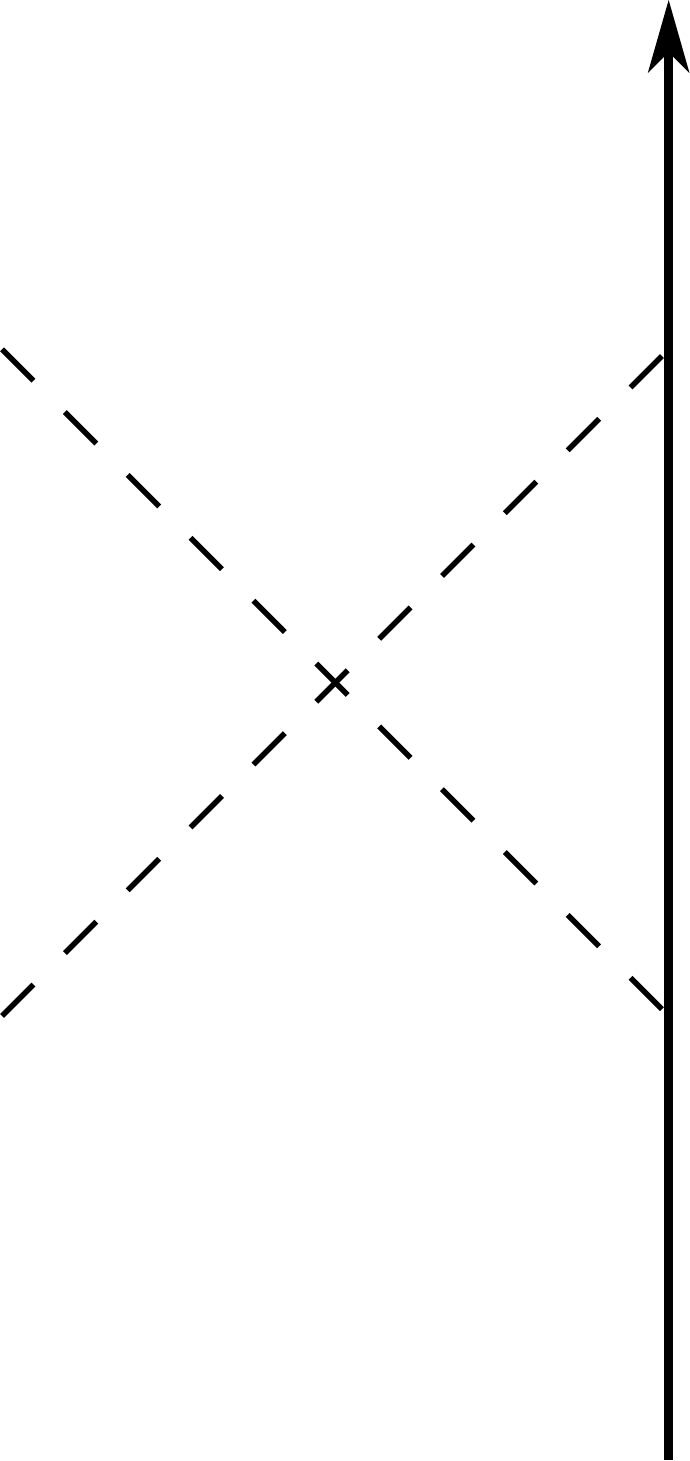}\\[0.5cm]$$

$$IHX:\mathfig{2}{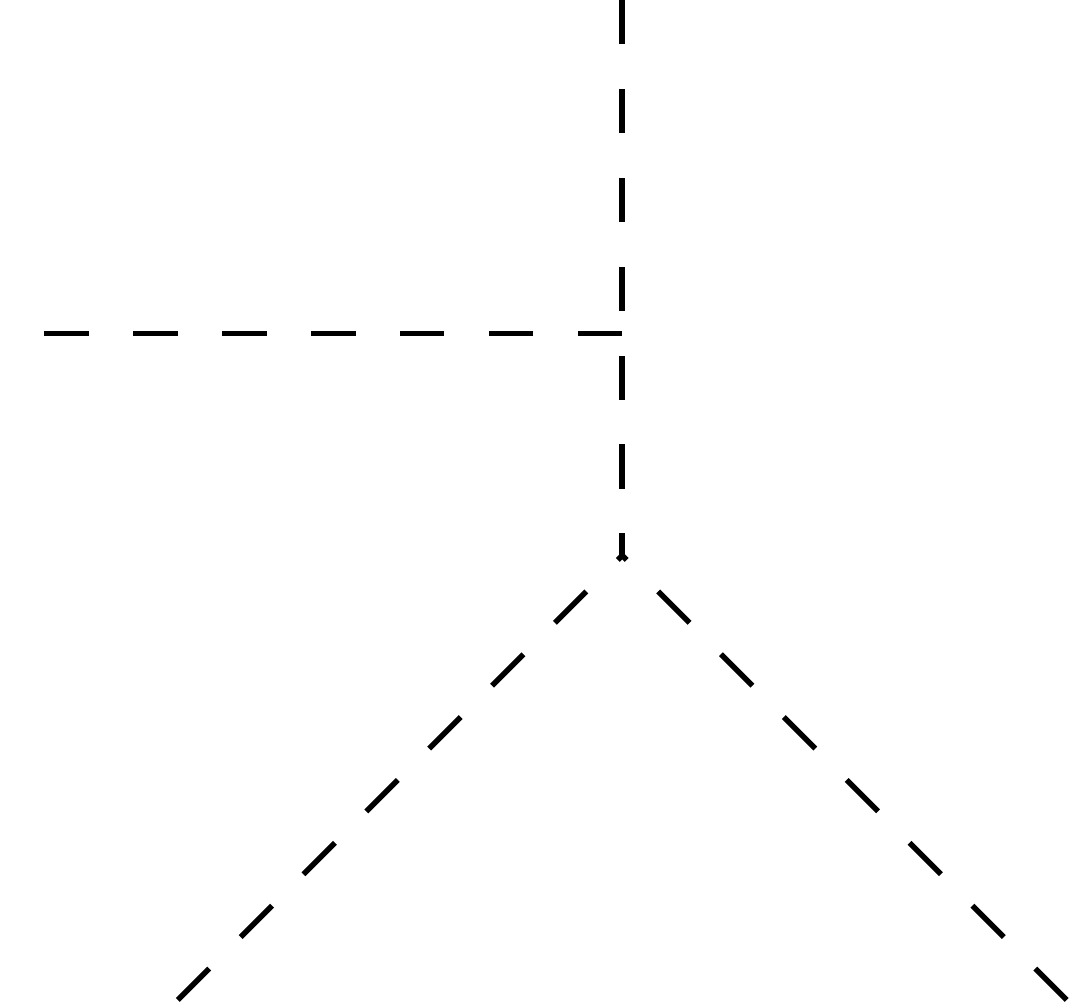} = \mathfig{2}{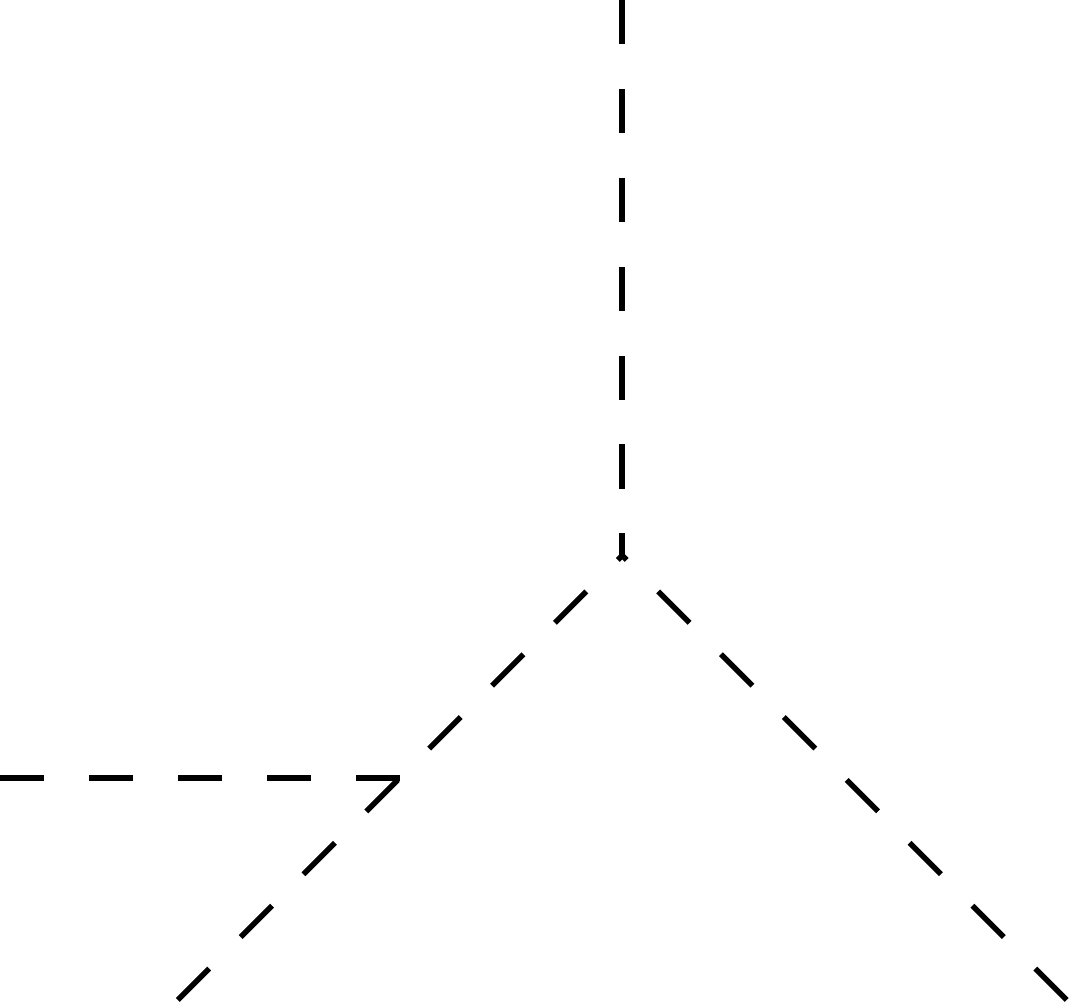} - \mathfig{2}{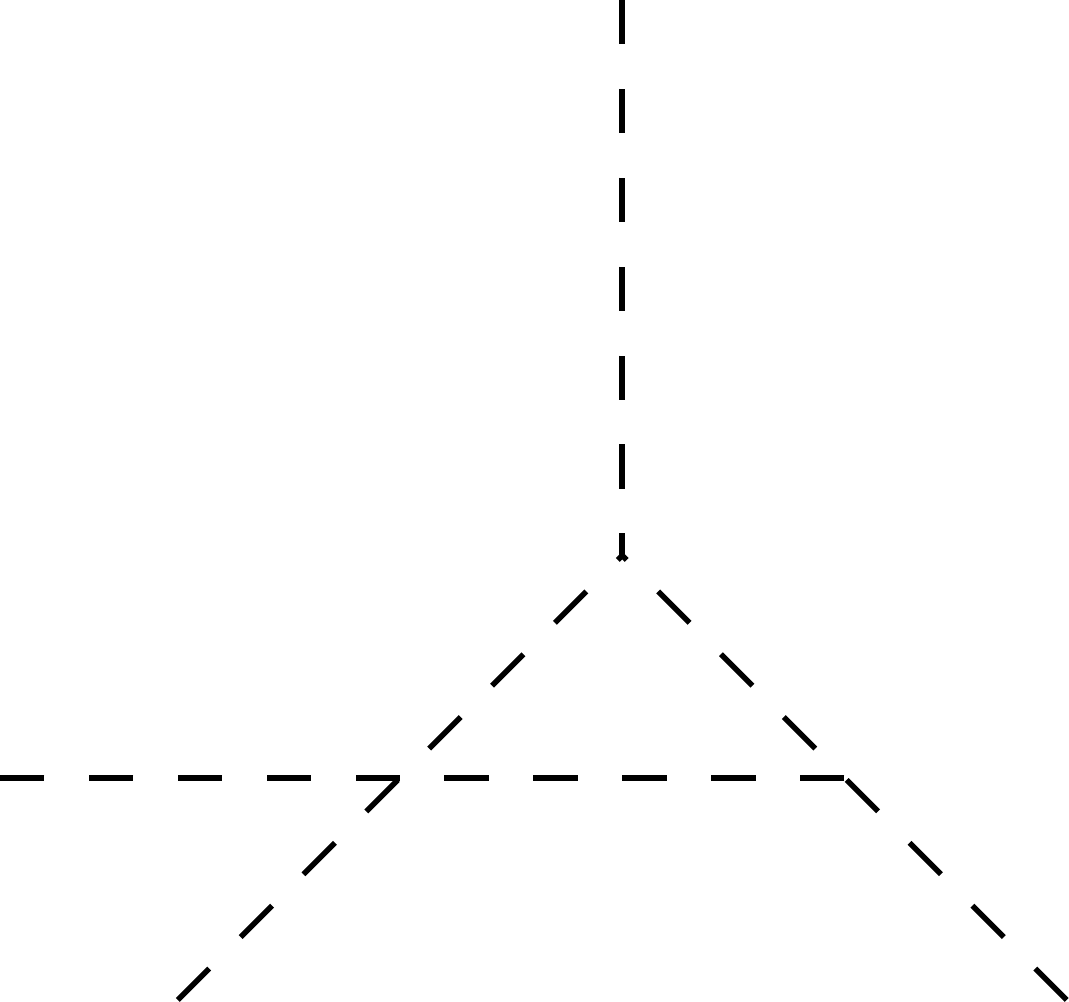}\\[0.5cm]$$

$$AS:\mathfig{2}{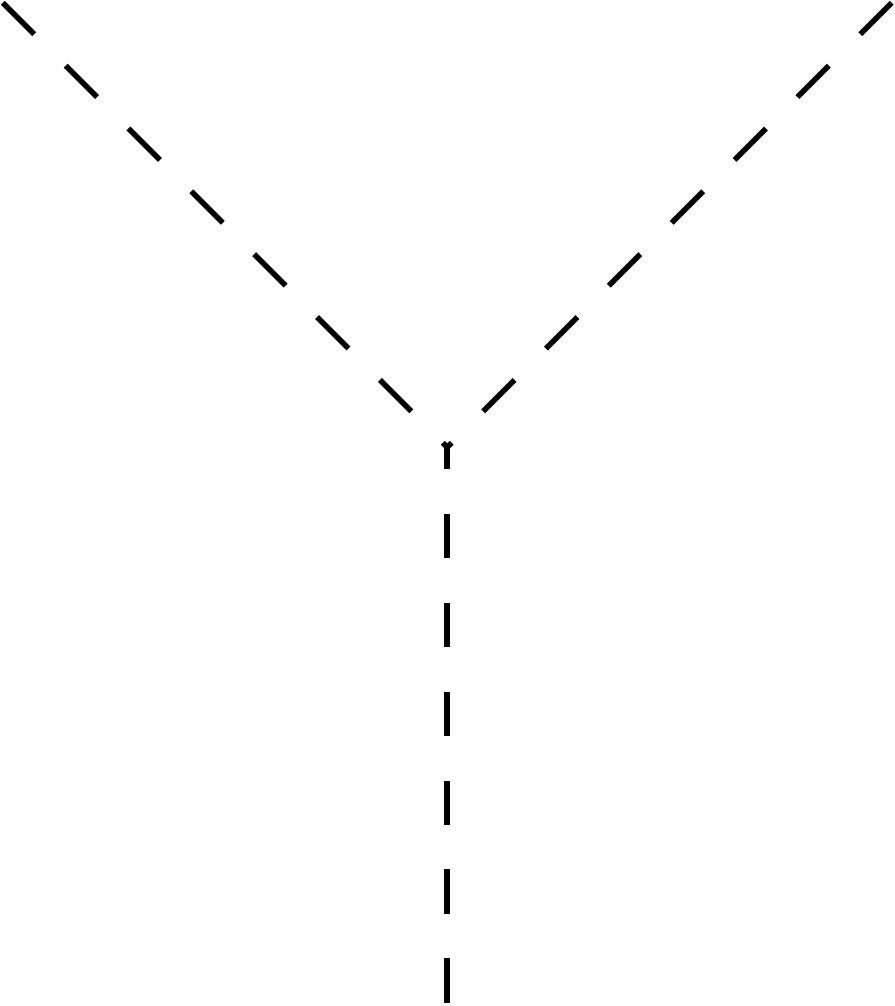} = \quad -\mathfig{2}{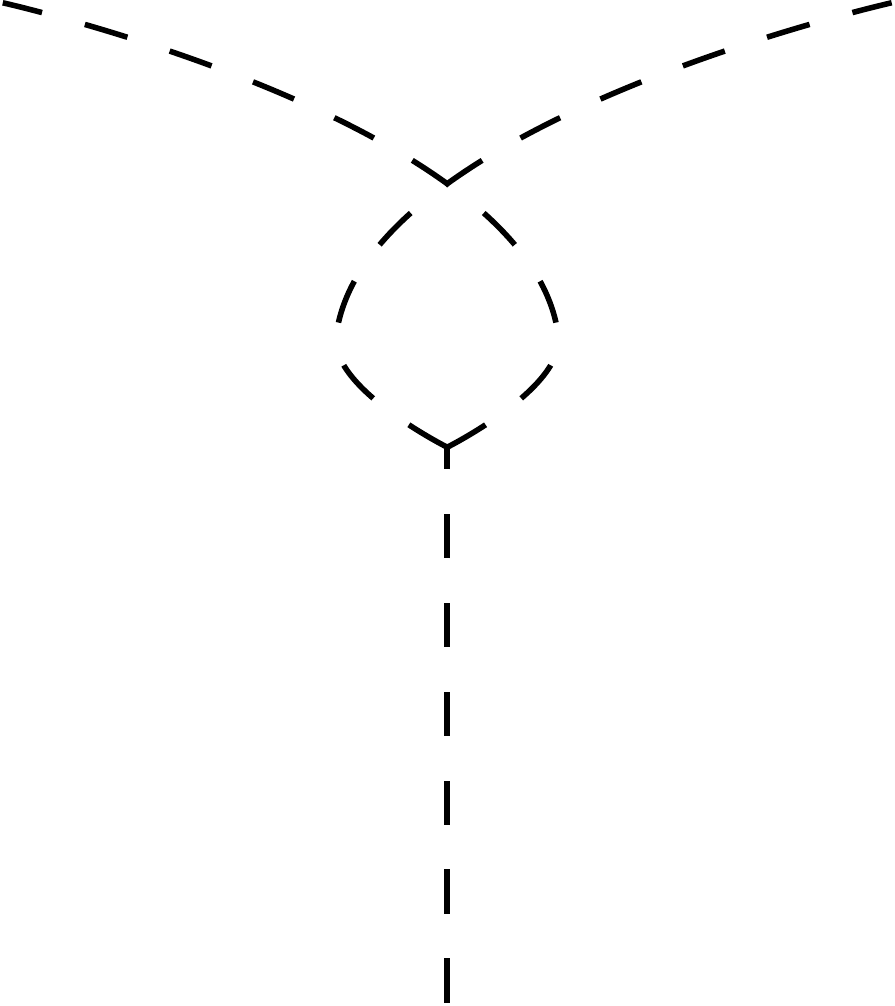}\\[0.5cm]$$

At this point we are specifically interested in the spaces $\mathcal{A}(P,\emptyset)$, i.e. spaces of Jacobi diagrams with no labeled vertices. We will denote those spaces by $\mathcal{A}^\partial(P)$. In these spaces the degree of all diagrams is even. It is common to define the degree in those spaces as half the number of vertices, but we will continue to use the above definition of degree, in order to be compatible with other spaces of Jacobi diagrams.

Let $\mathcal{A}^\partial$ be the category defined as follows: The objects of $\mathcal{A}^\partial$ are words in the symbols $\{+,-\}$. For $2$ such words $\omega_s,\omega_t$, the morphisms set $\mathcal{A}^\partial(\omega_s,\omega_t)$ is defined by: $\mathcal{A}^\partial(\omega_s,\omega_t):=\bigcup_{P\in\mathbb{P}(\omega_s,\omega_t)}\mathcal{A}^\partial(P)$. For $a\in \mathcal{A}^\partial(\omega_s,\omega_t)$ and $b\in\mathcal{A}^\partial(\omega_t,\omega_u)$, the composition $b\cdot a$ is obtained by putting $b$ above $a$ and composing the underlying patterns.

Recall the box notation, which is defined in figure \ref{box}. In this figure the vertical lines going through the box can be either solid or dashed. The sign $\varepsilon_i$ is $-1$ if the $i$-th line is a solid line with orientation opposite to the orientation of the box (i.e. the direction which the arrow in the box points to), and $+1$ otherwise.

\begin{figure}[ht]

$$\mathfig{4}{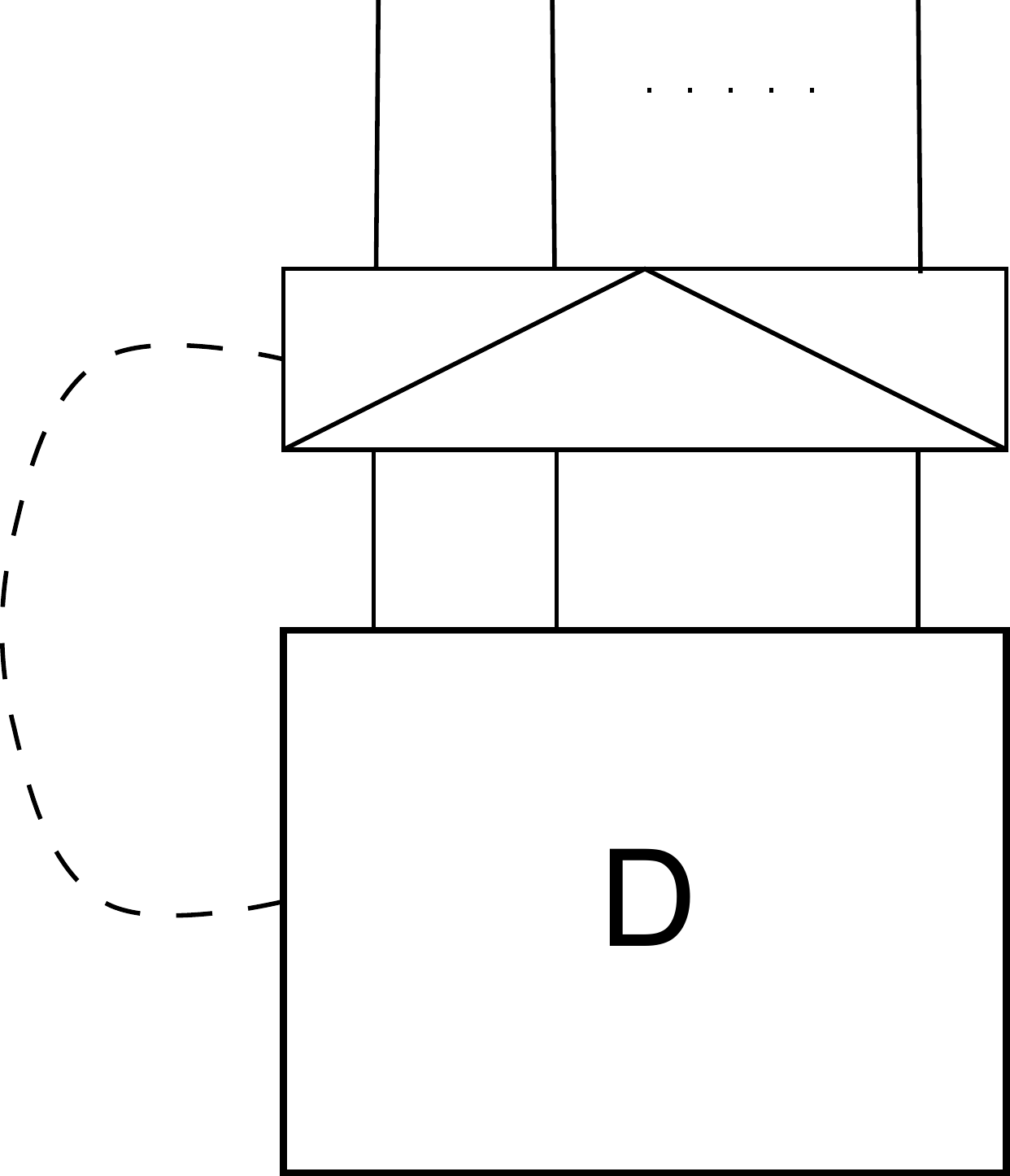} = \sum_i \varepsilon_i \mathfig{4}{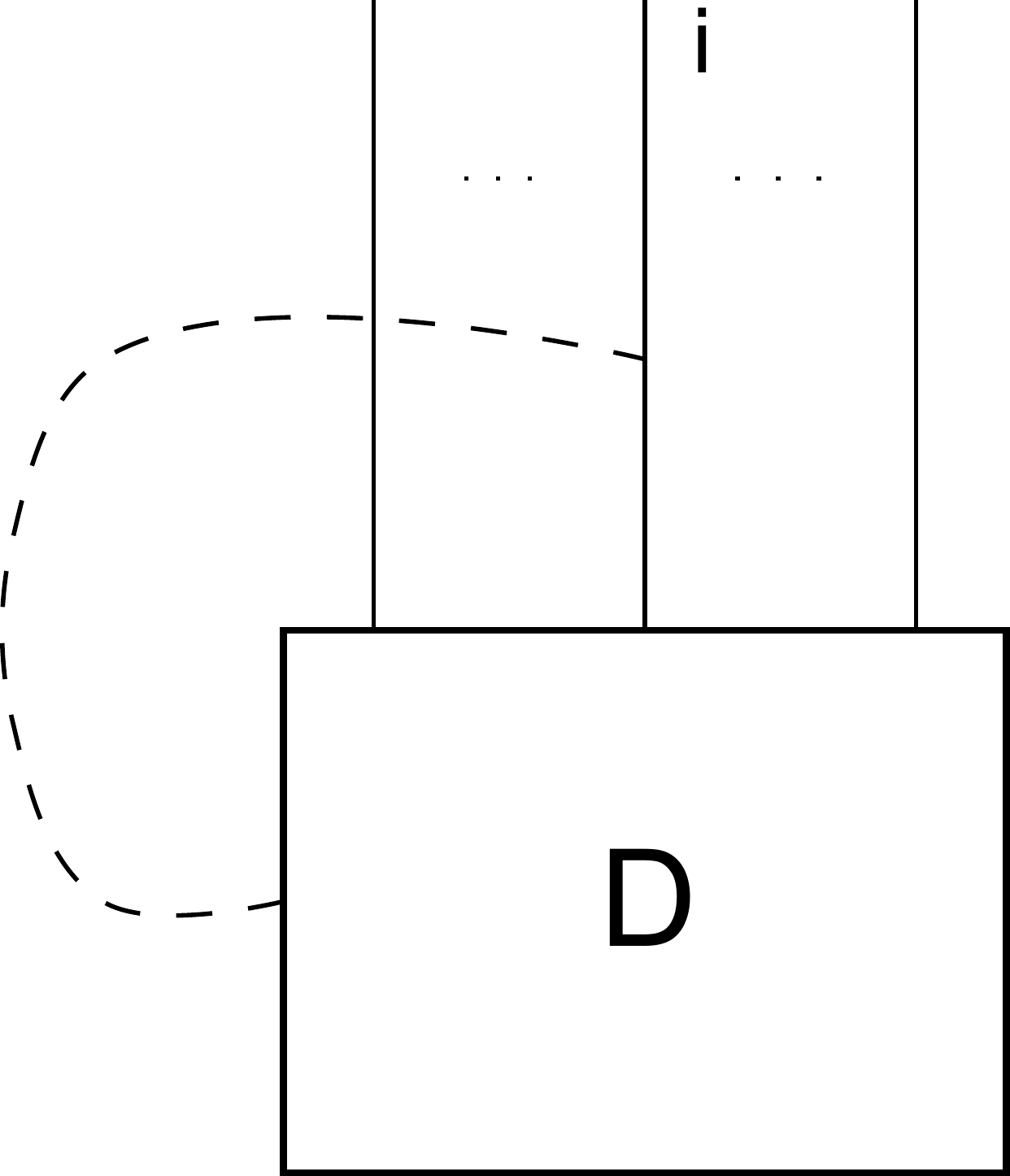}$$
\caption{The box notation}

\label{box}

\end{figure}

In each space $\mathcal{A}^\partial(P)$, we define $\textbf{I}(P)$ to be the subspace generated by all sums of the type given in figure \ref{I}. In this figure we assume that the lines which come out of the upper side of the box are exactly all the ends of the solid lines which lead to the points of $\partial^t(P)$. We call each such sum an $I$ relation.

\label{defA}
\begin{figure}

\centering
\includegraphics[height=4cm]{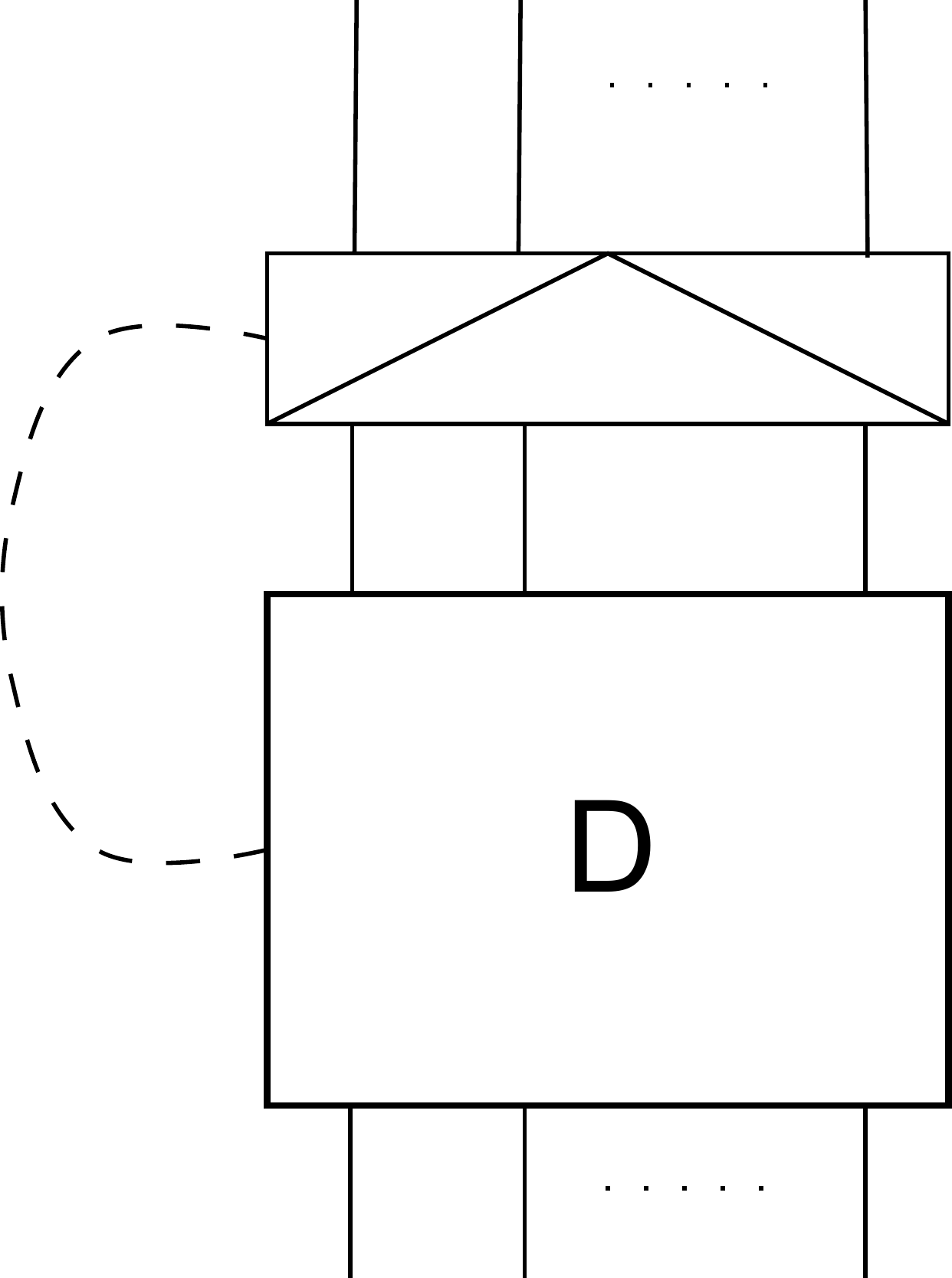}
\caption{An $I$ relation}
\label{I}
\end{figure}

It can be verified that the union of all $\textbf{I}(P)$ is a two-sided ideal of $\mathcal{A}^\partial$ (see \cite{Humbert}, Lemma 1.4.4). Therefore we can define the quotient category $\mathcal{A}^\partial/\textbf{I}$. We denote this category by $\mathcal{A}$.

\paragraph*{Remark:}
As we will see in the following sections, the categories of Jacobi diagrams decorated by $\partial$ are usually the targets of invariants of tangles in cobordisms whose boundary surfaces have one boundary component. The corresponding quotient categories with no decoration are used when the boundary surfaces have no boundary. Specifically, $\mathcal{A}^\partial$ is used to defined invariants of tangles in $D^2\times I$, and $\mathcal{A}$ is used to define invariants of tangles in $S^1\times I$. In the following sections we will be more interested in invariants of tangles in $\mathbb{T}\times I$. The categories $\mathcal{A}^\partial$ and $\mathcal{A}$ will only be used for technical reasons in the way to define those invariants.

We end this section by recalling a known notation (see \cite{Cheptea}, notation 3.13):
\paragraph*{Notation:} Let $\omega$ be a word of length $n$ in the symbols $\{+,-\}$, and $\omega_1,...,\omega_n$ other words in those symbols. The map $\Delta^\omega_{\omega_1,...,\omega_n}:\mathcal{A}^\partial(\omega,\omega)\rightarrow\mathcal{A}^\partial(\omega_1\cdot...\cdot\omega_n,\omega_1\cdot...\cdot\omega_n)$ is the map obtained by applying, for each $1\le i\le n$, the doubling map $\Delta:\mathcal{A}^\partial(\uparrow,\uparrow)$ iterated $|\omega_i|-1$ times on the $i$-th strand, and by applying the orientation-reversal map $S$ to each new strand whose corresponding symbol in $\omega_i$ does not agree with the $i$-th symbol in $\omega$. This map also induces a map on the quotient categories:  $\Delta^\omega_{\omega_1,...,\omega_n}:\mathcal{A}(\omega,\omega)\rightarrow\mathcal{A}(\omega_1\cdot...\cdot\omega_n,\omega_1\cdot...\cdot\omega_n)$

\subsection{Unordered Elliptic Jacobi Diagrams}

Let $\mathbb{D}_1(P)$ denote the set $\mathbb{D}(P,H_1(\mathbb{T}))$, i.e. the set of Jacobi diagrams over the pattern $P$ with labels coming from the first homology of the torus. Let $\mathcal{A}^\partial_1(P)$ be the quotient of $S_\mathbb{F}(\mathbb{D}_1(P))$ by the relations: $STU$, $IHX$, $AS$ and multilinearity. The multilinearity relation is defined as follows:

$$ (au+bv)\mathfigns{1}{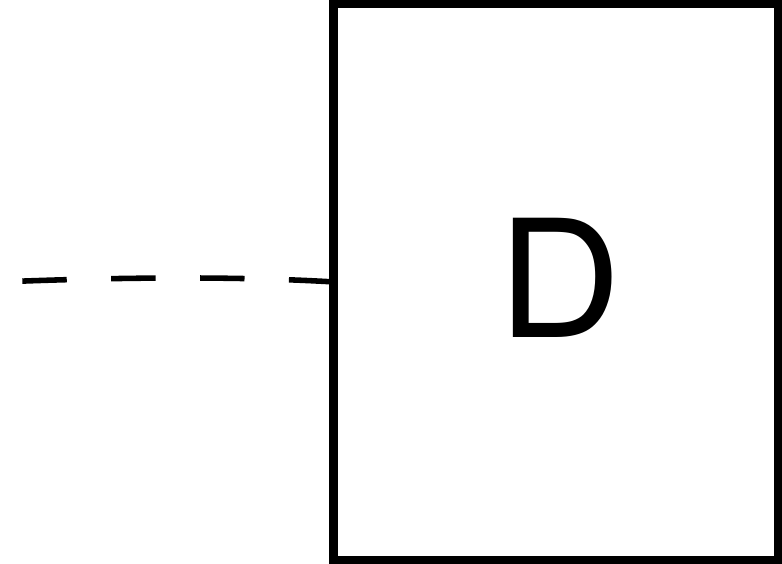}=a\left(u\mathfigns{1}{section2_multilinearity.pdf}\,\right)+b\left(v\mathfigns{1}{section2_multilinearity.pdf}\,\right) \quad \forall a,b\in\mathbb{F}\quad u,v\in H_1(\mathbb{T}) $$

The loops $x,y$ from figure \ref{torus} induce generators of $H_1(\mathbb{T})$, which we also denote by $x,y$. Given this basis (or any other basis), it is easy to see that $\mathcal{A}_1^\partial(P)$ is isomorphic to $\mathcal{A}^\partial(P,\{x,y\})$, and the multilinearity relation is no longer needed. In general it is better to work with the definition which allows all the elements of $H_1(\mathbb{T})$ as labels, because this definition does not depend on a choice of a basis, and also because $\mathcal{A}_1^\partial(P)$ defined this way has a natural action of the symplectic group (see \cite{Habiro09}). However, for our needs it would be more convenient to consider only diagrams with $x$ and $y$ labels, and forget about multilinearity.

\textbf{A strut} in an elliptic Jacobi diagram $D\in\mathbb{D}(P, \{x,y\})$ is a component in the diagram of the form $\begin{array}{c} u\\ \includegraphics[height=0.5cm]{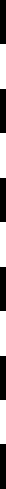}\\v\end{array}$ with $u,v\in\{x,y\}$. $D$ is called \textbf{top-substantial} if it has no struts labeled by $y$ on both vertices. Let $^{ts}\mathcal{A}_1^\partial(P)\subset \mathcal{A}_1^\partial(P)$ be the subspace generated by all the top-substantial diagrams. This restriction will allow us to define a composition of elliptic Jacobi diagrams (see also \cite{Cheptea}, section 3.1).

For $2$ elements $D_1\in^{ts}\mathcal{A}_1^\partial(P_1)$ and $D_2\in^{ts}\mathcal{A}_1^\partial(P_2)$ such that $P_1$ and $P_2$ are composable, we define the composition $D_2\cdot D_1$ to be the sum of all diagrams obtained by putting $D_2$ on top of $D_1$ and attaching \textbf{all} the $y$-labeled vertices of $D_1$ to \textbf{all} the $x$-labeled vertices of $D_2$ (thus, if the number of $y$-labeled vertices of $D_1$ is not equal to the number of $x$-labeled vertices of $D_2$, this sum is empty). This composition is extended linearly to a map $\mathcal{A}_1^\partial(P_2)\times\mathcal{A}_1^\partial(P_1)\rightarrow\mathcal{A}_1^\partial(P_2\cdot P_1)$.

The category $^{ts}\mathcal{A}_1^\partial$ is now defined as the category whose objects are words in the symbols $\{+,-\}$, and for any $2$ such words $\omega_s$ and $\omega_t$, $^{ts}\mathcal{A}_1^\partial(\omega_s,\omega_t):=\bigcup_{P\in\mathbb{P}(\omega_s,\omega_t)}\,^{ts}\mathcal{A}_1^\partial(P)$. The composition is defined as above.

The category $^{ts}\mathcal{A}_1^\partial$ is a monoidal category. The tensor product of two objects $\omega_1$ and $\omega_2$ is the concatenation $\omega_1\omega_2$, and the tensor product of $2$ diagrams $D_1$ and $D_2$ is obtained by putting $D_1$ alongside of $D_2$. In $^{ts}\mathcal{A}_1^\partial(\emptyset)$ this tensor product induces a multiplication. With respect to this multiplication we can define an exponent. In particular we have the identity elements:
$$id_\omega = 
\exp\left(\begin{array}{c} y\\ \includegraphics[height=0.7cm]{section2_strut.pdf}\\x\end{array}\right)
\otimes
\begin{array}{c}\omega \\  \mathfigns{1}{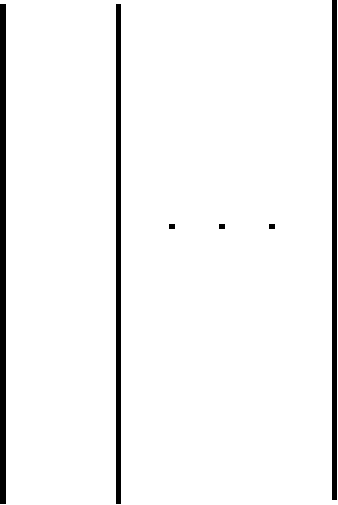}\\ \omega \end{array}$$.

Let $\textbf{I}_1(P)$ be the subspace of $^{ts}\mathcal{A}_1^\partial(P)$ generated by sums of the type shown in figure \ref{I_1}. In this figure we assume that the lines which come out of the upper side of the box are exactly all the ends of the solid lines which lead to the points of $\partial^t(P)$, and all the ends of dashed lines ending with the label $y$. We call each such sum an $I_1$ relation.

\begin{figure}[ht]

\centering
\includegraphics[height=4cm]{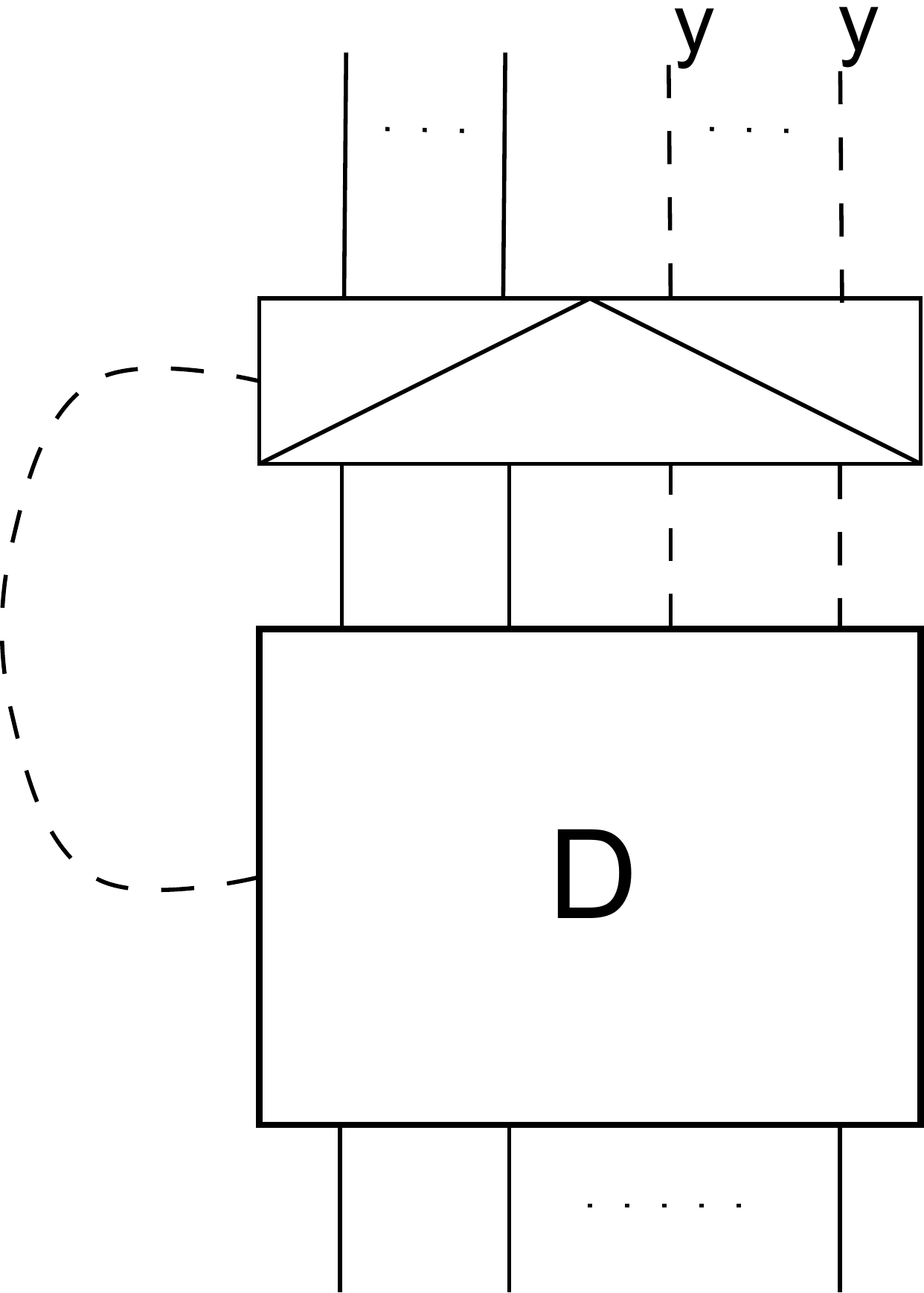}

\caption{An $I_1$ relation}
\label{I_1}
\end{figure}

It may be verified that the union of all $\textbf{I}_1(P)$ is a two-sided ideal of $^{ts}\mathcal{A}_1^\partial$ (by a similar argument to the one we mentioned in the previous section), which we denote by $\textbf{I}_1$. The quotient category $^{ts}\mathcal{A}_1^\partial/\textbf{I}_1$ is denoted $\mathcal{A}_1$. Denote the projection by $\pi:^{ts}\mathcal{A}_1^\partial\rightarrow \mathcal{A}_1$.

\subsection{Unordered Elliptic Jacobi Diagrams with no Struts} 
\label{subsection_no_struts}

Let $\mathbb{D}_1^y(P)\subset \mathbb{D}_1(P)$ be the subset of all diagrams $D$ such that each component of $D$ has at least one trivalent vertex or one vertex on $P$. In other words, $D$ has no struts. Since all the relations of $^{ts}\mathcal{A}_1^\partial(P)$ preserve this subspace, we get a subspace $\mathcal{A}^{\partial y}_1(P)$ of $^{ts}\mathcal{A}_1^\partial(P)$.

We use the spaces $\mathcal{A}^{\partial y}_1(P)$ to define a category $\mathcal{A}^{\partial y}_1$ similarly to the way we defined $\mathcal{A}^{\partial }_1$ but with a different composition. For $2$ elements $D_1\in\mathcal{A}_1^{\partial y}(P_1)$ and $D_2\in\mathcal{A}^{\partial y}_1(P_2)$ such that $P_1$ and $P_2$ are composable, we define the composition $D_2\cdot D_1$ to be the sum of all diagrams obtained by putting $D_2$ on top of $D_1$ and attaching \textbf{some} of the $y$-labeled vertices of $D_1$ to \textbf{some} of the $x$-labeled vertices of $D_2$. This composition is extended linearly to a map $\mathcal{A}_1^{\partial y}(P_2)\times\mathcal{A}_1^{\partial y}(P_1)\rightarrow\mathcal{A}_1^{\partial y}(P_2\cdot P_1)$.

Define maps $j^\partial_P:\mathcal{A}^{\partial y}_1(P)\rightarrow ^{ts}\mathcal{A}_1^\partial(P)$ by $u\mapsto \exp\left(\begin{array}{c} y\\ \includegraphics[height=0.5cm]{section2_strut.pdf}\\x\end{array}\right)
\otimes u$. It is easy to see that $j_P^\partial$ is injective. Also, given $u_1\in\mathcal{A}^{\partial y}(P_1)$ and $u_2\in\mathcal{A}^{\partial y}(P_2)$ with $P_1$ and $P_2$ composable, it may be verified that $j^\partial_{P_2\cdot P_1}(u_2\cdot u_1)=j^\partial_{P_2}(u_2)\cdot j^\partial_{P_1}(u_1)$. So $j^\partial_P$ induce an injective functor $j^\partial:\mathcal{A}^{\partial y}_1\rightarrow ^{ts}\mathcal{A}_1^\partial$.

Let $\textbf{I}^y_1(P)$ be the subspace of $\mathcal{A}_1^{\partial y}(P)$ generated by sums of the type shown in figure \ref{I^y_1}. Again, we assume that the lines which come out of the upper side of the box are exactly all the ends of the solid lines which lead to the points of $\partial^t(P)$, and all the ends of dashed lines ending with the label $y$. We call each such sum an $I_1^y$ relation. 

\begin{figure}[ht]

\centering
$$\mathfig{4}{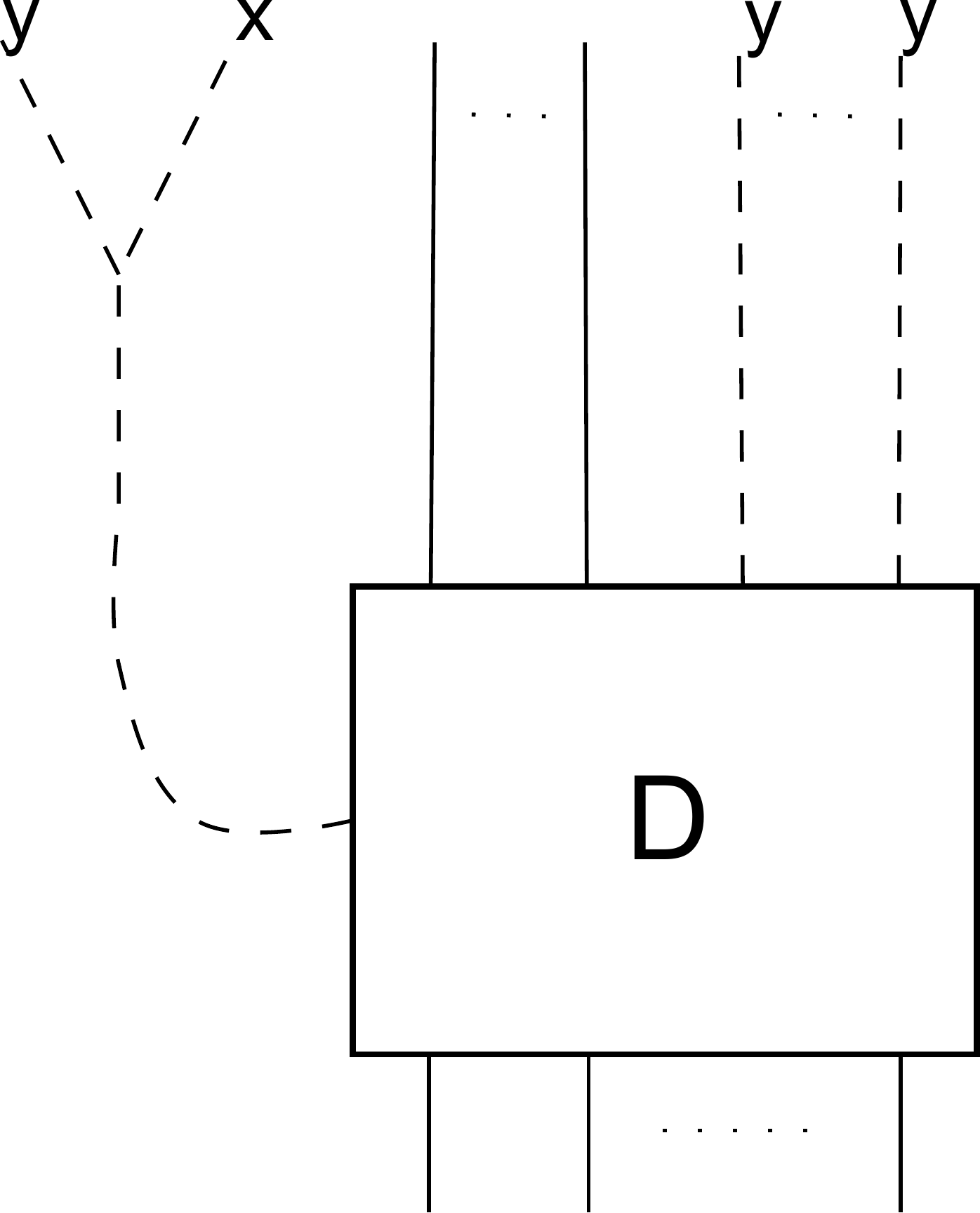} + \mathfig{4}{section2_I_1.pdf}$$

\caption{An $I^y_1$ relation}
\label{I^y_1}
\end{figure}
 
\begin{proposition}
The subspaces $\textbf{I}^y_1(P)$ induce a two-sided ideal of $\mathcal{A}_1^{\partial y}$.
\end{proposition}
\begin{proof}
Let the two diagrams of figure \ref{I^y_1} be denoted by $I_{xy}$ and $I_{\text{box}}$, respectively. Let $D_2$ be any diagram composable with $I_{xy}$ and $I_{\text{box}}$. We need to show that $D_2\circ(I_{xy}+I_{\text{box}})$ is in $\textbf{I}^y_1$ (the other order of composition is trivial).

Denote by $A_x$ the set of all $x$-labeled vertices of $D_2$. Let $U_x$ be a subset of $A_x$, $U_y$ a subset of the $y$ labels of $I_{\text{box}}$ with the same size of $U_x$, and $p:U_y\stackrel{\cong}{\longrightarrow}U_x$ a bijection. Denote by $D_2\circ_{p} I_{\text{box}}$ the diagram obtained by attatching the $x$-labels in $U_x$ to the $y$-labels in $U_y$ according to $p$. Similarly define $D_2\circ_{p} I_{xy}$. Also, for any $l\in A_x\setminus U_x$, define $D_2\circ_{(p,l)} I_{xy}$ to be the diagram obtained by attatching the $x$-labels in $U_x$ to the $y$-labels in $U_y$ according to $p$, and also attaching $l$ to the left-most $y$ label in $I_{xy}$. We have:
$$D_2\circ(I_{xy}+I_{\text{box}})=\sum_{(U_y,U_x,p)}\left(D_2\circ_{p} I_{\text{box}}+D_2\circ_{p} I_{xy}+\sum_{l\in A_x\setminus U_x}D_2\circ_{(p,l)} I_{xy}\right)$$
For a given triple $(U_y,U_x,p)$, we will show that the corresponding summand is an $I^y_1$ relation. Indeed, by using IHX and STU relations we can replace the box from $I_{\text{box}}$ by a box over all the top solid lines of $D_2$, all the edges leading to $y$ labels of $D$ and all the labels leading to the remaining $x$-labels of $D_2$. The summands of this box which are near $x$-labels get canceled by the sum $\sum_{l\in A_x\setminus U_x}D_2\circ_{(p,l)} I_{xy}$, and we are left with a new $I^y_1$ relation, as required.
\end{proof}

We denote the quotient category $\mathcal{A}_1^{\partial y}/\textbf{I}^y_1$ by $\mathcal{A}_1^y$, and the projection by $\pi^y:\mathcal{A}_1^{\partial y}\rightarrow \mathcal{A}_1^y$.

For any $u\in \textbf{I}^y_1(P)$ we have $j_P^\partial(u)\in\textbf{I}_1(P)$, so the functor $j^\partial$ induces a functor $j:\mathcal{A}_1^y\rightarrow \mathcal{A}_1$.
\label{defj}

\subsection{Ordered Elliptic Jacobi Diagrams}
\label{subsection_ordered}
In this subsection we define the categories $\mathcal{A}_1^{\partial <}$ and $\mathcal{A}_1^<$, which are isomorphic to $\mathcal{A}_1^{\partial y}$ and $\mathcal{A}_1^y$, respectively. Their definition is, in a sense, more complicated - we take more Jacobi diagrams, and quotient them by more relations. On the other hand, the composition rule in these categories is much simpler.

\textbf{An ordered elliptic Jacobi diagram} over a pattern $P$ is an elliptic Jacobi diagram in $\mathbb{D}_1^y(P)$, with the additional data of a linear order on the labeled vertices. Denote the set of all ordered elliptic Jacobi diagrams over $P$ by $\mathbb{D}_1^<(P)$. In figures we use the convention that a labeled vertex is bigger if it appears higher in the figure.

Let $\mathcal{A}_1^{\partial <}(P)$ be the quotient of $S_\mathbb{F}(\mathbb{D}_1^<(P))$ by the relations: STU, AS, IHX, multilinearity (although we will assume, as above, that the labels are only $x,y$ and there is no need for multilinearity), and STU-like. The STU-like relation is defined as follows:
$$\begin{array}{c} v\\w \end{array}
\mathfigns{1.5}{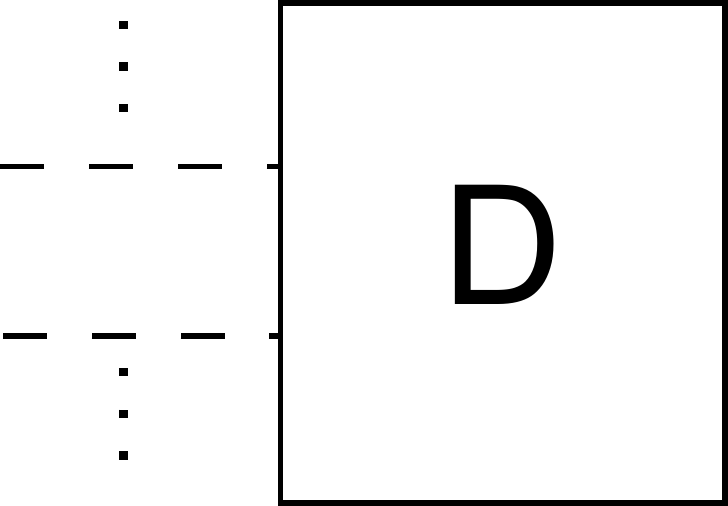} \quad- \quad
\begin{array}{c} w\\v \end{array}
\mathfigns{1.5}{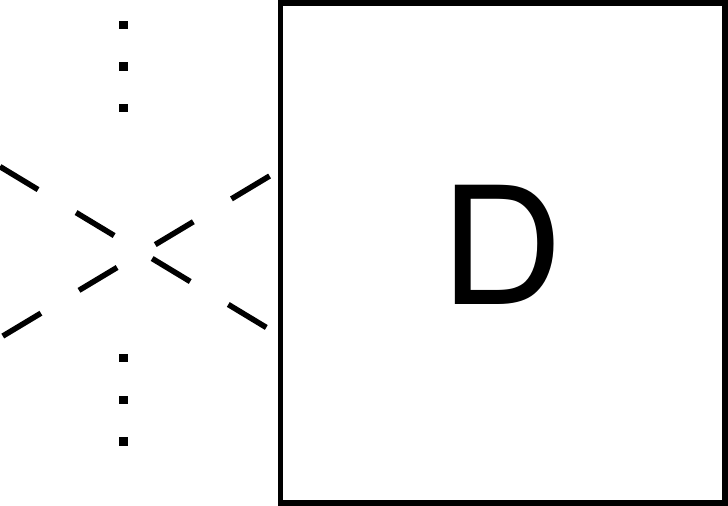} \quad=\quad
\left<v,w\right>
\mathfigns{1.5}{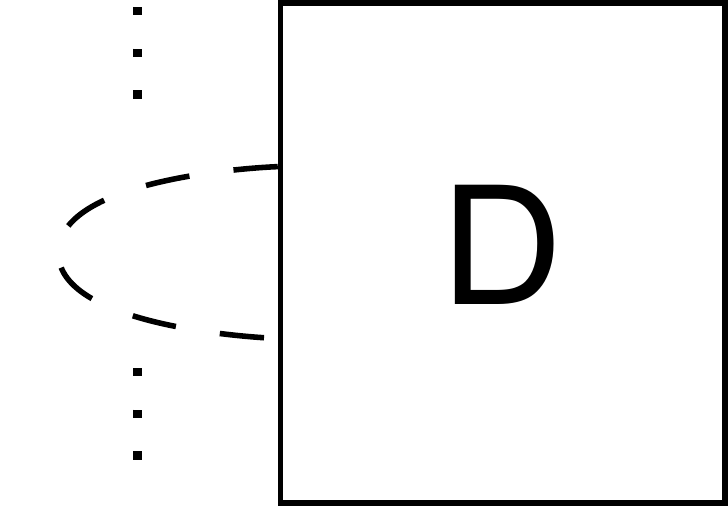}$$
$\left<\cdot,\cdot\right>$ is the intersection form on $H_1(\mathbb{T})$. Since we will only use the labels $x,y$, all we need to know is that $\left<y,x\right>=1$.

The category $\mathcal{A}_1^{\partial <}$ is defined in a way similar to the previous categories we defined in this section. The composition of $2$ diagrams $D_1,D_2$ is obtained by simply putting $D_2$ on top of $D_1$, and declaring all the labeled vertices of $D_2$ to be bigger than all the labeled vertices of $D_1$. This induces the composition of $\mathcal{A}_1^{\partial <}$ by linearity.

Let $k^\partial:\mathbb{D}_1\rightarrow \mathbb{D}_1^<$ be the map which sends a diagram $D$ to itself and declares all the $x$-labeled vertices to be smaller than all the $y$-labeled vertices. It may be verified that $k$ intertwines the compositions of $\mathcal{A}_1^{\partial y}$ and $\mathcal{A}_1^{\partial <}$, so it induces a functor $k^\partial:\mathcal{A}_1^{\partial y}\rightarrow\mathcal{A}_1^{\partial <}$.
\begin{proposition}
\label{proposition_k}
$k^\partial:\mathcal{A}_1^{\partial y}\rightarrow\mathcal{A}_1^{\partial <}$ is an isomorphism.
\end{proposition}

A short proof of this proposition, using an explicit formula for the inverse of $k^\partial$, can be found in \cite{Cheptea}. We will give here a different proof. The idea of the proof is, for any diagram $D\in\mathbb{D}_1^<(P)$, to iteratively use the STU-like relation to reduce the number of pairs of labeled vertices with $y<x$, until we get a representation of $D$ as a linear combination of diagrams from $\mathbb{D}_1^y(P)$. This simple idea is formalized using the language of filtrations and direct limits. In section \ref{section_t} we will use this technique several more times.

\begin{proof}
Since $k^\partial$ is the identity on objects, it is enough to show that for any pattern $P$, $k^\partial:\mathcal{A}_1^{\partial y}(P)\rightarrow\mathcal{A}_1^{\partial <}(P)$ is an isomorphism. For that purpose we will construct an inverse map $\varphi^\partial:\mathcal{A}_1^{\partial <}(P)\rightarrow\mathcal{A}_1^{\partial y}(P)$.

For any diagram $D\in\mathbb{D}_1^<(P)$, define $n_{y<x}(D)$ to be the number of pairs of labeled vertices in $D$ which are labeled by $x$ and $y$, and the $x$ vertex is bigger than the $y$ vertex. Define $(\mathbb{D}_1^<(P))^n:=\{D\in\mathbb{D}_1^<(P)\mid n_{y<x}(D)\le n\}$. The filtration $(\mathbb{D}_1^<(P))^n$ induces a filtration $S_\mathbb{F}((\mathbb{D}_1^<(P))^0)\subset S_\mathbb{F}((\mathbb{D}_1^<(P))^1)\subset S_\mathbb{F}((\mathbb{D}_1^<(P))^2)\subset\cdots$.

Denote by $(\mathcal{A}_1^{\partial <}(P))^n$ the quotient of $S_\mathbb{F}((\mathbb{D}_1^<(P))^n)$ by the STU, AS, IHX and STU-like relations which are contained in $S_\mathbb{F}((\mathbb{D}_1^<(P))^n)$. So the above filtration induces a sequence of maps: 
$$(\mathcal{A}_1^{\partial <}(P))^0\stackrel{k_0^\partial}{\longrightarrow} (\mathcal{A}_1^{\partial <}(P))^1\stackrel{k_1^\partial}{\longrightarrow} (\mathcal{A}_1^{\partial <}(P))^2\longrightarrow\cdots$$.

$(\mathcal{A}_1^{\partial <}(P))^0$ is isomorphic to $\mathcal{A}_1^{\partial y}(P)$, and the direct limit of the sequence is $\mathcal{A}_1^{\partial <}(P)$. The sequence of maps $k_l^\partial$ induces the map $k^\partial:\mathcal{A}_1^{\partial y}(P)\rightarrow\mathcal{A}_1^{\partial <}(P)$.

For $(l>0)$, let $\varphi^\partial_l:S_\mathbb{F}((\mathbb{D}_1^<(P))^l)\rightarrow S_\mathbb{F}((\mathbb{D}_1^<(P))^{l-1})$ be the map defined on a diagram $D$ as follows: If $n_{y<x}(D)<l$, $\varphi^\partial_l(D)=D$. Else, take the highest pair of consecutive vertices labeled $y<x$, and define $\varphi^\partial_l(D)$ by:

$$D:=\begin{array}{c} x\\y \end{array}
\mathfigns{1.5}{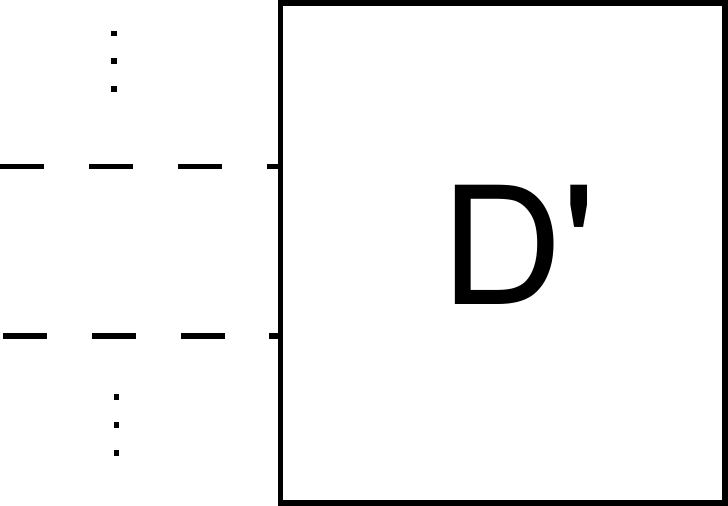} \quad \stackrel{\varphi^\partial_l}{\longmapsto} \quad
\begin{array}{c} y\\x \end{array}
\mathfigns{1.5}{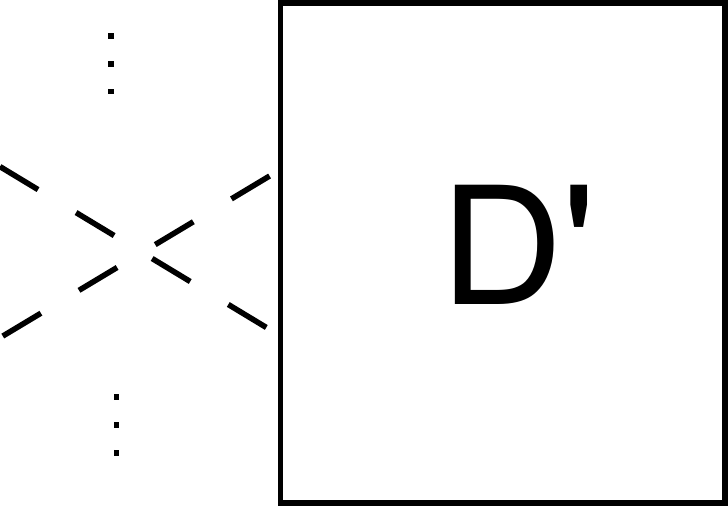} \quad-\quad
\mathfigns{1.5}{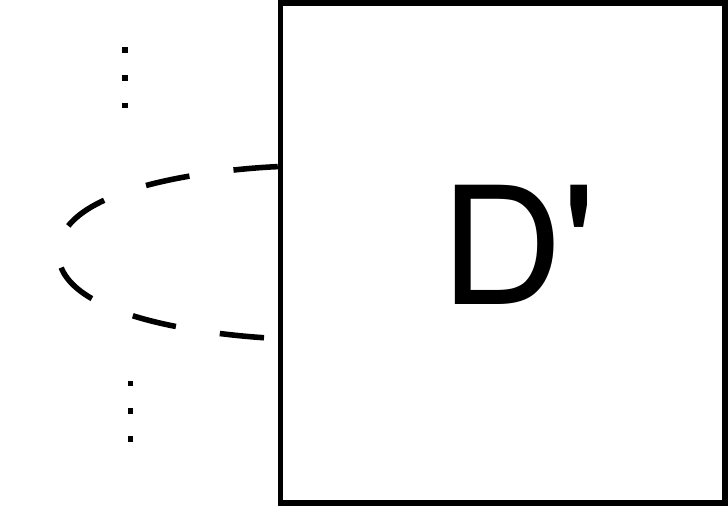}\\[0.5cm]$$

We claim that $\varphi^\partial_l$ induces a map $\varphi^\partial_l:(\mathcal{A}_1^{\partial <}(P))^l\rightarrow (\mathcal{A}_1^{\partial <}(P))^{l-1}$. Indeed, if $u\in S_\mathbb{F}((\mathbb{D}_1^<(P))^l)$ is an STU, AS or IHX relation, then $\varphi^\partial_l$ sends $u$ to a sum of corresponding relations in $S_\mathbb{F}((\mathbb{D}_1^<(P))^{l-1})$.

Suppose now  
$u=u_1+u_2+u_3=\begin{array}{c} v\\w \end{array}
\mathfigns{1}{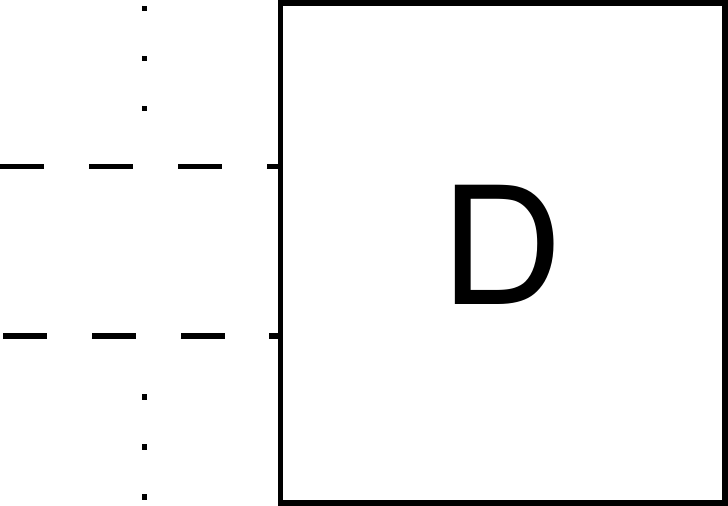} \,- \,
\begin{array}{c} w\\v \end{array}
\mathfigns{1}{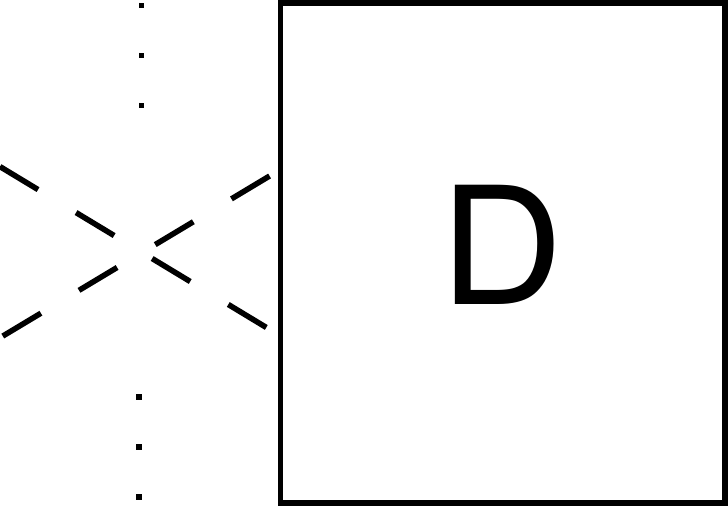} \,-\,
\left<v,w\right>
\mathfigns{1}{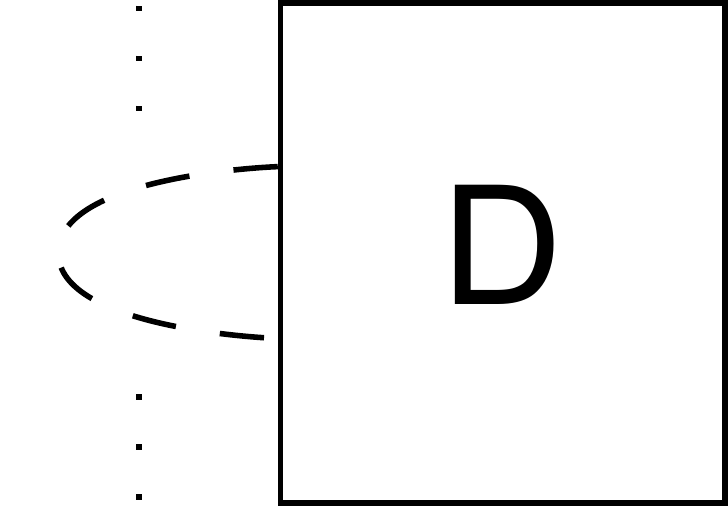}\,$ 
is an STU-like relation. If none of the labels $v,w$ belong to the highest $y<x$ pair in either $u_1$ or $u_2$, then $\varphi^\partial_l$ sends $u$ to a sum of STU-like relation in $S_\mathbb{F}((\mathbb{D}_1^<(P))^{l-1})$. If the pair $v,w$ is the highest $y<x$ pair in $u_i$ ($i=1$ or $i=2$), then by definition $\varphi^\partial_l(u)=0$ if $n_{y<x}(u_i)=l$, and otherwise $\varphi^\partial_l(u)=u$ is again an STU-like relations.

Suppose now that only $v$ or only $w$ belongs to the highest $y<x$ pair, either in $u_1$ or in $u_2$. Assume WLOG that this happens in $u_1$, and assume that $v=y$ is the label belonging to the highest $y<x$ pair. so we have:
$$u=u_1+u_2+u_3=\begin{array}{c} x\\y\\w \end{array}
\mathfigns{1.5}{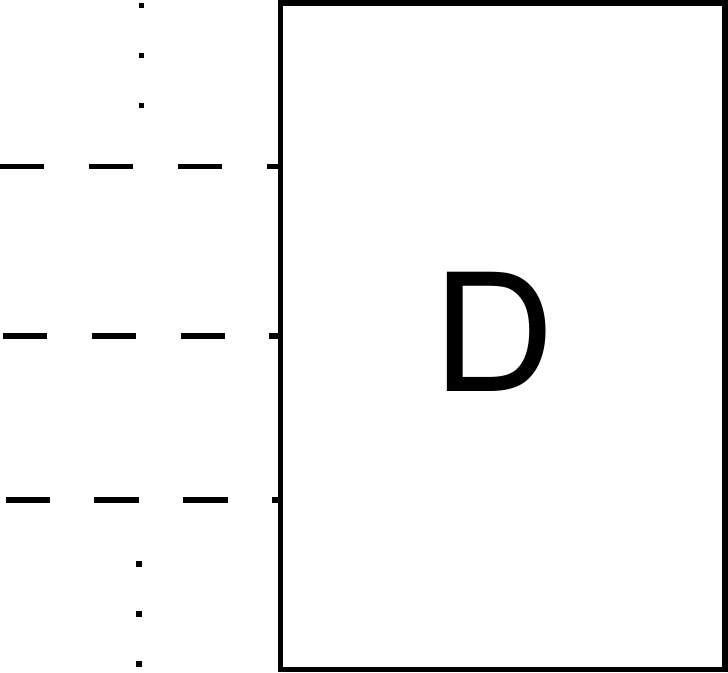} \,- \,
\begin{array}{c} x\\w\\y \end{array}
\mathfigns{1.5}{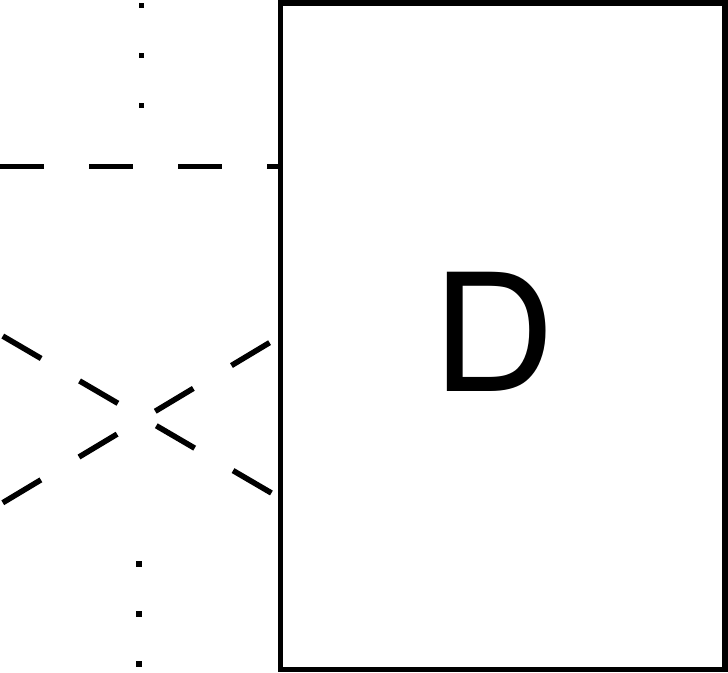} \,-\,
\left<y,w\right> \begin{array}{c} x\\ \, \\ \,  \end{array}
\mathfigns{1.5}{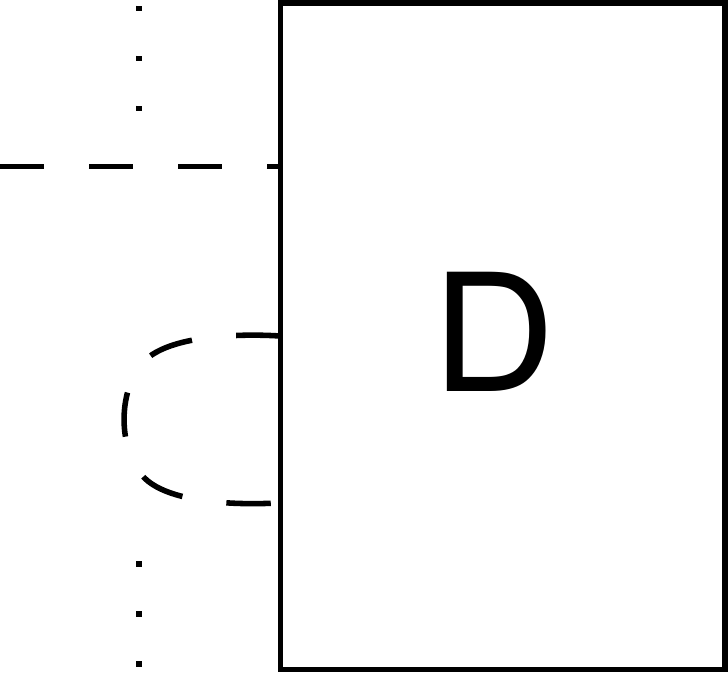}\,$$

If $w=x$ then $y<w$ is the highest $y<x$ in $u_2$, and we are back to the previous case. If $w=y$ and $n_{y<x}(u_1)<l$, then $\varphi^\partial_l(u)=u$ is again an STU-like relation. Otherwise we have the following calculation:

$$
\varphi^\partial_l(u)=\varphi_l^\partial\left(\begin{array}{c} x\\y\\y \end{array}
\mathfigns{1.5}{section2_proof_proof1} \,- \,
\begin{array}{c} x\\y\\y \end{array}
\mathfigns{1.5}{section2_proof_proof2} \,\right)=$$

$$=\,\begin{array}{c} y\\x\\y \end{array}
\mathfigns{1.5}{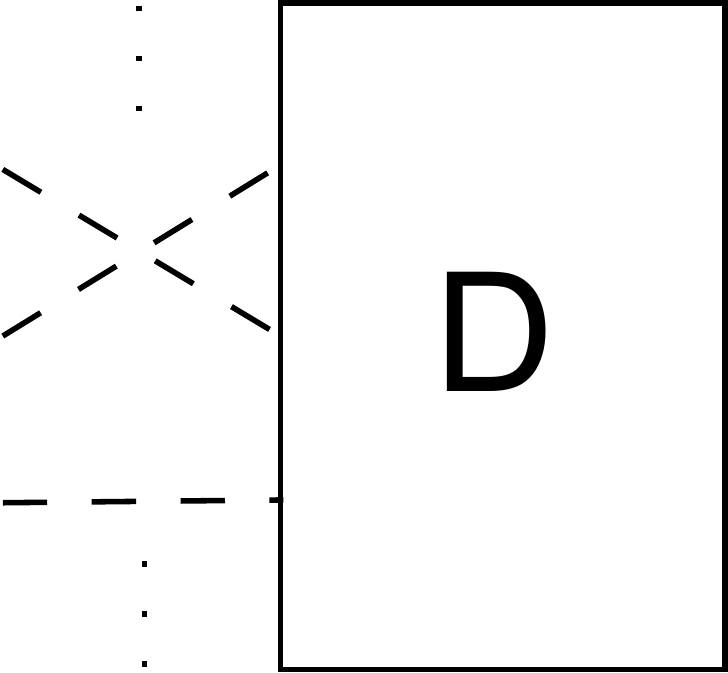} \,-
\,\begin{array}{c} \, \\ \, \\y \end{array}
\mathfigns{1.5}{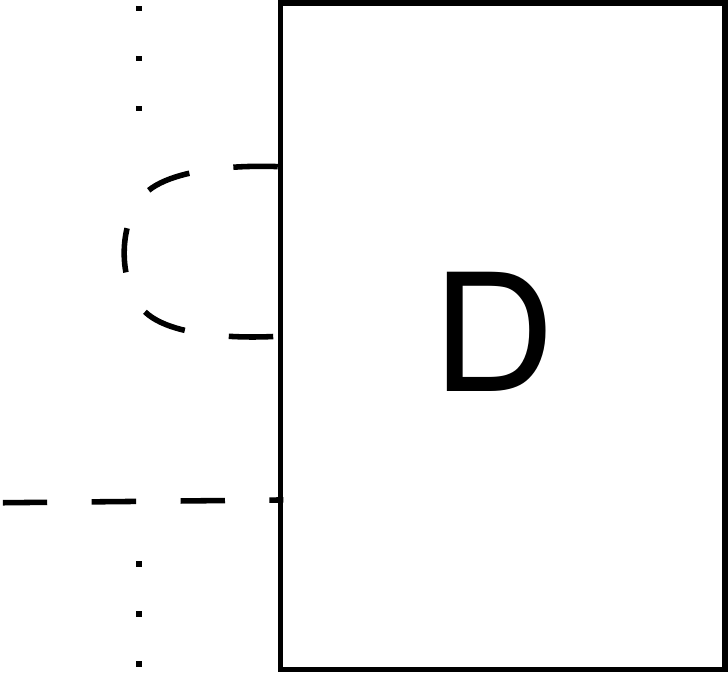} \,-
\,\begin{array}{c} y\\x\\y \end{array}
\mathfigns{1.5}{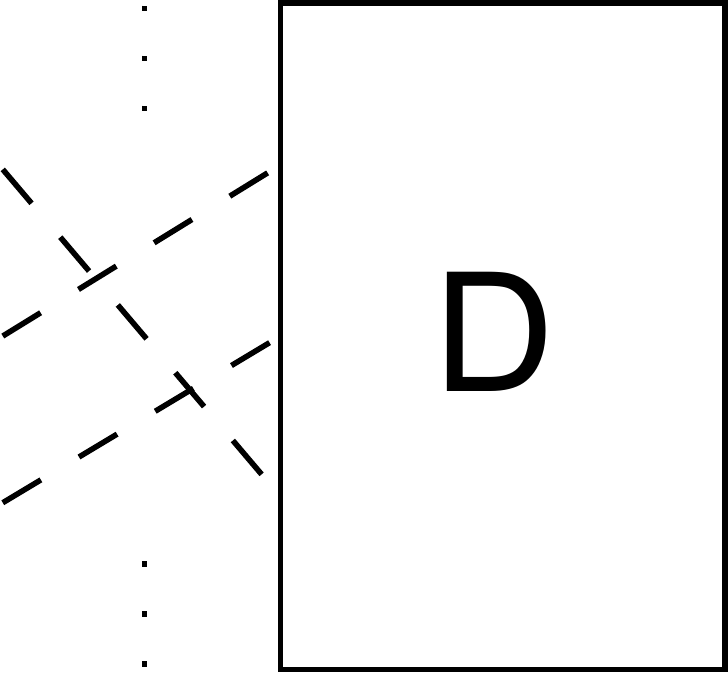} \,+ 
\,\begin{array}{c} \, \\y\\ \, \end{array}
\mathfigns{1.5}{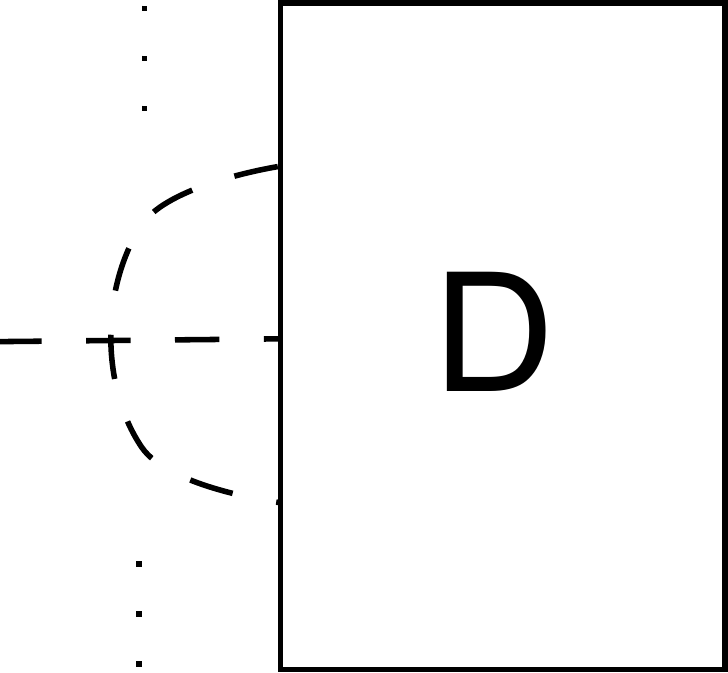} \,\approx $$

$$\approx\,\begin{array}{c} y\\y\\x \end{array}
\mathfigns{1.5}{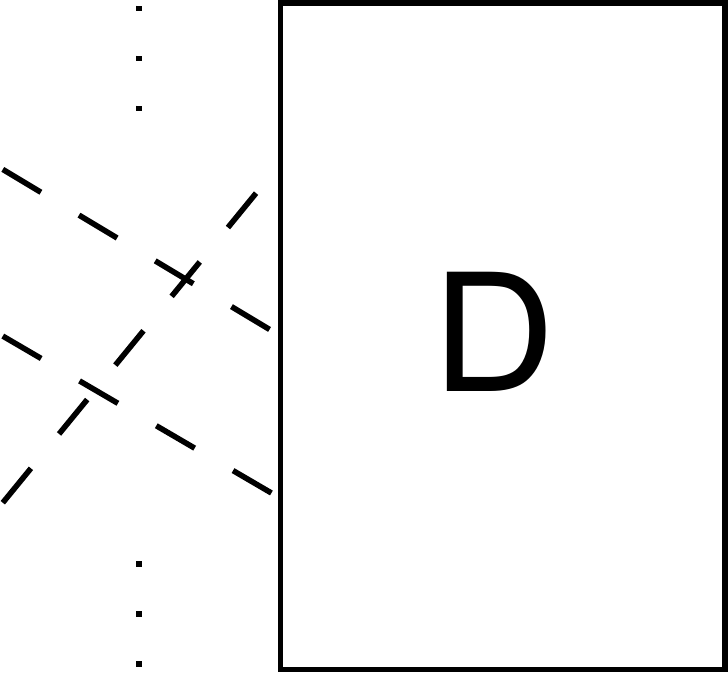} \,-
\begin{array}{c} \, \\y\\ \, \end{array}
\mathfigns{1.5}{section2_proof_proof24} \,-
\,\begin{array}{c} \, \\ \, \\y \end{array}
\mathfigns{1.5}{section2_proof_proof22} \,-$$

$$-\,\begin{array}{c} y\\y\\x \end{array}
\mathfigns{1.5}{section2_proof_proof31} \,+ 
\,\begin{array}{c} \, \\ \, \\ y \end{array}
\mathfigns{1.5}{section2_proof_proof22} \,+ 
\begin{array}{c} \, \\y\\ \, \end{array}
\mathfigns{1.5}{section2_proof_proof24}\,=\,0 $$

Note that all the diagrams in this calculation are indeed in $S_\mathbb{F}((\mathbb{D}_1^<(P))^{l-1})$. The case that $w=x$ is the label belonging to the highest $y<x$ pair in $u_1$ is similar. This completes the proof that $\varphi^\partial_l:(\mathcal{A}_1^{\partial <}(P))^l\rightarrow (\mathcal{A}_1^{\partial <}(P))^{l-1}$ is well defined.

$\varphi^\partial_l$ is the inverse of the map $k^\partial_{l-1}:(\mathcal{A}_1^{\partial <}(P))^{l-1}\rightarrow (\mathcal{A}_1^{\partial <}(P))^l$ defined above. Indeed, $\varphi^\partial_l\circ k^\partial_{l-1}=id$ by definition, and $k^\partial_{l-1}\circ \varphi^\partial_l$ sends an element $a\in(\mathcal{A}_1^{\partial <}(P))^l$ to an element equivalent to $a$ by STU-like. Therefore, the family $\{\varphi^\partial_l\}_l$ induces a map $\varphi^\partial:\mathcal{A}_1^{\partial <}(P)\rightarrow\mathcal{A}_1^{\partial y}(P)$ which is the inverse of $k^\partial$.
\end{proof} 

Let $\textbf{I}^<_1(P)$ be the subspace of $\mathcal{A}_1^{\partial <}(P)$ generated by sums of the type shown in figure \ref{I^<_1}. We assume that the lines which come out of the upper side of the box are exactly all the ends of the solid lines which lead to the points of $\partial^t(P)$. We call each such sum an $I_1^<$ relation. The subspaces $\textbf{I}^<_1(P)$ are together a two-sided ideal of $\mathcal{A}_1^{\partial <}$ (\cite{Humbert}, Lemma 2.4.4). We denote the quotient category by $\mathcal{A}_1^<$, and the projection by $\pi^<:\mathcal{A}_1^{\partial <}\rightarrow \mathcal{A}_1^<$.

\begin{figure}[ht]

\centering
$$\mathfig{4}{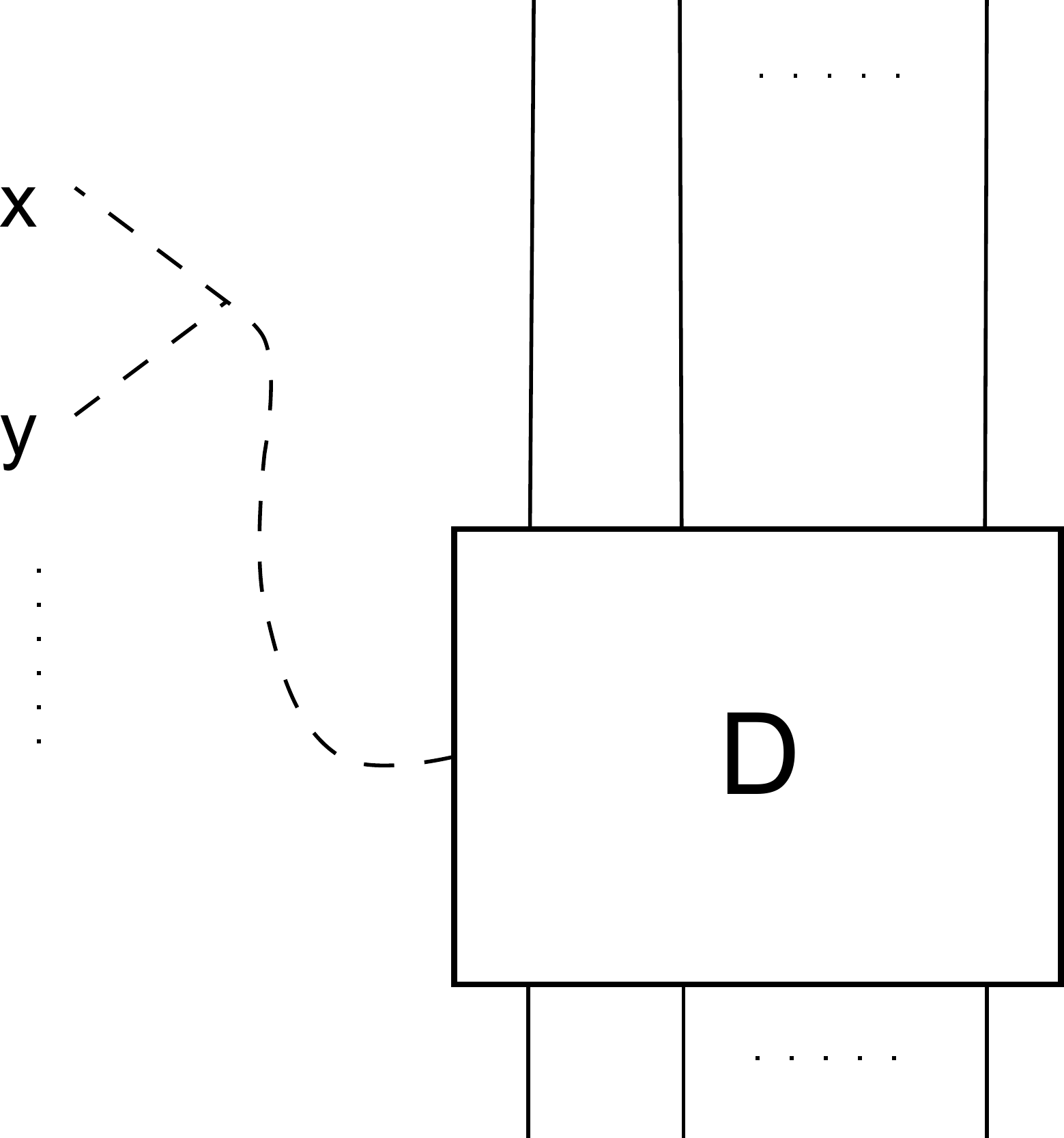} + \mathfig{4}{section2_I.pdf}$$

\caption{An $I^<_1$ relation}
\label{I^<_1}
\end{figure}

Every $I^y_1$ relation is mapped by $k^\partial$ to an $I^<_1$ relation. Indeed, if we take the $I^y_1$ as shown in figure \ref{I^y_1} and apply $k^\partial$ to it, the $x$ label in the first summand will be lower than all the $y$ labels of $D$. If we now use the STU-like relation to put the $x$ label above the $y$ labels of $D$, the extra summands we will get in the process will cancel the part of the box in the second summand which is on the $y$-labeled dashed lines, so we get an $I^<_1$ relation.

The isomorphism $k^\partial$ and its inverse $\varphi^\partial$ restrict to isomorphisms of the ideals $\textbf{I}_1^y$ and $\textbf{I}_1^<$. Therefore we have induced isomorphisms \xymatrix{ \mathcal{A}_1^y \ar[r]^k & \mathcal{A}_1^< \ar[l]^\varphi}.
\label{defk}

\subsection{Categories of Pattern-Connected Diagrams}
In this section we have defined several categories of Jacobi diagrams. All those categories and the maps between them are summarized in the following diagram. Note that the category $\mathcal{A}^\partial$ has an obvious inclusion into all the categories of elliptic diagrams.

\[\xymatrix{
&& \mathcal{A}^\partial\ar[r]\ar[dll]\ar[d]\ar[drr] & \mathcal{A}\\
^{ts}\mathcal{A}_1^\partial\ar[d]^\pi && \mathcal{A}_1^{\partial y}\ar[ll]_{j^\partial}\ar[d]^{\pi^y}\ar[rr]^{k^\partial} && \mathcal{A}_1^{\partial <}\ar[d]^{\pi^<}\ar[ll]\\
\mathcal{A}_1 && \mathcal{A}_1^y\ar[ll]_j\ar[rr]^k && \mathcal{A}_1^<\ar[ll]
}\]

Let $D$ be a Jacobi diagram in any of the sets of Jacobi diagrams defined above. We say that $D$ is \textbf{pattern-connected} if all the non-struts components of $D$ have at least one vertex on the pattern. We denote the subsets of pattern-connected Jacobi diagrams by $\mathbb{D}^p$, $\mathbb{D}_1^p$, $\mathbb{D}_1^{yp}$ and $\mathbb{D}_1^{<p}$. All the relations we saw respect those subsets, so we can define the corresponding categories of pattern-connected diagrams, which fall into the following diagram:

\[\xymatrix{
&& \mathcal{A}^{\partial p}\ar[r]\ar[dll]\ar[d]\ar[drr] & \mathcal{A}^p\\
^{ts}\mathcal{A}_1^{\partial p}\ar[d]^\pi && \mathcal{A}_1^{\partial yp}\ar[ll]_{j^\partial}\ar[d]^{\pi^y}\ar[rr]^{k^\partial} && \mathcal{A}_1^{\partial <p}\ar[d]^{\pi^<}\ar[ll]\\
\mathcal{A}_1^p && \mathcal{A}_1^{yp}\ar[ll]_j\ar[rr]^k && \mathcal{A}_1^{<p}\ar[ll]
}\]
Note that in all the categories of pattern-connected Jacobi diagrams we no longer need the AS and IHX relations, because they are implied by STU.

The category $\mathcal{A}^{\partial p}$ is the category which is denoted by $\textbf{A}$ in \cite{Humbert}, and the category $\mathcal{A}_1^{<p}$ is the category which is denoted by $\textbf{A}_1$ there. In the following sections we will describe $2$ different elliptic structures with respect to $(\mathcal{A}^{\partial p}\rightarrow \mathcal{A}_1^{<p})$.

\newpage
\section{The LMO Functor of Cobordisms with Embedded Tangles}
\label{section_LMO}

The LMO functor was defined by Cheptea, Habiro and Massuyeau (\cite{Cheptea}). It is a functor from the category of Lagrangian cobordisms to a certain category of Jacobi diagrams. In this section we extend the LMO functor to the category of Lagrangian cobordisms with embedded tangles. The extension is quite straight-forward, so most of this section may be seen as a review of CHM's work, with a slight generalization.

On the other hand, our construction will be more restricted than the construction of CHM. They deal with cobordisms between surfaces of any genus, with or without boundary. We will restrict ourselves to closed surfaces of genus 1, which is all we need here. We made this choice for convenience, to make the notation simpler, but the extension to any genus should be obvious.

At the end of this section we show how this extended LMO functor gives rise to an elliptic structure on the categories of Jacobi diagrams introduced in section \ref{section_jacobi}.

\subsection{The Category of Lagrangian Cobordisms with Embedded Tangles}
Recall that $\mathbb{T}$ is the torus $S^1\times S^1$. Denote by $\mathbb{T}^\partial$ a torus with one boundary component. A \textbf{cobordism of $\mathbb{T}^\partial$} is an oriented compact connected $3$-manifold $M$ with an isomorphism $m:\partial(\mathbb{T}^\partial\times[0,1])\stackrel{\cong}{\rightarrow}\partial M$. Similarly, a \textbf{cobordism of $\mathbb{T}$} is an oriented compact connected $3$-manifold $M$ with an isomorphism $m:\partial(\mathbb{T}\times[0,1])\stackrel{\cong}{\rightarrow}\partial M$.

Recall that for any $n\ge 0$ we defined the set of points $b_n$ in $\mathbb{T}$. We can define those sets of points also in  $\mathbb{T}^\partial$. Let $\omega_s$, $\omega_t$ be non-associative words in the symbols $\{+,-\}$, with lengths $|\omega_s|=m$, $|\omega_t|=n$. A \textbf{cobordism with an embedded tangle} of type $(\omega_s,\omega_t)$ (either of $\mathbb{T}^\partial$ or of $\mathbb{T}$) is a cobordism $M$ with an embedded framed oriented tangle $T\subset M$ satisfying $\partial T=m((b_m\times\{0\})\cup(b_n\times\{1\}))$, such that the orientations of $T$ around the boundary points correspond to the words $\omega_s$ and $\omega_t$, and the framings around the boundary points are all parallel to the boundary surfaces and parametrized (via $m$) as $(0,-1,0)$. As in definition \ref{def_tangle} above, we require that the tangles be piece-wise smooth and vertical near the boundary points.

Two cobordisms with tangles $(M_1,m_1,T_1)$ and $(M_2,m_2,T_2)$ are said to be equivalent if there is a homeomorphism $h:M_1\rightarrow M_2$ such that $h\circ m_1=m_2$ and $h(T_1)=T_2$ (including the framing and the orientation of the tangle).

The categories $CT^\partial$ and $CT$ (Cobordisms with Tangles) are defined as follows: The objects are non-associative words in $\{+,-\}$. For any two such words $\omega_s$, $\omega_t$, the set of morphisms $CT^\partial(\omega_s,\omega_t)$ ($CT(\omega_s,\omega_t)$) is the set of all equivalence classes of cobordisms of $\mathbb{T}^\partial$ ($\mathbb{T}$) with embedded tangles of type $(\omega_s,\omega_t)$. The composition is defined by simply putting one cobordism on top of the other.

We now explain how to represent a cobordism with tangle by another tangle embedded in a simpler manifold. In $D^2$ we have the sets of points $b_n$, and we choose $2$ more points $p$ and $q$. Let $\omega_s$, $\omega_t$ be non-associative words in $\{+,-\}$. \textbf{A representing tangle} of type $(\omega_s,\omega_t)$ is a framed oriented tangle $T$ embedded in $D^2\times I$ with the following properties:
\begin{itemize}
\item
The boundary of $T$ is $((b_{|\omega_s|}\cup\{p,q\})\times\{0\})\cup((b_{|\omega_t|}\cup\{p,q\})\times\{1\})$.
\item
There is a component whose boundary is $\{p,q\}\times\{0\}$, denoted by $x$, and there is a component whose boundary is $\{p,q\}\times\{1\}$, denoted by $y$.
\item
The orientations of $T$ around the boundary correspond to the words $(-+)(\omega_s)$ and $(-+)(\omega_t)$. (Note that the convention used in \cite{Cheptea} is opposite to ours, so they represent the orientations at $p$ and $q$ by $(+-)$, instead of $(-+)$.)
\item
We are given a subset $S$ of the closed components of $T$, and say that those components are marked for surgery.
\end{itemize}

Two representing tangles $T_1$ and $T_2$ are said to be equivalent if they can be related by a sequence of ambient isotopies and Kirby moves on the components marked for surgery. This means that we can add to $S$ a trivial closed component with $\pm1$ framing or remove such component from $S$, and we can slide any component over components of $S$.

The category $RT$ (Representing Tangles) is defined as follows: The objects are non-associative words in $\{+,-\}$. For any two such words $\omega_s$ and $\omega_t$, the morphisms set $RT(\omega_s,\omega_t)$ is the set of all equivalence classes of representing tangles of type $(\omega_s,\omega_t)$.

There is a simple operation $\circ:RT(\omega_s,\omega_t)\times RT(\omega_t,\omega_u)\rightarrow RT(\omega_s,\omega_u)$ which takes $2$ composable tangles and simply puts them one on top of the other. But the defintion of the composition in $RT$ is different. Let $T_1\in RT(\omega_s,\omega_t)$, $T_2\in RT(\omega_t,\omega_u)$ be representing tangles with subsets marked for surgery $S_1$ and $S_2$, respectively. We define the composition $T_2\cdot T_1$ to be $T_2\circ T_i(\omega_t)\circ T_1$, where $T_i(\omega)$ is the tangle shown in figure \ref{tangle_T_i}.

\begin{figure}[ht]

\centering
$\begin{array}{c}
p\quad\quad\quad\quad\quad q\quad\quad\quad b_{|\omega|}\\
\includegraphics[height=2 cm]{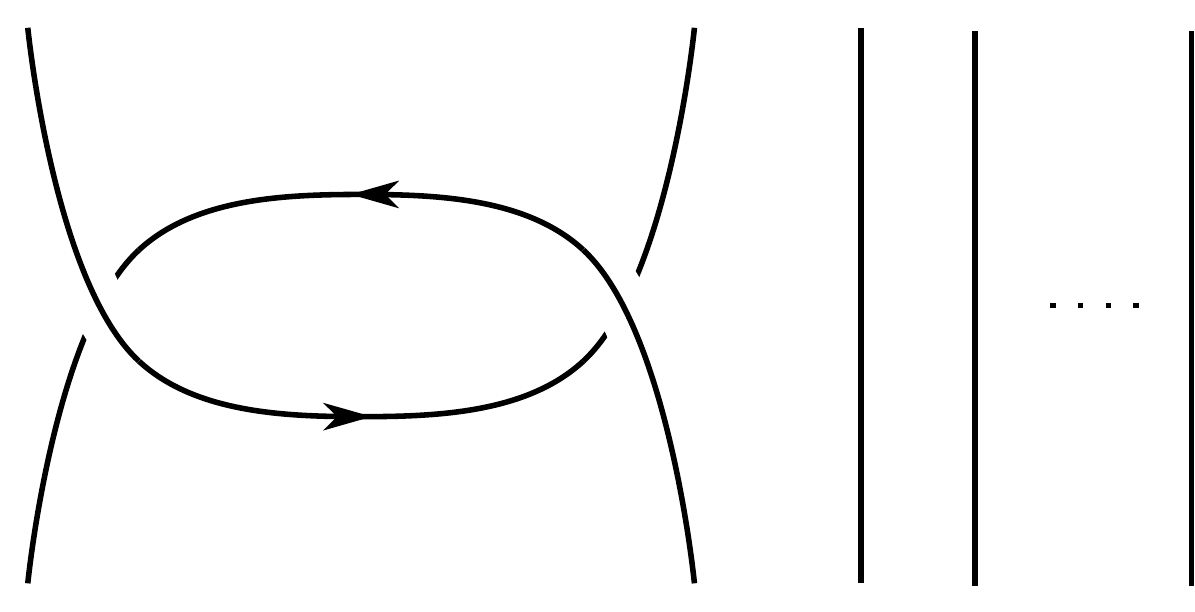}\\
p\quad\quad\quad\quad\quad q\quad\quad\quad b_{|\omega|}\\
\end{array}$

\caption{The tangle $T_i(\omega)$}
\label{tangle_T_i}
\end{figure}

The set of components marked for surgery in $T_2\cdot T_1$ is defined to be the union of $S_1$, $S_2$ and the (now closed) $y$ component of $T_1$ and $x$ component of $T_2$.
The identity in $RT(\omega,\omega)$ is:

\begin{center}
$\begin{array}{c}
p\quad\quad\quad\quad\quad q\quad\quad\quad b_{|\omega|}\\
\includegraphics[height=2 cm]{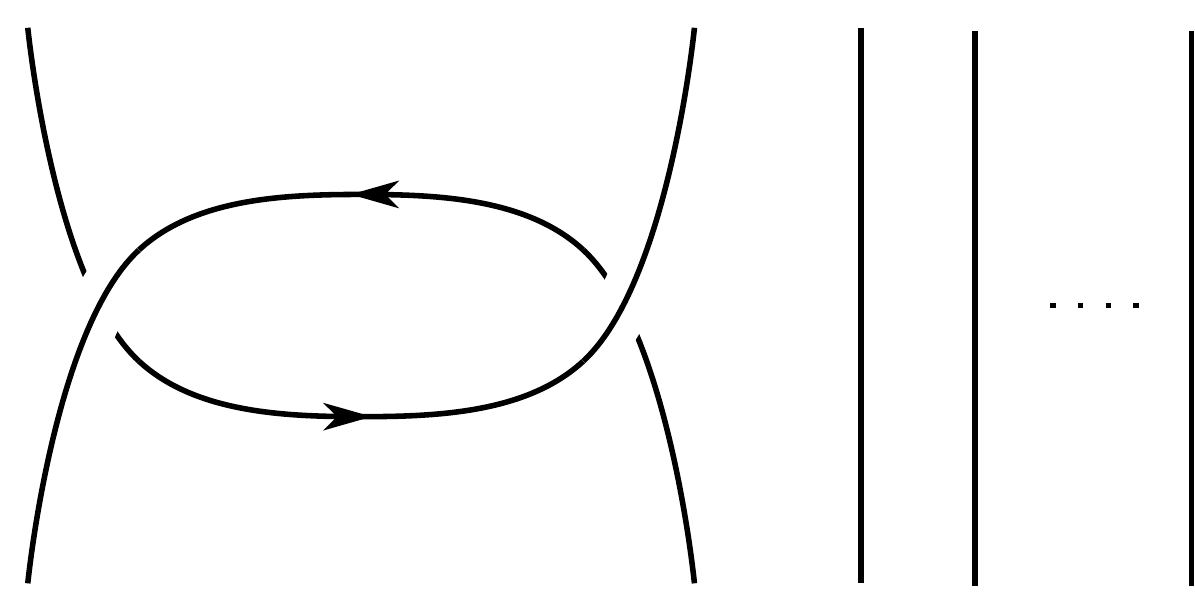}\\
p\quad\quad\quad\quad\quad q\quad\quad\quad b_{|\omega|}\\
\end{array}$
\end{center}

There is a functor $rep:RT\rightarrow CT^\partial$, which is the identity on objects, and for a representing tangle $T$, $rep(T)$ is the cobordism of $\mathbb{T}^\partial$ with embedded tangle obtained by removing a tubular neighborhood of $x,y$ from $D^2\times I$, and performing surgery on the components in $S$. The parametrization $m:\partial(\mathbb{T}^\partial\times[0,1])\rightarrow \partial rep(T)$ is chosen in such a way that the generators $x$ and $y$ of $\mathbb{T}^\partial$ from figure \ref{torus} are mapped to the following elements on the boundary of $rep(T)$:
\begin{minipage}{30pc}
\begin{center}
\includegraphics[width=8cm]{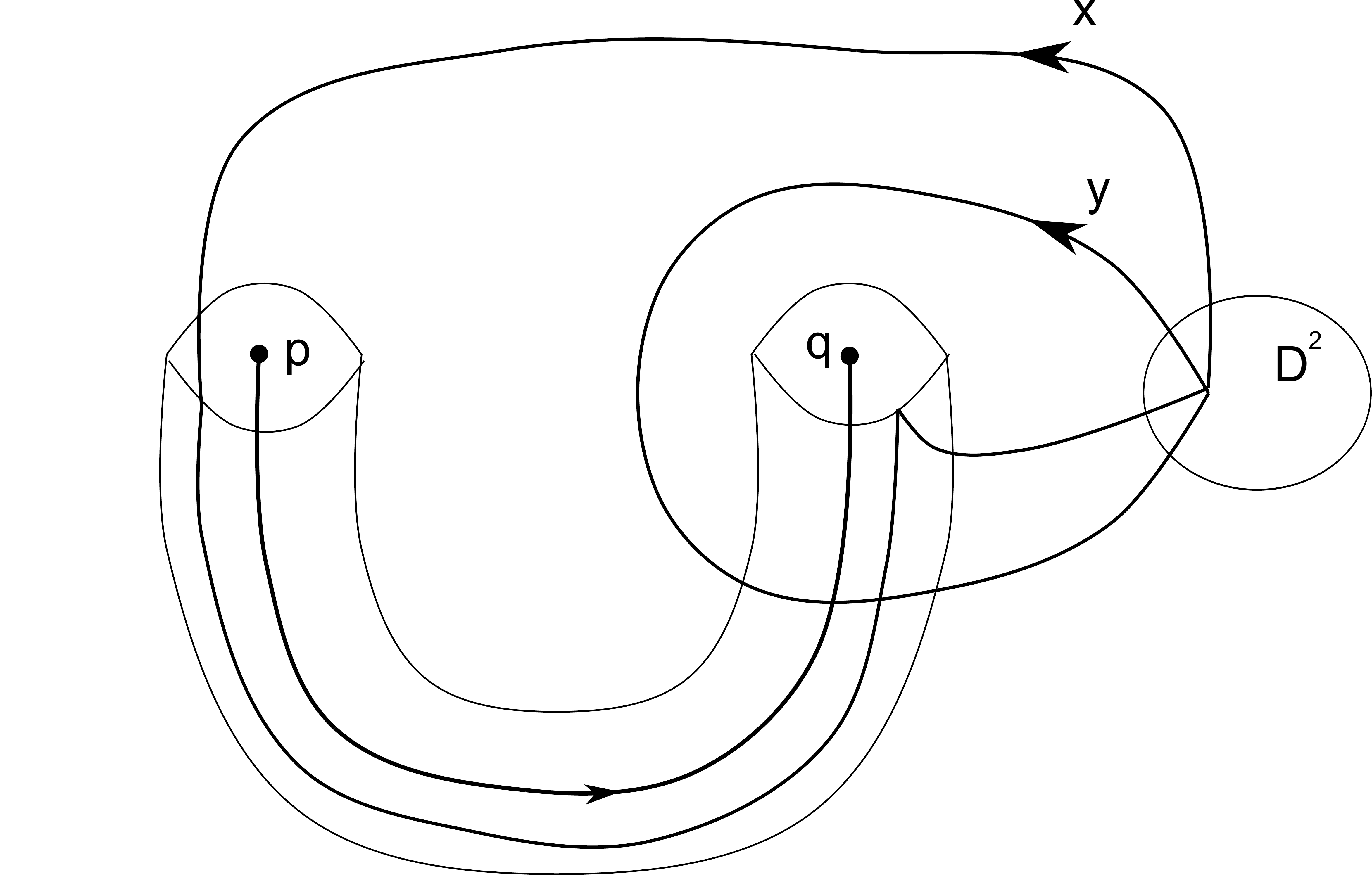}

\includegraphics[width=8cm]{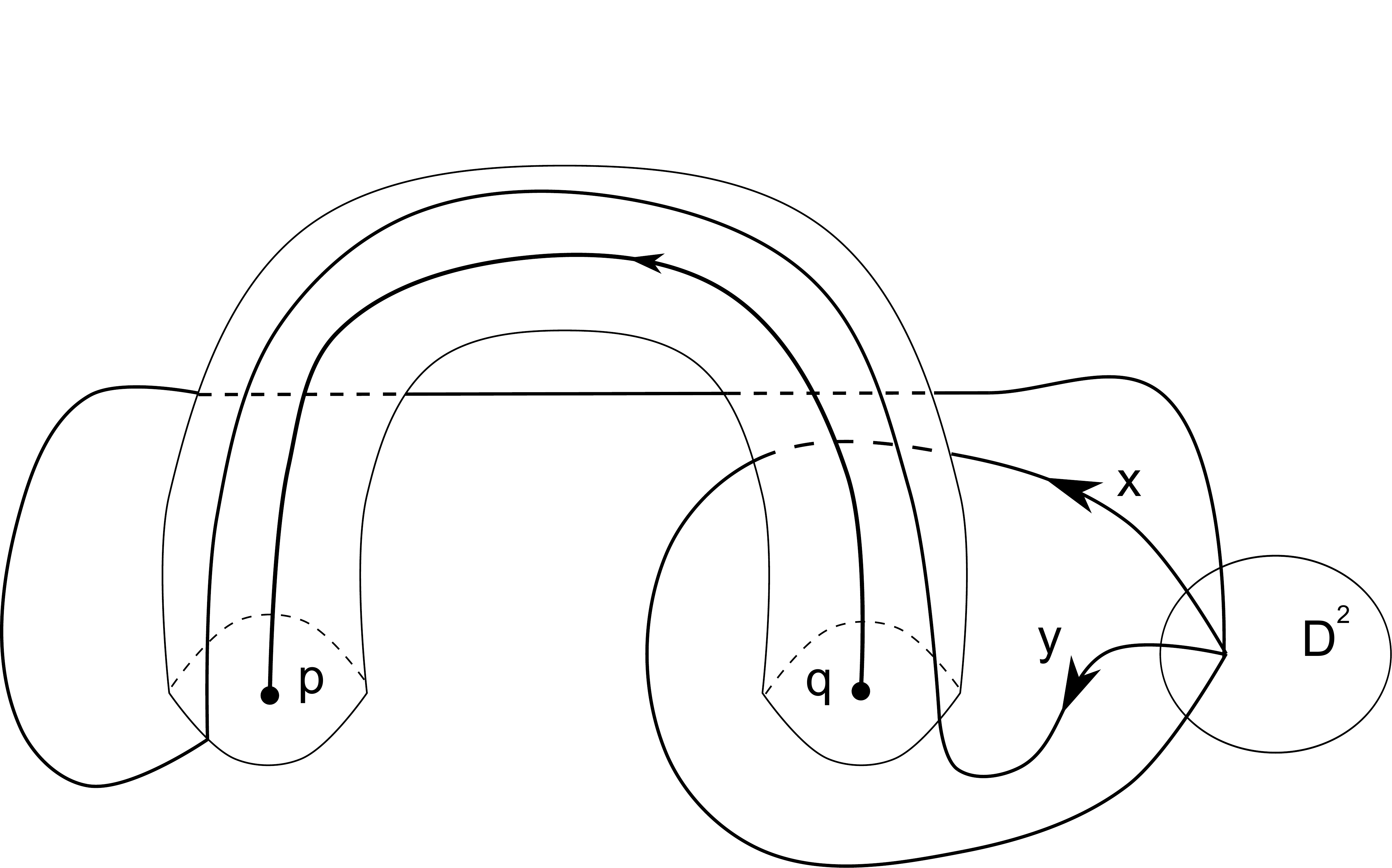}
\end{center}
\end{minipage}

The segments of the paths $x$ and $y$ which are parallel to the removed components of the tangle are determined by the framing of the tangle.

The functor $rep$ is very much related to the functor $D$ from Theorem 2.10 of \cite{Cheptea}. $D$ is a functor from the category of bottom-top tangles in homology cubes to the category of cobordisms. The category of bottom-top tangles in homology cubes, when restricted to tangles with one bottom component and one top component, is isomorphic to the category $RT$ via standard Kirby calculus. Therefore $rep$ factors through $D$, which proves that it is indeed a functor and an isomorphism.

In order to define the LMO functor we need to restrict to the subcategories $LCT^\partial\subset CT^\partial$ and $LCT\subset CT$ of Lagrangian cobordisms with embedded tangles. The exact definition of Lagrangian cobordisms is not important here, and can be found in \cite{Cheptea} (Definition 2.4). For our purposes it is enough to say that the corresponding subcategory $LRT$ of $RT$ is the subcategory of all representing tangles in which the determinant of the linking matrix of $S$ equals $\pm 1$, and the framing of $y$, after performing surgery on $S$, is $0$.

\subsection{Definition of the LMO Functor}
Let $T$ be a tangle in $LRT(\omega_s,\omega_t)$. Denote by $P\in\mathbb{P}(\omega_s,\omega_t)$ the skeleton of $T$. It is decomposed as $P=\{x,y\}\cup S \cup P'$.

Recall that a Drinfel'd associator is an element $\phi(A,B)$ in the exponent of the completed free Lie algebra generated by $A$ and $B$, which satisfies several identities (see, for example, \cite{BarNatan5}, Definition 3.1). We define $Z$ to be a functor from the category of tangles to the category of Jacobi diagrams $\mathcal{A}^\partial$. The tangle $T$ will be mapped by $Z$ to $Z(T)\in\mathcal{A}^\partial(P)$. $Z$ is a variant of the Kontsevich integral of tangles, which is defined over elementary tangles as follows:

$$Z\left(\begin{array}{c} u\,(v\,w)\\ \includegraphics[height=0.8cm]{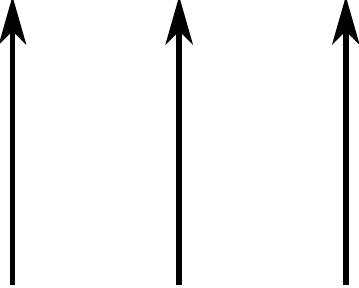}\\ (u\,v)\,w \end{array} \right)=\Delta^{+++}_{uvw}\left(\phi\left(\begin{array}{c} \includegraphics[height=0.8cm]{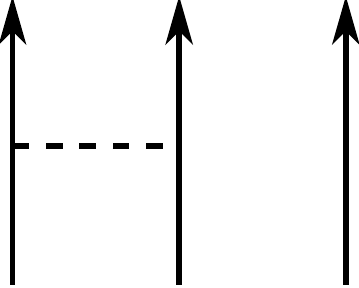}\, ,\, \includegraphics[height=0.8cm]{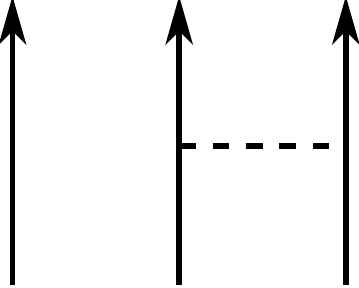} \end{array}\right)\right)$$
%$$Z\left(\begin{array}{c} (+\,+)\\ \includegraphics[height=0.8cm]{section4_right_crossing.pdf}\\ (+\,+)\end{array} \right)=\begin{array}{c} \, \\ \begin{array}{c}\\exp(\frac{1}{2}\\ \,\end{array} \includegraphics[height=1.5cm]{section4_jacobi_crossing.pdf}\begin{array}{c}) \\ \, \end{array}\end{array}\quad\quad
%Z\left(\begin{array}{c} (+\,+)\\ \includegraphics[height=0.8cm]{section4_left_crossing.pdf}\\ (+\,+)\end{array} \right)=\begin{array}{c} \, \\ \begin{array}{c}\exp(-\frac{1}{2}\\ \,\end{array} \includegraphics[height=1.5cm]{section4_jacobi_crossing.pdf}\begin{array}{c}) \\ \, \end{array}\end{array}$$
$$Z\left(\begin{array}{c} (+\,+)\\ \includegraphics[height=0.8cm]{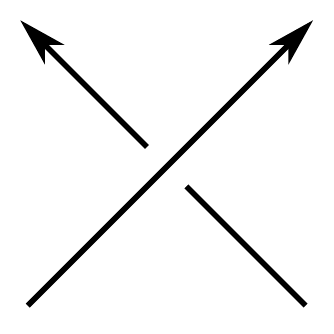}\\ (+\,+)\end{array} \right)=\exp(\frac{1}{2} \raisebox{-0.1cm}{\includegraphics[height=1.5cm]{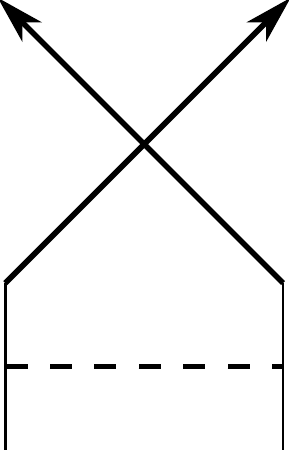}}\,)\quad\quad
Z\left(\begin{array}{c} (+\,+)\\ \includegraphics[height=0.8cm]{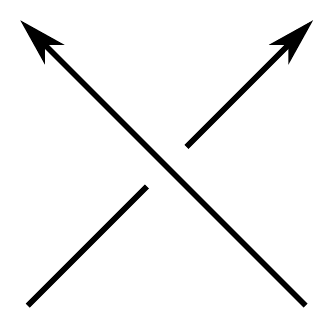}\\ (+\,+)\end{array} \right)=\exp(-\frac{1}{2}\raisebox{-0.1cm}{\includegraphics[height=1.5cm]{section4_jacobi_crossing.pdf}}\,)$$
$$Z\left(\begin{array}{c} \includegraphics[height=0.5cm]{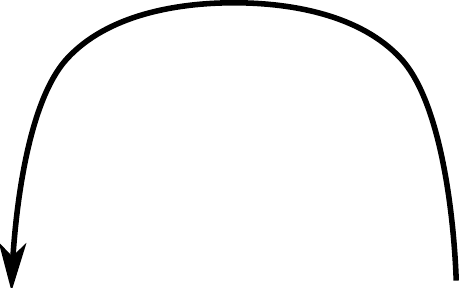}\\ (-\quad+)\end{array} \right)=\,\,\includegraphics[height=0.5cm]{section4_cap.pdf}\quad\quad
Z\left(\begin{array}{c}  (-\quad+)\\ \includegraphics[height=0.5cm]{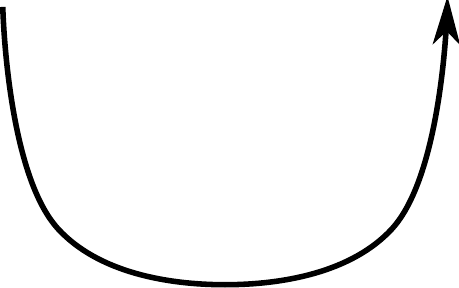}\end{array} \right)=\begin{array}{c}\, \\ \includegraphics[height=0.7cm]{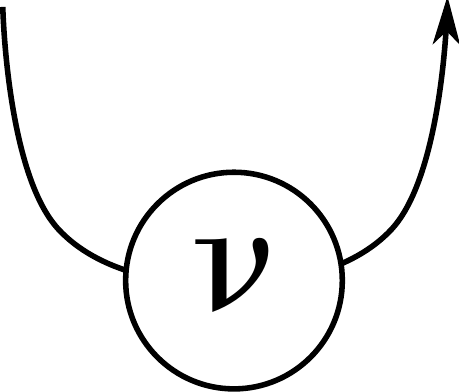}\end{array}$$
where $\nu\in\mathcal{A}^\partial(\uparrow)\cong\mathcal{A}^\partial(\includegraphics[height=0.3cm]{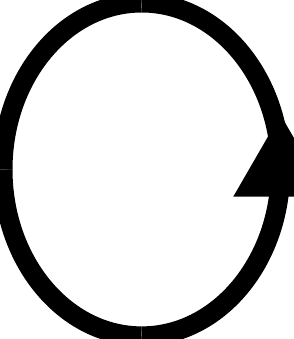})$ is the Kontsevich integral of the unknot with $0$ framing.

Let $Z^{\nu,S}(T)$ be the value obtained from $Z(T)$ by taking the connected sum of each component of $S$ with $\nu$.

In \cite{Cheptea} (after Lemma 4.9), an element $T_g\in\mathcal{A}(\emptyset,\{1^-,...,g^-,1^+,...,g^+\})$ is defined. We will consider  $T_1$ as an element of $\mathcal{A}(\emptyset,\{x,y\})$ via the labels change $1^-\mapsto x$ and $1^+\mapsto y$. For a word $\omega$ in the symbols $\{+,-\}$, let $id_\omega$ be the identity pattern in $\mathbb{P}(\omega,\omega)$, and let $T_1(\omega)\in^{ts}\mathcal{A}_1^\partial(id_\omega)$ be the element obtained by putting $T_1$ alongside the empty pattern $id_\omega$.

The LMO functor $LMO:LCT^\partial\cong LRT\rightarrow ^{ts}\mathcal{A}_1^\partial$ is defined as follows: A non-associative word $\omega$ is mapped to itself, forgetting the non-associative structure. A morphism $T\in LRT(\omega_s,\omega_t)$ is mapped to:
$$LMO(T):=T_1(\omega_t)\cdot\left(U_+^{-\sigma_+(S)}U_-^{-\sigma_-(S)}\int_S \chi^{-1}_{S\cup\{x,y\}}Z^{\nu,S}(T)\right)$$
where:
\begin{itemize}
\item
$\chi_{S'}:\mathcal{A}(P',S')\rightarrow \mathcal{A}^\partial(P'\cup S')$ is the symmetrization map defined in \cite{BarNatan} (section 5.2, in the proof of Theorem 8), applied to the components of $S'$ (which are first considered as labels). For a diagram $D\in\mathcal{A}(P',S')$, $\chi_{S'}(D)$ is the average of all the diagrams which are obtained by putting all the $s$ labeled vertices on the $s$ component of the pattern $S'$ (for all $s\in S'$), in any possible order.
\item
$\int_{S'}:$ is the \r{A}rhus integral on the labels of $S'$, as defined in \cite{BarNatan3} (section 2.1, specifically Definition 2.11).
\item
$U_\pm:=\int\chi^{-1}\left(\nu\# Z(\includegraphics[height=0.3cm]{section4_circle}_\pm)\right)$ with $\includegraphics[height=0.3cm]{section4_circle}_\pm$ being an unknot with framing $\pm1$.
\item
$\sigma_\pm(S)$ are the numbers of positive/negative eigenvalues of $Lk(S)$.
\end{itemize}
For a more detailed account of this construction, and a proof of invariance and functoriality, see \cite{Cheptea}. For our purposes it is almost enough to use this definition as a ``black box". The only important thing to notice is that the integral $\int_S$ commutes with $\chi^{-1}_{\{x,y\}}$, i.e. we have:
$$\int_S\chi^{-1}_{S\cup\{x,y\}}Z^{\nu,S}(T)=\chi^{-1}_{\{x,y\}}\int_S\chi^{-1}_SZ^{\nu,s}(T)$$

In order to define the LMO functor on cobordisms of $\mathbb{T}$ with embedded tangles, all we need to do is choose a representing tangle $T$, map it by $LMO$ to $^{ts}\mathcal{A}_1^\partial$, and then map it to $\mathcal{A}_1$ by the quotient map. Thus we get a functor $LMO:LCT\rightarrow \mathcal{A}_1$. The fact that this functor is well defined is also proved in \cite{Cheptea}, Theorem 6.2.

\subsection{Restrictions of the LMO Functor}
Let $HCT$ be the subcategory of $LCT$ containing only tangles in cobordisms which are homology cylinders. Homology cylinders are cobordisms which are homologically equivalent to the cylinder $\mathbb{T}\times I$ (an exact definition can be found in \cite{Cheptea} Definition 8.1). The LMO functor, when restricted to this category, gives values of the form $\exp\left(\begin{array}{c} y\\ \includegraphics[height=0.7cm]{section2_strut.pdf}\\x\end{array}\right)\otimes a$, where $a$ is a combination of diagrams from $\mathbb{D}_1^y$ (see \cite{Cheptea} section 8.2 and \cite{Habiro12} section 4.2). We can therefore compose the LMO functor with the inverse of the injective functor $j:\mathcal{A}_1^y\rightarrow \mathcal{A}_1$ defined above to get a functor $LMO^y:HCT\rightarrow \mathcal{A}_1^y$. Furthermore, we can also define $LMO^<:HCT\rightarrow \mathcal{A}_1^<$ by $LMO^<:= k\circ LMO^y$, $k$ being the isomorphism of categories defined in section \ref{subsection_ordered}.

If we further restrict $HCT$ to include only tangles in the trivial cylinder $\mathbb{T}\times I$, we get the category $q\tilde{T}_1$ of framed tangles in the thickened torus defined in Section \ref{def_category_tangles}. Restricting the above variants of the LMO functor to this subcategory we get $LMO^y:q\tilde{T}_1\rightarrow \mathcal{A}_1^{yp}$ and $LMO^<:q\tilde{T}_1\rightarrow \mathcal{A}_1^{<p}$.

\subsection{An Elliptic Structure on $\mathcal{A}_1^{<p}$}
In section \ref{def_category_tangles} we defined an elliptic structure on $q\tilde{T}_1$. The key ingredients were the tangles $X_{+,+}=\tilde{x}\cdot pt$ and $Y_{+,+}=\tilde{y}\cdot nt$ in $q\tilde{T}_1(++,++)$. We will now give an explicit description of those tangles via their representing tangles in $RT(++,++)$:

$$X_{+,+}=\quad\raisebox{-4.5cm}{\includegraphics[height=4cm]{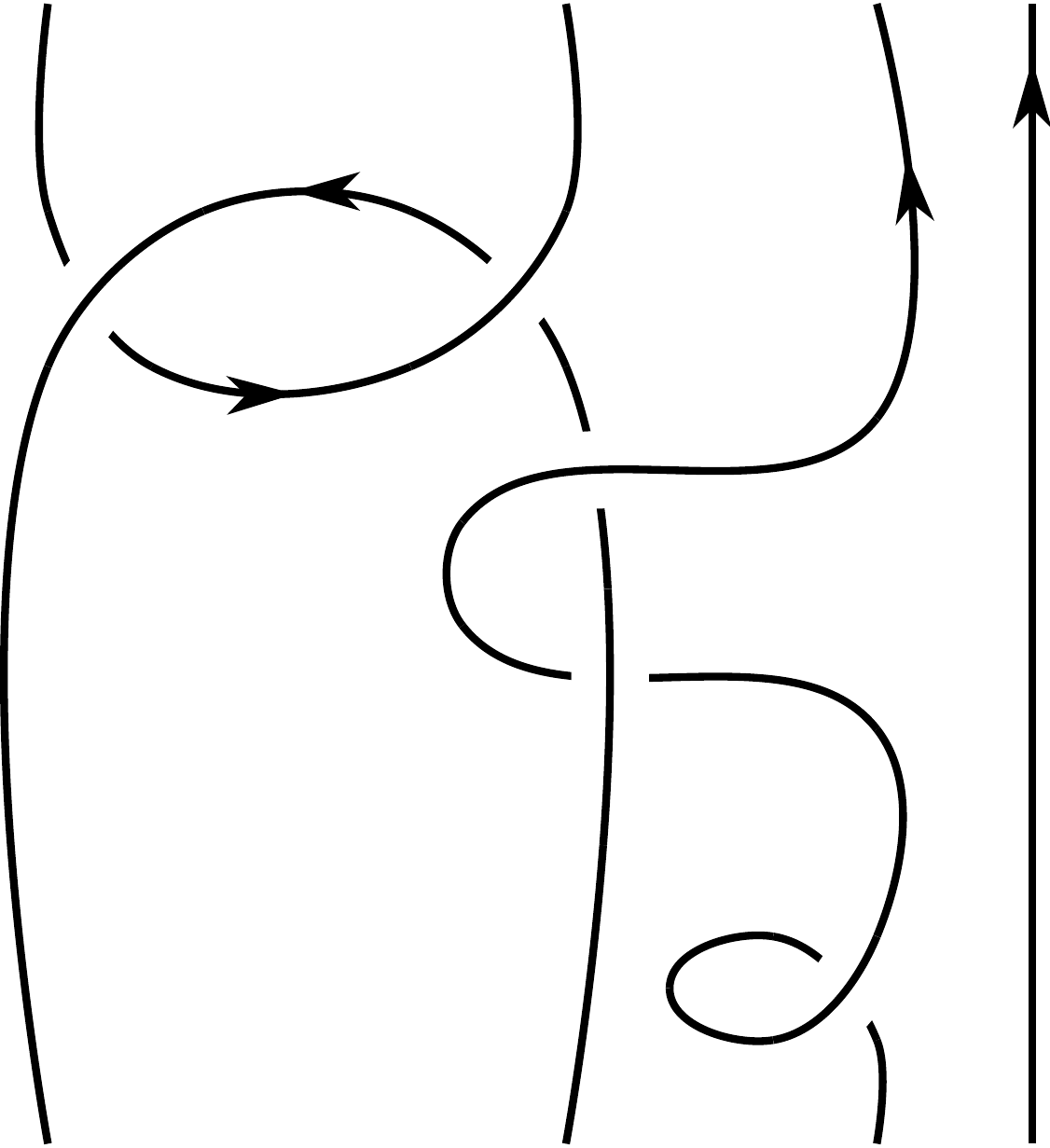}}\quad\quad
Y_{+,+}=\quad\raisebox{-4.5cm}{\includegraphics[height=4cm]{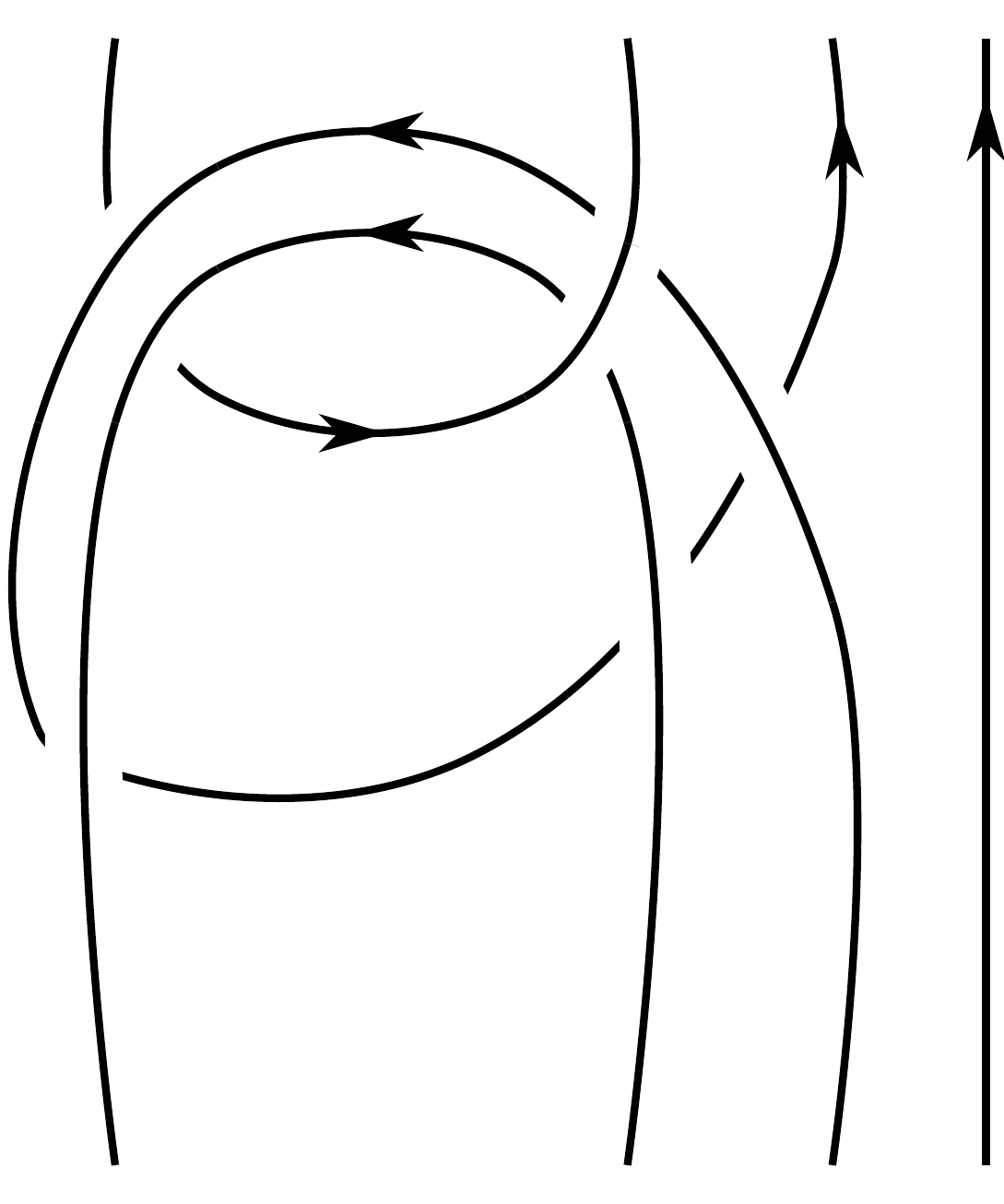}}$$

We also have:
$$X^{-1}_{+,+}=\quad\raisebox{-4.5cm}{\includegraphics[height=4cm]{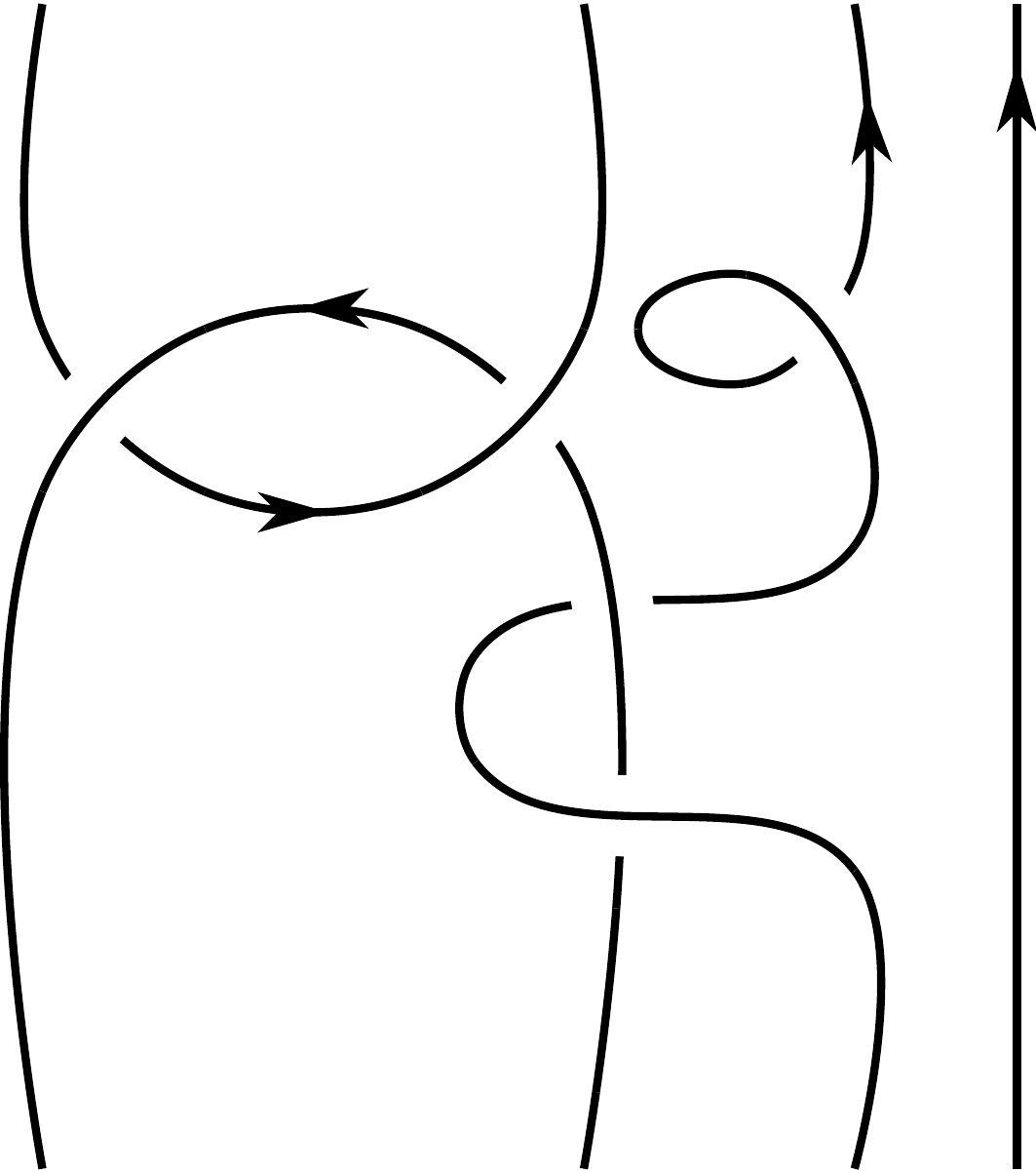}}\quad\quad
Y^{-1}_{+,+}=\quad\raisebox{-4.5cm}{\includegraphics[height=4cm]{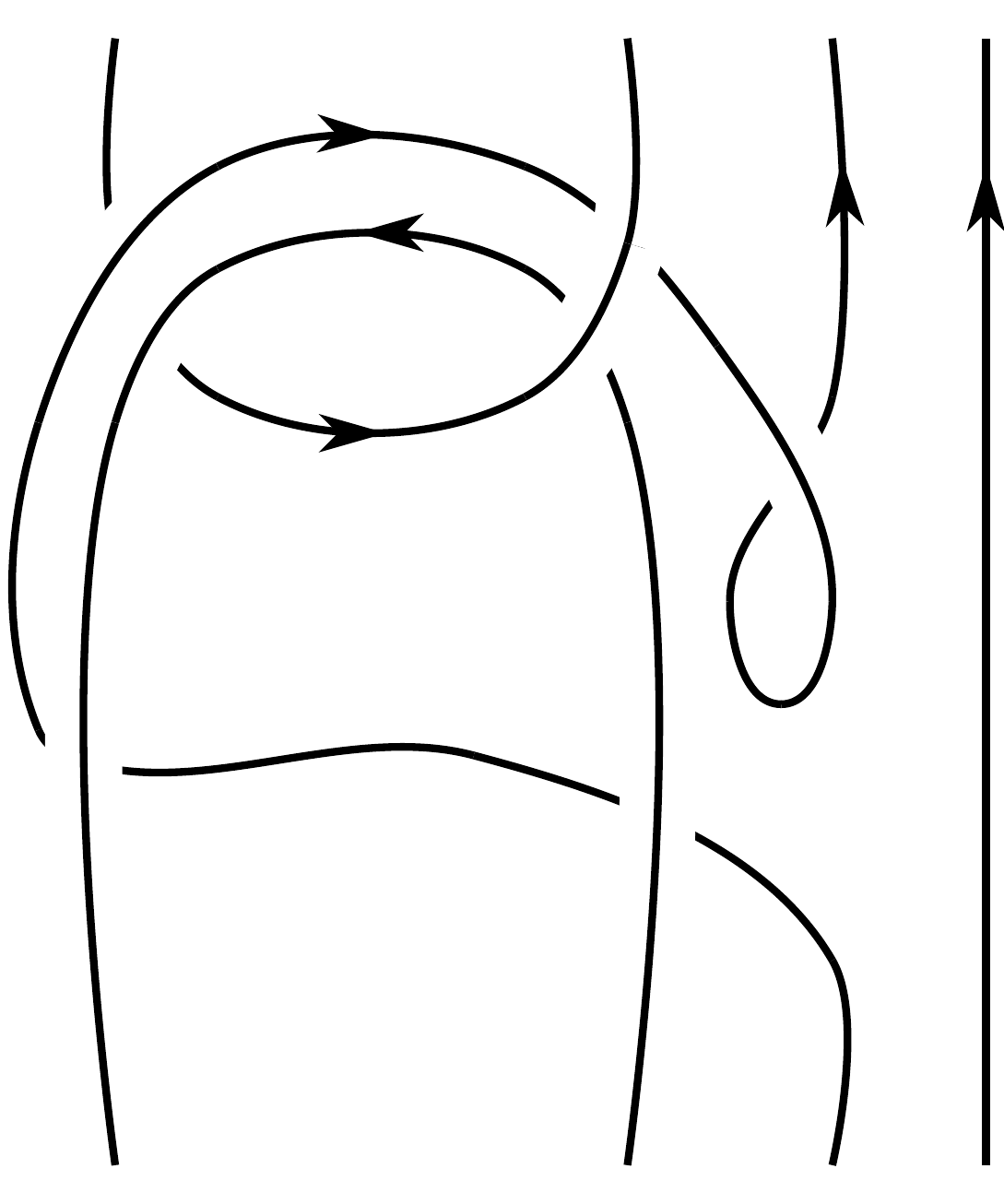}}$$

With this representation at hand we can verify the identities (\ref{elliptic1})-(\ref{elliptic4}) of definition \ref{def_elliptic} (we verify them only for $U=V=W=+$). In this verification we will use several times the ``slam-dunk" move (shown in figure \ref{slam_dunk}), which is implied by the Kirby moves. The equivalences which use this move are marked by a *.
$$(\ref{elliptic1})\quad X_{++,+}\quad=\quad\raisebox{-2.5cm}{\includegraphics[height=4cm]{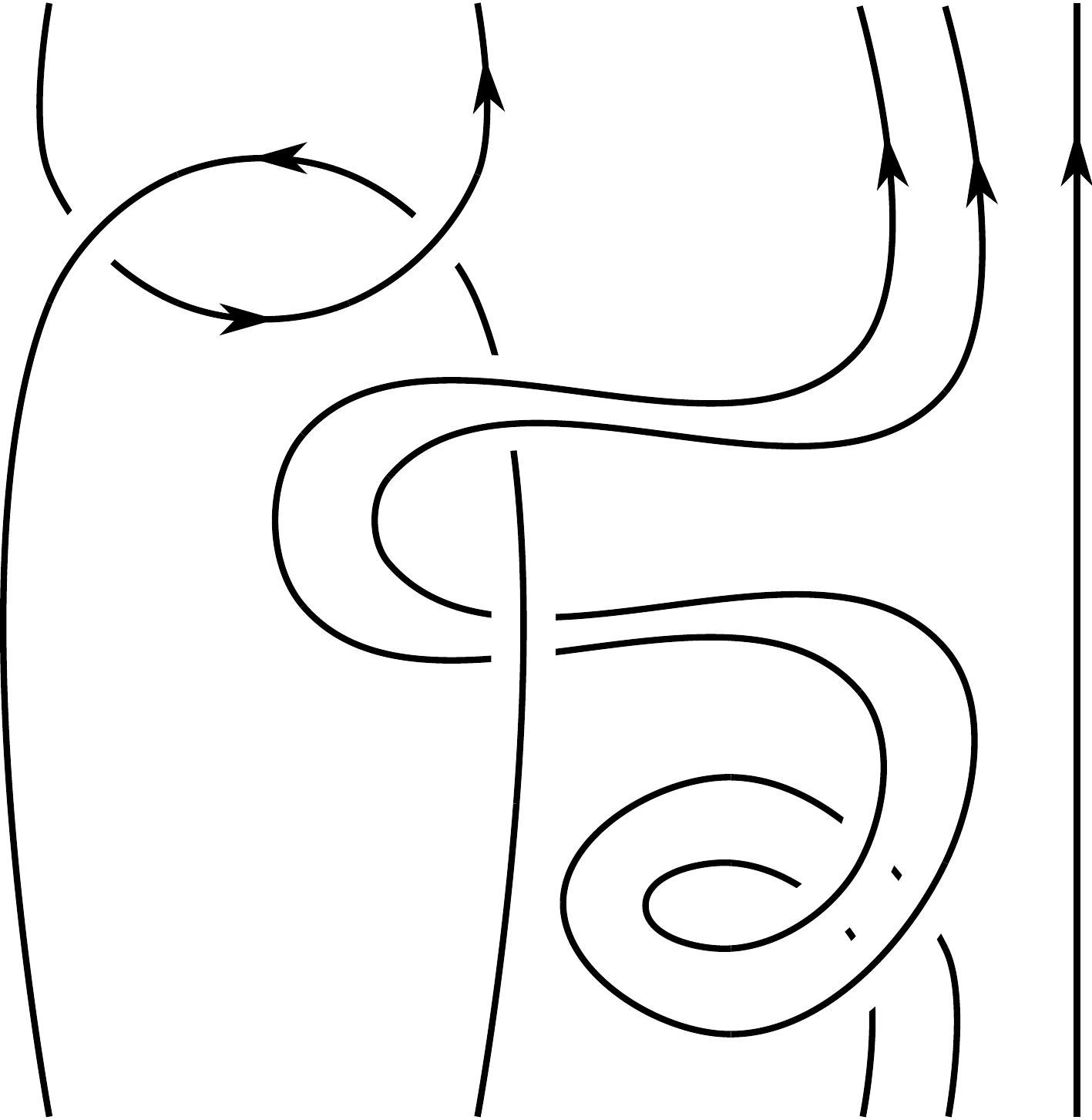}}\quad\approx\quad
\raisebox{-3.9cm}{\includegraphics[height=5.1cm]{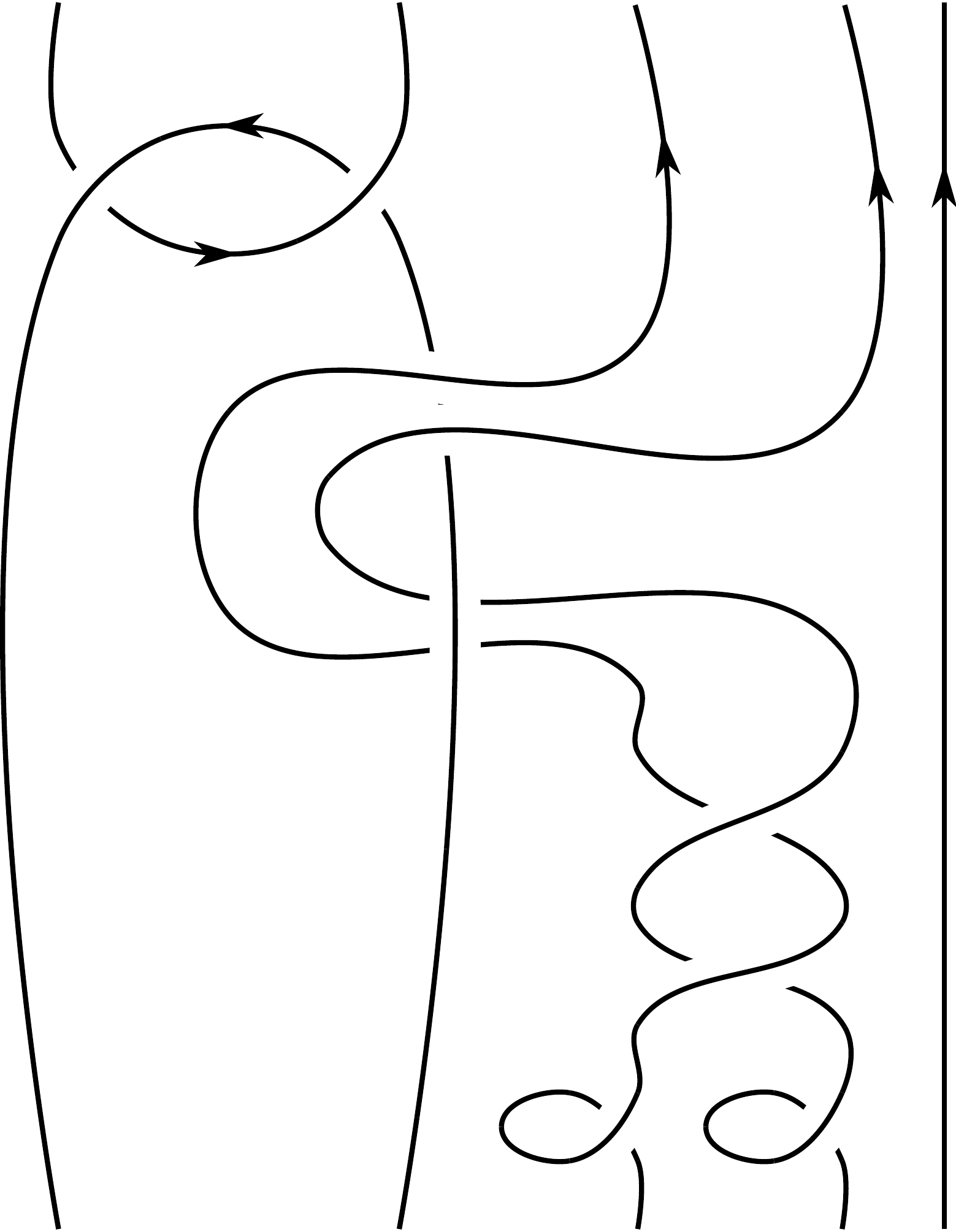}}\quad\approx$$

$$\approx\quad\raisebox{-4cm}{\includegraphics[height=6.4cm]{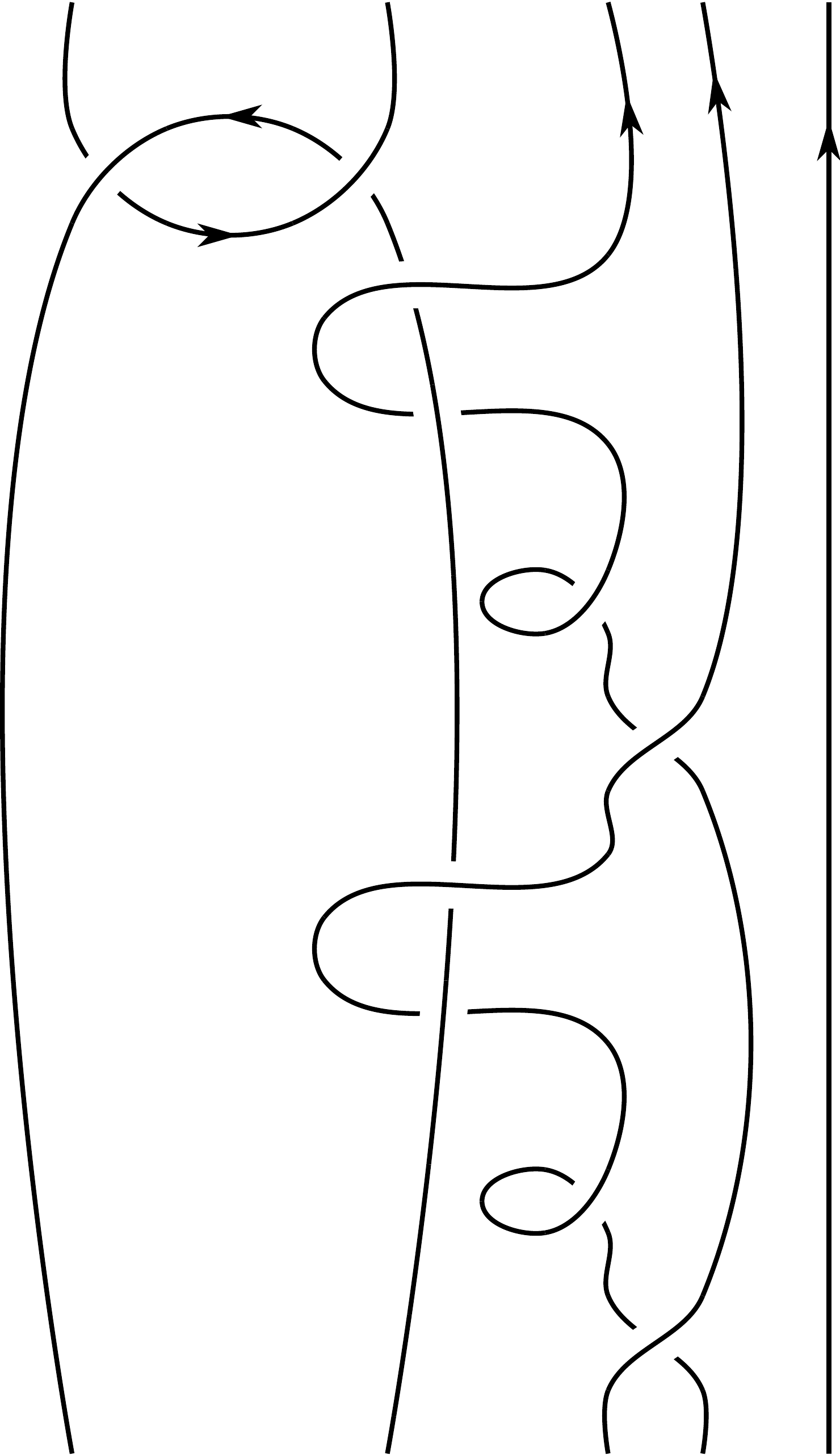}}\quad\stackrel{*}{\approx}\quad
\raisebox{-6cm}{\includegraphics[height=8cm]{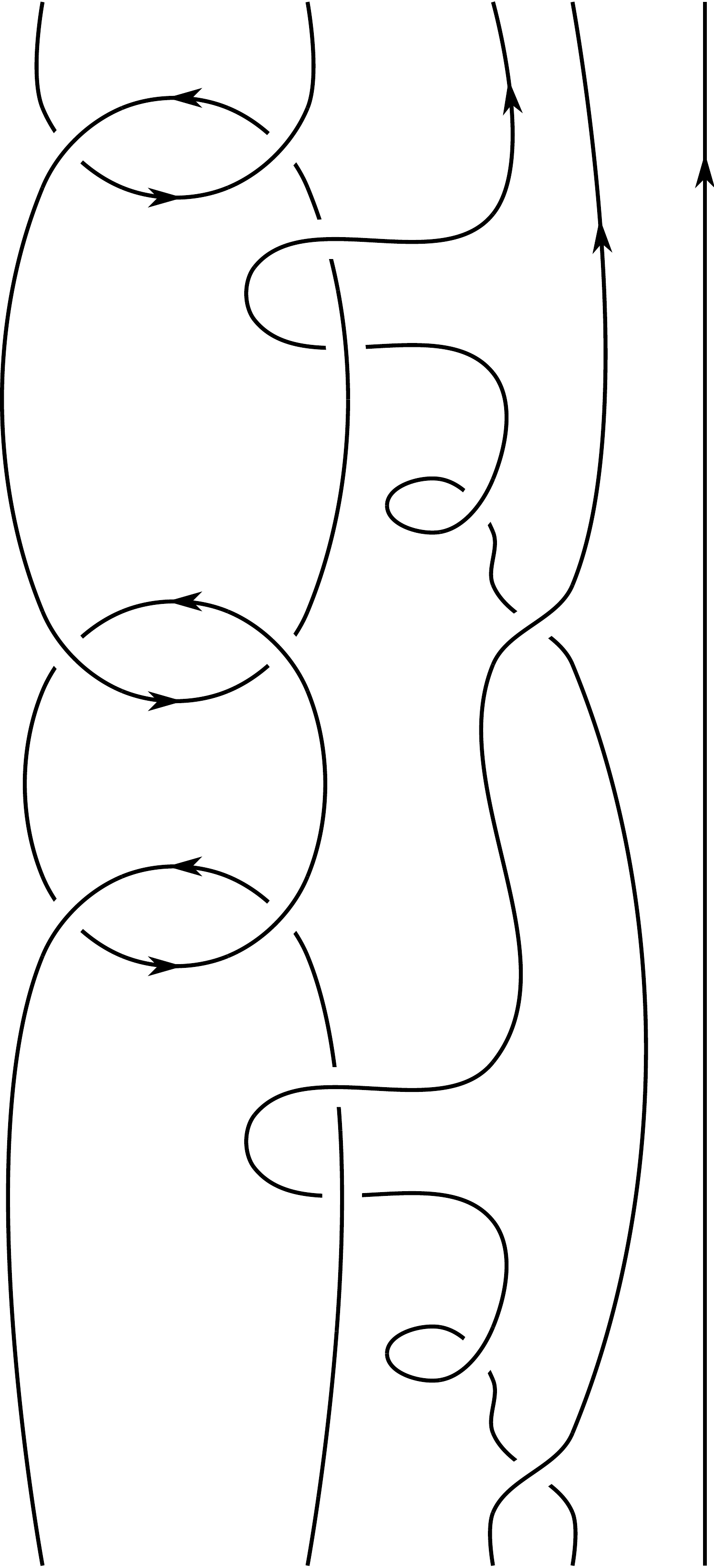}}\quad =\quad$$

$$=\quad X'_{+,++}\{c_{+,+}\otimes id_+\}X'_{+,++}\{c_{+,+}\otimes id_+\}$$

$$(\ref{elliptic2})\quad Y_{++,+}\quad=\quad\raisebox{-2.5cm}{\includegraphics[height=4cm]{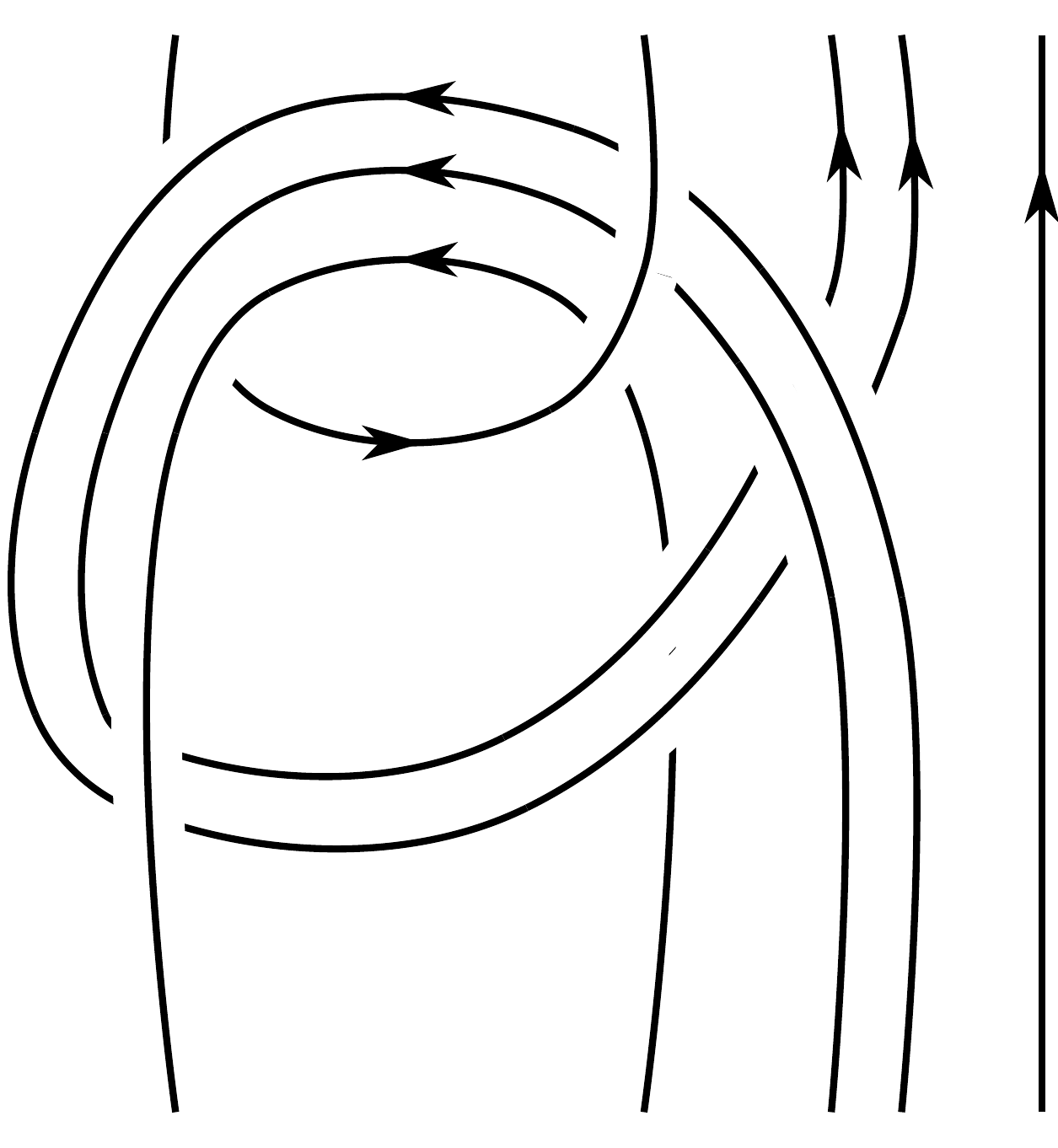}}\quad\approx\quad
\raisebox{-3.9cm}{\includegraphics[height=5.1cm]{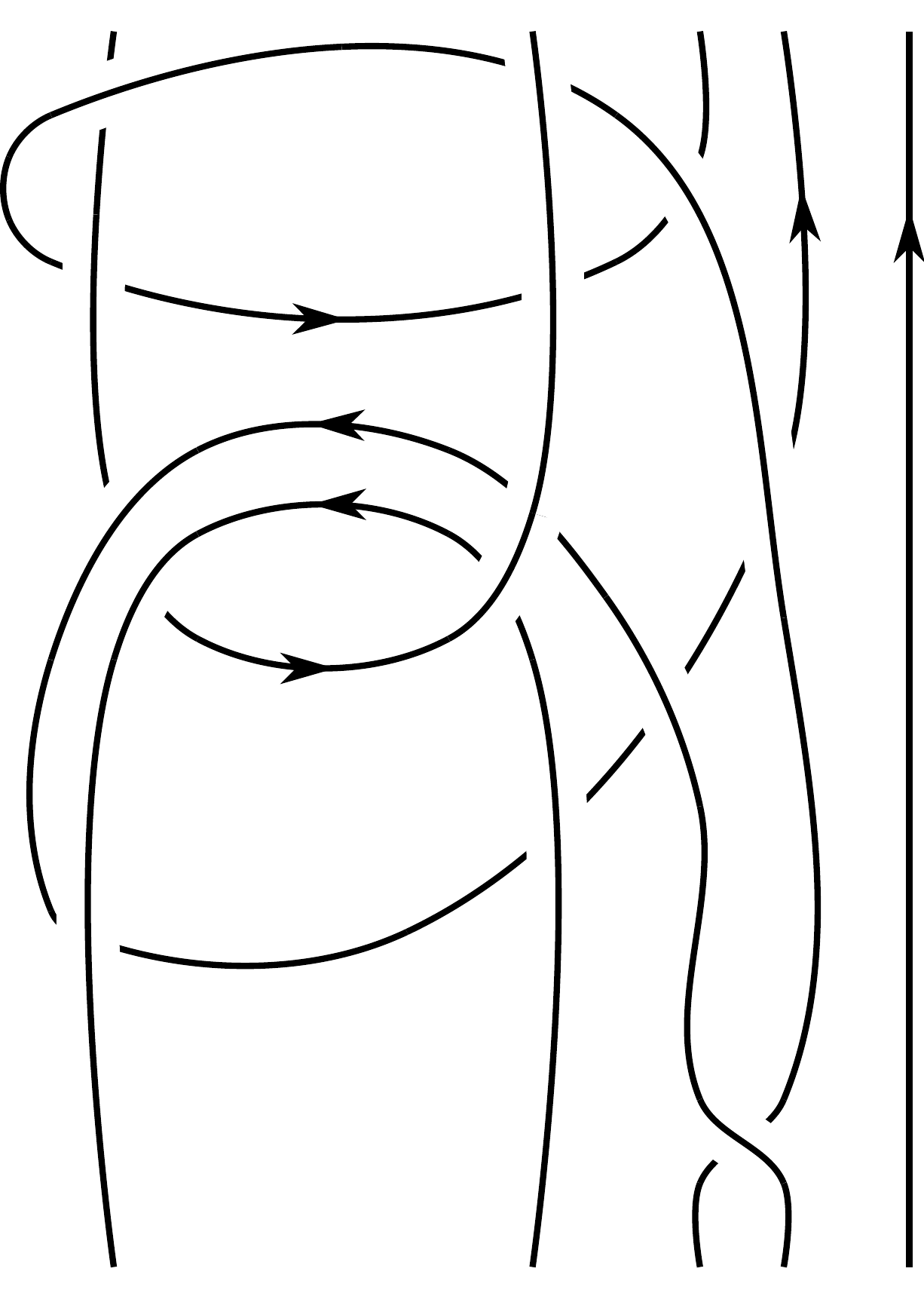}}\quad\stackrel{*}{\approx}$$

$$\stackrel{*}{\approx}\quad\raisebox{-4cm}{\includegraphics[height=6.4cm]{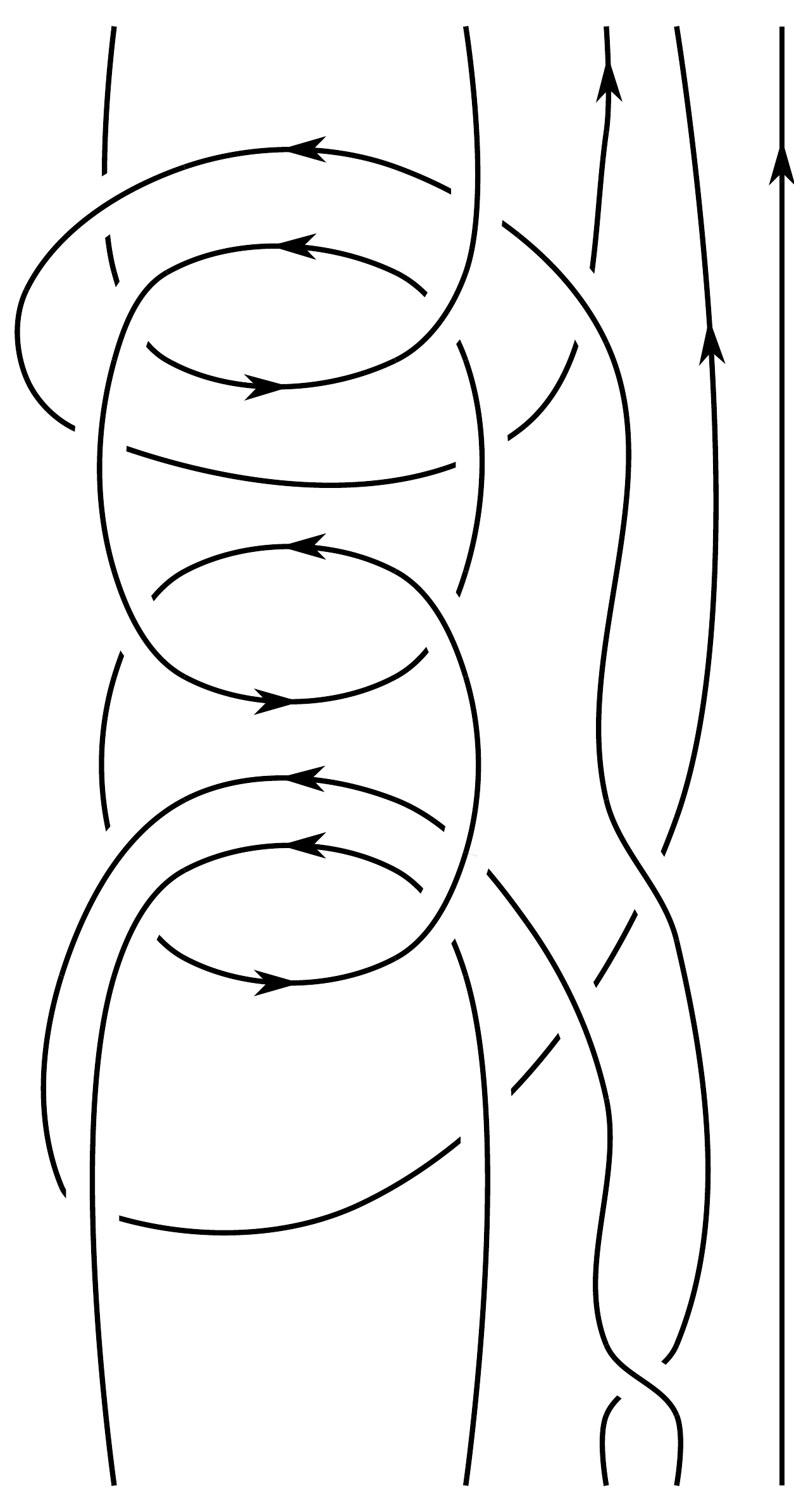}}\quad =\quad
 Y'_{+,++}\{c^{-1}_{+,+}\otimes id_+\}Y'_{+,++}\{c^{-1}_{+,+}\otimes id_+\}$$

$$(\ref{elliptic3})\quad Y_{+,+}X_{+,+}Y^{-1}_{+,+}X^{-1}_{+,+}\quad=$$

$$=\quad\raisebox{-11cm}{\includegraphics[height=12cm]{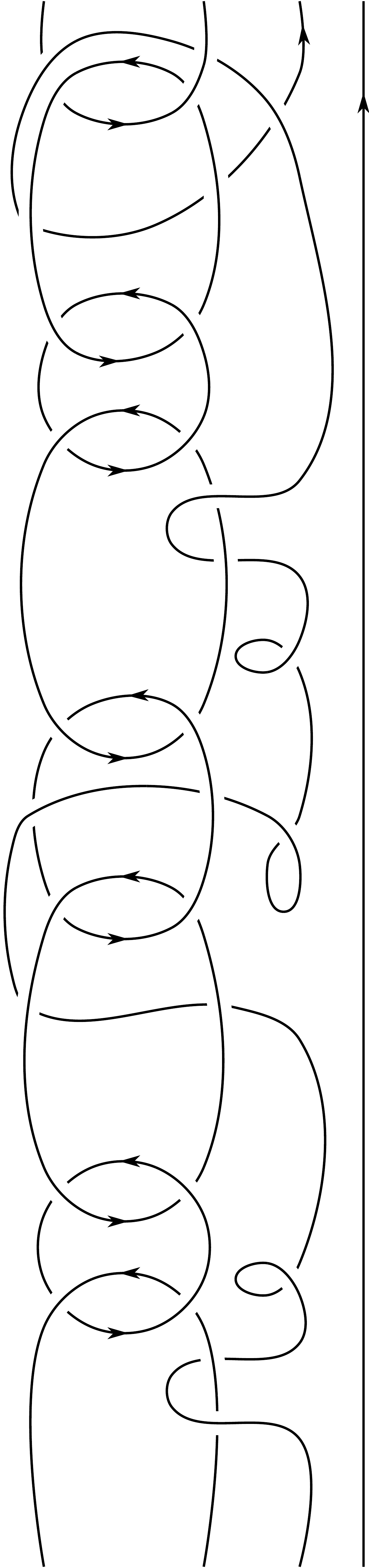}}\quad\stackrel{*}{\approx}\quad
\raisebox{-8cm}{\includegraphics[height=9.6cm]{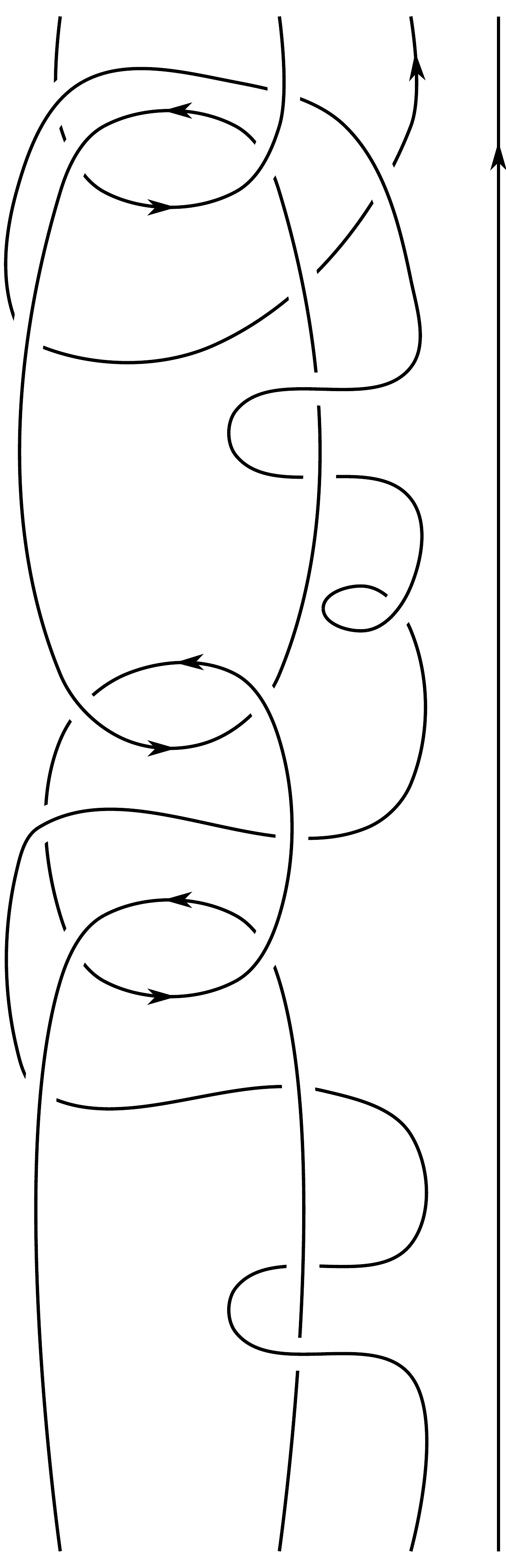}}\quad\stackrel{*}{\approx}
\raisebox{-3cm}{\includegraphics[height=4cm]{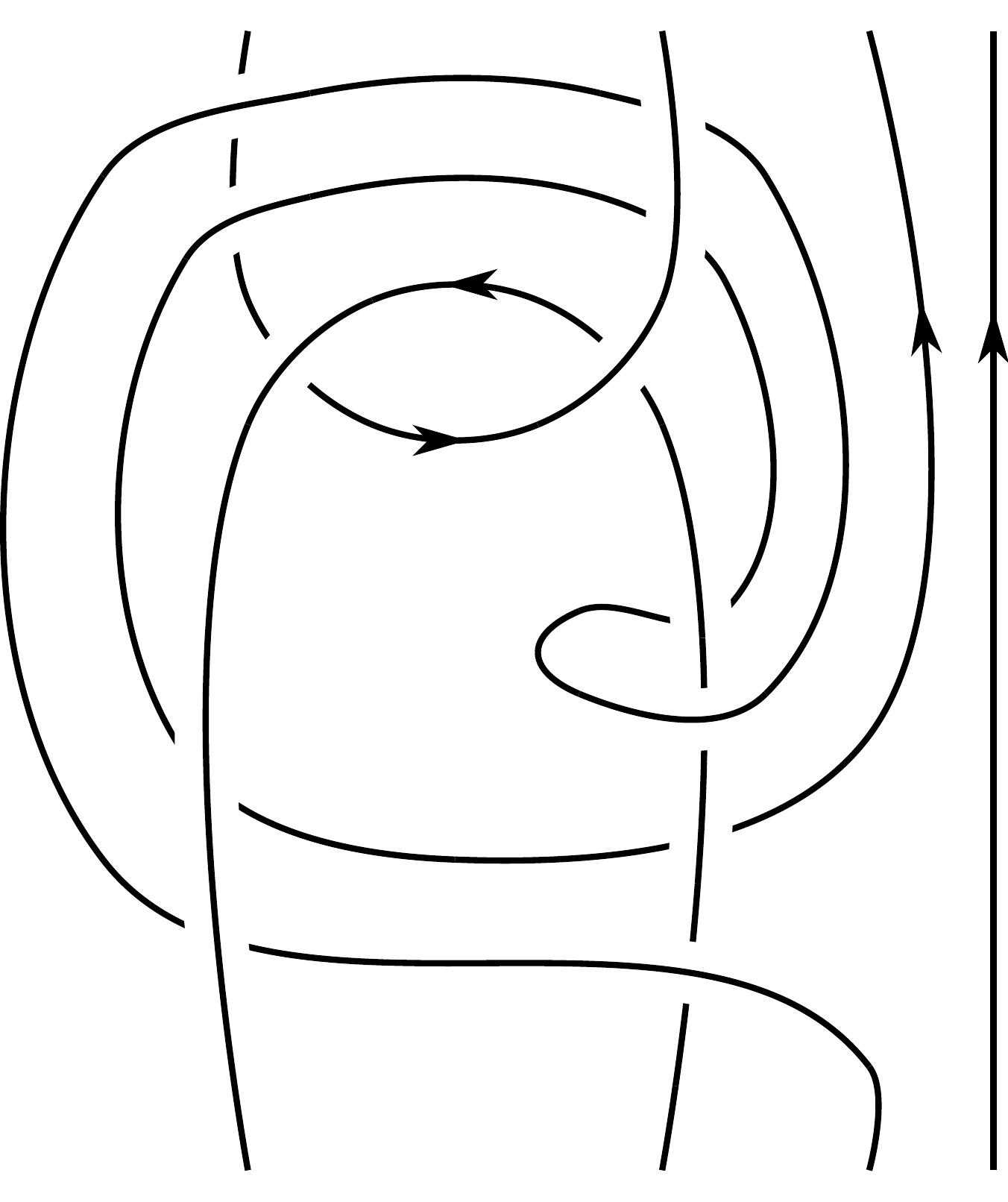}}\quad\approx$$

$$\approx\quad
\raisebox{-3cm}{\includegraphics[height=3.2cm]{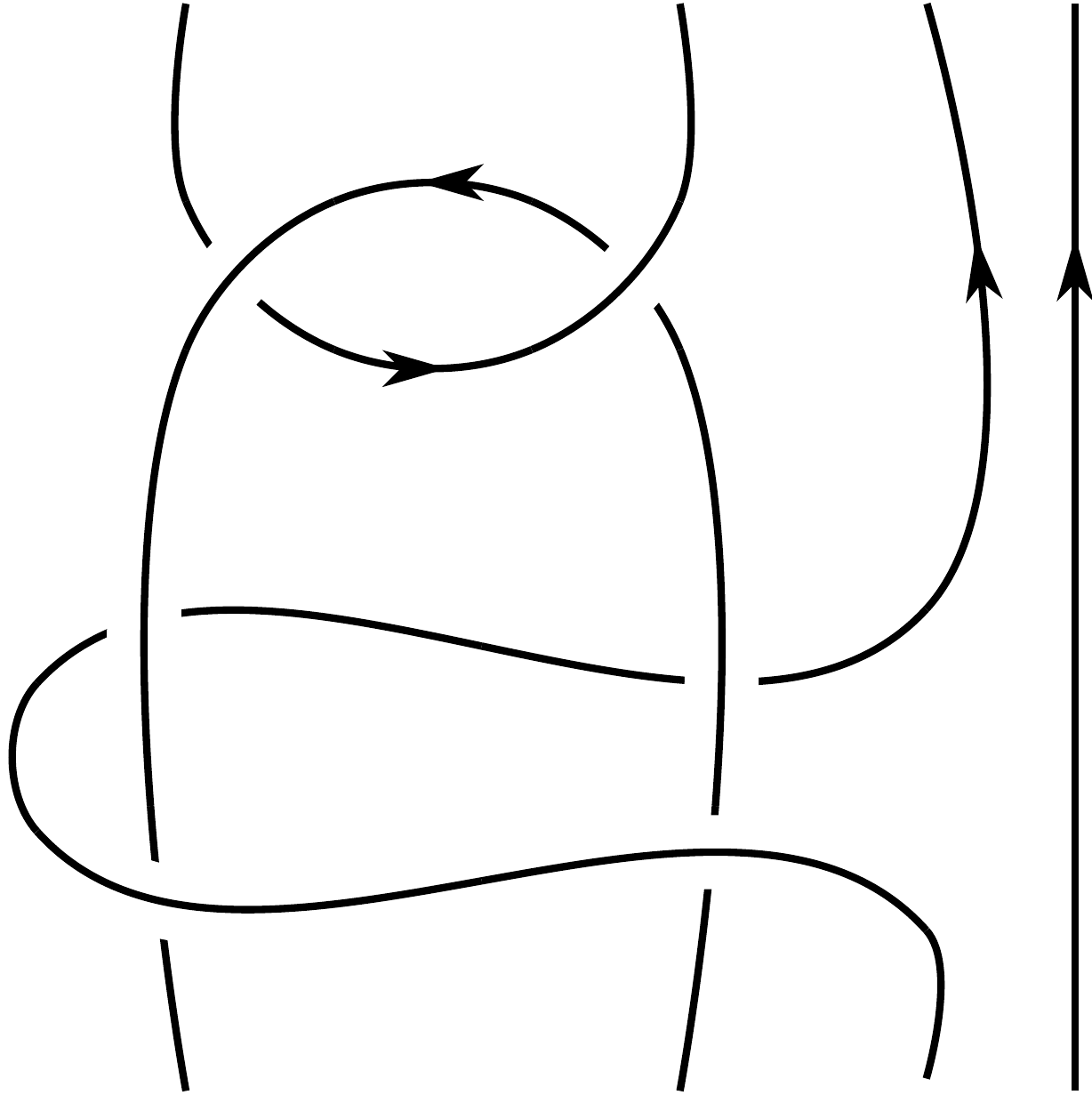}}\quad\approx\quad
\raisebox{-3cm}{\includegraphics[height=3.2cm]{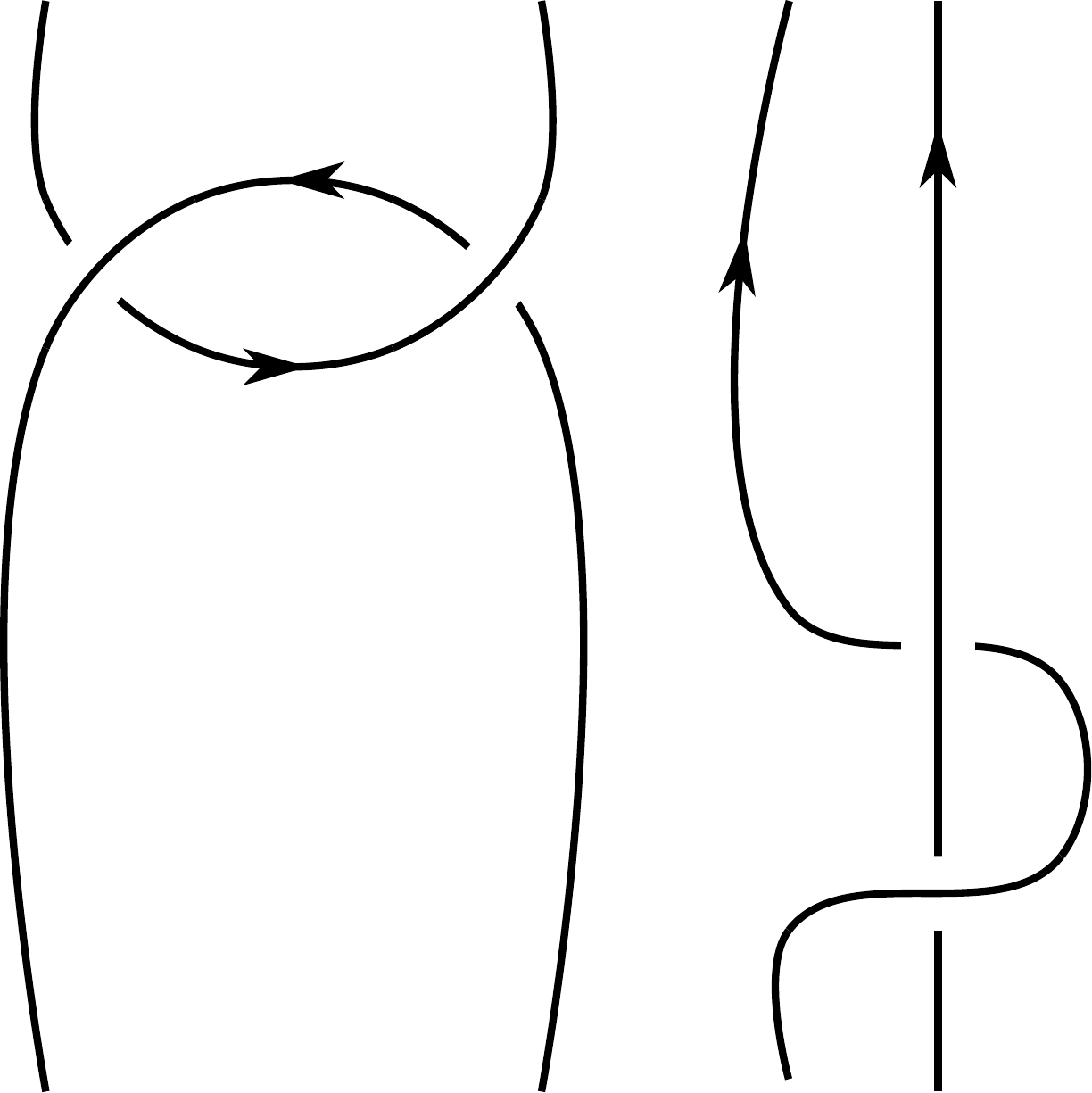}}\quad =\quad \{c_{+,+}c_{+,+}\}$$

$$(\ref{elliptic4})\quad \{c_{+,+}\otimes id_+\}X'_{+,++}\{c^{-1}_{+,+}\otimes id_+\}Y'_{+,++}\quad=$$

$$=\quad\raisebox{-4cm}{\includegraphics[height=6cm]{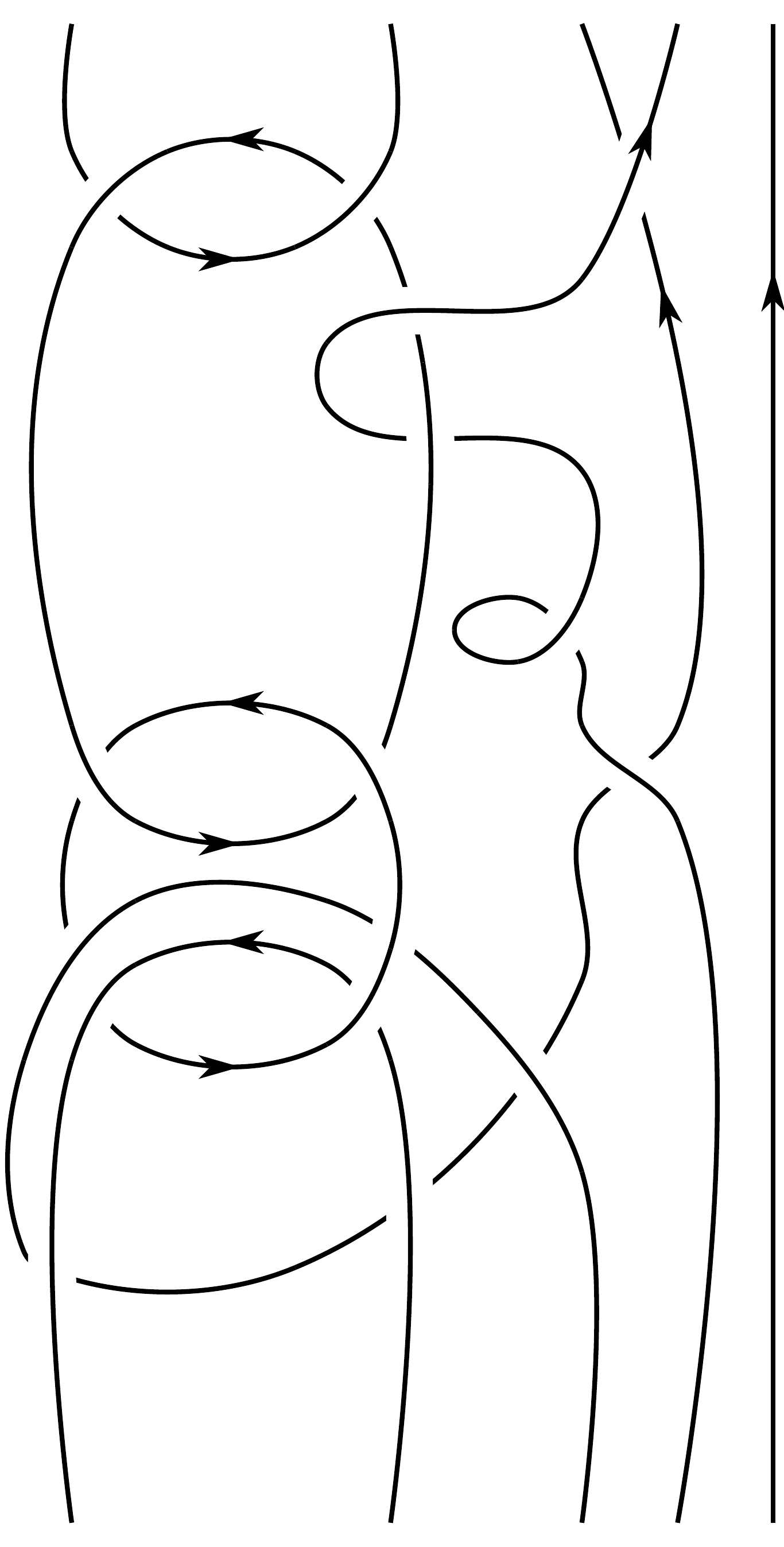}}\quad\stackrel{*}{\approx}\quad
\raisebox{-4cm}{\includegraphics[height=4.8cm]{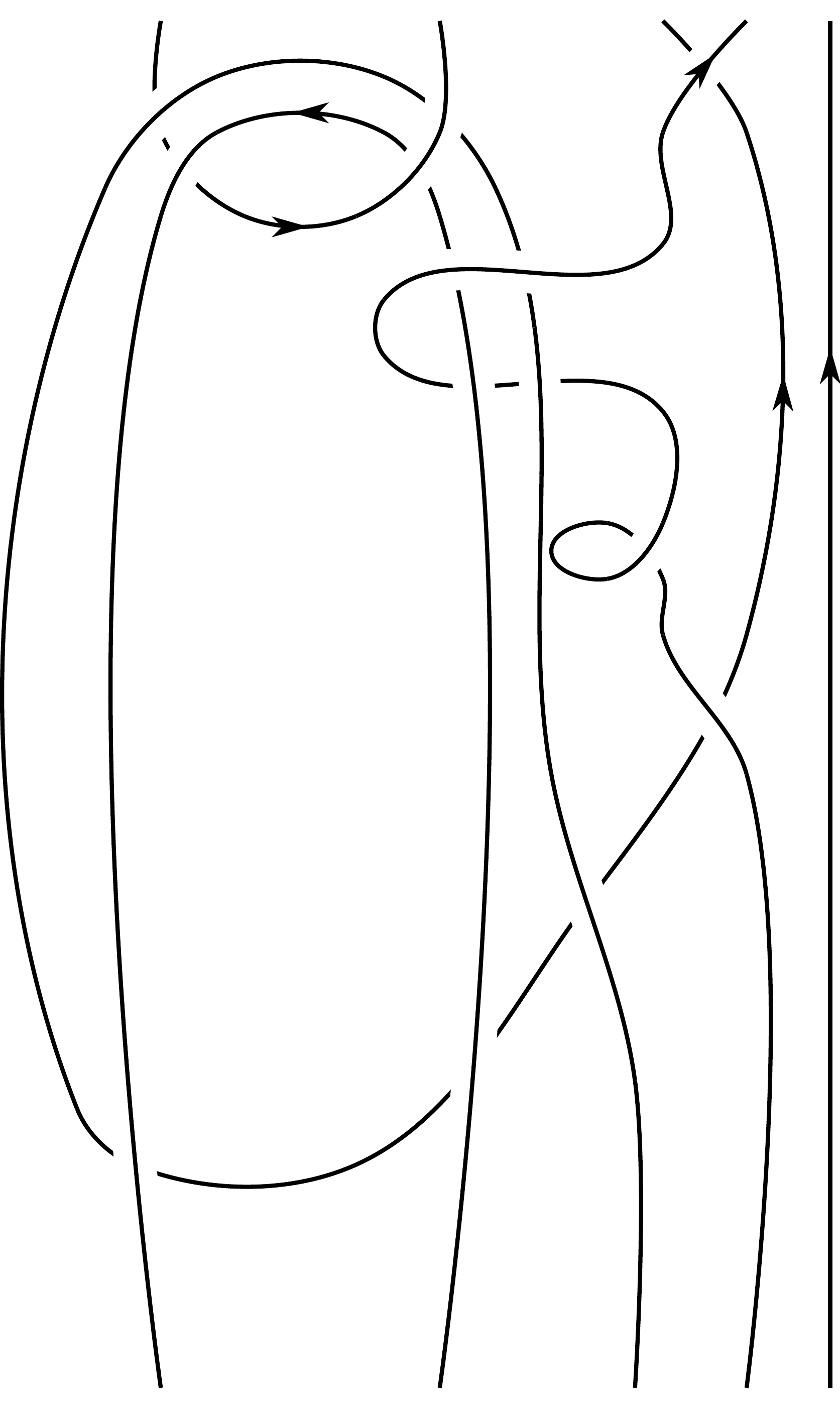}}\quad\approx$$

$$\approx\quad\raisebox{-4cm}{\includegraphics[height=4.8cm]{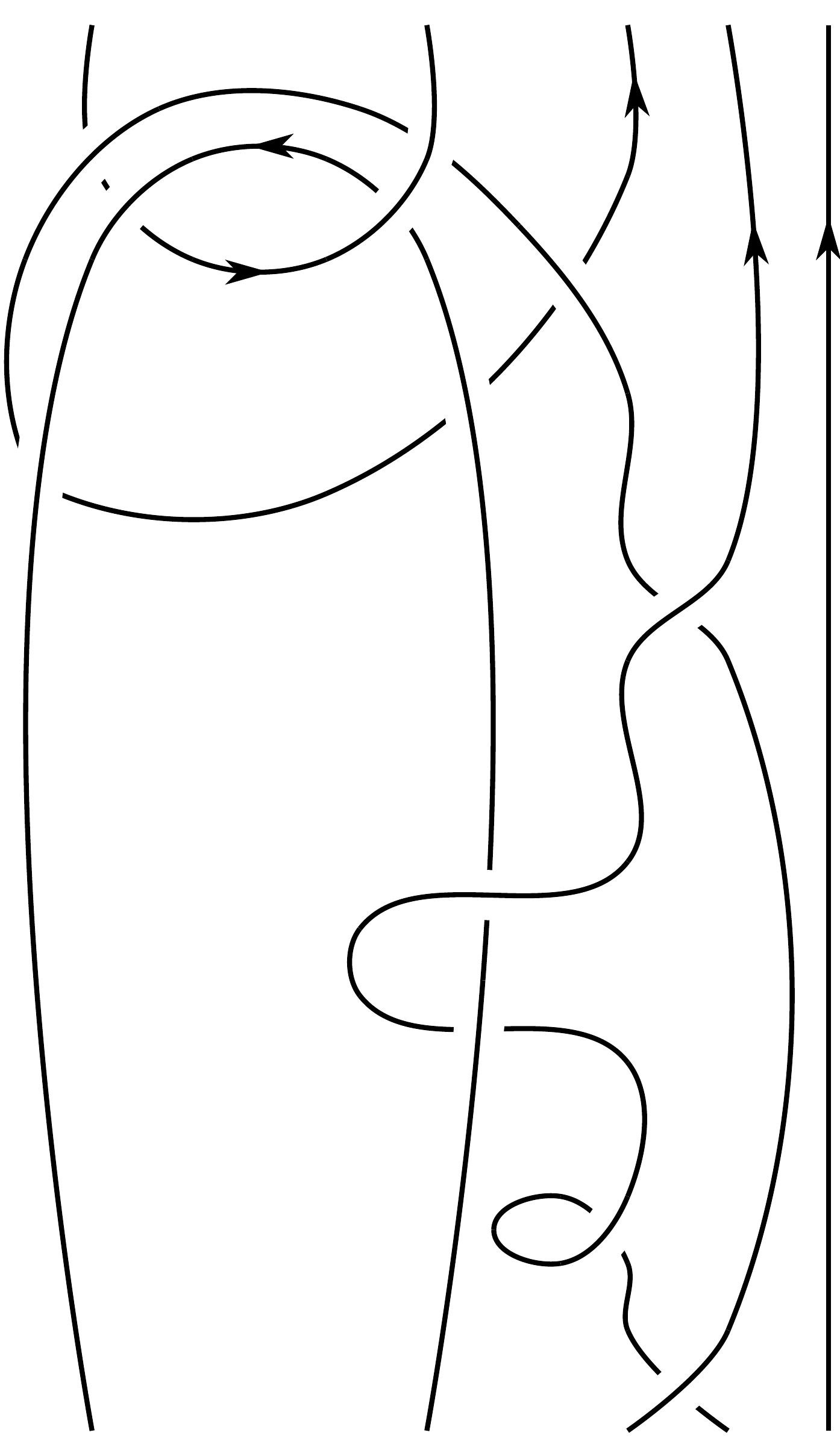}}\quad\stackrel{*}{\approx}\quad
\raisebox{-4cm}{\includegraphics[height=4.8cm]{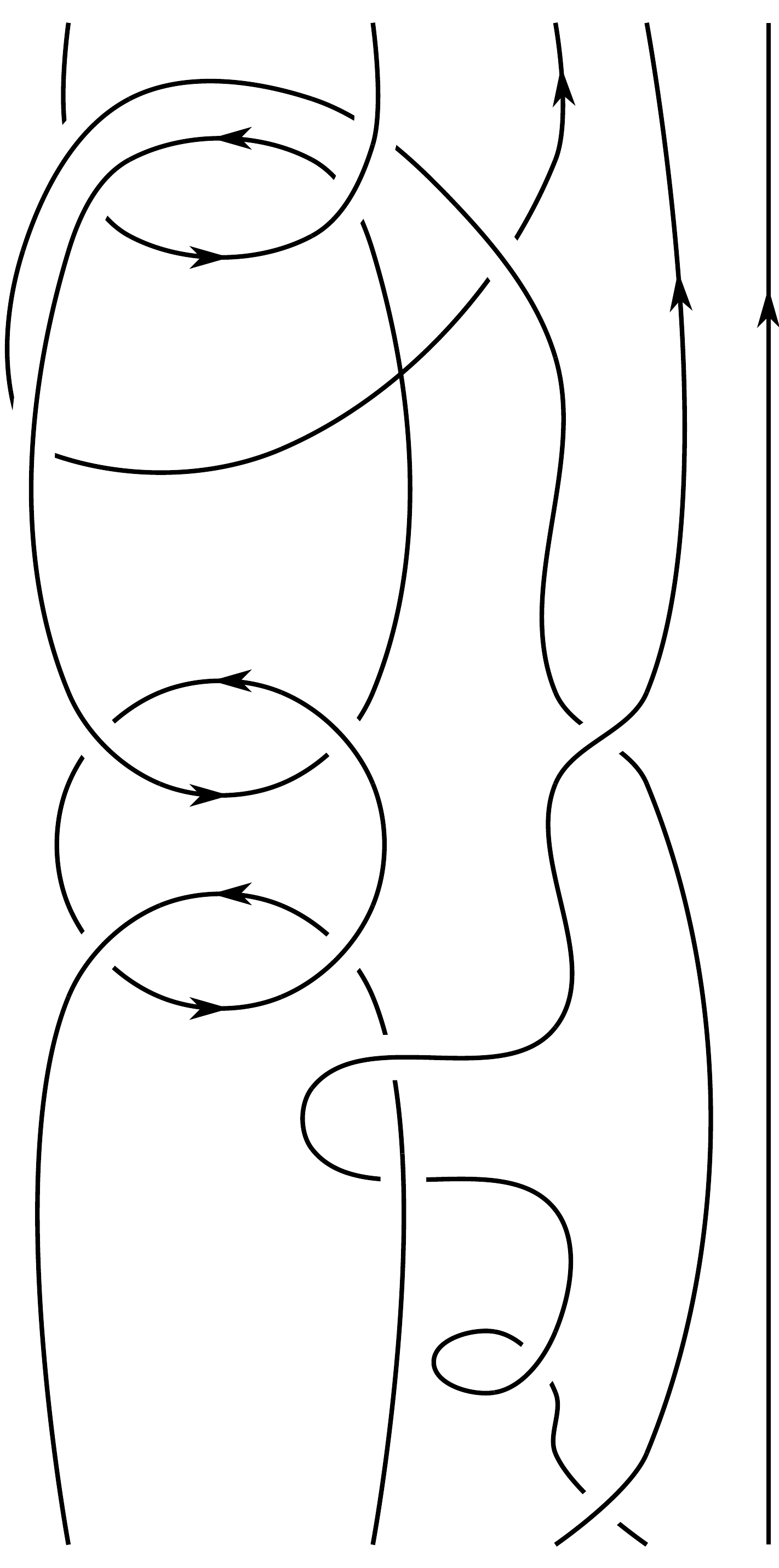}}\quad=$$

$$=\quad Y'_{+,++}\{c_{+,+}\otimes id_+\}X'_{+,++}\{c_{+,+}\otimes id_+\}$$

\begin{figure}[ht]

$$\raisebox{-2cm}{\includegraphics[height=4 cm]{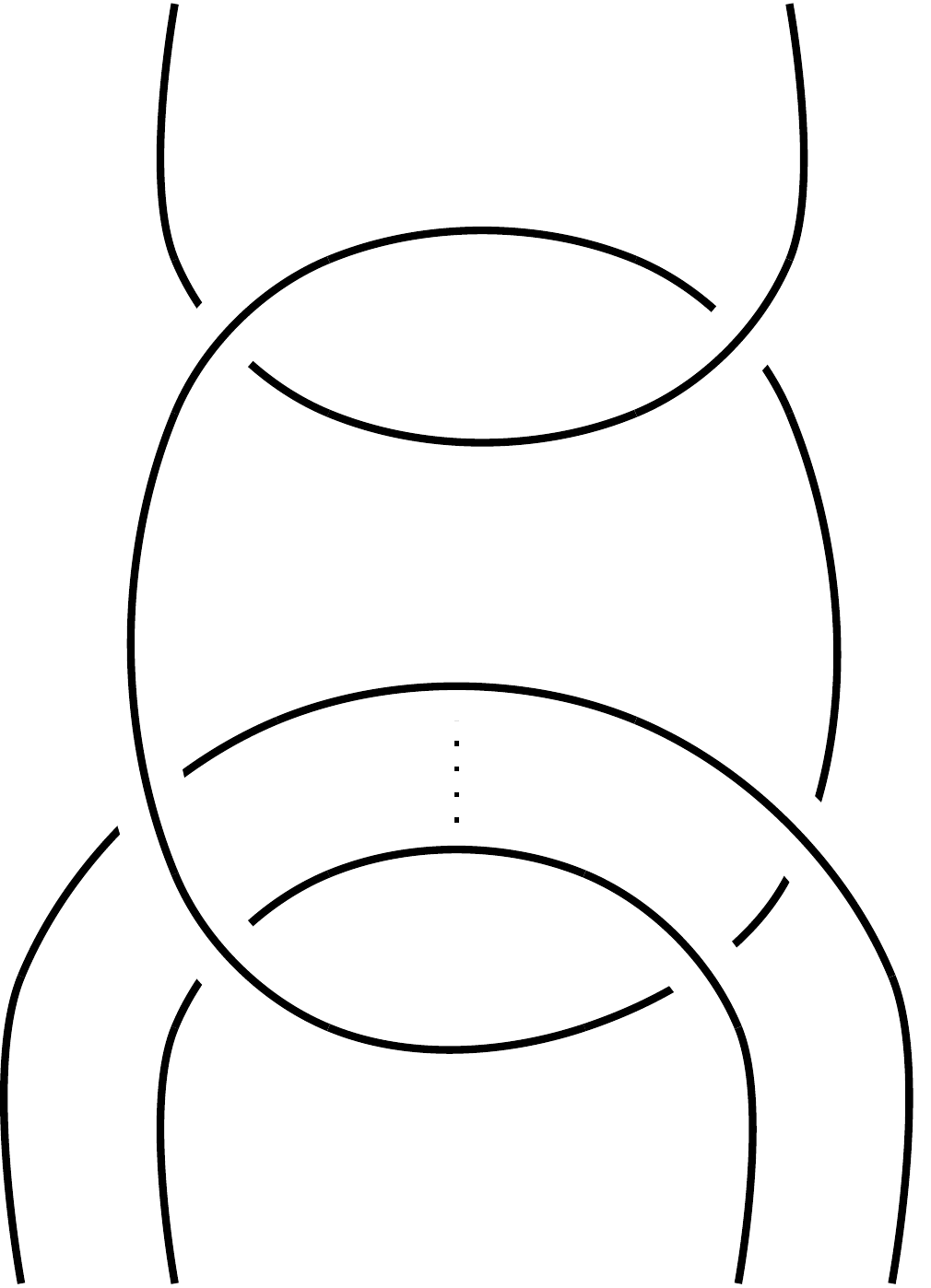}}\quad\approx\quad\raisebox{-2cm}{\includegraphics[height=4 cm]{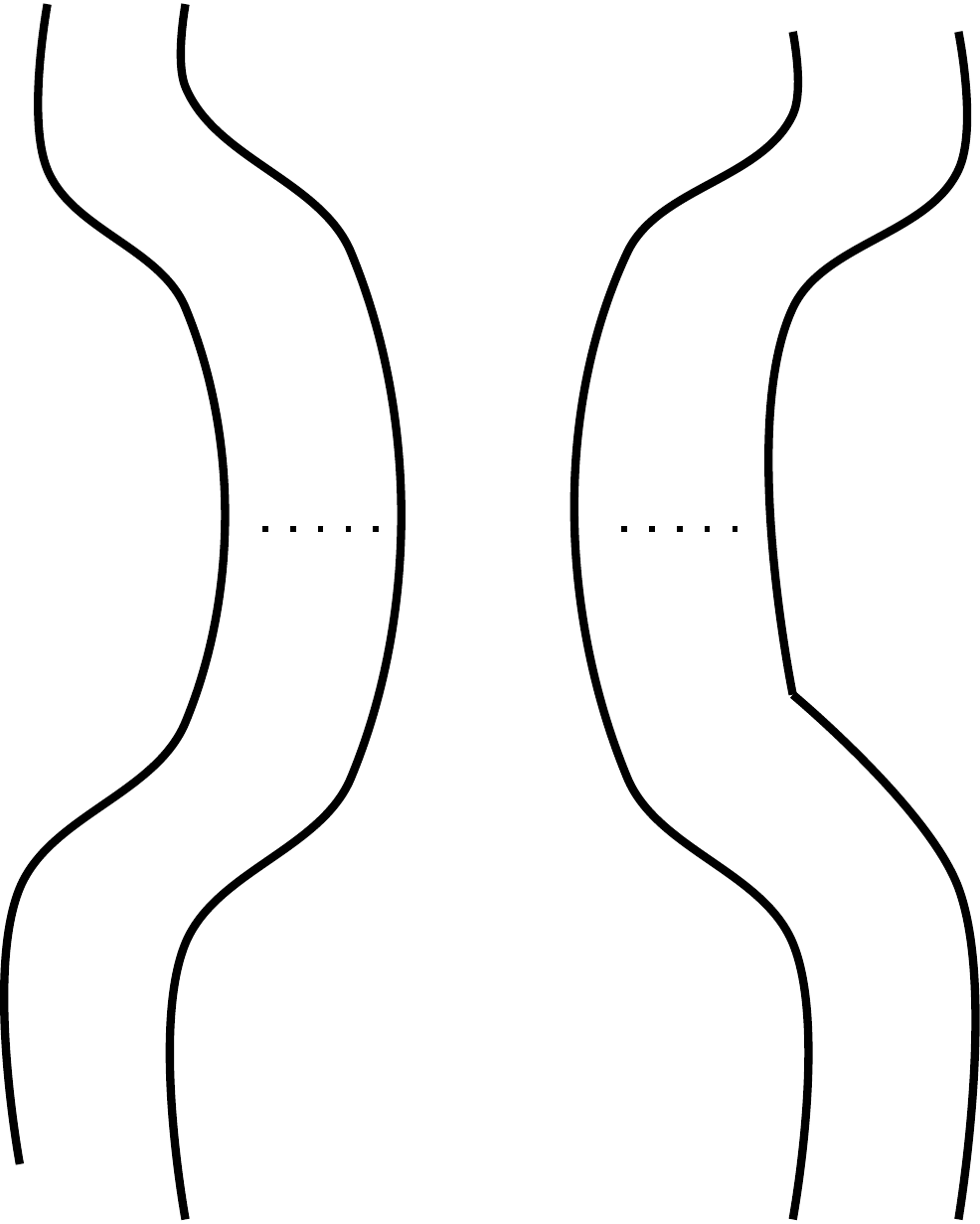}}$$

\caption{The ``slam-dunk" move}
\label{slam_dunk}
\end{figure}

Applying the maps $LMO^y$ and $LMO^<$ to $X_{+,+}$ and $Y_{+,+}$ we get elliptic structures relative to $\mathcal{A}^\partial\rightarrow\mathcal{A}_1^{yp}$ and $\mathcal{A}^\partial\rightarrow\mathcal{A}_1^{<p}$.
\label{elliptic_LMO}

\newpage
\section{The Lie Algebras $\textbf{t}_{1,n}$ and their Embeddings}
\label{section_t}

In the previous section we saw one way to define an elliptic structure relative to $\mathcal{A}^\partial\rightarrow\mathcal{A}_1^{<p}$, via the LMO functor. Another, more explicit, definition of an elliptic structure comes from specifying certain elements in the algebras $U\hat{\textbf{t}}_{1,n}$, and mapping them into $\mathcal{A}_1^{<p}$. In this section we will introduce those algebras and prove some propositions regarding their maps into $\mathcal{A}_1^{<p}$. The actual definition of the elliptic structure will appear in the next section.

\subsection{The Lie Algebras $\textbf{t}_{1,n}$}

\begin{definition} (\cite{Bezrukavnikov})
Let $\textbf{t}_{1,n}$ be the graded Lie algebra generated by $x_i$, $y_i$ ($1\le i\le n$) in degree $1$ and $t_{ij}$ ($1\le i,j\le n$, $i\neq j$) in degree $2$, with the relations:
$$[v_i,w_j]=<v,w>t_{ij}$$
$$[v_i,t_{jk}]=0$$
$$[x_i,y_i]=-\sum_{j\ne i} t_{ij}$$
where $1\le i,j,k\le n$ are distinct indices, $v,w\in\{x,y\}$, and $<\cdot,\cdot>$ is the intersection form of $H_1(\mathbb{T})$ (the symbols $x$ and $y$ are considered to be the generators of $H_1(\mathbb{T})$ from figure \ref{torus}).
$U\hat{\textbf{t}}_{1,n}$ is the degree completion of the universal enveloping algebra of $\textbf{t}_{1,n}$.
\end{definition}

Denote by $\uparrow^n$ the pattern in $\mathbb{P}(\underbrace{++\cdots+}_{\text{$n$ times}},\underbrace{++\cdots+}_{\text{$n$ times}})$ composed of $n$ up-going strands. There is a map $u_n:U\hat{\textbf{t}}_{1,n}\rightarrow\mathcal{A}_1^{<p}(\uparrow^n)$ defined by:
$$v_i\stackrel{u_n}{\longmapsto} \quad v\,\mathfigns{2}{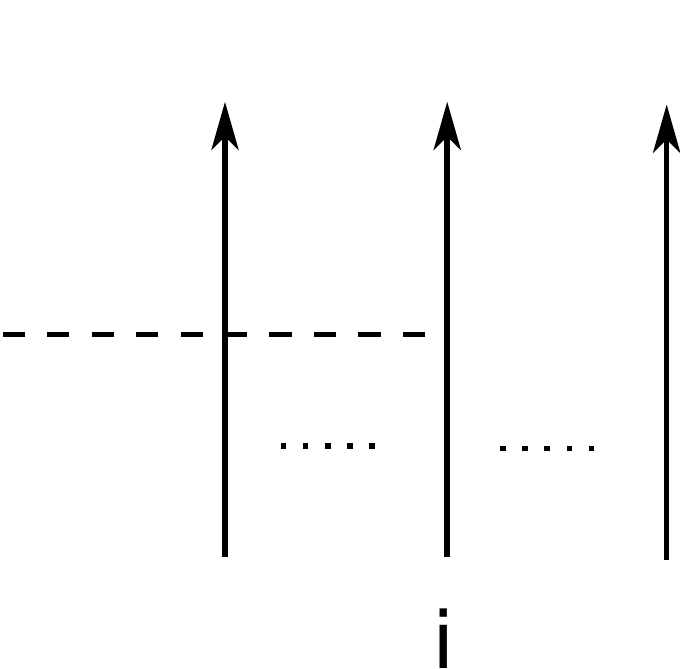}\quad$$
$$t_{ij}\stackrel{u_n}{\longmapsto}\mathfig{2}{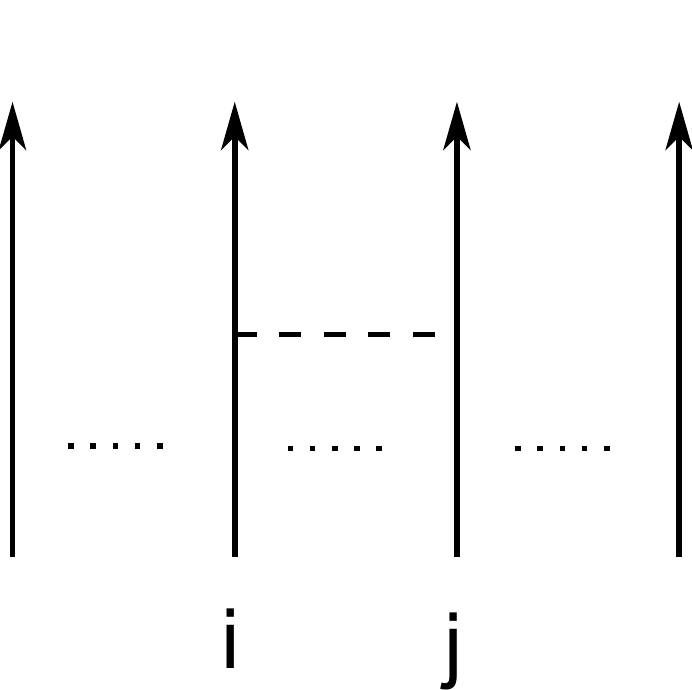}$$
It is easily verified that the defining relations of $\textbf{t}_{1,n}$ are mapped to relations in $\mathcal{A}_1^{<p}(\uparrow^n)$. $u_n$ is an algebra homomorphism, when $\mathcal{A}_1^{<p}(\uparrow^n)$ is considered as an algebra via the composition of $\mathcal{A}_1^{<p}$.

In the next section we will be interested in the subalgebra of $U\hat{\textbf{t}}_{1,n}$ generated by $\{x_i,y_i\, |\, 1\le i\le n-1\}$. Denote this subalgebra by $RU\hat{\textbf{t}}_{1,n}$ ($R$ for ``restricted"). We will want to show that for $n=2$ and $n=3$, $u_n |_{RU\hat{\textbf{t}}_{1,n}}$ is injective. In fact we will prove that for those values of $n$, $u_n$ is an isomorphism onto a certain quotient of a restriction of $\mathcal{A}_1^{<p}(\uparrow^n)$, as will be explained in section \ref{section_ordered}. In order to prove this theorem, we will define several spaces of Jacobi diagrams which are generated by less diagrams and less relations than $\mathcal{A}_1^{<p}(\uparrow^n)$. The following diagram summarizes the spaces and maps which we will encounter in this section. Note that we should have added a subscript $n$ to all the maps in order to specify the number of strands, but we omit this subscript to make the notations simpler:
\resizebox{1\textwidth}{!}{
$$\xymatrix{
 & RU\hat{\textbf{t}}_{1,n}\ar[lddd]^{\tilde{u}_f}\ar[rddd]^{\tilde{u}_s}\ar[rrr]\ar[ddd]^{\tilde{u}_o}\ar[rdd]^{u_s}\ar[ddrr]^{u_r} & & & U\hat{\textbf{t}}_{1,n}\ar[dd]^u\\
 \\
 &  & SR\mathcal{A}_1^{<p}(\uparrow^n)\ar[r]^s\ar[d]^{\pi_s} & R\mathcal{A}_1^{<p}(\uparrow^n)\ar[r]^r\ar[l]^\gamma\ar[d]^{\pi_r} & \mathcal{A}_1^{<p}(\uparrow^n)\ar[d]^\pi\\
FOSR\mathcal{A}_1^{<p}(\uparrow^n)\ar[r]^f\ar[uuur]^p & OSR\mathcal{A}_1^{<p}(\uparrow^n)\ar[r]^o\ar[l]^\alpha & SR\mathcal{A}_1^{<p}(\uparrow^n)/H_n\ar[r]^{\tilde{s}}\ar[l]^\beta & R\mathcal{A}_1^{<p}(\uparrow^n)/H_n\ar[r]^{\tilde{r}}\ar[l]^{\tilde{\gamma}} & \mathcal{A}_1^{<p}(\uparrow^n)/H_n 
}$$
}
\label{restrictions_diagram}

Some of the spaces and maps in this diagram are defined for any $n$, while others are defined only for $n=2$ or $n=3$, as will be clear in the following subsections.

\subsection{Restriction to the First $n-1$ Strands}
\label{section_restricted}
Let $R\mathbb{D}_1^{<p}(\uparrow^n)\subset\mathbb{D}_1^{<p}(\uparrow^n)$ be the subset of all diagrams with no vertices on the rightmost strand. Let $R\mathcal{A}_1^{<p}(\uparrow^n)$ be the quotient of $S_\mathbb{F}(R\mathbb{D}_1^{<p}(\uparrow^n))$ by all STU and STU-like relations which are contained in $S_\mathbb{F}(R\mathbb{D}_1^{<p}(\uparrow^n))$. There is an obvious map $r:R\mathcal{A}_1^{<p}(\uparrow^n)\rightarrow \mathcal{A}_1^{<p}(\uparrow^n)$. However, it is not a-priori clear that this map is injective, because in $\mathcal{A}_1^{<p}(\uparrow^n)$ there are $I_1^{<}$ relations, which relate elements from $R\mathbb{D}_1^{<p}(\uparrow^n)$ to elements outside this subset. The injectivitiy of $r$ is the goal of this subsection.

\begin{proposition}
\label{proposition_r}
$r:R\mathcal{A}_1^{<p}(\uparrow^n)\rightarrow \mathcal{A}_1^{<p}(\uparrow^n)$ is injective.
\end{proposition}

We will prove this proposition by representing $\mathcal{A}_1^{<p}(\uparrow^n)$ in a different way, which does not involve $I_1^<$ relations. However, we start by finding such representation for $\mathcal{A}_1^{yp}(\uparrow^n)$, and then we will use the isomorphism $k:\mathcal{A}_1^{<p}\rightarrow\mathcal{A}_1^{yp}$ to conclude.

\begin{definition}
Let $D\in \mathbb{D}_1^{yp}(\uparrow^n)$ be a diagram, and let $v$ be a vertex in $D$ on the rightmost strand. We call $v$ a \textbf{lonely vertex} if it belongs to a component of $D$ which does not have more vertices on the pattern.  Otherwise we call it a \textbf{non-lonely vertex}.

Let $L\mathbb{D}_1^{yp}(\uparrow^n)\subset\mathbb{D}_1^{yp}(\uparrow^n)$ be the subset of diagrams in which all the vertices on the rightmost strand are lonely vertices. Let $L\mathcal{A}_1^{yp}(\uparrow^n)$ denote the quotient of $S_\mathbb{F}(L\mathbb{D}_1^{yp}(\uparrow^n))$ by all STU relations contained in it (there are no $I_1^y$ relations in this subspace).
\end{definition}
\begin{proposition}
The obvious map $l:L\mathcal{A}_1^{yp}(\uparrow^n)\rightarrow \mathcal{A}_1^{yp}(\uparrow^n)$ is an isomorphism.
\end{proposition}
\begin{proof}
Let $D\in\mathbb{D}_1^{yp}(\uparrow^n)$. For any non-lonely vertex $v$ of $D$, let $n_l(v)$ be the total degree of all components with a lonely vertex higher than $v$. Let $n_l(D):=(\sum\limits_{\text{$v$ non-lonely}} n_l(v)\,)+\left|\{\text{$v$ non-lonely}\}\right|$. Let $(\mathbb{D}_1^{yp}(\uparrow^n))^m:=\{D\in\mathbb{D}_1^{yp}(\uparrow^n)|n_l(D)\le m\}$. We get a filtration of $\mathbb{D}_1^{yp}(\uparrow^n)$:
$$(\mathbb{D}_1^{yp}(\uparrow^n))^0\subseteq (\mathbb{D}_1^{yp}(\uparrow^n))^1\subseteq (\mathbb{D}_1^{yp}(\uparrow^n))^2\subseteq \cdots$$
which induces the sequence:
$$S_\mathbb{F}((\mathbb{D}_1^{yp}(\uparrow^n))^0)\subseteq S_\mathbb{F}((\mathbb{D}_1^{yp}(\uparrow^n))^1)\subseteq S_\mathbb{F}((\mathbb{D}_1^{yp}(\uparrow^n))^2)\subseteq \cdots$$
Let $L^m$ be the quotient of $S_\mathbb{F}((\mathbb{D}_1^{yp}(\uparrow^n))^m)$ by all STU, IHX and $I_1^y$ relations contained in it. We get a sequence:
$$L^0\stackrel{l_0}{\longrightarrow} L^1\stackrel{l_1}{\longrightarrow}L^2\stackrel{l_2}{\longrightarrow}\cdots$$
$L^0$ is isomorphic to $L\mathcal{A}_1^{yp}(\uparrow^n)$ (the IHX relations in this space are implied by STU, and there are no $I_1^y$ relations). The direct limit of this sequence is $\mathcal{A}_1^{yp}(\uparrow^n)$, and the maps $l_m$ induce the map $l$ at the limit. So, in order to complete the proof it is enough to define an inverse for each $l_m$.

Let $\eta_m:S_\mathbb{F}((\mathbb{D}_1^{yp}(\uparrow^n))^m)\rightarrow S_\mathbb{F}((\mathbb{D}_1^{yp}(\uparrow^n))^{m-1})$ be defined as follows: For $D$ with $n_l(D)<m$, $\eta_m(D)=D$. For $D$ with $n_l(D)=m$, we have $2$ cases. If the highest vertex on the right strand is non-lonely, define $\eta_m(D)$ by:

$$\mathfig{2}{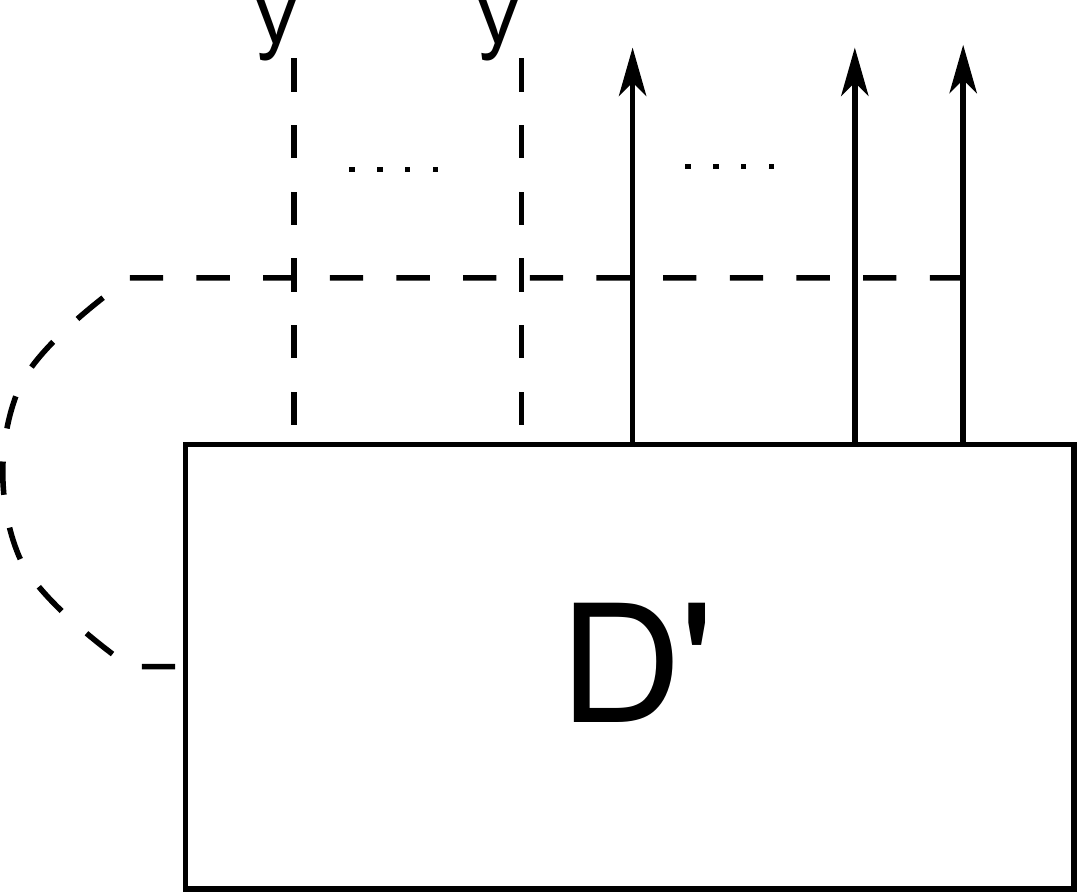}\stackrel{\eta_m}{\longmapsto}\quad-\mathfig{2}{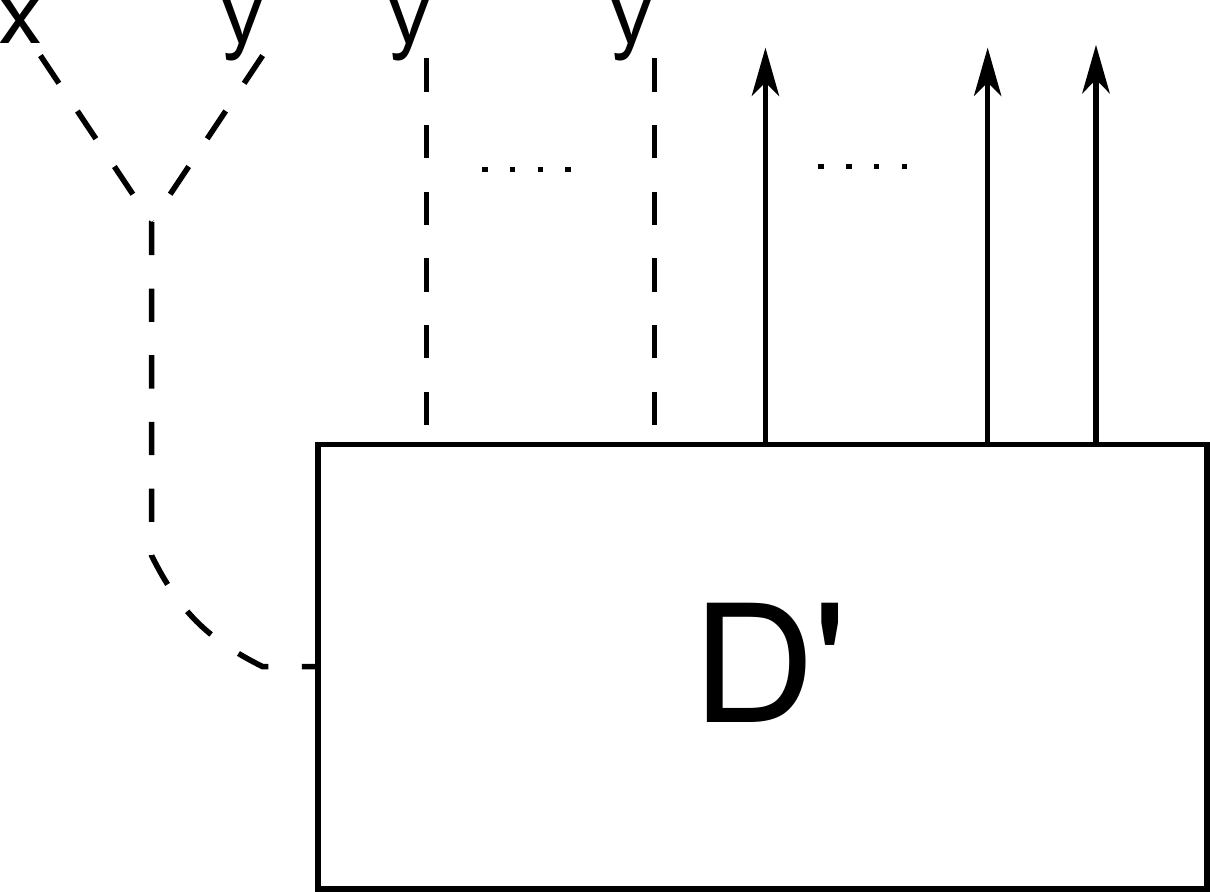}-\mathfig{2}{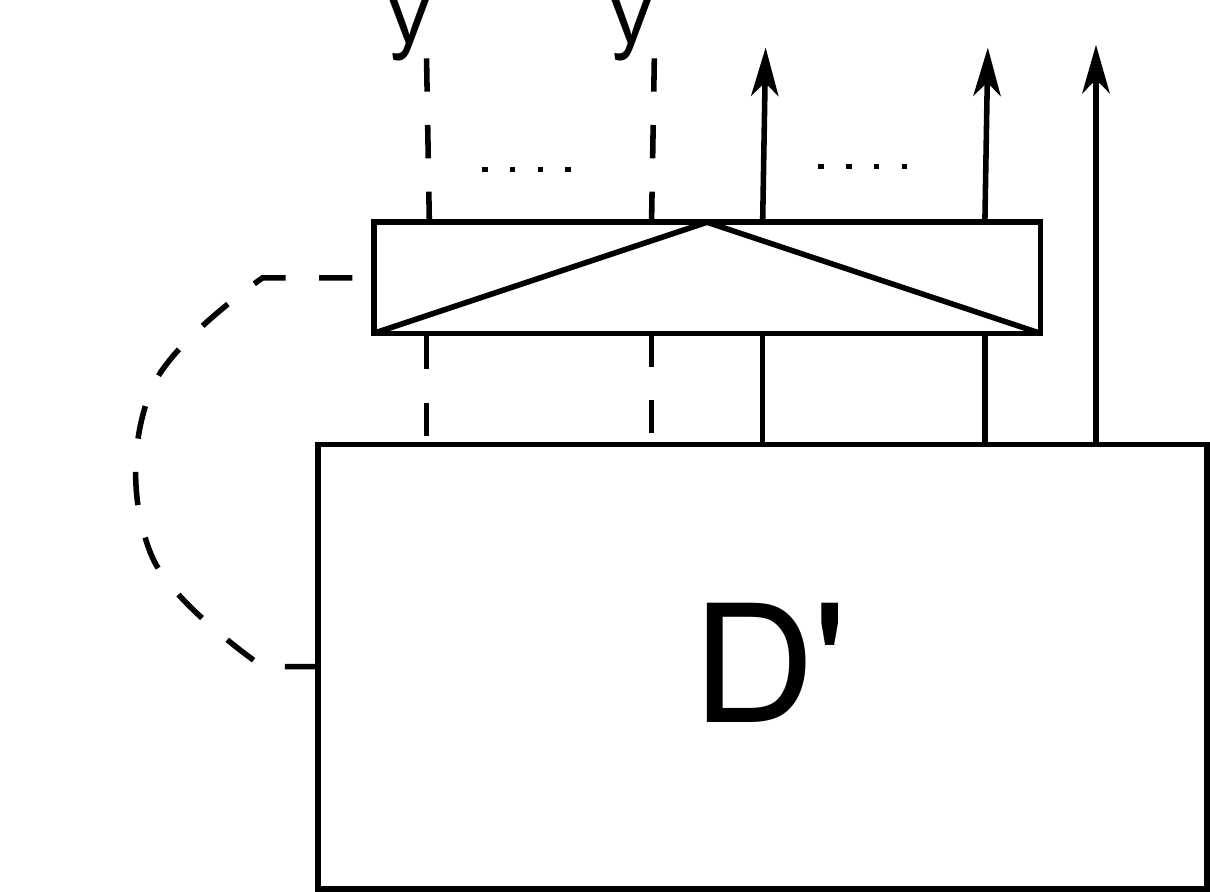}$$
In these figures we assume that there are no more $y$ labels inside $D'$.

If the highest vertex on the right strand in $D$ is a lonely vertex, denote the highest non-lonely vertex in $D$ by $v$, and define $\eta_m(D)$ by:
$$\mathfig{2.5}{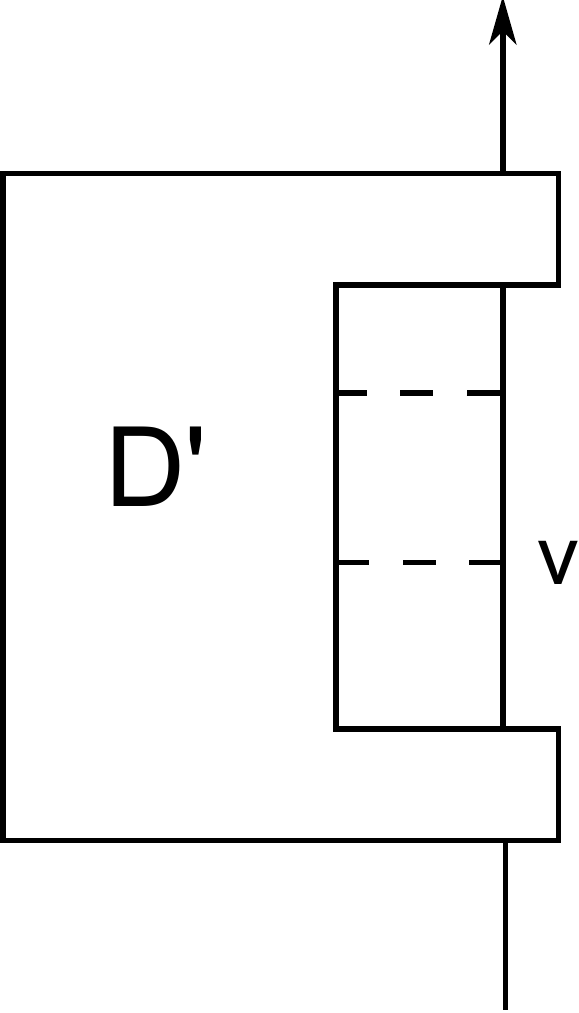}\stackrel{\eta_m}{\longmapsto}\quad\mathfig{2.5}{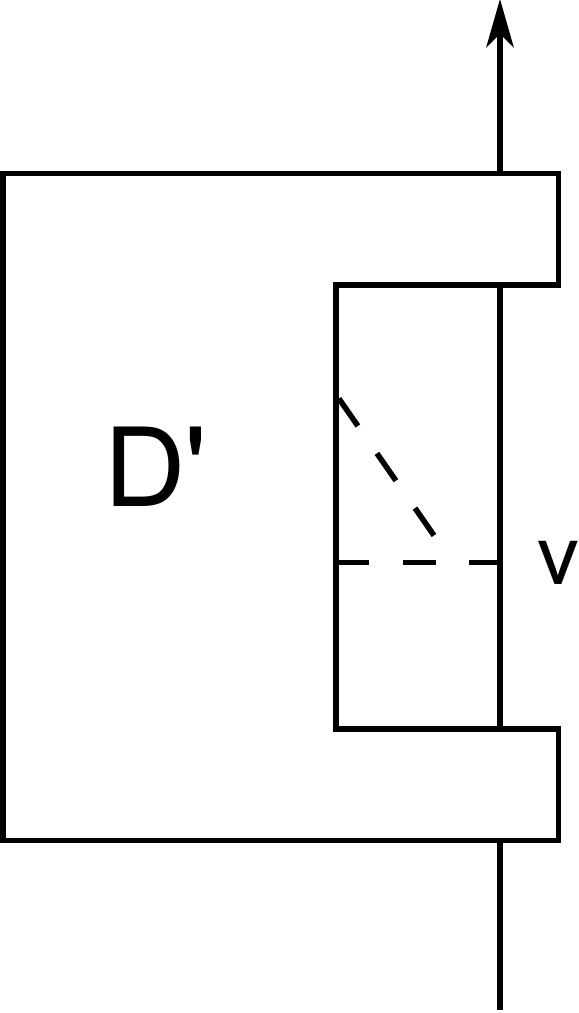}+\mathfig{2.5}{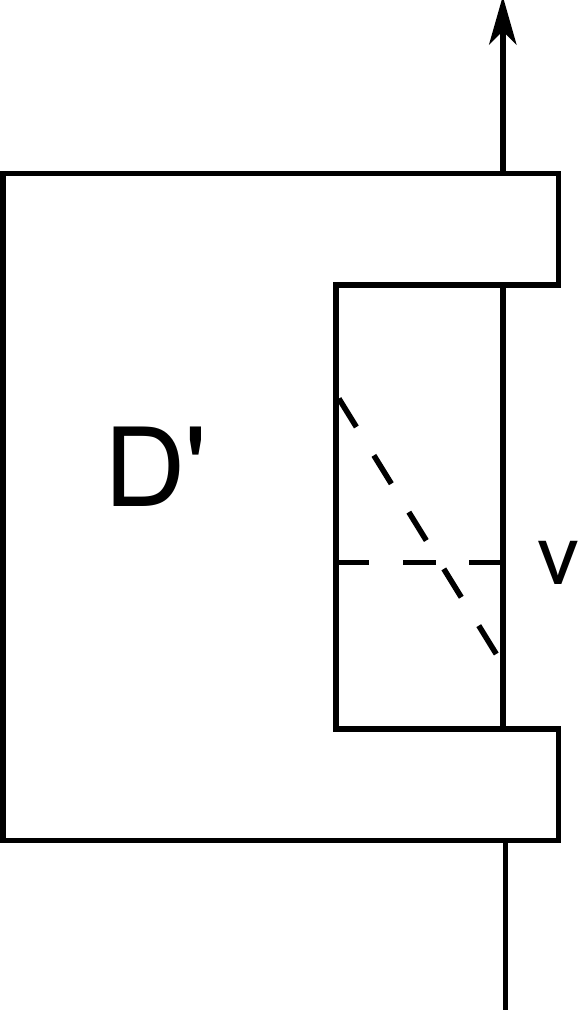}$$
We claim that $\eta_m$ induces a map $\eta_m:L^m\rightarrow L^{m-1}$. Indeed, if $u\in S_\mathbb{F}(\mathbb{D}_1^{yp}(\uparrow^n))^m)$ is an $I_1^y$ relation then $\eta_m(u)$ is either again an $I_1^y$ relation, or is equal to $0$ by definition of $\eta_m$. If $u$ is an IHX relation, $\eta_m(u)$ is a sum of IHX relations. Suppose now $u$ is an STU relation:
$$u=u_1-u_2-u_3=\mathfig{2.5}{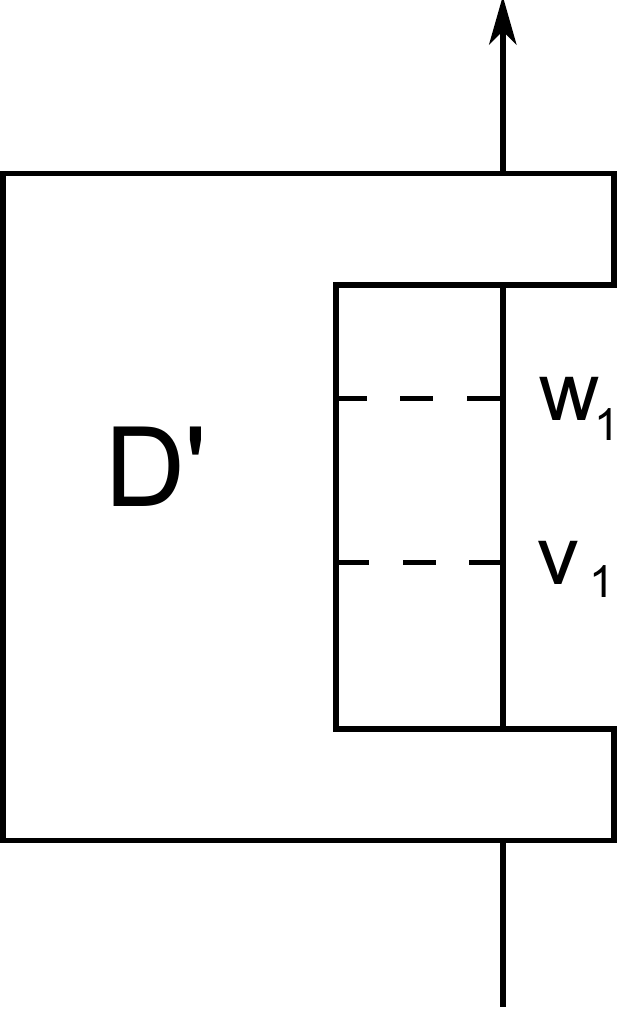}-\mathfig{2.5}{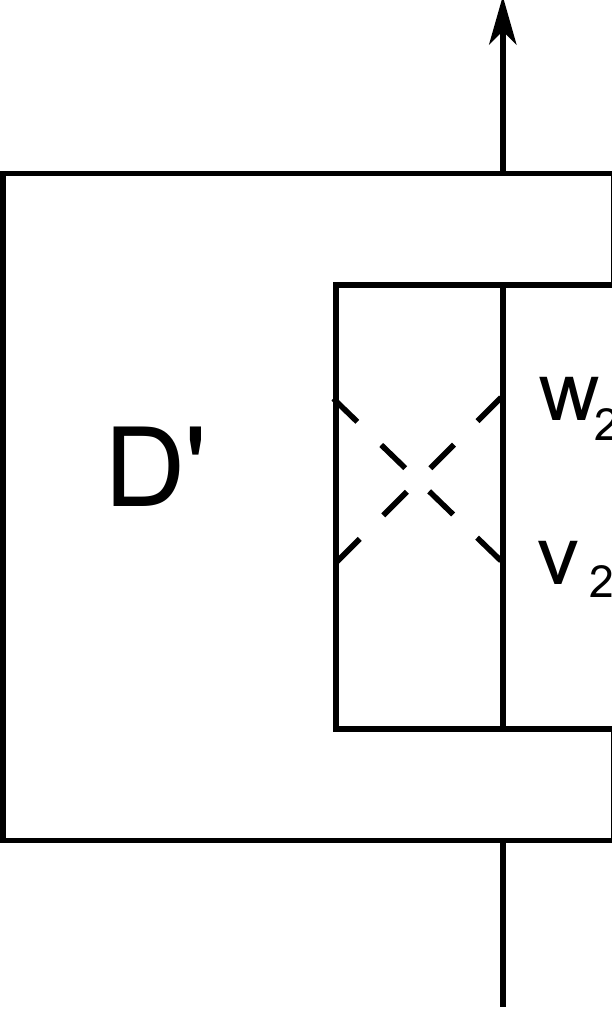}-\mathfig{2.5}{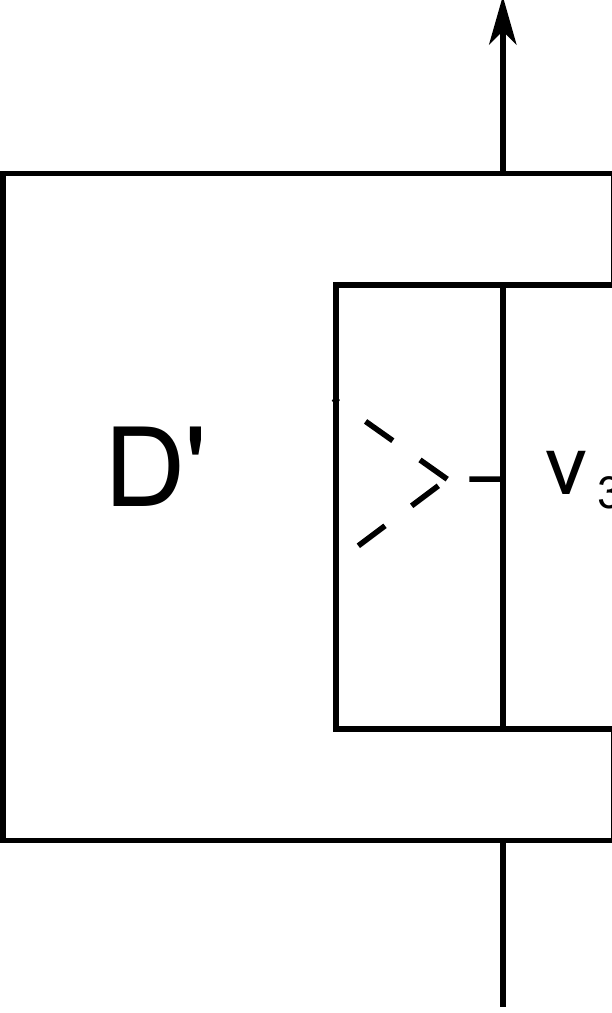}$$
(Note: the labels $\text{v}_i$, $\text{w}_i$ etc. are not part of the diagrams. We write them only to help keep track in the following computations.)

If none of the vertices $\text{v}_i$, $\text{w}_i$ are the highest non-lonely vertices in their respective diagrams or the vertices immediately above the highest non-lonely vertices, then $\eta_m(u)$ is a sum of $STU$ relations. Otherwise, we have to deal with several different cases (in all those cases, we assume that at least one of $u_1,u_2,u_3$ has $n_l(u_i)=m$, because otherwise $\eta_m(u)=u$):

\begin{enumerate}
\item[A.] If $\text{v}_3$ is the highest non-lonely vertex in $u_3$, and either $\text{w}_1$ or $\text{w}_2$ is a lonely vertex, then by definition $\eta_m(u)=0$.
\item[B.] If $\text{v}_3$ is the highest non-lonely vertex in $u_3$ and $\text{w}_1$,$\text{w}_2$ are non-lonely (and therefore also $\text{v}_1$ and $\text{v}_2$), then again we have $2$ cases:
\begin{itemize}

\item[B1.] If $\text{v}_3$ is the highest vertex on the right strand, we have:
$$\eta_m(u)\approx \mathfignum{3}{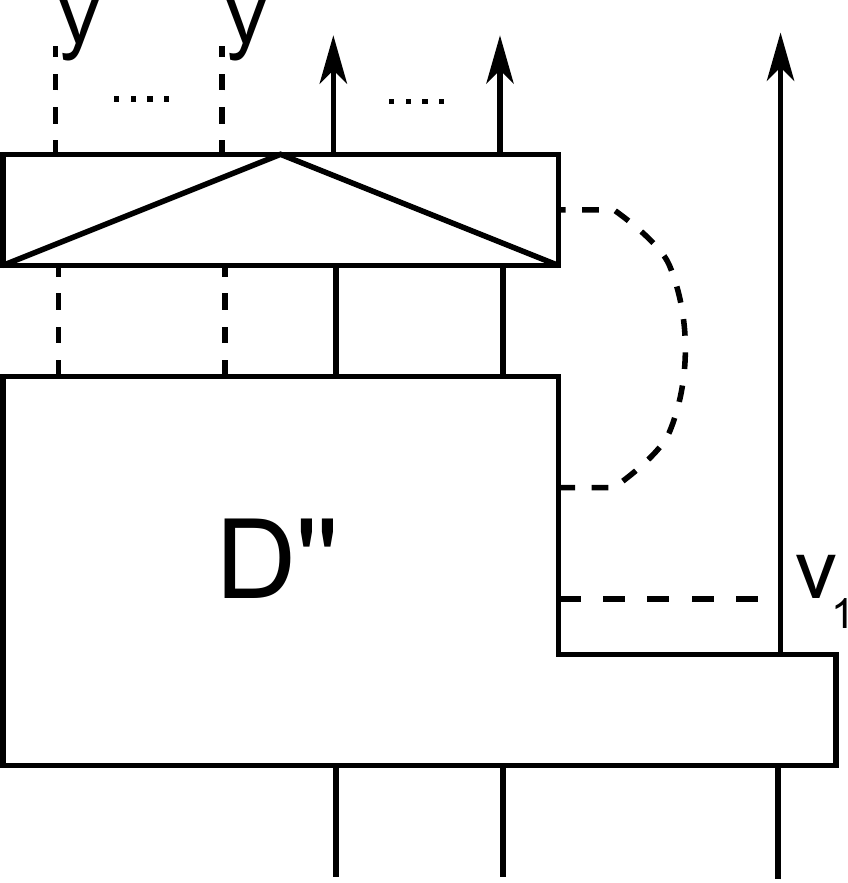}{1} - \mathfignum{3}{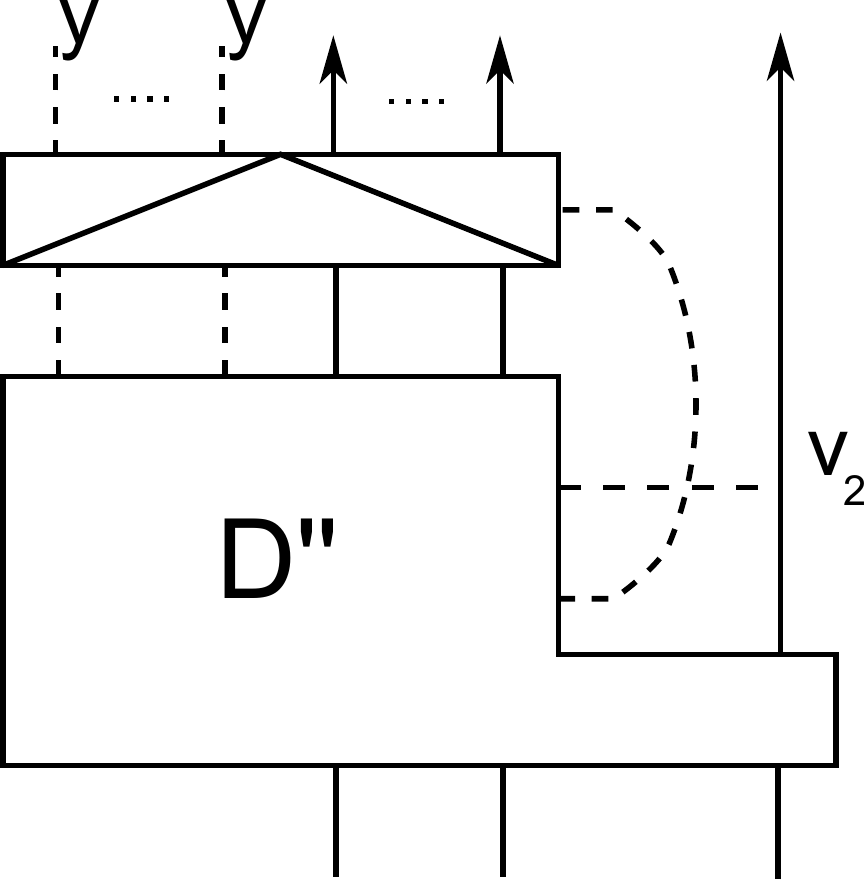}{2} - \mathfignum{3}{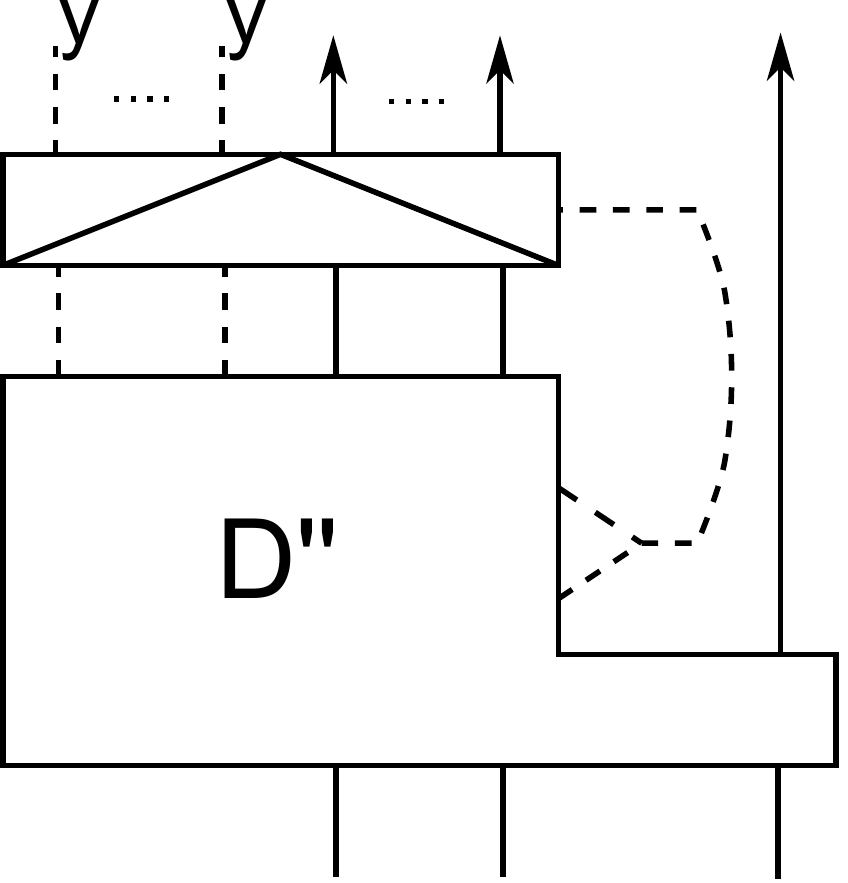}{3} +$$

$$+\mathfignum{3}{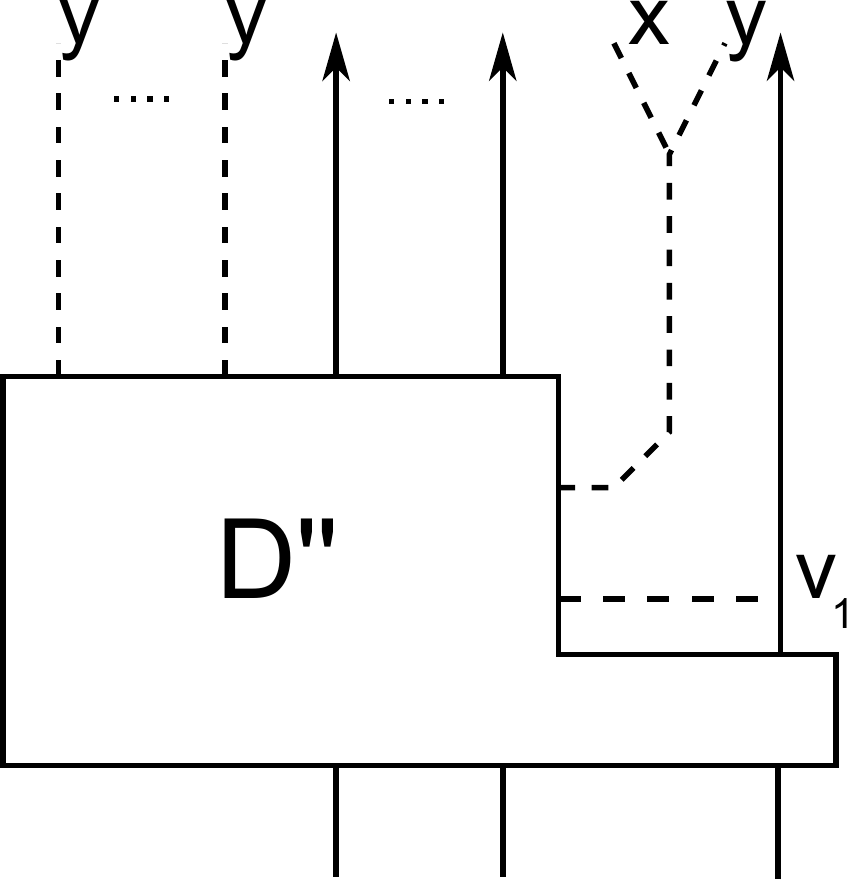}{4} - \mathfignum{3}{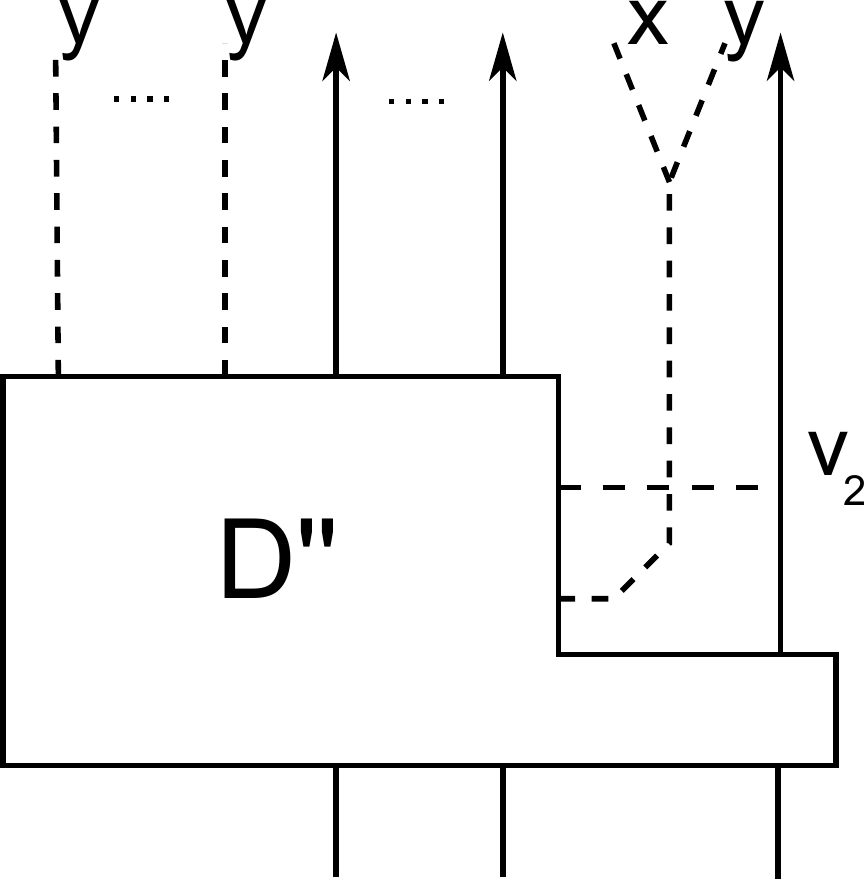}{5} - \mathfignum{3}{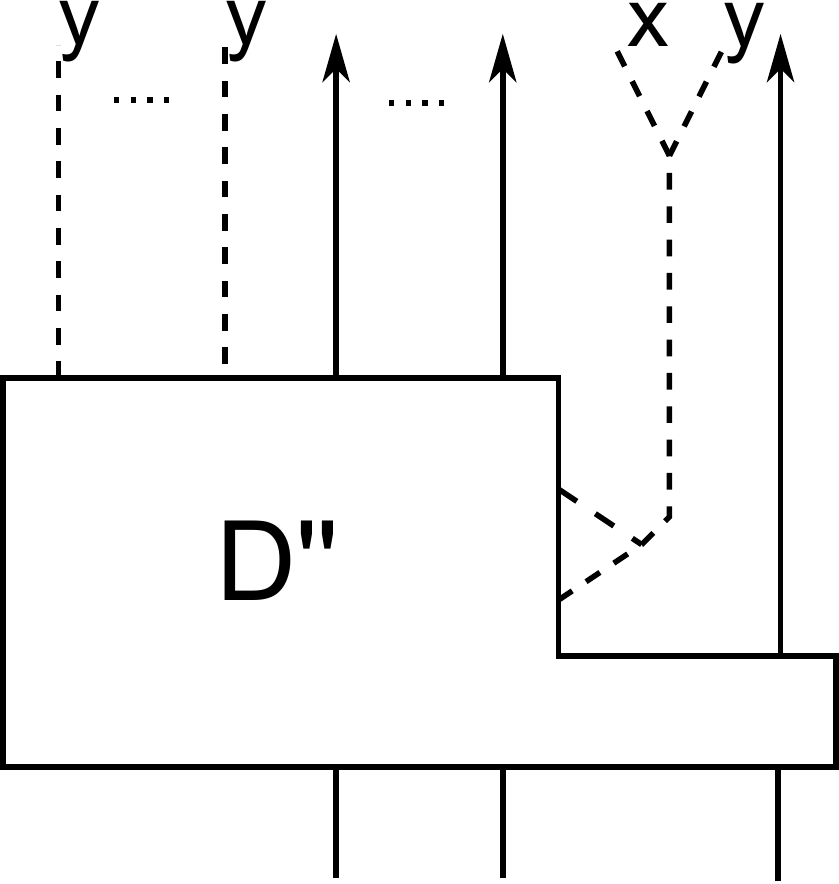}{6} \approx$$

$$-\mathfigbox{3}{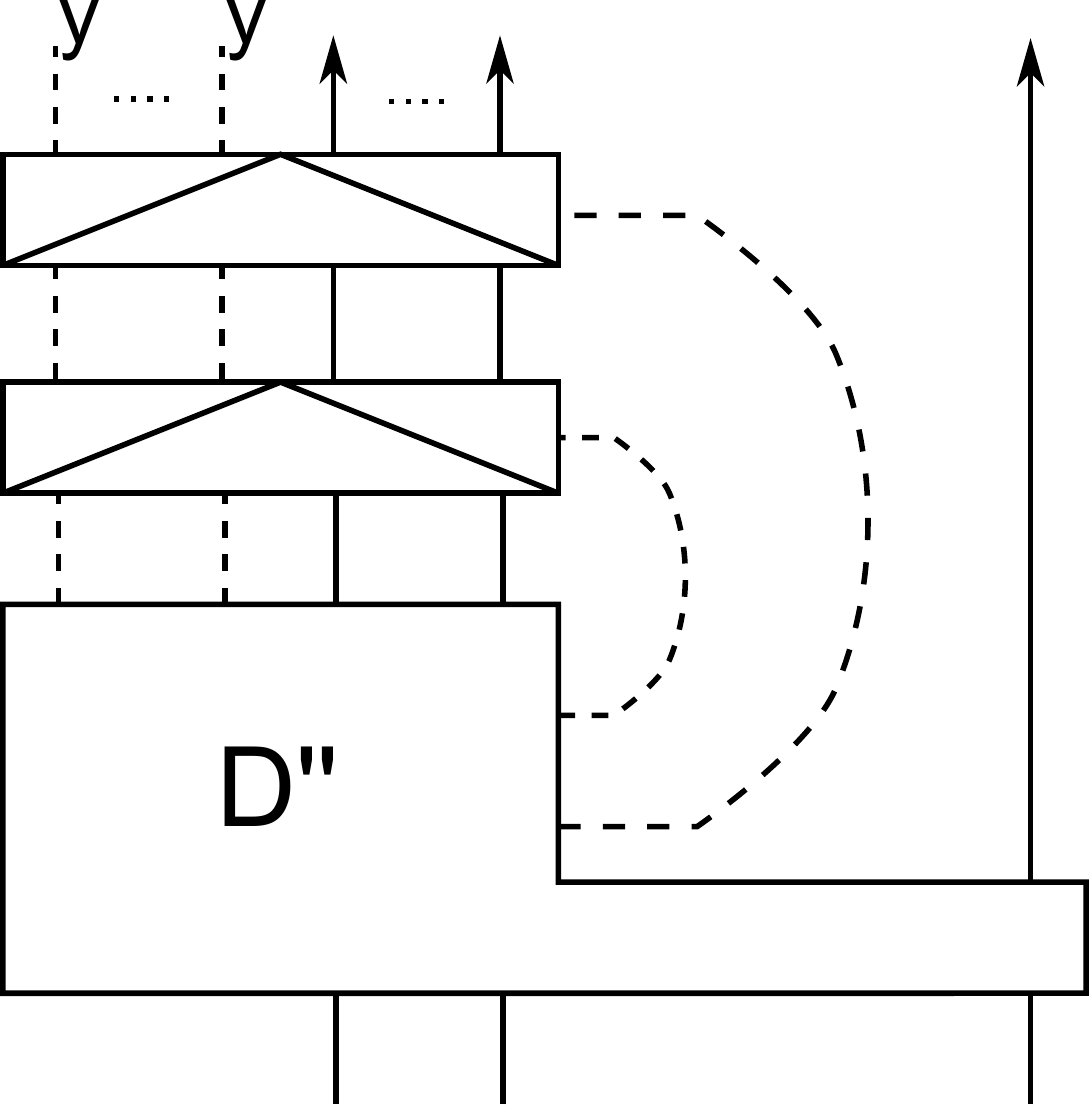}{1} - \mathfigbox{3}{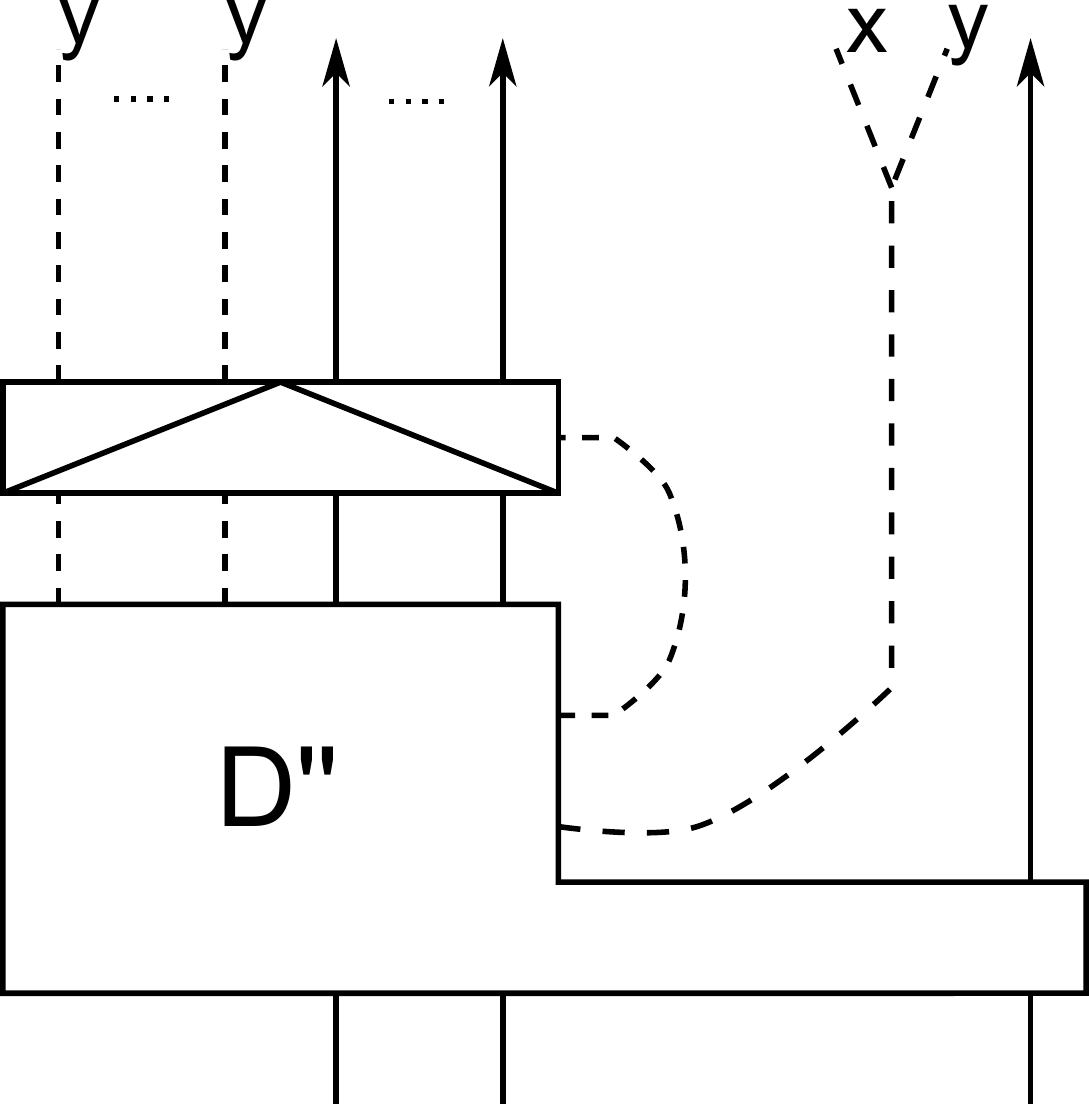}{2} + \mathfigbox{3}{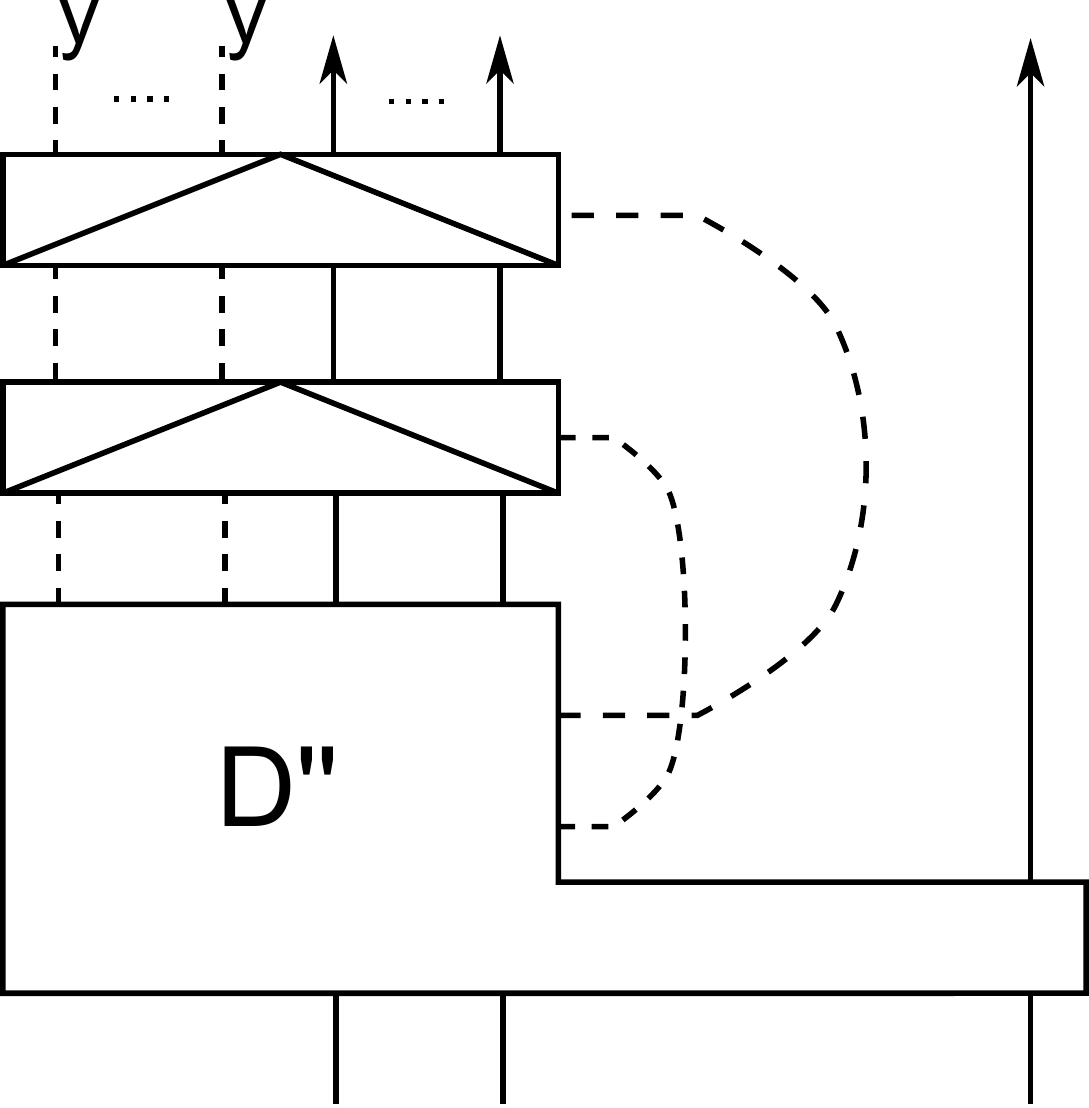}{3} +$$

$$+\mathfigbox{3}{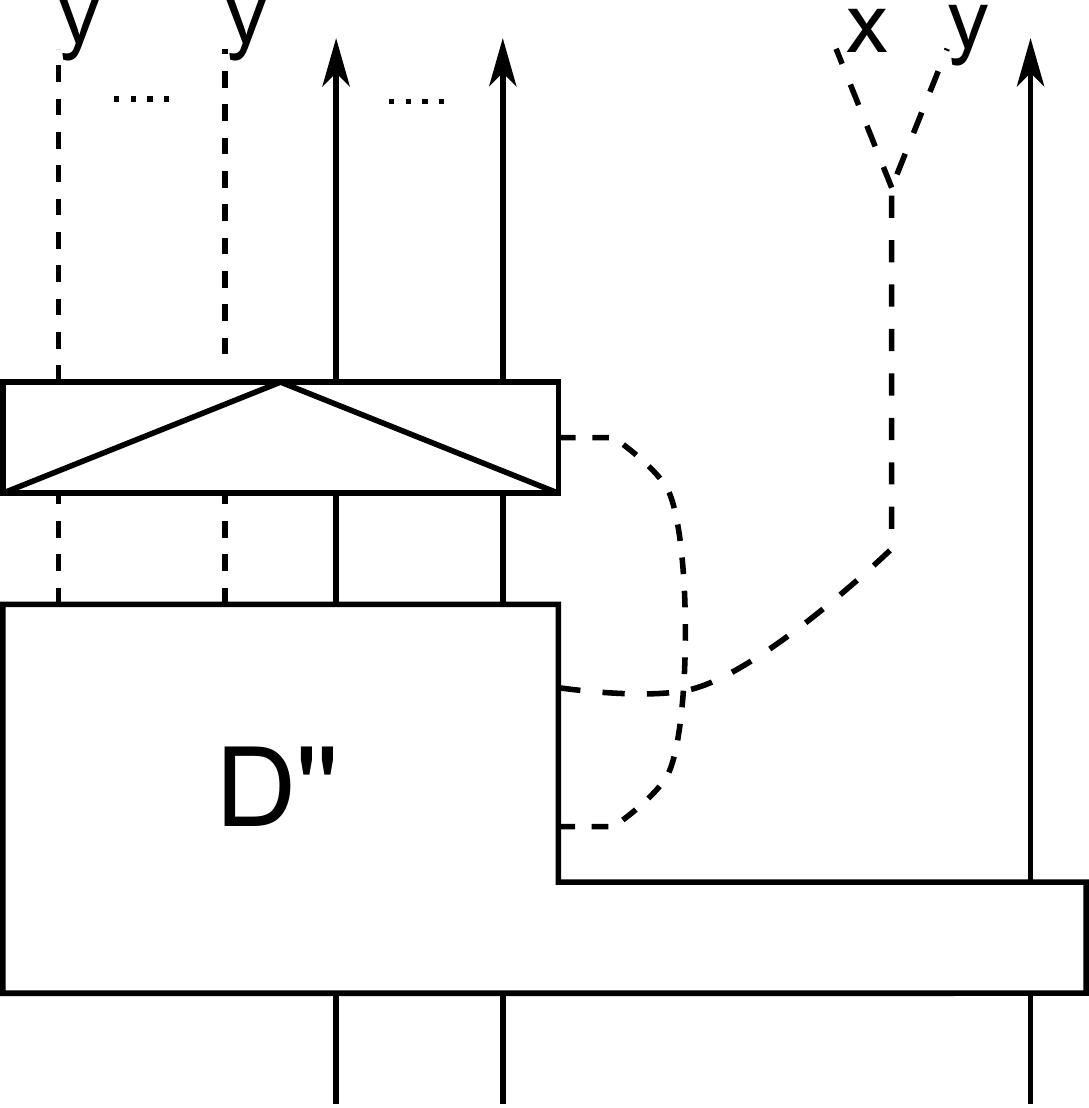}{4} - \mathfigbox{3}{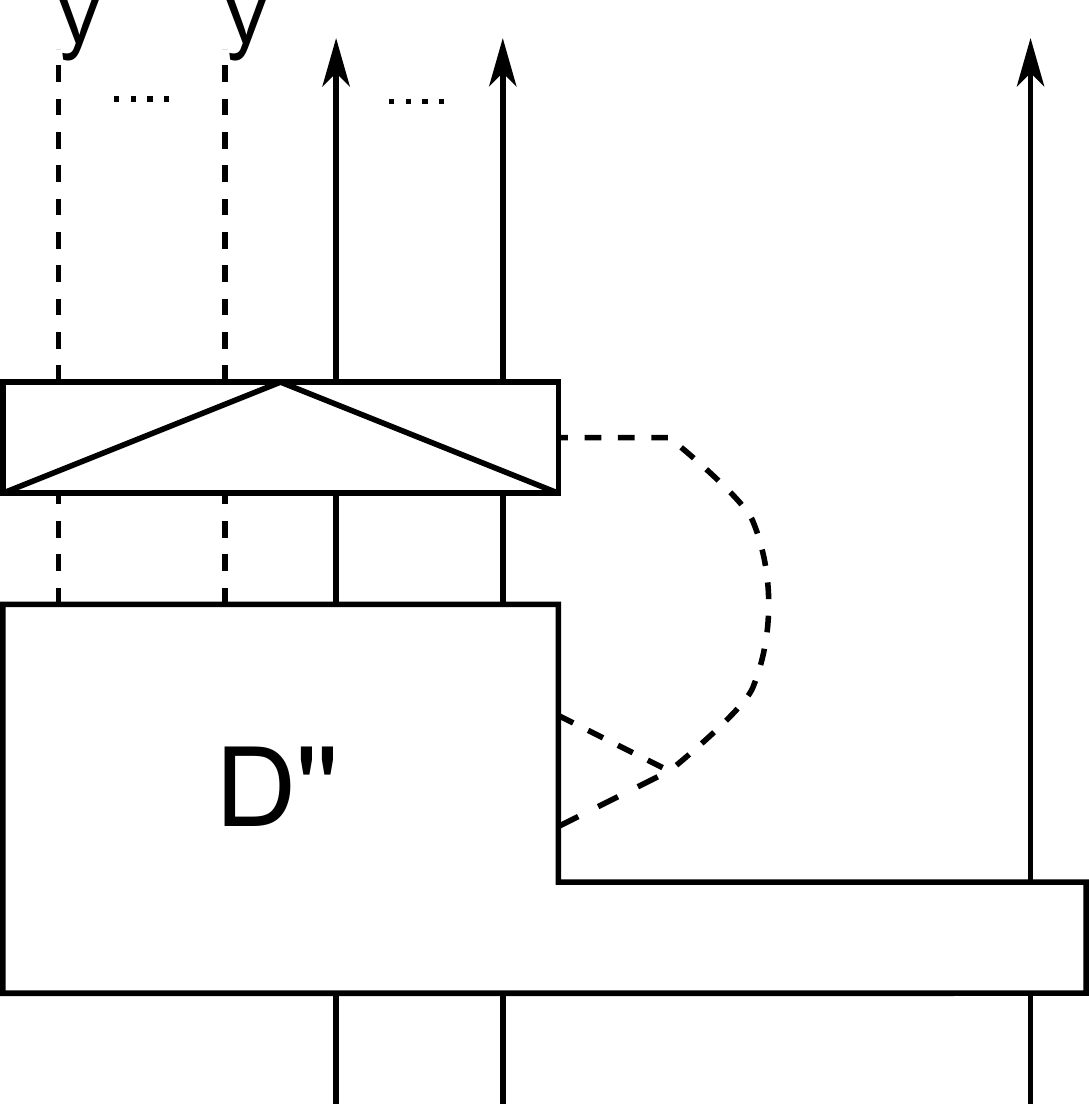}{5} - \mathfigbox{3}{section3_section32_proof_B1_24}{6} -$$

$$-\mathfigbox{3}{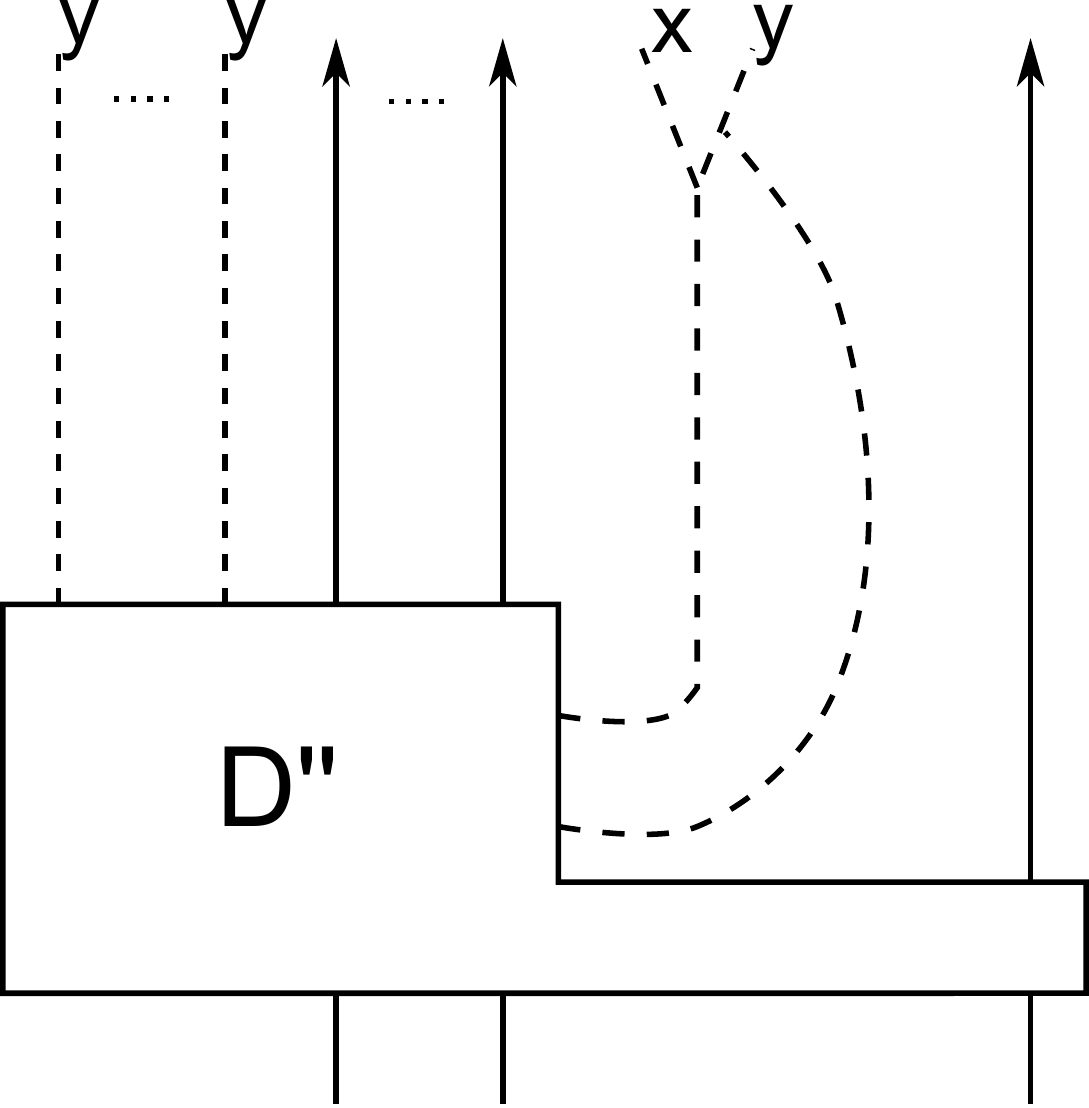}{7} - \mathfigbox{3}{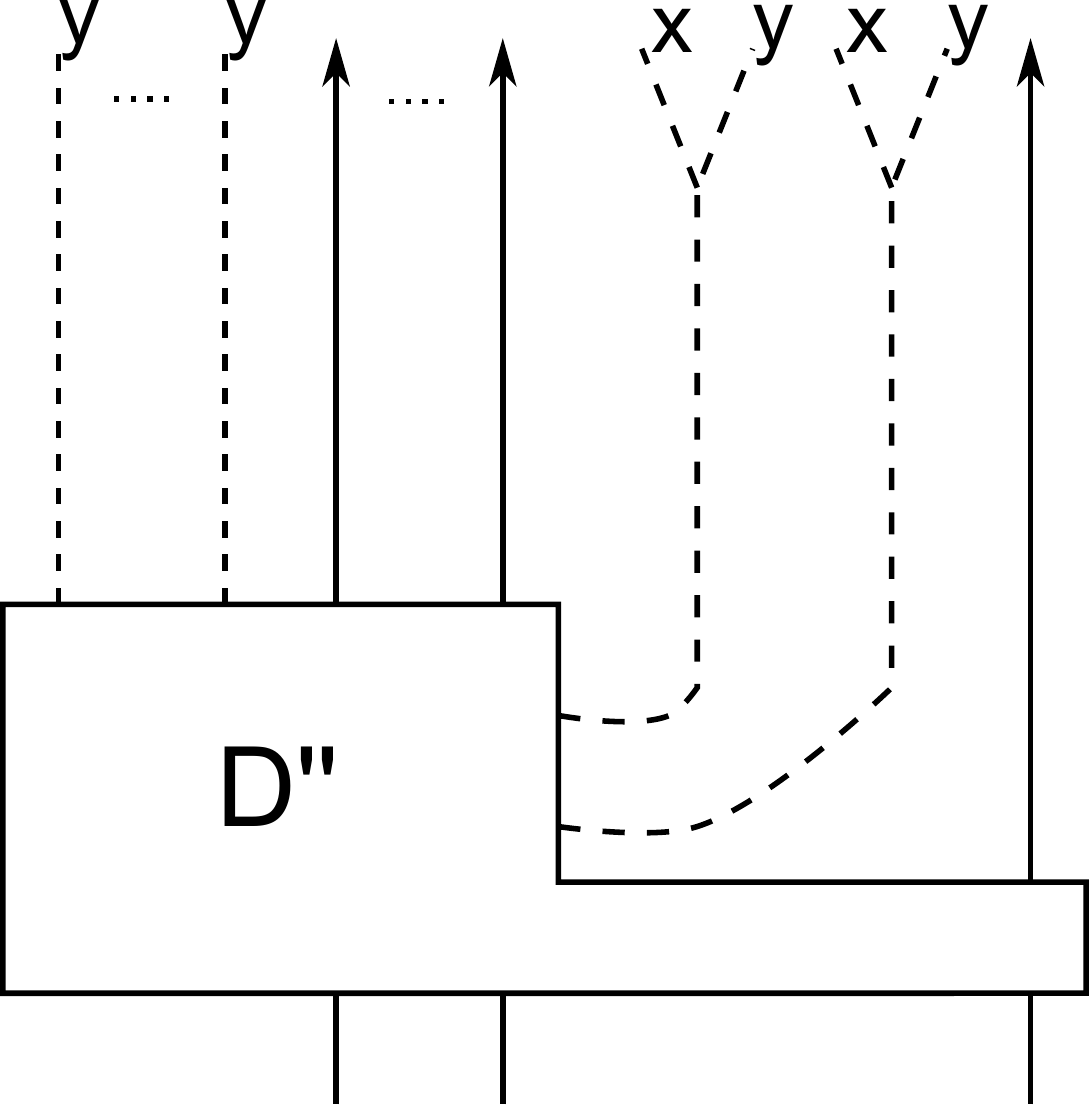}{8} + \mathfigbox{3}{section3_section32_proof_B1_22}{9} +$$

$$+\mathfigbox{3}{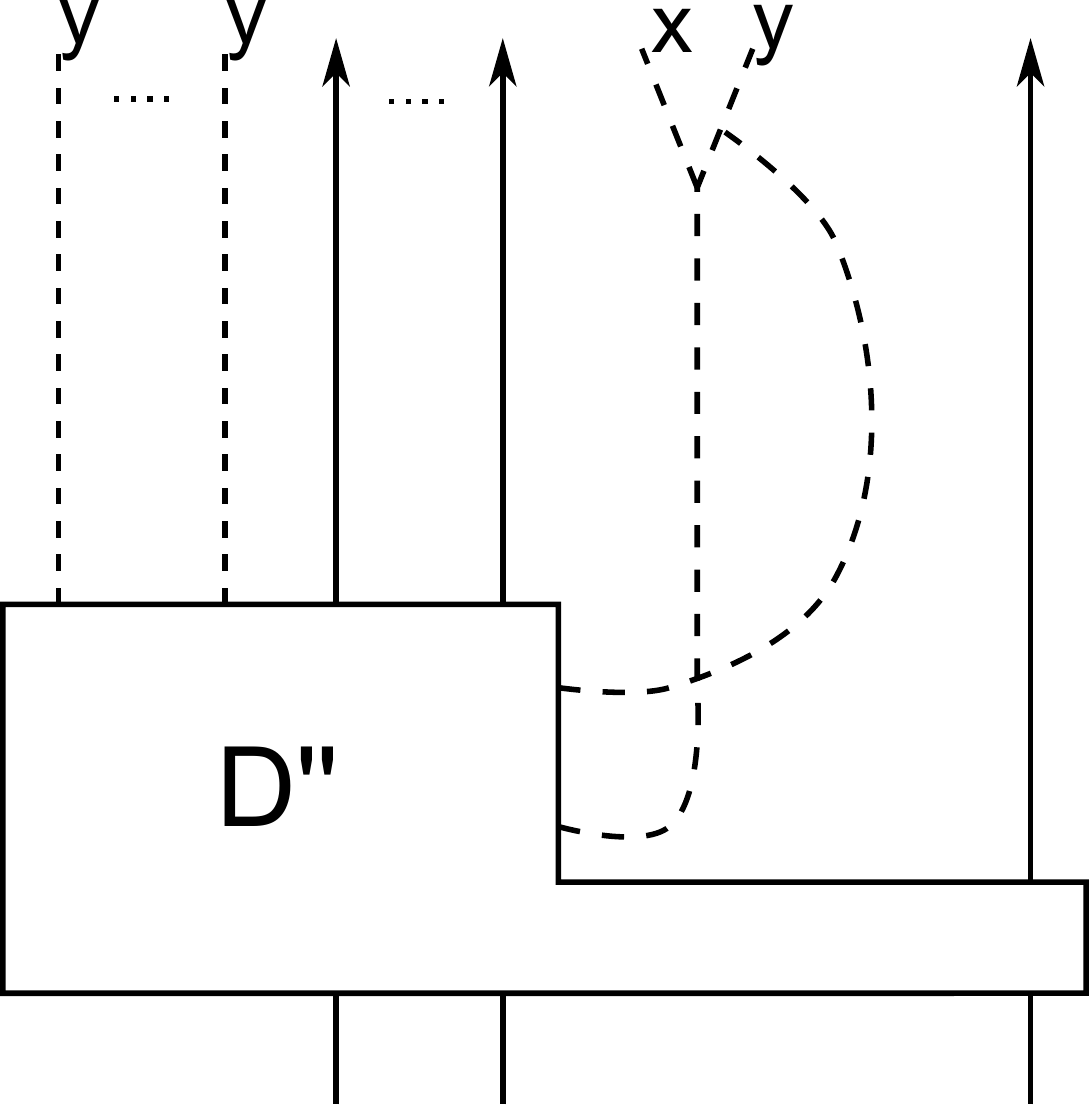}{10} + \mathfigbox{3}{section3_section32_proof_B1_28}{11} - \mathfigbox{3}{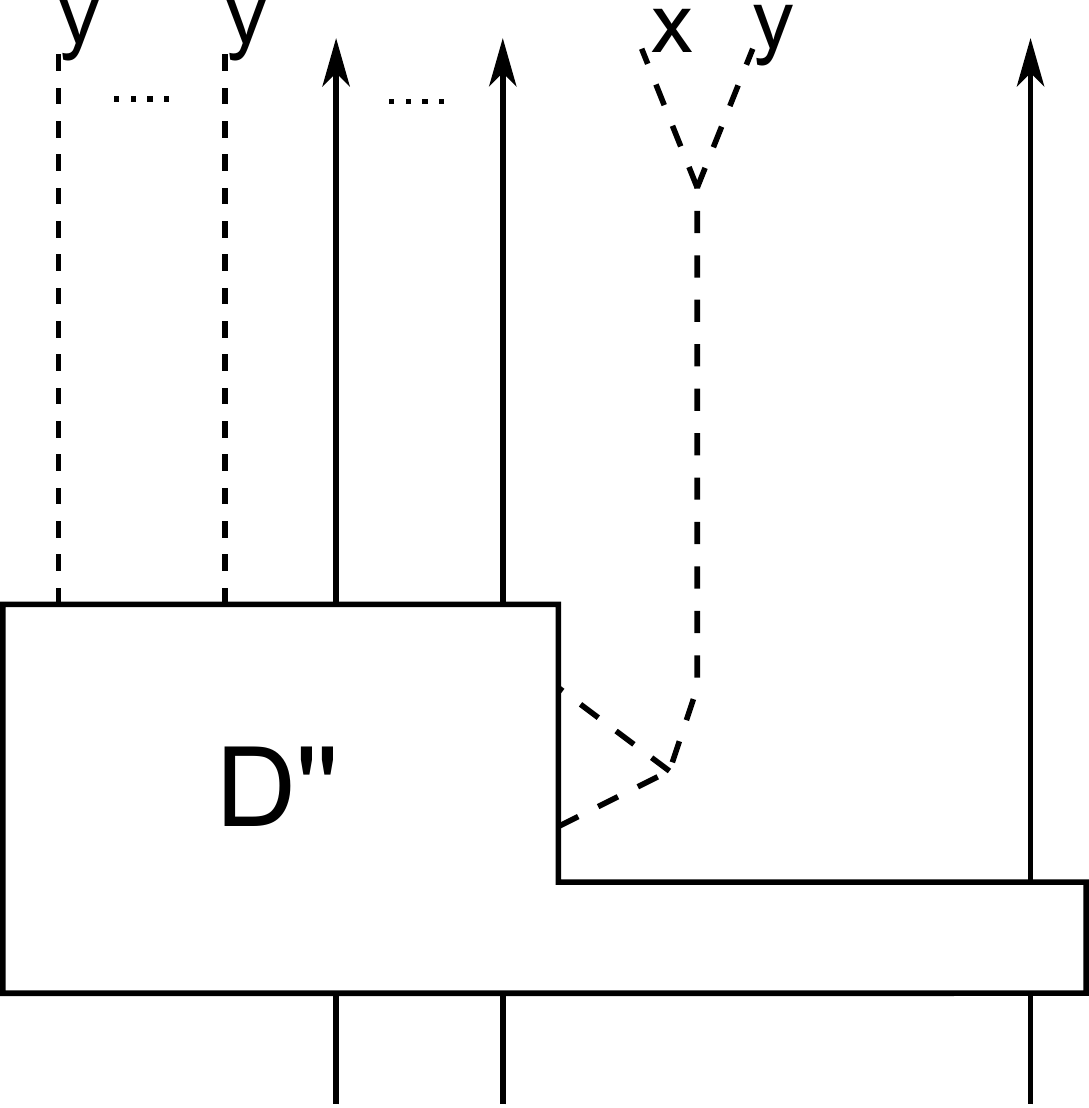}{12} \approx 0$$

We obtain the first equivalence by replacing \circlenum{1} with \framebox{1} and \framebox{2}, \circlenum{2} with \framebox{3} and \framebox{4}, \circlenum{4} with \framebox{6}, \framebox{7} and \framebox{8}, and \circlenum{5} with \framebox{9}, \framebox{10} and \framebox{11}. The second equivalence follows because: \framebox{1}, \framebox{3} and \framebox{5} are a sum of STU and IHX relations, \framebox{2} cancels \framebox{9}, \framebox{4} cancels \framebox{6}, \framebox{7}, \framebox{10} and \framebox{12} are an IHX relation, and \framebox{8} cancels \framebox{11}.

\item[B2.] If $\text{v}_3$ is not the highest vertex on the right strand, we have:

$$u=u_1-u_2-u_3=\mathfig{4}{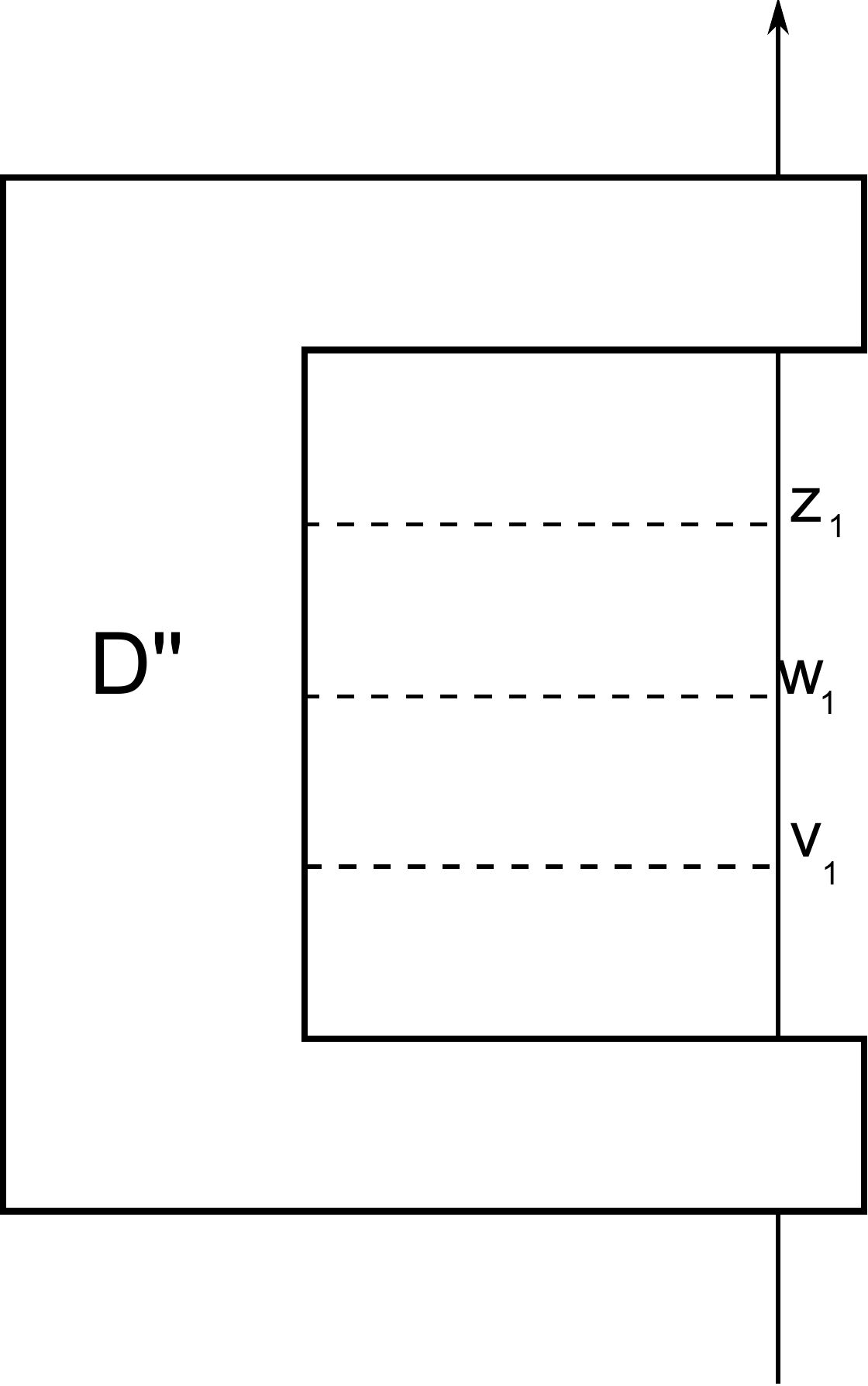}-\mathfig{4}{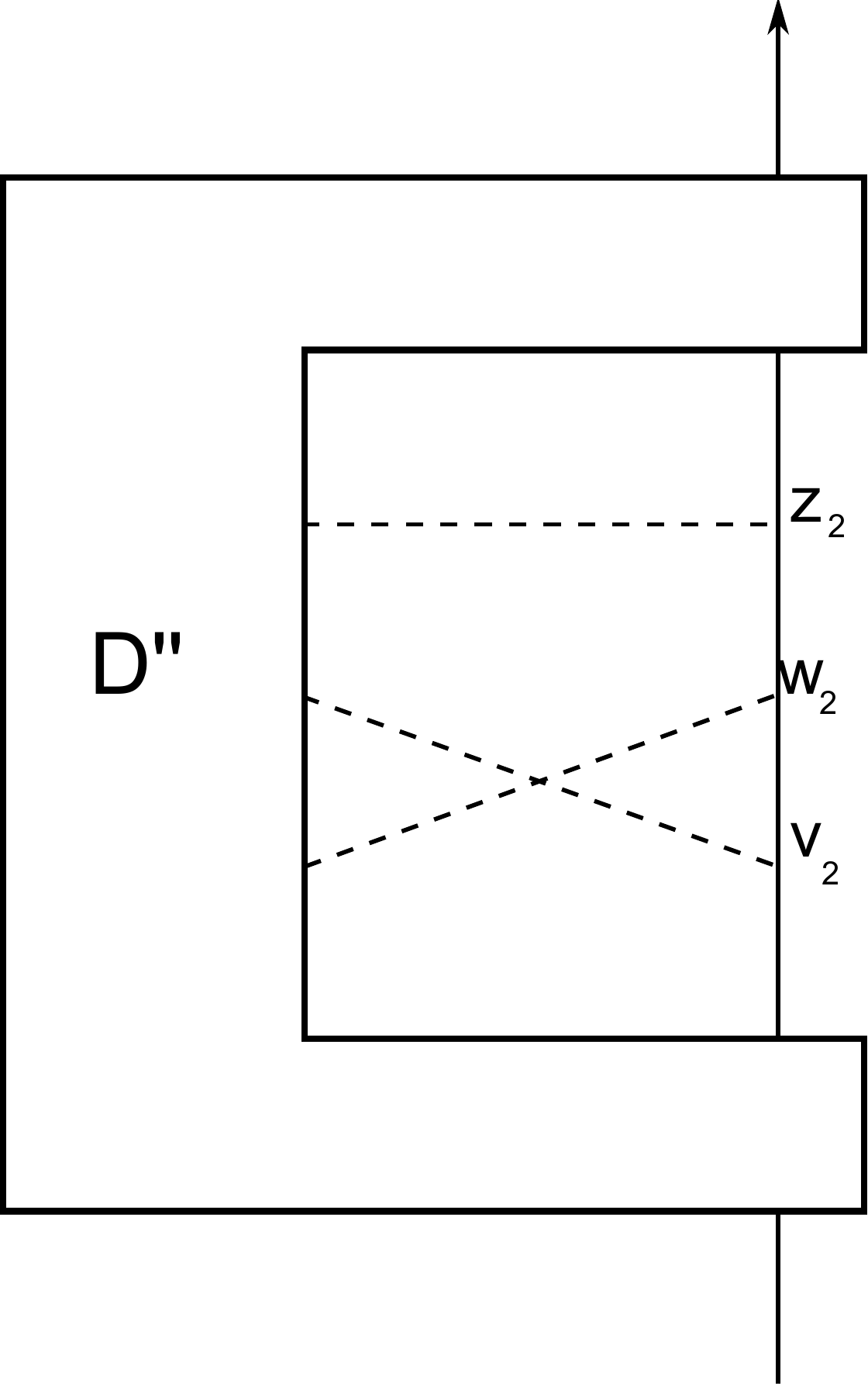}-\mathfig{4}{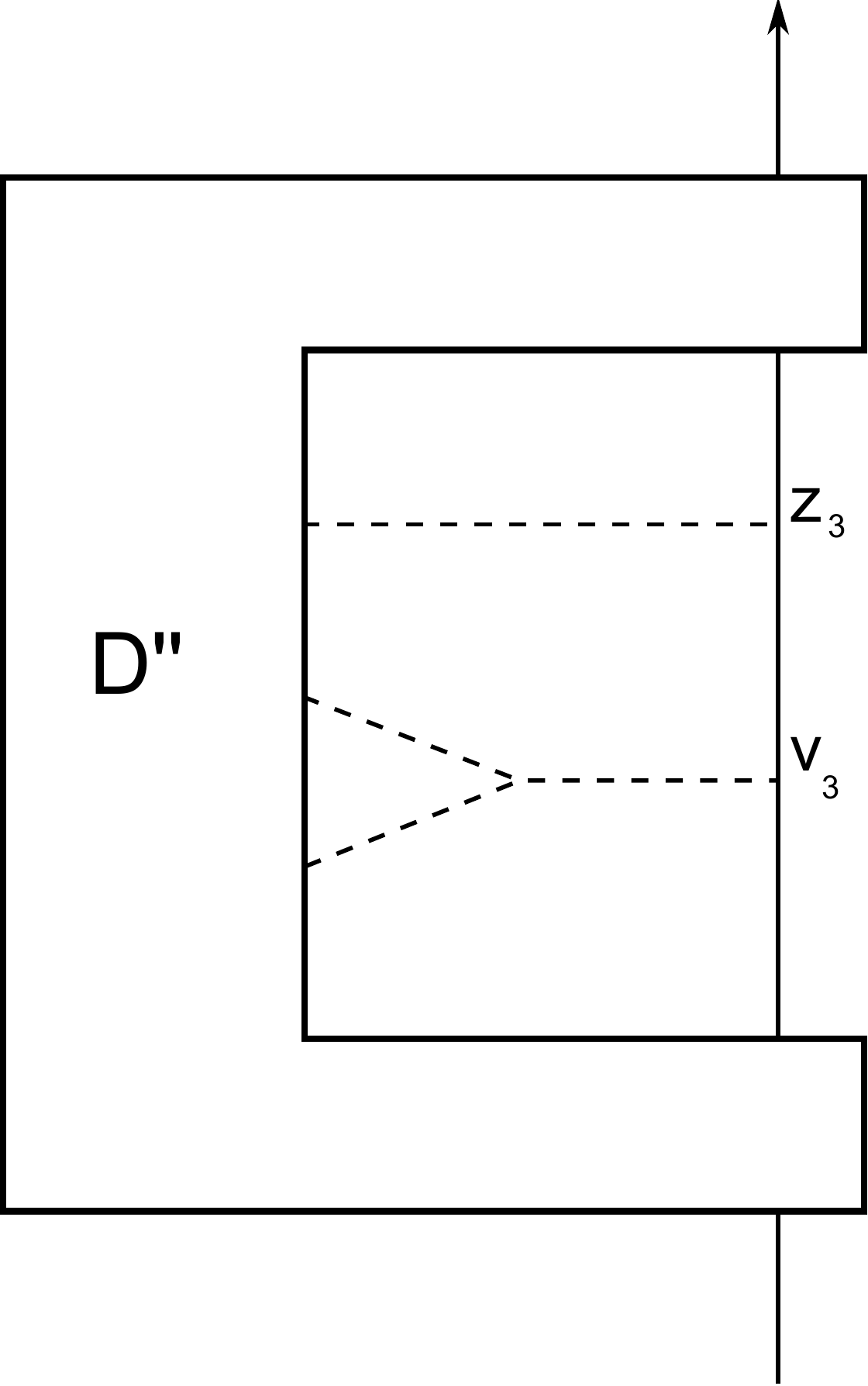}$$
and
$$\eta_m(u)\approx \mathfig{4}{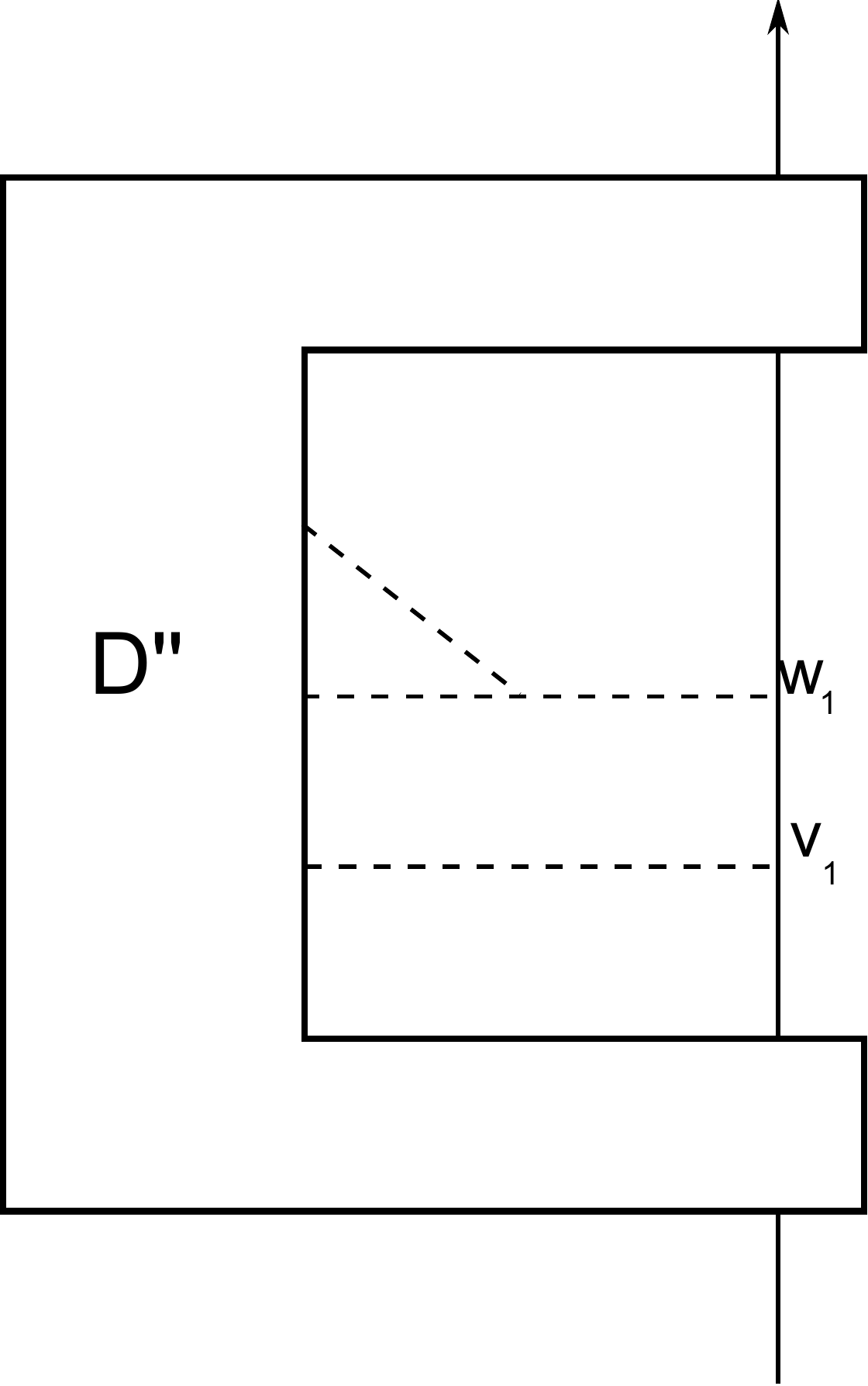} + \mathfig{4}{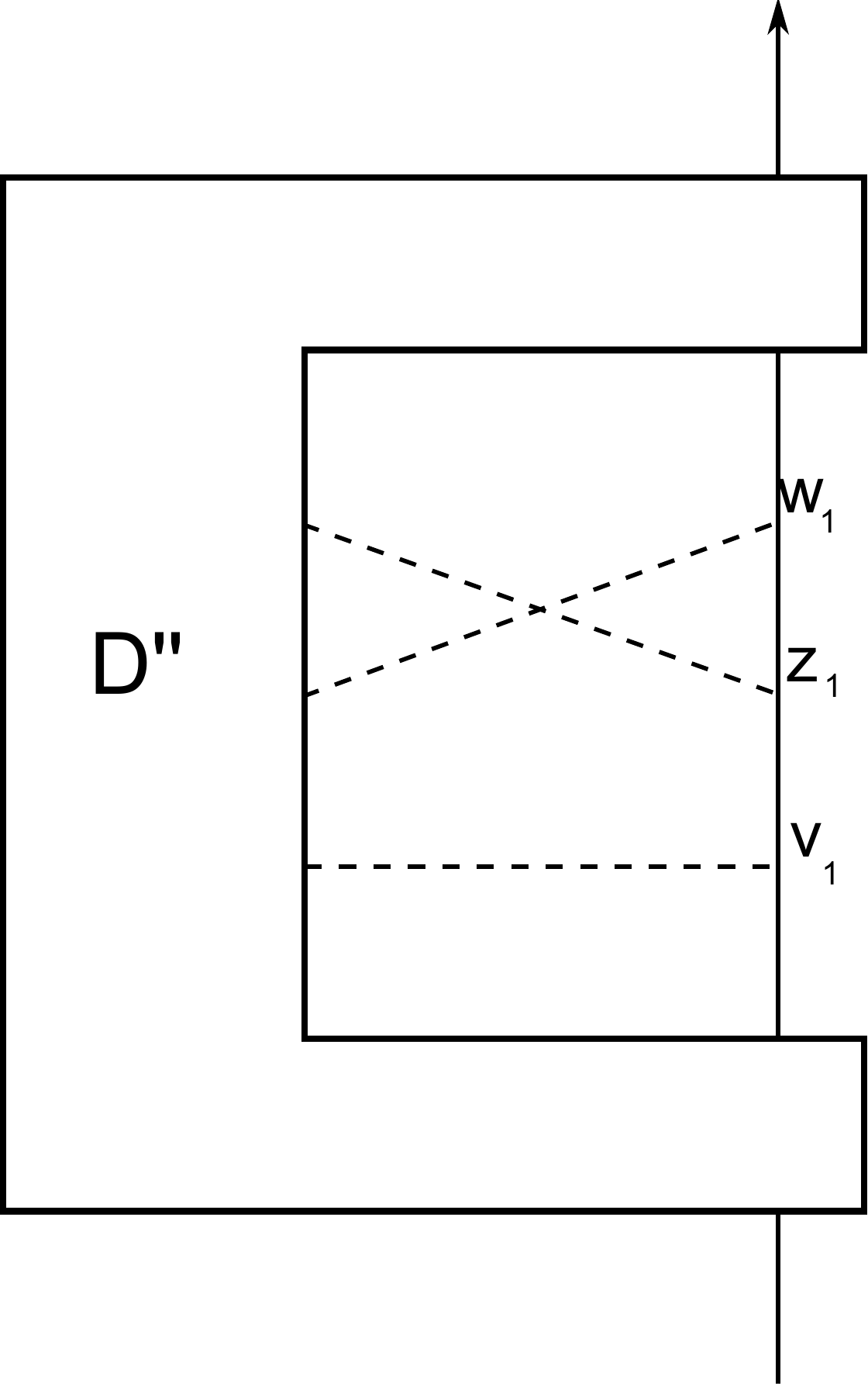} - \mathfig{4}{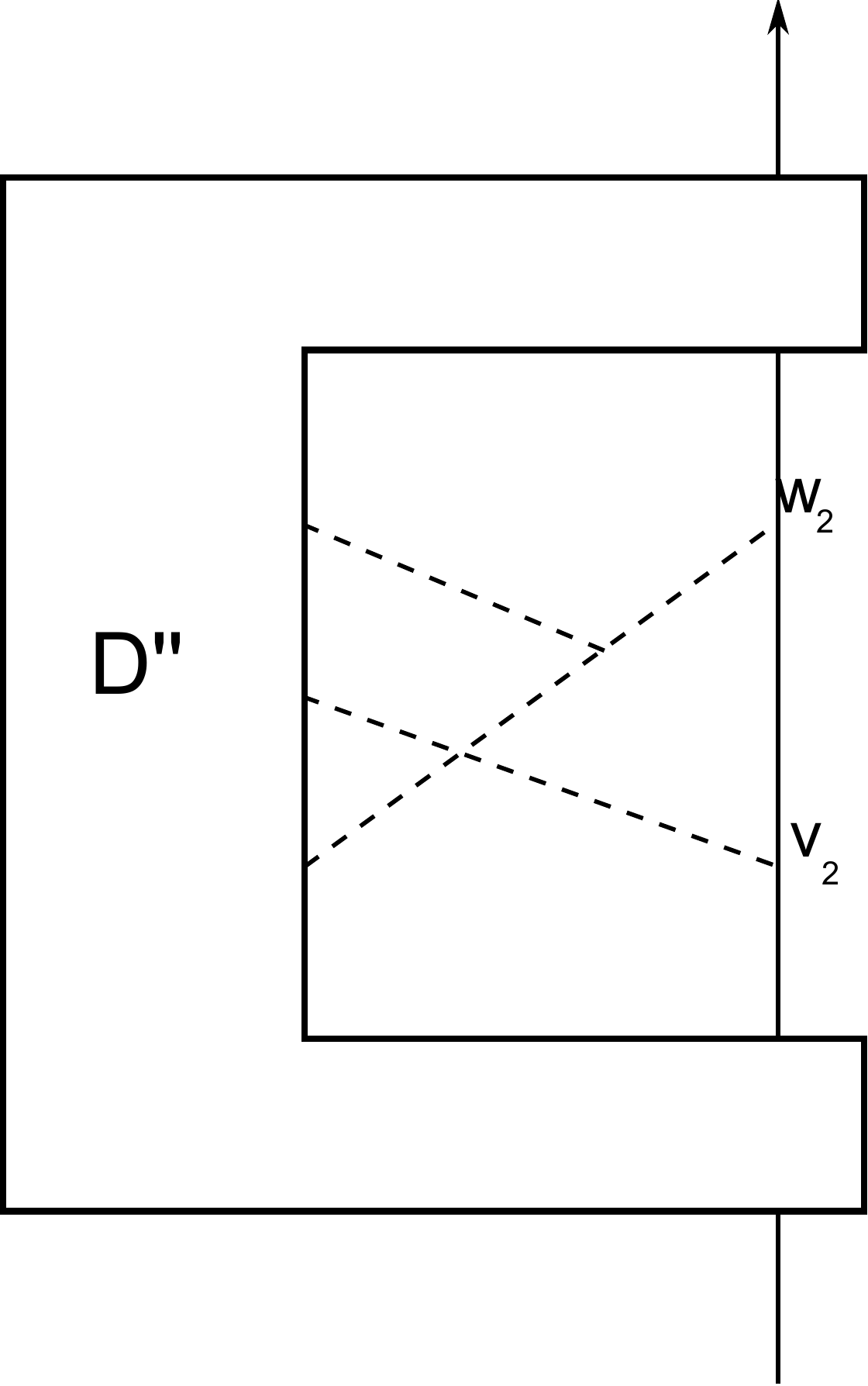} - $$ 

$$- \mathfig{4}{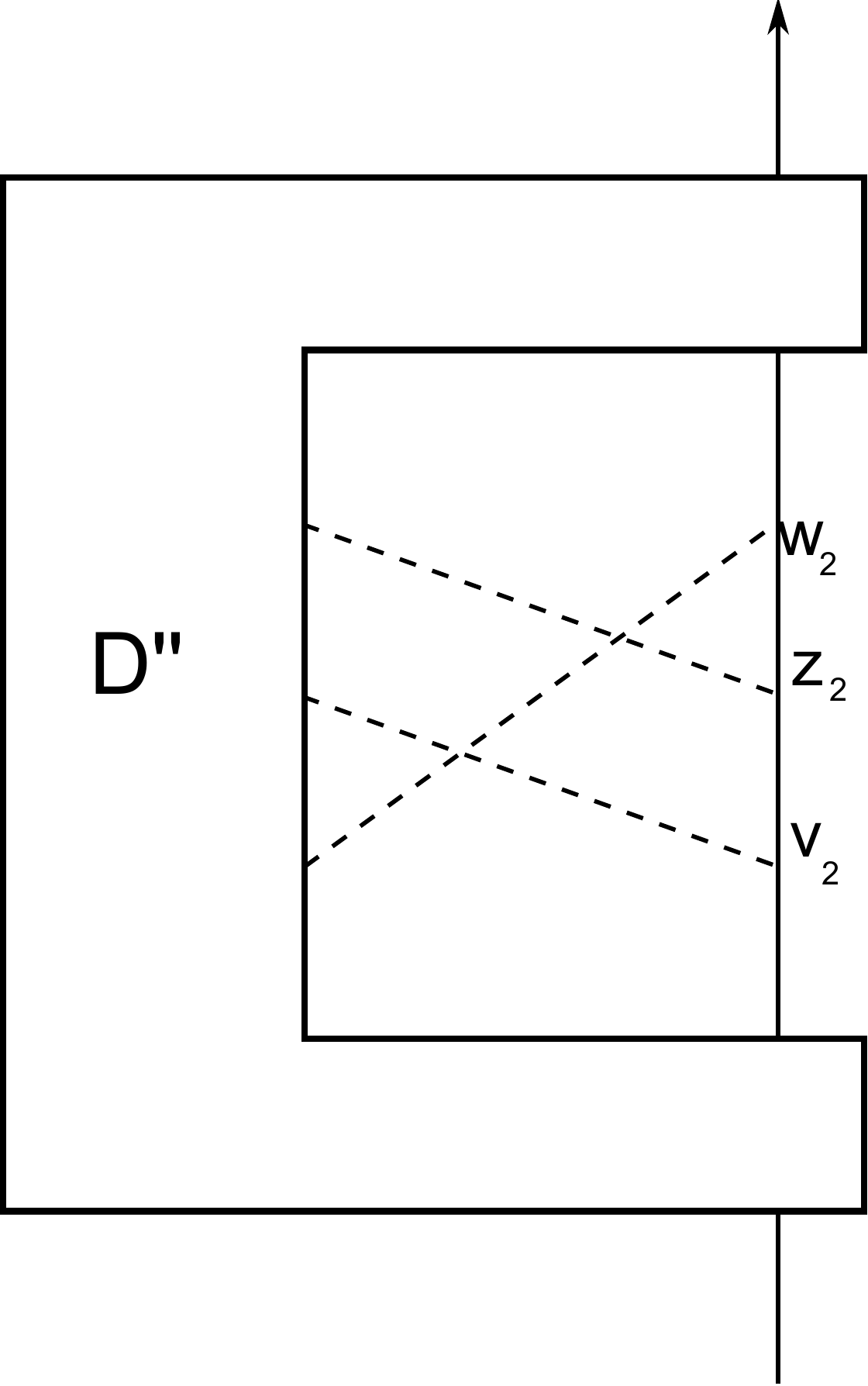} - \mathfig{4}{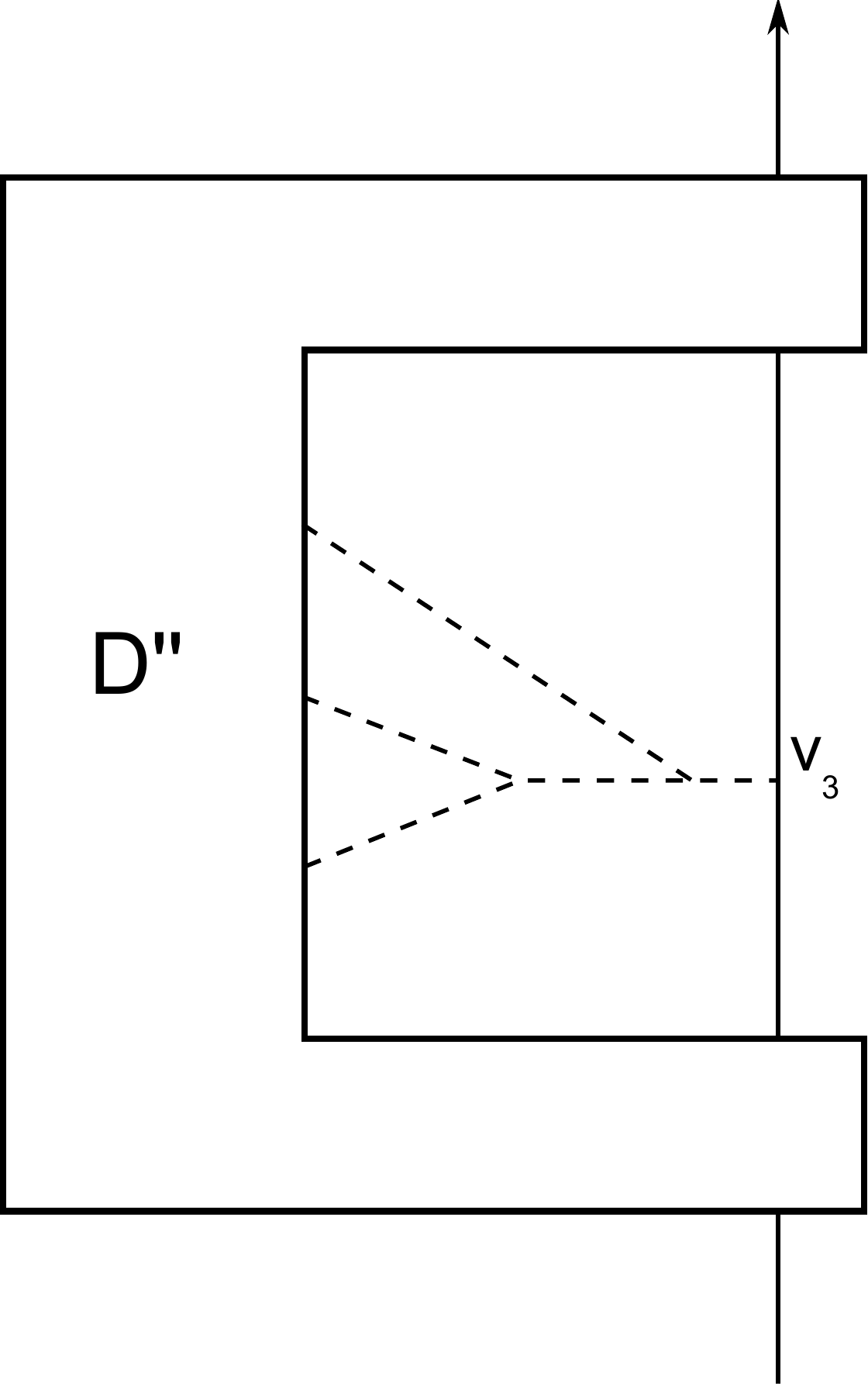} - \mathfig{4}{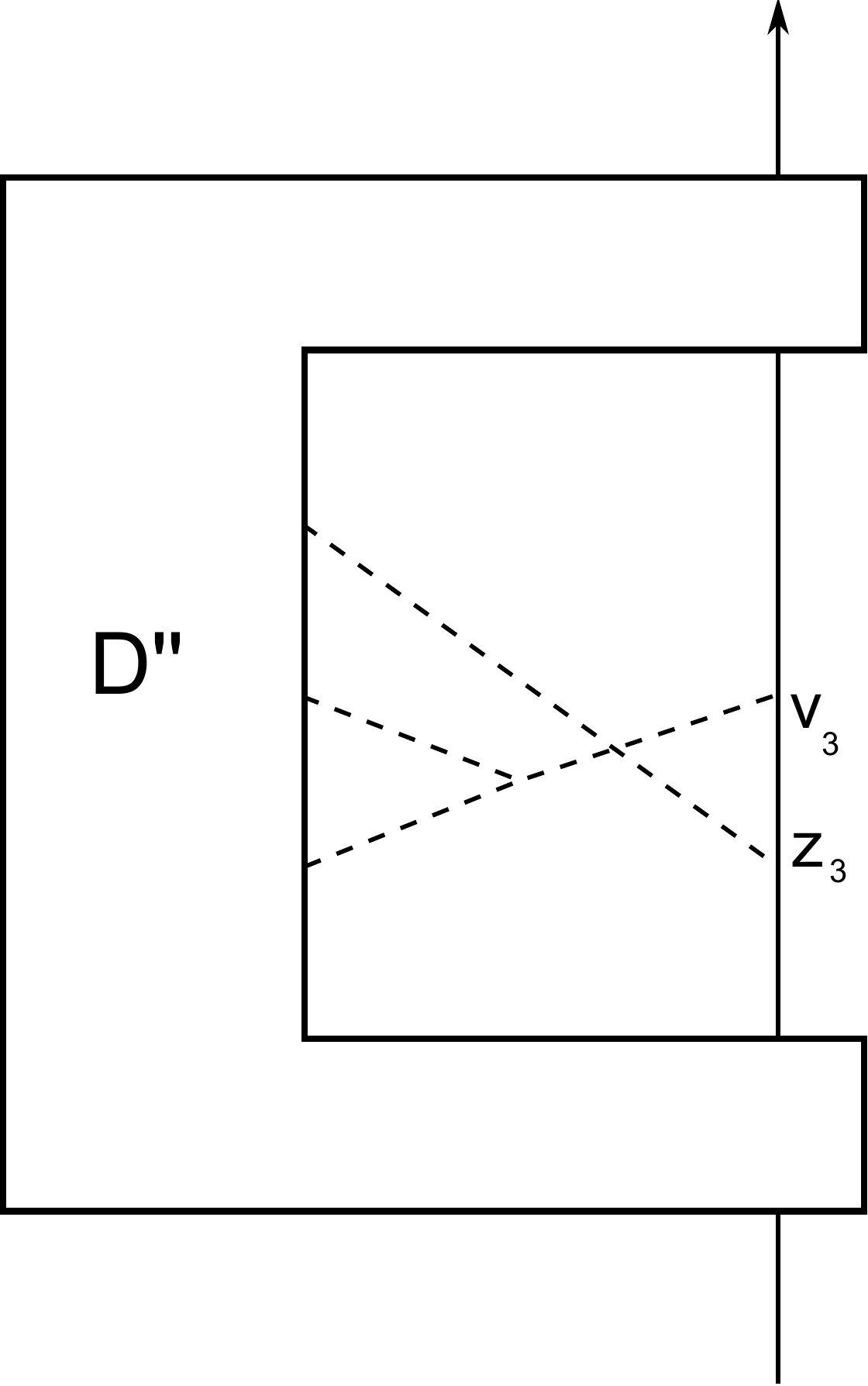} \approx$$

$$\approx \mathfignum{4}{section3_section32_proof_B2_21}{1} + \mathfignum{4}{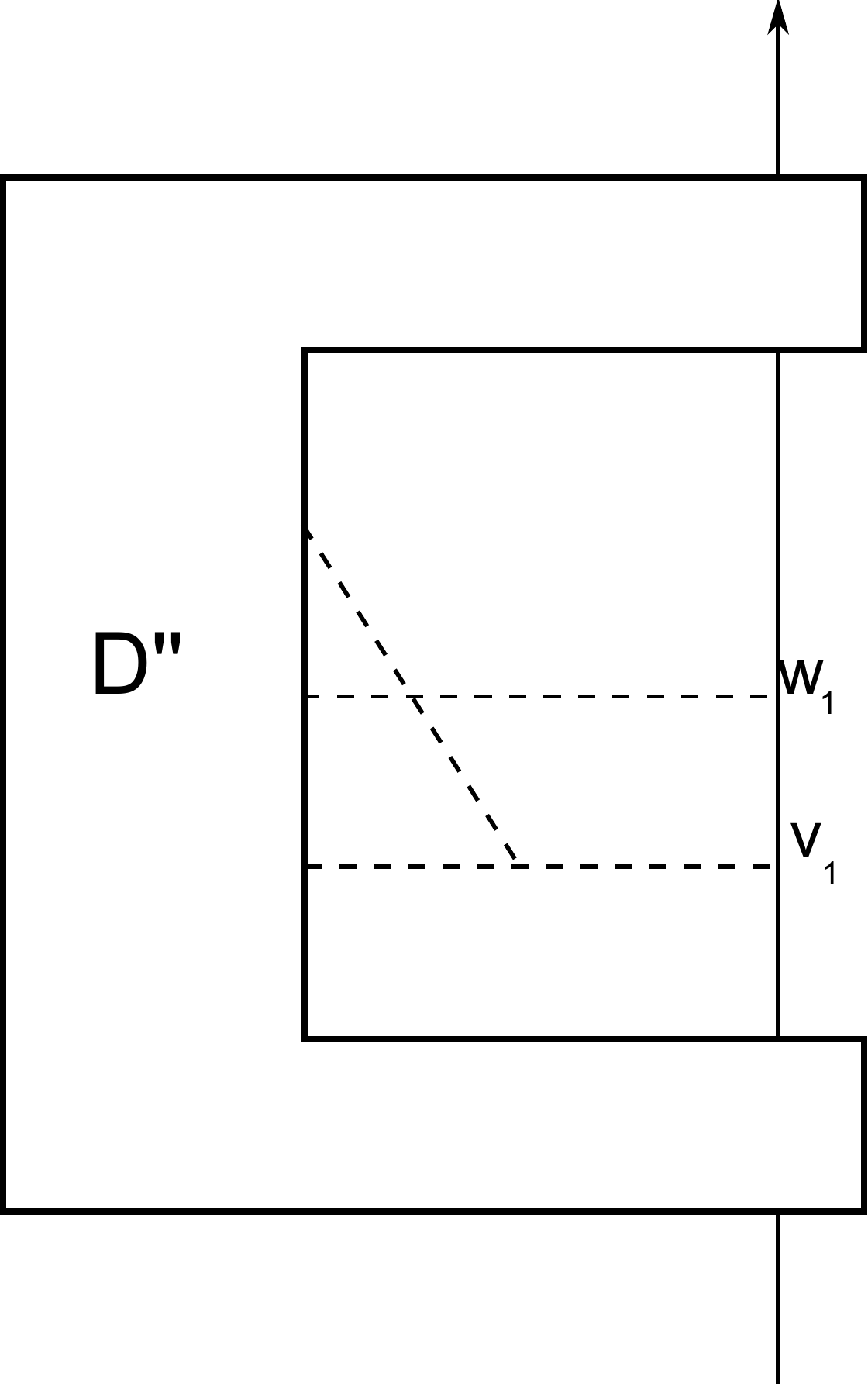}{2} + \mathfignum{4}{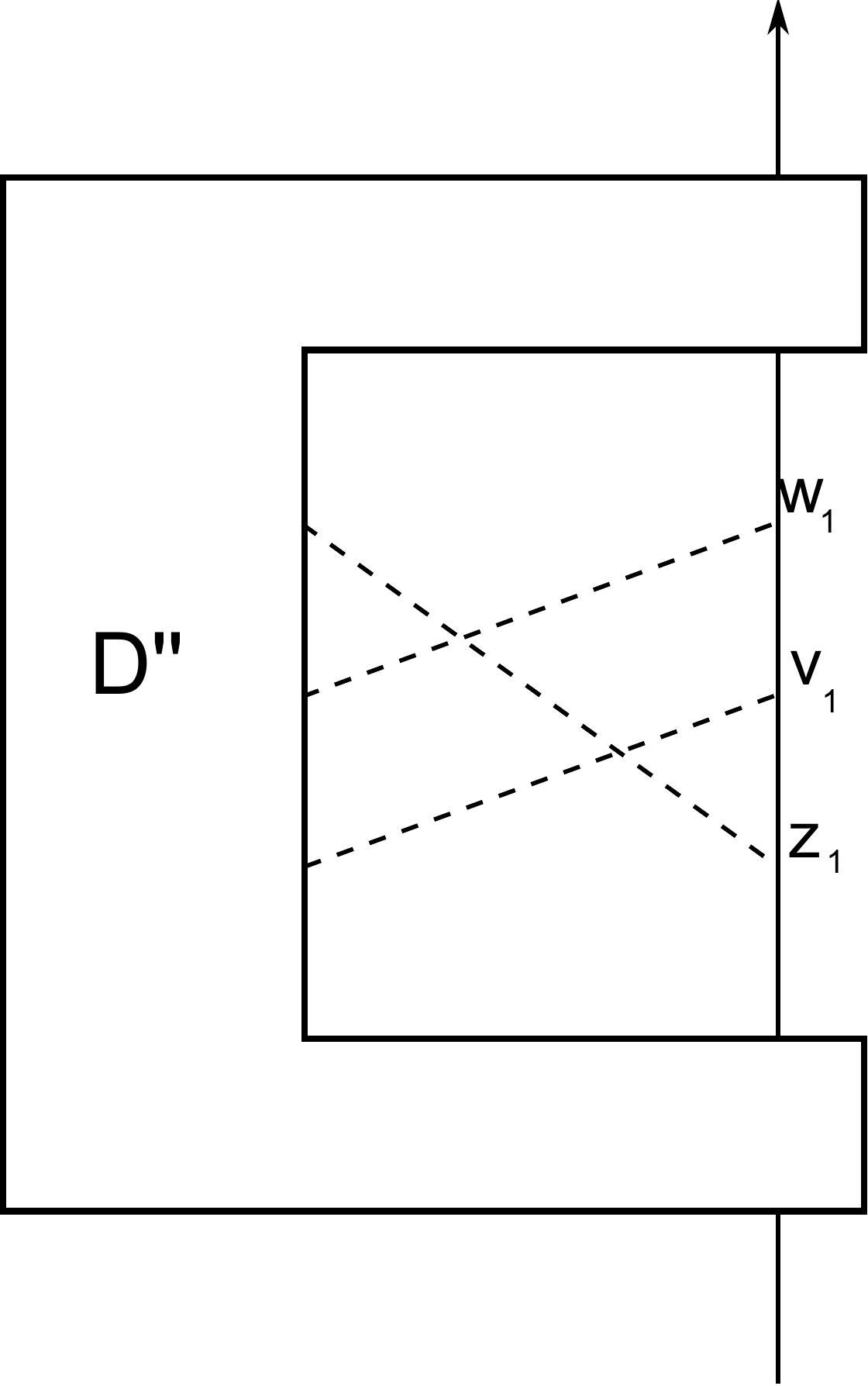}{3} -$$
$$- \mathfignum{4}{section3_section32_proof_B2_23}{4} - \mathfignum{4}{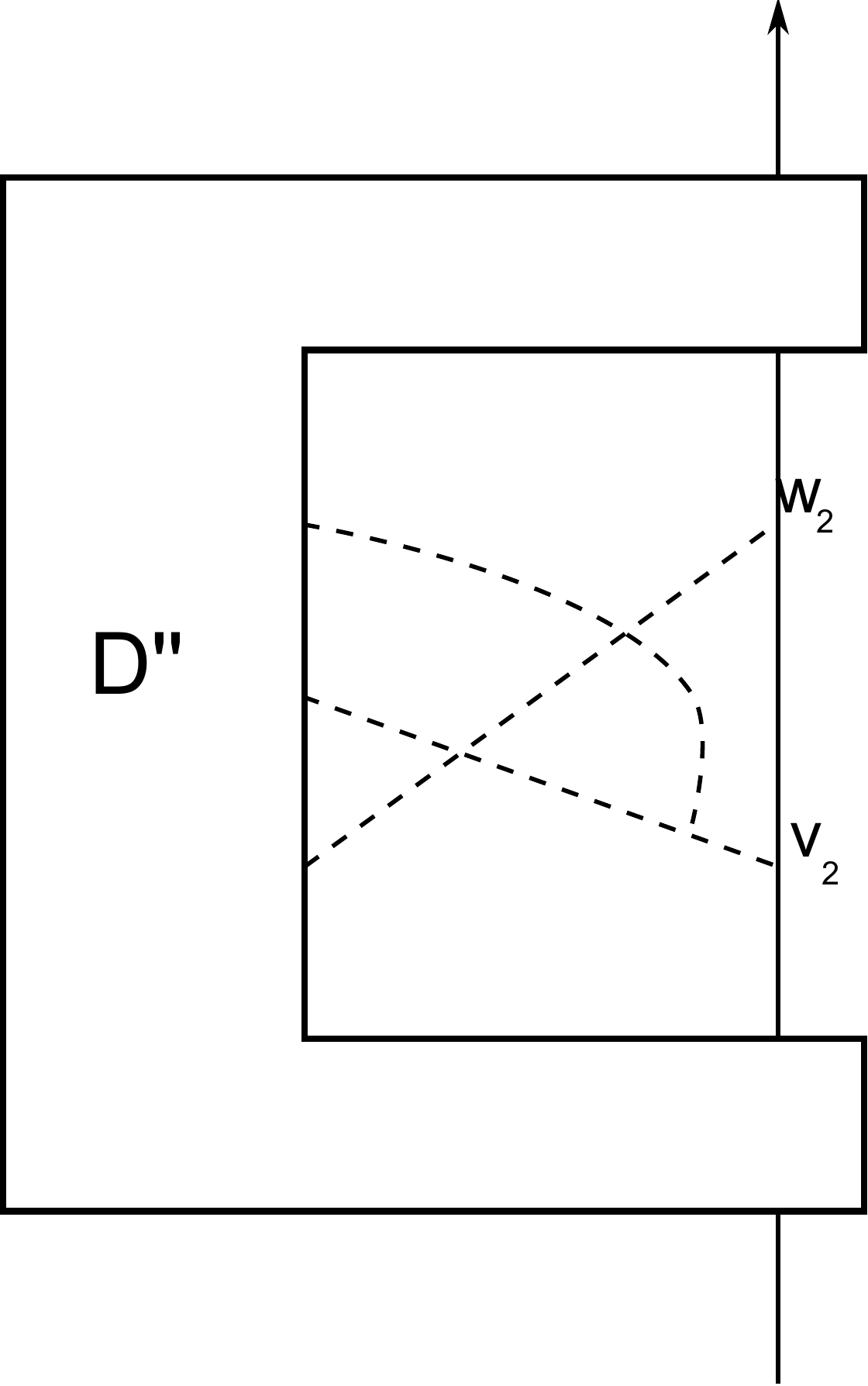}{5} - \mathfignum{4}{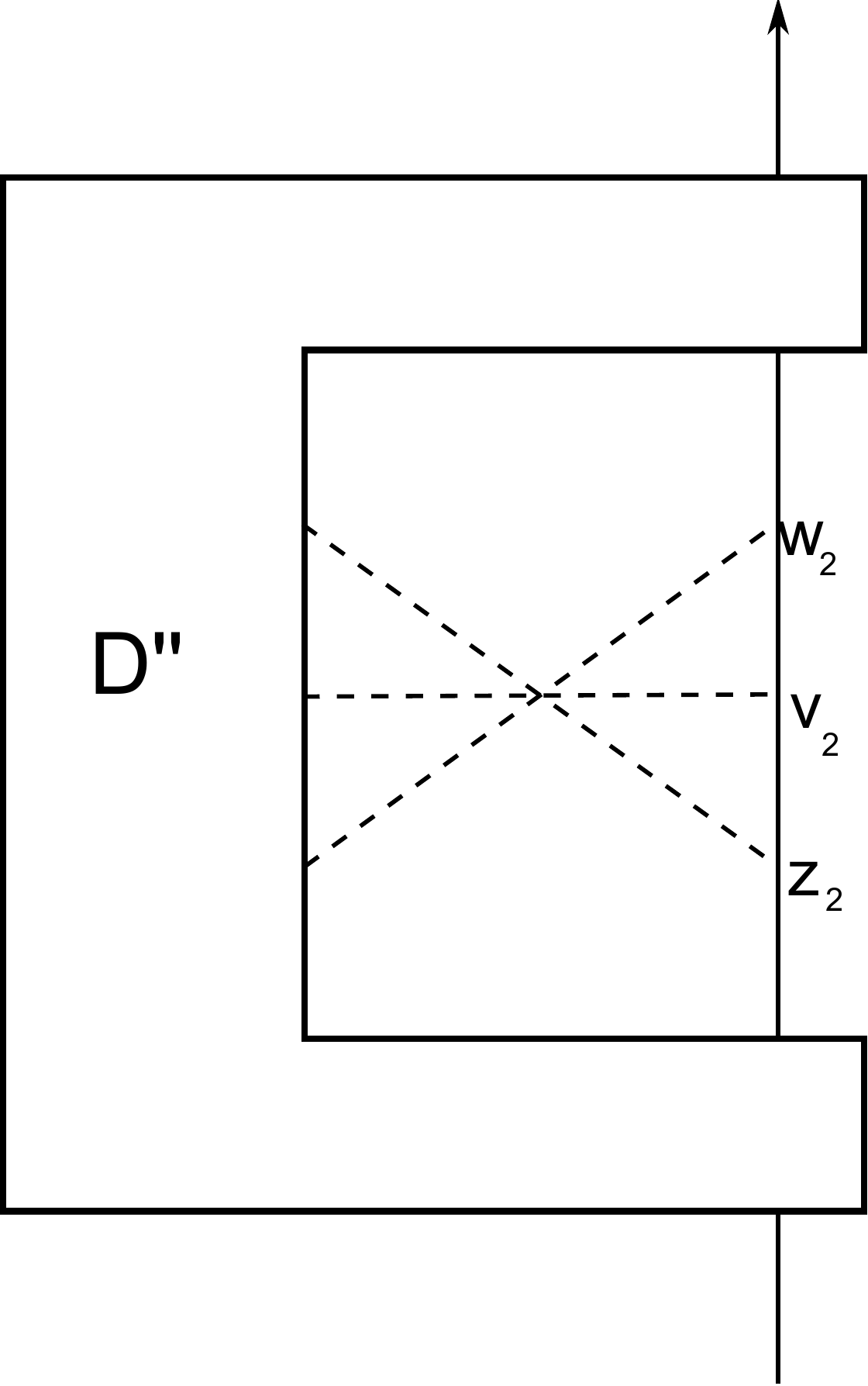}{6} - $$

$$-\mathfignum{4}{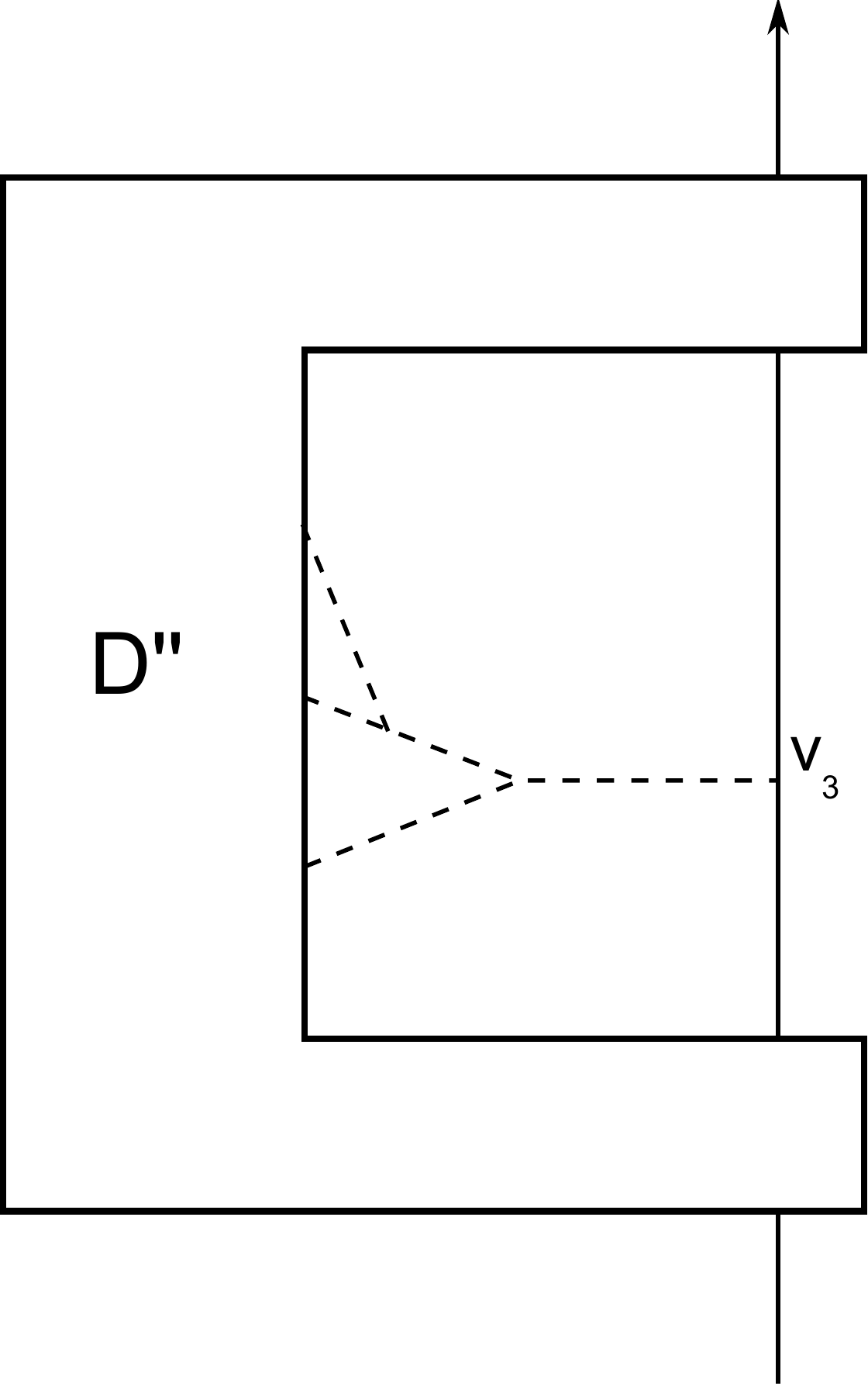}{7} - \mathfignum{4}{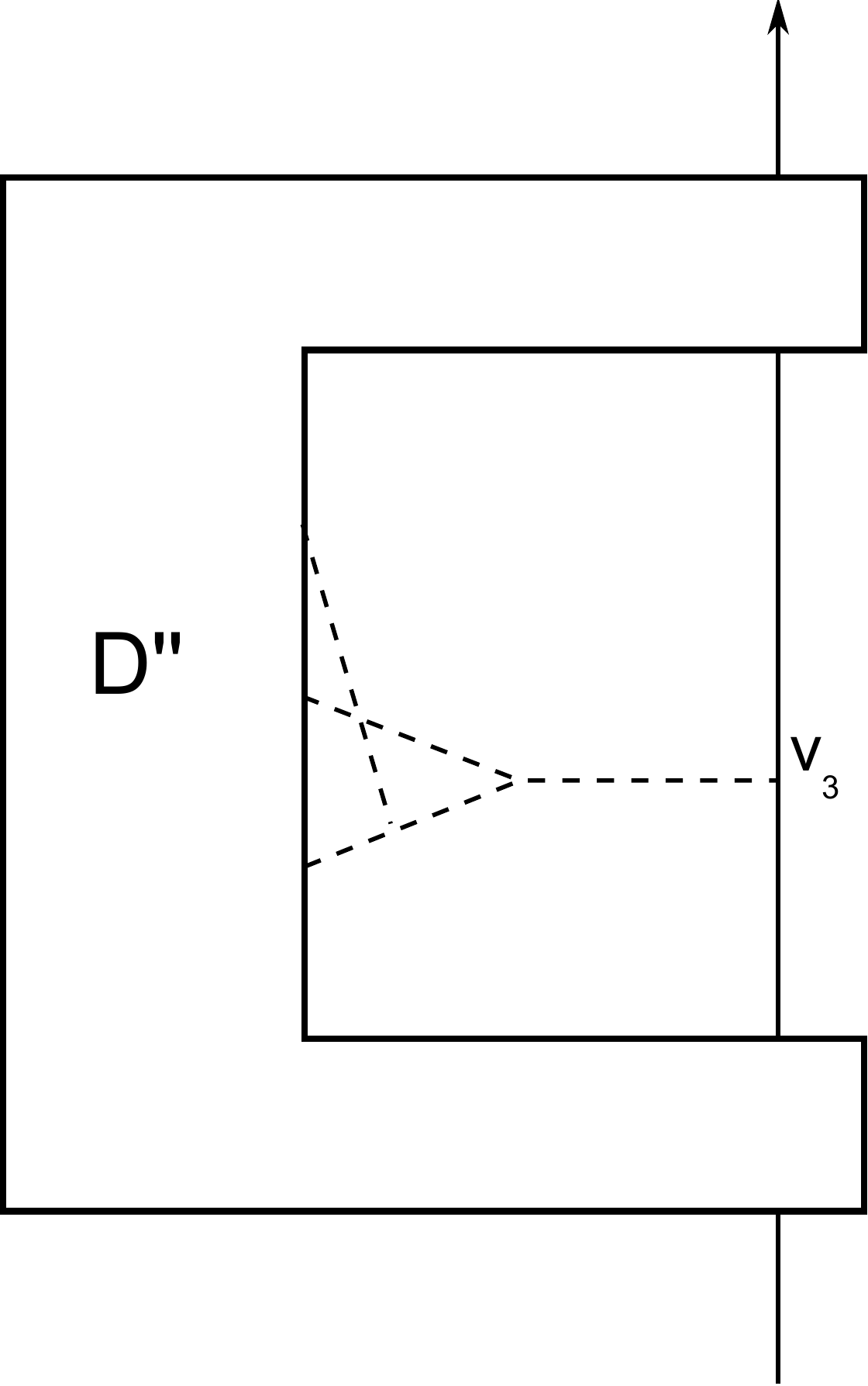}{8} - \mathfignum{4}{section3_section32_proof_B2_26}{9} \approx 0 $$

The last equivalence holds because \circlenum{1}, \circlenum{5} and \circlenum{7} are an IHX relation, \circlenum{2}, \circlenum{4} and \circlenum{8} are an IHX relation and \circlenum{3}, \circlenum{6} and \circlenum{9} are an STU relation.

\end{itemize}
\item[C.]
If $\text{v}_3$ is the lonely vertex immediately above the highest non-lonely vertex, then:
$$u=u_1-u_2-u_3=\mathfig{4}{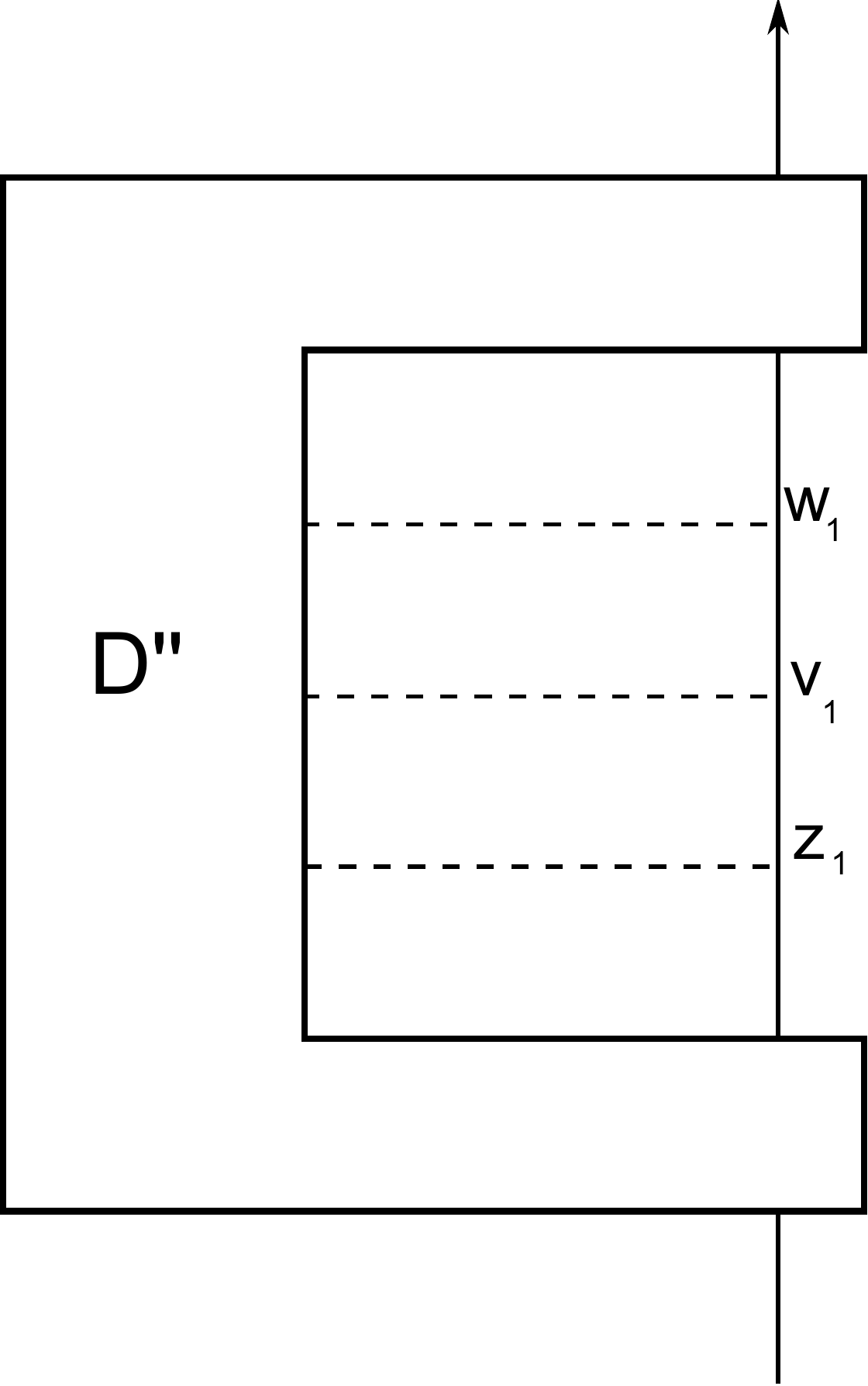}-\mathfig{4}{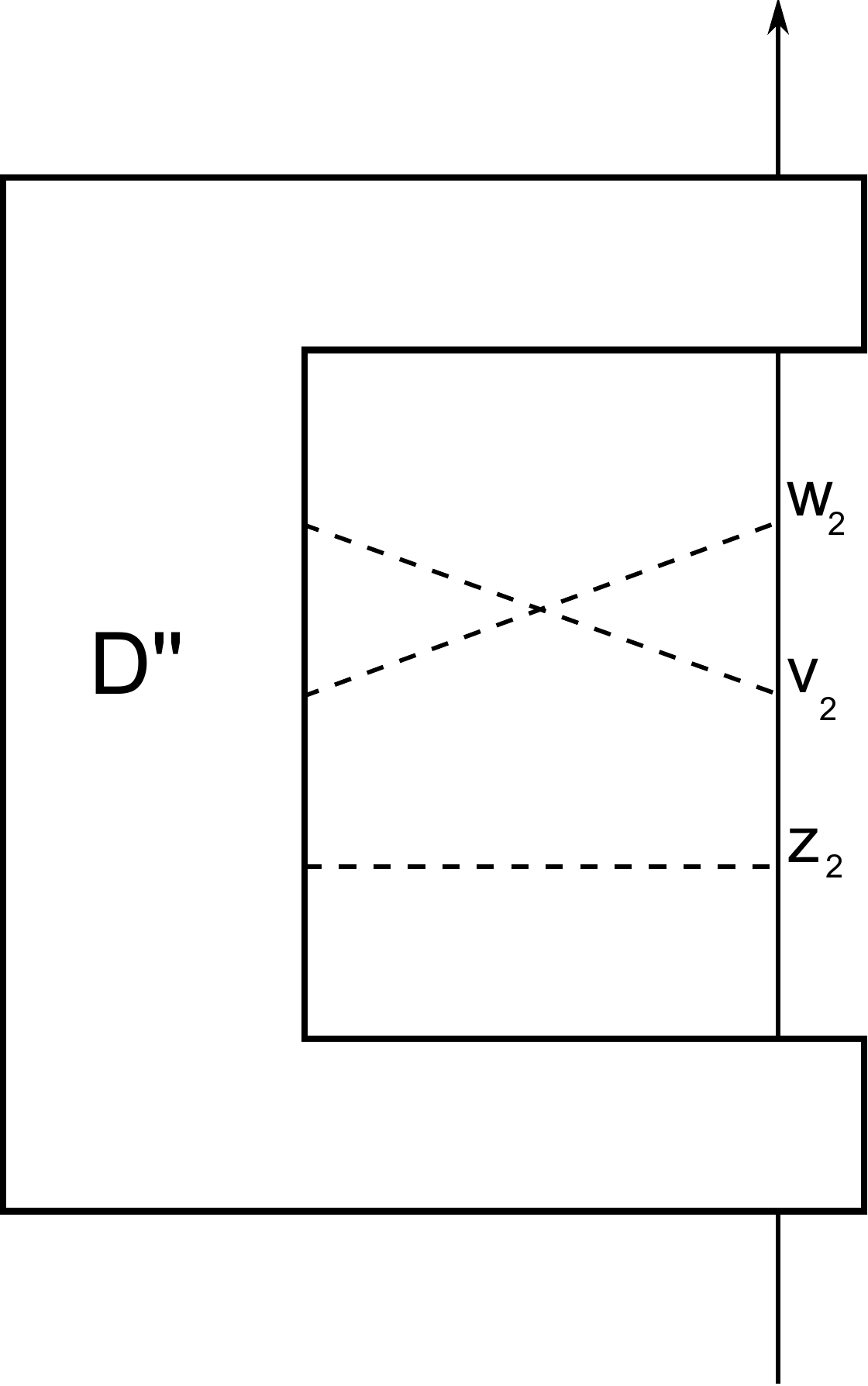}-\mathfig{4}{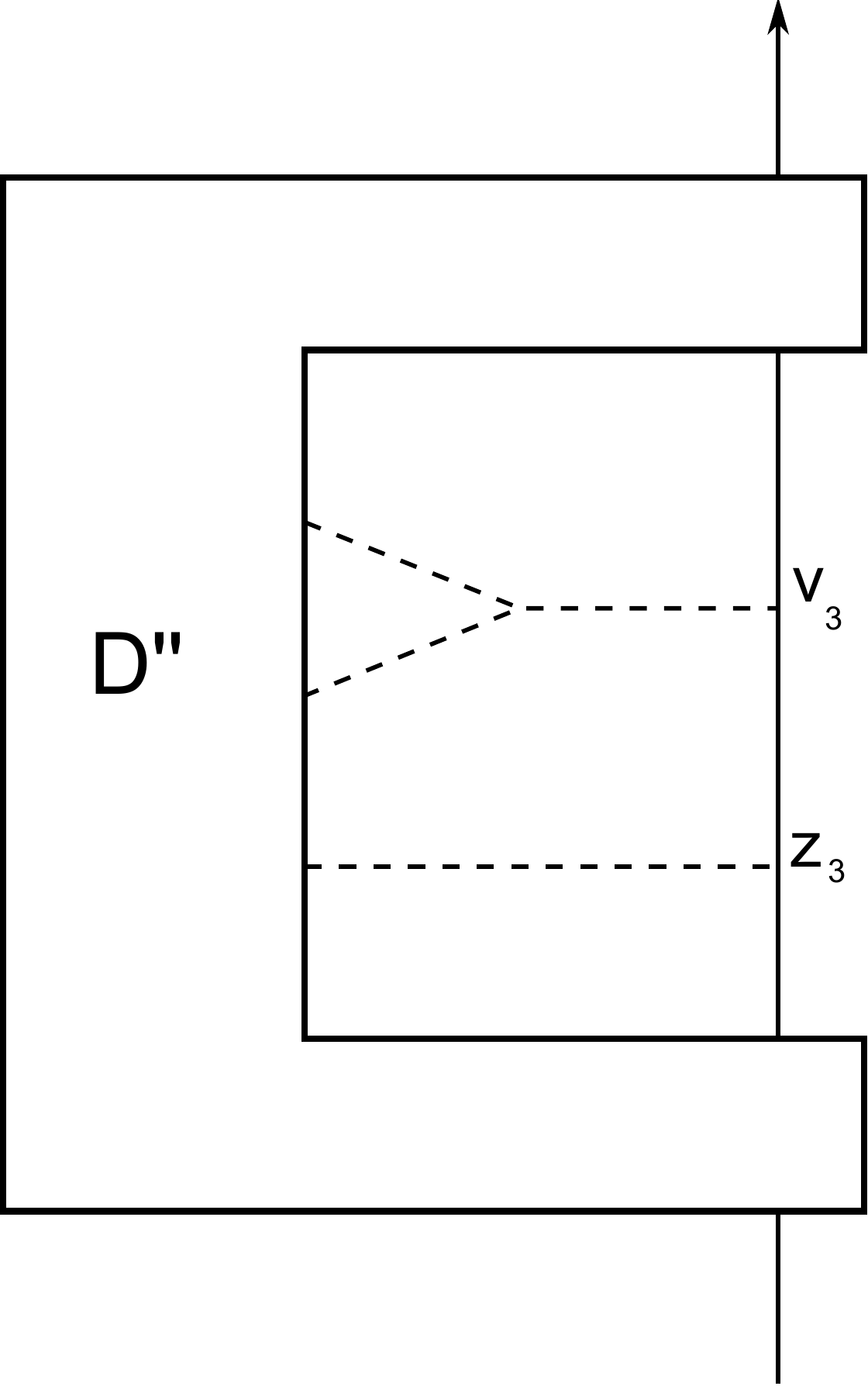}$$
and
$$\eta_m(u)\approx \mathfig{4}{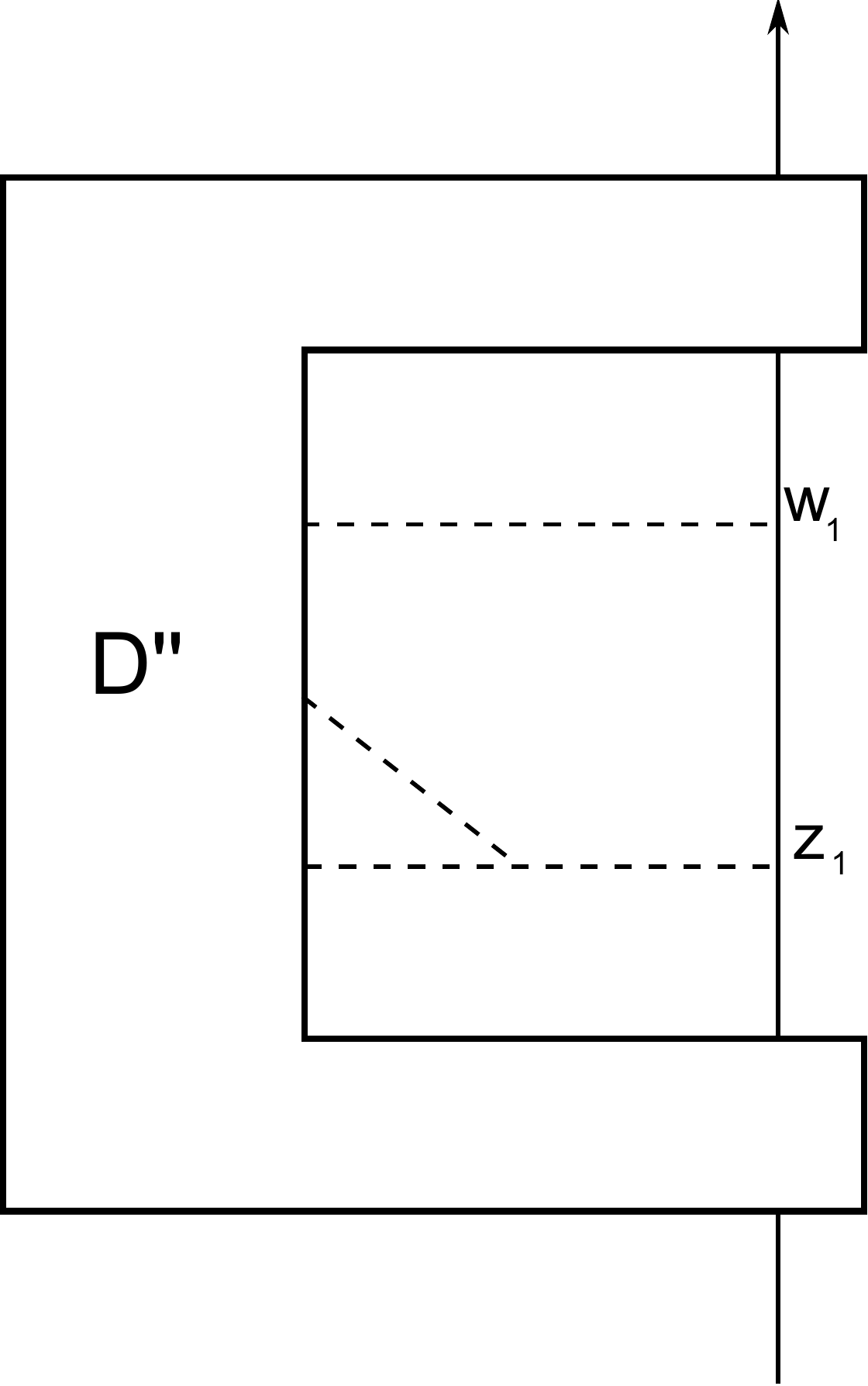} + \mathfig{4}{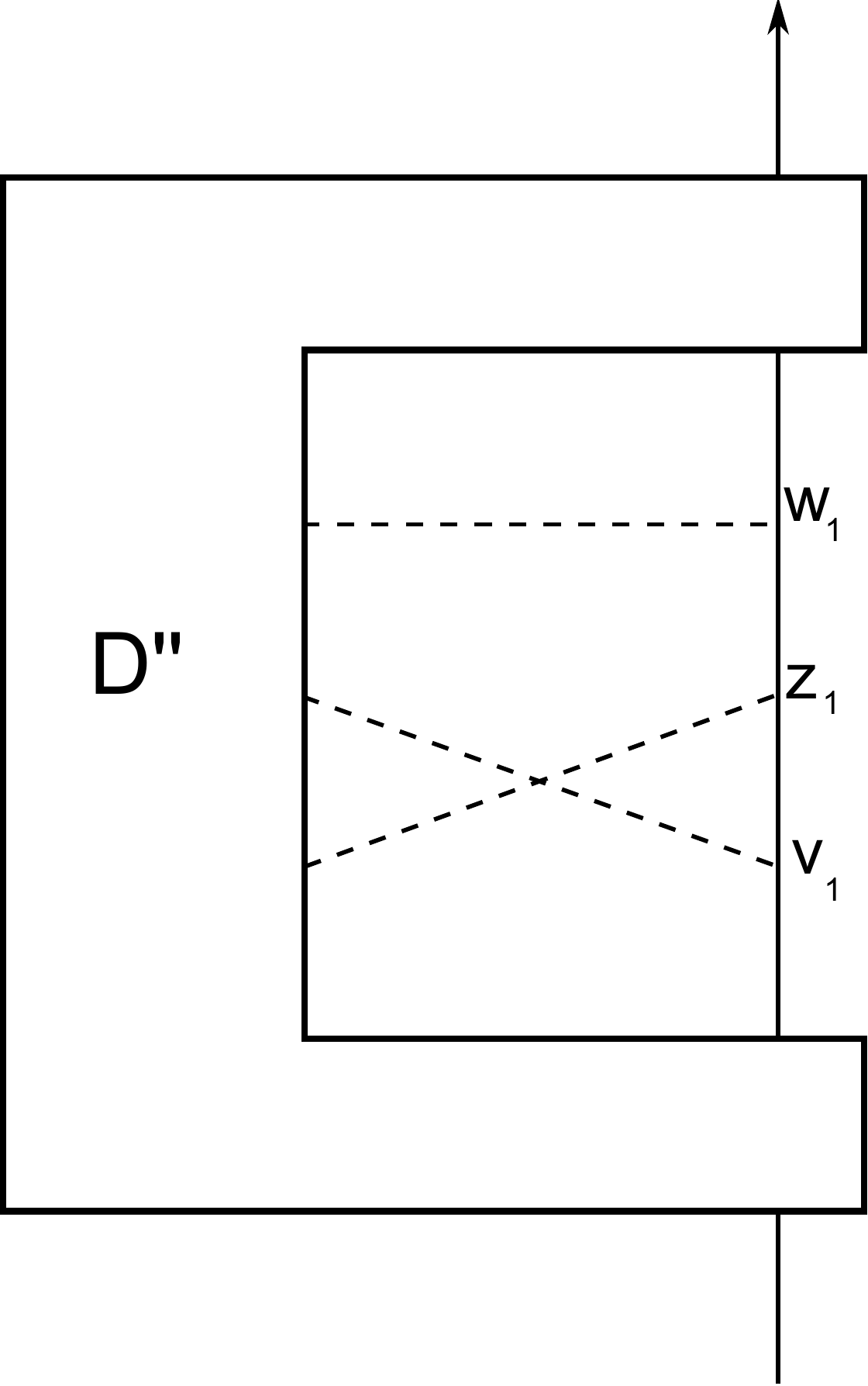} - \mathfig{4}{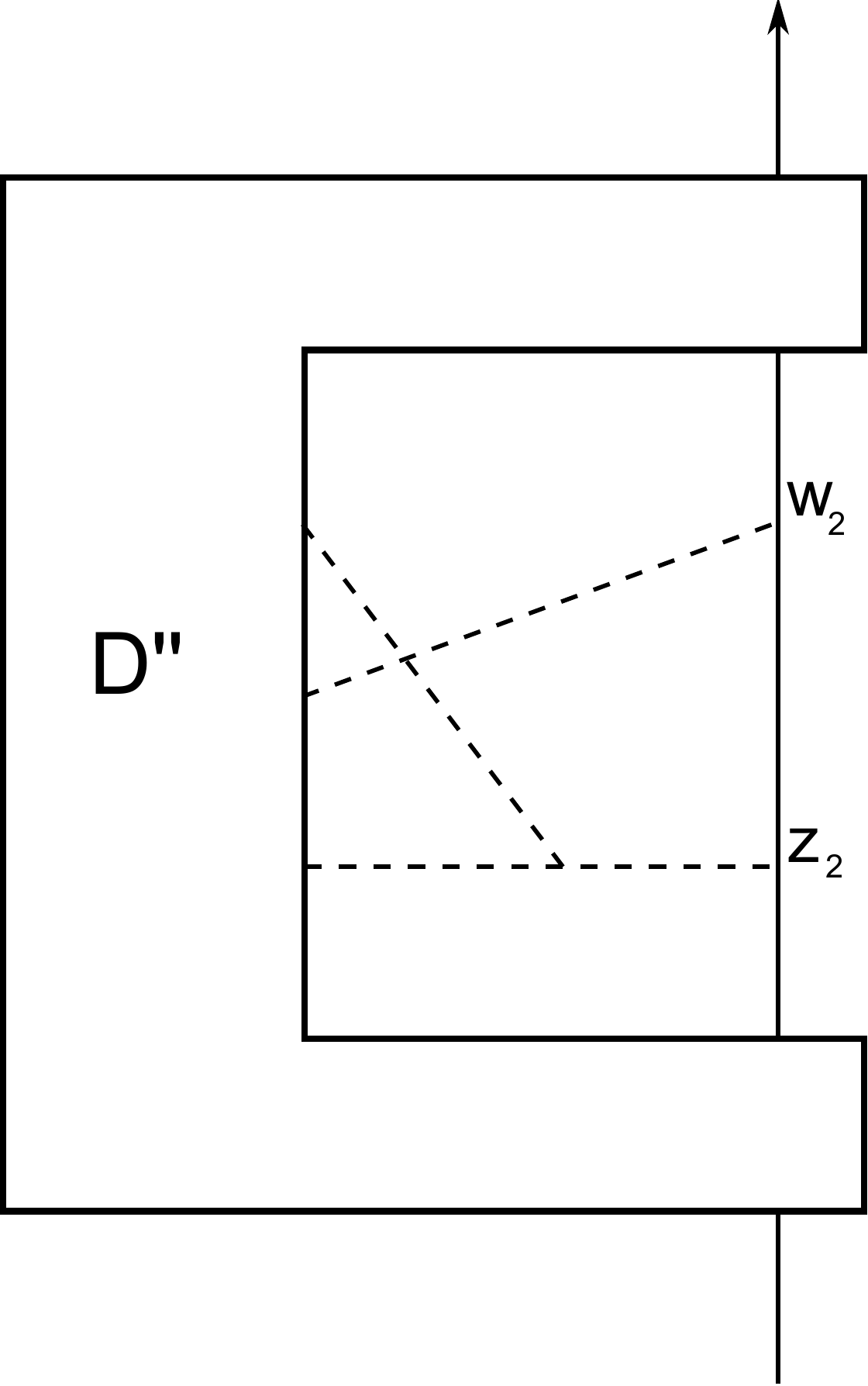} - $$

$$- \mathfig{4}{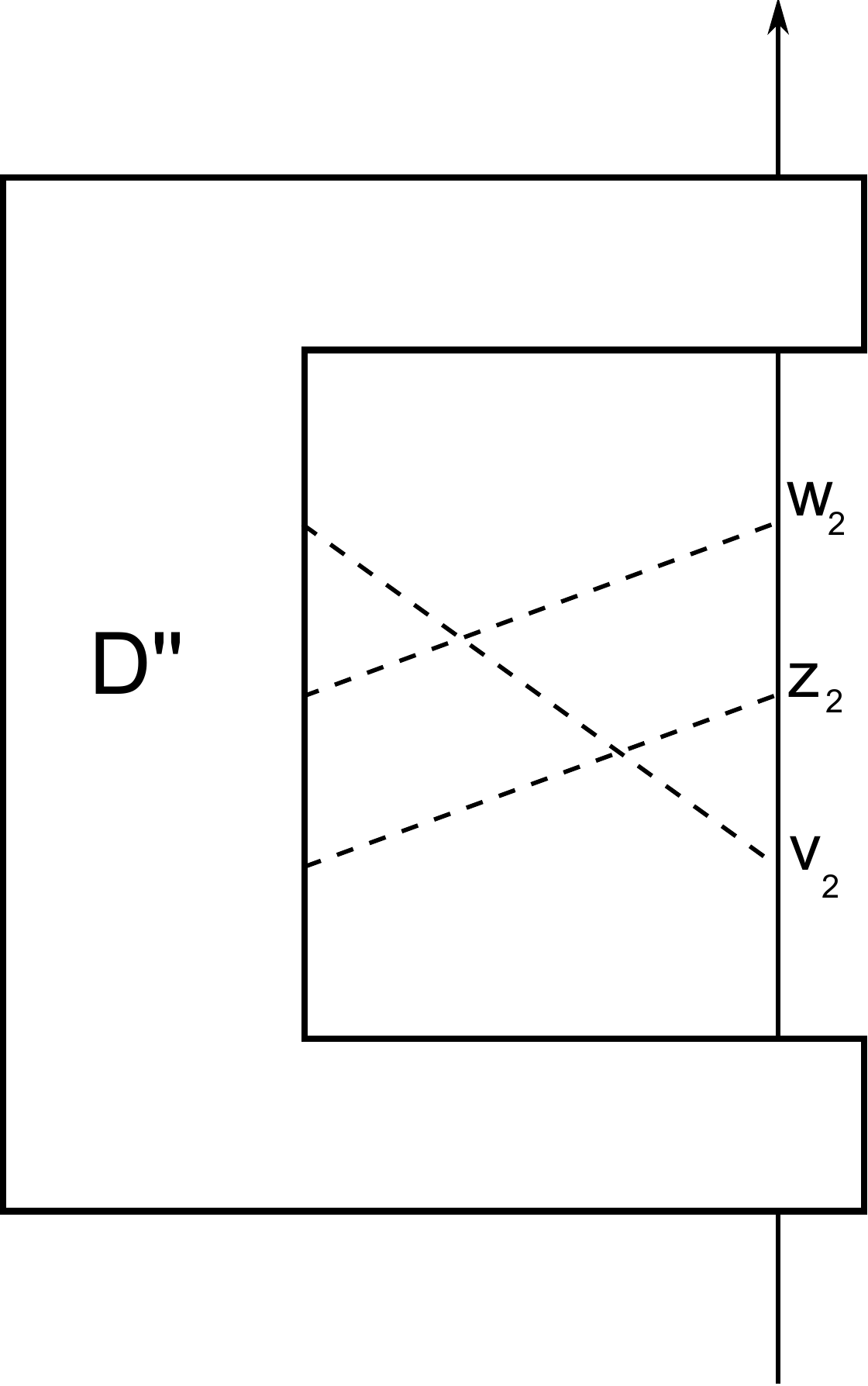} - \mathfig{4}{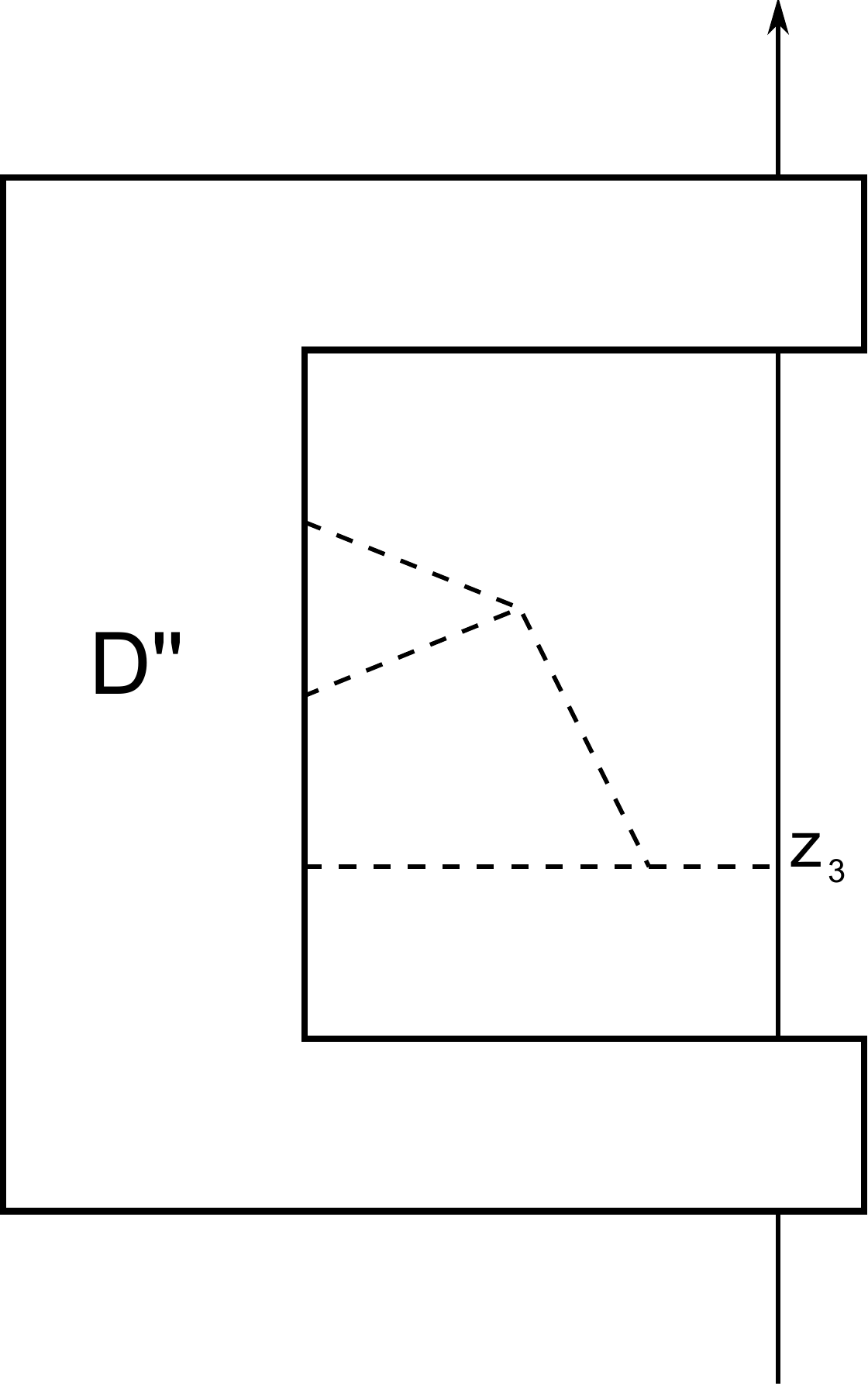} - \mathfig{4}{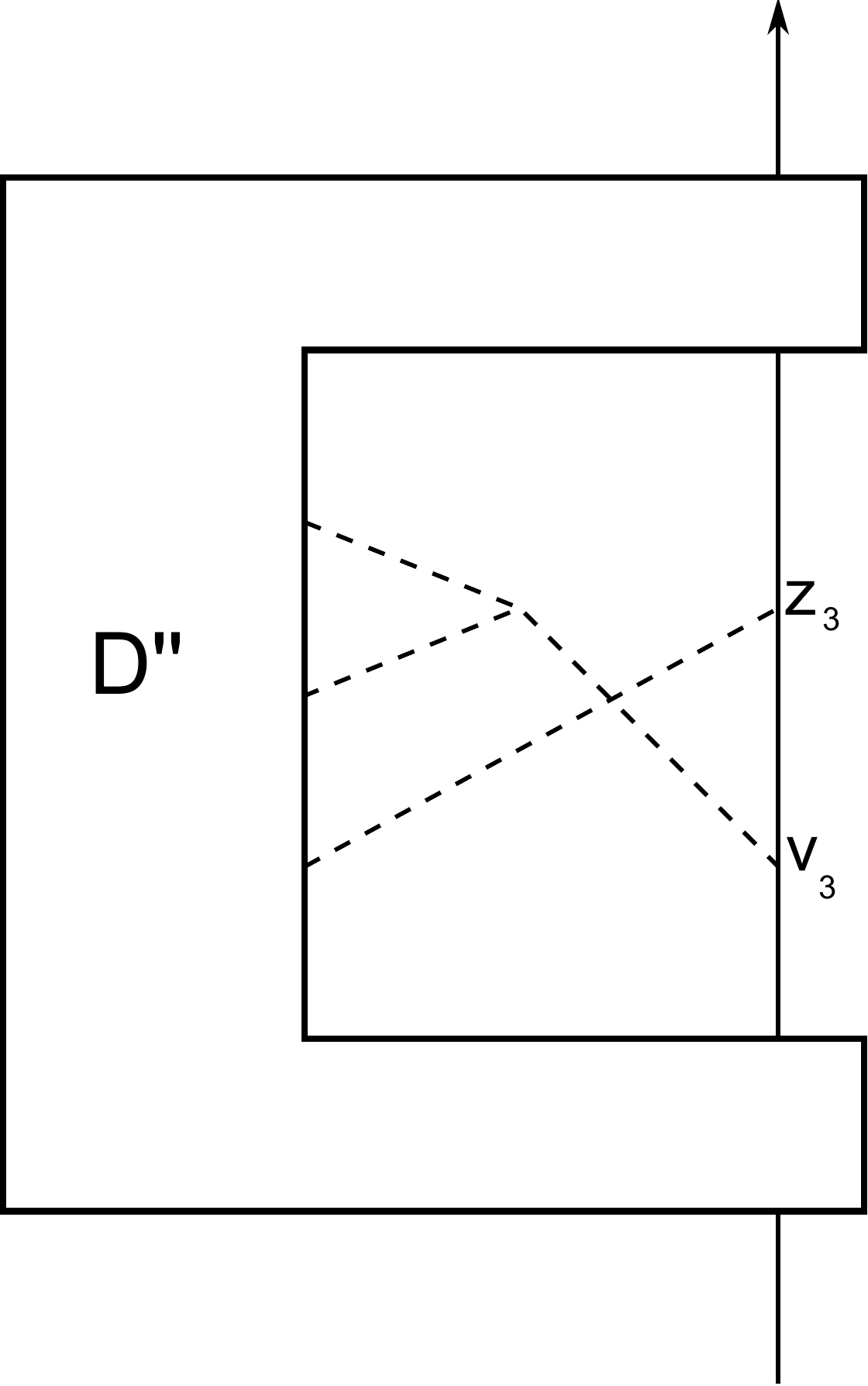} \approx $$

$$ \approx \mathfigbox{4}{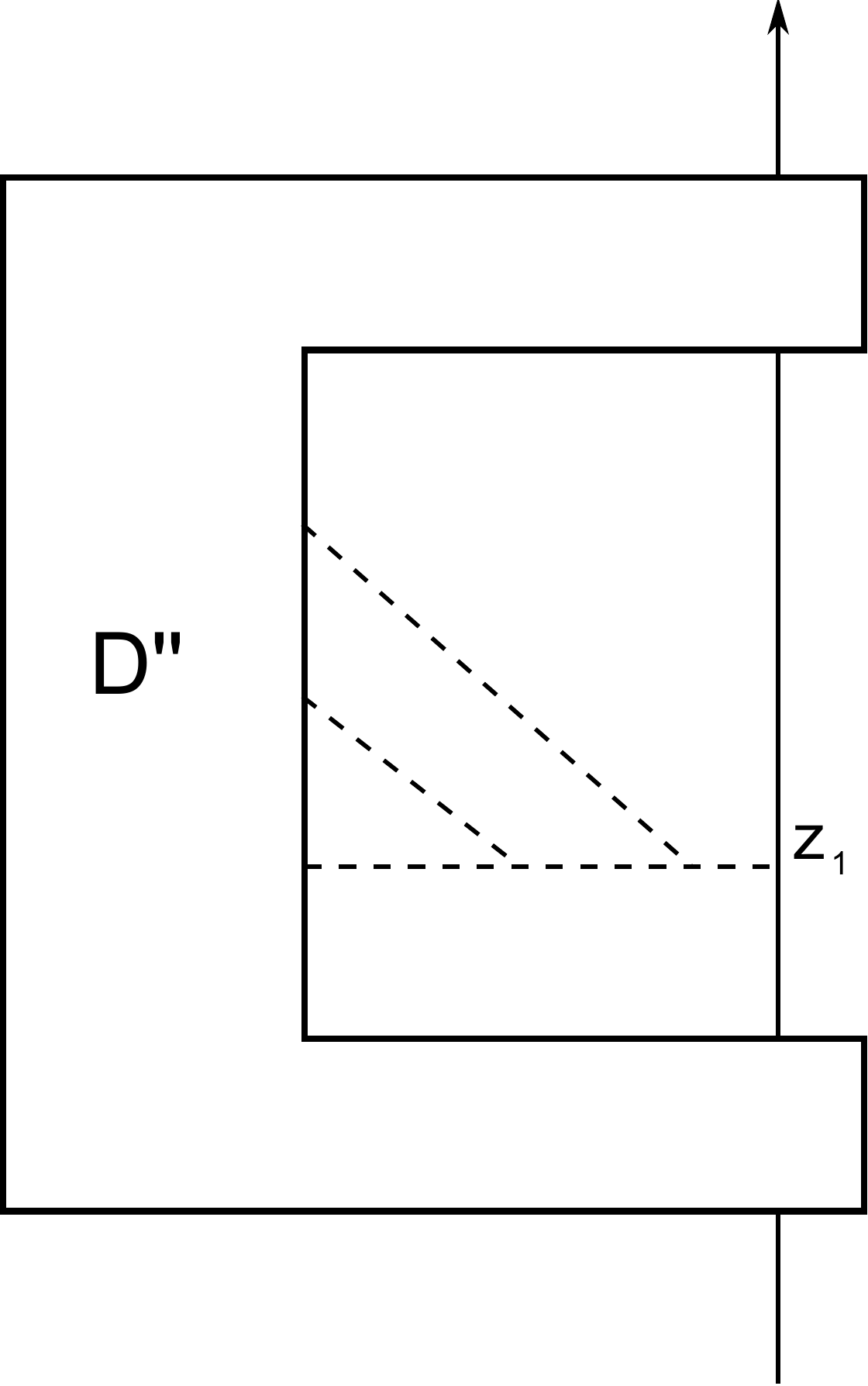}{1} + \mathfigbox{4}{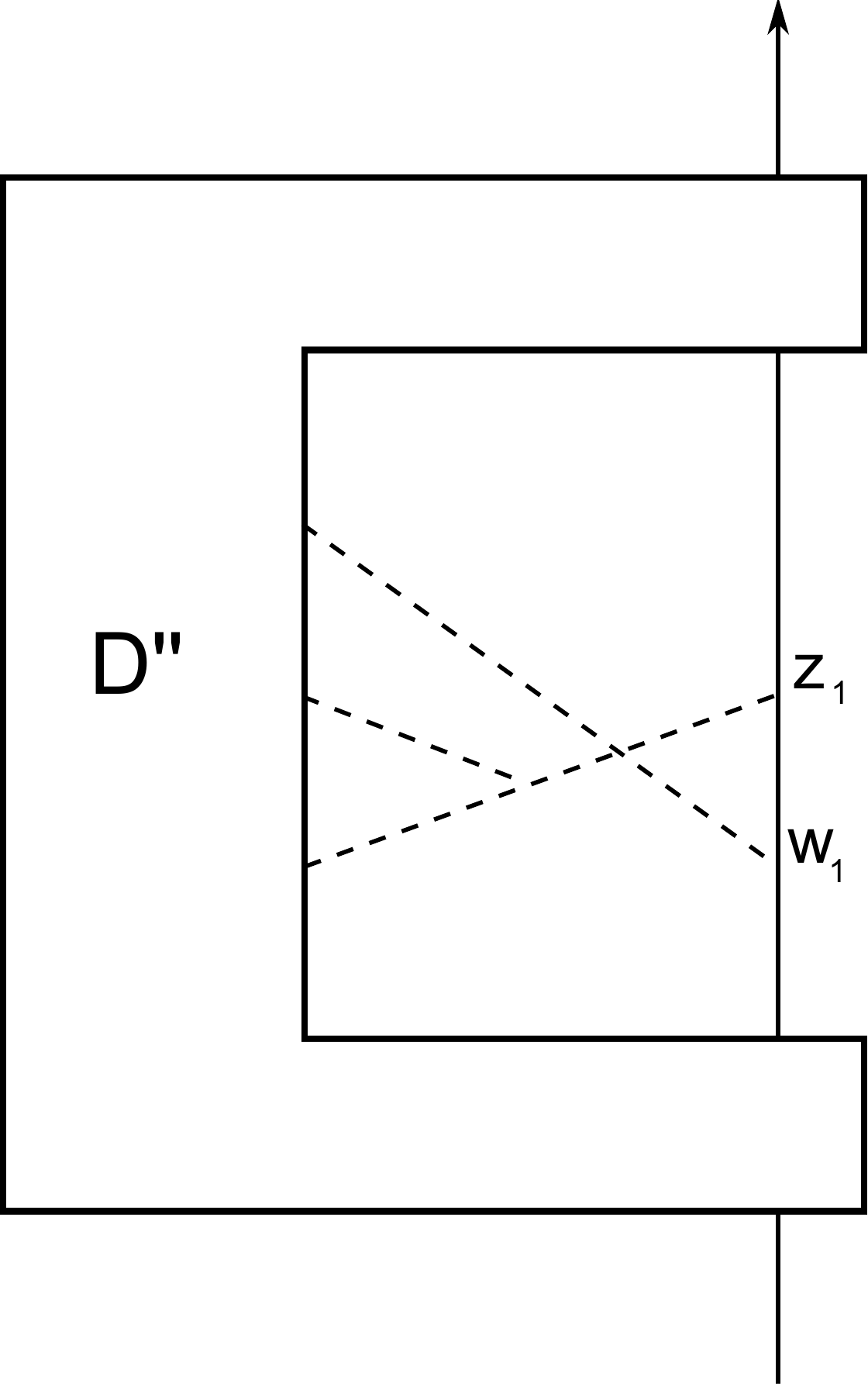}{2} + \mathfigbox{4}{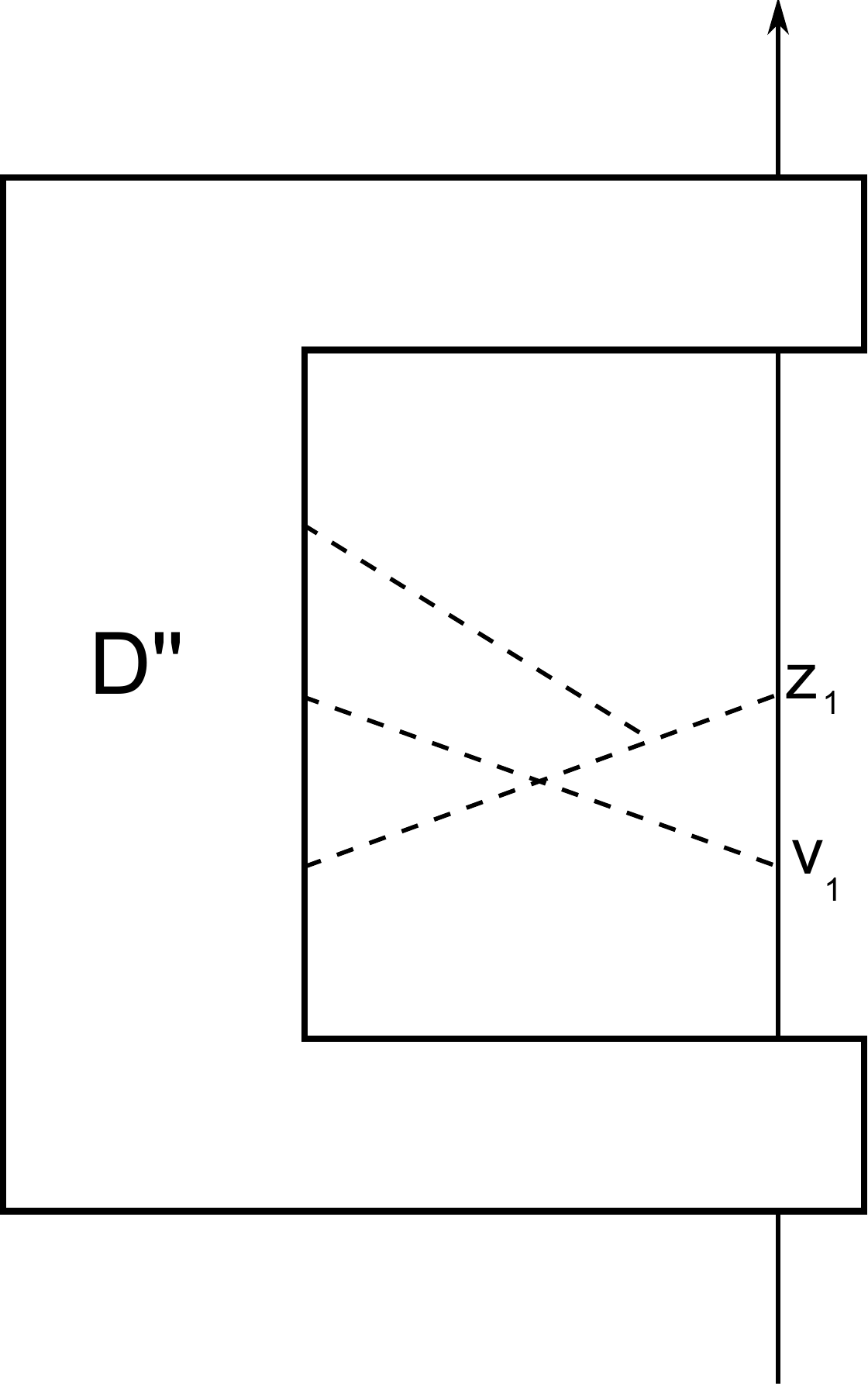}{3} +$$

$$ + \mathfigbox{4}{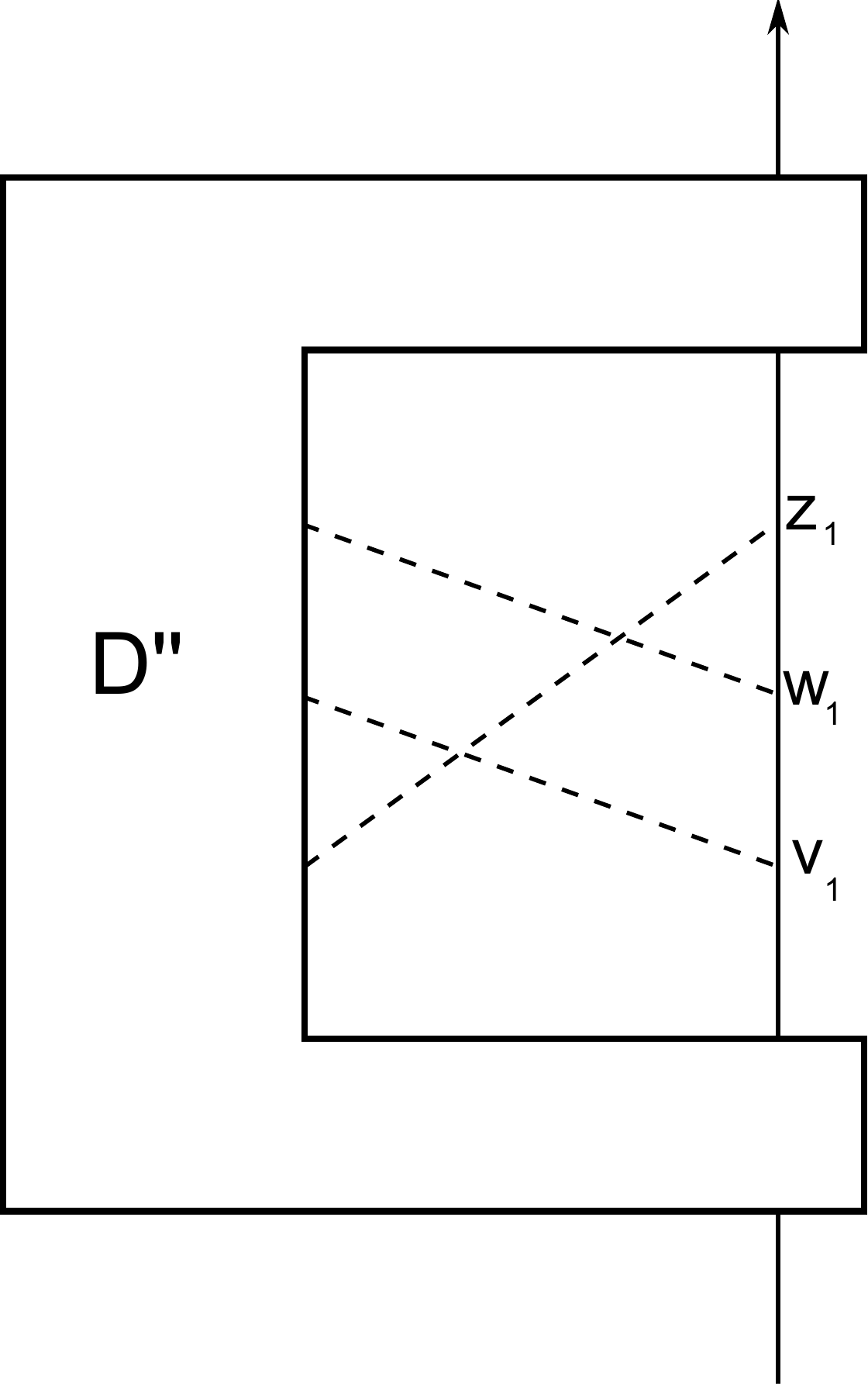}{4} - \mathfigbox{4}{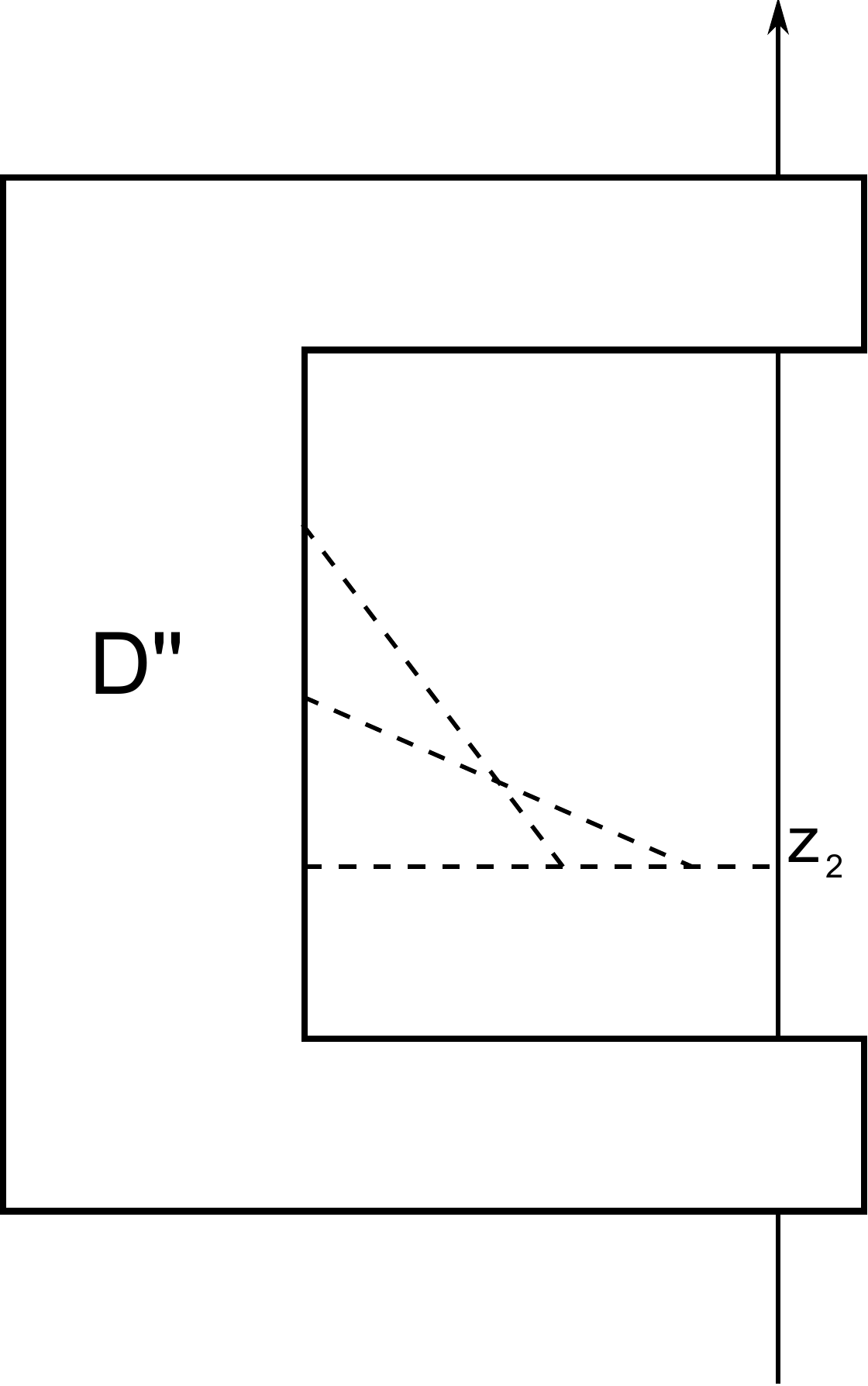}{5} - \mathfigbox{4}{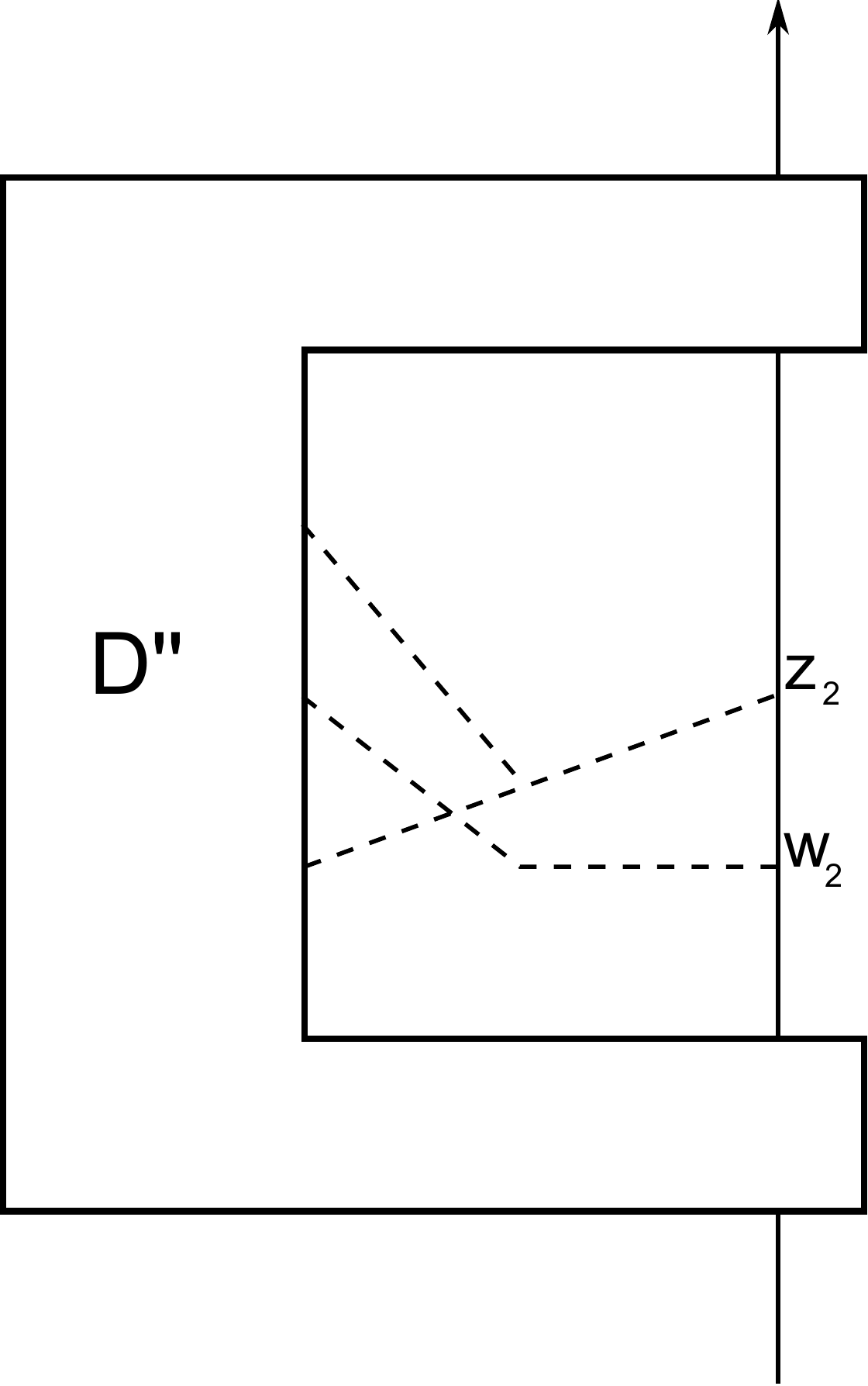}{6} - $$

$$-\mathfigbox{4}{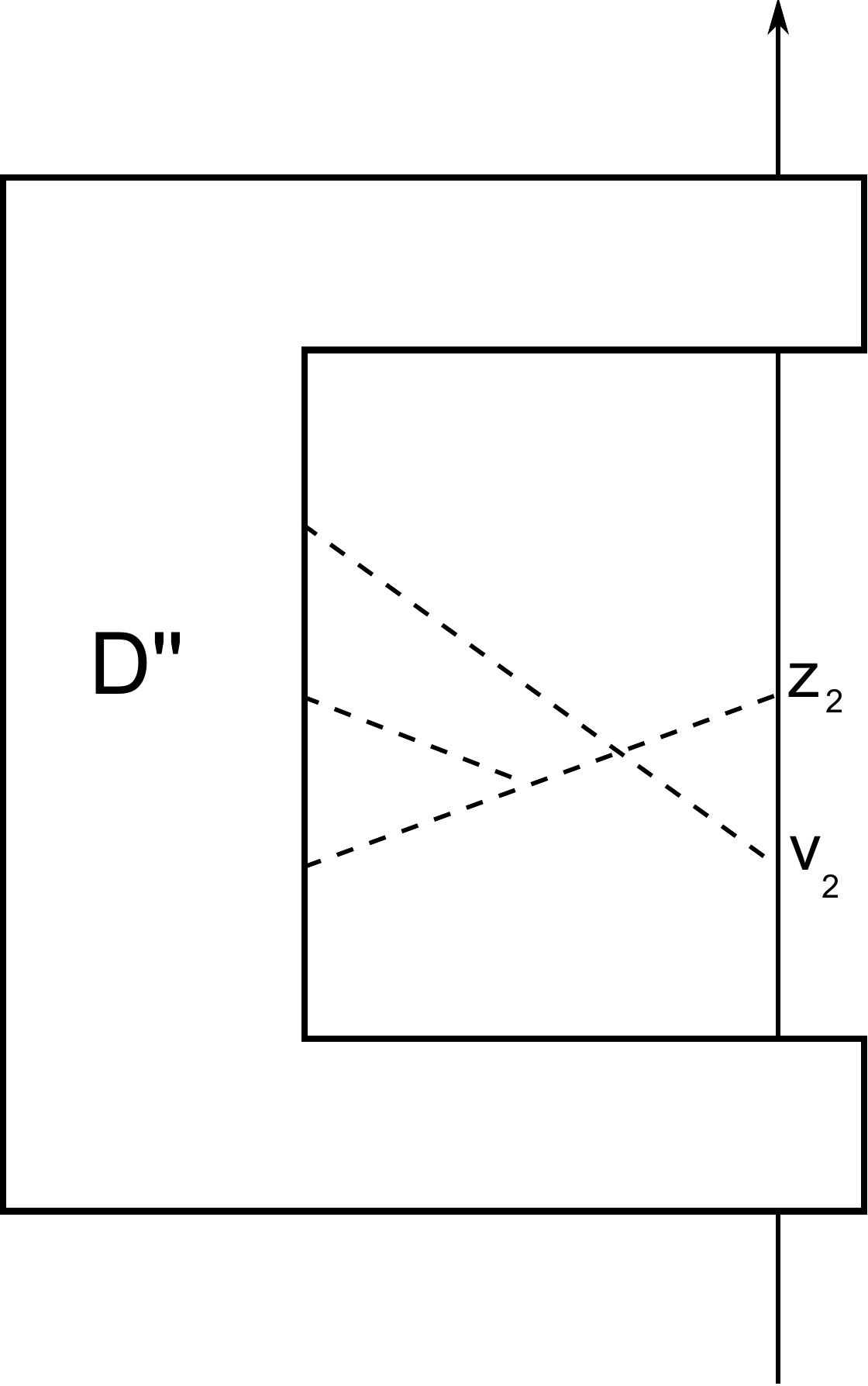}{7} - \mathfigbox{4}{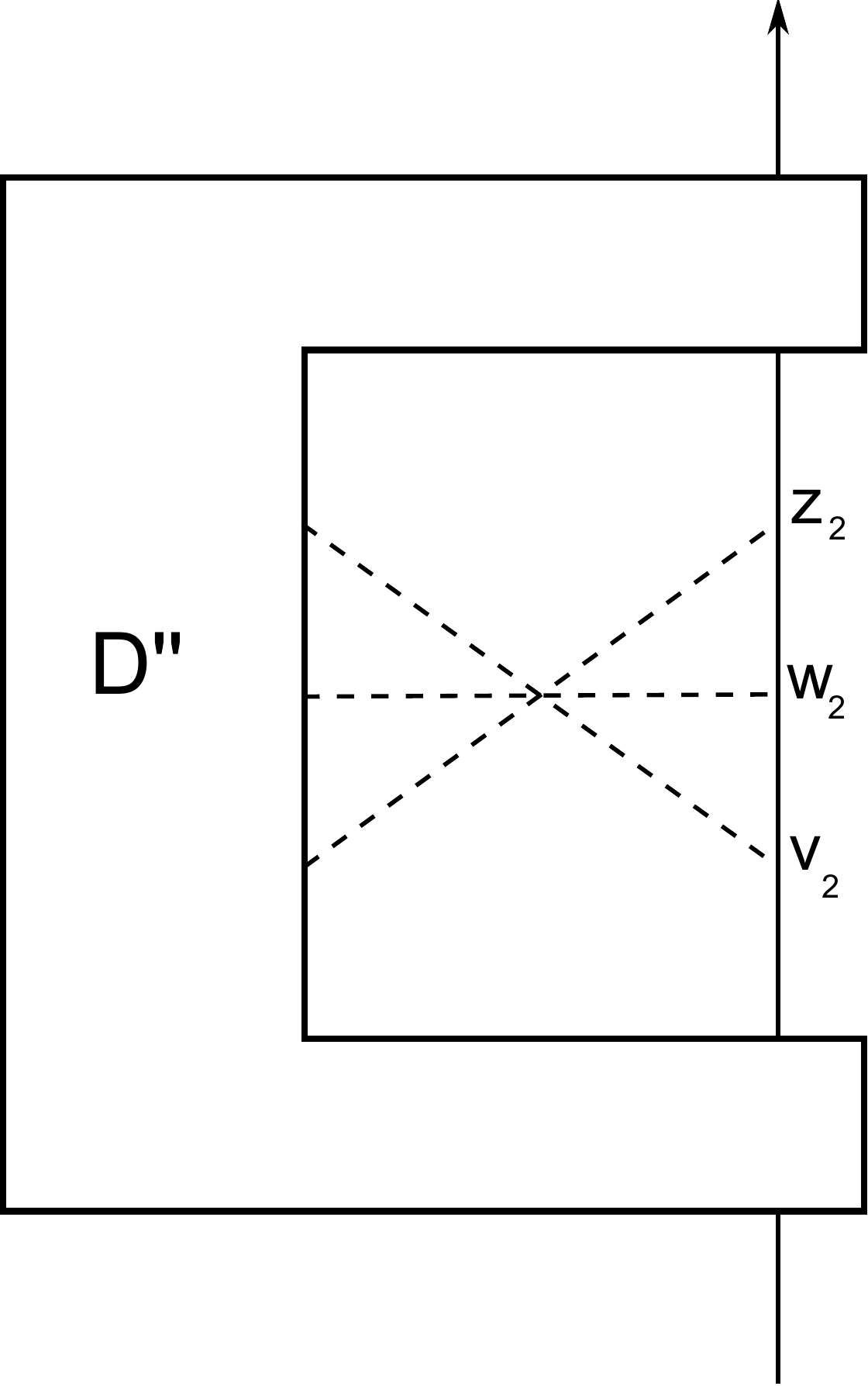}{8} - \mathfigbox{4}{section3_section32_proof_C_25}{9} -$$

$$- \mathfigbox{4}{section3_section32_proof_C_26}{10} \approx 0 $$

The last equivalence holds because \framebox{1}, \framebox{5} and \framebox{9} are an IHX relation, \framebox{2} cancels \framebox{7}, \framebox{3}, cancels \framebox{6}, and \framebox{4}, \framebox{8} and \framebox{10} are an STU relation.

\end{enumerate}

This completes the proof that $\eta_m:L^m\rightarrow L^{m-1}$ is well defined. It is easy to verify that it is the inverse of $l_{m-1}$, which completes the proof that $l$ is an isomorphism.

\end{proof}

We now have a representation of $\mathcal{A}_1^{yp}(\uparrow^n)$ as the quotient of the vector space ${S_\mathbb{F}(L\mathbb{D}_1^{yp}(\uparrow^n))}$ by STU relations alone. Recall the isomorphism $k:\mathcal{A}_1^{yp}(\uparrow^n)\rightarrow \mathcal{A}_1^{<p}(\uparrow^n)$ from section \ref{section_jacobi}. It is easy to see from the proof of proposition \ref{proposition_k} there that $k$ comes from an isomorphism $k:S_\mathbb{F}(\mathbb{D}_1^{yp}(\uparrow^n))\rightarrow S_\mathbb{F}(\mathbb{D}_1^{<p}(\uparrow^n))/_{\text{STU-like}}$. We denote by $LS_\mathbb{F}(\mathbb{D}_1^{<p}(\uparrow^n))$ the pre-image in $S_\mathbb{F}(\mathbb{D}_1^{<p}(\uparrow^n))$ of $k(S_\mathbb{F}(L\mathbb{D}_1^{yp}(\uparrow^n)))$.

Similarly, if $Rel^{yp}\subset S_\mathbb{F}(L\mathbb{D}_1^{yp}(\uparrow^n))$ is the subspace generated by all STU relations, denote by $Rel^{<p}$ the pre-image in $S_\mathbb{F}(\mathbb{D}_1^{<p}(\uparrow^n))$ of $k(Rel^{yp})$. It is easy to see that $Rel^{<p}\subset L S_\mathbb{F}(\mathbb{D}_1^{<p}(\uparrow^n))$ is the subspace generated by all STU and STU-like relations.

Finally we have:
\[\xymatrix{
S_\mathbb{F}(L\mathbb{D}_1^{yp}(\uparrow^n))/_{Rel^{yp}}\ar[r]^k_\simeq\ar[d]_\simeq &
L S_\mathbb{F}(\mathbb{D}_1^{<p}(\uparrow^n))/_{Rel^{<p}}\ar[d]^j \\
\mathcal{A}_1^{yp}(\uparrow^n) \ar[r]^k_\simeq & \mathcal{A}_1^{<p}(\uparrow^n)}\]
This shows that the obvious map $j:L S_\mathbb{F}(\mathbb{D}_1^{<p}(\uparrow^n))/_{Rel^{<p}}\rightarrow \mathcal{A}_1^{<p}(\uparrow^n)$ is actually an isomorphism. Thus we got a presentation of $\mathcal{A}_1^{<p}(\uparrow^n)$ without $I_1^<$ relations.

There is no simple description of $LS_\mathbb{F}(\mathbb{D}_1^{<p}(\uparrow^n))$ (i.e. it is not a span of a set of diagrams). However, it clearly contains $S_\mathbb{F}(R\mathbb{D}_1^{<p}(\uparrow^n))$. Therefore, the map $r:R\mathcal{A}^{<p}(\uparrow^n)\rightarrow\mathcal{A}_1^{<p}(\uparrow^n)$ can be decomposed as:
$$R\mathcal{A}^{<p}(\uparrow^n)\stackrel{r'}{\longrightarrow}LS_\mathbb{F}(\mathbb{D}_1^{<p}(\uparrow^n))/_{Rel^{<p}}\stackrel{j}{\longrightarrow}\mathcal{A}_1^{<p}(\uparrow^n)$$
where $r'$ is clearly injective, and therefore $r$ is also injective. This completes the proof of proposition \ref{proposition_r}.

\begin{remark}
\label{remark_restricted_diagram}
If $a\in\mathcal{A}_1^{<p}(\uparrow^n)$ is represented as a sum of diagrams with the property that each component has at least one vertex on one of the first $n-1$ strands, then $a$ is in the image of $r$.
\end{remark}

Clearly the image $u(RU\hat{\textbf{t}}_{1,n})\subset\mathcal{A}_1^{<p}(\uparrow^n)$ is contained in the image $r(R\mathcal{A}_1^{<p}(\uparrow^n))\subset \mathcal{A}_1^{<p}(\uparrow^n)$. Therefore $u$ induces a map $u_r:RU\hat{\textbf{t}}_{1,n}\rightarrow R\mathcal{A}_1^{<p}(\uparrow^n)$ satisfying $u=r\circ u_r$. To prove that $u|_{RU  \hat{ \textbf{t}_{1,n}}}$ is injective, it is enough to prove that $u_r$ is injective.

\subsection{Restriction to Diagrams with no Trivalent Vertices}
\label{section_simple}
Let $SR\mathbb{D}_1^{<p}(\uparrow^n)$ (Simple Restricted Diagrams) be the subset of $R\mathbb{D}_1^{<p}(\uparrow^n)$ which contains only the diagrams with no trivalent vertices. Let $SR\mathcal{A}_1^{<p}(\uparrow^n)$ be the quotient of $S_\mathbb{F}(SR\mathbb{D}_1^{<p}(\uparrow^n))$ by STU-like and 4T relations. The 4T (four terms) relation is defined as follows:
$$\mathfig{3}{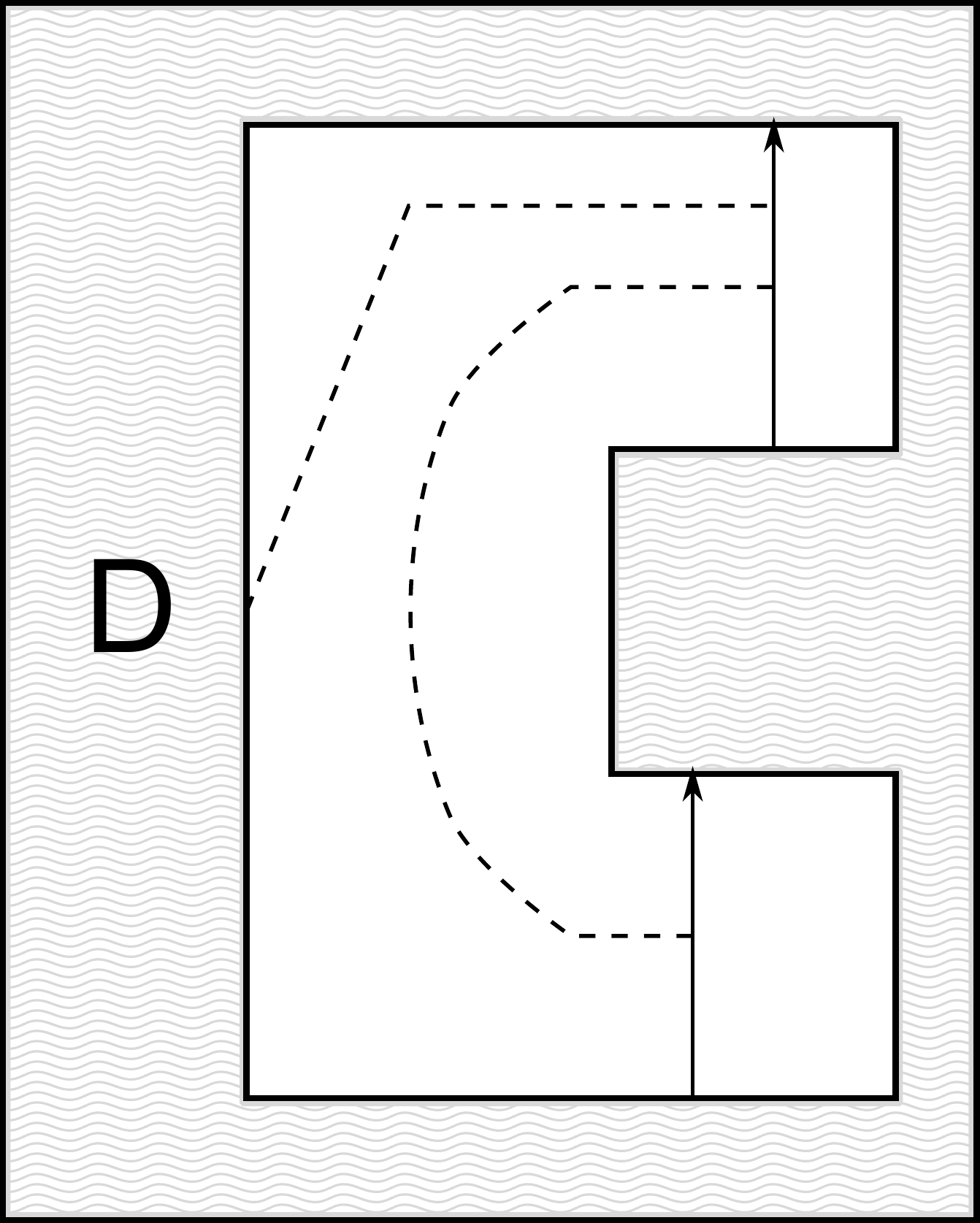}-\mathfig{3}{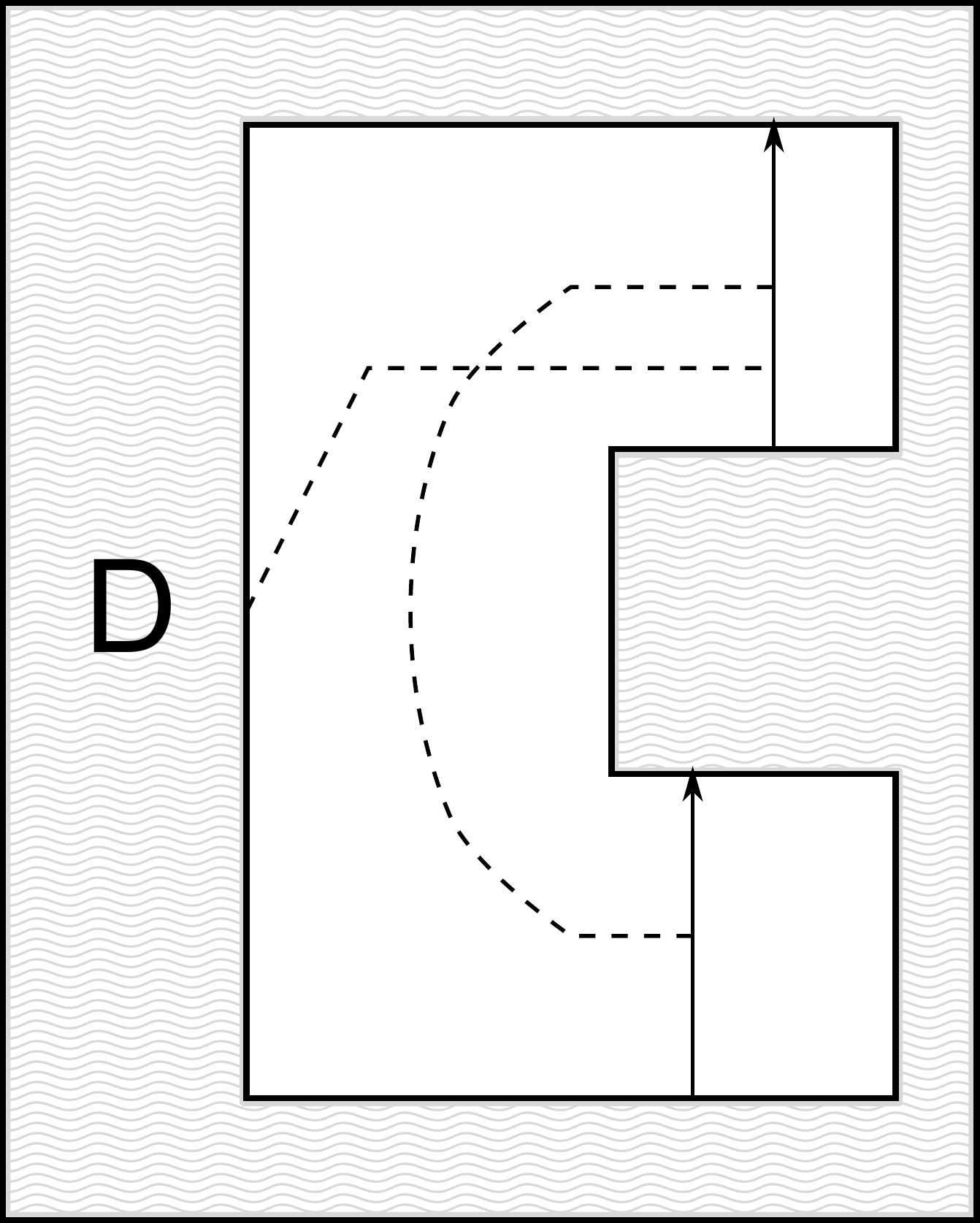}=\mathfig{3}{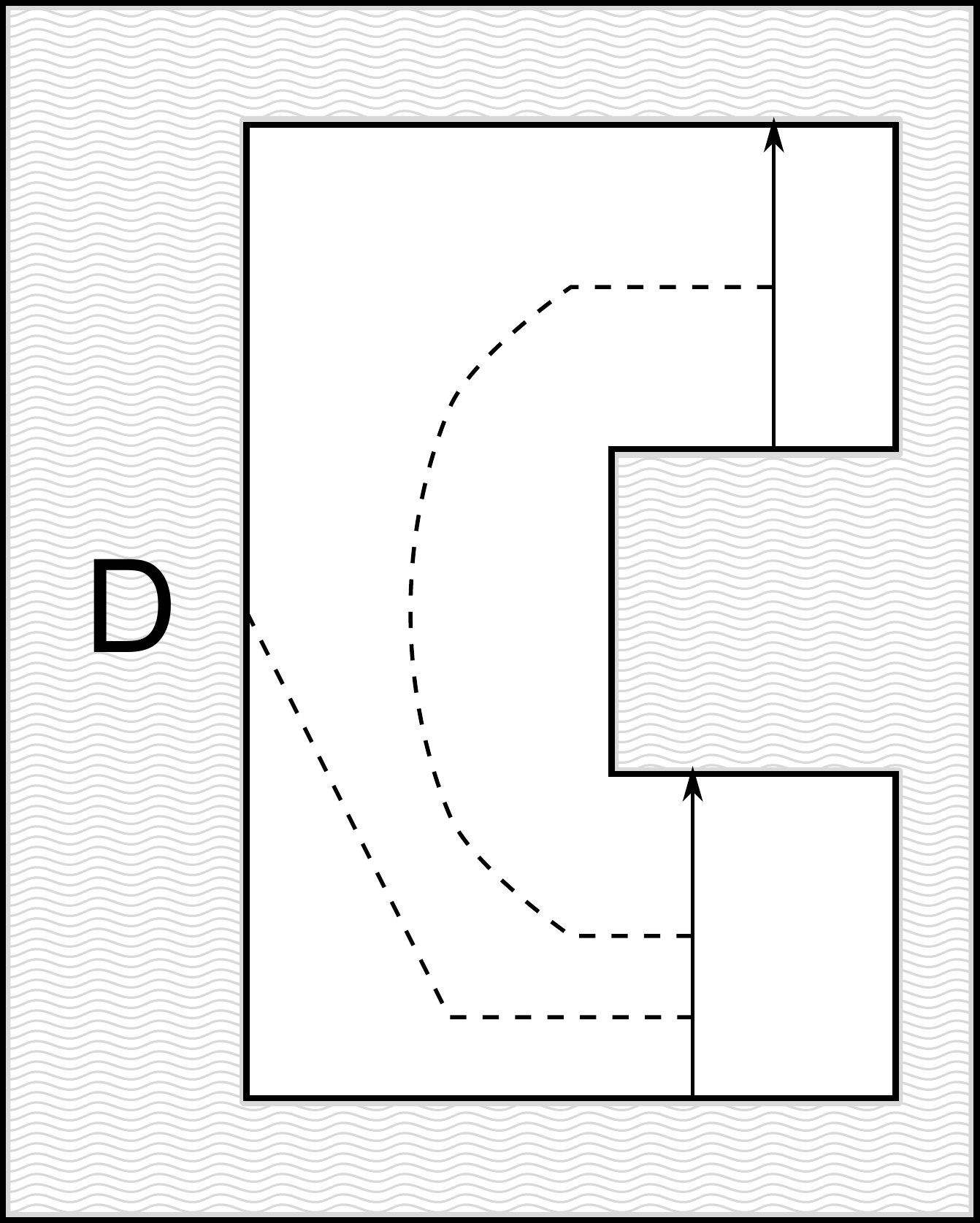}-\mathfig{3}{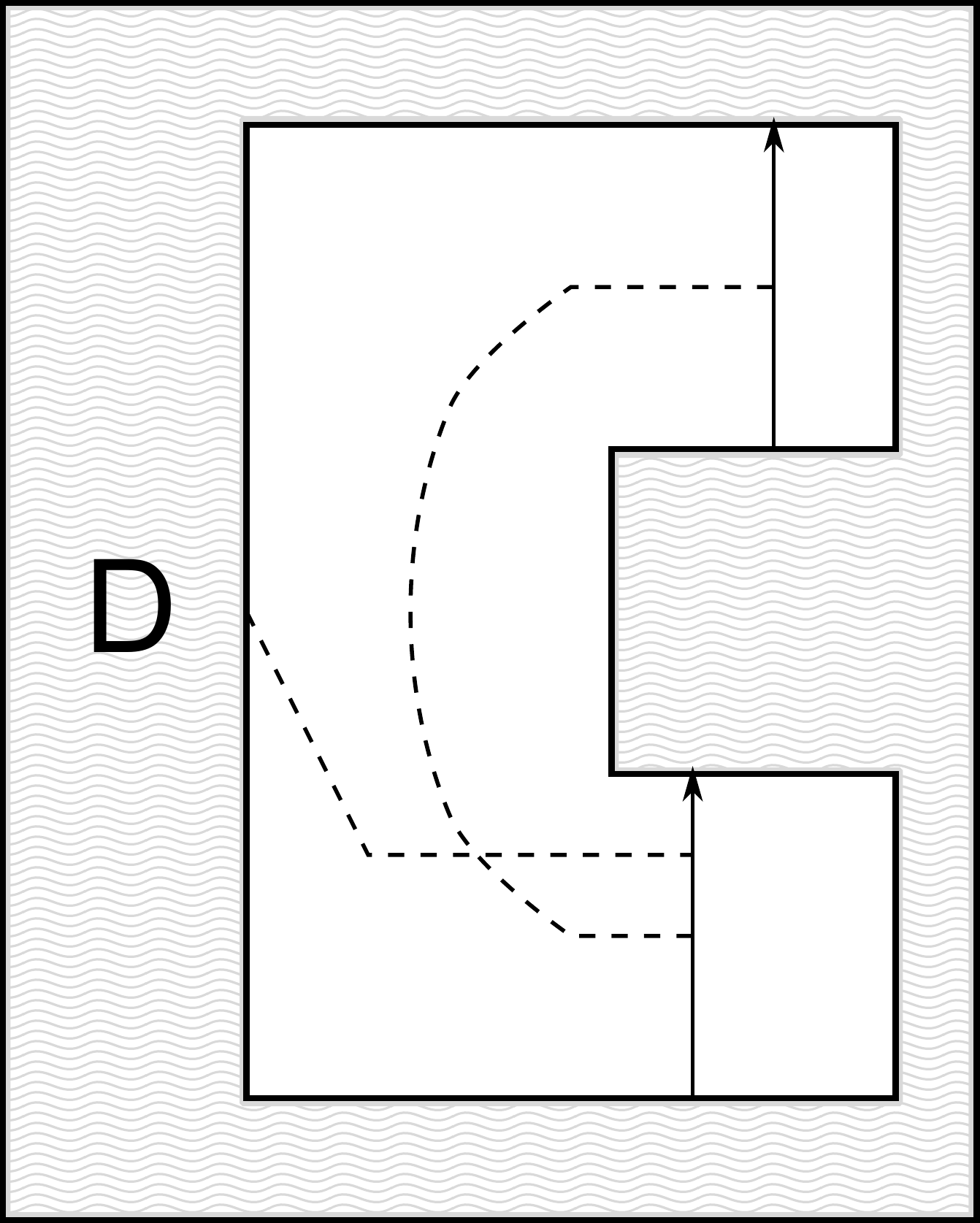}$$
Note that in $R\mathcal{A}_1^{<p}(\uparrow^n)$ the 4T relation is implied by STU.
\begin{proposition} The obvious map $s:S_\mathbb{F}(SR\mathbb{D}_1^{<p}(\uparrow^n))\rightarrow S_\mathbb{F}(R\mathbb{D}_1^{<p}(\uparrow^n)$ induces an isomorphism $s:SR\mathcal{A}_1^{<p}(\uparrow^n)\rightarrow R\mathcal{A}_1^{<p}(\uparrow^n)$.
\label{proposition_s}
\end{proposition}
This proposition is a slight generalization of Theorem 6 of \cite{BarNatan}.
\begin{proof}
For a diagram $D\in R\mathbb{D}_1^{<p}(\uparrow^n)$, let $n_{triv}(D)$ be the number of trivalent vertices in $D$. Let $(R\mathbb{D}_1^{<p}(\uparrow^n))^m$ be the subset of $R\mathbb{D}_1^{<p}(\uparrow^n)$ containing all diagrams $D$ with $n_{triv}(D)\le m$. We get a filtration of $S_\mathbb{F}((R\mathbb{D}_1^{<p}(\uparrow^n))$:
$$S_\mathbb{F}((R\mathbb{D}_1^{<p}(\uparrow^n))^0)\subseteq S_\mathbb{F}((R\mathbb{D}_1^{<p}(\uparrow^n))^1)\subseteq S_\mathbb{F}((R\mathbb{D}_1^{<p}(\uparrow^n))^2)\subset\cdots $$
Let $S^m$ be the quotient of $S_\mathbb{F}((R\mathbb{D}_1^{<p}(\uparrow^n))^m)$ by STU, STU-like and 4T relations. We get a sequence $S^0\stackrel{s^0}{\rightarrow} S^1\stackrel{s^1}{\rightarrow} S^2\rightarrow \cdots$. $S^0$ is $SR\mathcal{A}_1^{<p}(\uparrow^n)$, and the direct limit of the sequence is $R\mathcal{A}_1^{<p}(\uparrow^n)$. As in the previous proofs, we need to find an inverse to $s^m$.

Let $\gamma^m:S_\mathbb{F}((R\mathbb{D}_1^{<p}(\uparrow^n))^m)\rightarrow S_\mathbb{F}((R\mathbb{D}_1^{<p}(\uparrow^n))^{m-1})$ be defined as follows: Given a diagram $D\in (SR\mathbb{D}_1^{<p}(\uparrow^n))^m$, let $i(D)$ be the left-most strand with a vertex whose component has trivalent vertices. Let $v(D)$ be the highest such vertex on $i(D)$, and let $w(D)$ be the trivalent vertex sharing an edge with $v(D)$. If $n_{triv}(D)<m$, define $\gamma^m(D)=D$. Otherwise, define $\gamma^m(D)$ by:

$$\mathfig{3}{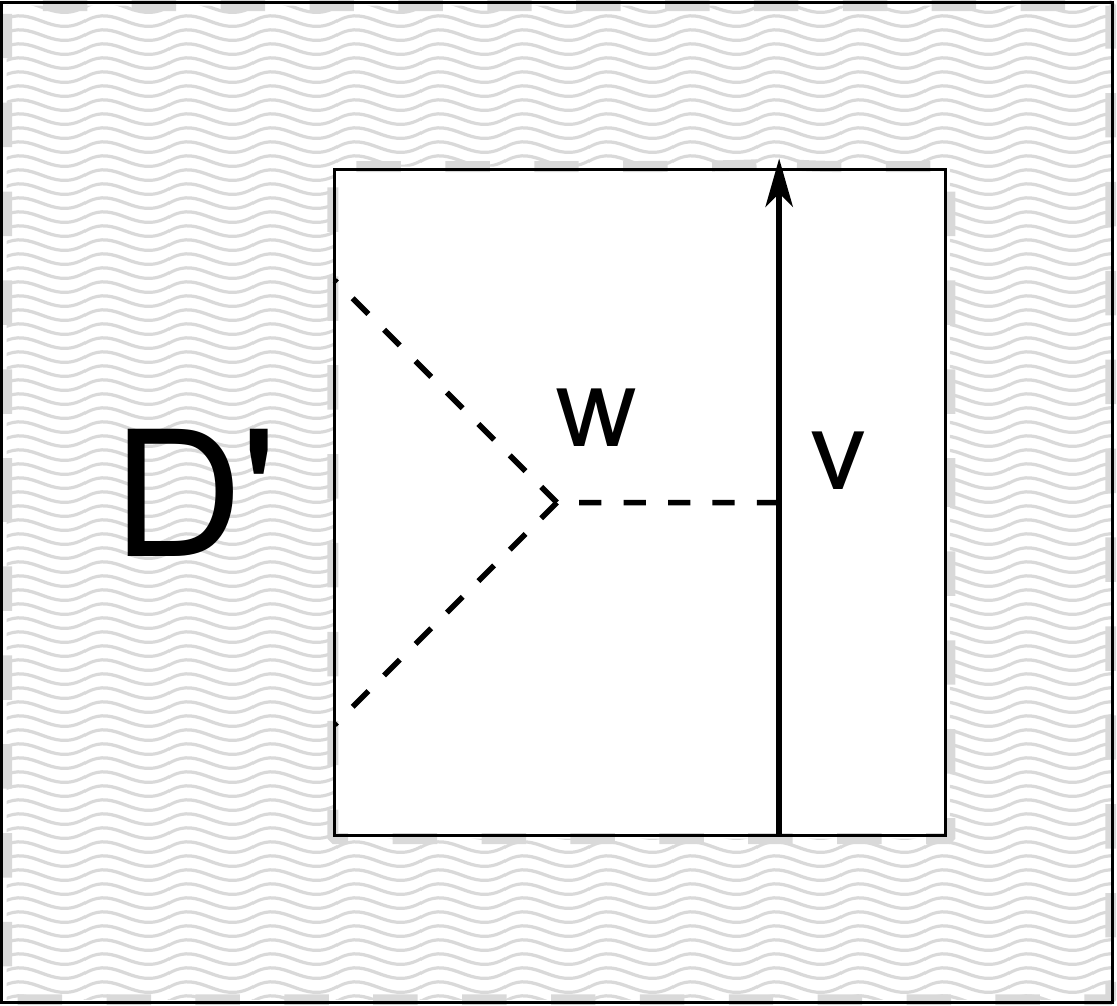}\stackrel{\gamma^m}{\longmapsto}\mathfig{3}{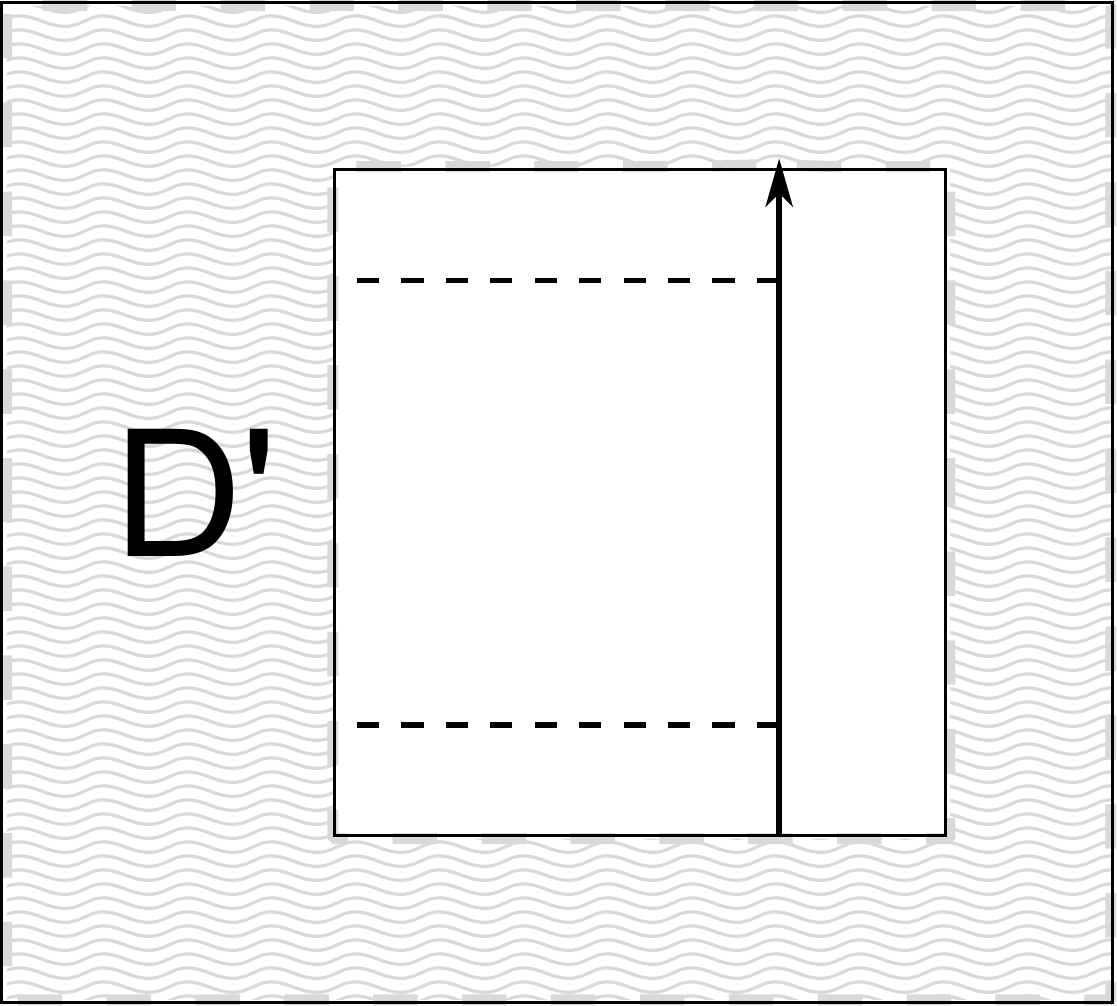}-\mathfig{3}{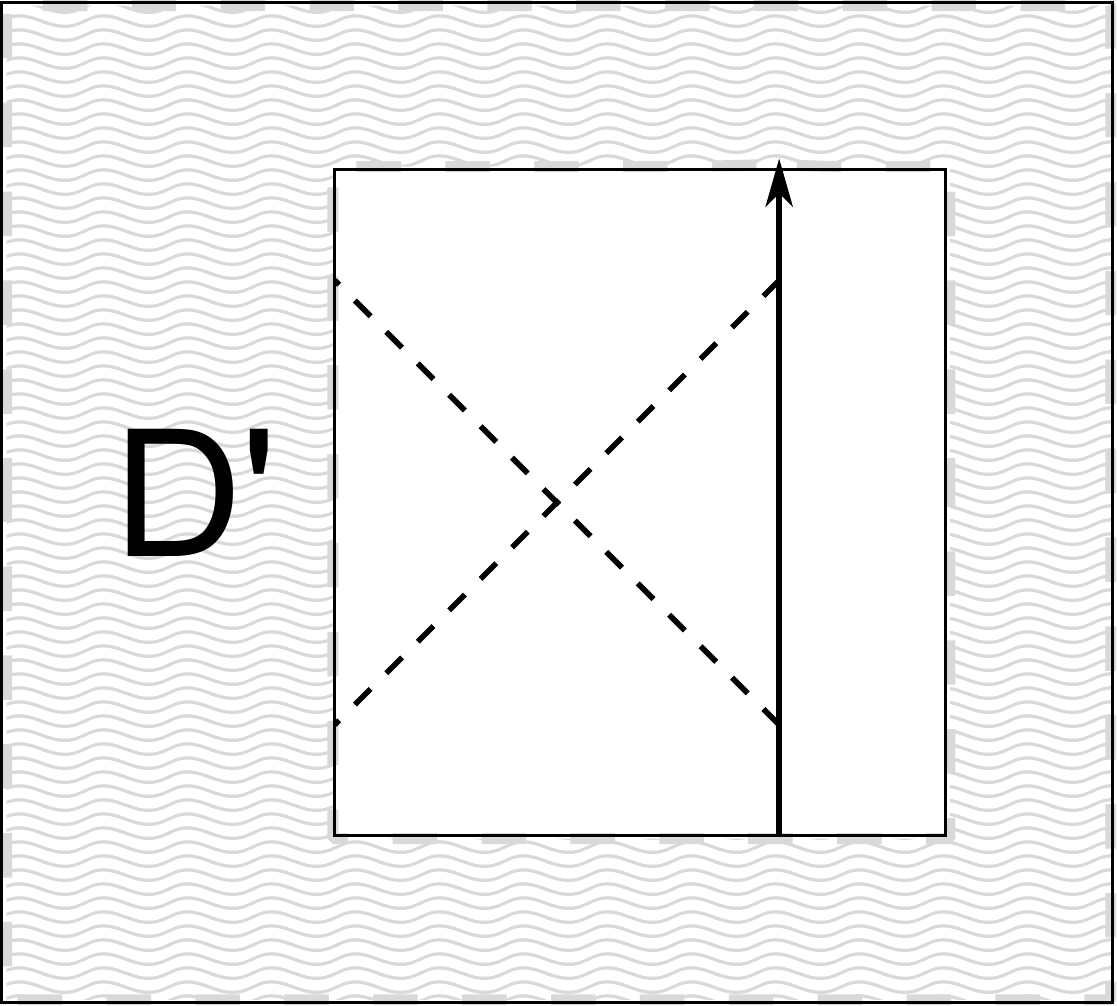}$$

We claim that $\gamma^m$ induces a map $\gamma^m:S^m\rightarrow S^{m-1}$. Indeed, if $u$ is an STU-like relation, then $\gamma^m(u)$ is a sum of STU-like relations. 

Let $u=u_1+u_2+u_3$ be an STU relation, and suppose $n_{triv}(u_1)=m$. If the STU relation $u$ does not involve the vertices $v(u_1)$ and $w(u_1)$, then $\gamma^m(u)$ is a sum of STU relations. If the STU relation involves both $v(u_1)$ and $w(u_1)$, then by definition $\gamma^m(u)=0$. And if the STU relation involves $w(u_1)$ but not $v(u_1)$, then $\gamma^m(u)$ is a 4T relation.

Now let $u=u_1+u_2+u_3+u_4$ be a 4T relation. If $u$ does not involve, in any of its summands, the vertex $v(u_i)$, then $\gamma^m(u)$ is a sum of 4T relations. If the 4T relation $u$ involves, in any of its summands, the vertex $v(u_i)$, then $\gamma^m(u)$ is either equal to a sum of 4T relations, or equivalent to it via STU (depending on whether $v(u_i)$ is involved in the 4T relation in all the $u_i$'s or only in $2$ of them). This follows from the fact that the following sum:
$$\mathfig{4}{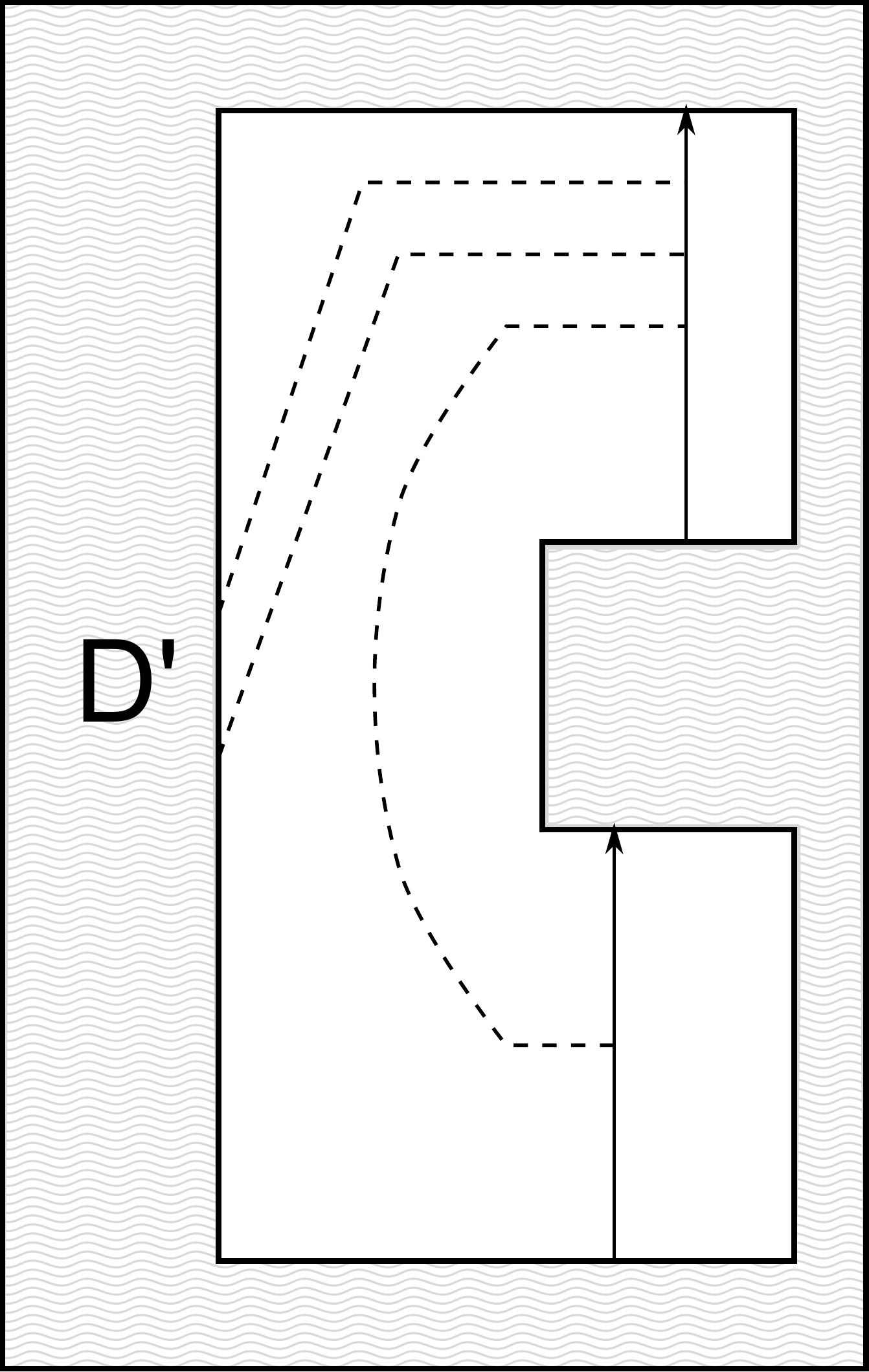}-\mathfig{4}{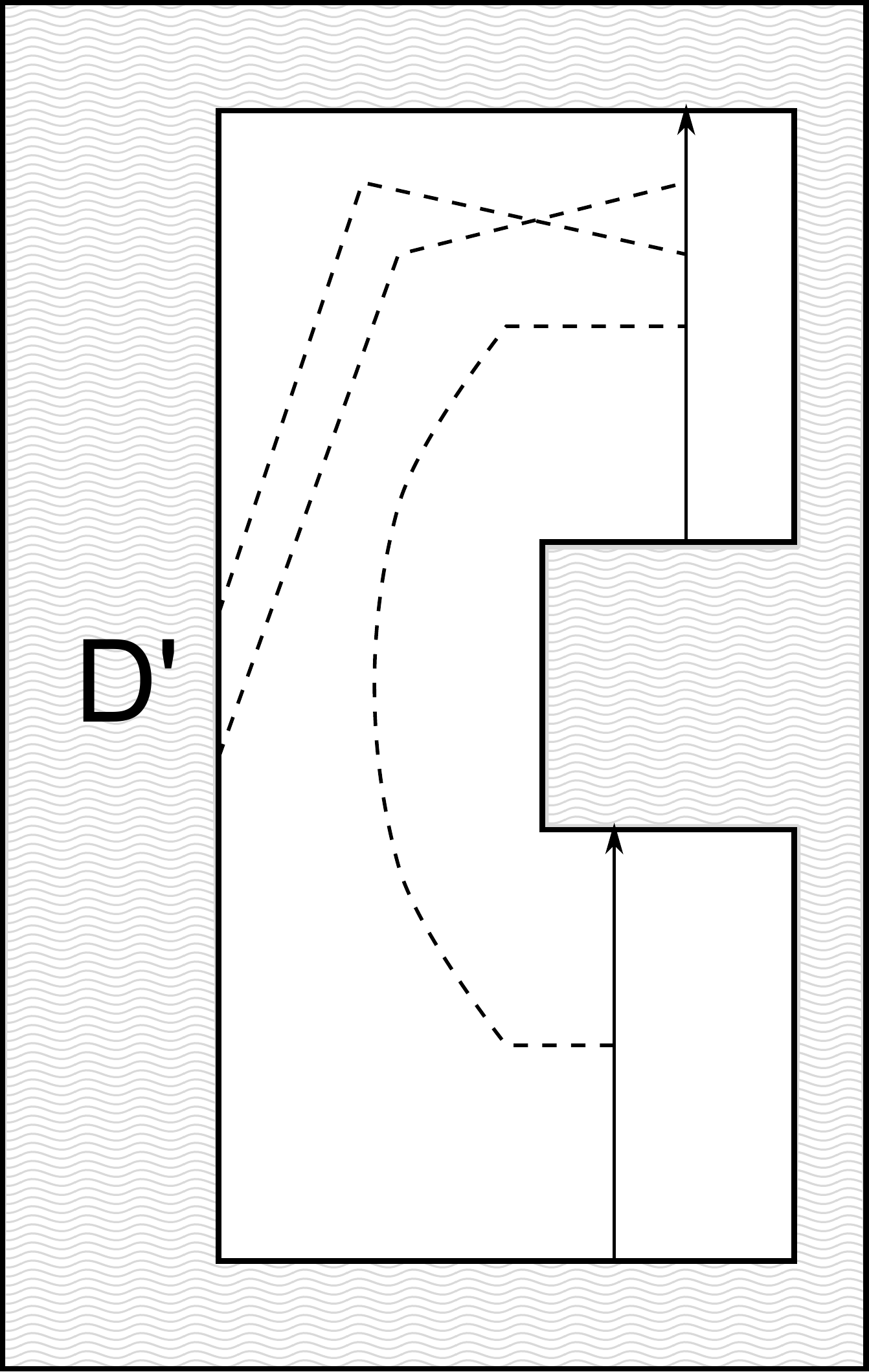}-\mathfig{4}{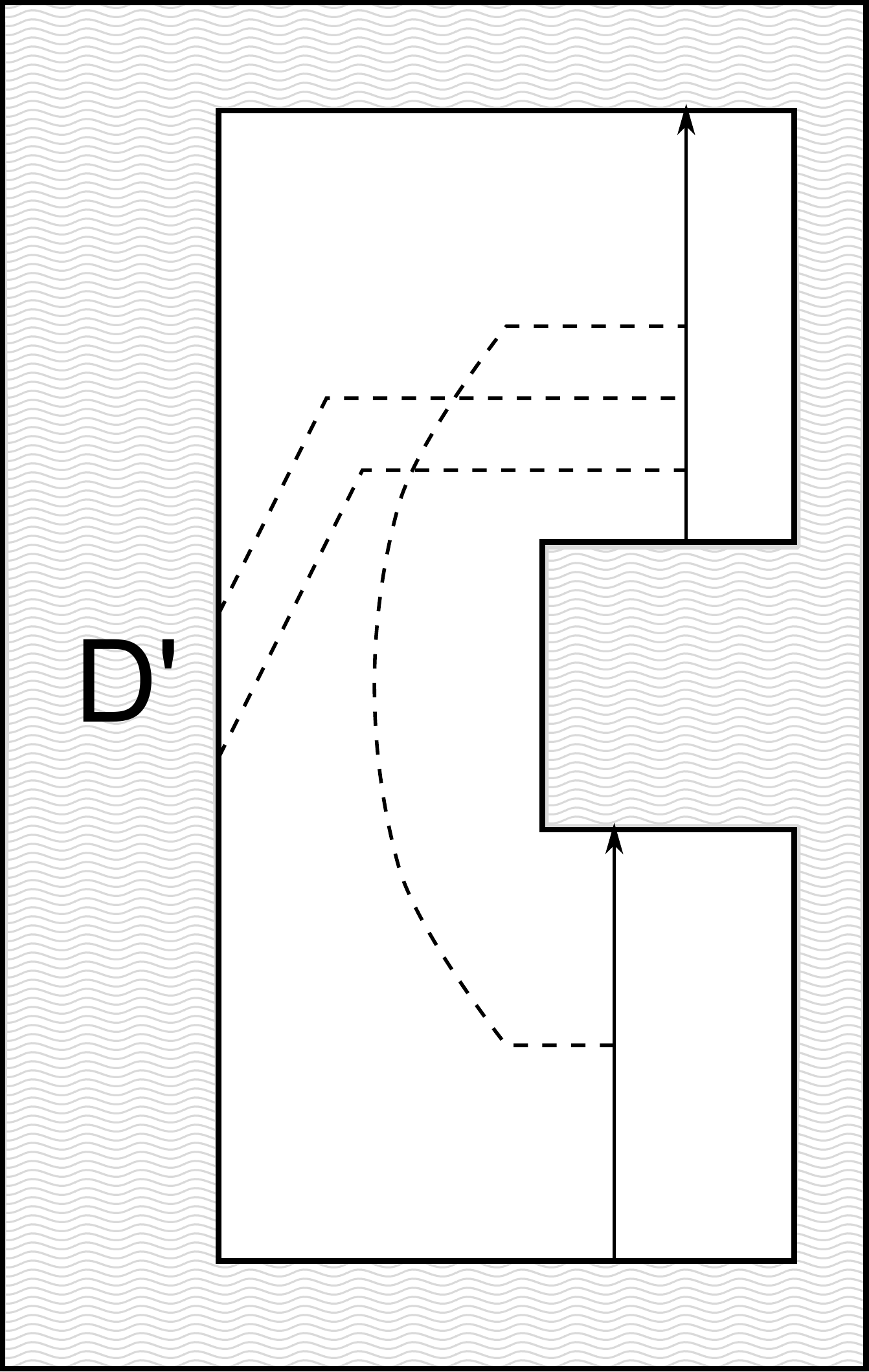}+\mathfig{4}{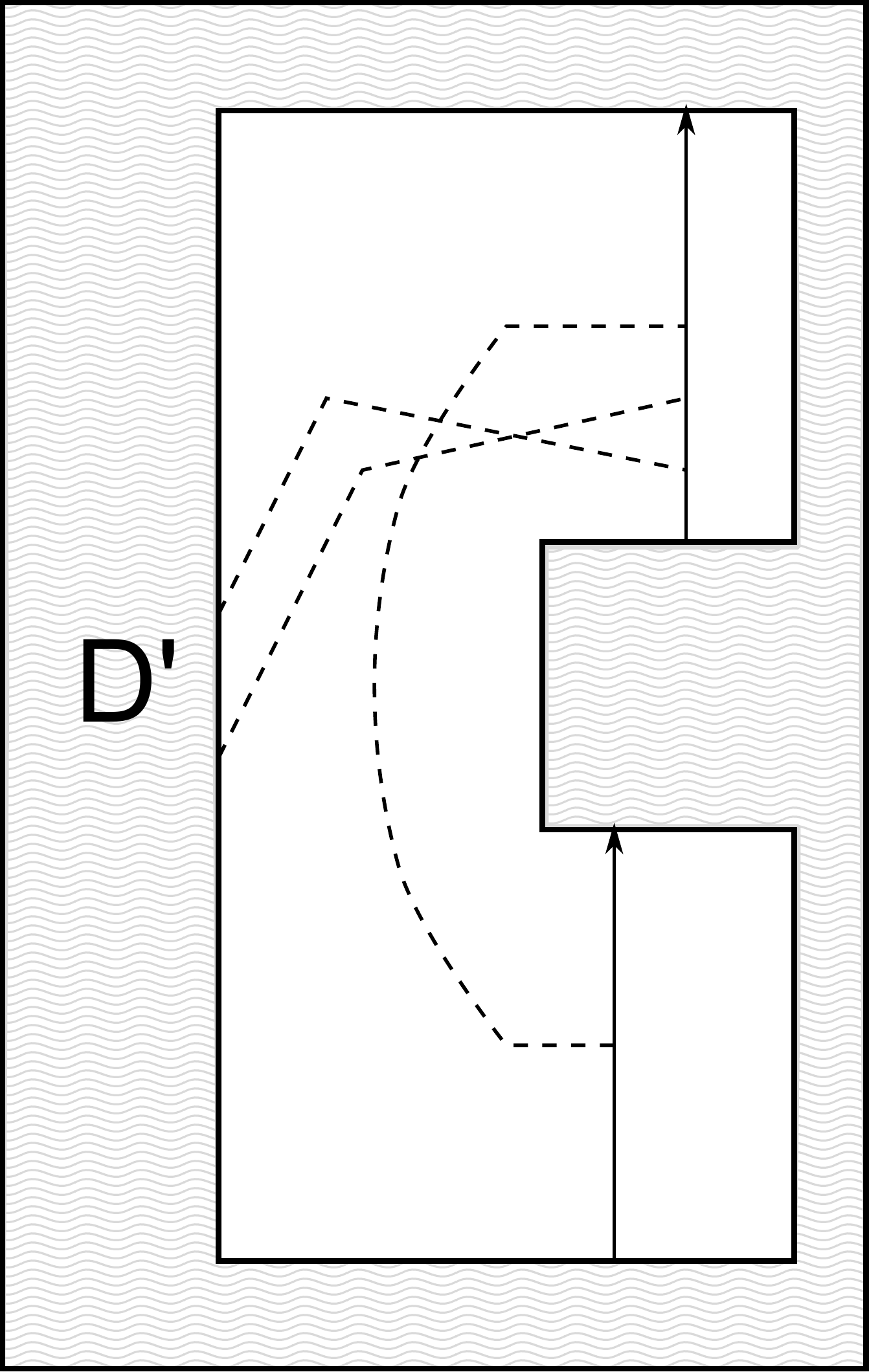}+$$
$$+\mathfig{4}{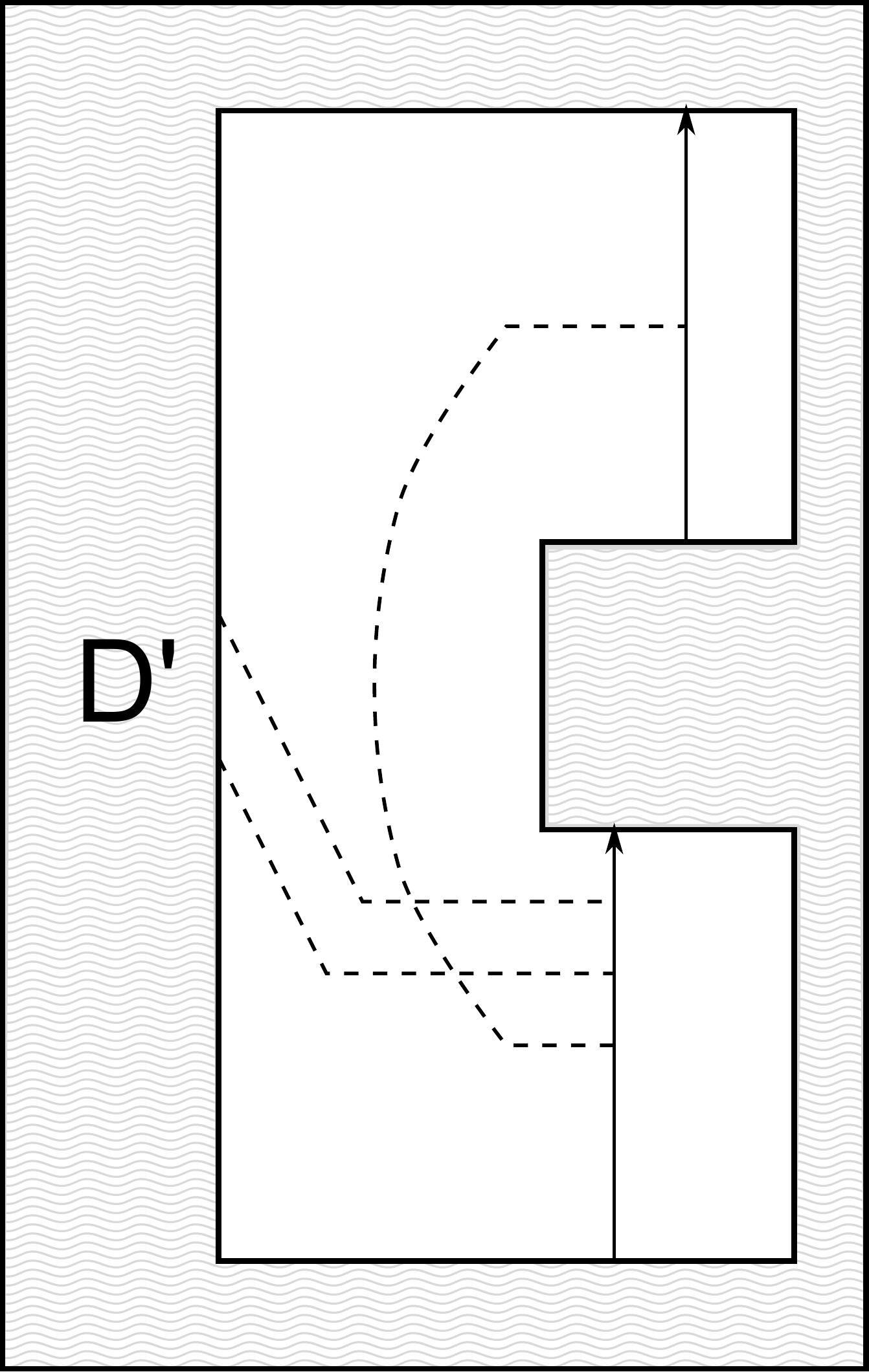}-\mathfig{4}{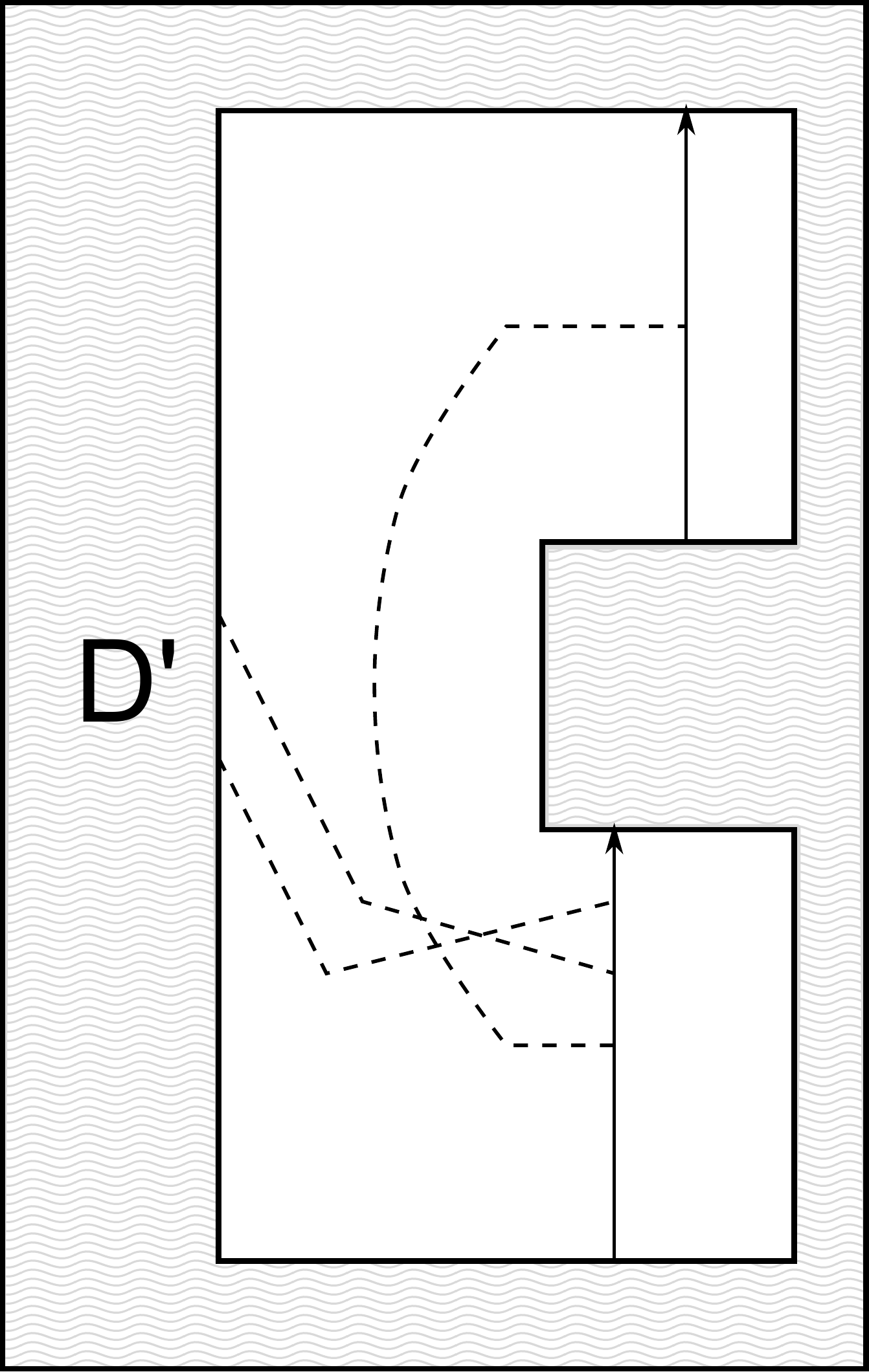}-\mathfig{4}{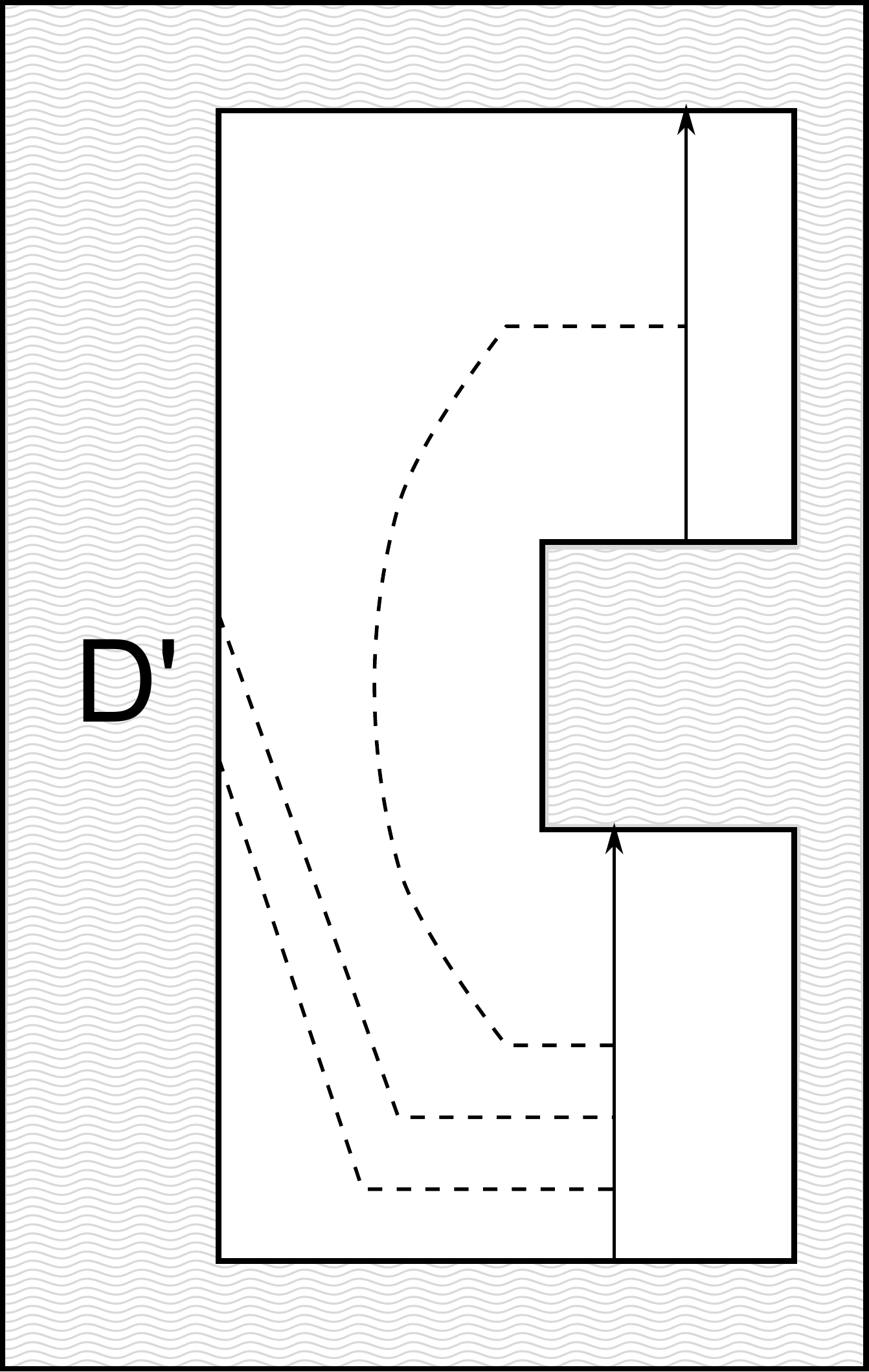}+\mathfig{4}{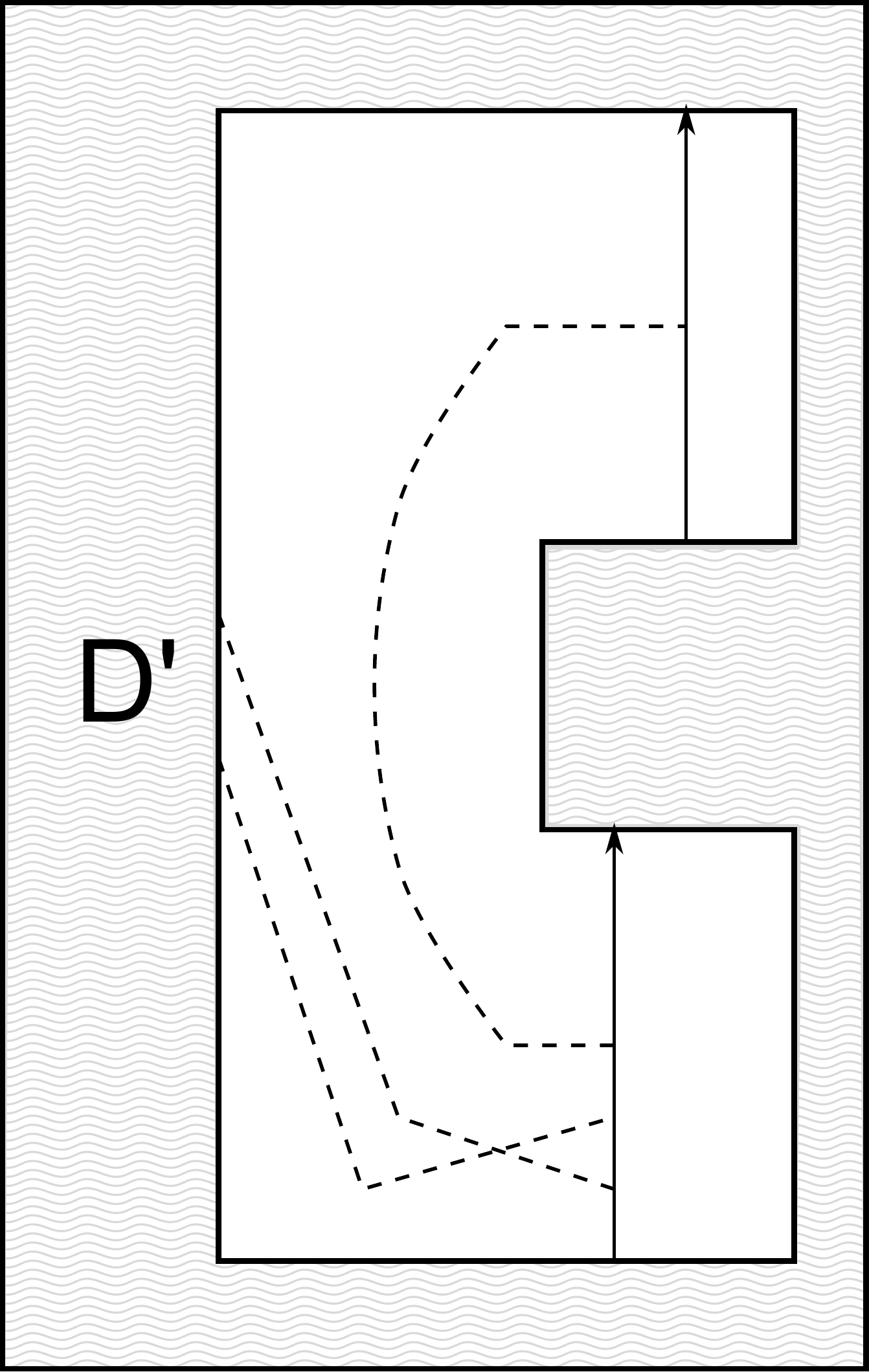}$$
can be written as a sum of $8$ 4T relations, with $24$ of the $32$ summands cancelling in pairs.

It is easy to verify that $\gamma^m$ is the inverse of $s^{m-1}$, which completes the proof.

\end{proof}

The map $u_r:RU\hat{\textbf{t}}_{1,n}\rightarrow R\mathcal{A}_1^{<p}(\uparrow^n)$ induces a map $u_s:RU\hat{\textbf{t}}_{1,n}\rightarrow SR\mathcal{A}_1^{<p}(\uparrow^n)$ by composing with the isomorphism $\gamma:R\mathcal{A}_1^{<p}(\uparrow^n)\rightarrow SR\mathcal{A}_1^{<p}(\uparrow^n)$. Thus, to prove the injectivity of $u_r$ it is enough to show that $u_s$ is injective.

\subsection{Restriction to Ordered Diagrams}
\label{section_ordered}

From now on we restrict our attention to $n=2$ and $n=3$. We wish to show that $u_s:RU\hat{\textbf{t}}_{1,n}\rightarrow SR\mathcal{A}_1^{<p}(\uparrow^n)$ is injective. In fact we will prove that $u_s$ is an isomorphism onto a quotient of $SR\mathcal{A}_1^{<p}(\uparrow^n)$. But first we need a definition.

\begin{definition}
A diagram in $SR\mathbb{D}_1^{<p}(\uparrow^n)$ may have $2$ kinds of edges: edges with a labeled vertex, which we call \textbf{labeled edges}, and edges with both vertices on the pattern, which we call \textbf{chords}.

Let $H_n\subset SR\mathbb{D}_1^{<p}(\uparrow^n)$ be the subset of all diagrams with a chord which has $2$ vertices on the same strand. ($H$ here stands for ``homotopy" - see \cite{BarNatan2} for an explanation of this notation).
\end{definition}

We will denote by the same notation $H_n$ also the image of $H_n$ in $R\mathcal{A}_1^{<p}(\uparrow^n)$ via $s$, and the image in $\mathcal{A}_1^{<p}(\uparrow^n)$ via $r\circ s$. The projections will be denoted by $\pi_s:SR\mathcal{A}_1^{<p}(\uparrow^n)\rightarrow SR\mathcal{A}_1^{<p}(\uparrow^n)/H_n$, $\pi_r:R\mathcal{A}_1^{<p}(\uparrow^n)\rightarrow R\mathcal{A}_1^{<p}(\uparrow^n)/H_n$ and $\pi:\mathcal{A}_1^{<p}(\uparrow^n)\rightarrow \mathcal{A}_1^{<p}(\uparrow^n)/H_n$. We also denote $\tilde{u}_s:=\pi_s\circ u_s$, $\tilde{u}_r:=\pi_r\circ u_r$ and $\tilde u:= \pi \circ u\vert_{RU  \hat {\textbf{t}_{1,n}}}$. Most of those spaces and maps can be seen in the diagram in section \ref{restrictions_diagram}.

\begin{remark}
If a diagram $a\in\mathcal{A}_1^{<p}(\uparrow^n)$ has the property described in remark \ref{remark_restricted_diagram}, and in addition it has a component with more than one vertex on the same strand, or a component with a loop, then $a$ is in the image of $H_n$ via the isomorphism $r\circ s$ (a diagram with a loop is equivalent to a sum of diagrams in $H_n$ after applying the STU relation to all the trivalent vertices).
\end{remark}
\begin{theorem}
\label{theorem_u_n}
The map $\tilde{u}_s:RU\hat{\textbf{t}}_{1,n}\rightarrow SR\mathcal{A}_1^{<p}(\uparrow^n)/H_n$ is an isomorphism for $n=2,3$.
\end{theorem}

This theorem implies that $\tilde{u}_r$ is also an isomorphism, and $\tilde{u}$ is injective.

In order to prove this theorem we will need to restrict $SR\mathbb{D}_1^{<p}(\uparrow^n)$ further. Let $OSR\mathbb{D}_1^{<p}(\uparrow^n)$ (ordered simple restricted diagrams) be the subset of $SR\mathbb{D}_1^{<p}(\uparrow^n)$ containing all diagrams $D$ which are ordered, in the following sense: If $2$ labeled edges have vertices on the same strand, then the order of the labels corresponds to the order of the vertices along the strand. An example for a diagram in $OSR\mathbb{D}_1^{<p}(\uparrow^3)$ is given in figure \ref{ordered_example}.

\begin{figure}[ht]

\centering
\includegraphics[height=4 cm]{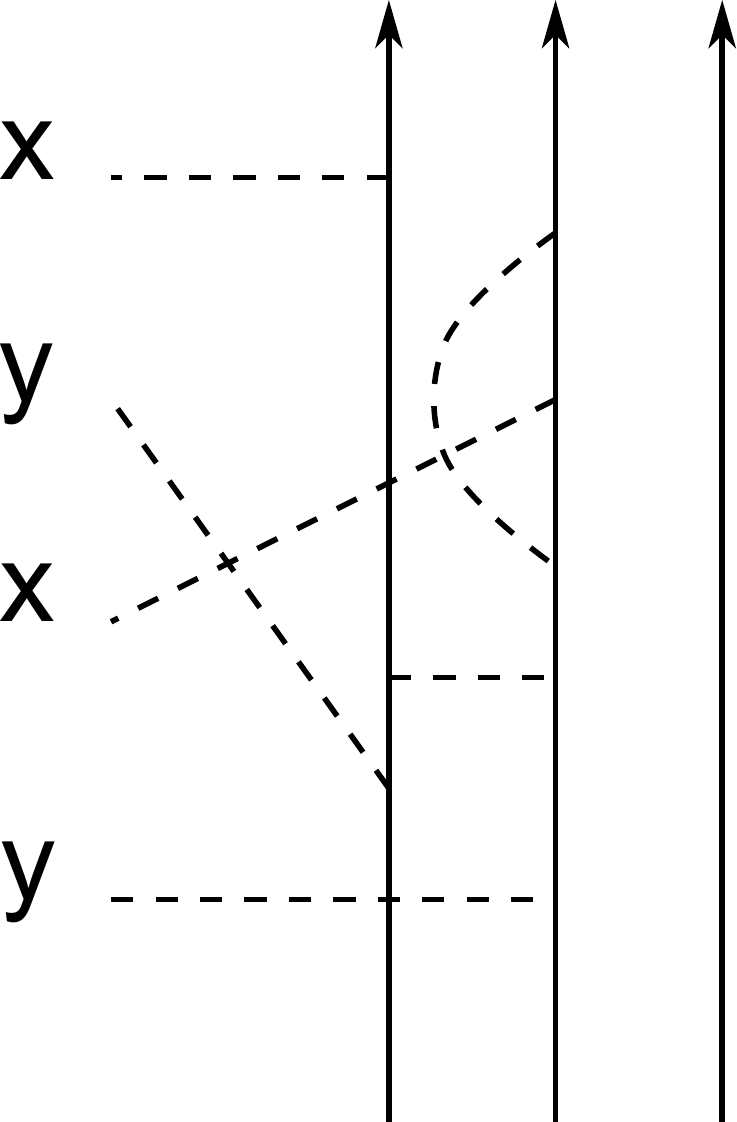}

\caption{An example for a diagram in $OSR\mathbb{D}_1^{<p}(\uparrow^3)$}
\label{ordered_example}
\end{figure}

We define an O4T (ordered 4 terms) relation to be either a 4T relation or the following relation:

$$\mathfig{4.5}{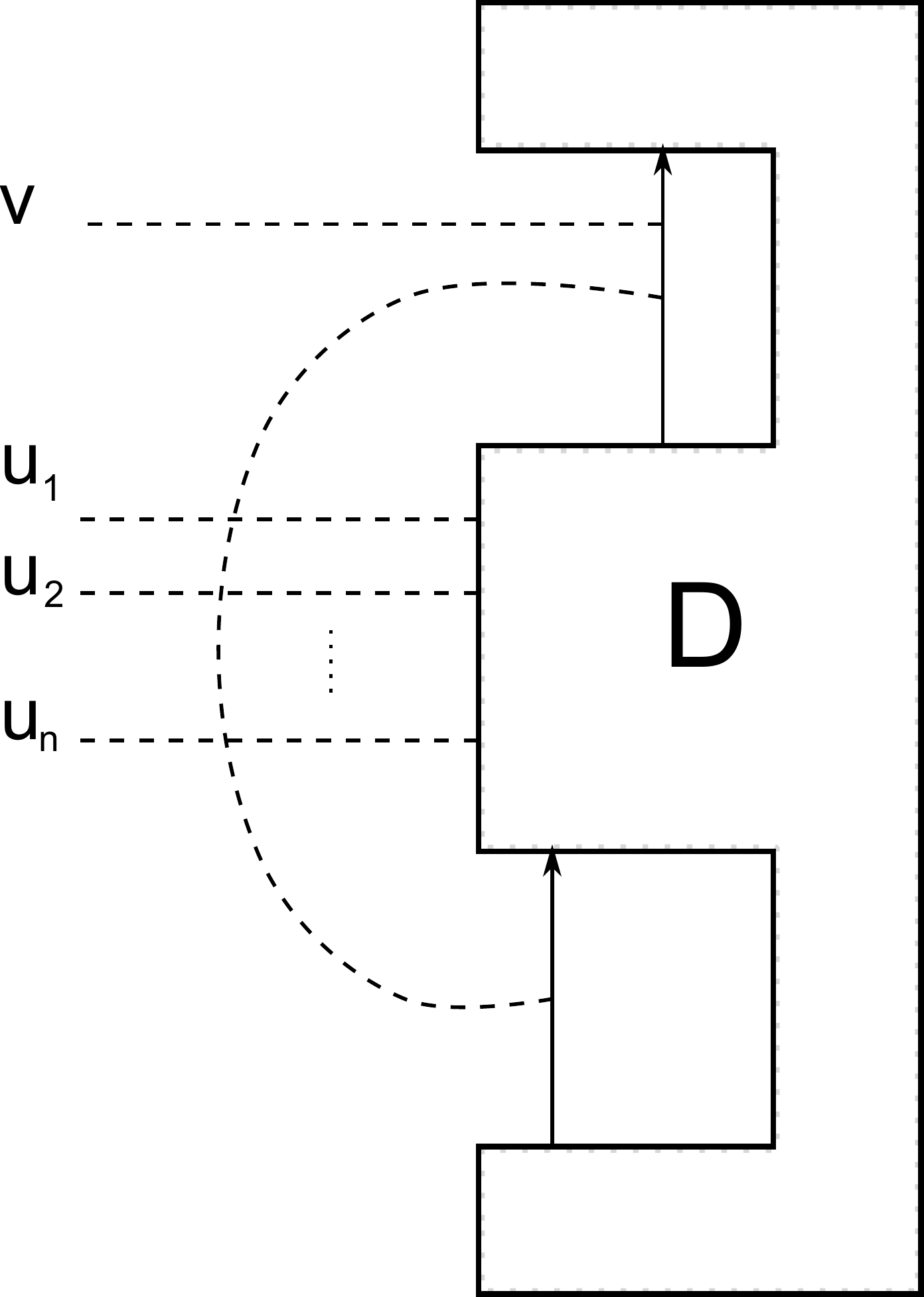}-\mathfig{4.5}{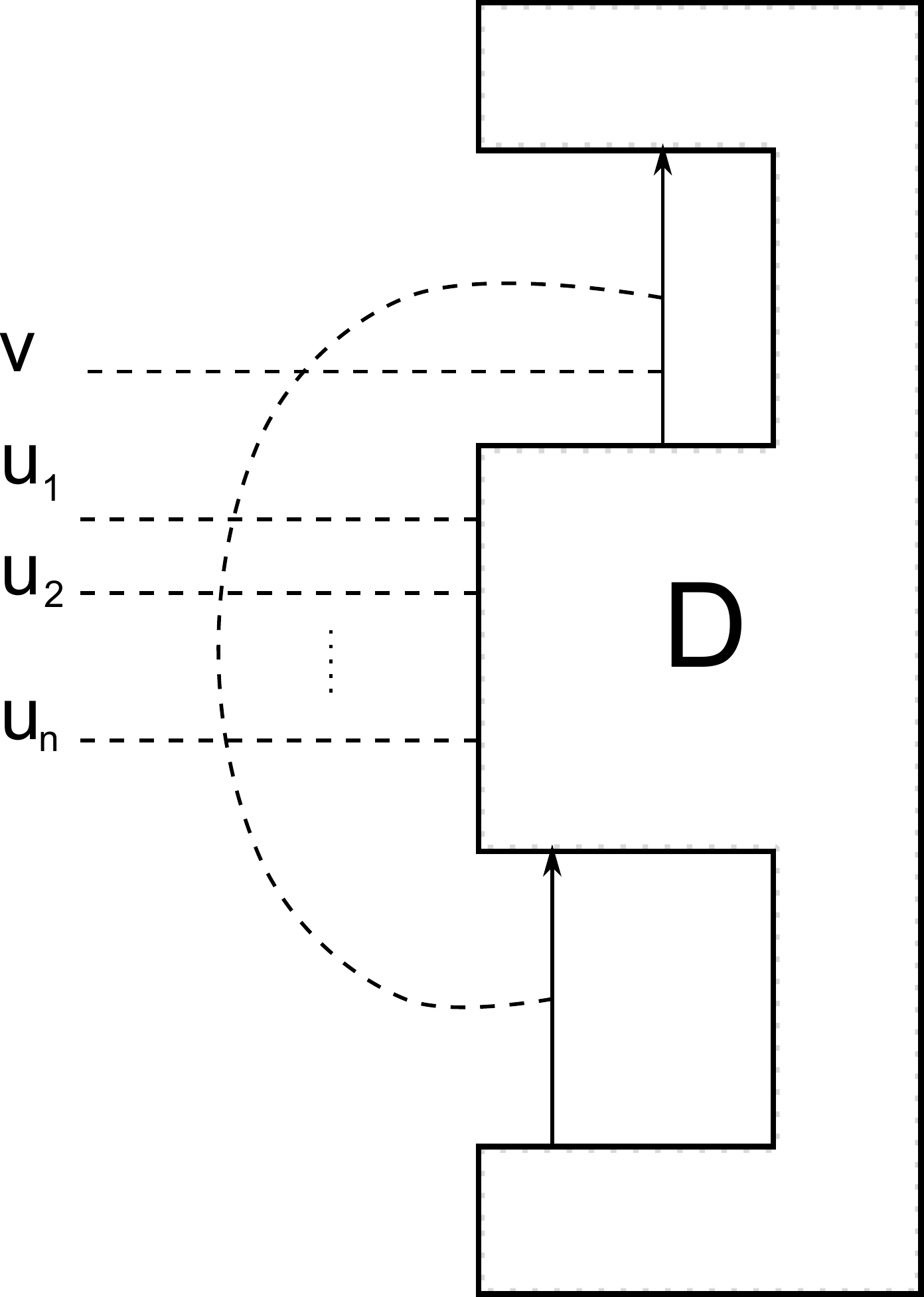}=$$
$$=\mathfig{4.5}{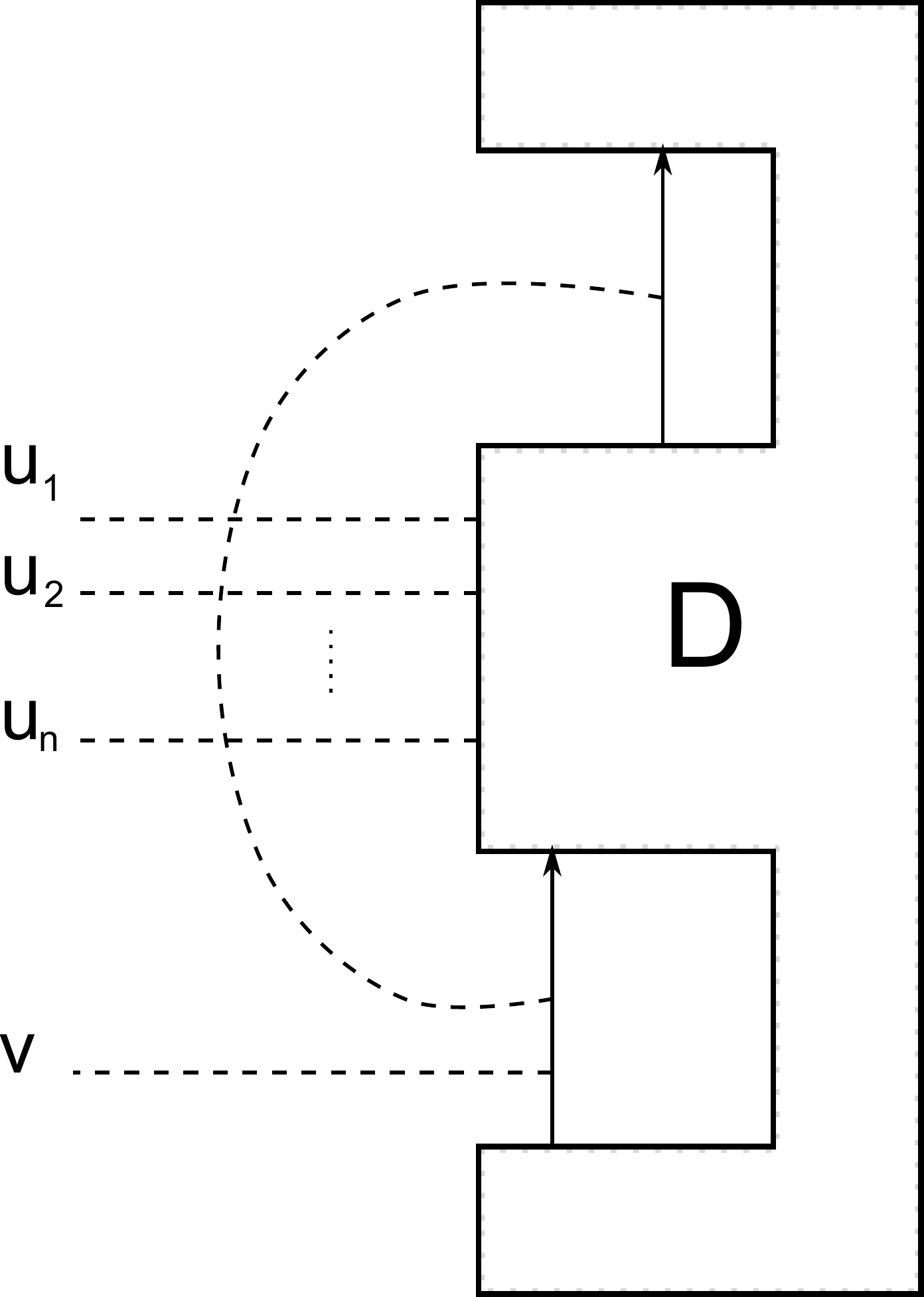}-\mathfig{4.5}{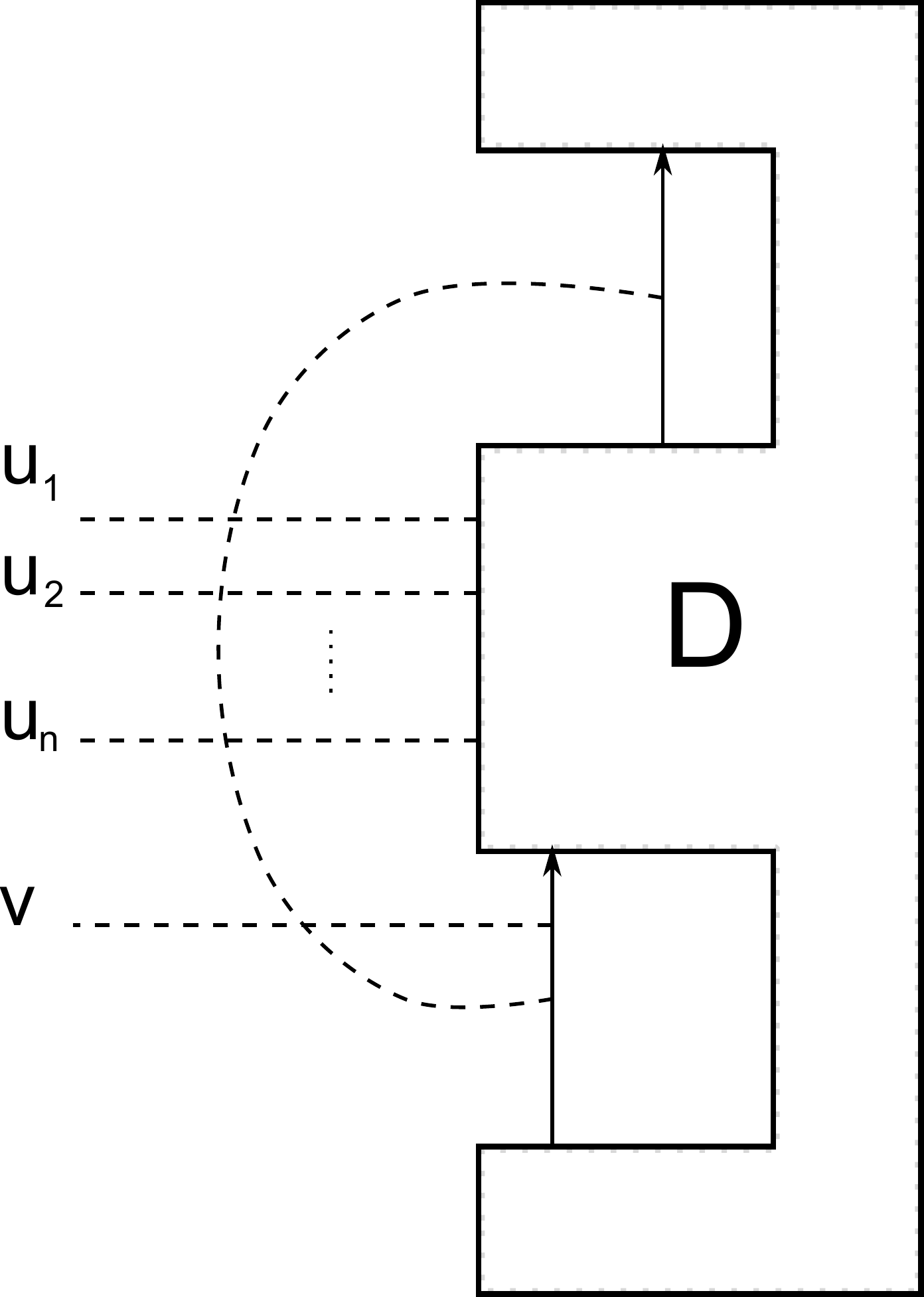}$$

\begin{remark} In $SR\mathcal{A}_1^{<p}(\uparrow^n)/H_n$ ($n=2,3$), we have:
$$u\quad:=\mathfig{4}{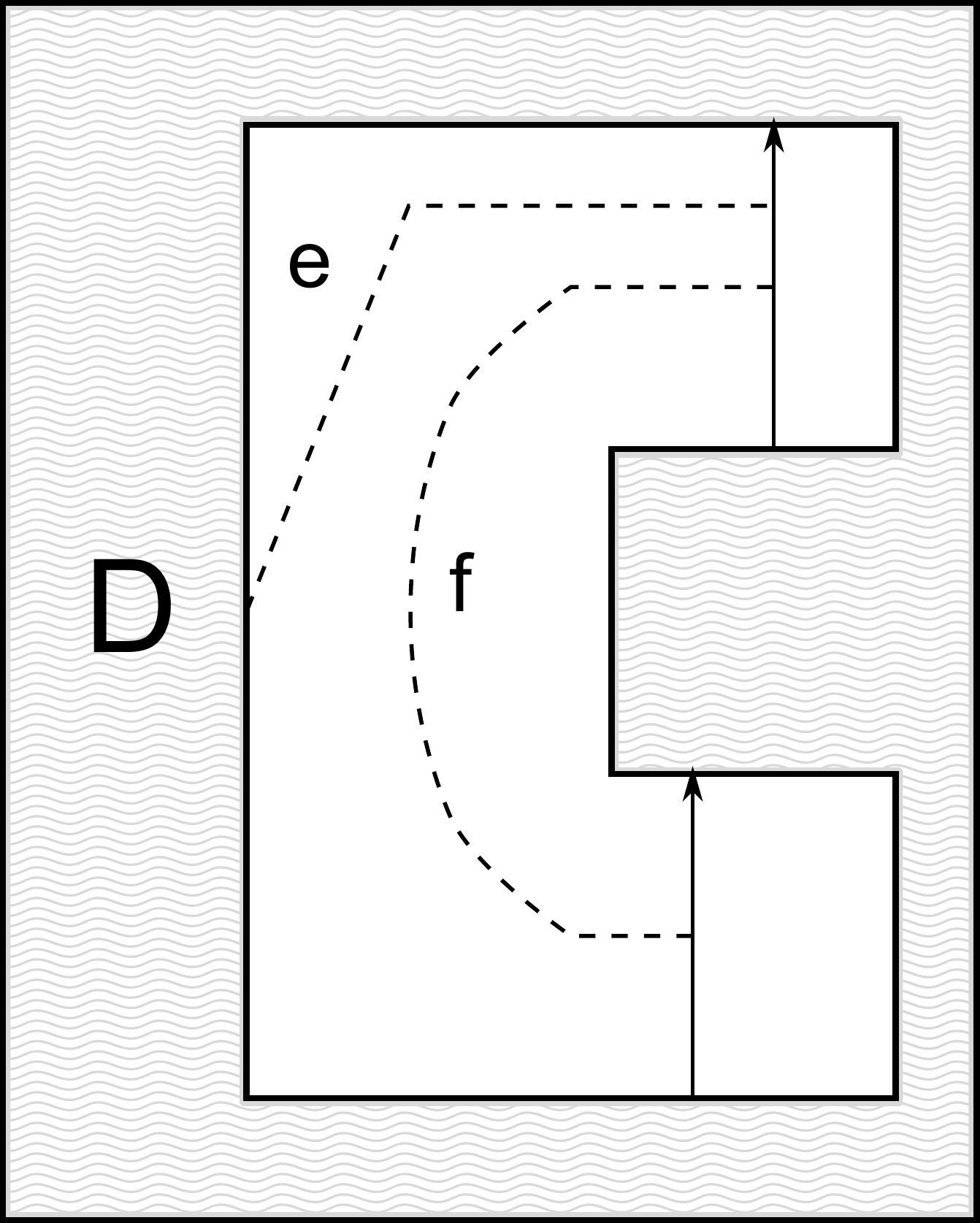} - \mathfig{4}{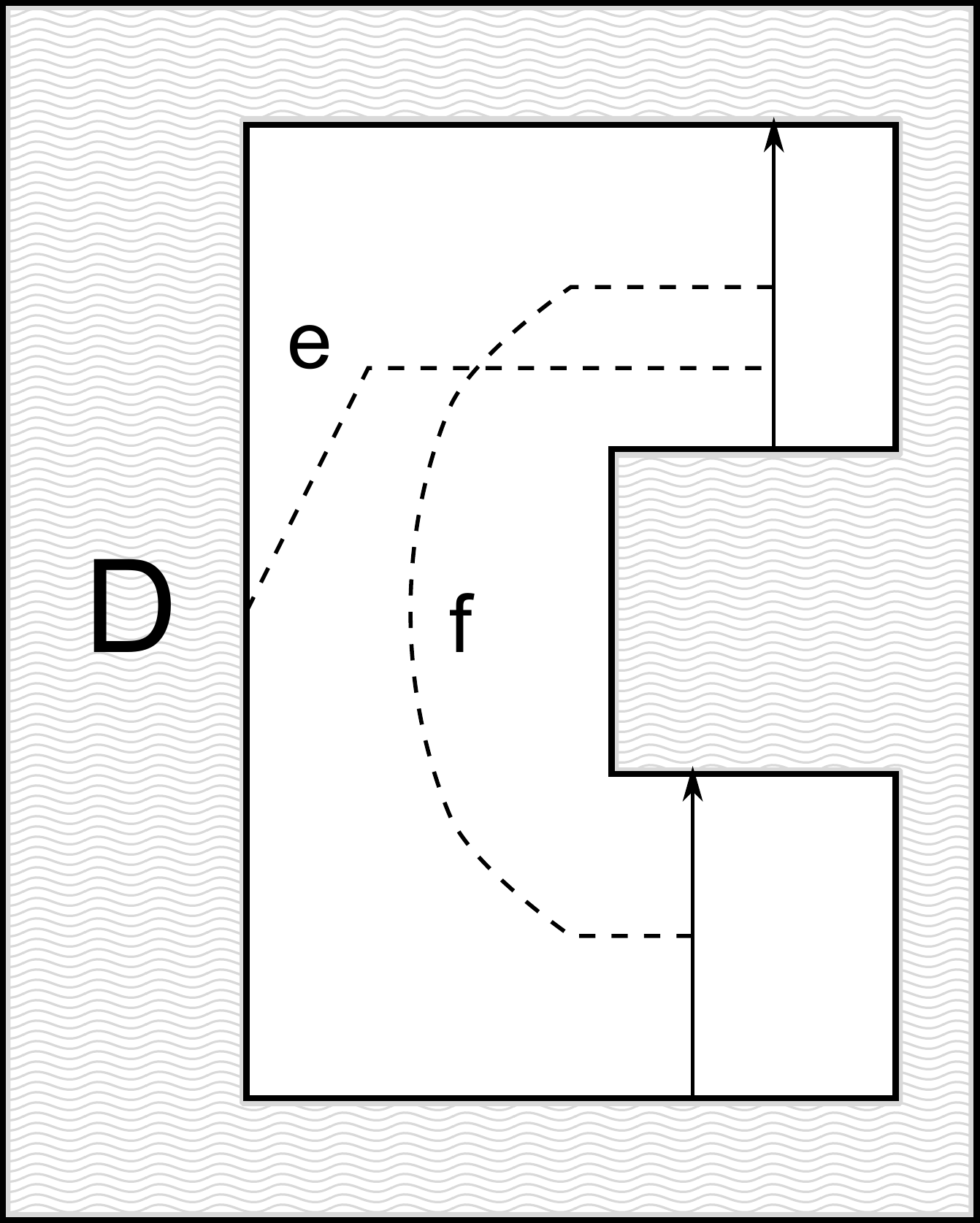}=\quad 0$$
if $\textbf{e}$ is a chord. Indeed, if $\textbf{f}$ has both vertices on the same strand, then by definition $u\in H_n$. Otherwise, suppose without loss of generality that the upper vertex of $\textbf{f}$ in the figures is on strand $1$, and the lower vertex is on strand $2$. If the second vertex of $\textbf{e}$ is on strand $1$, then again $u\in H_n$. Otherwise the other vertex of $\textbf{e}$ is on strand $2$, and $u$ is equivalent via 4T to an element in $H_n$.
\label{remark_H}
\end{remark}

Using this remark we see that in $SR\mathcal{A}_1^{<p}(\uparrow^n)/H_n$, O4T is implied by 4T, STU-like and $H_n$. Let $OSR\mathcal{A}_1^{<p}(\uparrow^n)$ be the quotient of $S_\mathbb{F}(OSR\mathbb{D}_1^{<p}(\uparrow^n))$ by O4T, STU-like and $H_n$. There is an obvious map $o:OSR\mathcal{A}_1^{<p}(\uparrow^n)\rightarrow SR\mathcal{A}_1^{<p}(\uparrow^n)/H_n$.

\begin{proposition}
$o$ is an isomorphism.
\end{proposition}

\begin{proof}
Let $D\in SR\mathbb{D}_1^{<p}(\uparrow^n)$, and let $e_1$ and $e_2$ be a pair of unordered labeled edges in $D$. Define $n_o(e_1,e_2):=\#\{\text{labeled vertices between $e_1$ and $e_2$}\}+1$, and $n_o(D):=\sum\limits_{\text{$e_1$,$e_2$ unordered in $D$}}n_o(e_1,e_2)$. $n_o$ induces a filtration:
$$(SR\mathbb{D}_1^{<p}(\uparrow^n))^0\subset (SR\mathbb{D}_1^{<p}(\uparrow^n))^1\subset (SR\mathbb{D}_1^{<p}(\uparrow^n))^2\subset\cdots$$

Let $O^m$ be the quotient of $S_\mathbb{F}((SR\mathbb{D}_1^{<p}(\uparrow^n))^m)$ by O4T, STU-like and $H_n$. We get a sequence: $$O^0\stackrel{o^0}{\longrightarrow}O^1\stackrel{o^1}{\longrightarrow}O^2\stackrel{o^2}{\longrightarrow}\cdots$$

$O^0$ is $OSR\mathcal{A}_1^{<p}(\uparrow^n)$, and the direct limit of the sequence if $SR\mathcal{A}_1^{<p}(\uparrow^n)/H_n$. As usual, we need to find an inverse to $o^m$.

Let $\beta^m:S_\mathbb{F}((SR\mathbb{D}_1^{<p}(\uparrow^n))^m)\rightarrow S_\mathbb{F}((SR\mathbb{D}_1^{<p}(\uparrow^n))^{m-1})$ be defined as follows: if $n_o(D)<m$, $\beta^m(D)=D$. Otherwise, let $v$ be the highest label involved in an unordered pair, and define:

$$\begin{array}{c} v\\w \end{array}
\mathfigns{1.5}{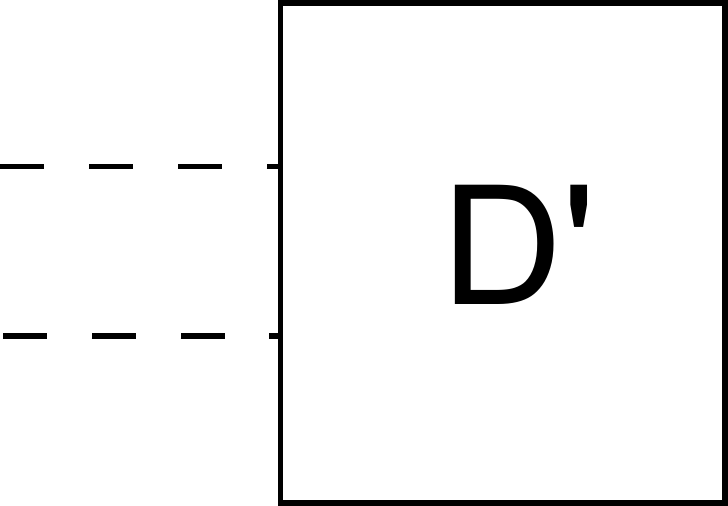} \quad \stackrel{\beta^m}{\longmapsto} \quad
\begin{array}{c} w\\v \end{array}
\mathfigns{1.5}{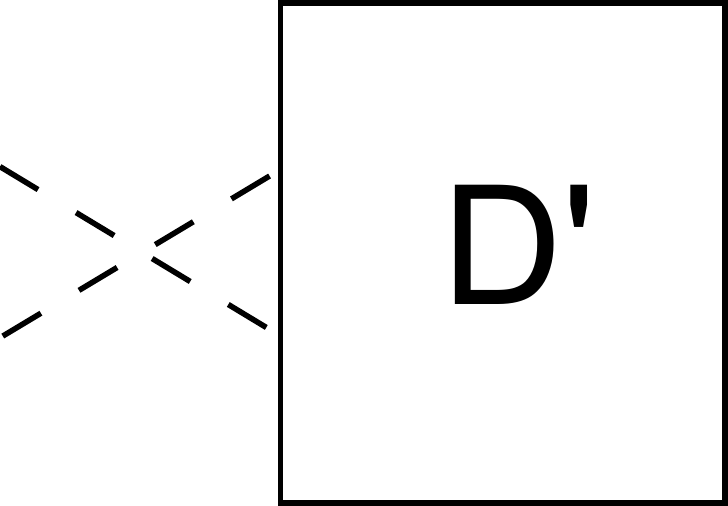} \quad+\quad
\left<v,u\right>
\mathfigns{1.5}{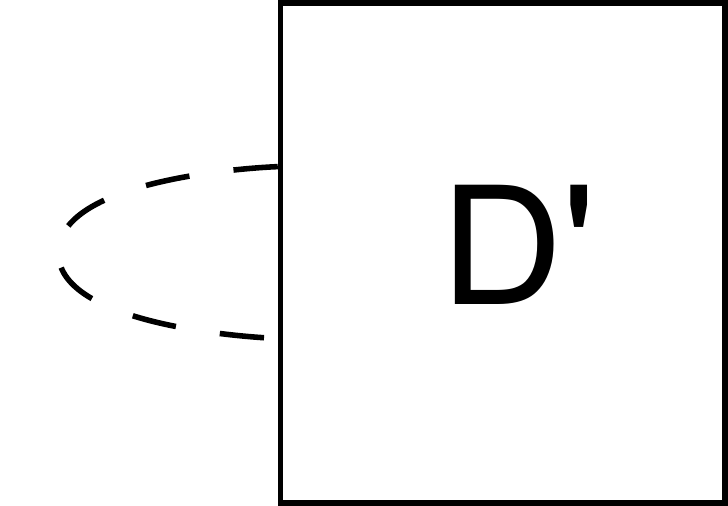}$$

We claim that $\beta^m$ induces a map $\beta^m:O^m\rightarrow O^{m-1}$. Indeed, if $u\in H_n$ then $\beta^m(u)$ is also in $H_n$. If $u$ is an STU-like relation:
$$u=u_1+u_2+u_3=\begin{array}{c} v\\w \end{array}
\mathfigns{1.5}{section3_section34_beta_def1} \quad - \quad
\begin{array}{c} w\\v \end{array}
\mathfigns{1.5}{section3_section34_beta_def2} \quad-\quad
\left<v,u\right>
\mathfigns{1.5}{section3_section34_beta_def3}$$
there are sevaral cases: If $v$ is the highest label in $u_2$ involved in an unordered pair, or $w$ is the highest label in $u_2$ involved in an unordered pair, then by definition $\beta^m(u)=0$. If the highest label involved in an unordered pair in $u_1$ (and therefore also in $u_2$) is $z$ which is immediately above $v$, then we have:

$$\beta^m(u)=\beta^m\left(\begin{array}{c} z\\v\\w \end{array}\mathfigns{1.5}{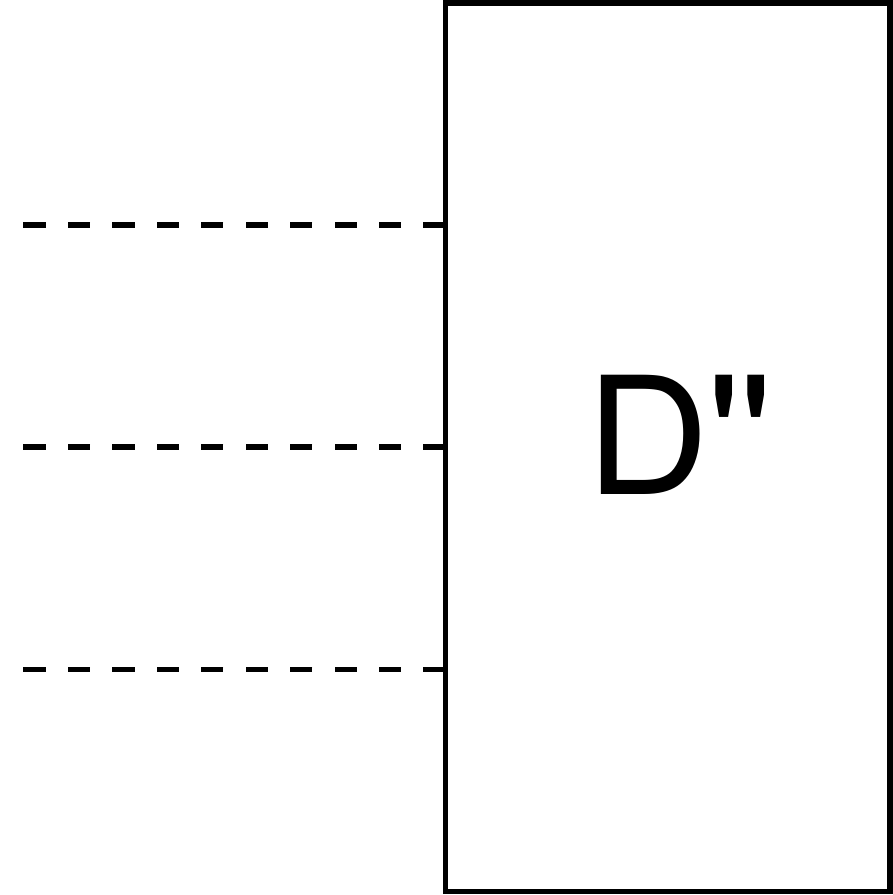}\quad - \quad
\begin{array}{c} z\\w\\v \end{array}\mathfigns{1.5}{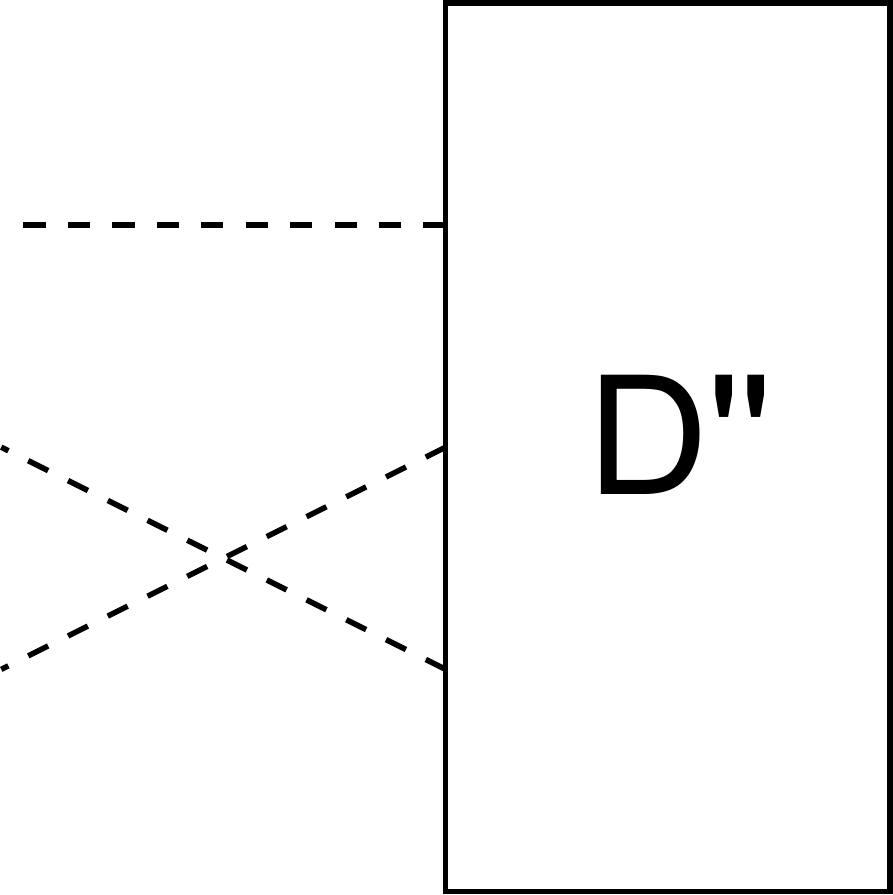}\quad - \quad
<v,w>\cdot\begin{array}{c} z\\ \, \\ \,  \end{array}\mathfigns{1.5}{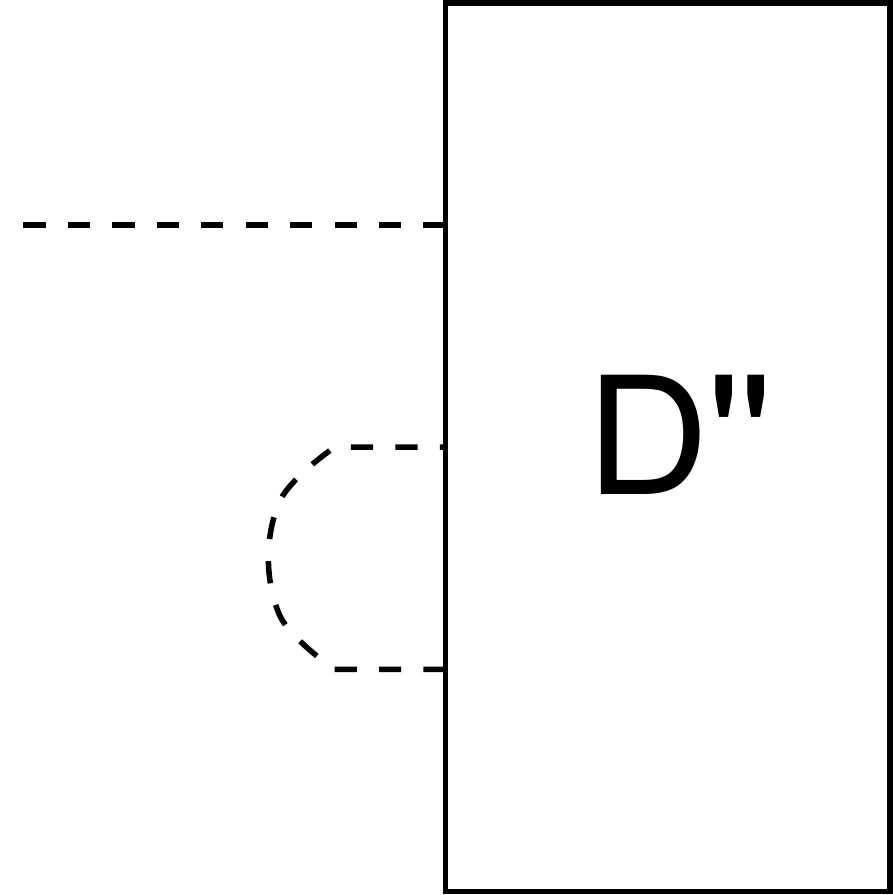}\quad\right) \approx$$
$$\approx\quad\begin{array}{c} v\\z\\w \end{array}\mathfigns{1.5}{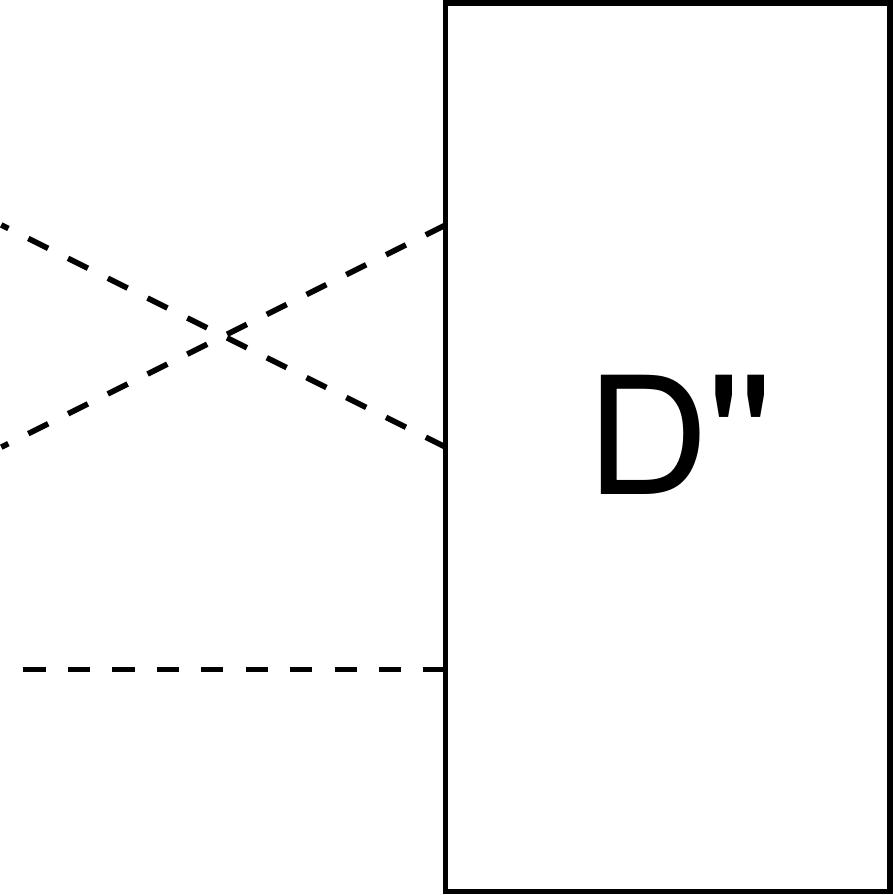}\quad + \quad
<z,v>\cdot\begin{array}{c} \, \\ \, \\w \end{array}\mathfigns{1.5}{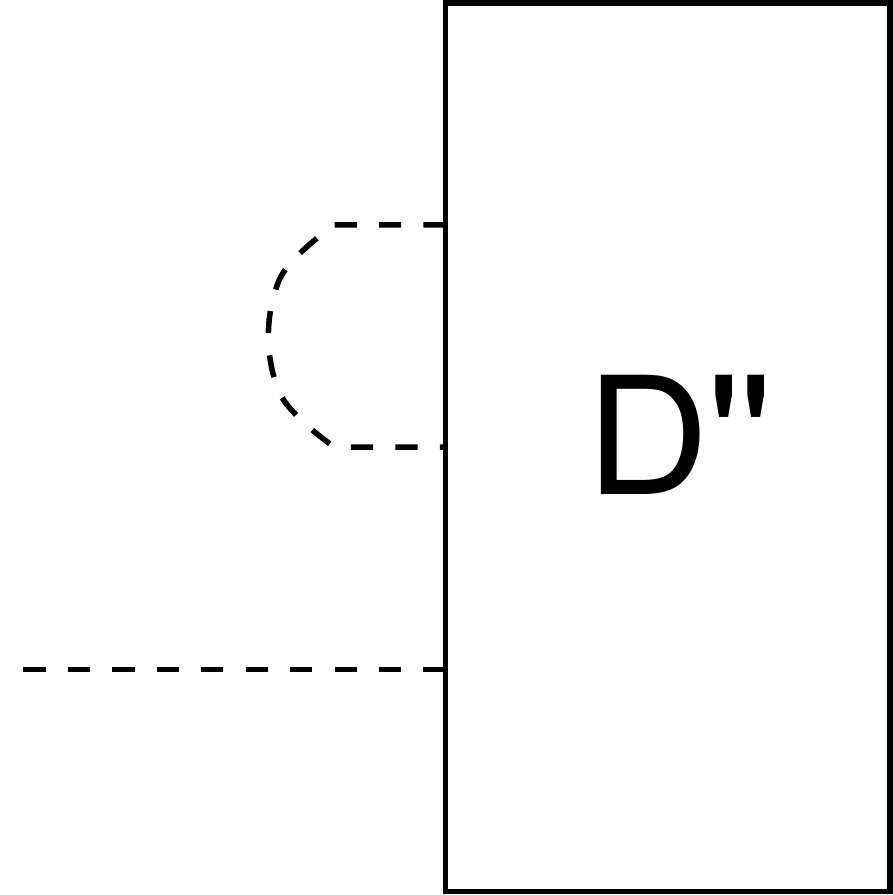}\quad - \quad
\begin{array}{c} w\\z\\v \end{array}\mathfigns{1.5}{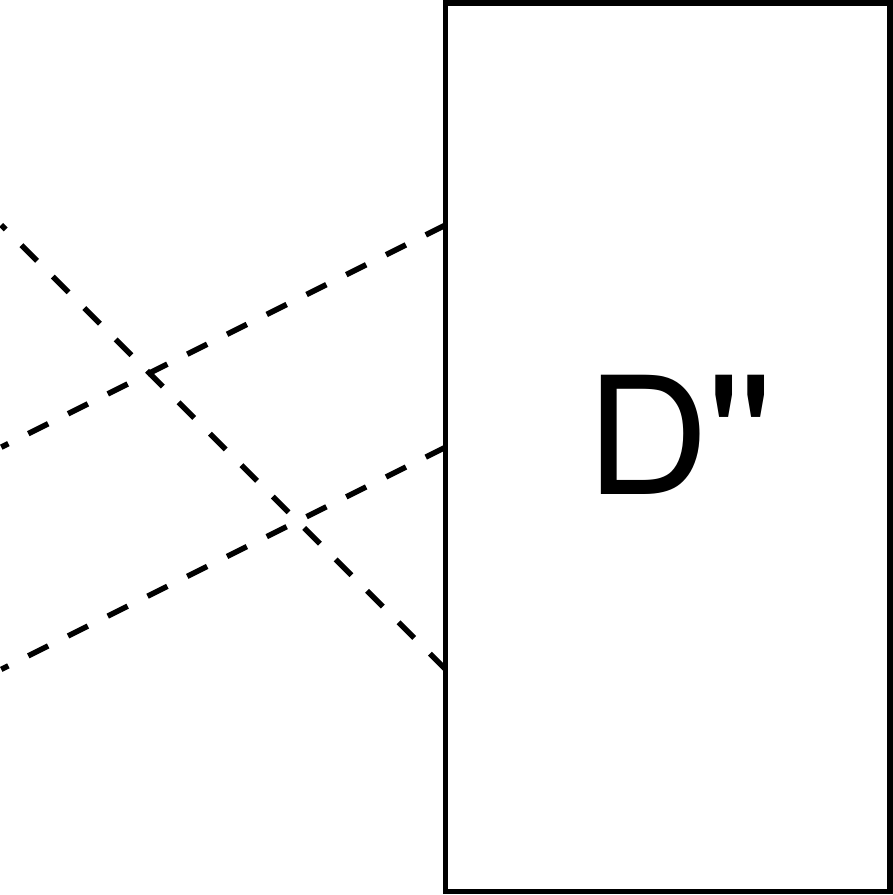}\quad - \quad$$
$$-\quad<z,w>\cdot\begin{array}{c} \, \\v\\ \, \end{array}\mathfigns{1.5}{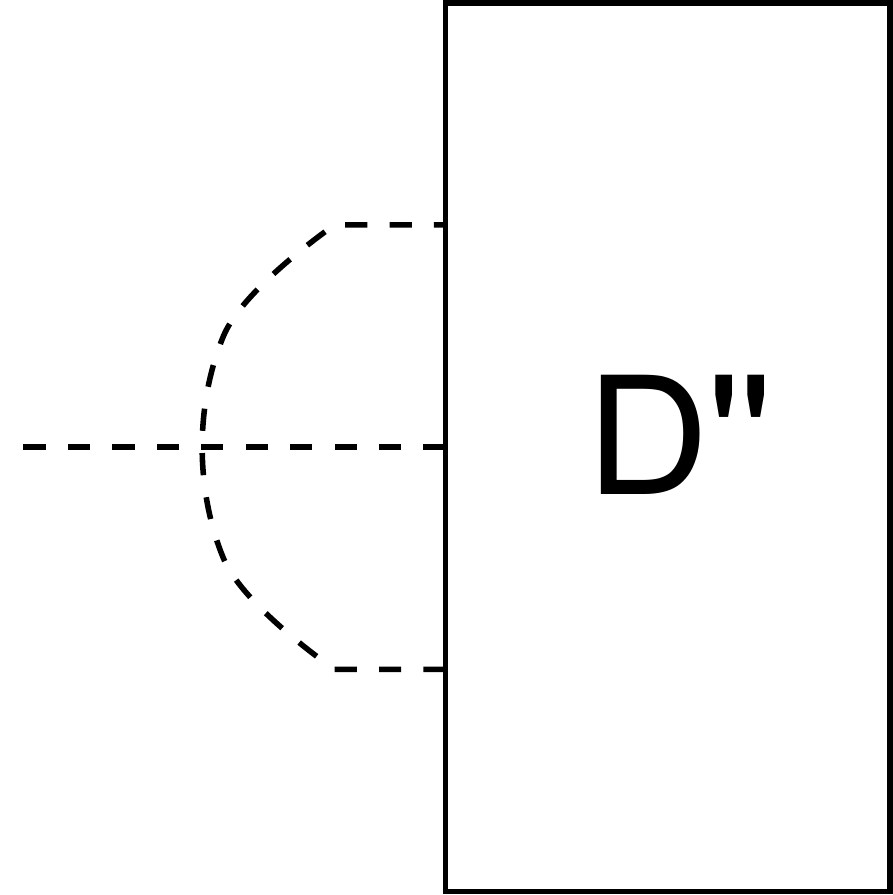}\quad - \quad
<v,w>\cdot\begin{array}{c} z\\ \, \\ \, \end{array}\mathfigns{1.5}{section3_section34_proof_proof13}\quad \approx \quad
\begin{array}{c} v\\w\\z \end{array}\mathfignumns{1.5}{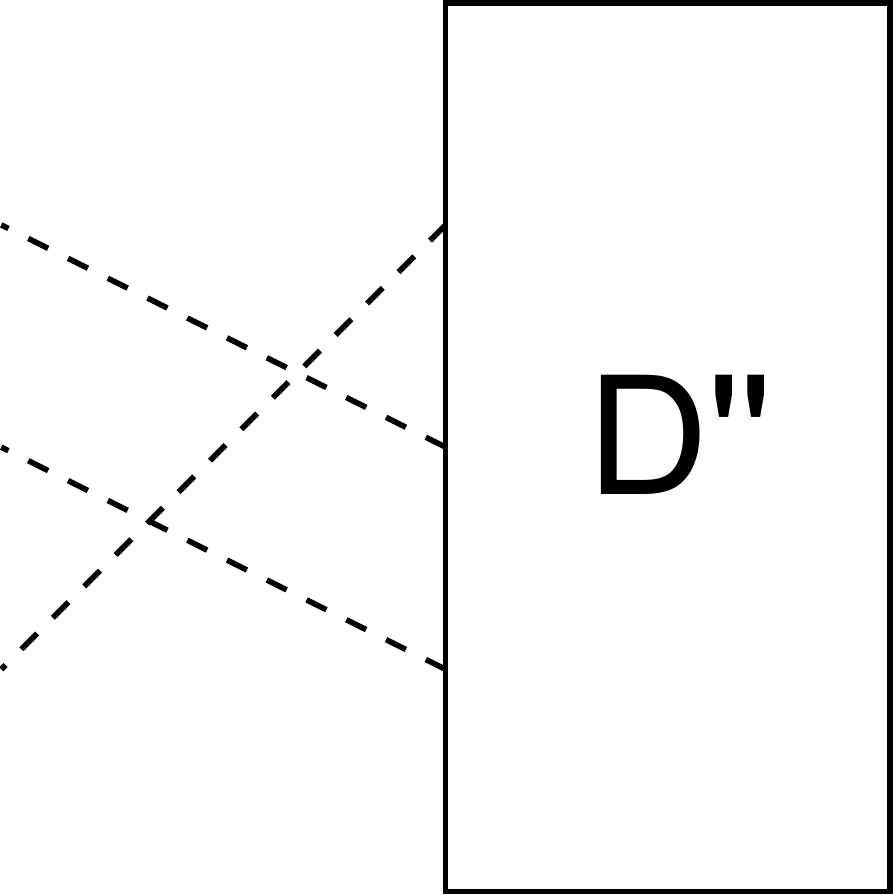}{1}\quad + \quad$$
$$+\quad<z,w>\cdot\begin{array}{c} \, \\v\\ \, \end{array}\mathfignumns{1.5}{section3_section34_proof_proof24}{2}\quad + \quad
<z,v>\cdot\begin{array}{c} \, \\ \, \\w \end{array}\mathfignumns{1.5}{section3_section34_proof_proof22}{3}\quad - \quad
\cdot\begin{array}{c} w\\v\\z \end{array}\mathfignumns{1.5}{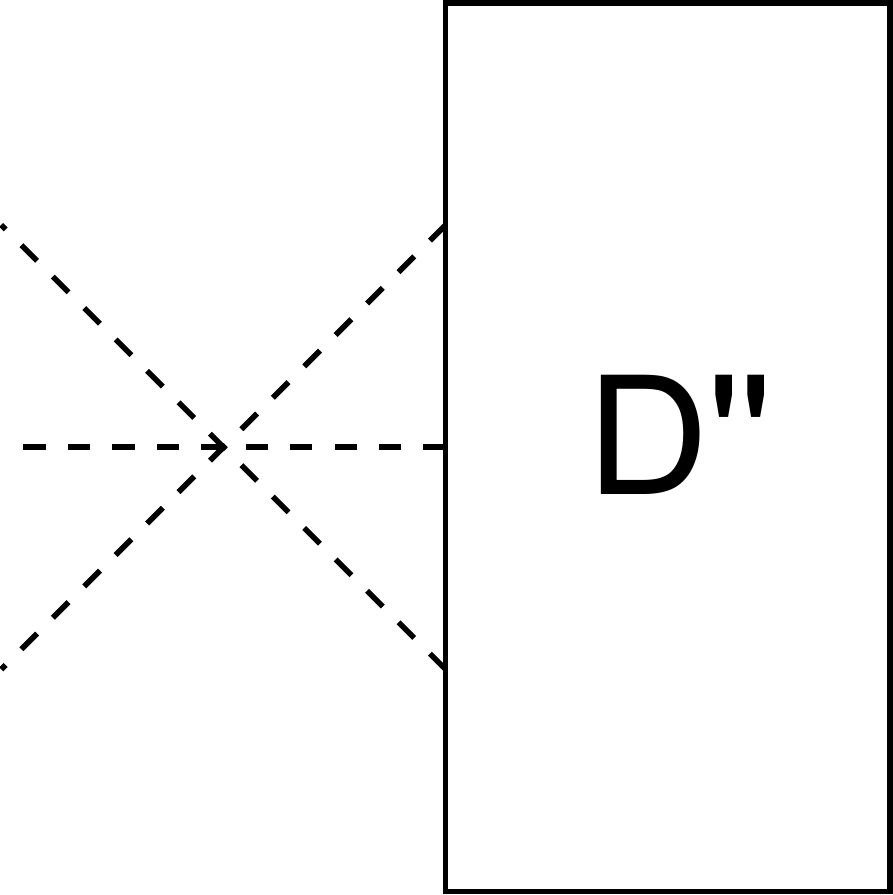}{4}\quad - \quad$$
$$-\quad<z,v>\cdot\begin{array}{c} \, \\ \, \\w \end{array}\mathfignumns{1.5}{section3_section34_proof_proof22}{5}\quad - \quad
<z,w>\cdot\begin{array}{c} \, \\v\\ \, \end{array}\mathfignumns{1.5}{section3_section34_proof_proof24}{6}\quad - $$
$$- \quad
<v,w>\cdot\begin{array}{c} z\\ \, \\ \, \end{array}\mathfignumns{1.5}{section3_section34_proof_proof13}{7}\quad \approx \quad 0$$
The last equivalence is true because \circlenum{2} cancels \circlenum{6}, \circlenum{3} cancels \circlenum{5}, and \circlenum{1}, \circlenum{4} and \circlenum{7} are STU-like.

In all other cases, $\beta^m(u)$ is a sum of STU-like relations.

If $u$ is an O4T relation:
$$u=\mathfig{4}{section3_section34_remark1}-\mathfig{4}{section3_section34_remark2}+$$
$$+\mathfig{4}{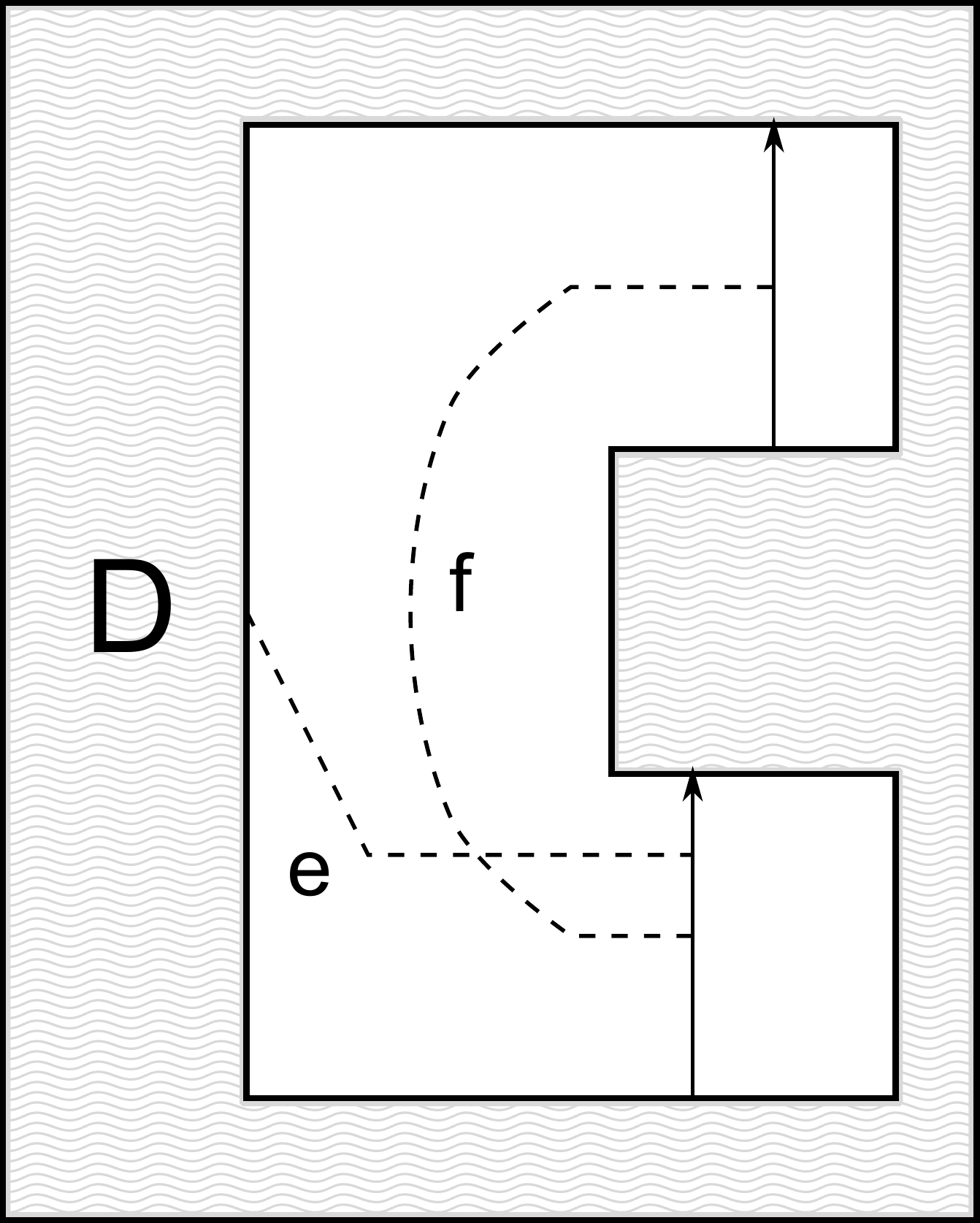}-\mathfig{4}{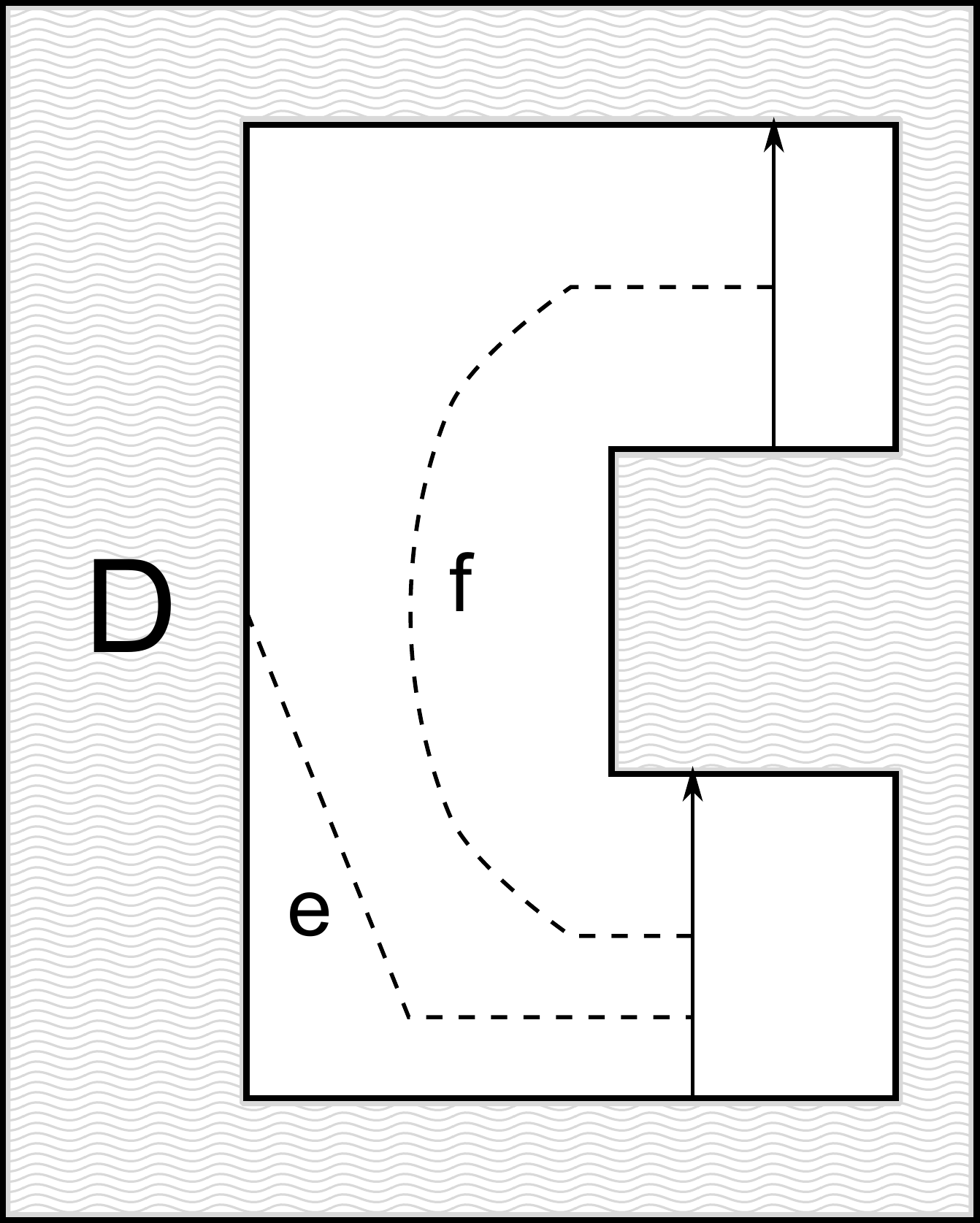}$$
then again we have to deal with several cases: If $\textbf{e}$ is a chord, then $\beta^m(u)$ is a sum of O4T relations. Similarly, if $\textbf{e}$ is a labeled edge, and its label is not ``touched" by $\beta^m$ in either of the summands of $u$, then again $\beta^m(u)$ is a sum of O4T relations. And if $\textbf{e}$ is labeled and its label is touched by $\beta^m$ in some (or all) of the summands of $u$, then $\beta^m(u)$ is an O4T relation + some terms which are equivalent to $0$ according to remark \ref{remark_H}.

It is easy to see that $\beta^m$ is the inverse of $o^{m-1}$, which completes the proof.
\end{proof}

The map $\tilde{u}_s:RU\hat{\textbf{t}}_{1,3}\rightarrow SR\mathcal{A}_1^{<p}(\uparrow^n)/H_3$ induces a map $\tilde{u}_o:RU\hat{\textbf{t}}_{1,3}\rightarrow OSR\mathcal{A}_1^{<p}(\uparrow^3)$ by composing with the isomorphism $\beta:SR\mathcal{A}_1^{<p}(\uparrow^n)/H_3\rightarrow OSR\mathcal{A}_1^{<p}(\uparrow^3)$. Thus, in order to prove theorem \ref{theorem_u_n} it is enough to prove that $\tilde{u}_o$ is an isomorphism.

For $n=2$ there are no relations in $OSR\mathcal{A}_1^{<p}(\uparrow^n)$, thus it is the free algebra generated by $\tilde{u}_o(x_1)$ and $\tilde{u}_o(y_1)$. $RU\hat{\textbf{t}}_{1,2}$ itself is the free algebra generated by $x_1$ and $y_1$. Therefore, we have completed the proof of theorem \ref{theorem_u_n} for $n=2$. For $n=3$ we will need yet another restriction, which is the content of the next (and final) subsection.

\subsection{Restriction to Fully Ordered Diagrams}
We begin with some notations:

In a diagram $D\in OSR\mathbb{D}_1^{<p}(\uparrow^3)$, each component is an edge. We number the strands from left to right. An edge with one labeled vertex and the other vertex on strand $i$ will be called an $i$ labeled edge. An edge with one vertex on strand $i$ and the other vertex on edge $j$ will be called an $i$-$j$ edge. Note that we will only have $i=1$ or $i=2$.

Let $FOSR\mathbb{D}_1^{<p}(\uparrow^3)\subset OSR\mathbb{D}_1^{<p}(\uparrow^3)$ (fully ordered simple restricted diagrams) be the union of $H_3$ with the subset of all diagrams with the following property: the labels of all $1$ labeled edges are smaller than the labels of all the $2$ labeled edges, and the vertices of all $1$ labeled edges on strand $1$ are lower than the vertices on strand $1$ of all $1$-$2$ edges. Denote by $FOSR\mathcal{A}_1^{<p}(\uparrow^3)$ the quotient of $FOSR\mathbb{D}_1^{<p}(\uparrow^3)$ by STU-like, O4T and $H_3$.

\begin{proposition}
The obvious map $f:FOSR\mathcal{A}_1^{<p}(\uparrow^3)\rightarrow OSR\mathcal{A}_1^{<p}(\uparrow^3)$ is an isomorphism.
\end{proposition}

\paragraph{Note:} There is no obvious multiplication in $FOSR\mathcal{A}_1^{<p}(\uparrow^3)$. The content of the proposition is that $f$ is an isomorphism of vector spaces. After we show that $f$ is indeed an isomorphism, it will induce a multiplication on $FOSR\mathcal{A}_1^{<p}(\uparrow^3)$ by pulling back the multiplication of $OSR\mathcal{A}_1^{<p}(\uparrow^3)$.

\begin{proof}
Let $D\in OSR\mathbb{D}_1^{<p}(\uparrow^3)$ be a diagram not in $H_3$. A pair of edges $e_1$, $e_2$ is an unordered pair if $e_1$ is a $1$ labeled edge and $e_2$ is either a $2$ labeled edge with a smaller label or a $1$-$2$ edge with a lower vertex on strand $1$. In the first case define $n_f(e_1,e_2):=\#\{\text{labeled vertices between $e_1$ and $e_2$}\}+1$, and in the second case define $n_f(e_1,e_2):=\#\{\text{vertices on strand $1$ between $e_1$ and $e_2$}\}+1$. Define $n_f(D):=\sum\limits_{\text{$e_1$,$e_2$ unordered}}n_f(e_1,e_2)$.

Let $(OSR\mathbb{D}_1^{<p}(\uparrow^3))^m$ be the union of $H_3$ and all diagrams $D$ with $n_f(D)\le m$. Let $F^m$ be the quotient of $S_\mathbb{F}((OSR\mathbb{D}_1^{<p}(\uparrow^3))^m)$ by STU-like, O4T and $H_3$. We get a sequence: $$F^0\stackrel{f^0}{\longrightarrow} F^1\stackrel{f^1}{\longrightarrow} F^2\stackrel{f^2}{\longrightarrow}\cdots$$
The direct limit of the sequence is $OSR\mathcal{A}_1^{<p}(\uparrow^3)$, and $F^0$ is isomorphic to $FOSR\mathcal{A}_1^{<p}(\uparrow^3)$. Therefore, what we need to do, as usual, is to find an inverse to $f^m$.

Let $\alpha^m:S_\mathbb{F}((OSR\mathbb{D}_1^{<p}(\uparrow^3))^m)\rightarrow S_\mathbb{F}((OSR\mathbb{D}_1^{<p}(\uparrow^3))^{m-1})$ be defined as follows: If $D\in H_3$ or $n_f(D)<m$, $\alpha^m(D)=D$. Else, find the highest $1$ labeled edge in $D$ such that either of the following holds:
\begin{enumerate}
\item[A.] The labeled vertex immediately below it belongs to a $2$ labeled edge.
\item[B.] The vertex immediately below it on strand $1$ belongs to a $1$-$2$ edge.
\end{enumerate}
Denote this edge by $e$. If A. applies to $e$, define:
$$\mathfig{3}{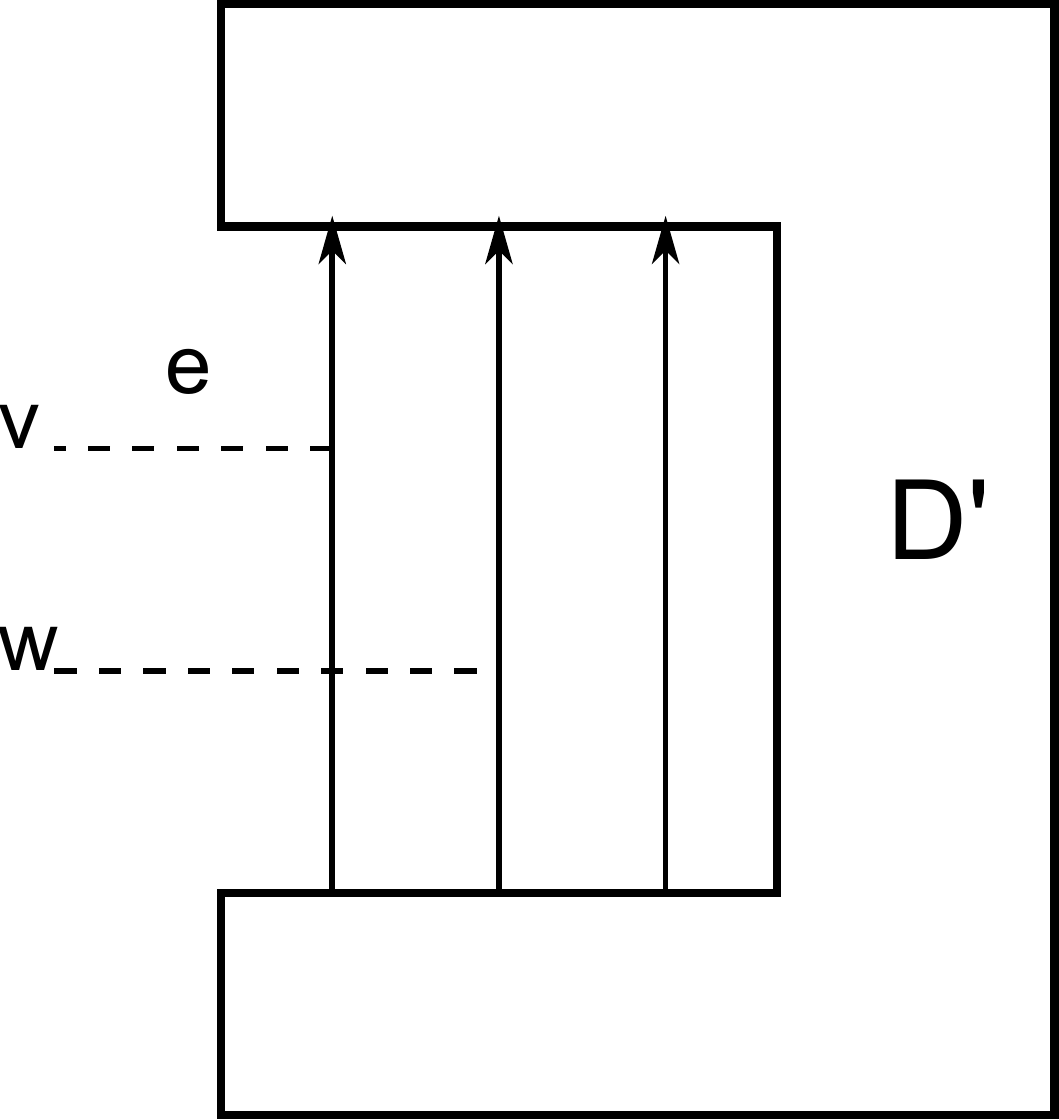}\stackrel{\alpha^m}{\longmapsto}\mathfig{3}{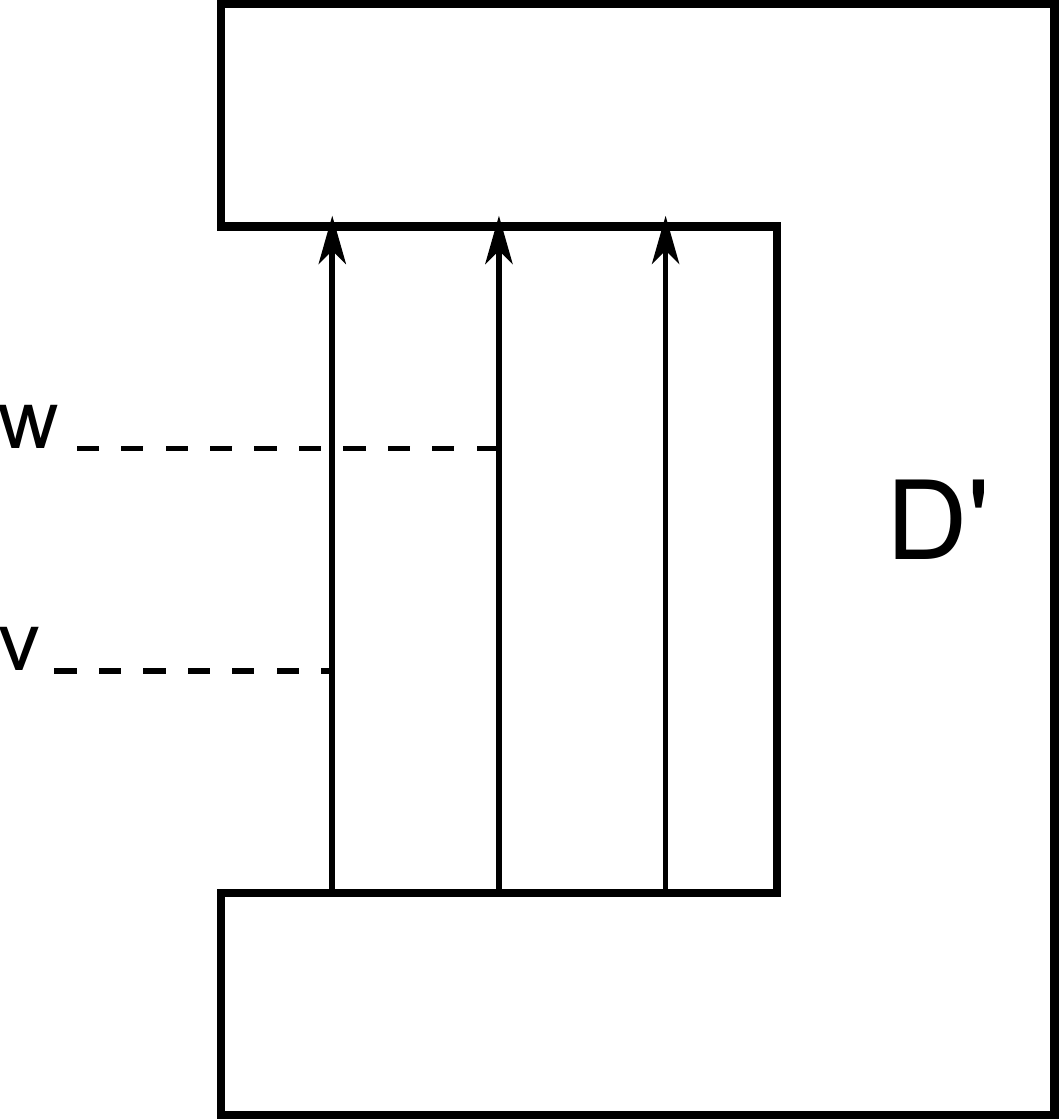}+\quad <v,w> \mathfigns{3}{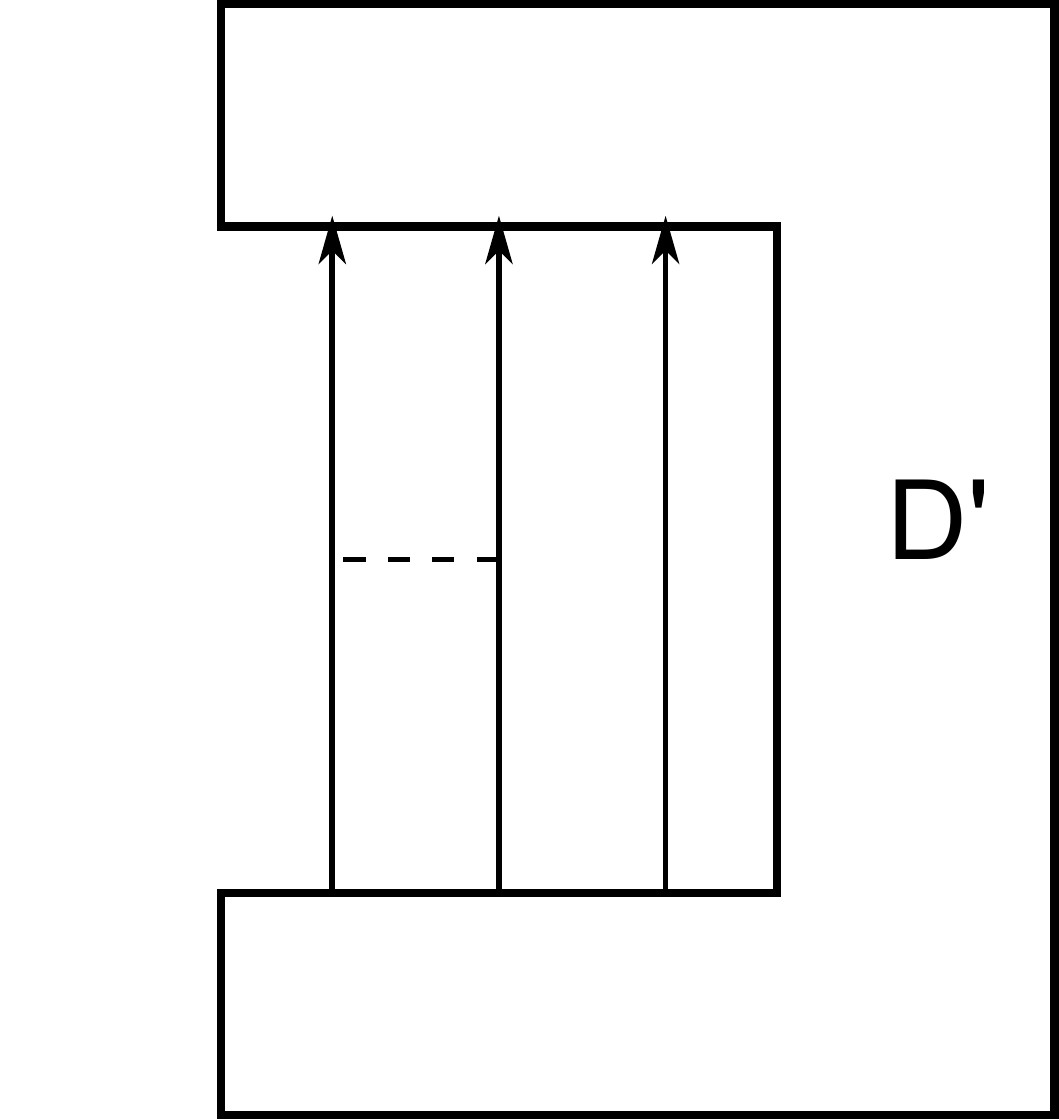}$$
If only B. applies to $e$, define:
$$\mathfig{3}{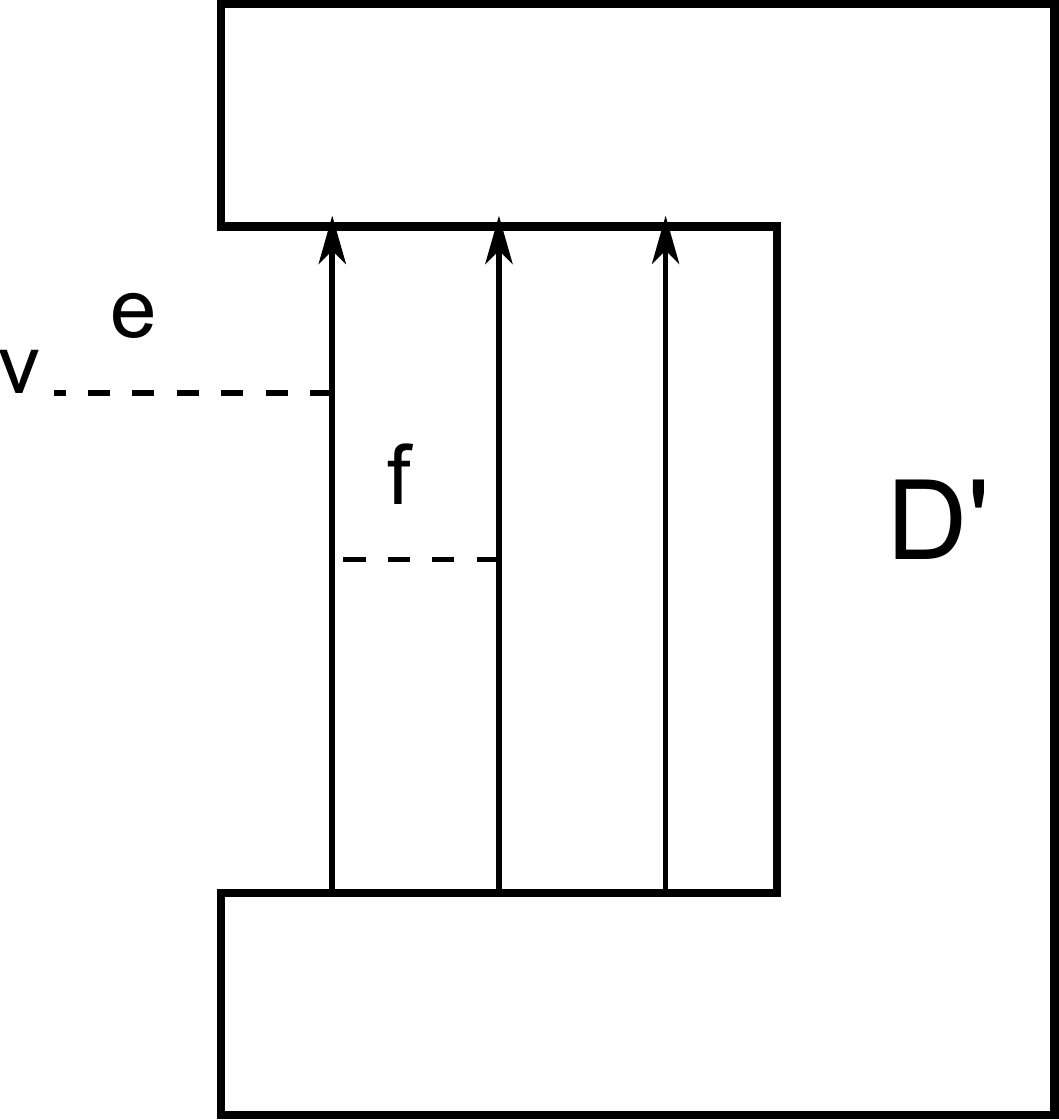}\stackrel{\alpha^m}{\longmapsto}$$
$$\mathfig{3}{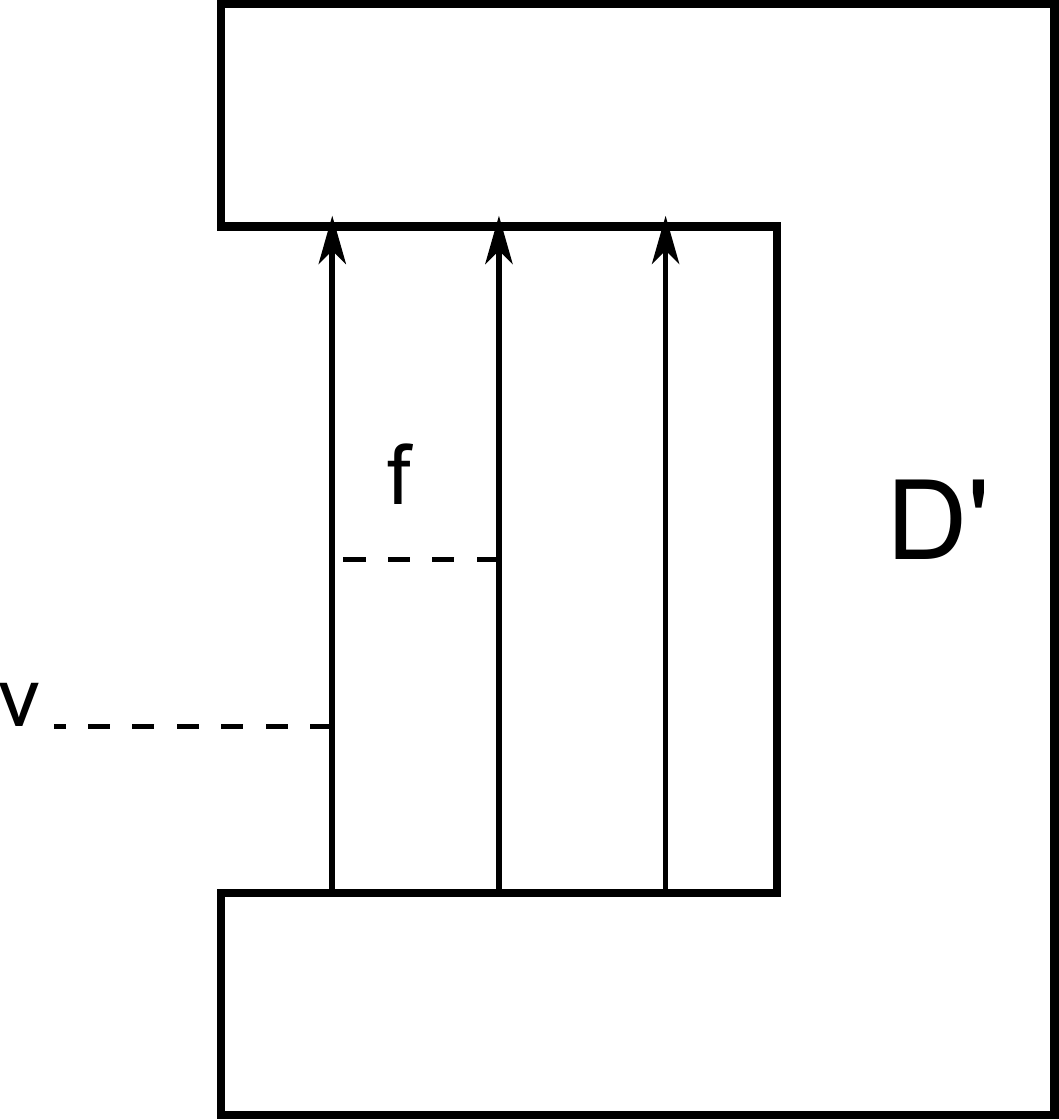}+\mathfignum{3}{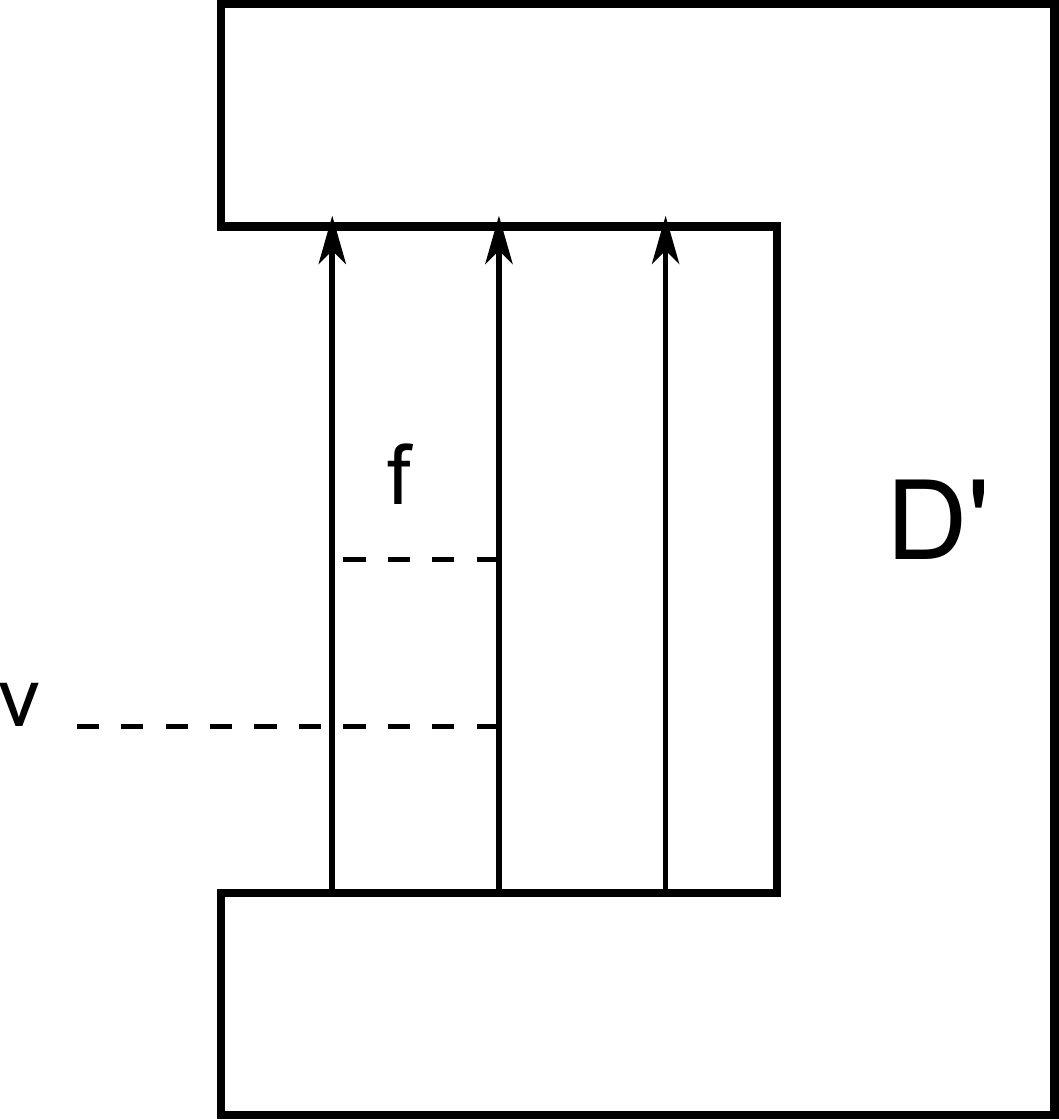}{*}-\mathfignum{3}{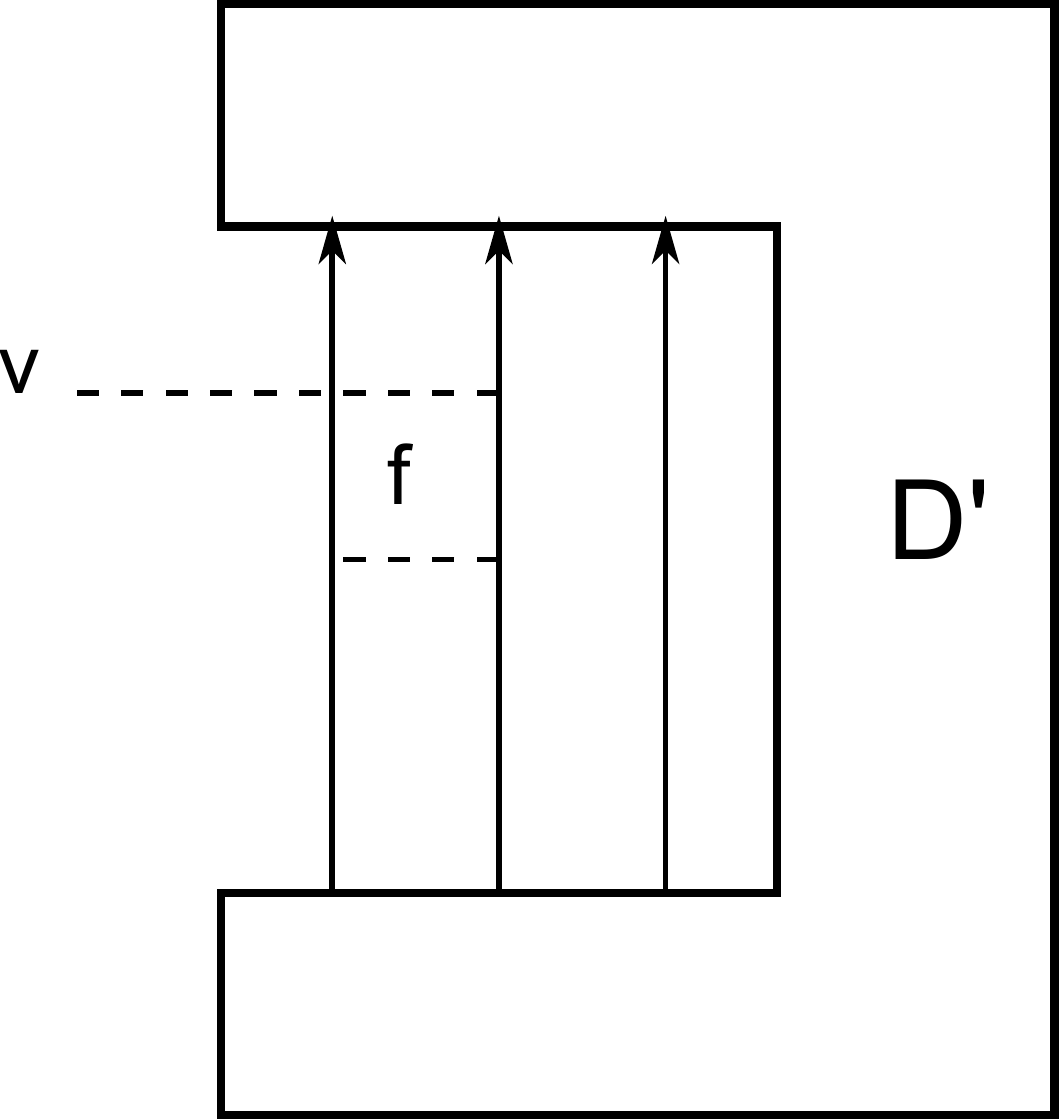}{*}$$
In the last $2$ summands (marked by \circlenum{*}) we should specify the location of the label $v$ in the linear order of the labels in $D$. This is determined as follows: If there is a $2$ labeled edge above the vertex of $f$, locate $v$ as the label immediately below it. Otherwise, locate $v$ as the highest vertex. This choice guaranties that $\alpha^m(D)$ is indeed in $OSR\mathcal{A}_1^{<p}(\uparrow^3)$.

We claim that $\alpha^m$ induces a map $\alpha^m:F^m\rightarrow F^{m-1}$. Indeed, assume $u$ is an STU-like relation. Since we are in $OSR\mathcal{A}_1^{<p}(\uparrow^3)$, we must have: $$u=u_1+u_2+u_3=$$
$$=\mathfig{3}{section3_section35_alpha11}-\mathfig{3}{section3_section35_alpha12}-\quad <v,w> \mathfigns{3}{section3_section35_alpha13}$$

We have $n_f(u_1)>n_f(u_2),n_f(u_3)$. Assume $n_f(u_1)=m$. If $v$ is the highest label with properties A. or B. then by definition $\alpha^m(u)=0$. Otherwise $\alpha^m(u)$ is equivalent to a sum of STU-like relations.

Assume now that $u$ is an O4T relation:
$$u=u_1+u_2+u_3+u_4=$$
$$=\mathfig{3}{section3_section35_alpha21}-\mathfig{3}{section3_section35_alpha22}+$$ $$+\mathfig{3}{section3_section35_alpha24}-\mathfig{3}{section3_section35_alpha23}$$

Clearly we have $n_f(u_1)>n_f(u_2)$ and $n_f(u_3)>n_f(u_4)$. Therefore, potentially we might have $n_f(u_i)=m$ only for $i=1$ and $i=3$. Assume first that only $n_f(u_1)=m$. If $v$ is not the highest label in $u_1$ with properties A. or B., then $\alpha^m(u)$ is equivalent to an O4T relation. If $v$ is the highest such label and only property B. applies to it, then by definition $\alpha^m(u)=0$. And if property A. also applies to it, then we have:

$$u=\mathfig{3.5}{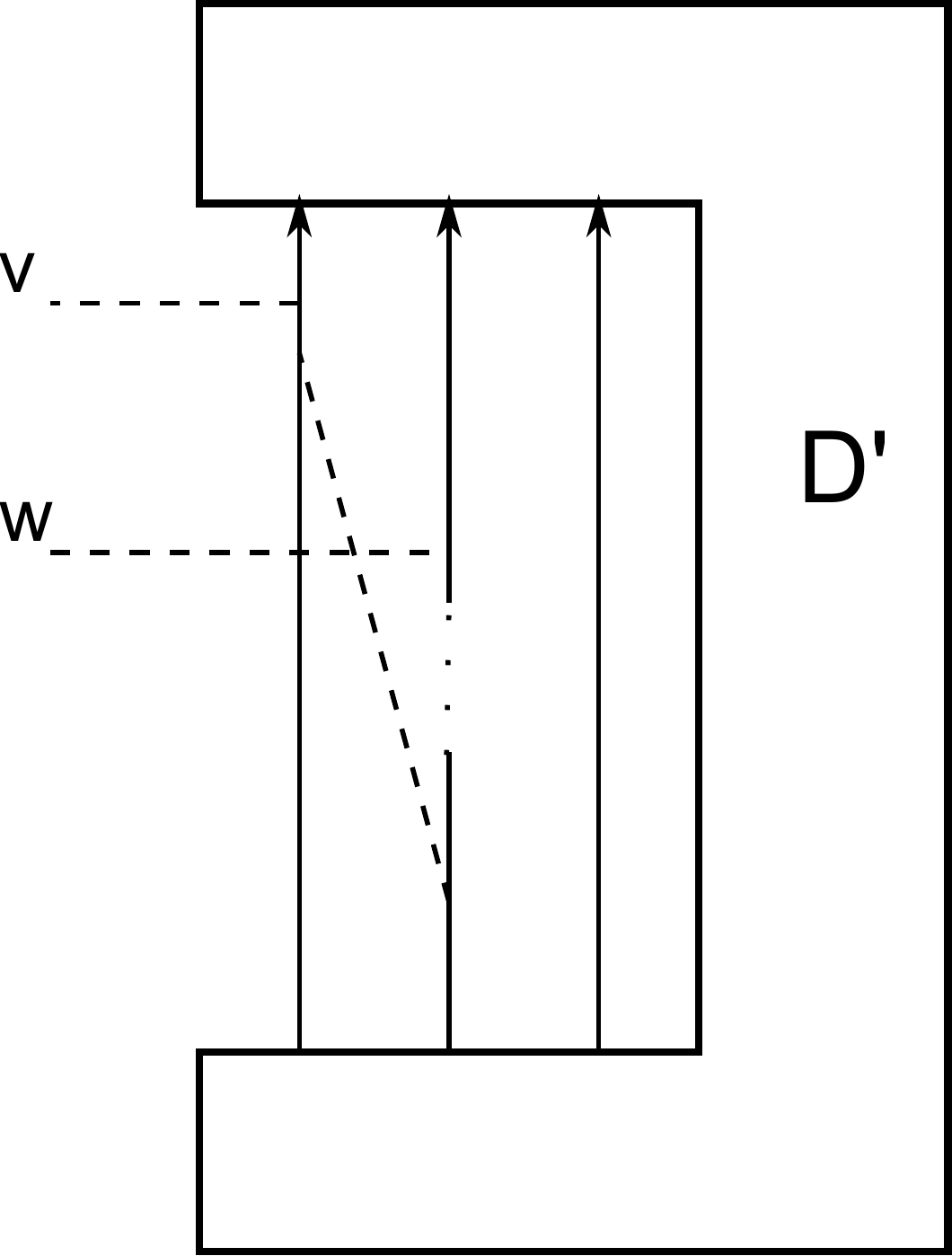}-\mathfig{3.5}{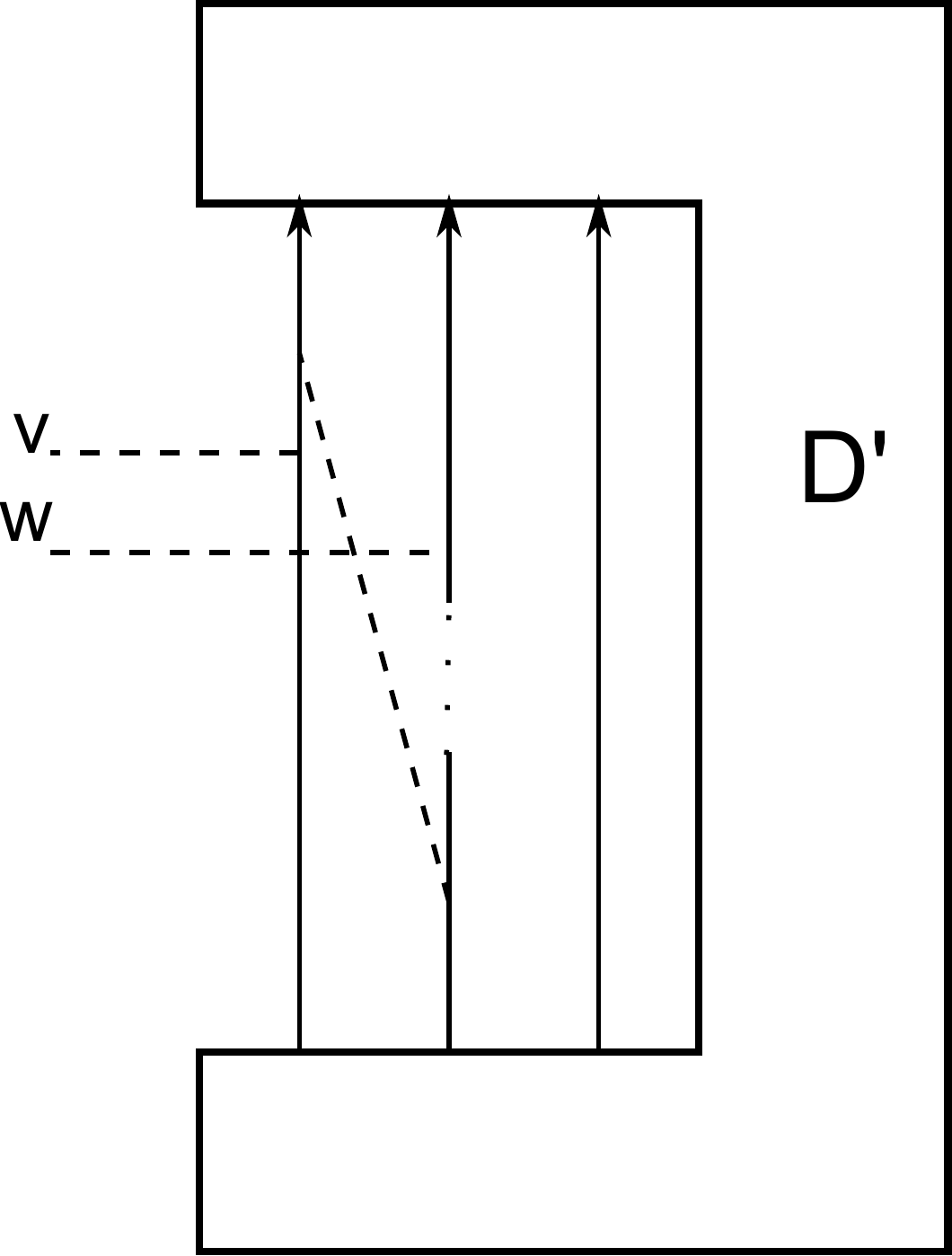}+\mathfig{3.5}{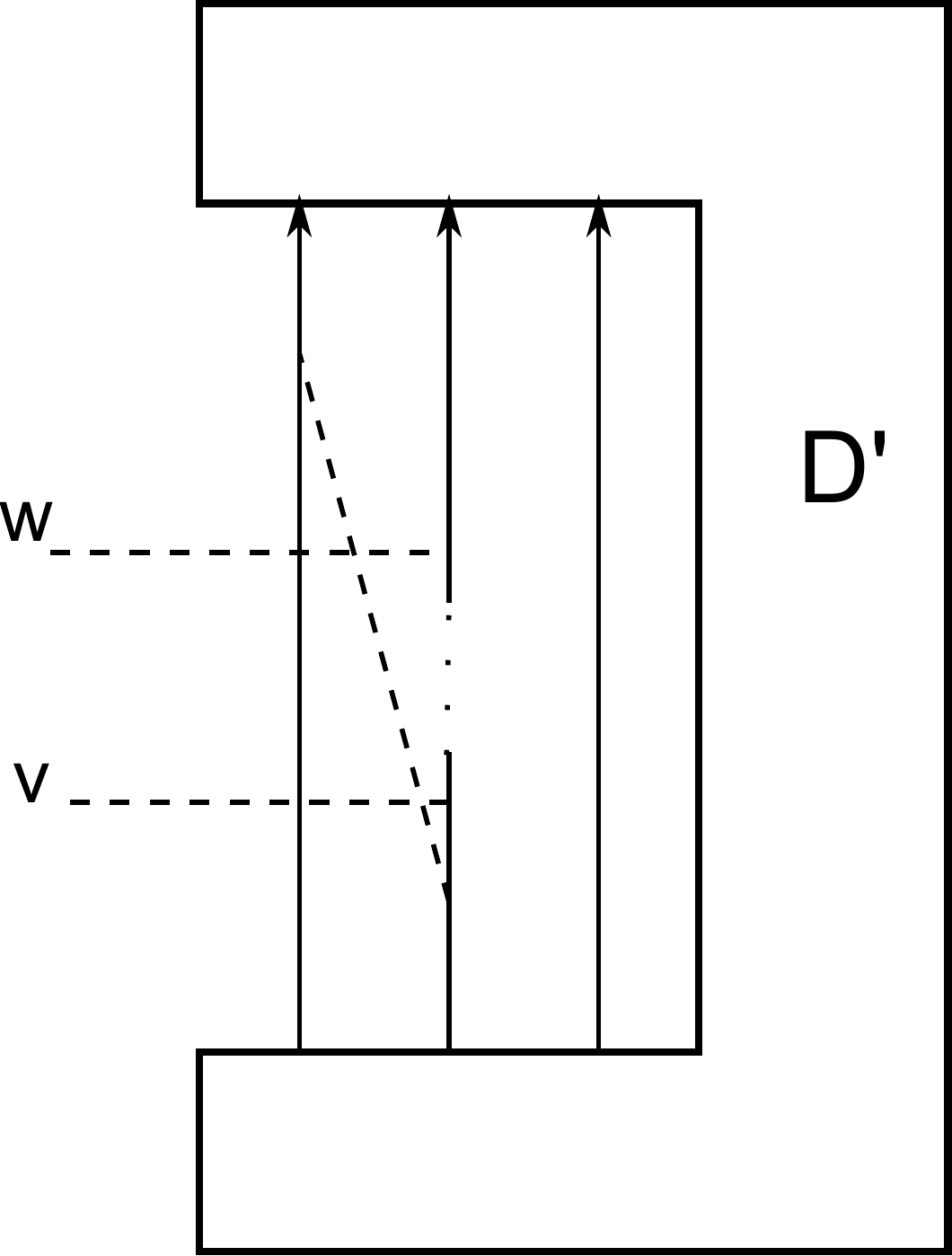}-$$ 

$$ -\mathfig{3.5}{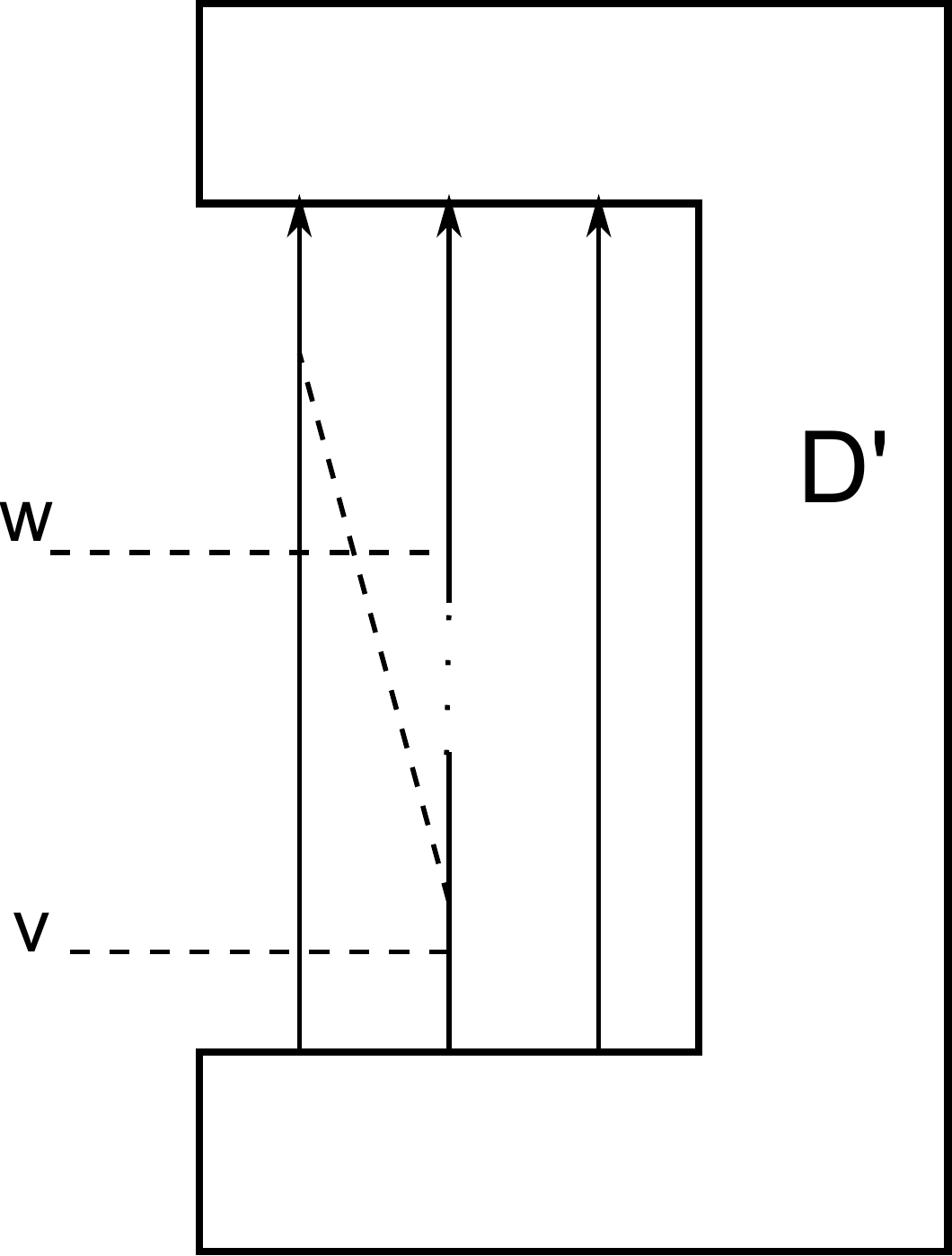}\stackrel{\alpha^m}{\longmapsto}\mathfignum{3.5}{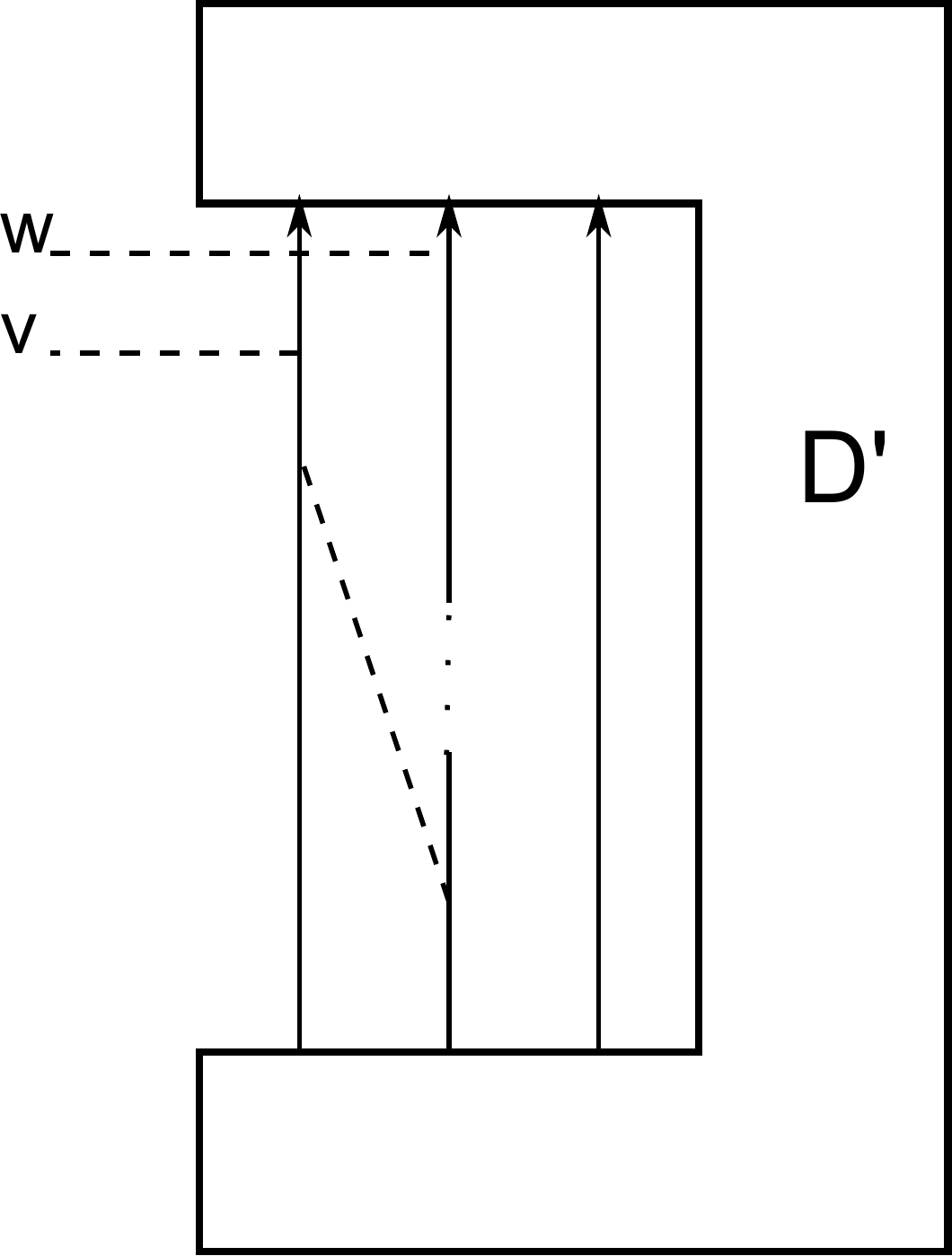}{1}+\quad<v,w>\mathfignumns{3.5}{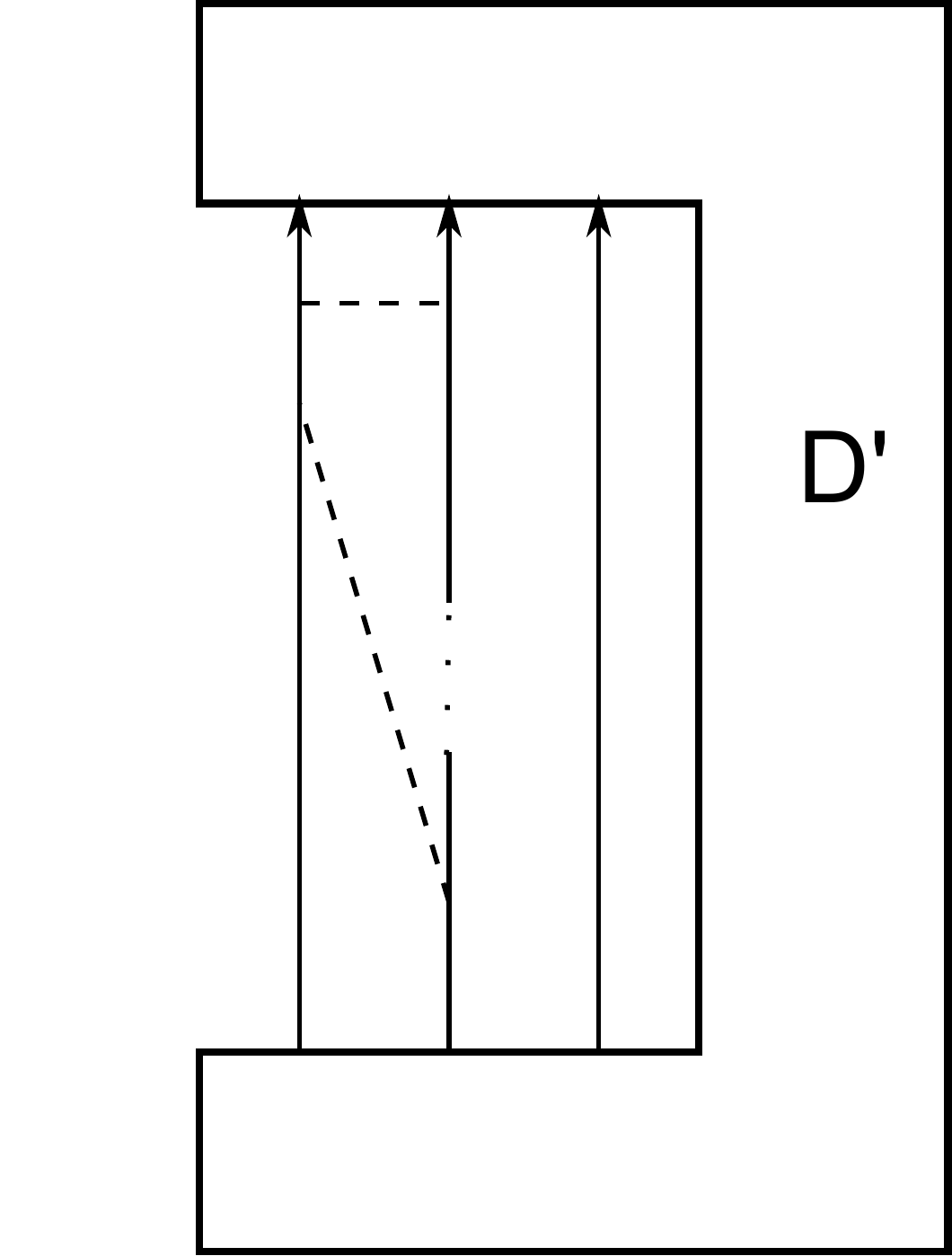}{2}\quad - $$
 
$$ - \mathfignum{3.5}{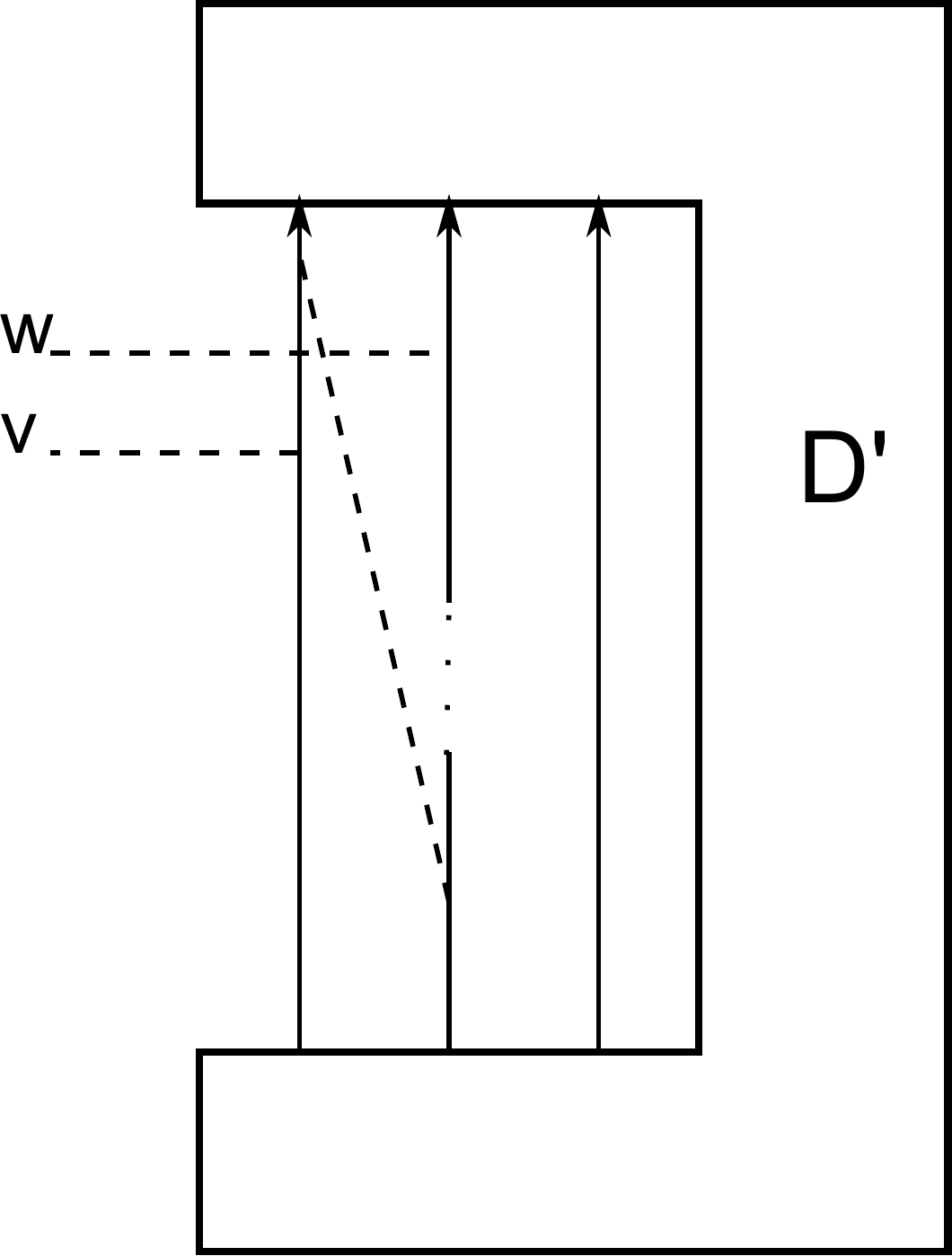}{3} - <v,w>\mathfignumns{3.5}{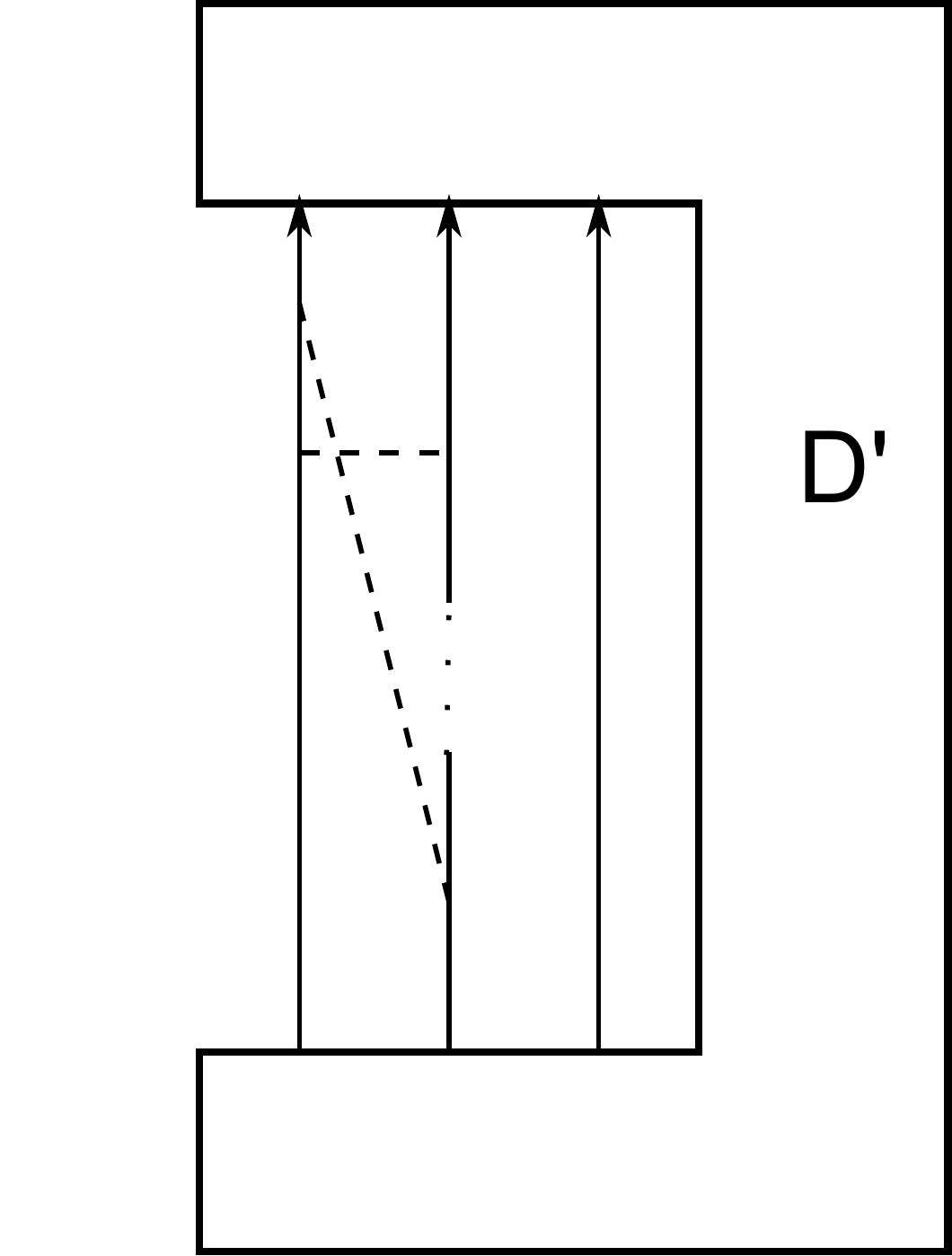}{4}\quad+\mathfignum{3.5}{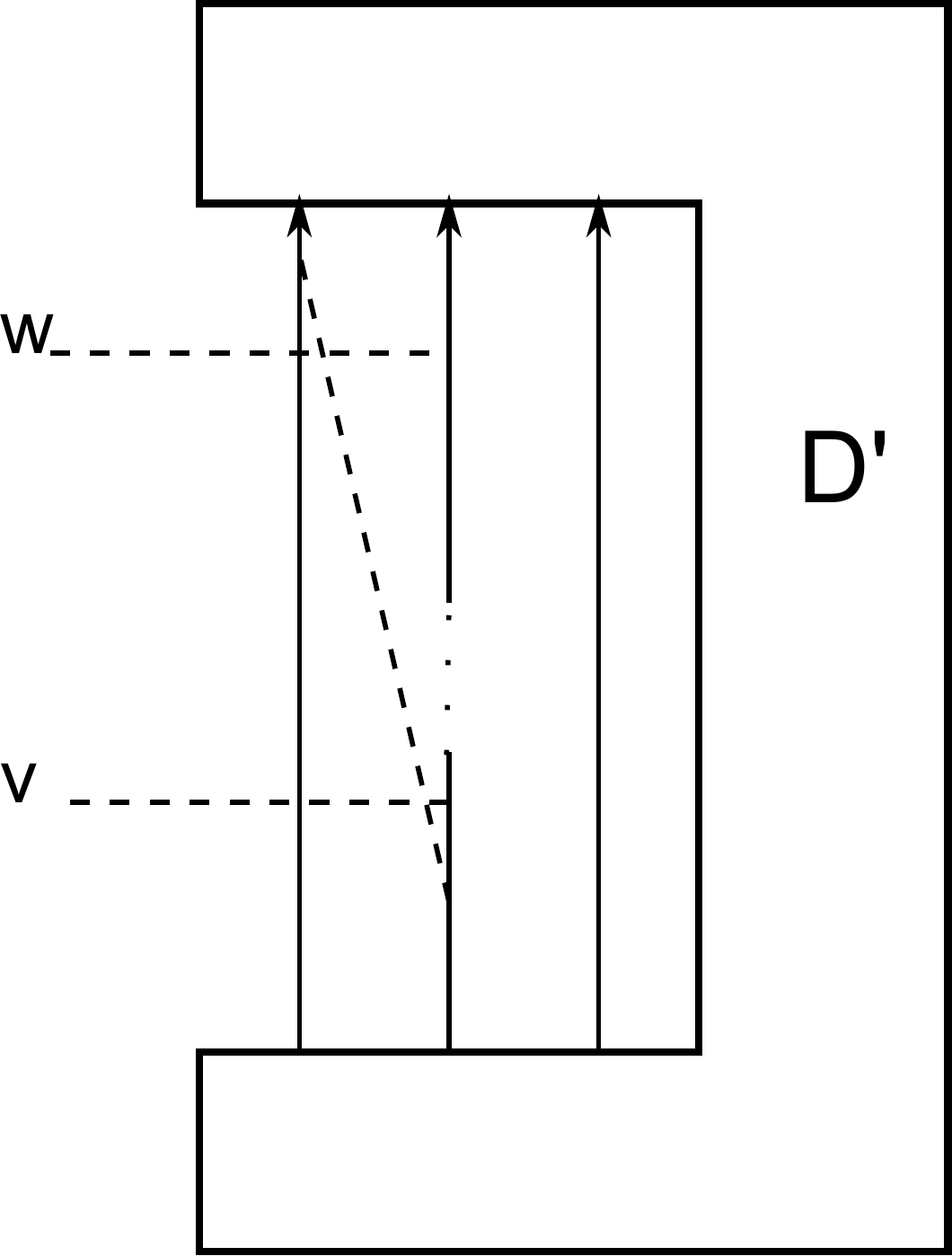}{5}-$$
$$-\mathfignum{3.5}{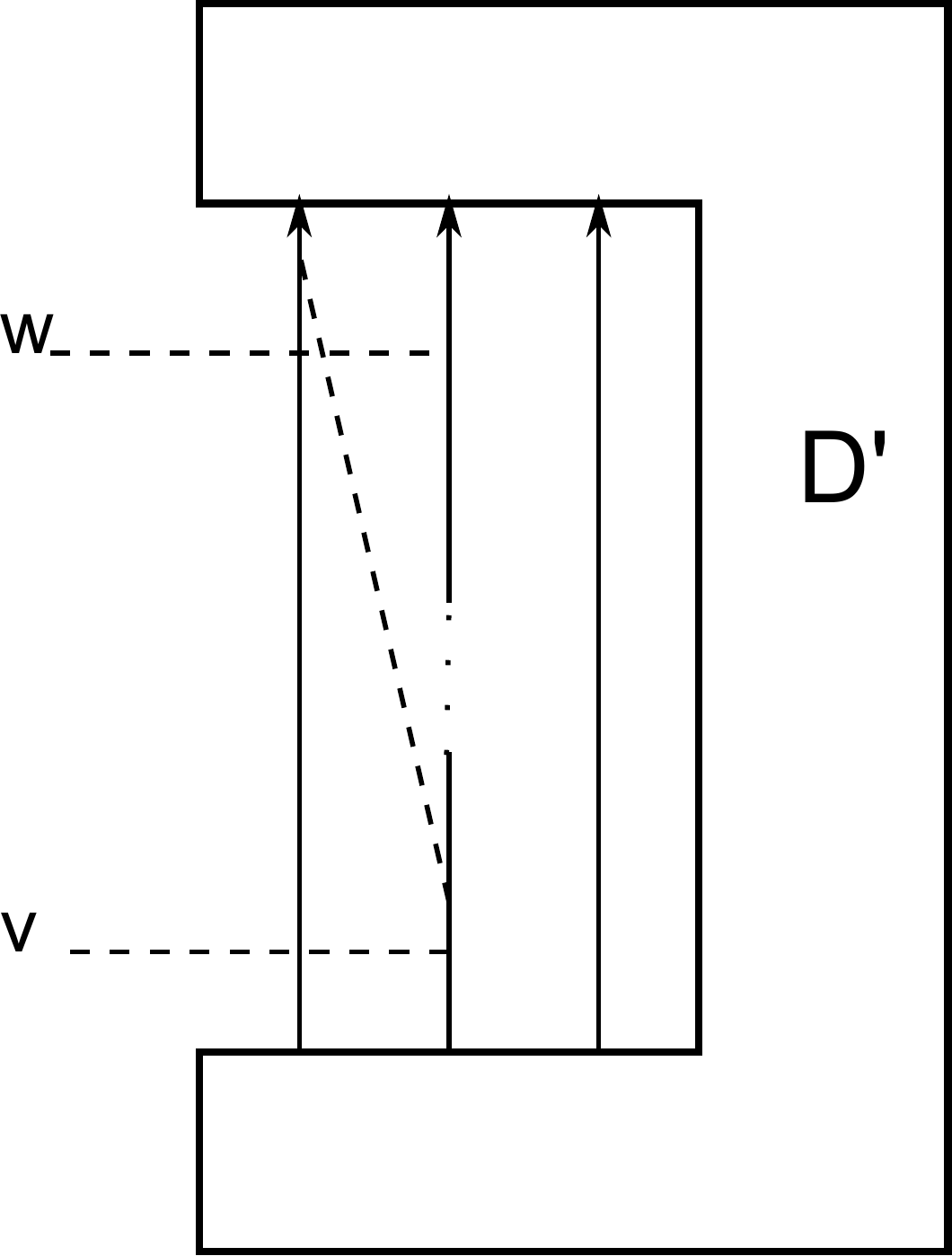}{6}\approx 0 $$
The last equivalence is true because \circlenum{1}, \circlenum{3}, \circlenum{5} and \circlenum{6} are an O4T relation, and \circlenum{2} and \circlenum{4} are equivalent to 0 by remark \ref{remark_H}. Note that in some of the diagrams involved in this calculation we might actually have a different order of the labels or a different order along strand $2$, but this does not affect the actual calculation.

If only $n_f(u_3)=m$ the argument is similar. If $n_f(u_1)=n_f(u_3)=m$ then the computation is more complicated, since we need to consider several possibilities, but the principles of the calculation are the same, and we leave it to the reader.

It is easy to see that $\alpha^m$ is the inverse of $f^{m-1}$, thus we have completed the proof.
\end{proof}

The map $\tilde{u}_o:RU\hat{\textbf{t}}_{1,n}\rightarrow OSR\mathcal{A}_1^{<p}(\uparrow^n)$ induces a map $\tilde{u}_f:RU\hat{\textbf{t}}_{1,n}\rightarrow FOSR\mathcal{A}_1^{<p}(\uparrow^n)$ by composing with the isomorphism $\alpha:OSR\mathcal{A}_1^{<p}(\uparrow^n)\rightarrow FOSR\mathcal{A}_1^{<p}(\uparrow^n)$. Thus, in order to prove theorem \ref{theorem_u_n} it is enough to prove that $\tilde{u}_f$ is an isomorphism.

In $S_\mathbb{F}(FOSR\mathbb{D}_1^{<p}(\uparrow^3))$ there are no STU-like relations. The only O4T relations involve two $1$-$2$ edges, and modulo $H_3$ they reduce to the relation:
$$\mathfig{3.5}{section3_section35_proof22}=\mathfig{3.5}{section3_section35_proof24}$$.

According to this observation we may further restrict $FOSR\mathbb{D}_1^{<p}(\uparrow^3)$. Let $\widetilde{FOSR\mathbb{D}_1^{<p}}(\uparrow^3)$ be the subset of $FOSR\mathbb{D}_1^{<p}(\uparrow^3)$ containing all diagrams in which the $1$ labeled edges are lower than any other edge (as before), and the vertices of all the $1$-$2$ edges have the same order on both strands. $FOSR\mathcal{A}_1^{<p}(\uparrow^3)$ will then be isomorphic to the span of $\widetilde{FOSR\mathbb{D}_1^{<p}}(\uparrow^3)\cup H_3$ quotiented by $H_3$, which is simply $S_\mathbb{F}(\widetilde{FOSR\mathbb{D}_1^{<p}}(\uparrow^3))$. We denote this space by $\widetilde{FOSR\mathcal{A}_1^{<p}}(\uparrow^3)$.

Each diagram in $\widetilde{FOSR\mathcal{A}_1^{<p}}(\uparrow^3)$ is a product of the elements $\tilde{u}_f(x_i)$, $\tilde{u}_f(y_i)$ ($i=1,2$) and $\tilde{u}_f(t_{12})$. Therefore, there is an obvious map $\widetilde{FOSR\mathcal{A}_1^{<p}}(\uparrow^3)\rightarrow RU\hat{\textbf{t}}_{1,3}$. We denote by $p$ the composition $p:FOSR\mathcal{A}_1^{<p}(\uparrow^3)\stackrel{\cong}{\rightarrow}\widetilde{FOSR\mathcal{A}_1^{<p}}(\uparrow^3)\rightarrow RU\hat{\textbf{t}}_{1,3}$. We claim that $p$ is an inverse to $\tilde{u}_f$.

Indeed, $\tilde{u}_f\circ p$ is clearly the identity. As for $p\circ \tilde{u}_f$, we need to show that for any $D\in RU\hat{\textbf{t}}_{1,3}$, the image $p\circ\alpha\circ\tilde{u}_o(D)$ in $RU\hat{\textbf{t}}_{1,3}$ is equivalent to $D$ via the relations of $U\hat{\textbf{t}}_{1,3}$. The only map in this composition which actually changes the underlying diagram of $D$ is $\alpha$. Following the definition of the maps $\alpha^m$ (which compose $\alpha$) shows that all we need is to verify the following relations in $U\hat{\textbf{t}}_{1,3}$:
$$-[v_1,t_{12}]=[v_2,t_{12}]$$
$$[v_1,w_2]=<v,w>t_{12}$$
Those relations indeed hold in $U\hat{\textbf{t}}_{1,3}$ (see \cite{Humbert} Definition 2.1.1 and Lemma 2.1.2). This completes the proof of theorem \ref{theorem_u_n}.

\newpage\section{Elliptic Associators and the LMO Functor}
\label{section_elliptic}
In this section we introduce the concept of elliptic associators and the specific associator defined in \cite{Calaque} and \cite{Enriquez}. We then study the relation between this elliptic associator and the elliptic structure relative to $\mathcal{A}^\partial\rightarrow\mathcal{A}_1^{<p}$ induced by the LMO functor, which we described in section \ref{section_LMO}.

\subsection{Elliptic Associators}
\begin{definition} (\cite{Humbert})
Let $\hat{f}(A,B)$ be the completed free Lie algebra generated by $A$ and $B$. Let $\phi\in\exp(\hat{f}(A,B))$ be a Drinfel'd associator. A pair $X(A,B), Y(A,B)\in\exp(\hat{f}(A,B))$ is called \textbf{an elliptic associator} with respect to $\phi$ if it satisfies the following identity in $U\hat{\textbf{t}}_{1,2}$:
\begin{equation}
\label{associator_identity1}
Y(x_1,y_1)X(x_1,y_1)Y^{-1}(x_1,y_1)X^{-1}(x_1,y_1)=\exp(t_{12})
\end{equation}
and the following $3$ identities in $U\hat{\textbf{t}}_{1,3}$:
\begin{multline}
\label{associator_identity2}
X(x_1+x_2,y_1+y_2)=\\
\phi(t_{12},t_{23})^{-1}X(x_1,y_1)\phi(t_{12},t_{23})\exp(t_{12}/2)\cdot\\
\cdot\phi(t_{12},t_{13})^{-1}X(x_2,y_2)\phi(t_{12},t_{13})\exp(t_{12}/2)
\end{multline}
\begin{multline}
\label{associator_identity3}
Y(x_1+x_2,y_1+y_2)=\\
\phi(t_{12},t_{23})^{-1}Y(x_1,y_1)\phi(t_{12},t_{23})\exp(-t_{12}/2)\cdot\\
\cdot\phi(t_{12},t_{13})^{-1}Y(x_2,y_2)\phi(t_{12},t_{13})\exp(-t_{12}/2)
\end{multline}
\begin{multline}
\label{associator_identity4}
\phi(t_{12},t_{23})^{-1}Y(x_1,y_1)\phi(t_{12},t_{23})\exp(t_{12}/2)\phi(t_{12},t_{13})^{-1}\cdot\\
\cdot X(x_2,y_2)\phi(t_{12},t_{13})\exp(t_{12}/2)=\exp(t_{12}/2)\phi(t_{12},t_{13})^{-1}\cdot\\ \cdot X(x_2,y_2)\phi(t_{12},t_{13})\exp(-t_{12}/2)\phi(t_{12},t_{23})^{-1}Y(x_1,y_1)\phi(t_{12},t_{23})
\end{multline}

\end{definition}

If $X(A,B),Y(A,B)$ is an elliptic associator, then it is easy to see that $\Delta^{++}_{\omega_1,\omega_2}(u_2(X(x_1,y_1)))$ and $\Delta^{++}_{\omega_1,\omega_2}(u_2(Y(x_1,y_1)))$ define an elliptic structure relative to $\mathcal{A}^\partial\rightarrow\mathcal{A}_1^{<p}$.

We will now describe a specific elliptic associator, which was defined in \cite{Calaque} and \cite{Enriquez}.
\paragraph*{Notations:} In the completed Lie algebra $\hat{f}(A,B)$, denote:
$$T:=[B,A]$$
$$\tilde{A}:=\frac{\text{ad} B}{e^{\text{ad}B}-1}(A)=A-\frac{1}{2}[B,A]+\frac{1}{12}[B,[B,A]]+\cdots$$
\paragraph*{Note:} The coefficients which appear in this expansion are the Bernoulli numbers, which are denoted by $B_i$.
\begin{definition}
Given a Drinfel'd associator $\phi$, let $e(\phi)=(X_\phi,Y_\phi)$ be defined by:
$$X_\phi(A,B)=\phi(\tilde{A},T)\exp(\tilde{A})\phi(\tilde{A},T)^{-1}$$
$$Y_\phi(A,B)=\exp(T/2)\phi(-\tilde{A}-T,T)\exp(B)\phi(\tilde{A},T)^{-1}$$
\end{definition}
\begin{theorem}
\label{theorem_elliptic}
$e(\phi)$ is an elliptic associator relative to $\phi$.
\end{theorem}
A proof of this theorem is given in \cite{Enriquez} (Proposition 3.8, see also \cite{Calaque} Proposition 5.3). Our goal in this section is to give a different proof of this theorem, based on the following theorem, which relates $e(\phi)$ to the elliptic structure relative to $\mathcal{A}^\partial\rightarrow\mathcal{A}_1^{<p}$ defined in section \ref{elliptic_LMO} via the LMO functor. Note that $X_\phi(x_1,y_1)$ and $Y_\phi(x_1,y_1)$ both belong to $RU\hat{\textbf{t}}_{1,2}$.
\begin{theorem}
\label{theorem_related}
$$\tilde{u}_2(X_\phi(x_1,y_1))=LMO^<(X_{+,+})\in\mathcal{A}_1^{<p}(\uparrow^2)/H_2$$
$$\tilde{u}_2(Y_\phi(x_1,y_1))=LMO^<(Y_{+,+})\in\mathcal{A}_1^{<p}(\uparrow^2)/H_2$$
\end{theorem}
Theorem \ref{theorem_related} would not hold if we replace $\tilde{u}_2$ by $u_2$ (see Remark \ref{remark_not_equal} below). Therefore, the elliptic structure relative to $\mathcal{A}^\partial\rightarrow\mathcal{A}_1^{<p}$ which is induced by the LMO functor is not the same elliptic structure which is induced by $e(\phi)$. This theorem says that among all the elliptic structures that come from elliptic associators, the elliptic structure induced by $e(\phi)$ is, in a sense, the ``closest" to the one induced by the LMO functor.

\begin{remark}
Using the same techniques one might be able to define associators for higher genus, by pulling back the value of the LMO functor on the right tangles.
\end{remark}
The rest of this section is dedicated to proving theorems \ref{theorem_elliptic} and \ref{theorem_related}.

\subsection{Proof of Theorem \ref{theorem_related}}
We begin with several lemmas. Here and in the following proofs we denote by $t_{ij}\in\mathcal{A}(\uparrow^n,S)$ the diagram with a single edge connecting the $i$ strand and the $j$ strand.
\begin{lemma}
\label{lemma0}
Given a word $\omega$ of length $3$ and words $\omega_1$, $\omega_2$ and $\omega_3$ in $\{+,-\}$, $\Delta^\omega_{\omega_1,\omega_2,\omega_3}\phi(t_{12},t_{23})=1$ (i.e. the empty diagram) in $\mathcal{A}(\uparrow^{|\omega_1|+|\omega_2|+|\omega_3|})$ (as defined in section \ref{defA}).
\end{lemma}
\begin{proof}
It is enough to show that $\phi(t_{12},t_{23})=1$ in $\mathcal{A}(\uparrow^3)$. Indeed, using an $I$ relation we get: $$\phi(t_{12},t_{23})=\phi\left(\,\raisebox{-0.3cm}{\includegraphics[height=0.7cm]{section4_t_12.pdf}}\, ,\,\raisebox{-0.3cm}{\includegraphics[height=0.7cm]{section4_t_23.pdf}}\,\right)=\phi\left(-\,\raisebox{-0.3cm}{\includegraphics[height=0.7cm]{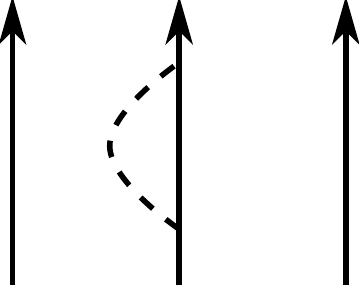}}-\,\raisebox{-0.3cm}{\,\includegraphics[height=0.7cm]{section4_t_23.pdf}}\,,\,\raisebox{-0.3cm}{\includegraphics[height=0.7cm]{section4_t_23.pdf}}\,\right)=1$$
because both $\,\raisebox{-0.3cm}{\includegraphics[height=0.7cm]{section5_t_22.pdf}}\,$ and $\,\raisebox{-0.3cm}{\includegraphics[height=0.7cm]{section4_t_23.pdf}}\,$ commute with $\,\raisebox{-0.3cm}{\includegraphics[height=0.7cm]{section4_t_23.pdf}}\,$ (using the STU relation).
\end{proof}

\begin{lemma}
\label{lemma_chi}
Assume that in the pattern $\uparrow^{n+1}$ the left strand is labeled $x$, and let $a\in\mathcal{A}(\uparrow^{n+1},\{y\})$. Recall the map $j:\mathcal{A}_1^y\rightarrow \mathcal{A}_1$ defined in section \ref{defj}, and the map $k:\mathcal{A}_1^y \longrightarrow\mathcal{A}_1^< $ defined in section \ref{defk}. Then we have $\chi^{-1}_x\left(\exp\left(y\raisebox{-0.3cm}{\includegraphics[height=0.7cm]{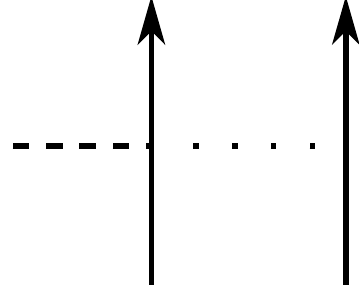}}\right)\cdot a\right)\in Im(j)\subset ^{ts}\mathcal{A}_1(\uparrow^n)$, and $k\circ j^{-1}\circ\chi^{-1}_x\left(\exp\left(y\raisebox{-0.3cm}{\includegraphics[height=0.7cm]{section5_y_1}}\right)\cdot a\right)$ is the element obtained from $a$ by replacing each vertex on $x$ by the following sum:
$$\mathfig{4}{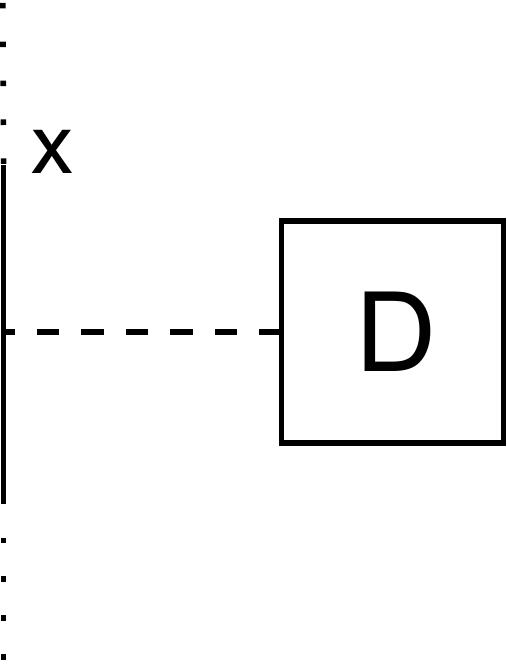}\longmapsto\sum_{i=0}^{\infty}B_i\mathfig{4}{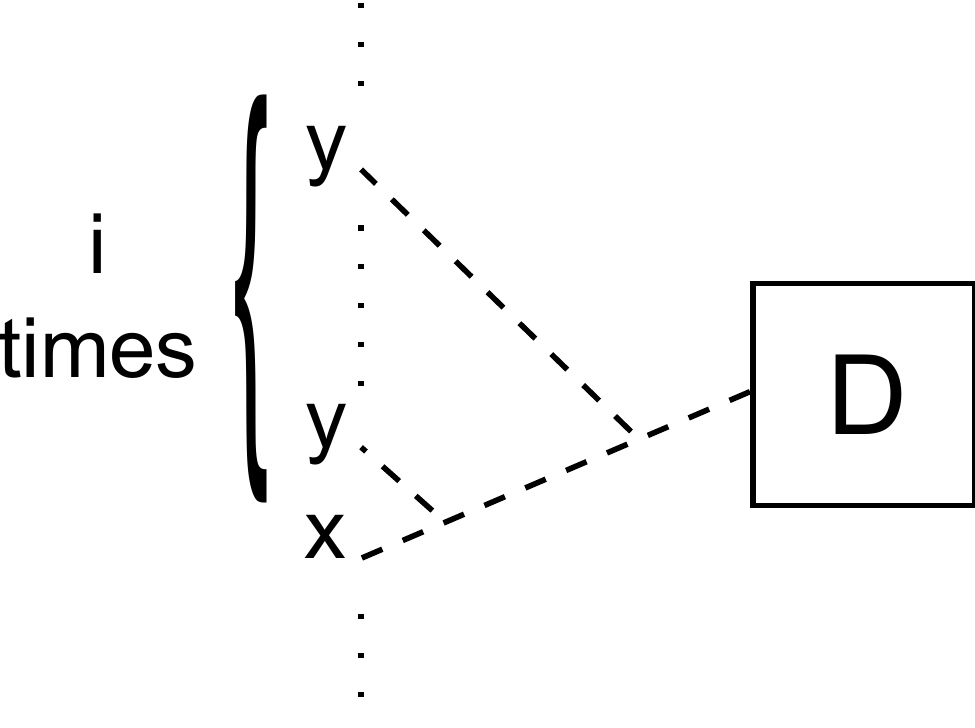}$$
where $B_i$ are the Bernoulli numbers.
\end{lemma}

\begin{proof}
Recall the element $\lambda(a,b;r)\in\mathcal{A}(\emptyset,\{a,b,r\})$ defined in \cite{Cheptea} by:
$$\lambda(a,b;r)=\chi^{-1}_r\left(\mathfig{3}{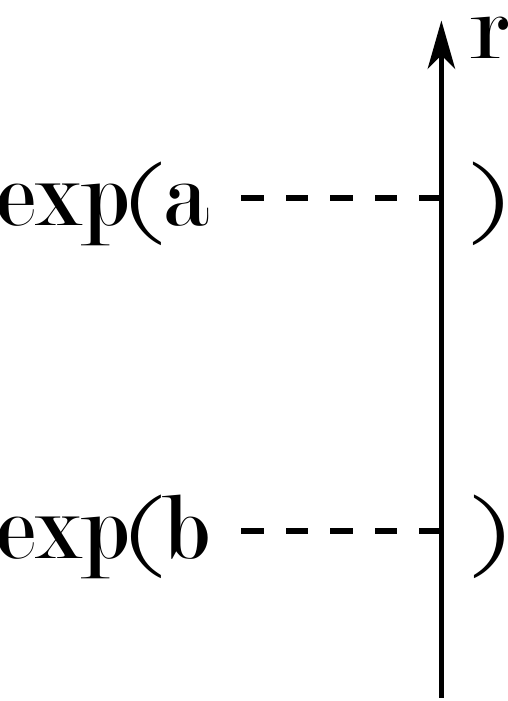}\right)=$$
$$=\boldsymbol{\exp_\coprod}\left(\begin{array}{c} b\\ \includegraphics[height=1.3cm]{section2_strut.pdf}\\r\end{array} +\sum_{\begin{array}{c} n\ge0\\u_1,...,u_n\in\{a,b\}\end{array}} r(u_1,...,u_n)\mathfig{2.5}{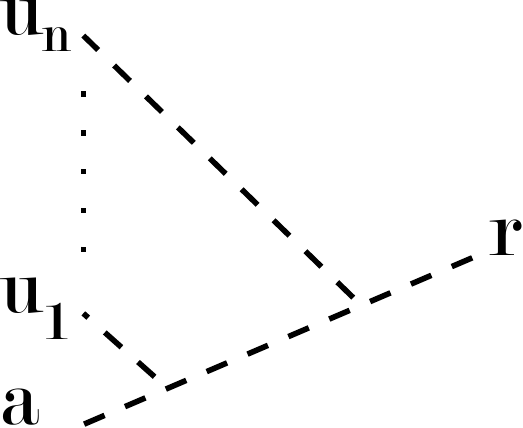}\right)$$
where $r(u_1,...,u_n)$ are some coefficients determined by the Baker-Campbell-Hausdorff formula. In particular we have $r(\underbrace{b,...,b}_{\text{n times}})=B_n$ (follows from \cite{Dynkin}, formula (12)).

$\lambda(a,b;r)$ has the following property: If $D\in\mathbb{D}(P\cup\{\uparrow_r\},S)$ is of the following form:
$$D=\mathfig{4}{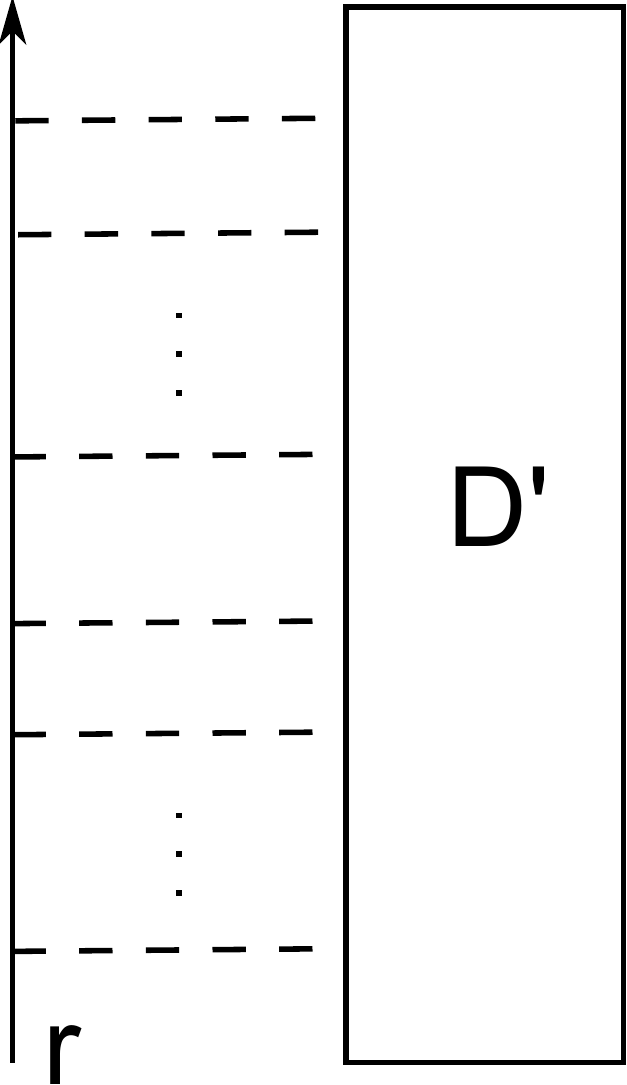}$$
then we have:
$$\chi^{-1}_r(D)=\left<\chi^{-1}_{r_1,r_1}\left(\mathfig{4}{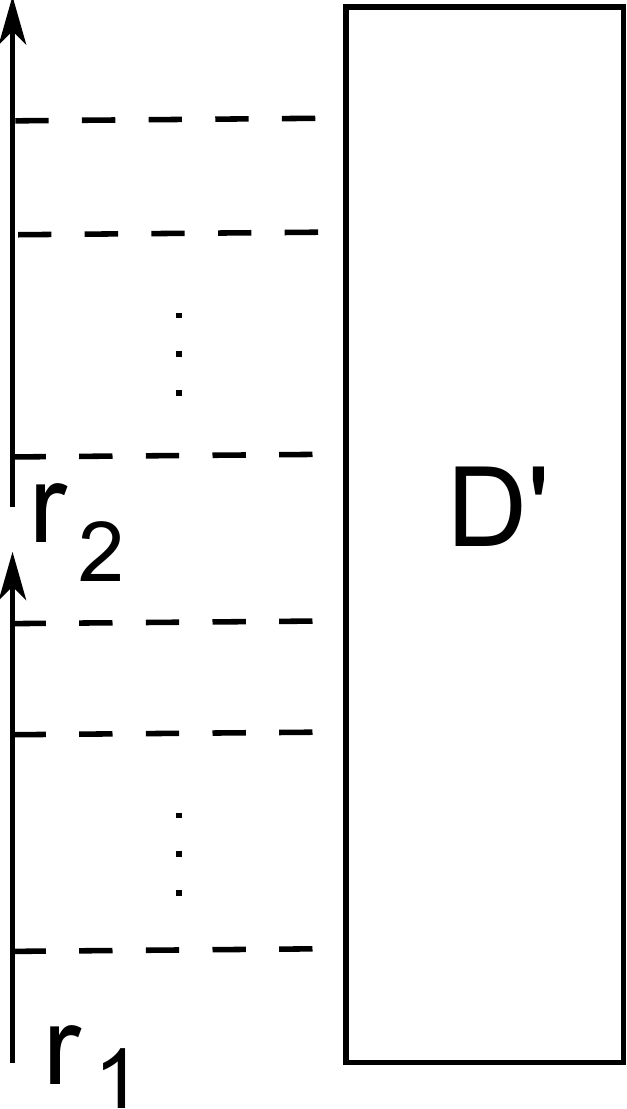}\right),\lambda(r_1,r_2;r)\right>\raisebox{-0.5cm}{$r_1,r_2$}$$
where the form $<D_1,D_2>_{r_1,r_2}$ is defined on diagrams $D_1,D_2$ to be the sum of all ways to glue all vertices labeled by $r_1$ in $D_1$ to all vertices labeled $r_1$ in $D_2$ and all vertices labeled by $r_2$ in $D_1$ to all vertices labeled $r_2$ in $D_2$.

Now, it is enough to prove the theorem for $a$ which is a diagram. Suppose we have $m\ge1$ vertices on $x$ (the case $m=0$ is trivial), and denote $x_{(m)}=x$.

$$\chi^{-1}_{x_{(m)}}\left(\mathfig{3.5}{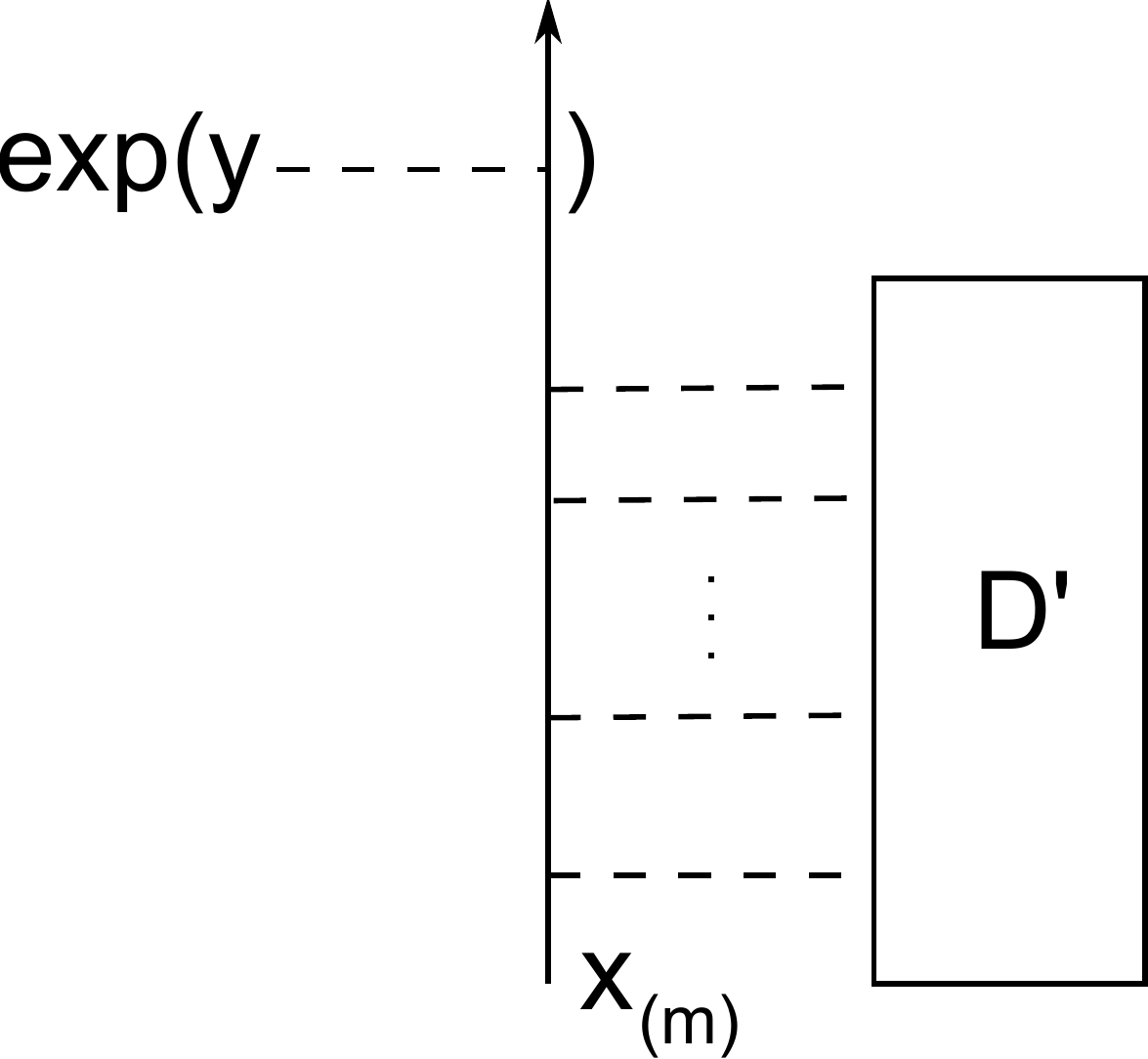}\right)=$$
$$=\left<\chi^{-1}_{x_m,x_{(m-1)}}\left(\mathfig{4}{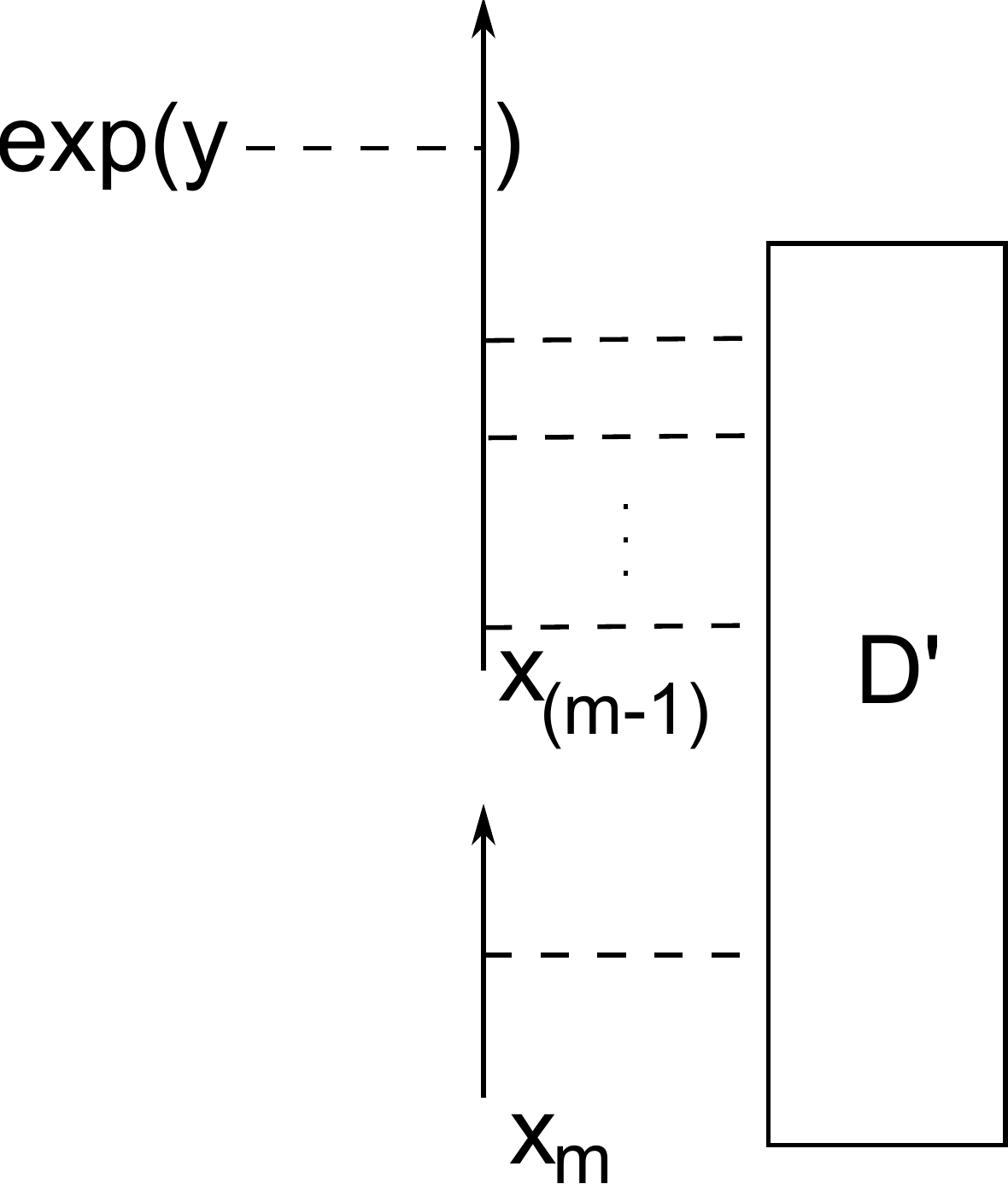}\right),\lambda(x_m,x_{(m-1)};x_{(m)})\right>\raisebox{-0.5cm}{$x_m,x_{(m-1)}$}=$$
$$=\psi_{x_{(m-1)},\tilde{x}_{(m-1)}}\left(\sum_{k_m\ge 0}B_{k_m}\cdot \chi^{-1}_{x_{(m-1)}}\left(
\exp\left(\begin{array}{c} \tilde{x}_{(m-1)}\\ \includegraphics[height=0.7cm]{section2_strut.pdf}\\x_{(m)}\end{array}\right)\cdot
\mathfig{5}{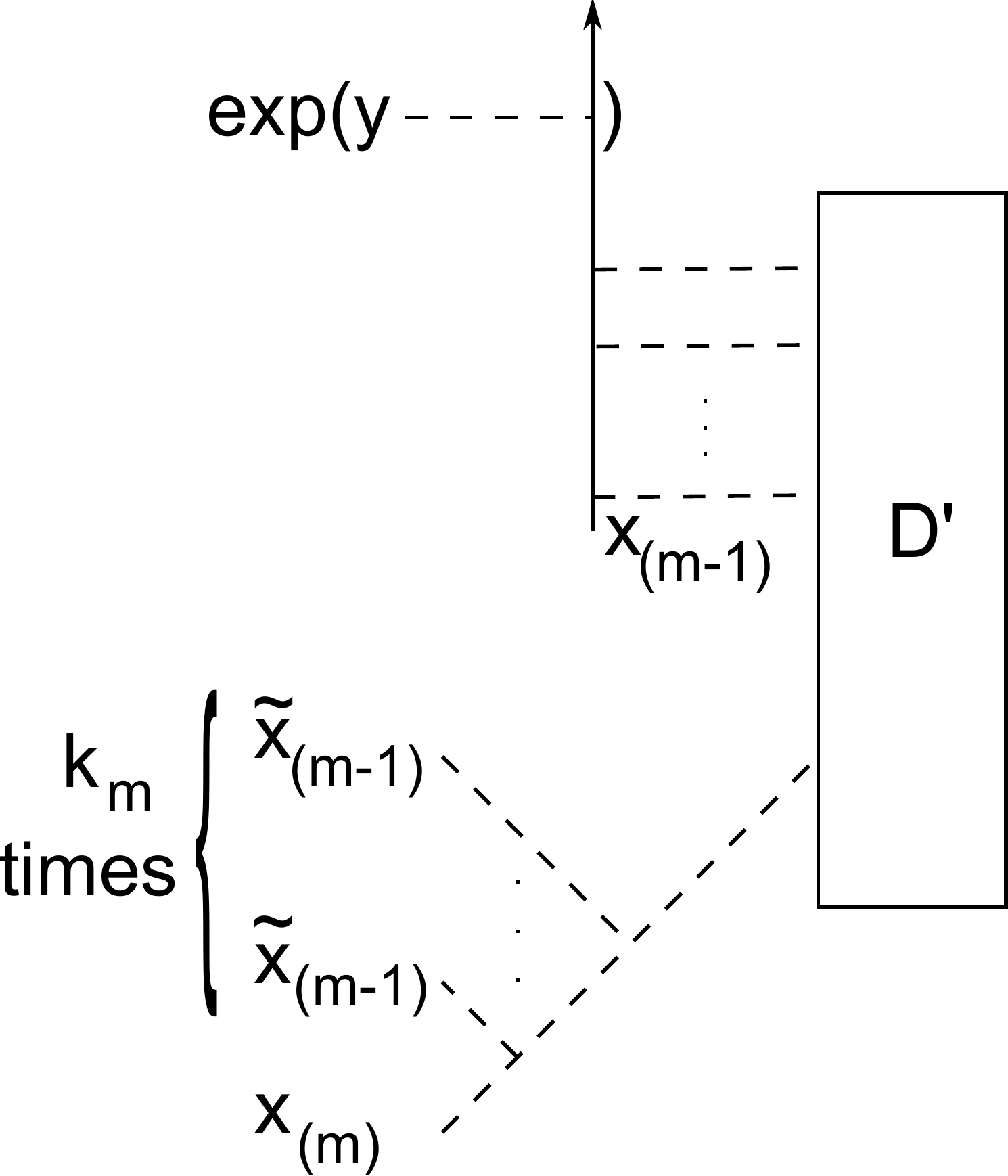}\right)\right)$$
where $\psi_{z,\tilde{z}}(D)$ for a diagram $D$ which contains the labels $z$ and $\tilde{z}$ is defined to be the sum of all ways to glue all vertices labeled $z$ to all vertices labeled $\tilde{z}$.

Repeating this process recursively we get:
$$\chi^{-1}_{x_{(m)}}\left(\mathfig{3.5}{section5_proof_lemma3}\right)=$$
$$=\psi_{x_{(m-1)},\tilde{x}_{(m-1)}}\circ\cdots\circ\psi_{x_{(0)},\tilde{x}_{(0)}}(\sum_{k_1,...,k_m\ge 0}B_{k_1}\cdots B_{k_m}$$
$$\exp\left(\begin{array}{c} \tilde{x}_{(m-1)}\\ \includegraphics[height=0.7cm]{section2_strut.pdf}\\x_{(m)}\end{array}\right)\cdots
\exp\left(\begin{array}{c} \tilde{x}_{(0)}\\ \includegraphics[height=0.7cm]{section2_strut.pdf}\\x_{(1)}\end{array}\right)\cdot
\exp\left(\begin{array}{c} y\\ \includegraphics[height=0.7cm]{section2_strut.pdf}\\x_{(0)}\end{array}\right)
\cdot
\mathfignum{5}{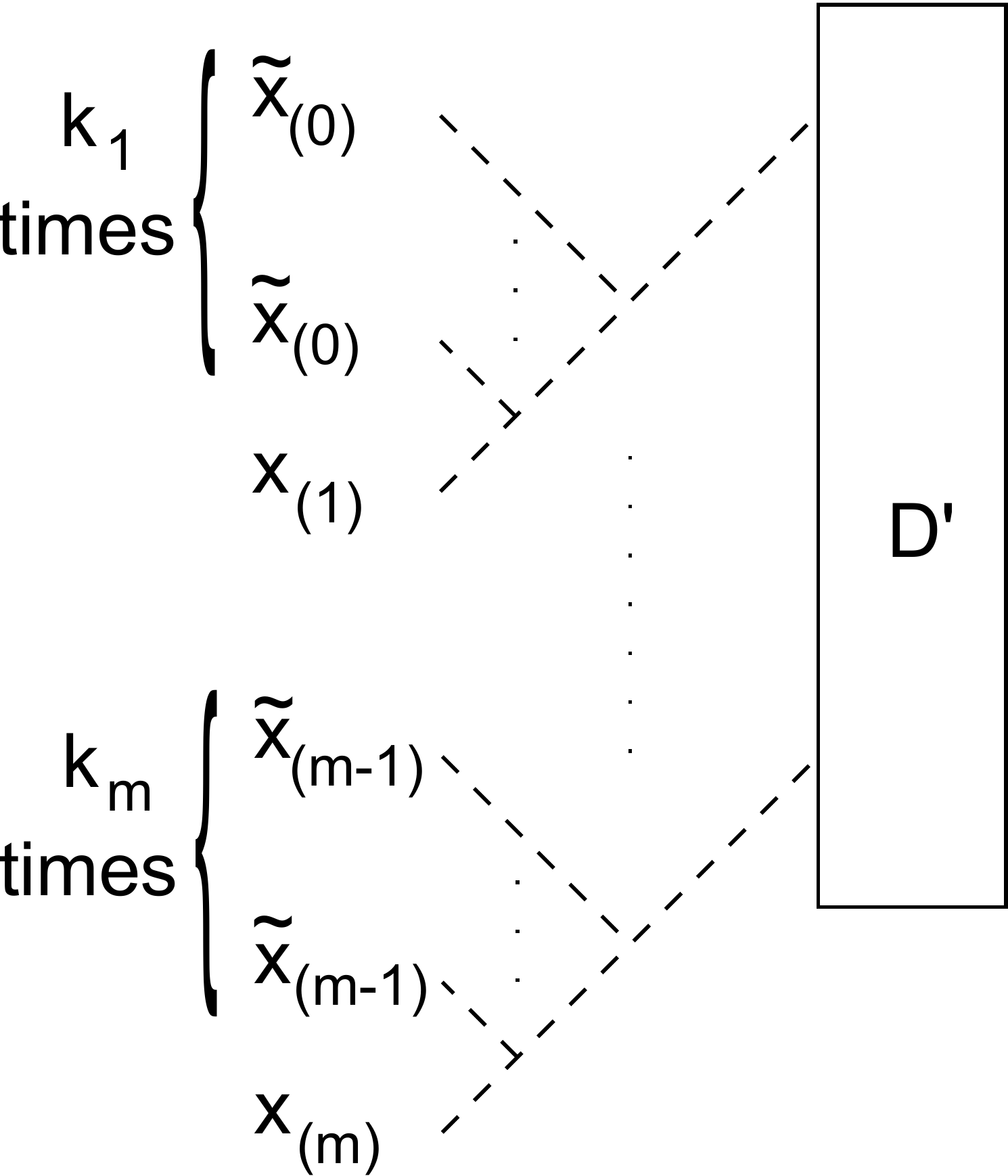}{*})$$

For a specific choice of $k_1,...,k_m$, we wish to describe the element we get by applying \linebreak $\psi_{x_{(m-1)},\tilde{x}_{(m-1)}}\circ\cdots\circ\psi_{x_{(0)},\tilde{x}_{(0)}}$ to the corresponding summand. A careful examination shows that the element we get is a product of $\exp\left(\begin{array}{c} y\\ \includegraphics[height=0.7cm]{section2_strut.pdf}\\x_{(m)}\end{array}\right)$ with the sum of all the diagrams which can be produced from \circlenum{*} by the following process:
\begin{enumerate}
\item
Change all the $\tilde{x}_{(0)}$ labels to $y$.
\item
For $i=1,...,m-1$:
\begin{itemize}
\item
Attach some of the $x_{(i)}$ labels to some of the $\tilde{x}_{(i)}$ labels.
\item
Change all the remaining $x_{(i)}$ labels to $x_{(i+1)}$.
\item
Change all the remaining $\tilde{x}_{(i)}$ labels to $y$
\end{itemize}
\end{enumerate}

This sum can be described more shortly as the sum of all ways to glue some of the $x_{(i)}$ labels to $\tilde{x}_{(j)}$ labels with $j\ge i$, and then change all the remaining $x_{(i)}$ labels to $x_{(m)}=x$ and all the remaining $\tilde{x}_{(i)}$ labels to $y$.

The result of the above calculation clearly belongs to the image of $j$. It is not difficult to see that $j^{-1}\circ\chi^{-1}_x\left(\exp\left(y\raisebox{-0.3cm}{\includegraphics[height=0.7cm]{section5_y_1}}\right)\cdot D\right)$ has a simpler presentation when mapped by $k$ to $\mathcal{A}_1^{<}(\uparrow^n)$, as:
$$\sum_{k_1,...,k_m\ge 0}B_{k_1}\cdots B_{k_m}\mathfig{5}{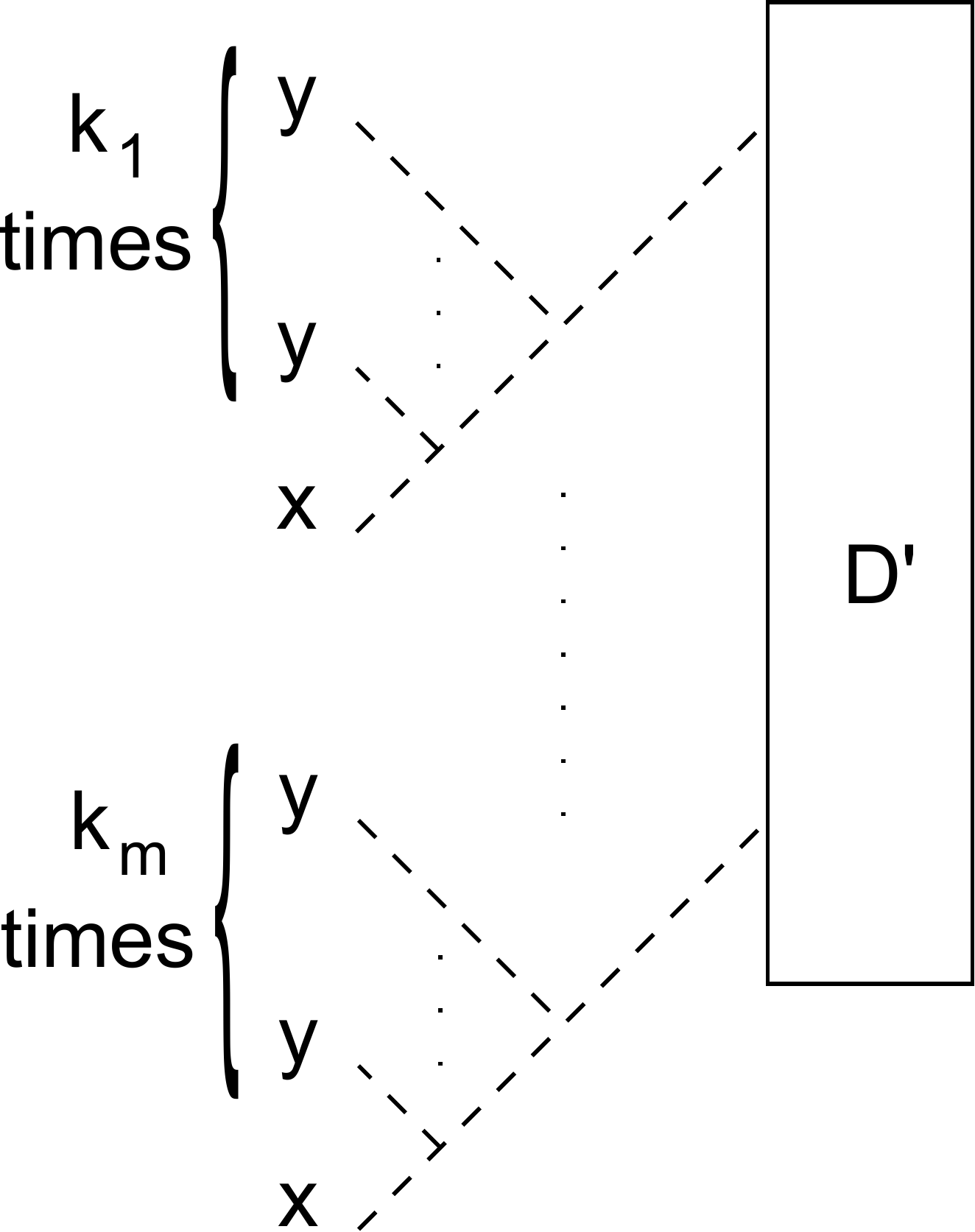}$$
Thus we have completed the proof of the lemma.

\end{proof}

\begin{lemma}
Recall the notation $\tilde{A}:=\frac{\text{ad} B}{e^{\text{ad}B}-1}(A)$ in $\hat{f}(A,B)$. Similarly we have in $U\hat{\textbf{t}}_{1,2}$: $\tilde{x}_1=\cfrac{\text{ad}\,y_1}{e^{\text{ad}\,y_1}-1}(x_1)$. Then we have:
$$\tilde{u}_2(\tilde{x}_1)\approx\sum_{i=0}^{\infty}B_i\mathfig{3}{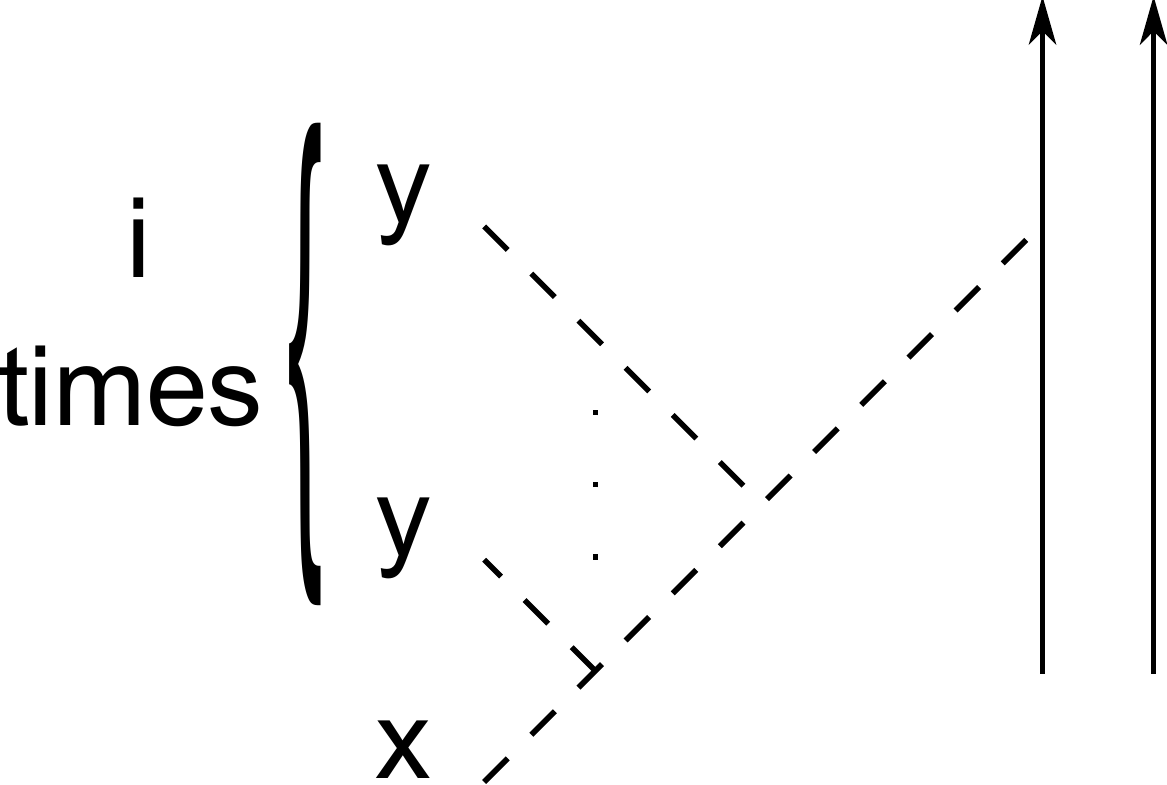}\mod H_2$$
\label{lemma_u}
\end{lemma}

\begin{proof}
$\tilde{x}_1=\sum_{i=0}^{\infty}B_i[\underbrace{y_1[\cdots[y_1}_{\text{i times}},x_1]\cdots]]$ by definition. We will prove by induction on $i\ge1$ the following identity, which will imply the lemma:
$$u_2([\underbrace{y_1[\cdots[y_1}_{\text{i times}},x_1]\cdots]])=\mathfig{2.5}{section5_lemma3_proof1}+\mathfig{2.5}{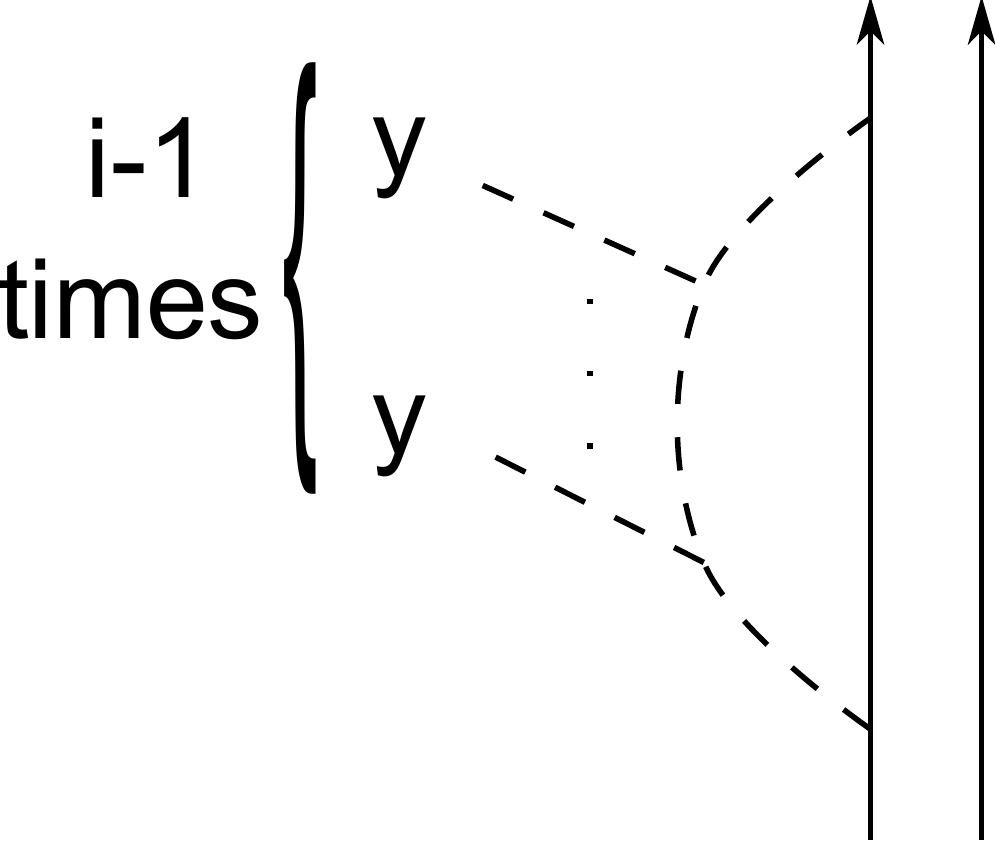}$$

For $i=1$ we have:
$$u_2([y_1,x_1])=\mathfig{2}{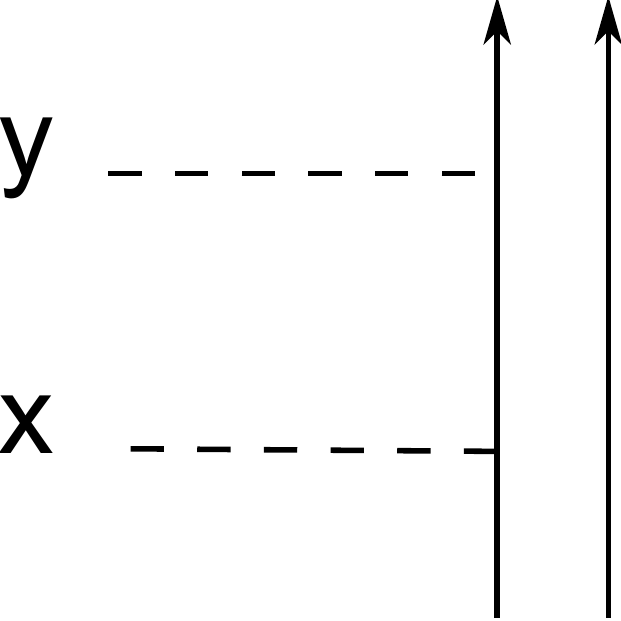}-\mathfig{2}{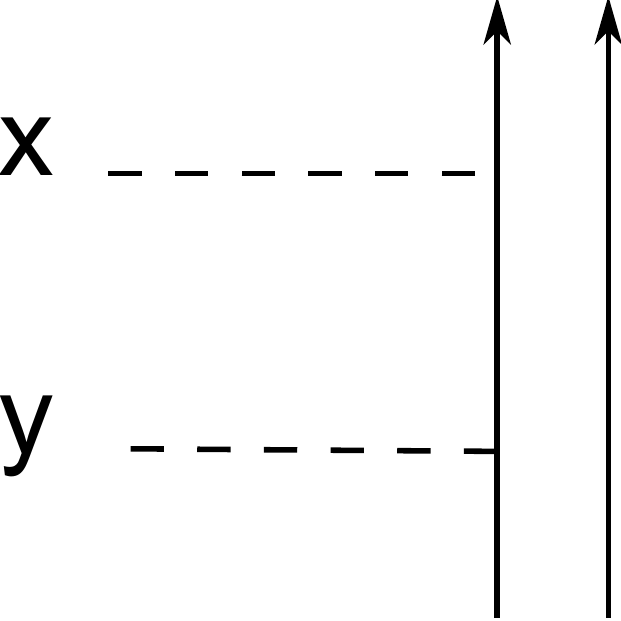}=\mathfig{2}{section5_lemma3_proof2_1}-$$
$$-\mathfig{2}{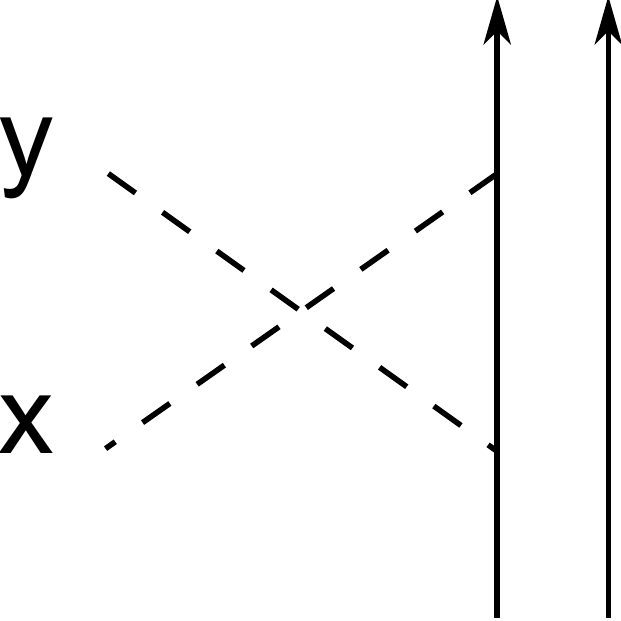}+\mathfig{2}{section5_lemma3_proof2_3}-\mathfig{2}{section5_lemma3_proof2_2}=$$
$$=\mathfig{2}{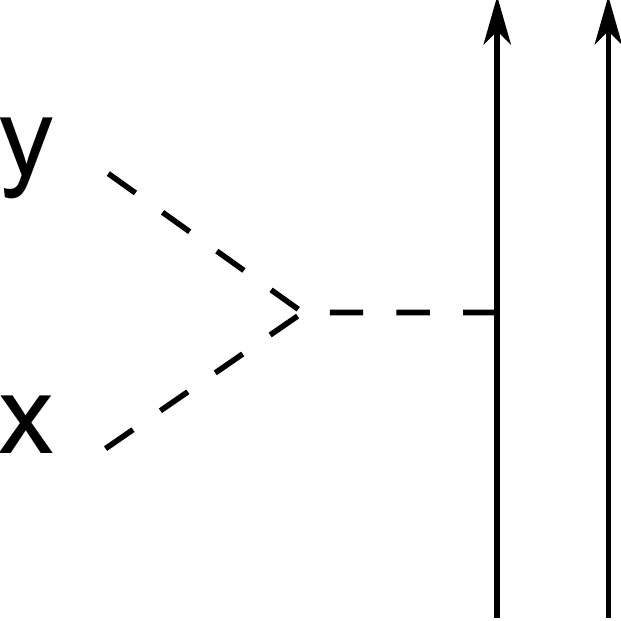}+\mathfig{2}{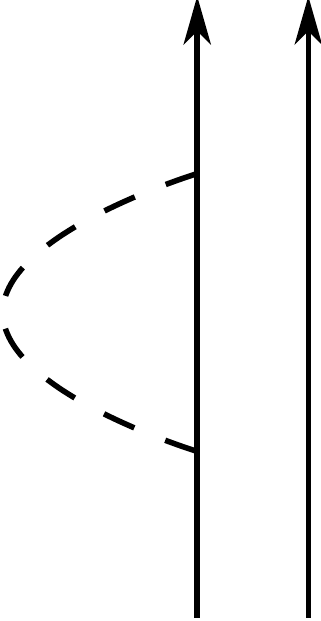}$$

Assume we proved the lemma for $i_0$. Then for $i_0+1$:

$$u_2([\underbrace{y_1[\cdots[y_1}_{\text{$i_0+1$ times}},x_1]\cdots]])=\mathfig{3}{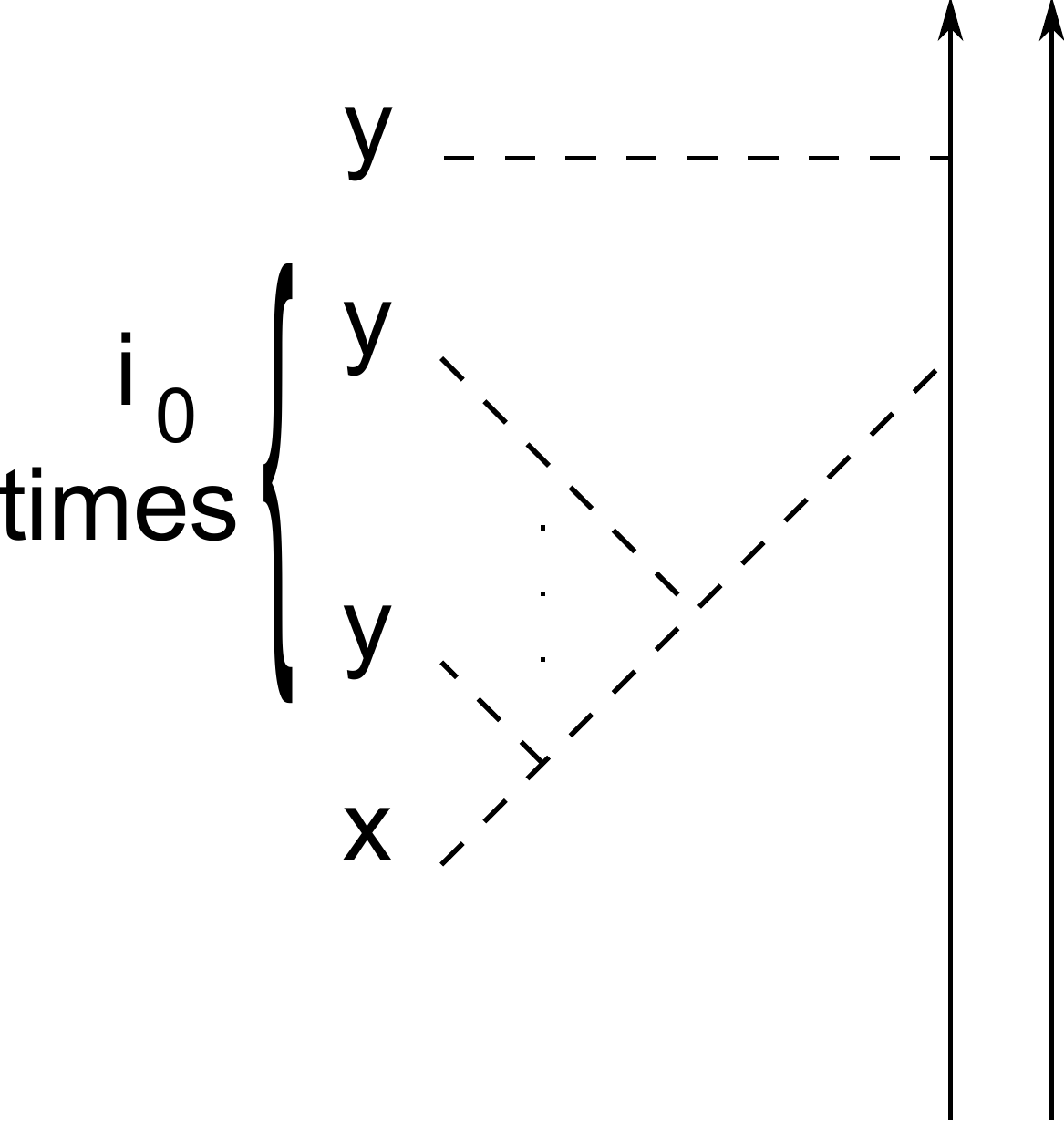}-\mathfig{3}{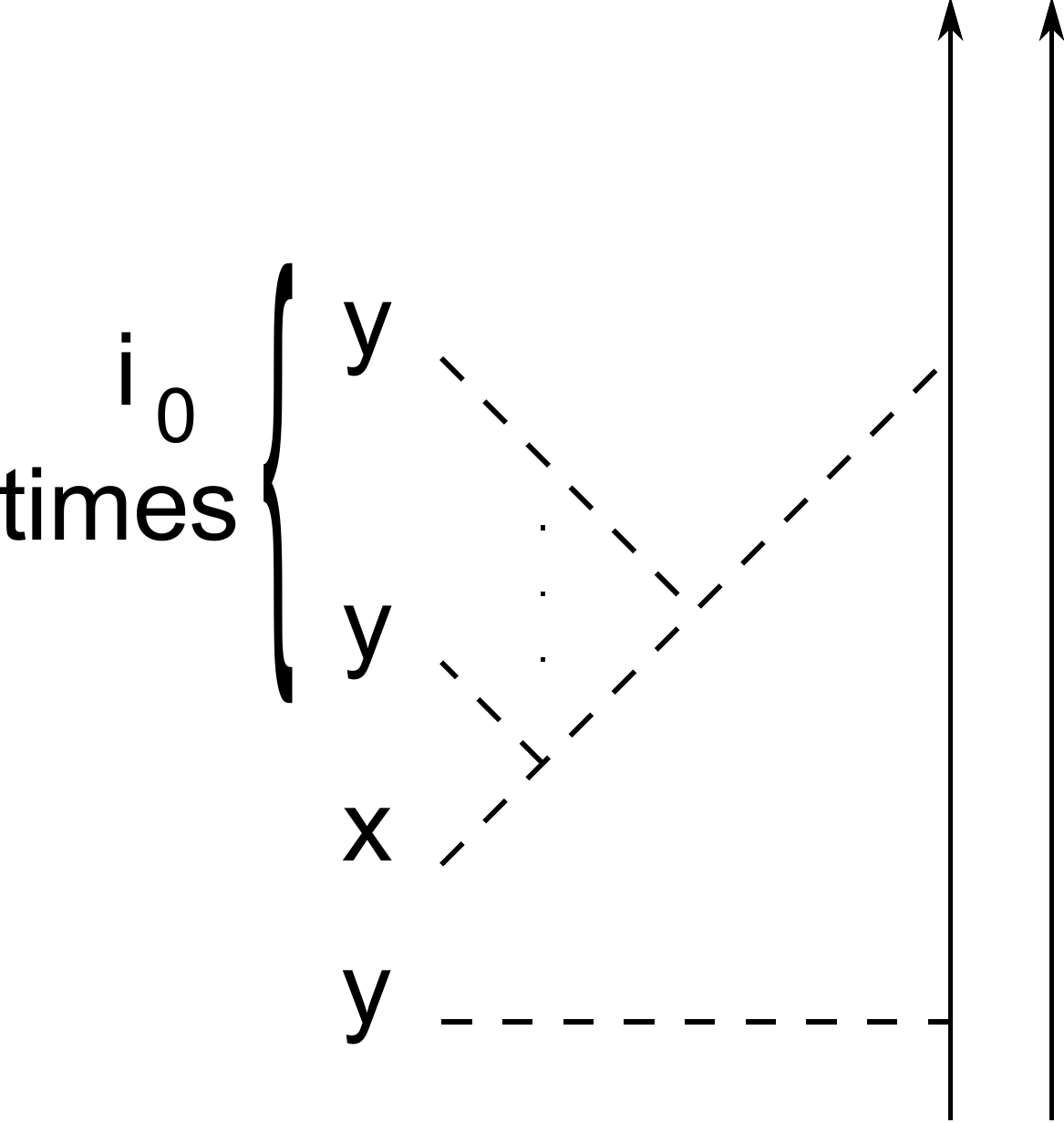}+$$
$$\overbrace{\mathfig{3}{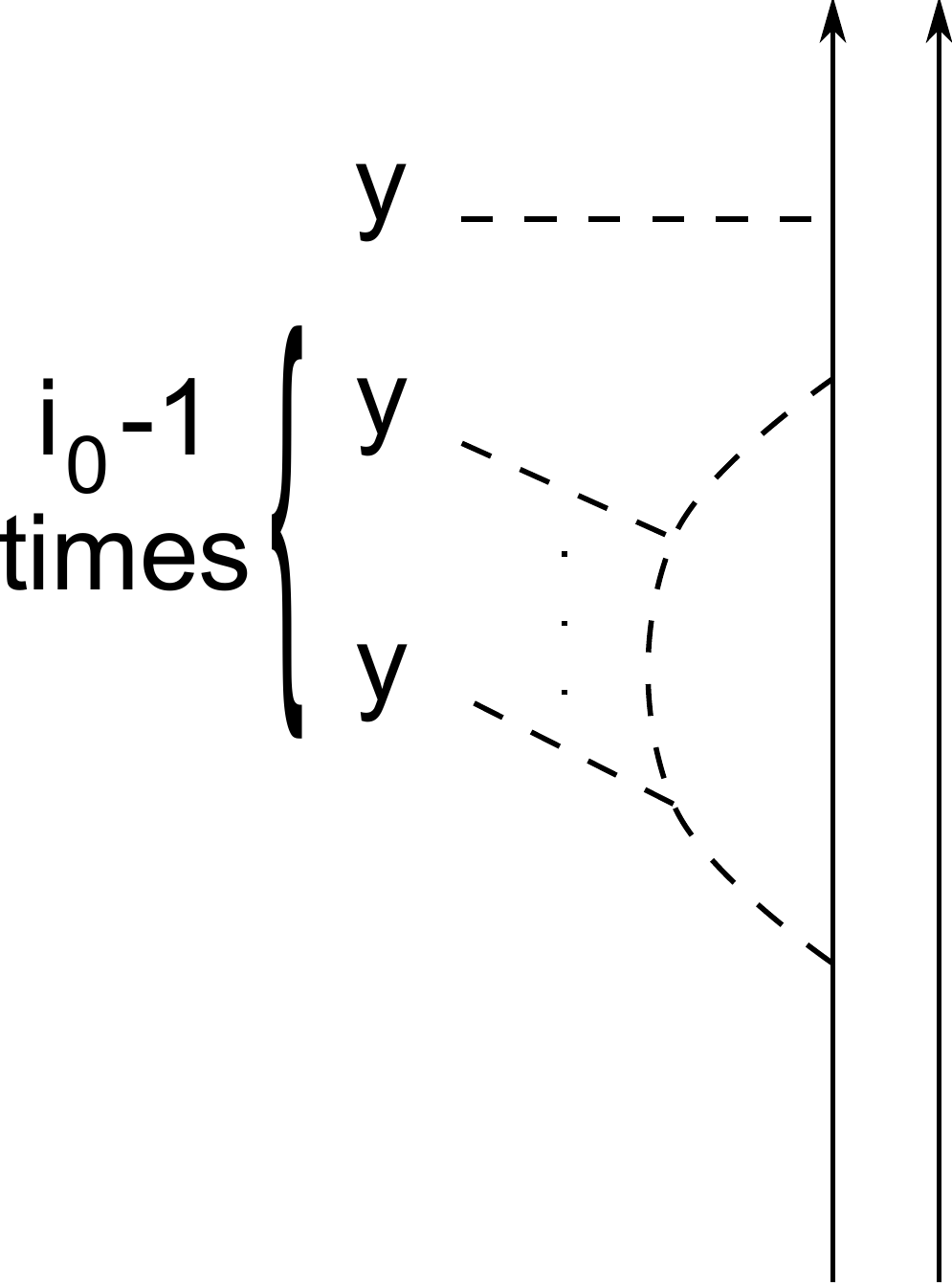}-\mathfig{3}{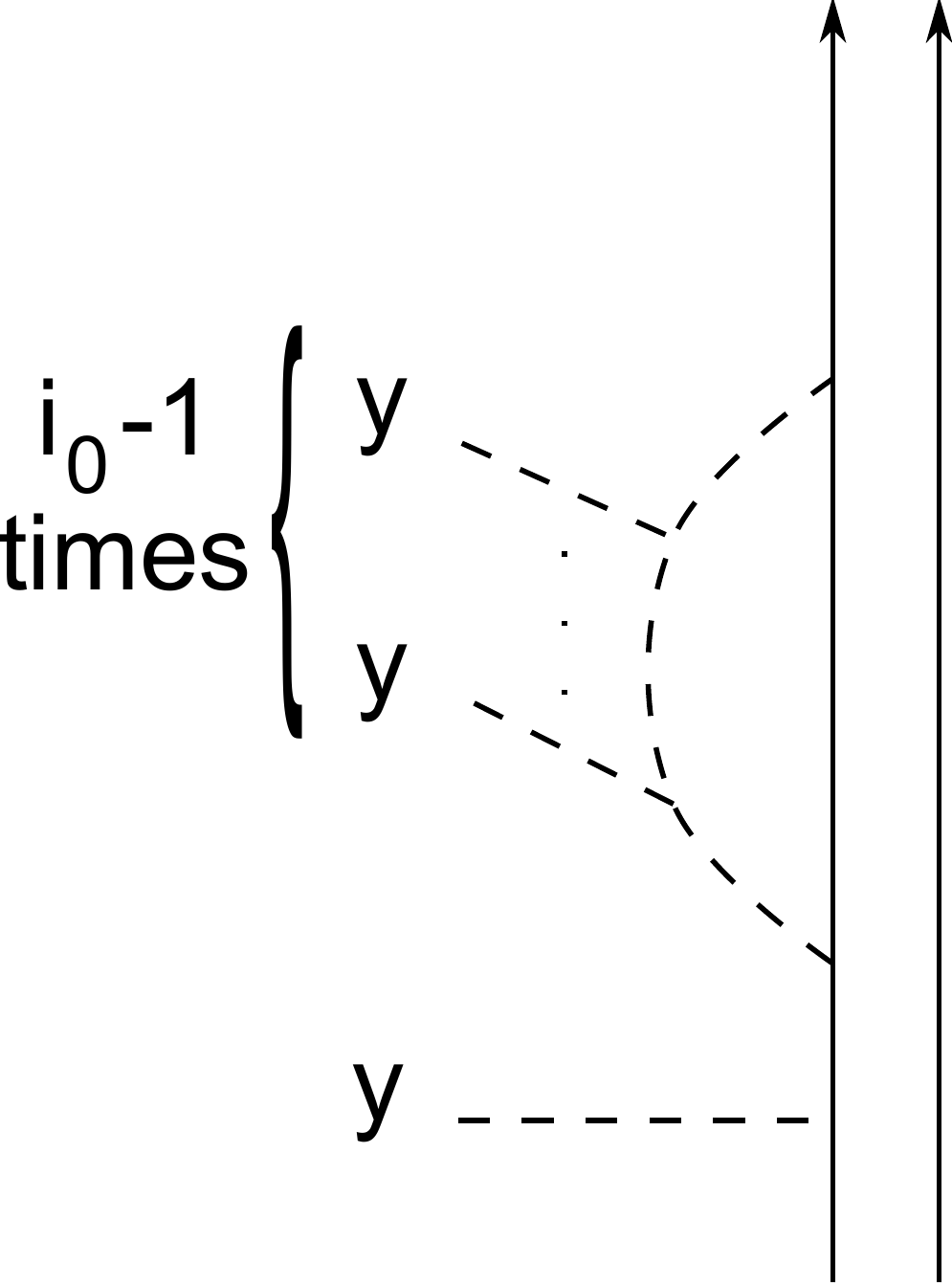}}^{=0}=\mathfig{3}{section5_lemma3_proof3_1}-$$
$$-\mathfig{3}{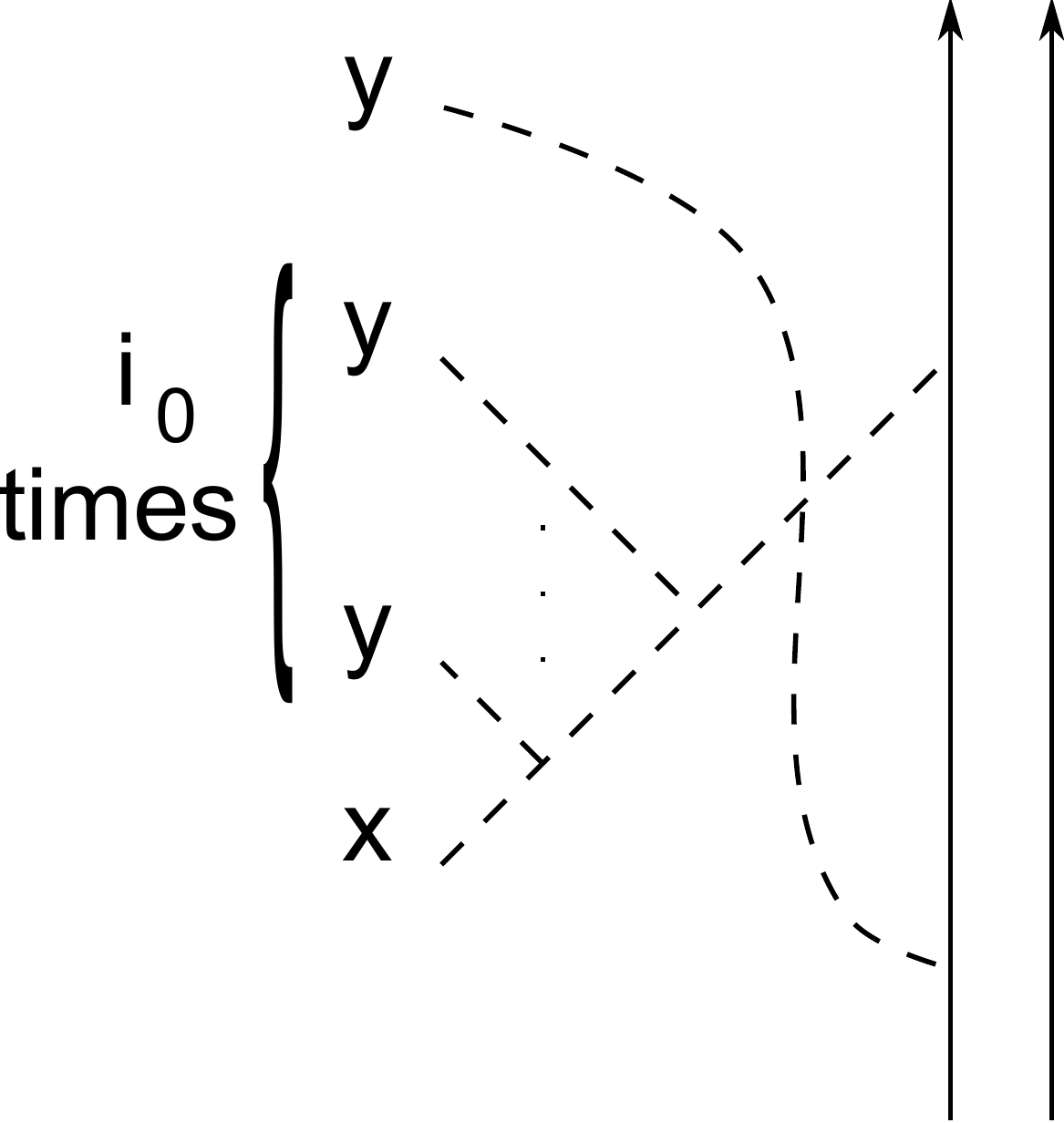}+\mathfig{3}{section5_lemma3_proof3_5}-\mathfig{3}{section5_lemma3_proof3_2}=$$
$$=\mathfig{2}{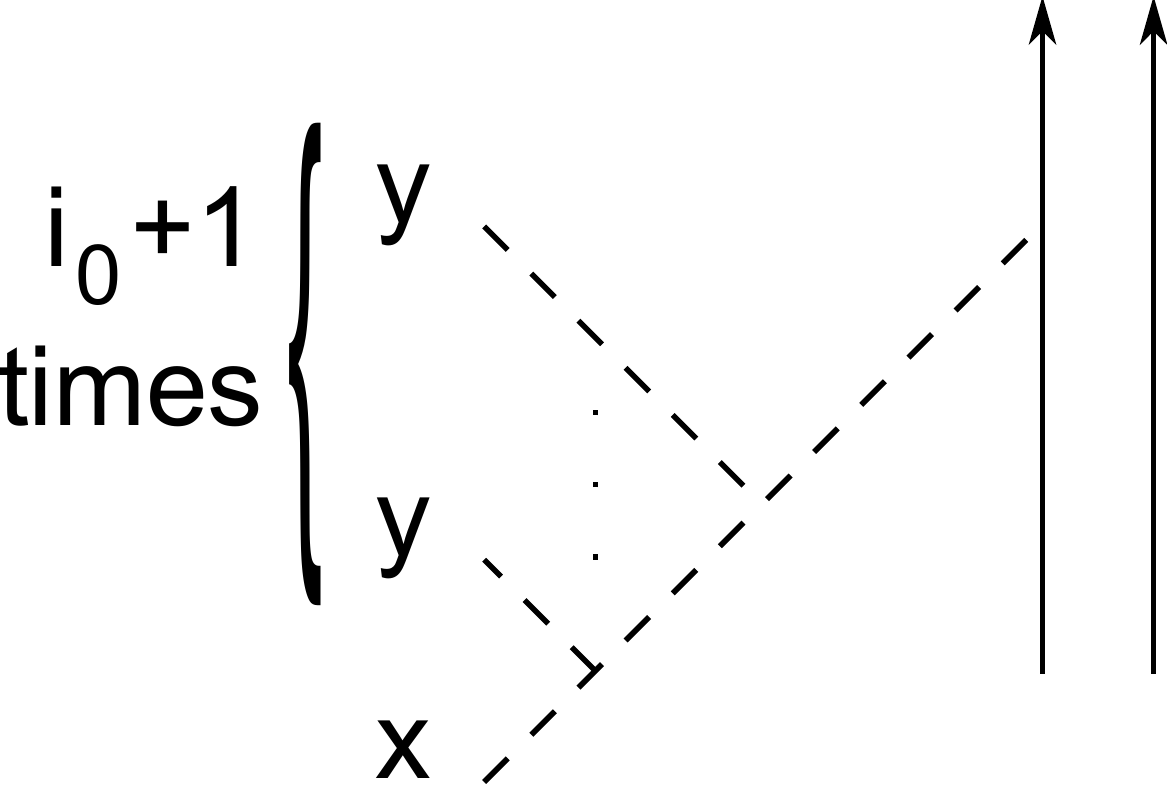}+\mathfig{2}{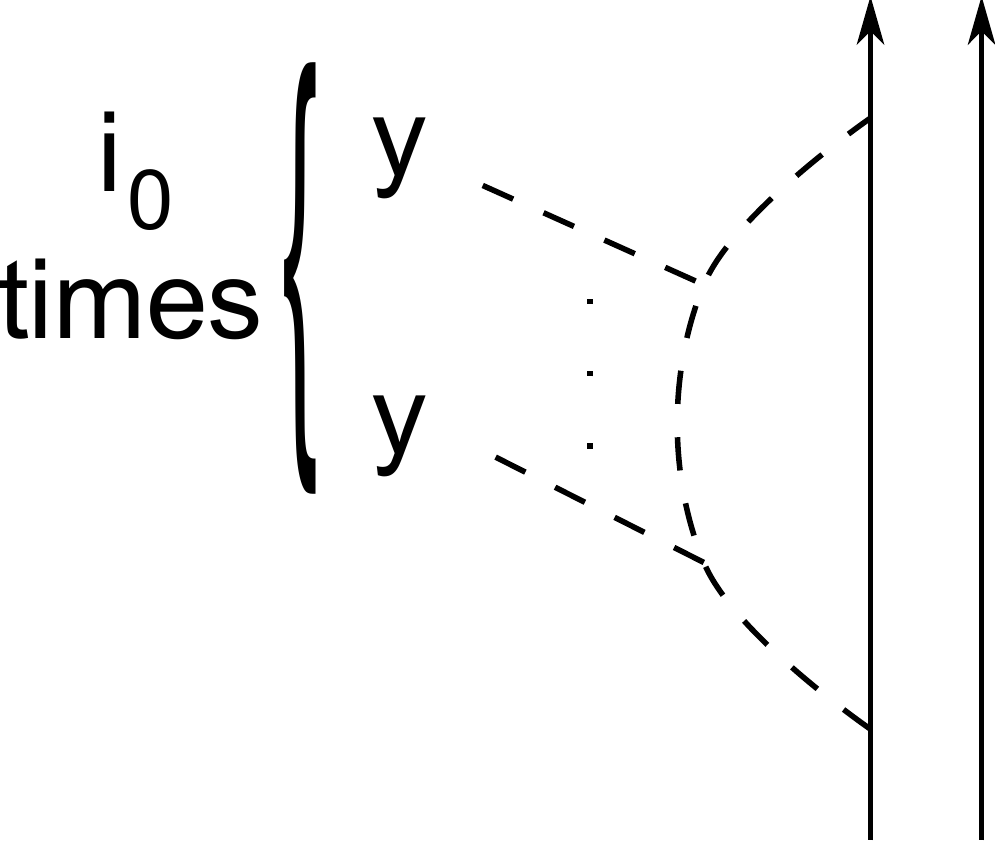}$$

\end{proof}
We are now ready to prove theorem \ref{theorem_related}.

\begin{proof}[Proof of theorem \ref{theorem_related}]
We represent the tangle $X_{+,+}$ as follows:
$$\mathfig{7}{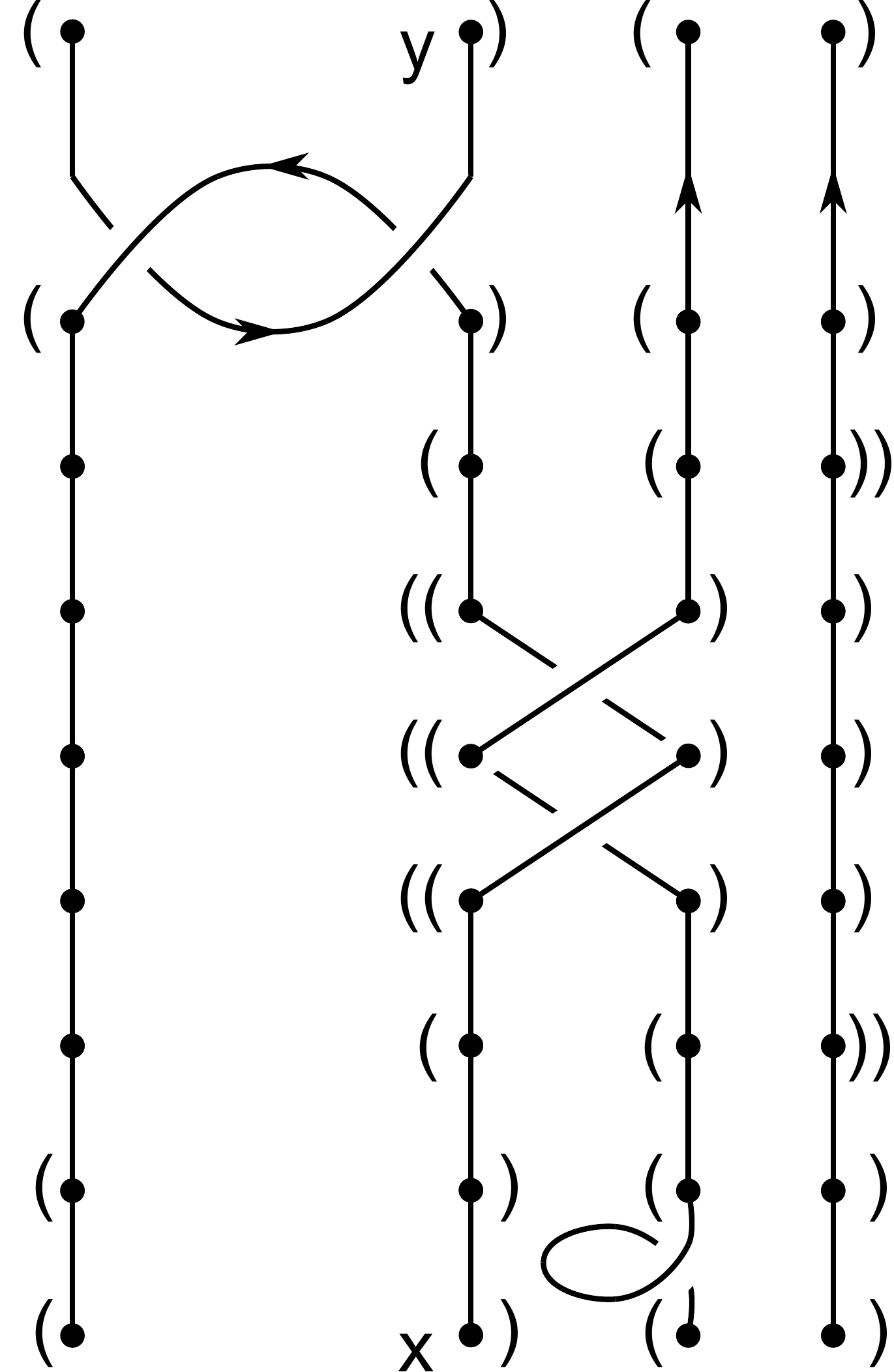}$$

We will calculate $\chi_x(LMO(X_{+,+}))\in\mathcal{A}(\uparrow_x\uparrow^2,\{y\})$. Inside $\mathcal{A}(\uparrow_x\uparrow^2,\{y\})$ we have the subspace $H_2$ which is spanned by all diagrams with a component which has more than one vertex on the second strand from the right, or has a loop. For an element $a\in H_2\subset \mathcal{A}(\uparrow_x\uparrow^2,\{y\})$ which is mapped by $\chi^{-1}_x$ to the image of $j:\mathcal{A}_1^{yp}\rightarrow \mathcal{A}_1^p$, we have $k\circ j^{-1}\circ \chi^{-1}_x(a)\in H_2\subset\mathcal{A}_1^{<p}(\uparrow^n)$.

$\chi_x(T_1(++))\circ\chi^{-1}_y\circ Z\left(\mathfig{2}{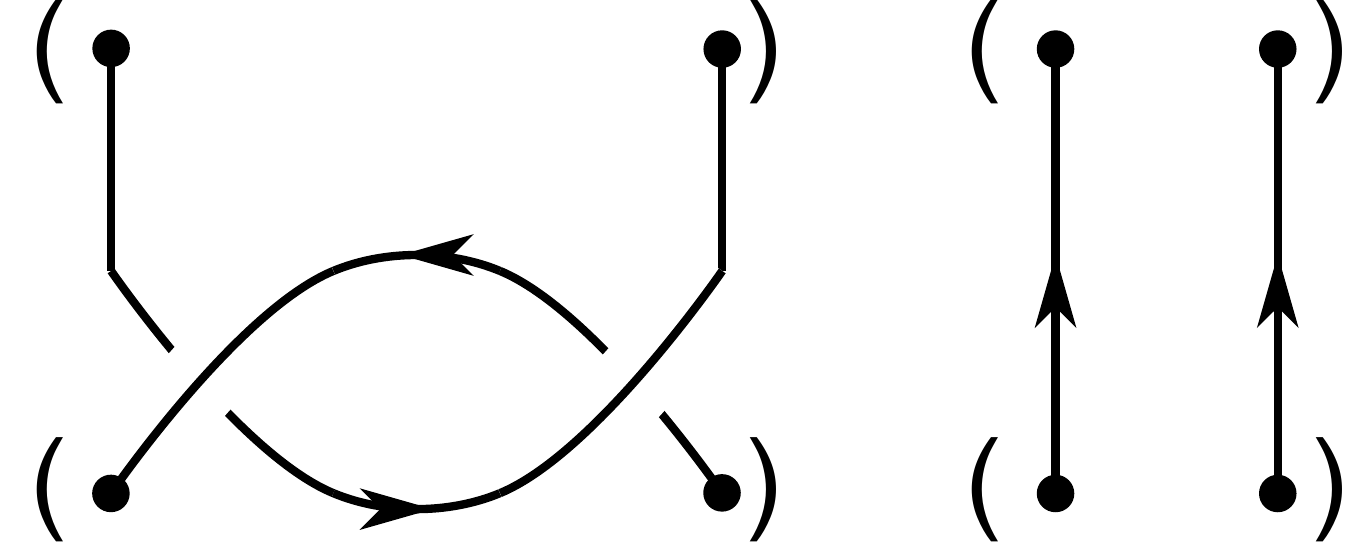}\right)=$

$\quad=\quad  \mathfigns{2.4}{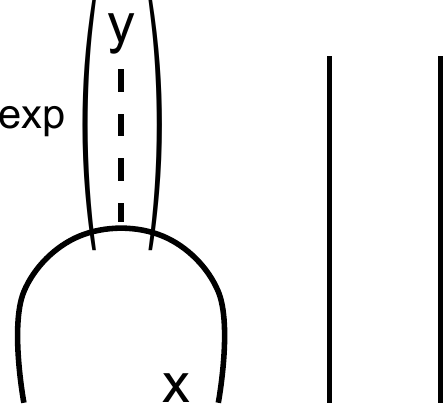}$

$Z\left(\mathfig{1}{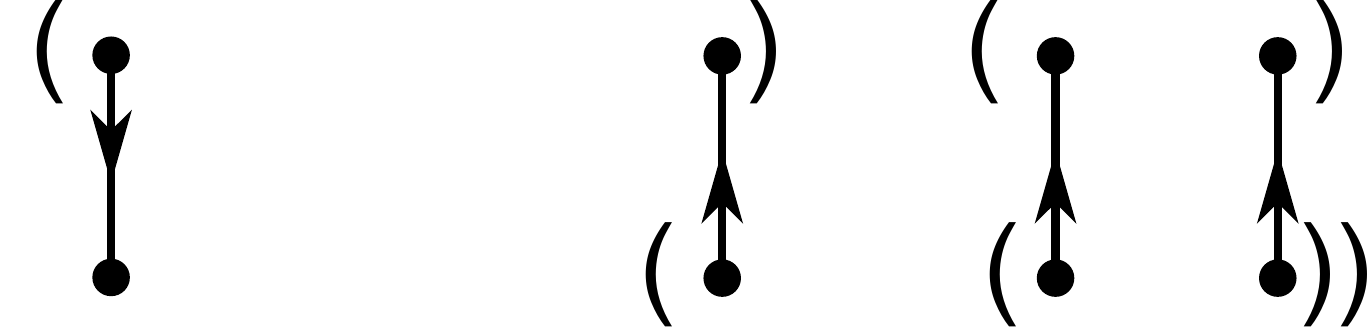}\right)=1\quad\text{by lemma \ref{lemma0}}$

$Z\left(\mathfig{1}{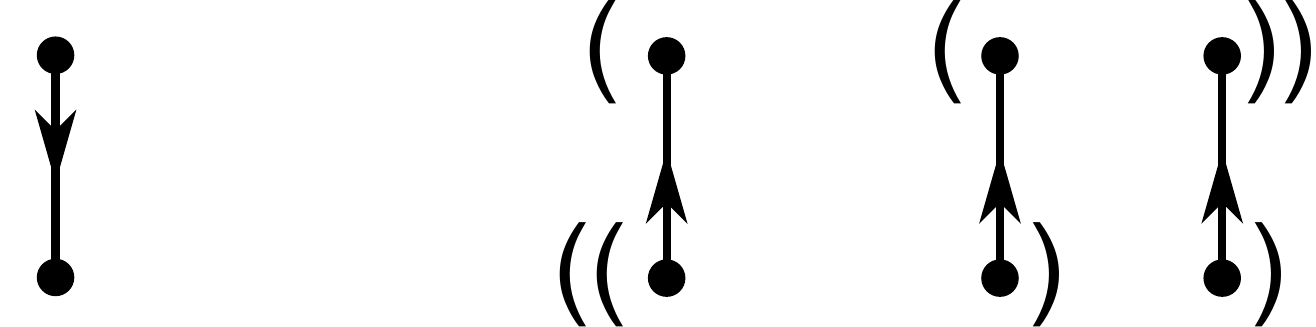}\right)=\phi(t_{23},t_{34})$

$Z\left(\mathfig{2}{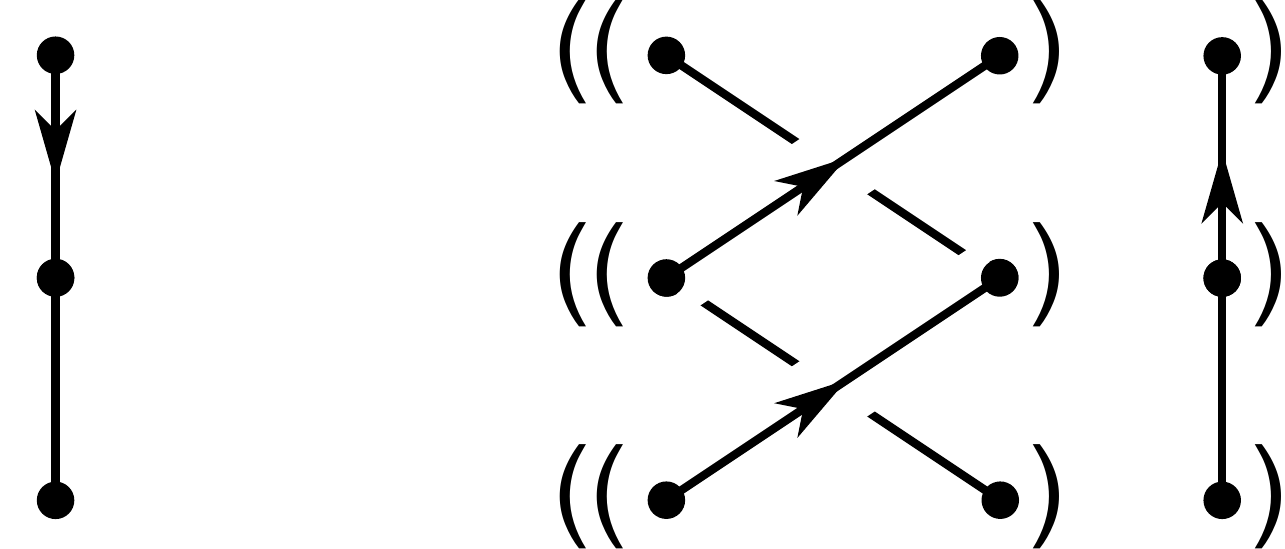}\right)=\exp(t_{23})$

$Z\left(\mathfig{1}{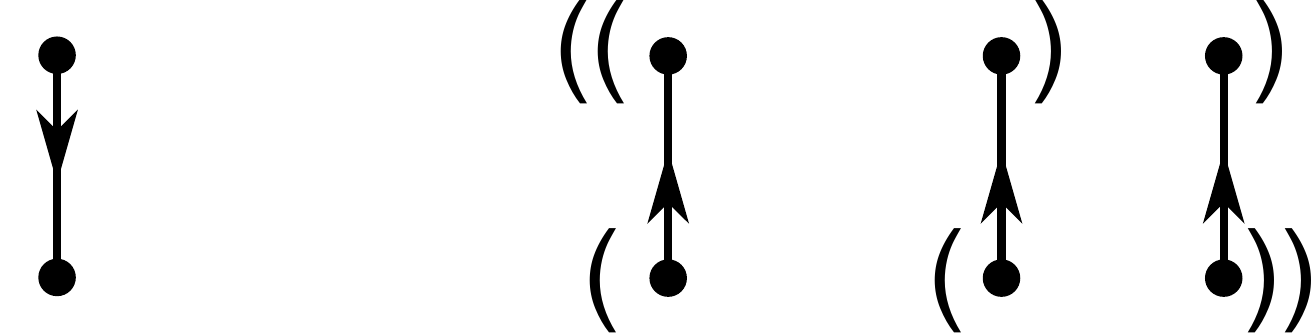}\right)=\phi(t_{23},t_{34})^{-1}$

$Z\left(\mathfig{1}{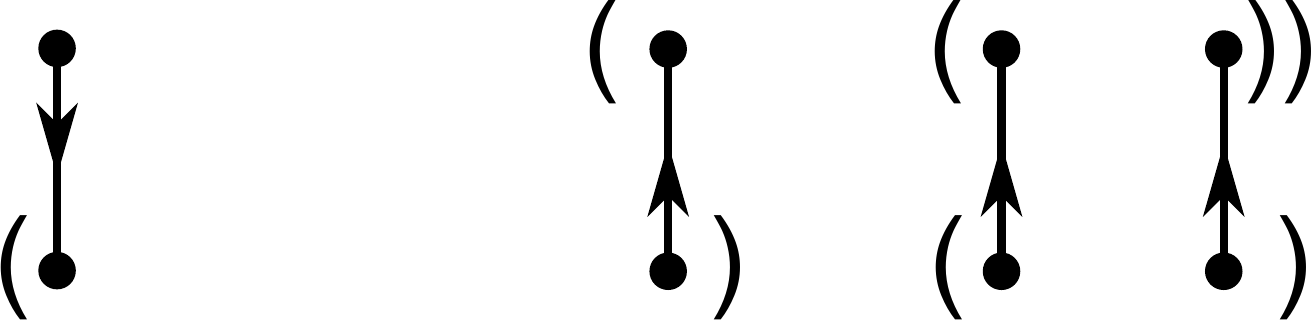}\right)=1\quad\text{by lemma \ref{lemma0}}$

$Z\left(\mathfig{1}{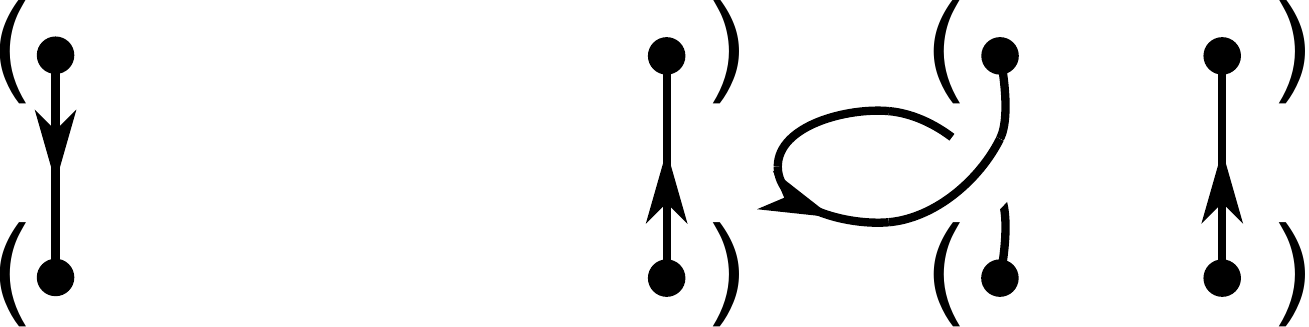}\right)\approx 1 \mod H_2$

Putting it all together we get only $3$ strands, since the $2$ left strands are connected at the top, and become one strand labeled by $x$. Note that in all the the above diagrams we have no vertex on the left strand, therefore after the composition the $y$-labeled edge is at the top of the $x$ strand. Hence we get the following element of $\mathcal{A}(\uparrow_x\uparrow^2,\{y\})$:
$$\chi_x(LMO(X_{+,+}))\approx \exp\left(\quad y\mathfigns{1}{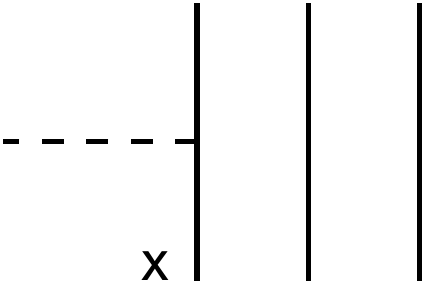}\quad\right)\cdot \phi(t_{12},t_{23})\cdot \exp(t_{12}) \cdot \phi(t_{12},t_{23})^{-1}\mod H_2$$

The proof of the theorem for $X_{+,+}$ is now completed by Lemmas \ref{lemma_chi} and \ref{lemma_u}.

For $Y_{+,+}$ we find it easier to carry out the calculation on $Y^{-1}_{+,+}$. We use the following presentation:
$$\mathfig{10}{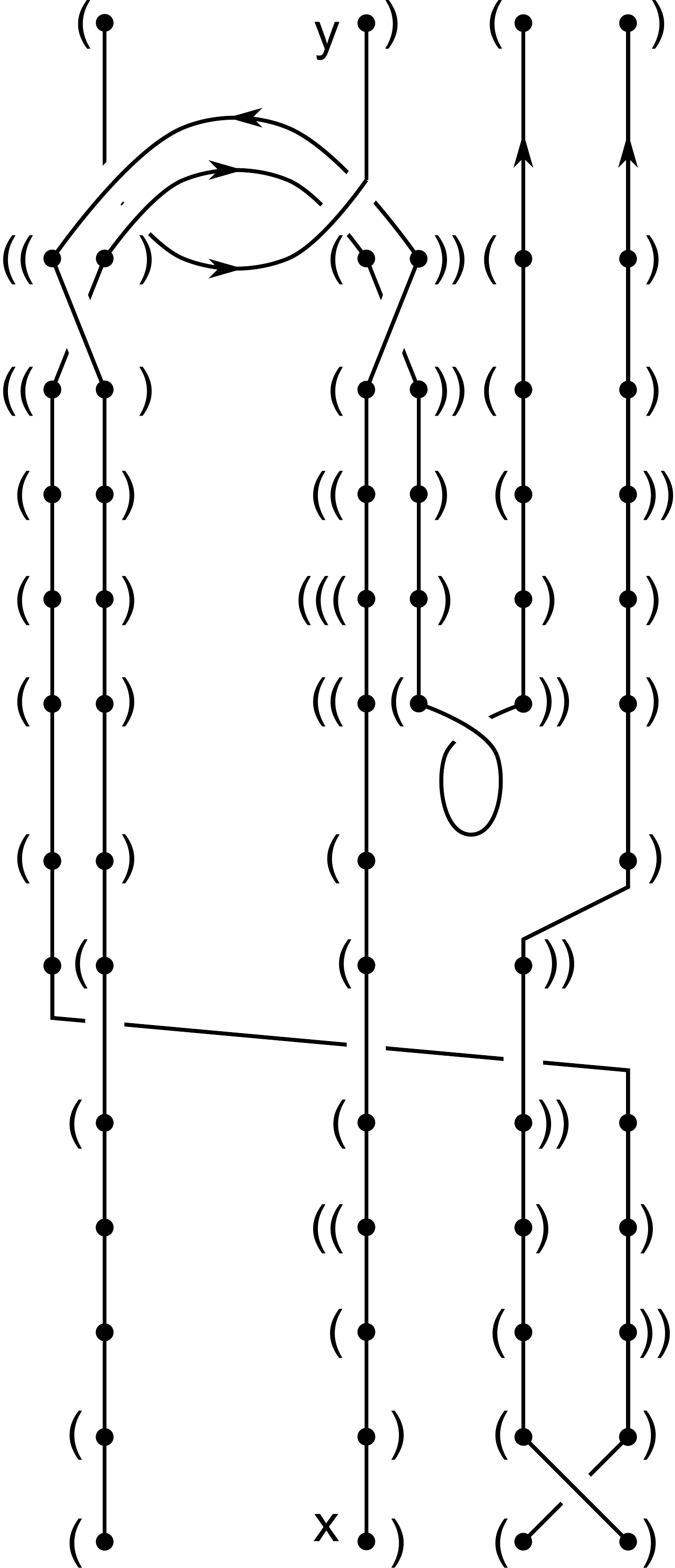}$$

In the following calculation, note that in some of the slices we get elements which by themselves are not equivalent to $1\,\mod H_2$, but they become equivalent to $1$ after composing all the diagrams.

First we need to calcualte:
$$\chi_x(T_1(++))\circ\chi^{-1}_y\circ Z\left(\mathfig{2}{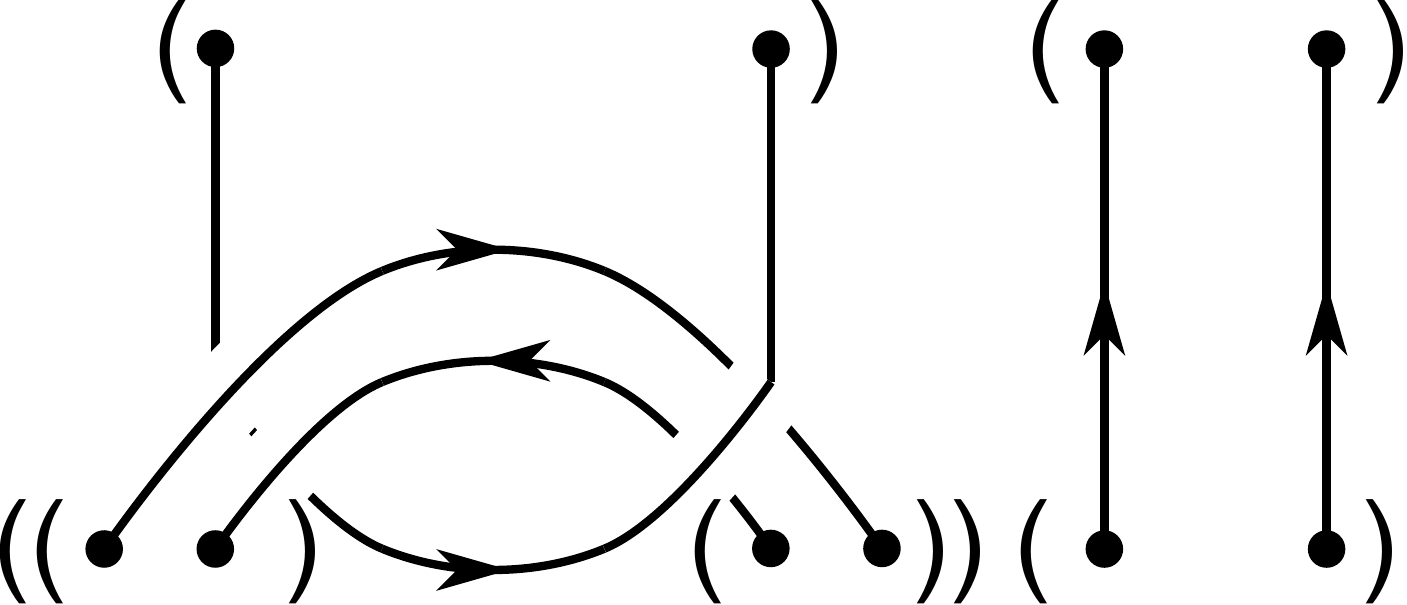}\right)$$
According to Lemma 5.5 of \cite{Cheptea}, this can be computed as:
$$\mathfig{4}{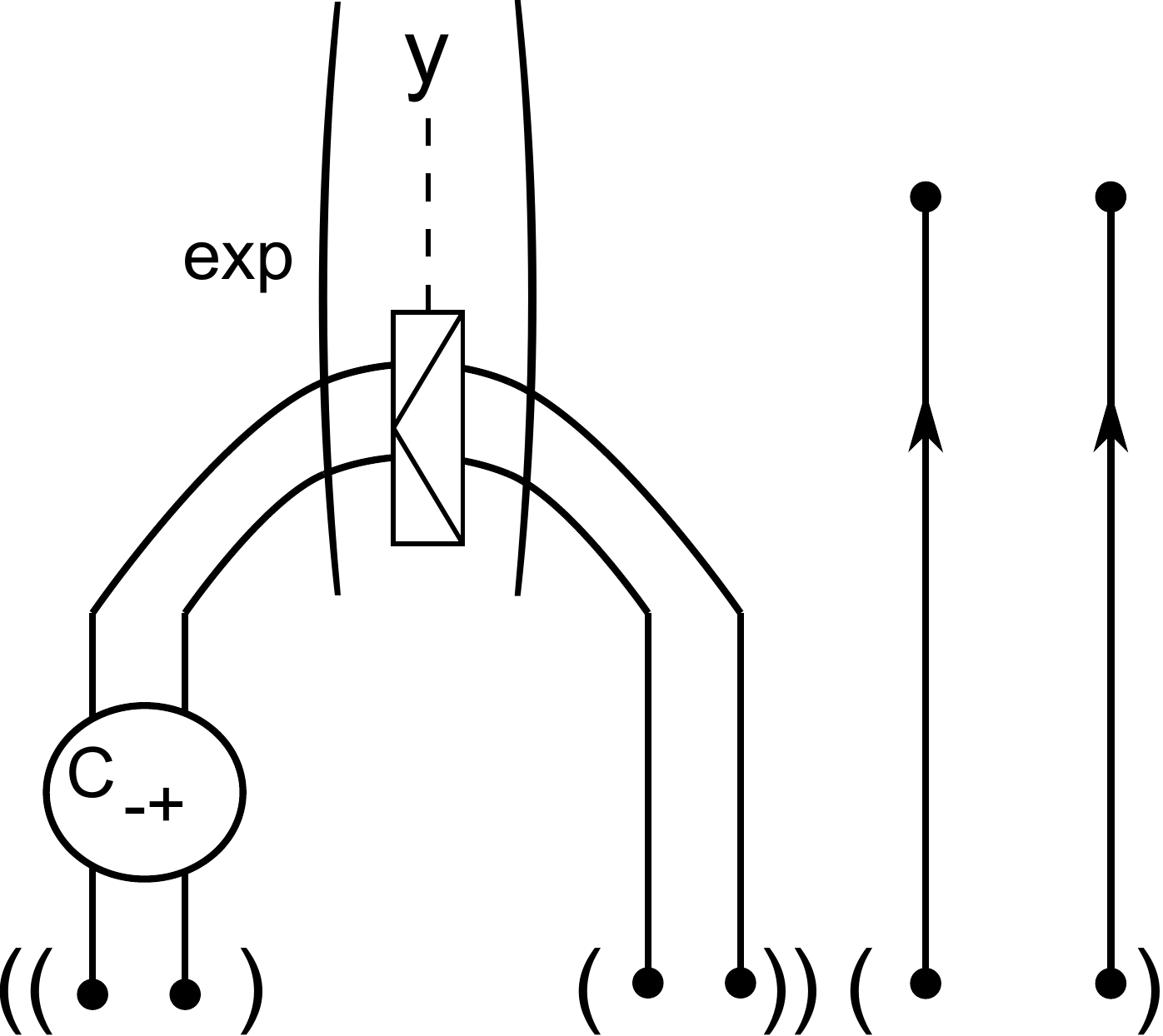}$$
where $C_{-+}$ is $Z(\mathfig{0.9}{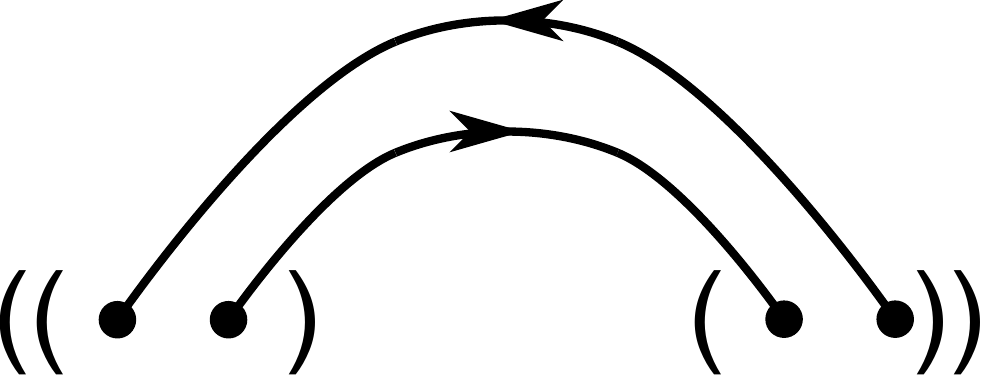})\in\mathcal{A}(\mathfig{0.9}{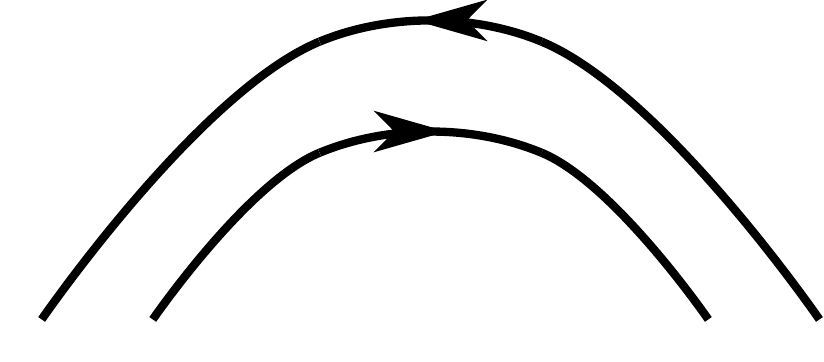}\cong \mathcal{A}(\uparrow\uparrow)$. We will show at the end of the proof that $C_{++}\equiv 1 \mod H_2$. Therefore:

$\chi_x(T_1(++))\circ\chi^{-1}_y\circ Z\left(\mathfig{2}{section5_theorem_related_Y11}\right)=$

$\quad=\exp\left(\quad y\mathfigns{1}{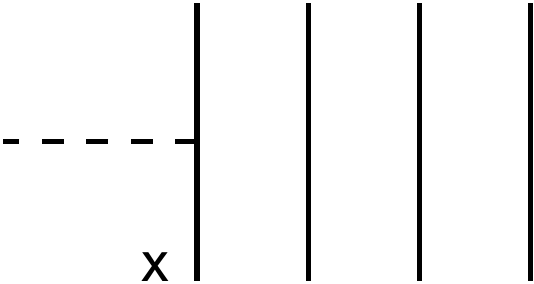}\quad\right)\exp\left(-\quad y\mathfigns{1}{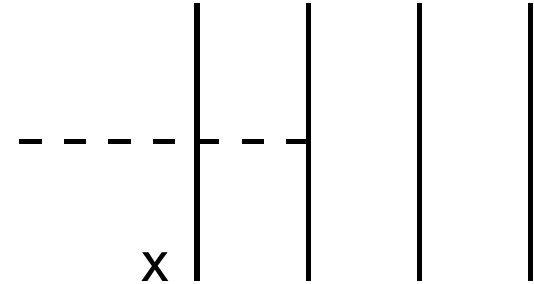}\quad\right)$

$Z\left(\mathfig{1}{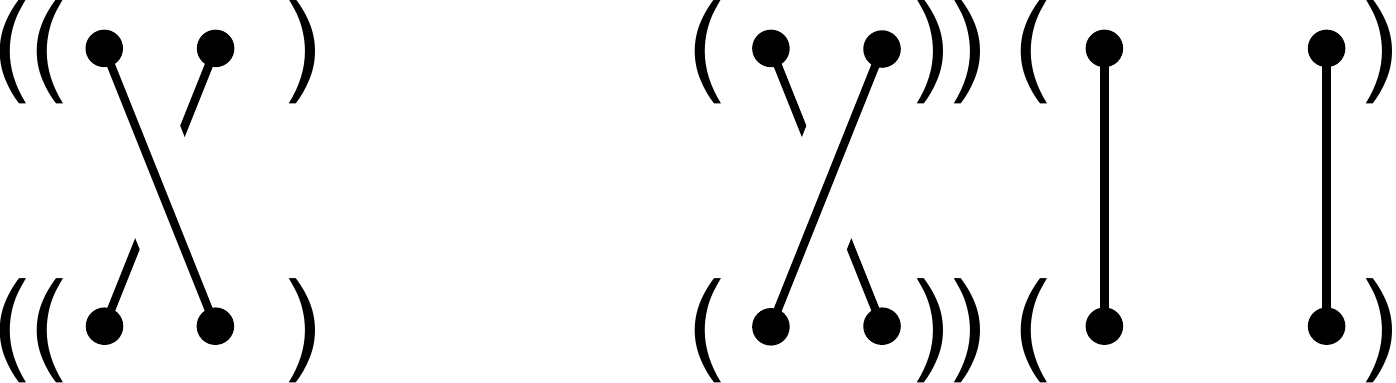}\right) = \exp(-t_{12}/2)\exp(t_{34}/2)$. Since those exponents commute with the box coming from the above tangle, they cancel each other and do not contribute to the final expression.

$Z\left(\mathfig{1}{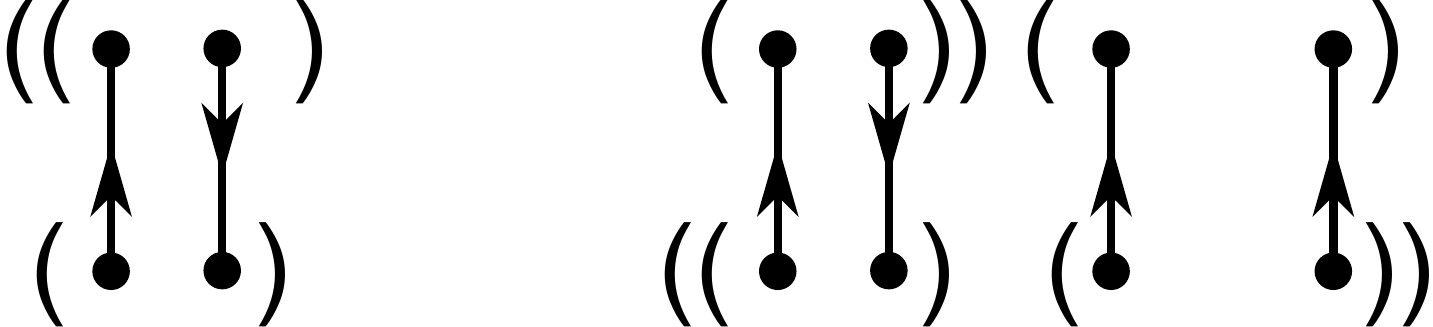}\right)= 1 \quad\text{by lemma \ref{lemma0}}$

$Z\left(\mathfig{1}{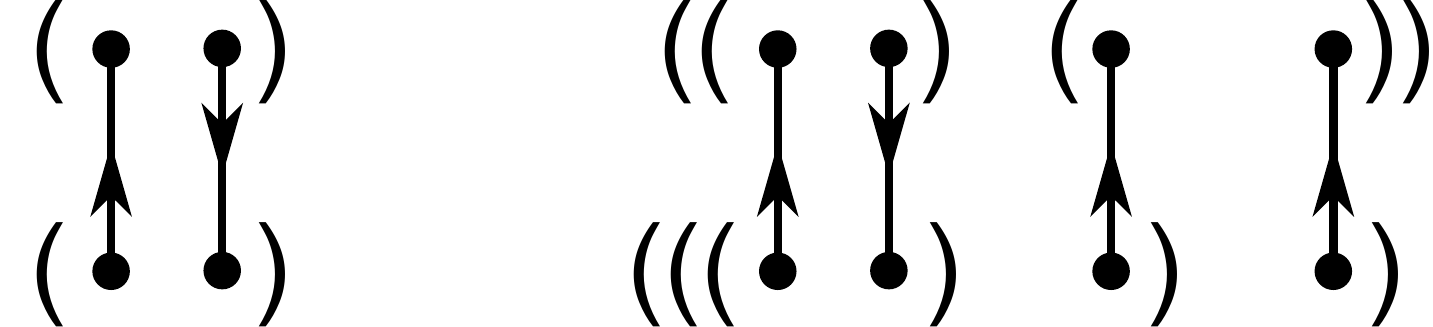}\right)=\phi(t_{35}-t_{45},t_{56})\approx$

$\quad\approx\phi(t_{35},t_{56})\mod H_2\quad\text{(after composition)}$

$Z\left(\mathfig{1}{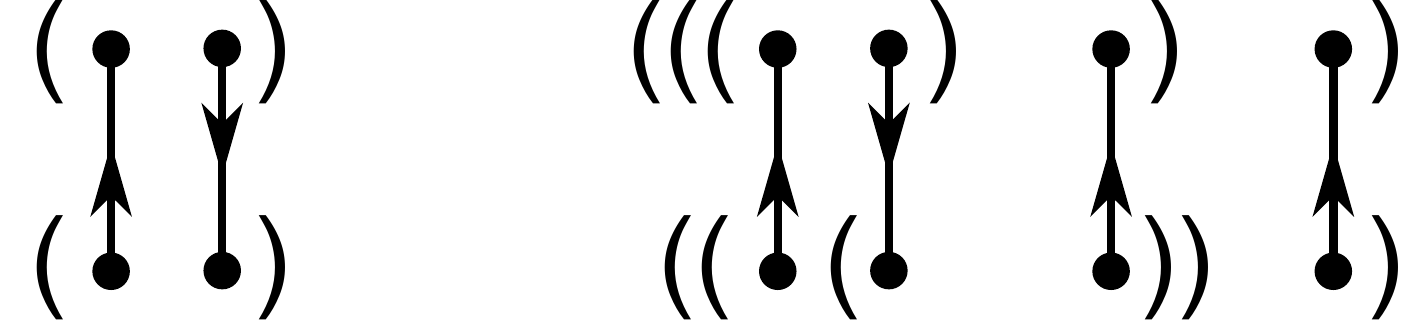}\right)=\phi(-t_{34},-t_{45})^{-1}\approx 1\mod H_2\quad\text{(after composition)}$

$Z\left(\mathfig{1}{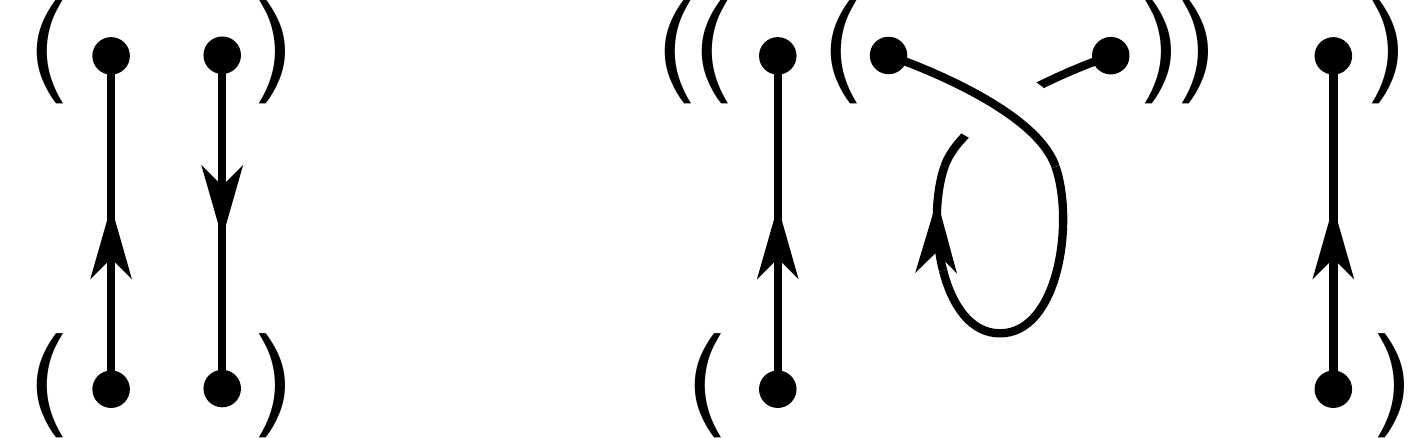}\right)\approx 1 \mod H_2$  
 
$Z\left(\mathfig{1}{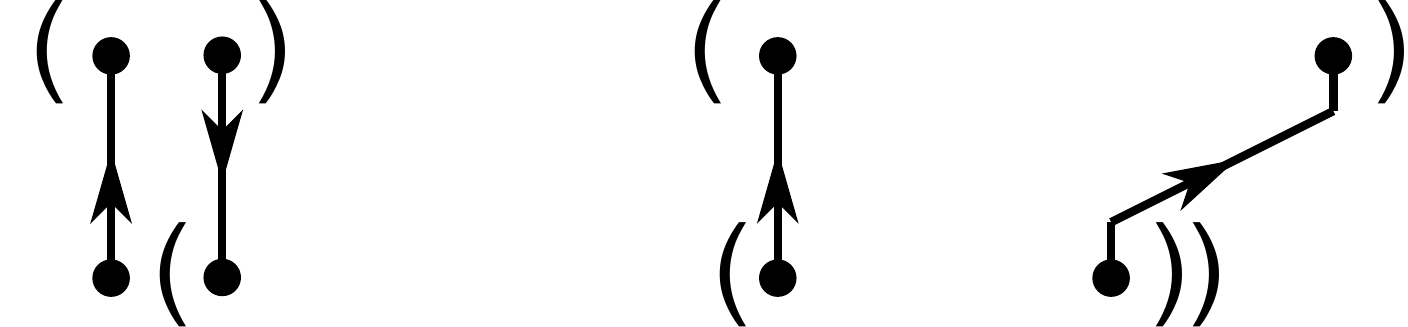}\right)=1\quad\text{by lemma \ref{lemma0}}$

$Z\left(\mathfig{1}{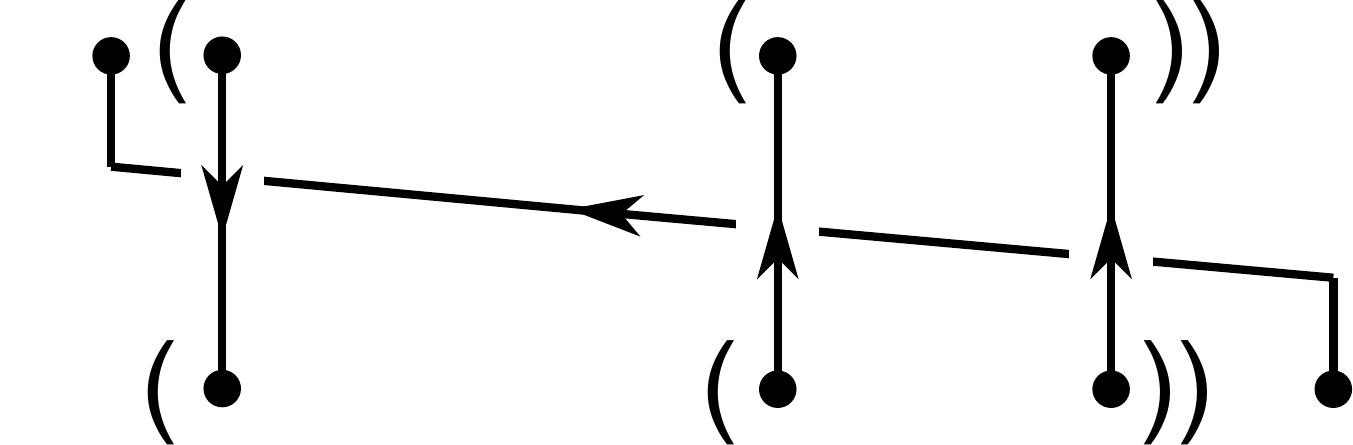}\right)=\exp\left(\frac{1}{2}\mathfig{1}{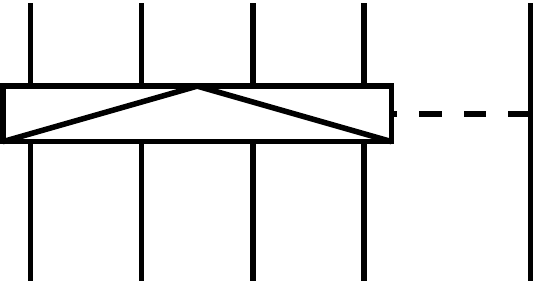}\right)=\quad\text{(using an $I$ relation)}$

$\quad=\exp\left(-\frac{1}{2}\mathfig{1}{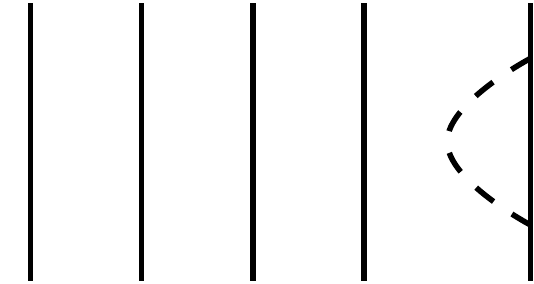}\right)\approx 1\mod H_2$

$Z\left(\mathfig{1}{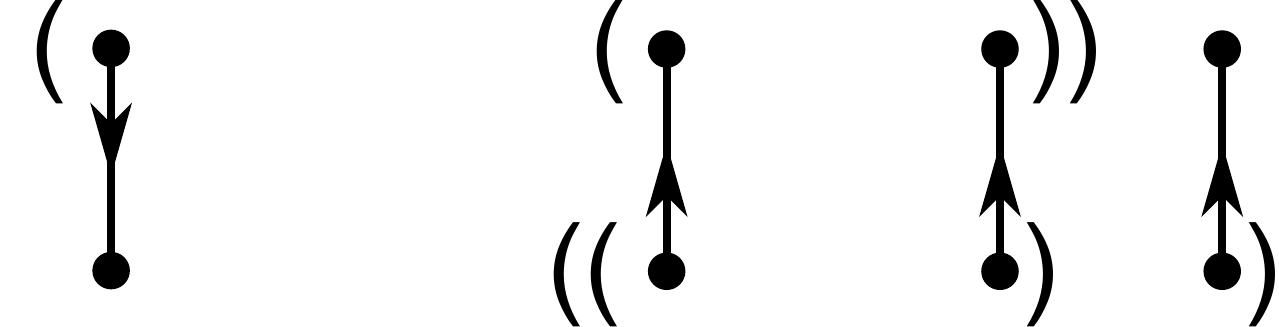}\right)=1\quad\text{by lemma \ref{lemma0}}$

$Z\left(\mathfig{1}{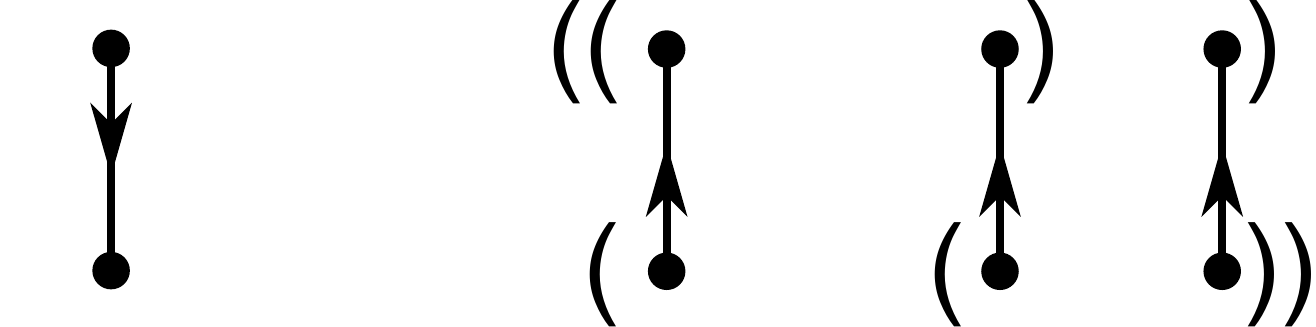}\right)=\phi(t_{23},t_{34})^{-1}=\phi(-t_{24}-t_{34},t_{34})^{-1}$

$\quad\text{(because $t_{23}+t_{24}+t_{34}$ commutes with $t_{23}$ and $t_{34}$)}$

$Z\left(\mathfig{1}{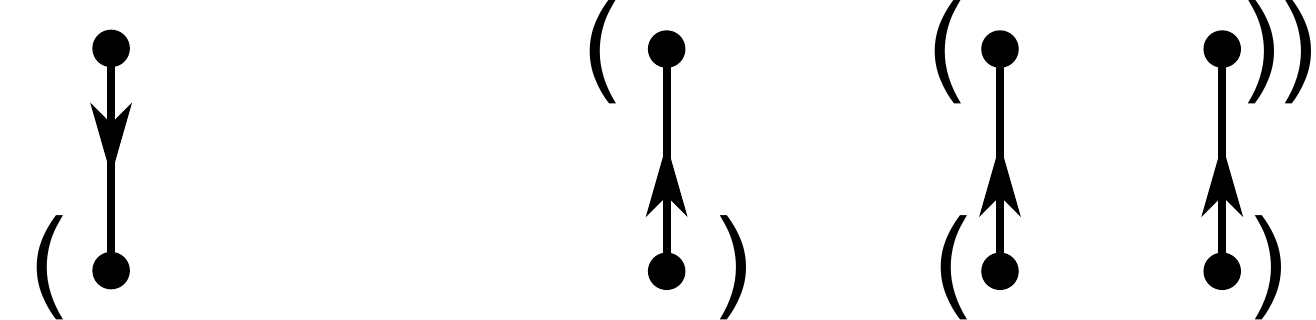}\right)= 1 \quad\text{by lemma \ref{lemma0}}$

$Z\left(\mathfig{1}{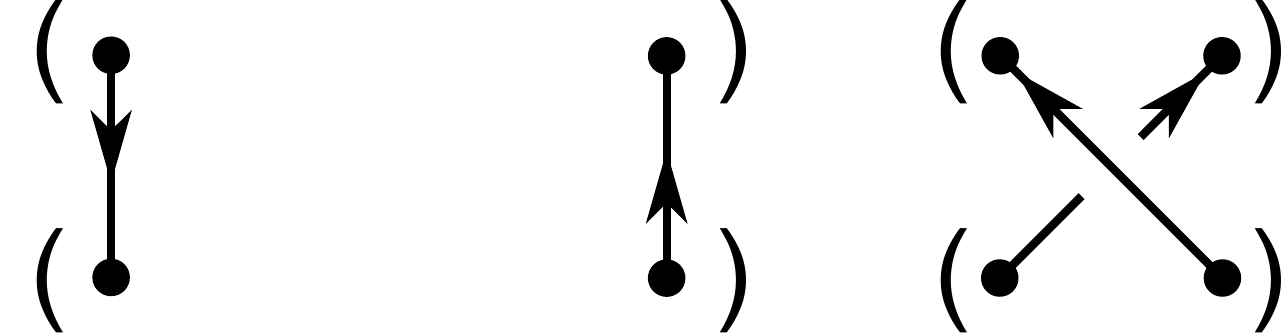}\right)=\exp(-t_{34}/2)$

Putting it all together we get the following element of $\mathcal{A}(\uparrow_x\uparrow^2,\{y\})$:

$$\chi_x(LMO(Y^{-1}_{+,+}))\approx\exp\left(\quad y\mathfigns{1}{section5_theorem_related_y}\quad\right)\cdot$$
$$\cdot\phi(t_{12},t_{23})\exp\left(-\quad y\mathfigns{1}{section5_theorem_related_y_2}\quad\right)\phi(-t_{12}-t_{23},t_{23})^{-1}\exp(-t_{23}/2)\quad\mod H_2$$

Note that when we travel along strand $x$ we encounter $\exp\left(\quad y\mathfigns{1}{section5_theorem_related_yx}\quad\right)$ after the associator $\phi(t_{12},t_{23})$, whereas when we travel along the second strand we encounter $\exp\left(-\quad y\mathfigns{1}{section5_theorem_related_y_2}\quad\right)$ before this associator.

By Lemmas \ref{lemma_chi} and \ref{lemma_u} we get:
$$LMO^<(Y^{-1}_{+,+})=\tilde{u}_2(\phi(\tilde{x}_1,t)\exp(-y_1)\phi(-\tilde{x}_1-t,t)^{-1}\exp(-t/2))$$
Therefore:
$$LMO^<(Y_{+,+})=\tilde{u}_2(\exp(t/2)\phi(-\tilde{x}_1-t,t)\exp(y_1)\phi(\tilde{x}_1,t)^{-1})$$
which is the expression we wanted to get.

In order to complete the proof we only need to show that $C_{-+}\approx 1 \mod H_2$. $C_{-+}$ is the composition of $\mathfig{1}{double_cap_tangle}$ with $\phi(-t_{23},-t_{34})\phi(-t_{12},-t_{23}+t_{24})\in\mathcal{A}(\downarrow\uparrow\downarrow\uparrow)$.
$t_{23}$ is in $H_2$ (after the composition), therefore we are left with $\phi(-t_{12},t_{24})$. This can be written as the exponent of a sum of iterated commutators, where the innermost commutator is:
$$\mathfig{2}{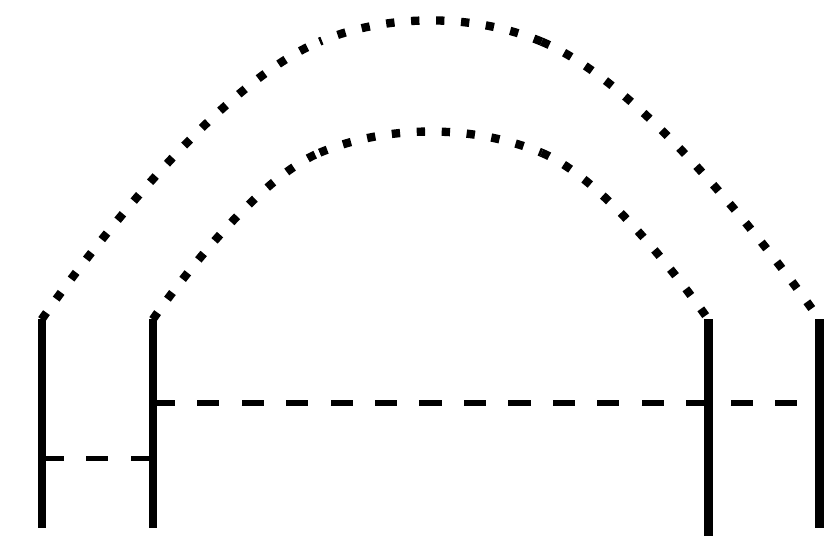}-\mathfig{2}{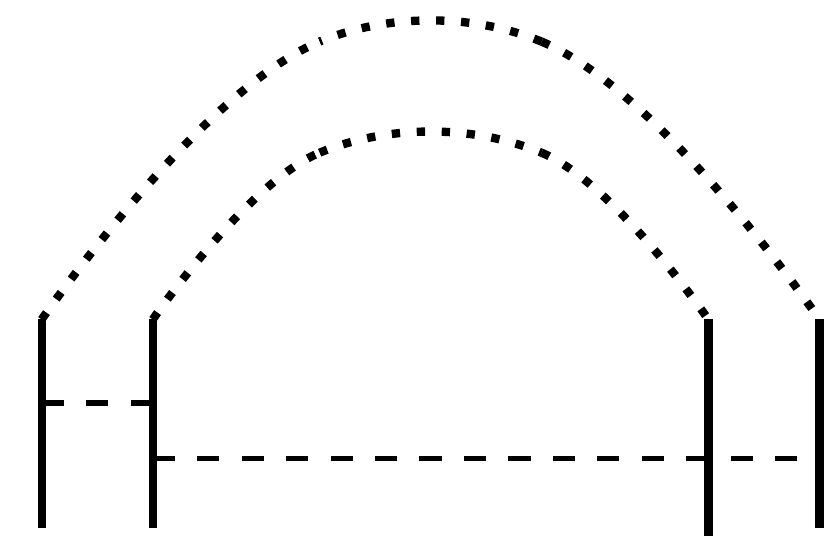}$$
Using several STU and IHX relations we can ``transfer'' the nodes on the rightmost strand to nodes on the leftmost strand, in the price of adding many more diagrams in which this node is transferred to the lower capped strand. But all those extra diagrams are in $H_2$. Therefore we are left with the commutator:
$$\mathfig{2}{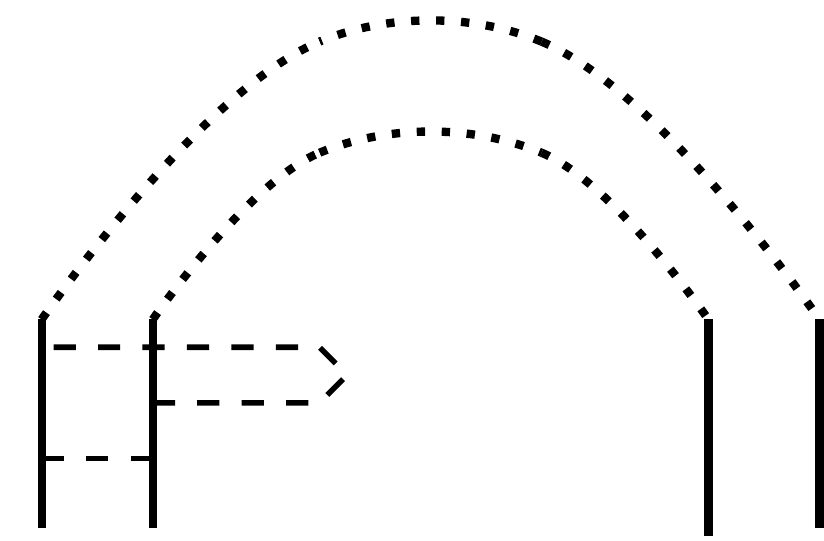}-\mathfig{2}{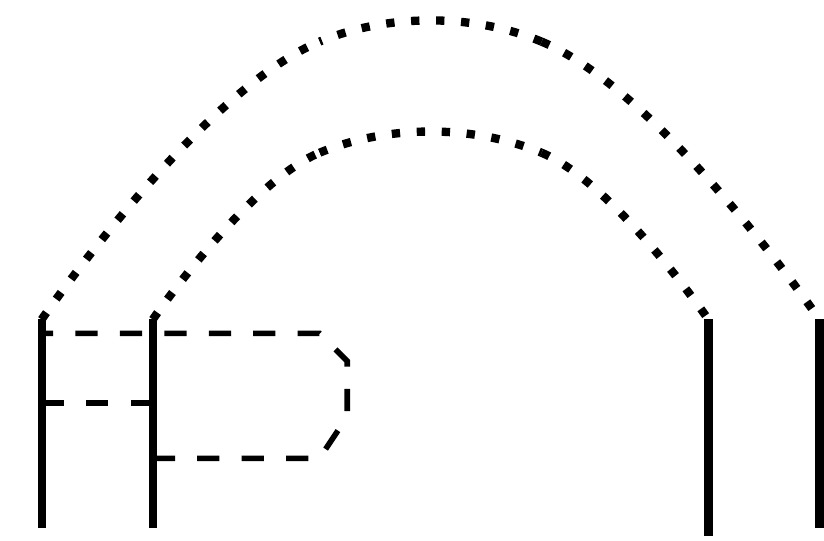}$$
which, by another STU relation (on the left strand), is also in $H_2$. Thus we have completed the proof.

\end{proof}

\begin{remark}
\label{remark_not_equal}
Looking at the above calculation we see why the $H$ relations were needed. For example, in $LMO^<(X_{+,+})\in\mathcal{A}_1^{<p}(\uparrow^2)$ we have all the high degree components of $Z\left(\mathfig{0.6}{section5_theorem_related_X1}\right)$ which do not vanish (with the exception of the $2$ degree component which vanishes, see the proof of Theorem \ref{theorem_elliptic} in the next section), and are not in the image of the map $u_2$, so they do not come from the elliptic associator $e(\phi)$.
\end{remark}

\subsection{Proof of Theorem \ref{theorem_elliptic}}
In this section we use the relation we found between $e(\phi)$ and the LMO functor to give an alternative proof that $e(\phi)$ is an elliptic associator. In fact, most of the proof follows from the fact that $LMO^<(X_{+,+})$ and $LMO^<(Y_{+,+})$ define an elliptic structure on $\mathcal{A}_1^{<p}(\uparrow^n)$.

Indeed, identities (\ref{elliptic3}) and (\ref{elliptic4}) hold for $X_{+,+}$ and $Y_{+,+}$ in $q\tilde{T}_1$. Applying $LMO^<$ to both sides and taking the quotient with $H_2$ (for \ref{elliptic3}) and $H_3$ (for \ref{elliptic4}) we get the same identities for $LMO^<(X_{+,+})$ and $LMO^<(Y_{+,+})$ in $R\mathcal{A}_1^{<p}(\uparrow^2)/H_2$ and $R\mathcal{A}_1^{<p}(\uparrow^3)/H_3$. Pulling them back to $RU\hat{\textbf{t}}_{1,2}\subset U\hat{\textbf{t}}_{1,2}$ and $RU\hat{\textbf{t}}_{1,3}\subset U\hat{\textbf{t}}_{1,3}$ by $\tilde{u}_{r,2}^{-1}$ and $\tilde{u}_{r,3}^{-1}$, and using theorem \ref{theorem_related}, we get identities (\ref{associator_identity1}) and (\ref{associator_identity4}) for $X_\phi$ and $Y_\phi$.

Similarly, identities (\ref{elliptic1}) and (\ref{elliptic2}) hold for $X_{+,+}$ and $Y_{+,+}$ in $q\tilde{T}_1$. We can now again apply $LMO^<$ to both sides, take the quotient with $H_3$ and pull back to $RU\hat{\textbf{t}}_{1,3}\subset U\hat{\textbf{t}}_{1,3}$ by $\tilde{u}_{r,3}^{-1}$. Theorem \ref{theorem_related} shows that at the right hand sides we get the right hand sides of identities (\ref{associator_identity2}) and (\ref{associator_identity3}) for $X_\phi$ and $Y_\phi$. However, we still need to show that the left hand sides are equal. More explicitly, we need to show that $\tilde{u}_{r,3}^{-1}\circ LMO^<(Z_{++,+})=Z_\phi(x_1+x_2,y_1+y_2)$ for $Z=X$ and $Z=Y$. This will be the content of the following proof.

\begin{proof}[Proof of theorem \ref{theorem_elliptic}]

Let $\Delta_1:RU\hat{\textbf{t}}_{1,2}\rightarrow RU\hat{\textbf{t}}_{1,3}$ be the map defined by : $v_1\mapsto v_1+v_2$ and $v_2\mapsto v_3$ ($v=x$ or $v=y$). As explained above, we need to calculate $\tilde{u}_{r,3}^{-1}\circ LMO^<(Z_{++,+})$ for $Z=X,Y$, and show that they are equal to $\Delta_1\circ \tilde{u}_{r,2}^{-1} LMO^<(Z_{+,+})$. We will use the same decomposition of $X_{+,+}$ and $Y_{+,+}$ to simple tangles as we used above. For most of those simple tangles $T$ we can show that:
\begin{equation}
\label{T_identity}
\tilde{u}_{r,3}^{-1}\circ LMO^<(\Delta^{++}_{++,+}(T))=\Delta_1\circ \tilde{u}_{r,2}^{-1}(LMO^<(T))
\end{equation}
In the few cases where this identity does not hold, the extra components we get will eventually cancel each other.

For a tangle $T$ with no cups or caps, we can use the identity $LMO^<(\Delta^{++}_{++,+}(T))=\Delta^{++}_{++,+}\circ LMO^<(T)$. Decompose $LMO^<(T)$ as $LMO^<(T)_u+LMO^<(T)_H$, where $LMO^<(T)_u$ is in the image of $u_2$, and $LMO^<(T)_H$ is in $H_2$. Clearly, $\tilde{u}_{r,3}^{-1}\circ\Delta^{++}_{++,+}(LMO^<(T)_u)=\Delta_1\circ \tilde{u}_{r,2}^{-1}(LMO^<(T)_u)$. Therefore, in order to prove identity (\ref{T_identity}) for such $T$, it is enough to show that $\Delta^{++}_{++,+}(LMO^<(T)_H)$ is in $H_3$.

In the calculation of the $LMO^<$ functor of $X_{+,+}$ and $Y_{+,+}$ we encountered several contributions to the $H_2$ part. First, according to lemma \ref{lemma_chi}, each edge with a vertex on the $x$ strand, after applying $k\circ j^{-1}\circ\chi^{-1}_x$, became:
\begin{equation}
\sum_{i=0}^{\infty}B_i\mathfig{3}{section5_a_lemma2}
\label{sum}
\end{equation}
The $H_2$ part of this sum is, according to lemma \ref{lemma_u}: $\sum_{i=1}^\infty B_i\mathfig{2}{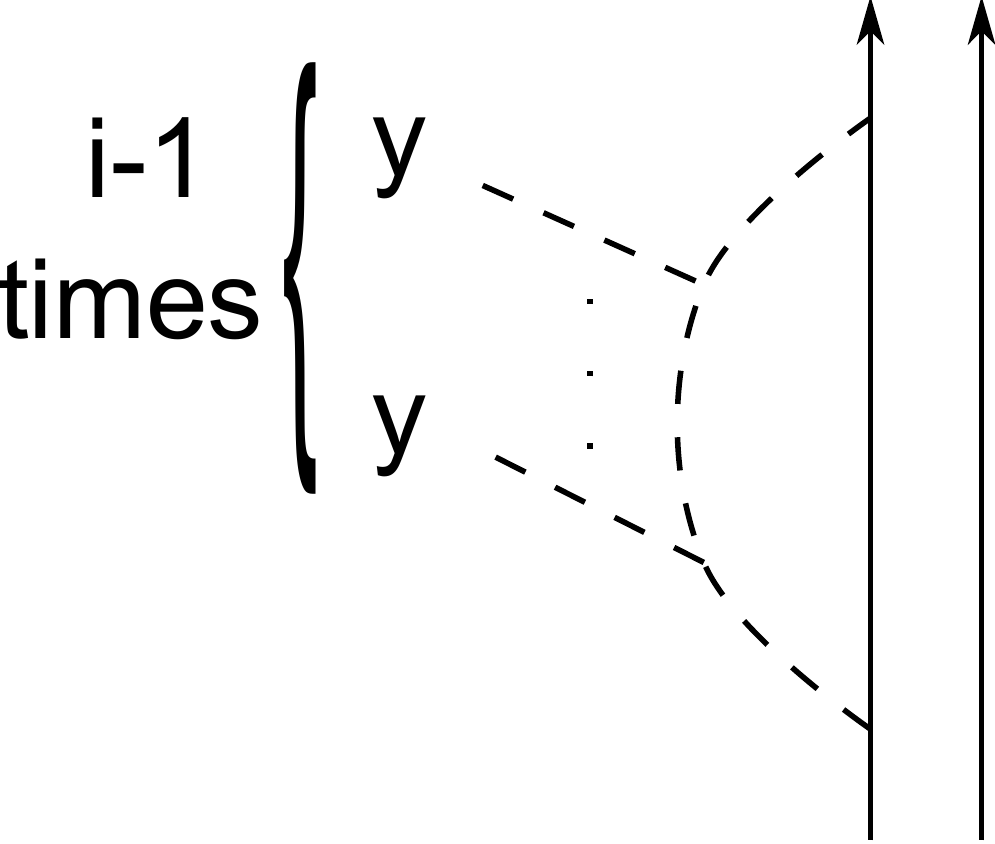}$. For $i$ even we have:

$$\Delta^{++}_{++,+}\left(\mathfig{2}{section5_section53_tilde1}\right)=\mathfig{2}{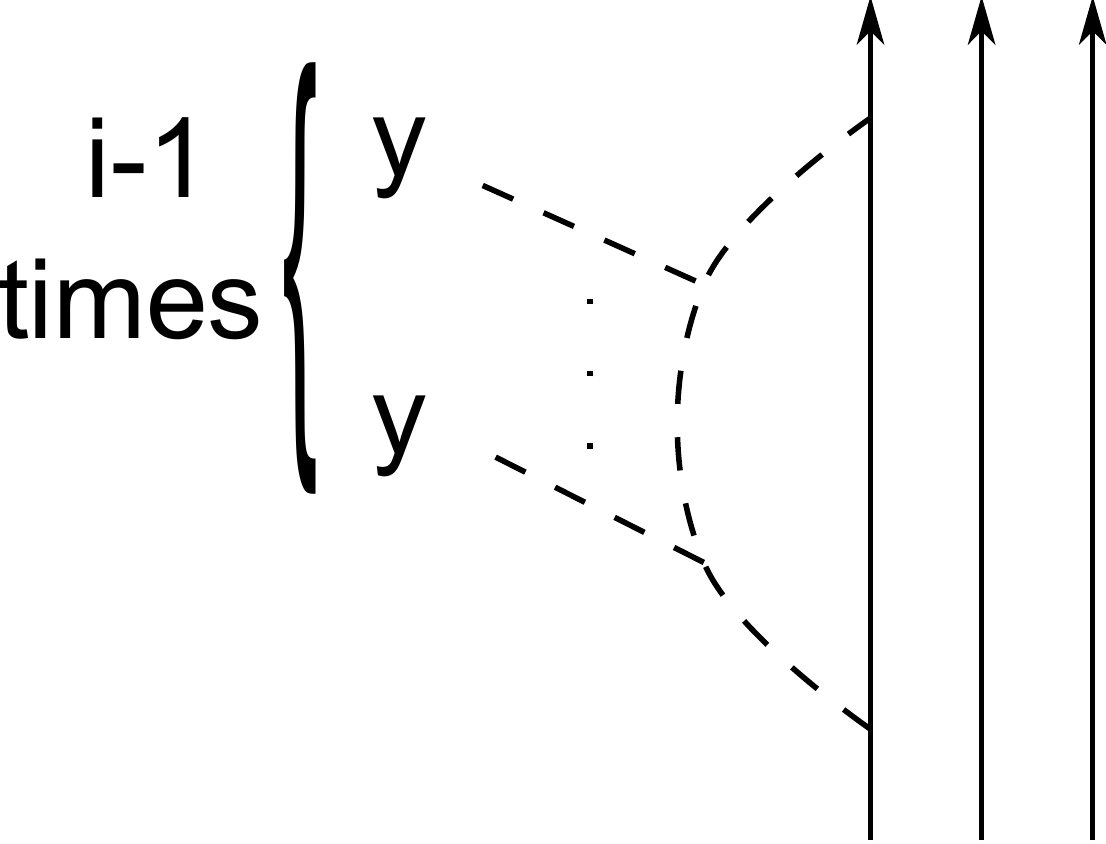}+\mathfig{2}{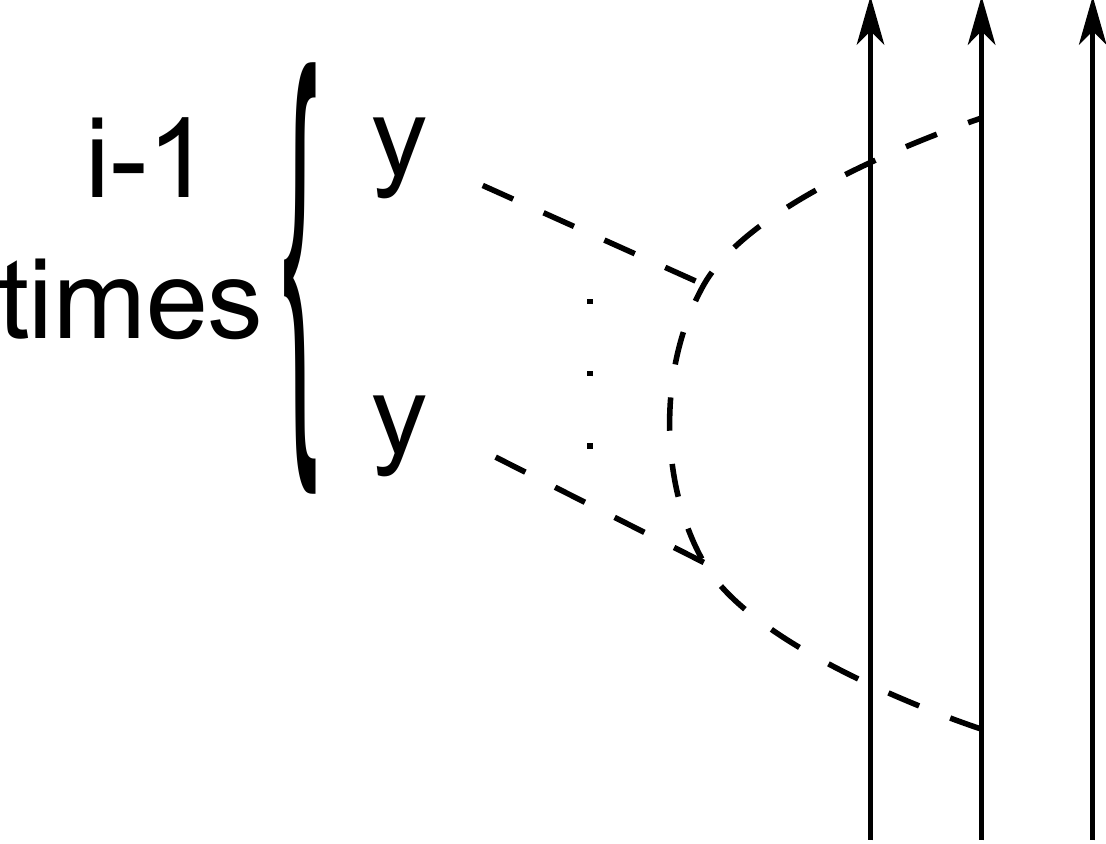}+$$
$$+\mathfig{2}{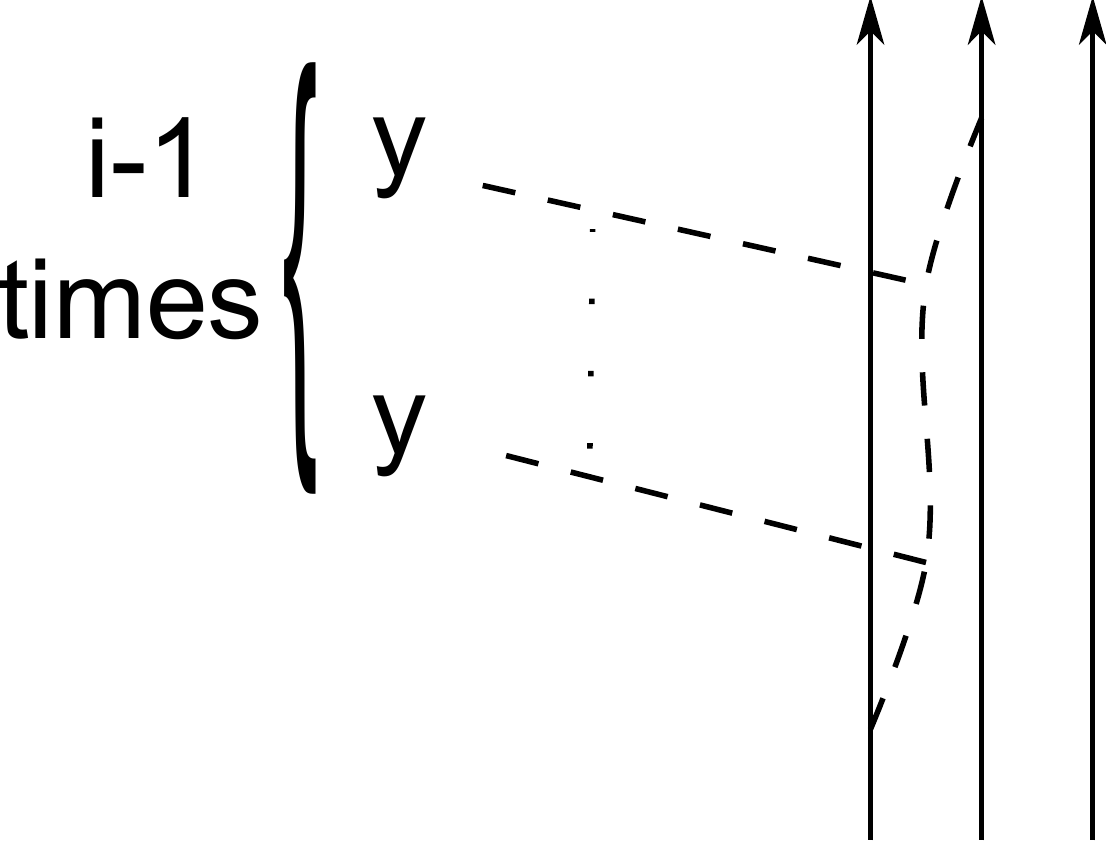}+\mathfig{2}{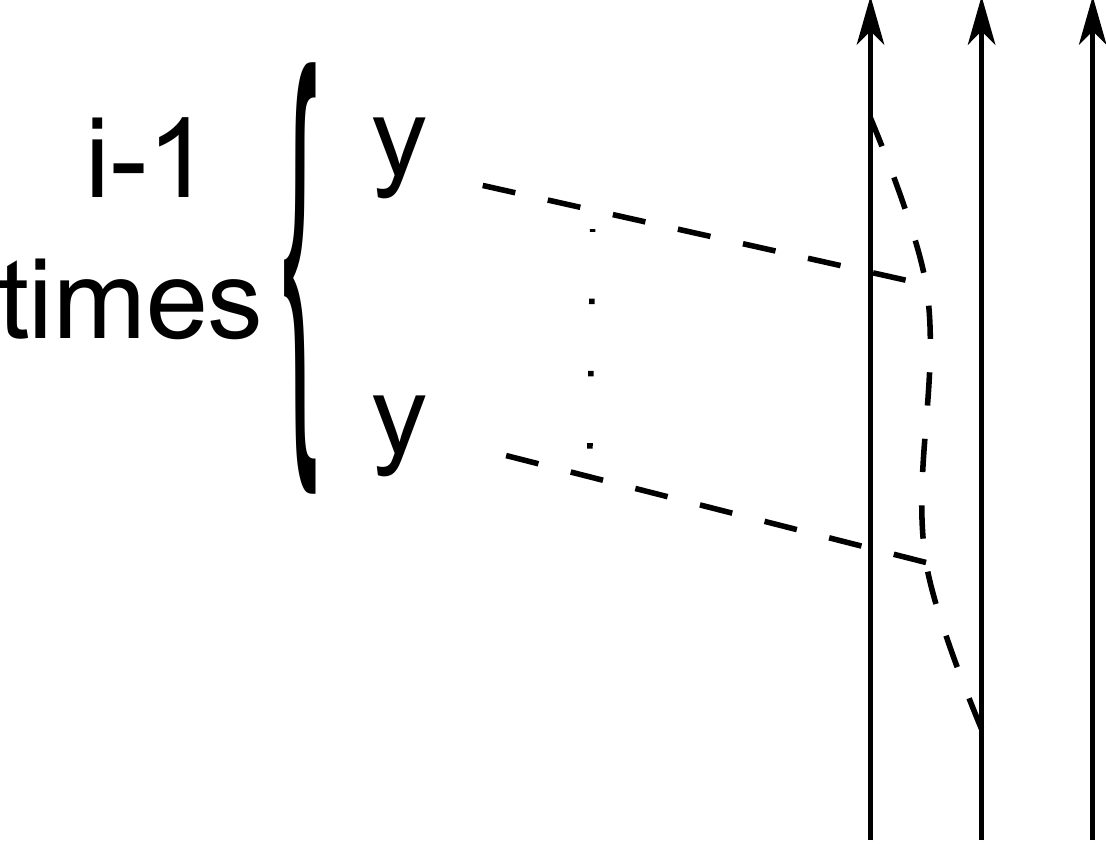}=$$
$$=\mathfig{2}{section5_section53_tilde2}+\mathfig{2}{section5_section53_tilde3}\in H_3$$

For $i$ odd, the only non-zero coefficient in the sum (\ref{sum}) is $B_1=-\frac{1}{2}$. When this sum appears in an associator, the $-\frac{1}{2}\mathfig{1}{section5_lemma3_proof2_6}$ summand cancels, because it commutes with everything else.

In $LMO^<(X_{+,+})$ the sum (\ref{sum}) appears twice inside an associator, so the $H_2$ part of those tangle-parts indeed maps to $H_3$. There is also one occurance of this sum which appears inside an exponent. Therefore we are left with $\exp(-\frac{1}{2}\mathfig{1}{section5_lemma3_proof2_6})$ which is in $H_2$, but is not mapped by $\Delta^{++}_{++,+}$ to $H_3$. However, we will immediately see that this part eventually cancels.

In $LMO^<(X_{+,+})$ we also have $Z\left(\mathfig{1}{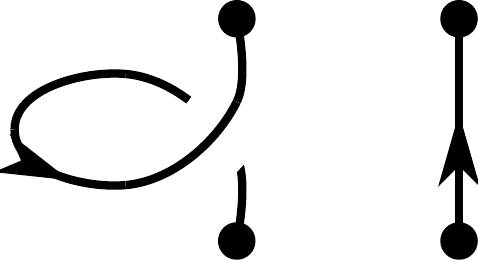}\right)$ which has a cup and a cap. We need to calculate $Z\left(\Delta^{++}_{++,+}\left(\mathfig{1}{section5_section53_loop}\right)\right)=Z\left(\mathfig{2}{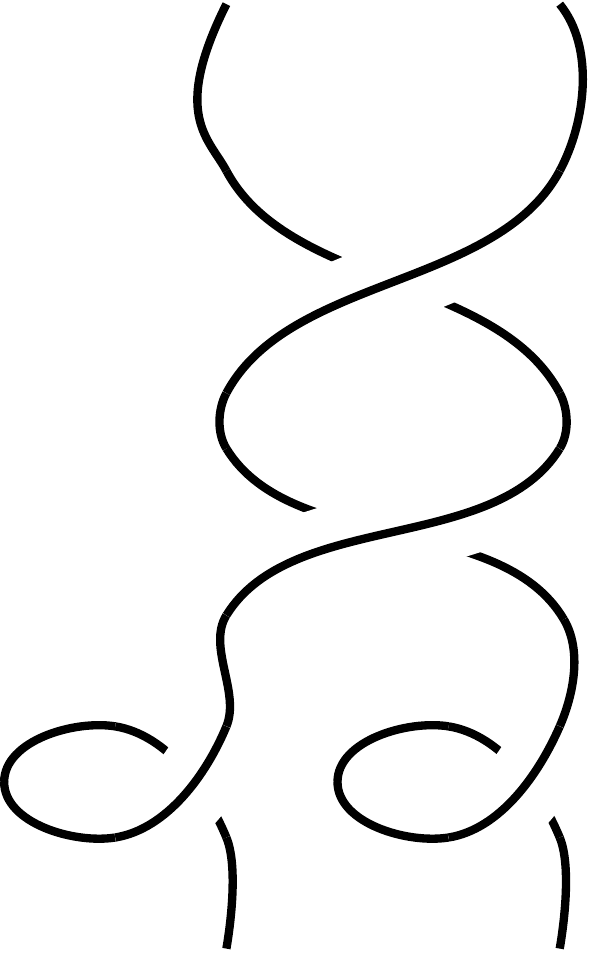}\right)$. This can be written as $\exp\left(t_{34}+u\right)$, where $u$ is in $H_3$. So we are left with $\exp\left(t_{34}\right)$. But this part cancels with $\exp\left(-t_{34}\right)$ coming from applying $\Delta^{++}_{++,+}$ to the exponent of the sum (\ref{sum}) (because they commute with everything in between). This concludes the proof for $LMO^<(X_{+,+})$.

In $LMO^<(Y_{+,+})$, all the occurences of the sum (\ref{sum}) are in associators, so their $H_2$ part is mapped to $H_3$. But there are several more contributions to the $H_2$ part. First, we had:

$Z\left(\mathfig{1}{section5_theorem_related_Y5}\right)=\exp\left(-\frac{1}{2}\mathfig{1}{section5_theorem_related_chord}\right)$

Applying $\Delta^{++}_{++,+}$ to $\exp\left(-\frac{1}{2}\mathfig{1}{section5_theorem_related_chord}\right)$ we get $\exp(-t_{56})$ which is not in $H_3$. However, we will soon see that this exponent cancels with another exponent.

Another contribution to the $H_2$ part comes from $Z\left(\mathfig{1}{section5_theorem_related_Y7}\right)$. So we need to calculate $Z(\Delta^{++}_{++,+}\left(\mathfig{1}{section5_theorem_related_Y7}\right)=Z\left(\mathfig{1.5}{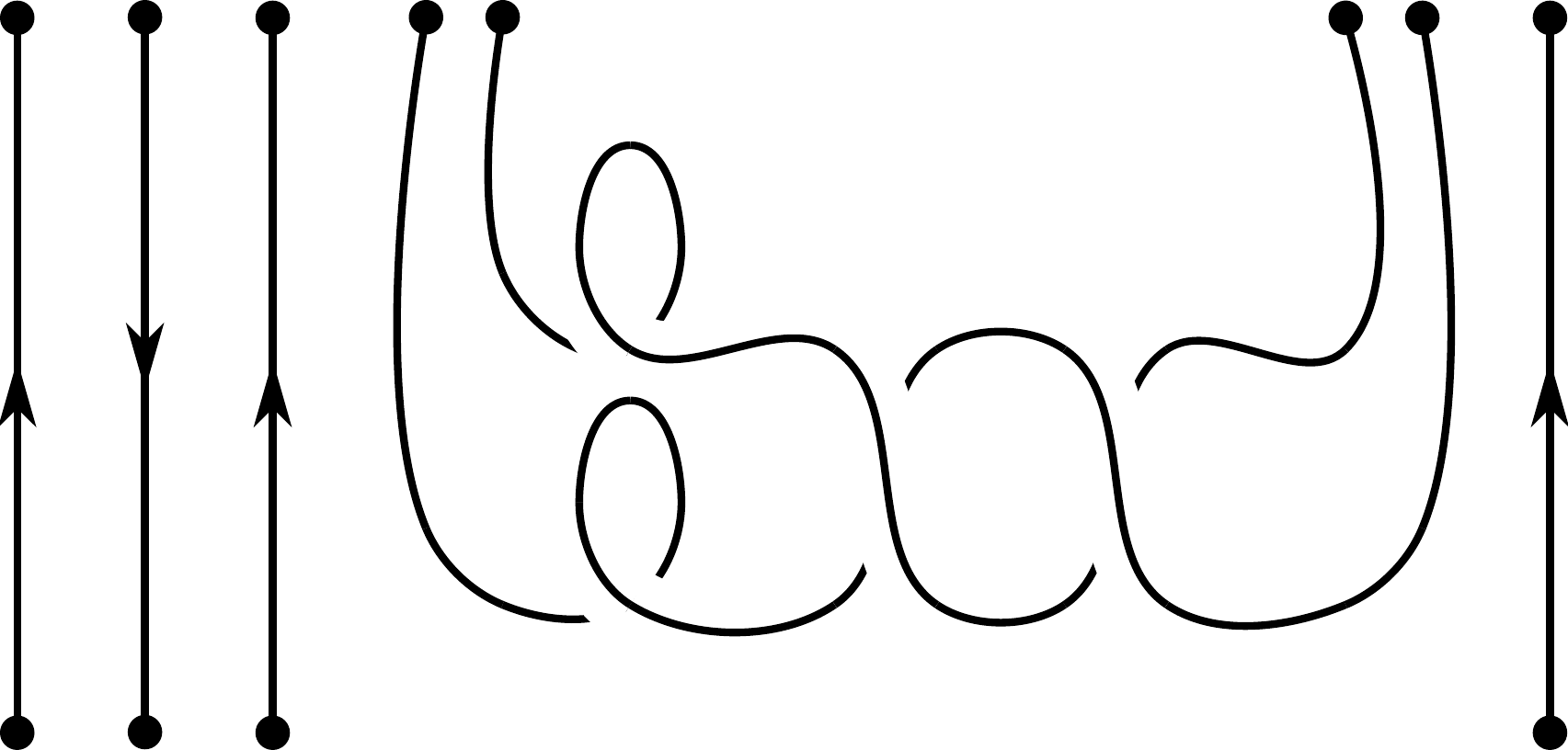}\right)$. This is equal to $\exp(t_{45}+u)$ with $u\in H_3$. So we are left with $\exp(t_{45})$ which is not in $H_3$. But this cancels out with the above $\exp(-t{56})$ (because they commute with everything in between).

We will now deal with the $H_2$ parts which come from: $Z\left(\mathfig{1}{section5_theorem_related_Y9}\right)$ and $Z\left(\mathfig{1}{section5_theorem_related_Y8}\right)$, which are:

$$\phi\left(\mathfig{1}{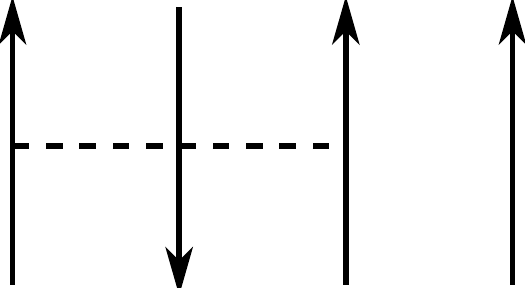}-\mathfig{1}{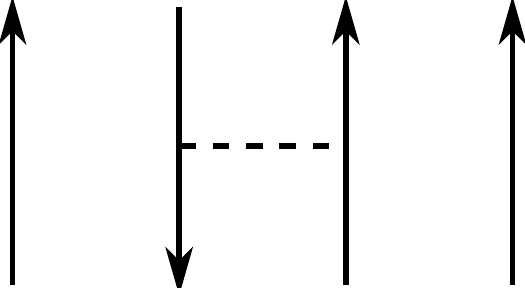},\mathfig{1}{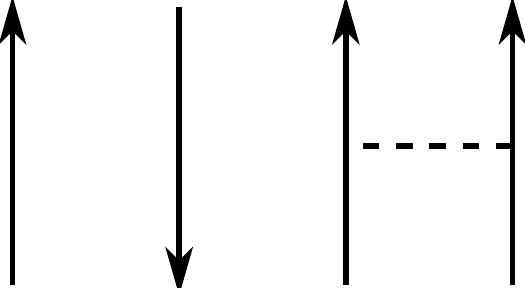}\right)\cdot$$
$$\phi^{-1}\left(\mathfig{1}{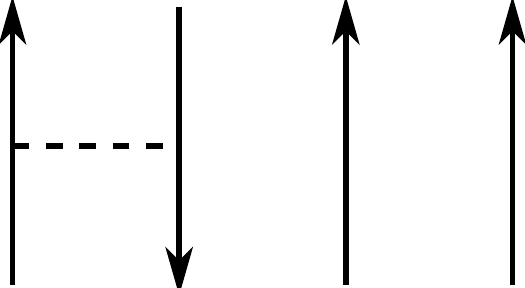},\mathfig{1}{section5_section53_t_23}\right)\cdot\left(\mathfig{1}{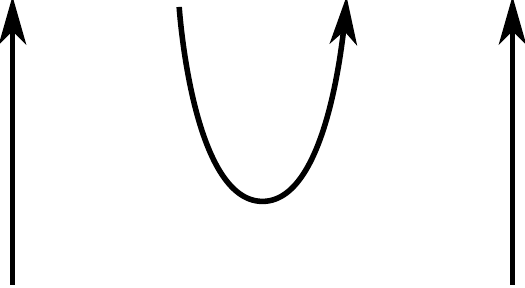}\right)$$

(For convenience we ignored the $2$ left strands.)

$\phi^{-1}\left(\mathfig{1}{section5_section53_t_12},\mathfig{1}{section5_section53_t_23}\right)$ is an exponent of a sum of iterated commutators in $\mathfig{1}{section5_section53_t_12}$ and $\mathfig{1}{section5_section53_t_23}$. The innermost commutator in each of those iterated commutators is $\left[\mathfig{1}{section5_section53_t_12},\mathfig{1}{section5_section53_t_23}\right]$, so it is enough to show that applying $\Delta^{++}_{++,+}$ to this commutator multiplied by $\left(\mathfig{1}{section5_section53_cup}\right)$ is in $H_3$. (Note that we apply $\Delta^{++}_{++,+}$ to the right strands. The left strand will disappear when we apply $\chi^{-1}_x$.) And indeed:

$$\Delta^{++}_{++,+}\left(\mathfig{2}{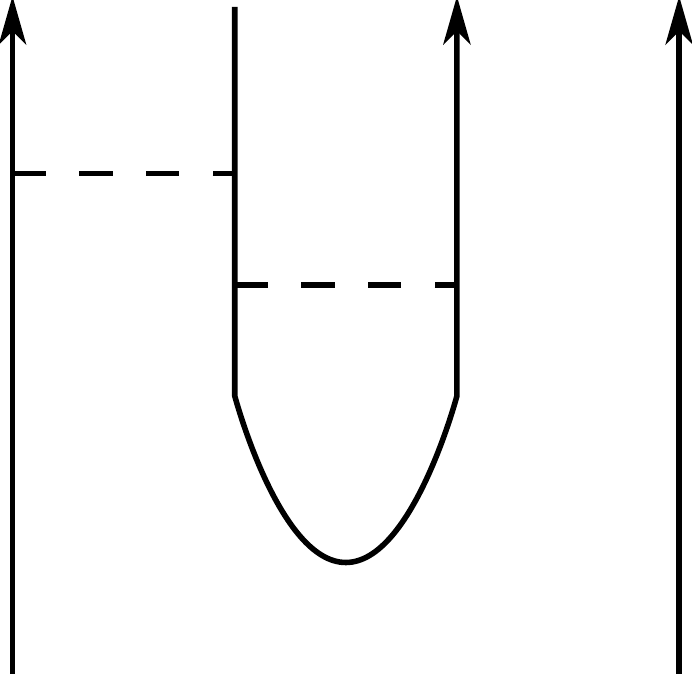}-\mathfig{2}{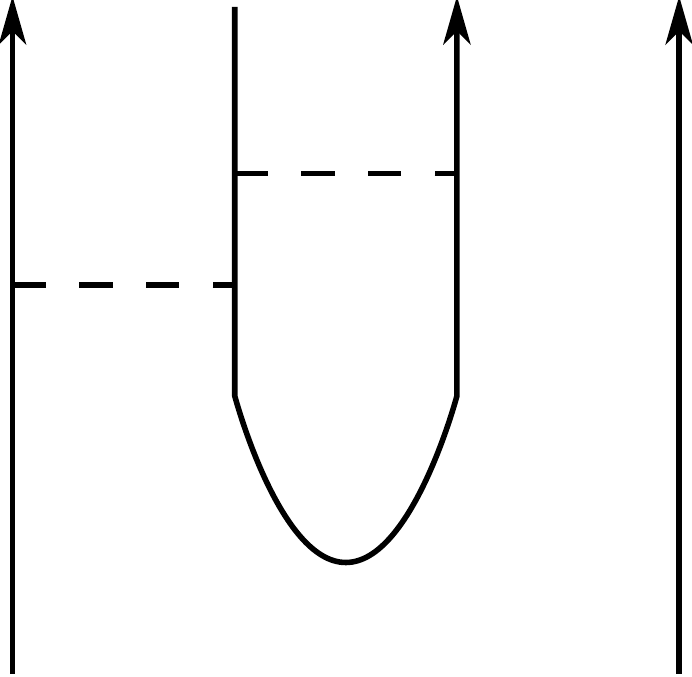}\right)=$$
$$=
\underbrace{\mathfig{2}{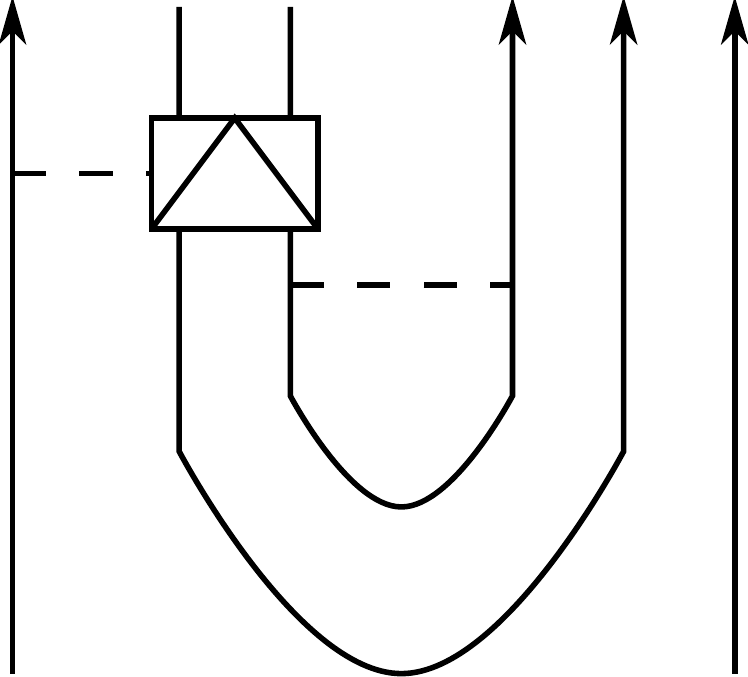}+\mathfig{2}{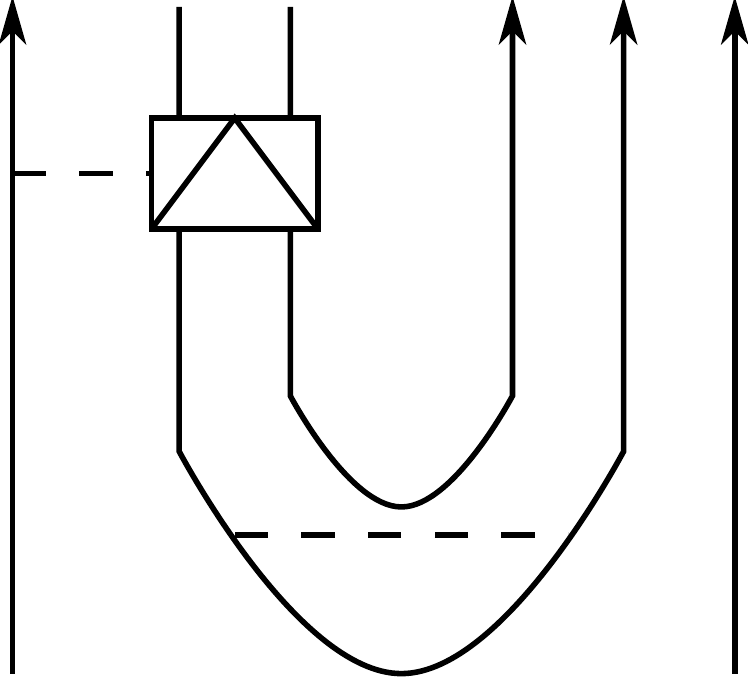}}_{\in H_3}-$$
$$-\underbrace{\mathfig{2}{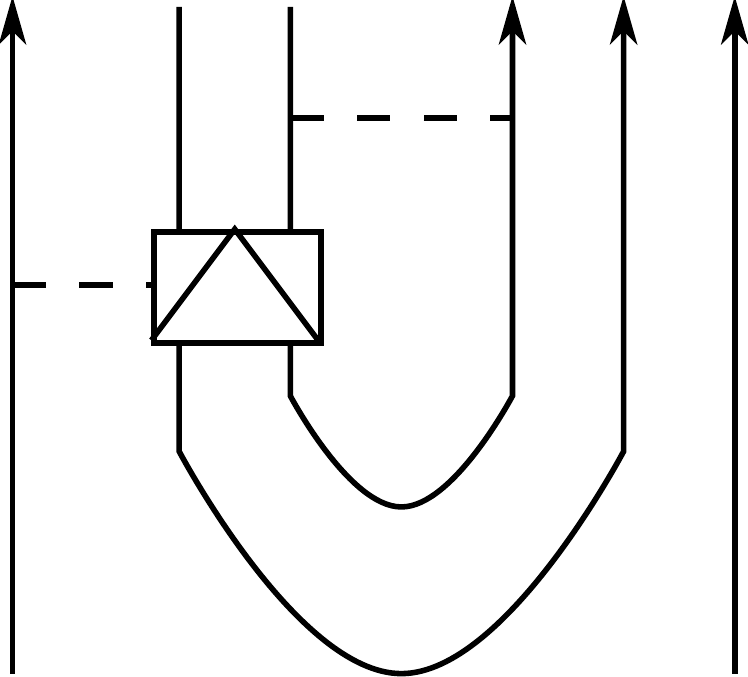}-\mathfig{2}{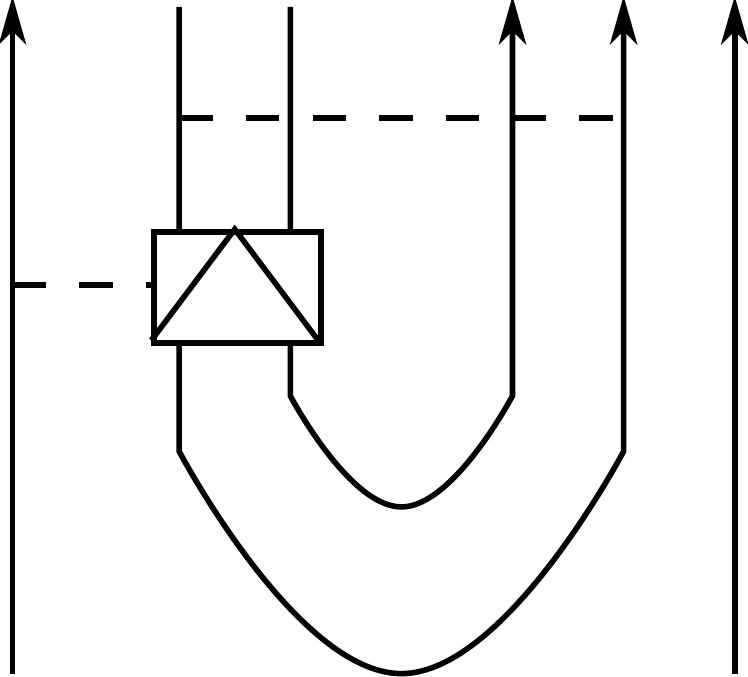}}_{\in H_3}+\mathfig{2}{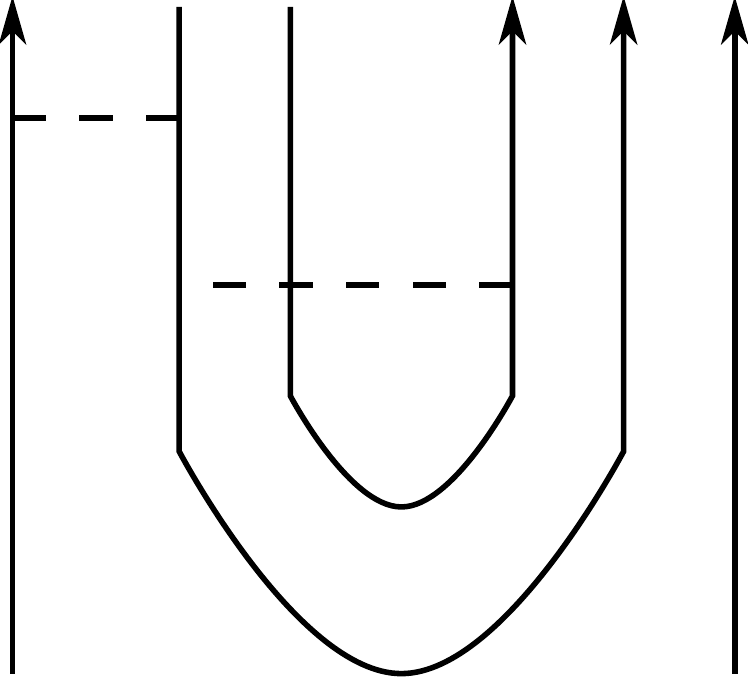}-\mathfig{2}{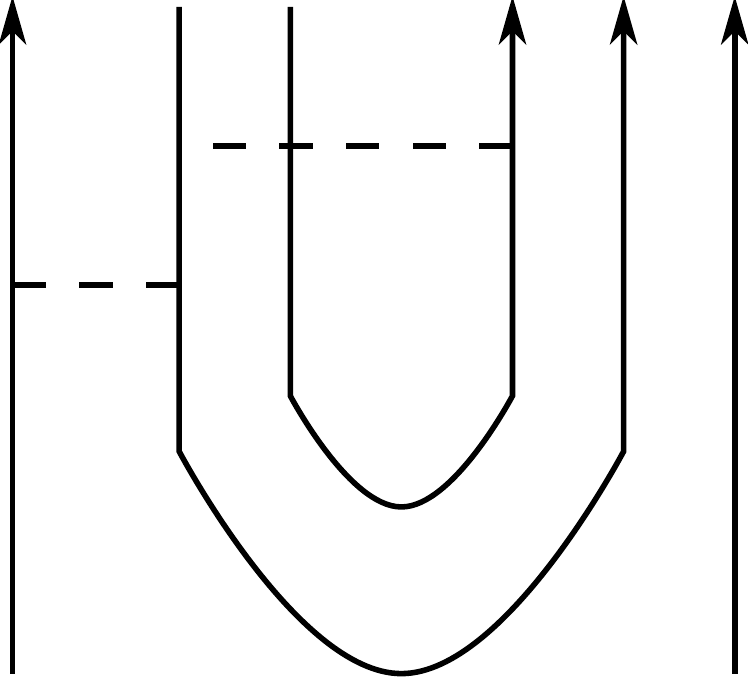}+$$
$$+\underbrace{\mathfig{2}{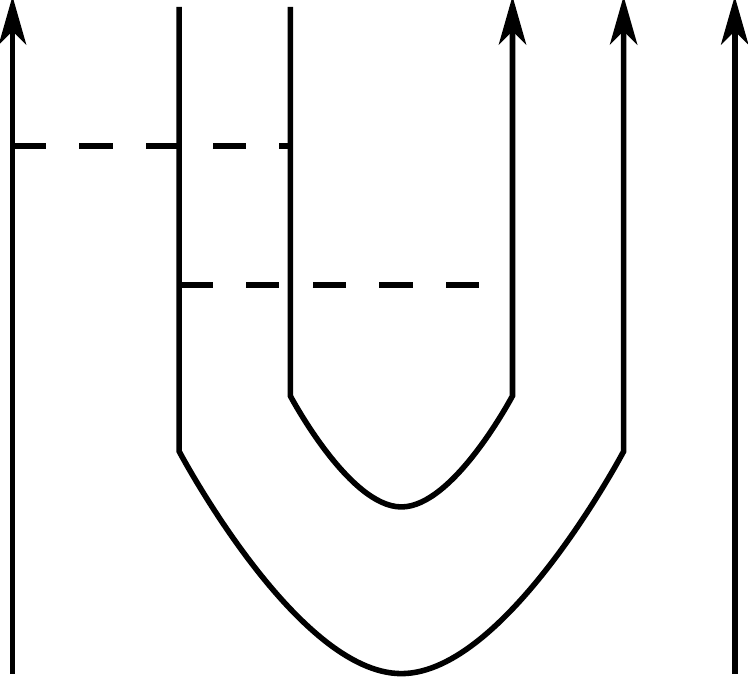}-\mathfig{2}{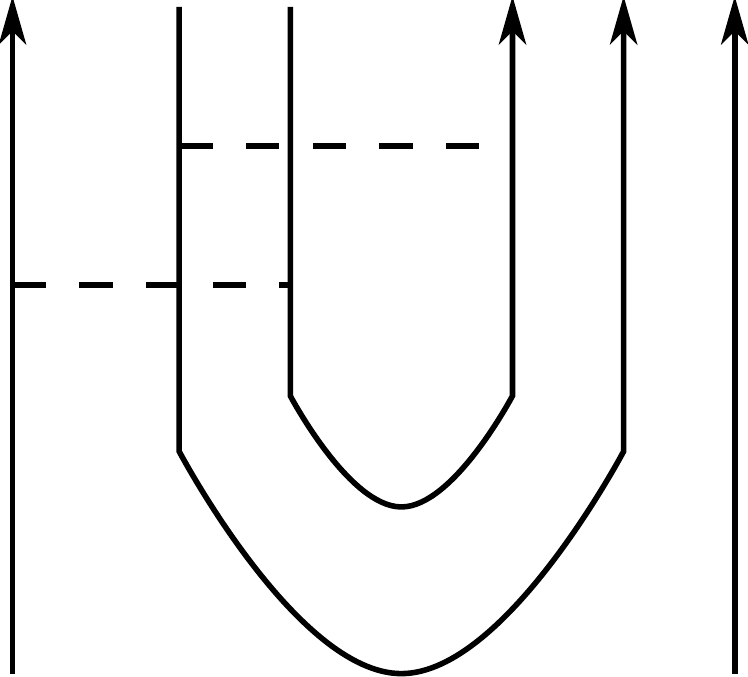}}_{\text{canceling}}+\underbrace{\mathfig{2}{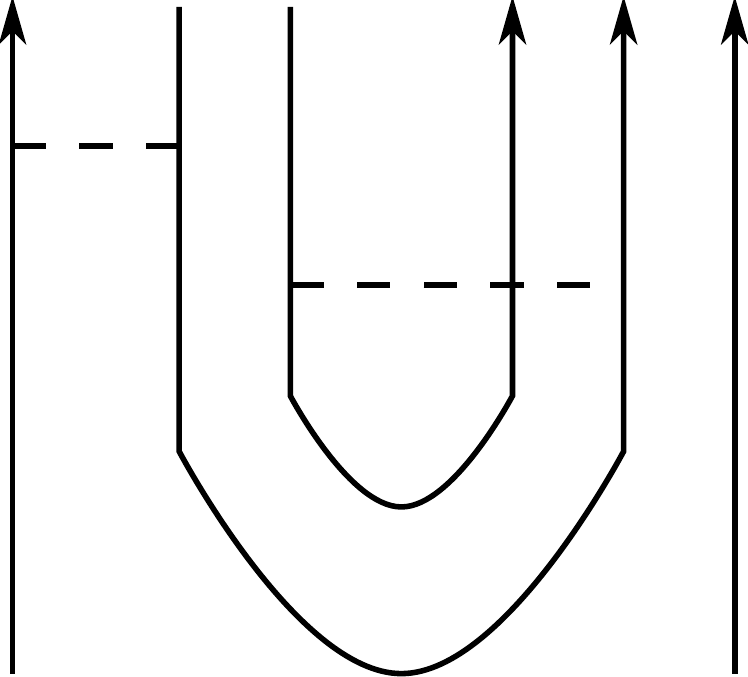}-\mathfig{2}{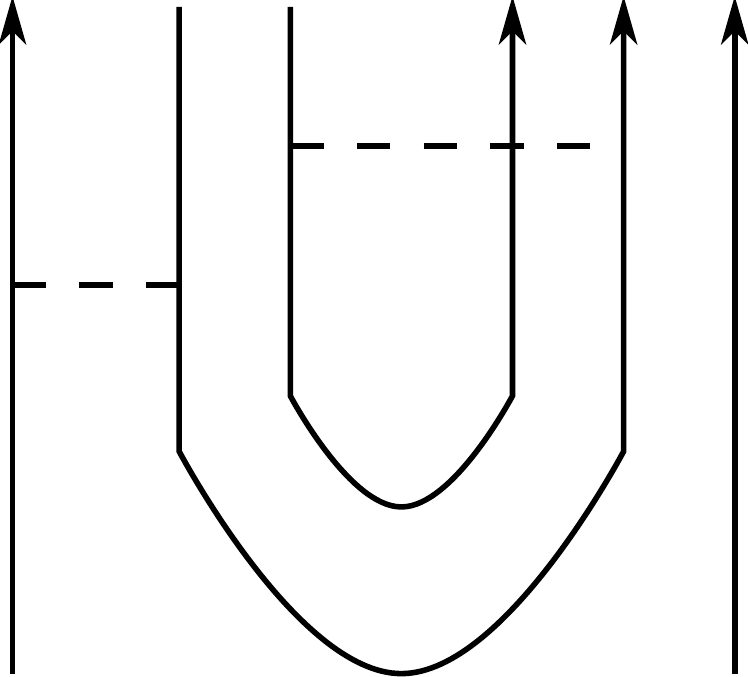}}_{\text{canceling}}+$$
$$+\mathfig{2}{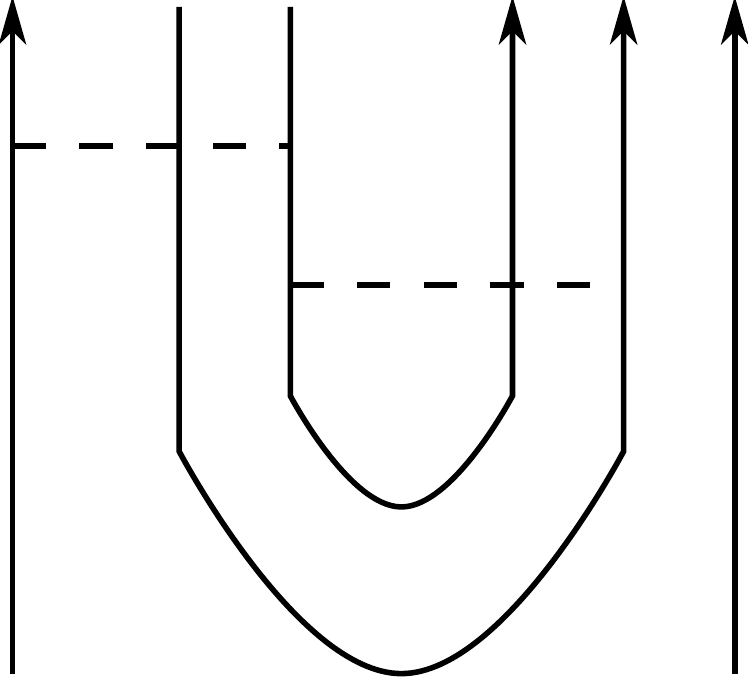}-\mathfig{2}{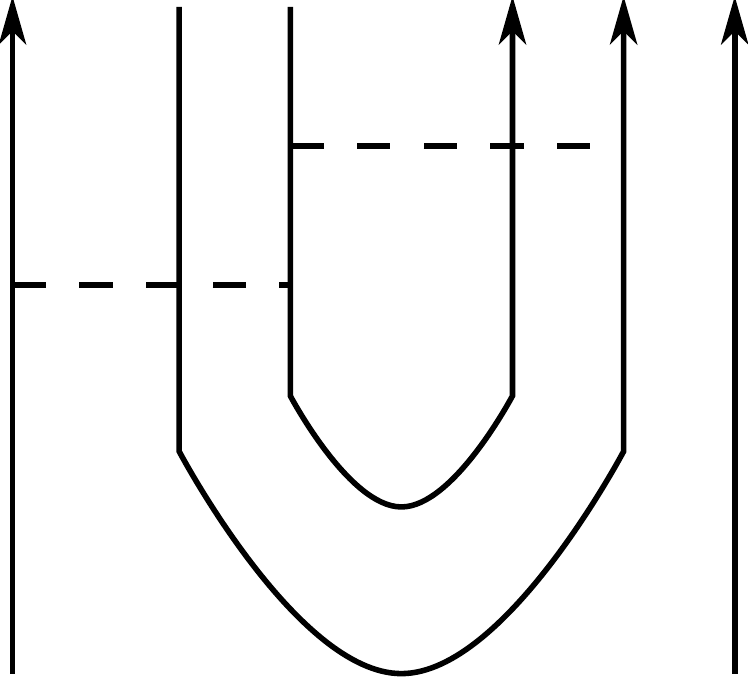}\approx\mathfig{2}{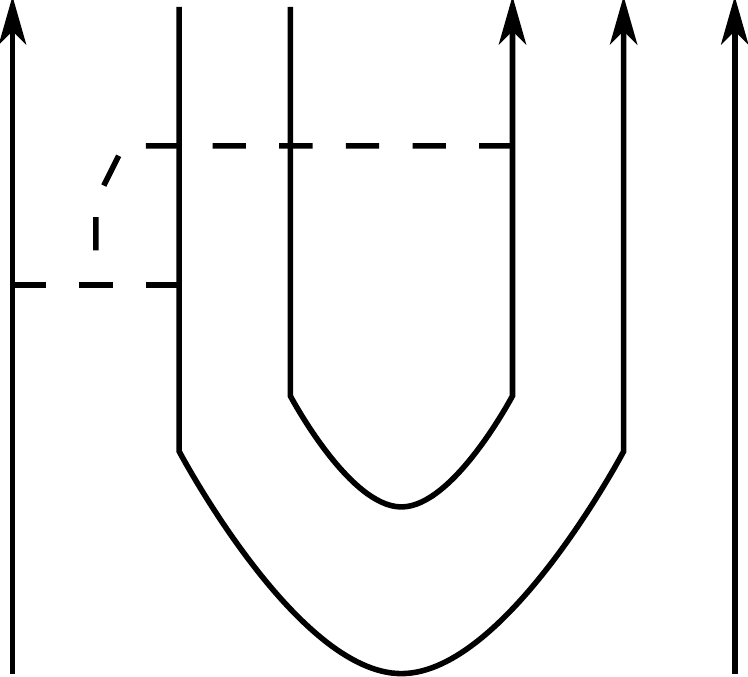}+\mathfig{2}{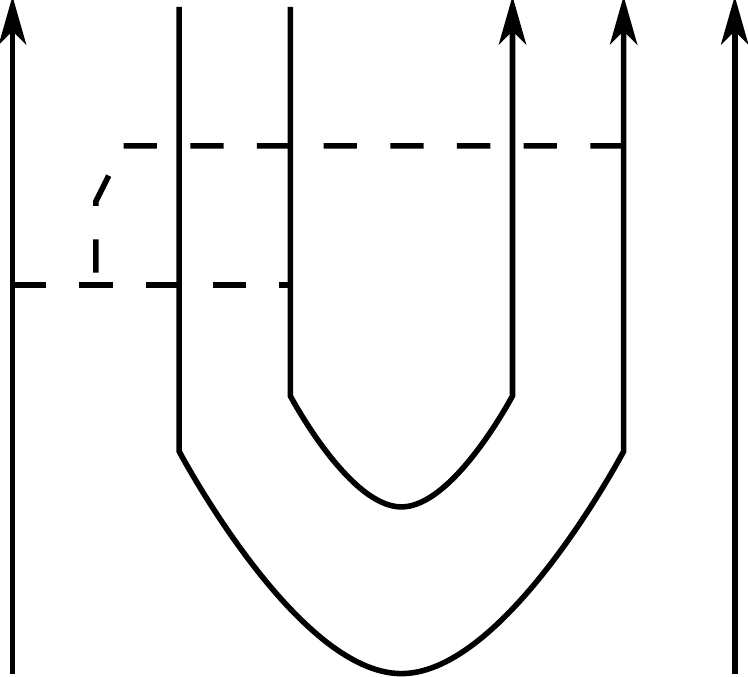}=0$$

We are left with summands in which the empty diagram of: $$\phi^{-1}\left(\mathfig{1}{section5_section53_t_12},\mathfig{1}{section5_section53_t_23}\right)\cdot\left(\mathfig{1}{section5_section53_cup}\right)$$ is multiplied by the $H_2$ part of: $$\phi\left(\mathfig{1}{section5_section53_t_13}-\mathfig{1}{section5_section53_t_23},\mathfig{1}{section5_section53_t_34}\right)$$
This associator is an exponent of a sum of iterated commutators of $\mathfig{1}{section5_section53_t_13}$, $\mathfig{1}{section5_section53_t_23}$ and $\mathfig{1}{section5_section53_t_34}$. If $\mathfig{1}{section5_section53_t_23}$ does not appear in this commutator, it does not contribute to the $H_2$ part. So we may consider only commutators in which $\mathfig{1}{section5_section53_t_23}$ appears. We may assume the commutator is one sided. It is also enough to consider commutators of the type $\left[\mathfig{1}{section5_section53_t_23},u\right]$ where $u$ is a commutator in $\mathfig{1}{section5_section53_t_13}$ and $\mathfig{1}{section5_section53_t_34}$ (because all the commutators we consider here have this type of commutator as their inner part). So we need to consider elements of the following form:
$$\mathfig{2}{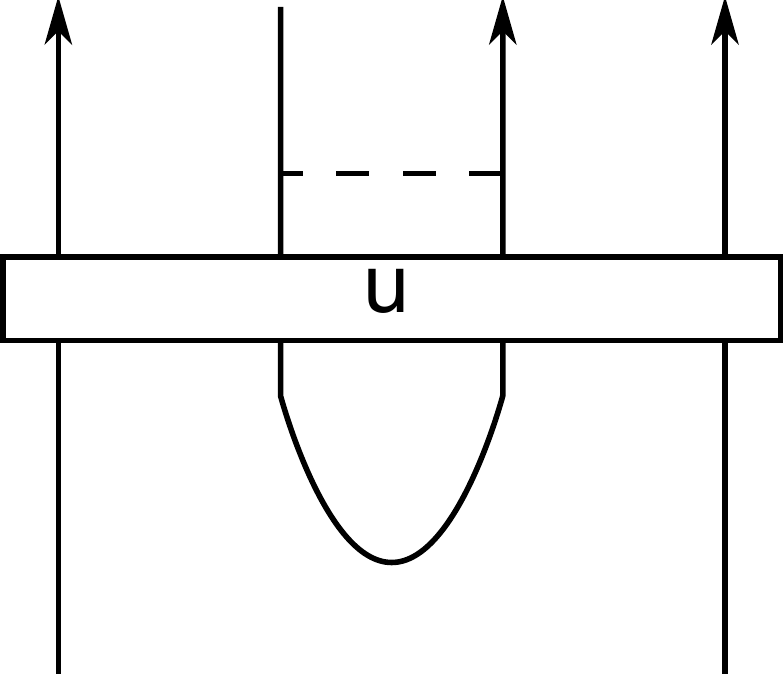}-\mathfig{2}{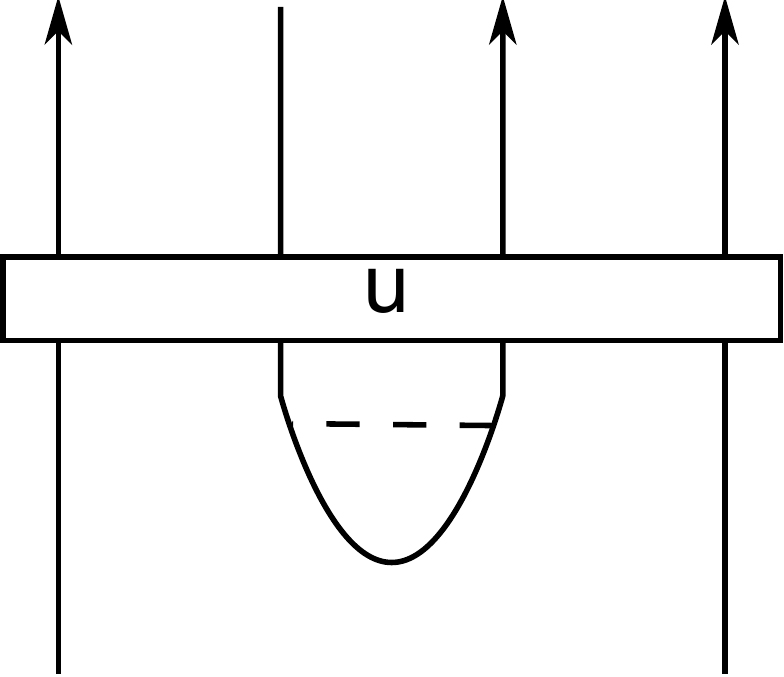}$$
where $u$ has no vertices on the down-going strand $\downarrow$. It is easy to prove (by induction) that by applying $\Delta$ to the strands $2$ and $3$ of $u$ we get $u_1+u_2$, where $u_1$ is obtained by putting all the vertices of strand $3$ (from the left) in $\uparrow\downarrow\uparrow\uparrow$ on strand $4$ of $\uparrow\downarrow\downarrow\uparrow\uparrow\uparrow$, and $u_2$ is obtained by putting all the vertices of strand $3$ in $\uparrow\downarrow\uparrow\uparrow$ on strand $5$ of $\uparrow\downarrow\downarrow\uparrow\uparrow\uparrow$. So we have:

$$\Delta^{++}_{++,+}\left(\mathfig{2}{section5_section53_proof2_1}-\mathfig{2}{section5_section53_proof2_2}\right)=$$
$$=\underbrace{\mathfig{2}{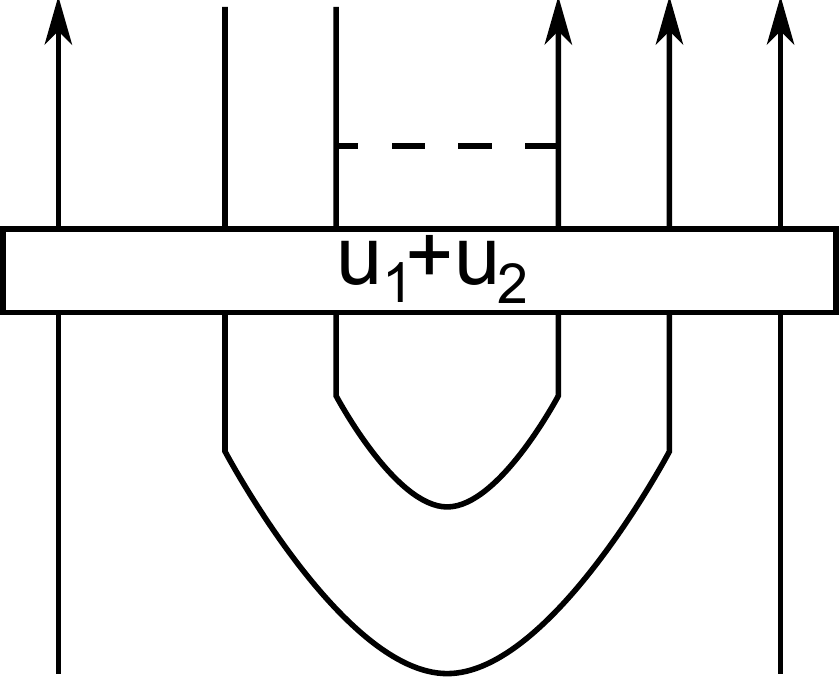}-\mathfig{2}{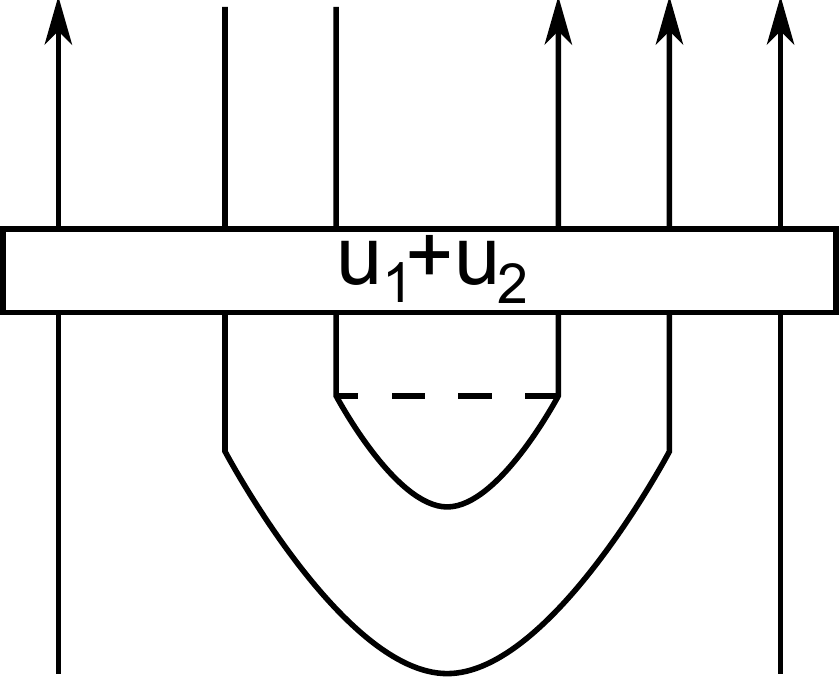}+\mathfig{2}{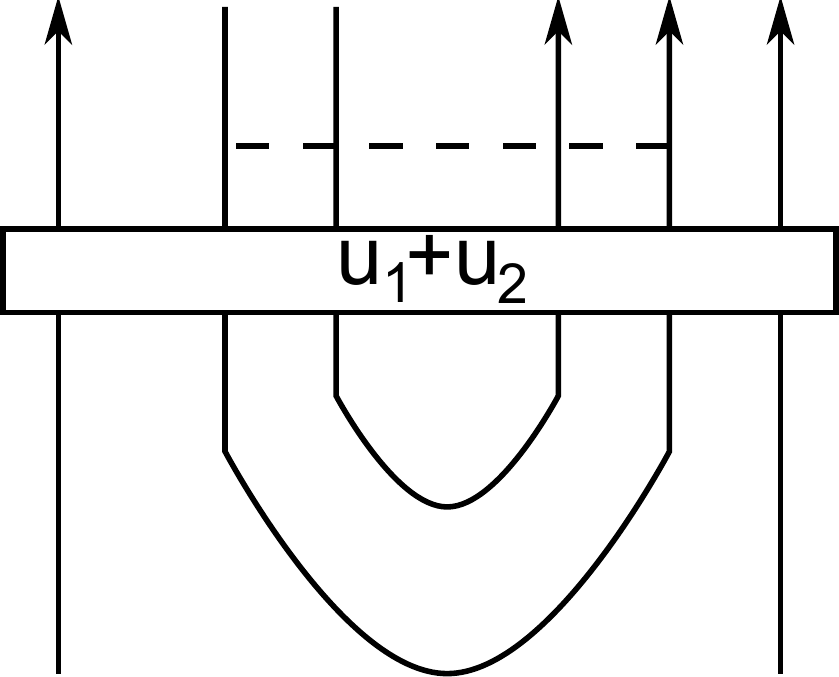}-\mathfig{2}{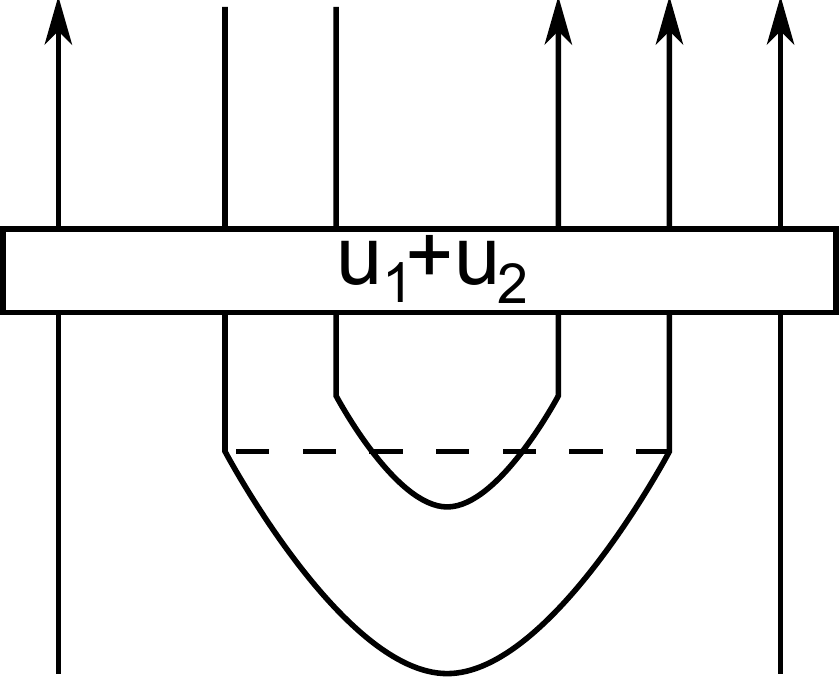}}_{\in H_3}+$$
$$+\underbrace{\mathfig{2}{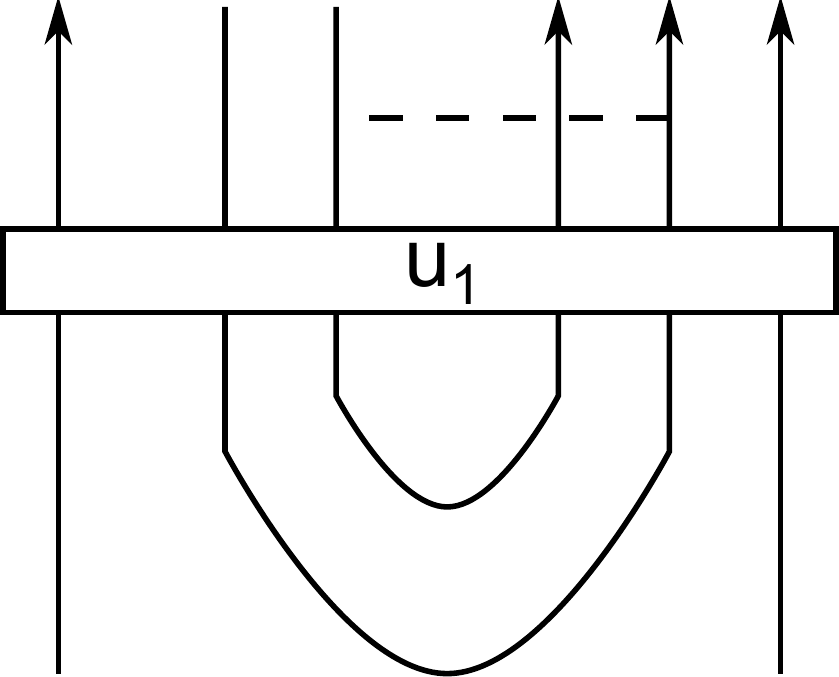}-\mathfig{2}{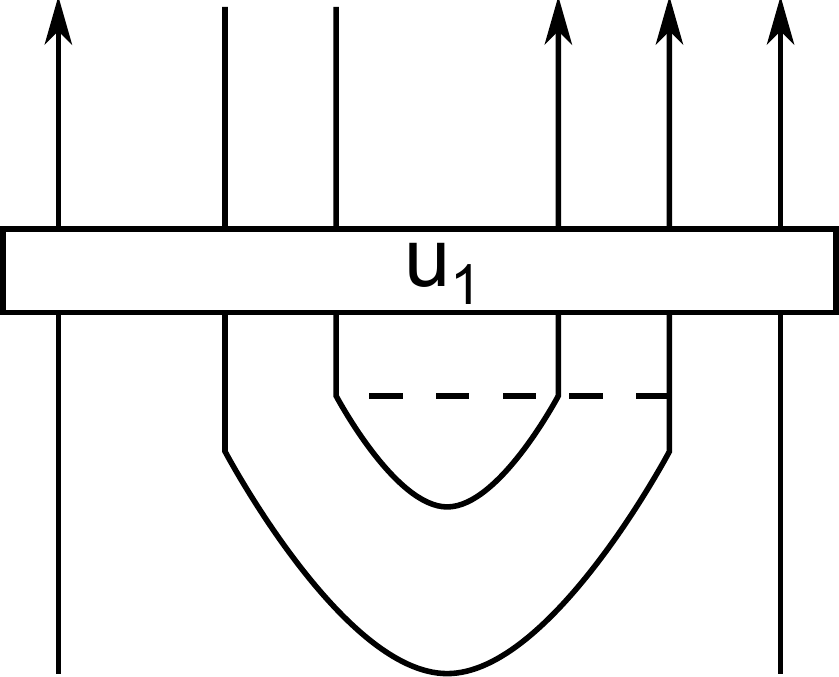}}_{\text{canceling}}+\underbrace{\mathfig{2}{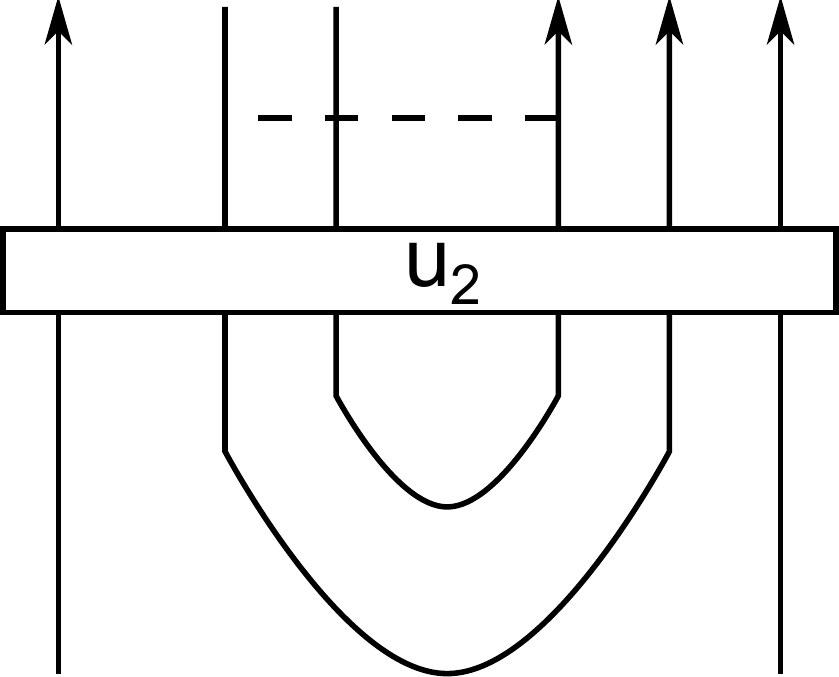}-\mathfig{2}{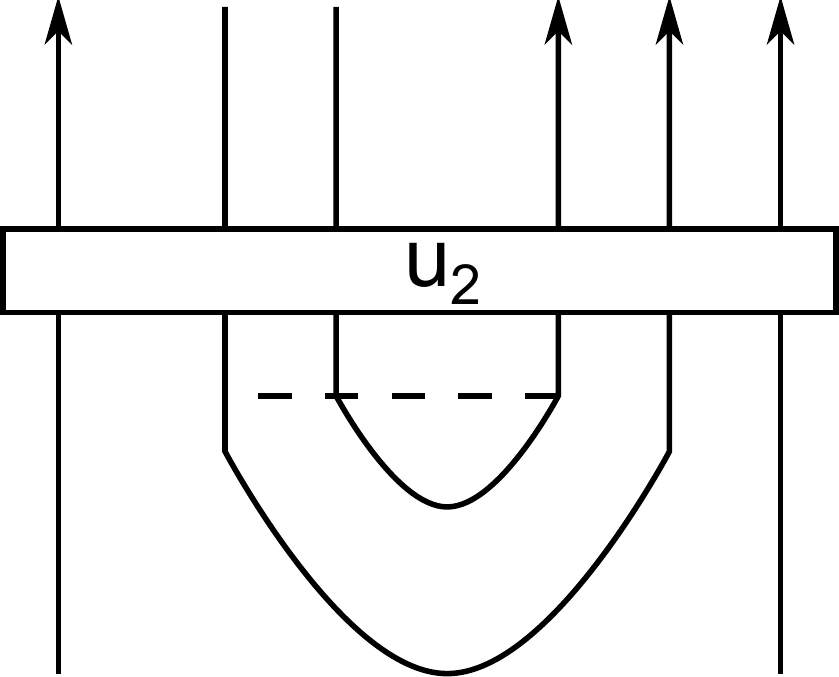}}_{\text{canceling}}+$$
$$+\mathfig{2}{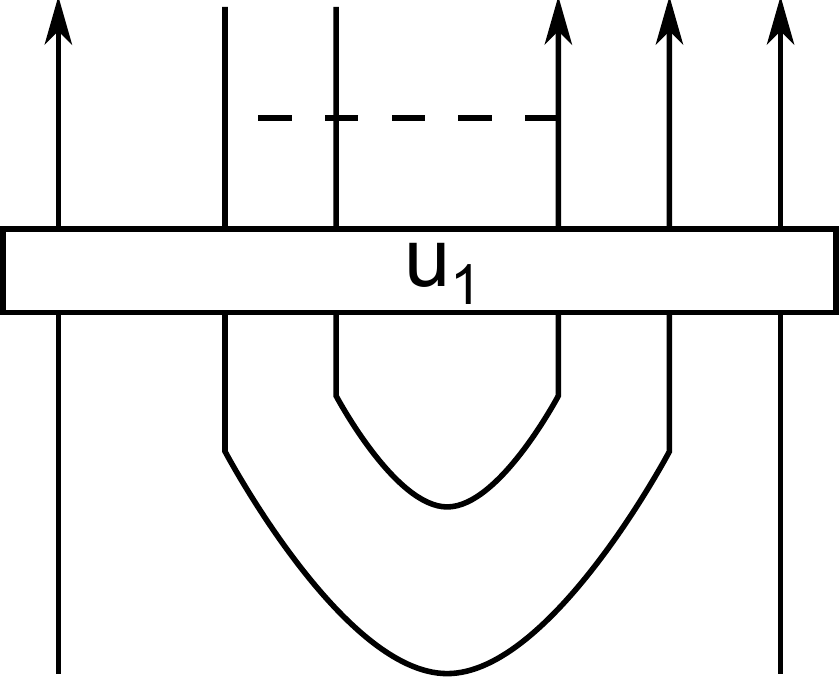}-\mathfig{2}{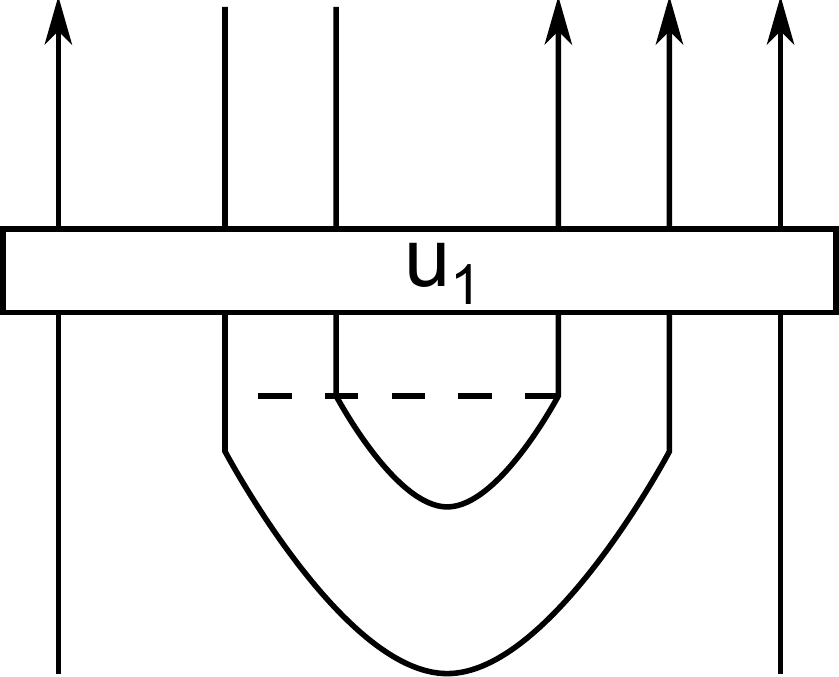}+\mathfig{2}{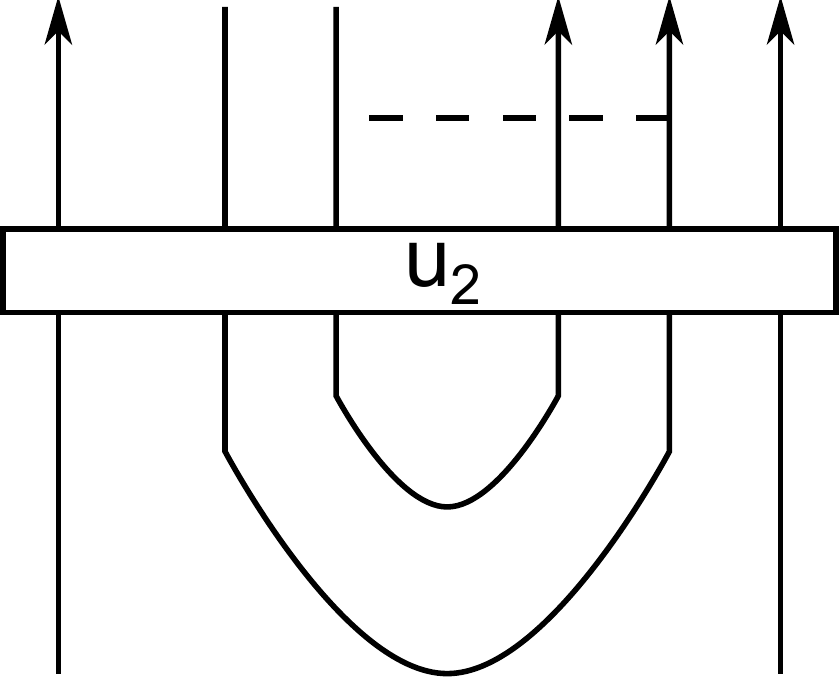}-\mathfig{2}{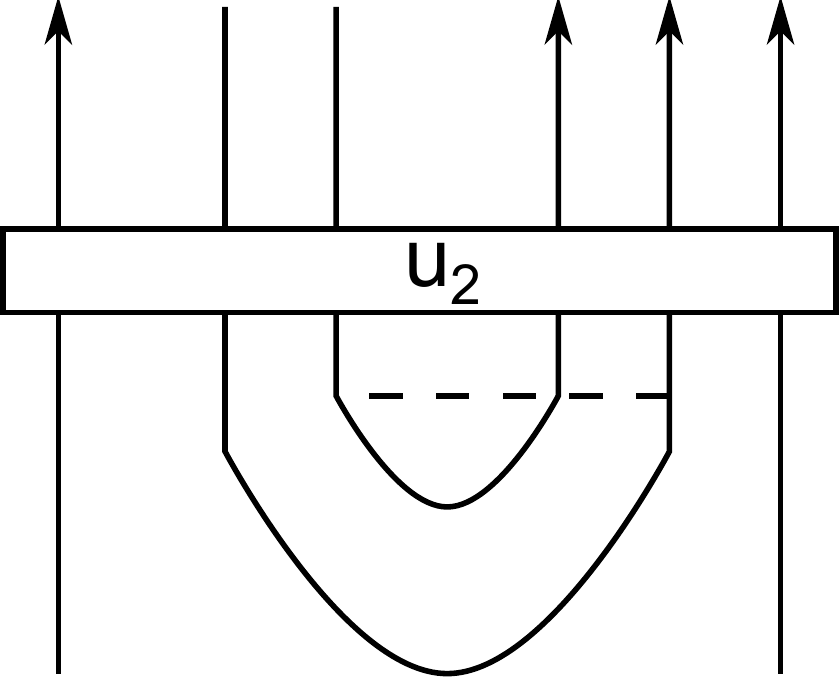}$$
$u_1$ and $u_2$ can be written as a sum of connected diagrams. The diagrams in this sum which have more than $1$ vertex on strand $4$ (for $u_1$) or on strand $5$ (for $u_2$) are already in $H_3$. For diagrams with only one vertex on those strands, the above sum equals $0$ via the STU relation.

At last we have to deal with the $H_2$ part coming from $C_{-+}$. Most of the summands in $C_{-+}$ are mapped to $H_3$ by the exact same argument we have just seen. We only need to show that $\Delta^{++}_{++,+}\left(\mathfig{1}{double_cap_tangle}\circ \phi(-t_{12},t_{24})\right)$ is in $H_3$.

$\phi(-t_{12},t_{24})$ is an exponent of a linear combination of commutators in $t_{12}$ and $t_{24}$. In each summand of the exponent, one of the commutators is closest to the caps at the top. This commutator can be written as a one-component diagram with only one vertex on strand $2$. So after applying $\Delta^{++}_{++,+}$ we get two copies of this diagram, each with a vertex on each copy of strand $2$. The same argument from the end of the proof of theorem \ref{theorem_related}, which showed that $C_{-+}$ is in $H_2$, shows now that each of these copies is in $H_3$. This completes the proof for $LMO^<(Y_{+,+})$, and hence the proof of theorem \ref{theorem_elliptic}.

\end{proof}

\end{onehalfspace}
\end{document}